\documentclass[numbers=enddot,12pt,final,onecolumn,notitlepage,abstracton]{scrartcl}
\usepackage[headsepline,footsepline,manualmark]{scrlayer-scrpage}
\usepackage[all,cmtip]{xy}
\usepackage{amssymb}
\usepackage{amsmath}
\usepackage{amsthm}
\usepackage{framed}
\usepackage{comment}
\usepackage{color}
\usepackage{tabu}
\usepackage[sc]{mathpazo}
\usepackage[T1]{fontenc}
\usepackage{needspace}
\usepackage[breaklinks]{hyperref}
\providecommand{\U}[1]{\protect\rule{.1in}{.1in}}
\theoremstyle{definition}
\newtheorem{theo}{Theorem}[section]
\newenvironment{theorem}[1][]
{\begin{theo}[#1]\begin{leftbar}}
{\end{leftbar}\end{theo}}
\newtheorem{lem}[theo]{Lemma}
\newenvironment{lemma}[1][]
{\begin{lem}[#1]\begin{leftbar}}
{\end{leftbar}\end{lem}}
\newtheorem{prop}[theo]{Proposition}
\newenvironment{proposition}[1][]
{\begin{prop}[#1]\begin{leftbar}}
{\end{leftbar}\end{prop}}
\newtheorem{defi}[theo]{Definition}
\newenvironment{definition}[1][]
{\begin{defi}[#1]\begin{leftbar}}
{\end{leftbar}\end{defi}}
\newtheorem{remk}[theo]{Remark}
\newenvironment{remark}[1][]
{\begin{remk}[#1]\begin{leftbar}}
{\end{leftbar}\end{remk}}
\newtheorem{coro}[theo]{Corollary}
\newenvironment{corollary}[1][]
{\begin{coro}[#1]\begin{leftbar}}
{\end{leftbar}\end{coro}}
\newtheorem{conv}[theo]{Convention}

\newtheorem{quest}[theo]{Question}
\newenvironment{question}[1][]
{\begin{quest}[#1]\begin{leftbar}}
{\end{leftbar}\end{quest}}
\newtheorem{warn}[theo]{Warning}

\newtheorem{soln}{Solution}

\newtheorem{conj}[theo]{Conjecture}

\newtheorem{exam}[theo]{Example}
\newenvironment{example}[1][]
{\begin{exam}[#1]\begin{leftbar}}
{\end{leftbar}\end{exam}}
\newenvironment{statement}{\begin{quote}}{\end{quote}}
\let\sumnonlimits\sum
\let\prodnonlimits\prod
\let\cupnonlimits\bigcup
\let\capnonlimits\bigcap
\renewcommand{\sum}{\sumnonlimits\limits}
\renewcommand{\prod}{\prodnonlimits\limits}
\renewcommand{\bigcup}{\cupnonlimits\limits}
\renewcommand{\bigcap}{\capnonlimits\limits}
\newenvironment{verlong}{}{}
\newenvironment{vershort}{}{}
\newenvironment{noncompile}{}{}
\excludecomment{verlong}
\includecomment{vershort}
\excludecomment{noncompile}
\newcommand{\kk}{{\mathbf{k}}}
\newcommand{\xx}{{\mathbf{x}}}
\newcommand{\id}{{\operatorname{id}}}
\newcommand{\ev}{\operatorname{ev}}
\newcommand{\Adm}{\operatorname{Adm}}
\newcommand{\pack}{\operatorname{pack}}
\newcommand{\Mon}{\operatorname{Mon}}
\newcommand{\Comp}{{\operatorname{Comp}}}
\newcommand{\QSym}{{\operatorname{QSym}}}

\newcommand{\Powser}{\mathbf{k}\left[\left[x_1,x_2,x_3,\ldots\right]\right]}
\newcommand{\Par}{\operatorname{Par}}
\newcommand{\bdd}{\operatorname{bdd}}
\newcommand{\sign}{\operatorname{sign}}
\newcommand{\Stab}{\operatorname{Stab}}

\newcommand{\EE}{{\mathbf{E}}}
\newcommand{\FF}{{\mathbf{F}}}
\newcommand{\bk}{{\mathbf{k}}}
\newcommand{\NN}{{\mathbb{N}}}
\newcommand{\ZZ}{{\mathbb{Z}}}
\newcommand{\QQ}{{\mathbb{Q}}}
\ihead{Double posets and the antipode of $\QSym$}
\ohead{page \thepage}
\begin{document}

\author{Darij Grinberg}

\title{Double posets and the antipode of $\QSym$}

\date{version 3.2 (April 28, 2017, typos corrected 20 June 2026) \\[1pc]
Corrected version of:
\href{https://doi.org/10.37236/6660}{%
Darij Grinberg, \textit{Double Posets and the Antipode of QSym},
The Electronic Journal of Combinatorics \textbf{24} (2), P2.22.}
}

\maketitle

\begin{abstract}
A quasisymmetric function is assigned to every double poset (that is,
every finite set endowed with two partial orders) and any weight function
on its ground set. This generalizes well-known objects such as monomial
and fundamental quasisymmetric functions, (skew) Schur functions, dual
immaculate functions, and quasisymmetric
$\left(P, \omega\right)$-partition enumerators.
We prove a formula for the antipode of this function that
holds under certain conditions (which are satisfied when the second order
of the double poset is total, but also in some other cases); this
restates (in a way that to us seems more natural) a
result by Malvenuto and Reutenauer, but our proof is new and
self-contained. We generalize it further to an even more comprehensive
setting, where a group acts on the double poset by automorphisms.
\\[0.2pc]

\textbf{Keywords:}
antipodes,
double posets,
Hopf algebras,
posets,
P-partitions,
quasisymmetric functions.
\\[0.2pc]

\textbf{MSC2010 Mathematics Subject Classifications:}
05E05, 05E18.
\end{abstract}

\begin{verlong}
\tableofcontents
\end{verlong}

\section{Introduction}
\label{sec:in}

Double posets and $\EE$-partitions (for $\EE$ a double poset)
have been introduced by Claudia Malvenuto and Christophe
Reutenauer \cite{Mal-Reu-DP}; their goal was to construct a
combinatorial Hopf algebra which harbors a noticeable amount of
structure, including an analogue of the Littlewood-Richardson
rule and a lift of the internal product operation of the
Malvenuto-Reutenauer Hopf algebra of permutations. In this note,
we shall employ these same notions to restate in a simpler form,
and reprove in a more elementary fashion, a formula for the
antipode in the Hopf algebra $\QSym$ of quasisymmetric functions
due to (the same) Malvenuto and Reutenauer
\cite[Theorem 3.1]{Mal-Reu}. We then further generalize this
formula to a setting in which a group acts on the double poset
(a generalization inspired by Katharina Jochemko's
\cite{Joch}).

\begin{vershort}
In the present version of the paper, some (classical and/or
straightforward) proofs are missing or sketched. A more detailed
version exists, in which at least a few of these proofs are
elaborated on more\footnote{It can be downloaded from \newline
\url{http://www.cip.ifi.lmu.de/~grinberg/algebra/dp-abstr-long.pdf} .
It is also archived as an ancillary file on
\url{http://arxiv.org/abs/1509.08355v3}, although the former
website is more likely to be updated.}.
\end{vershort}
\begin{verlong}
The present version of the paper is the
\textit{detailed version}\footnote{This version can be downloaded
from \newline
\url{http://www.cip.ifi.lmu.de/~grinberg/algebra/dp-abstr-long.pdf} .
It is also archived as an ancillary file on
\url{http://arxiv.org/abs/1509.08355v3}, although the former
website is more likely to be updated.}.
A standard version is also available\footnote{at
\url{http://www.cip.ifi.lmu.de/~grinberg/algebra/dp-abstr.pdf}
and
\url{http://arxiv.org/abs/1509.08355v3}}.
The two versions differ in that the detailed version contains extra
details in various proofs (although the level of detail is not
always consistent).
\end{verlong}

A short summary of this paper has been submitted to the FPSAC
conference \cite{Gri-extabs}.

\subsection*{Acknowledgments}

Katharina Jochemko's work \cite{Joch} provoked this research.
I learnt a lot about $\QSym$ from Victor Reiner. The SageMath
computer algebra system \cite{SageMath} was used for some
computations that suggested one of the proofs.

\subsection*{Note on the published version of this paper}

The document you are reading is the
\begin{verlong}
detailed version of a
\end{verlong}
preprint of a paper (of the
same title) that was
published in the
\href{http://www.combinatorics.org/}{Electronic Journal of
Combinatorics}
in 2017
(\href{https://doi.org/10.37236/6660}{Volume 24, Issue 2,
doi:10.37236/6660}).
The published version differs from
\begin{vershort}
this preprint
\end{vershort}
\begin{verlong}
the standard version of this preprint
\end{verlong}
insubstantially\footnote{The main difference is that in the
published version, the long footnote in
Section~\ref{sect.qsym-intro} has been relegated into a
separate subsection (\S 2.2), whereas the remainder of
Section~\ref{sect.qsym-intro} has become \S 2.1.
Other than this, the two versions differ in formatting and
editorialization.}.

\section{Quasisymmetric functions}
\label{sect.qsym-intro}

Let us first briefly introduce the notations that will be used in the
following.

We set $\NN = \left\{0, 1, 2, \ldots\right\}$. A \textit{composition}
means a finite sequence of positive integers. We let $\Comp$ be the set
of all compositions. For $n \in \NN$, a \textit{composition of $n$}
means a composition whose entries sum to $n$ (that is, a composition
$\left(\alpha_1, \alpha_2, \ldots, \alpha_k\right)$ satisfying
$\alpha_1 + \alpha_2 + \cdots + \alpha_k = n$).

Let $\kk$ be an arbitrary commutative ring. We shall keep $\kk$ fixed
throughout this paper.
We consider the $\kk$-algebra $\Powser$ of formal
power series in infinitely many (commuting) indeterminates
$x_1, x_2, x_3, \ldots$ over $\kk$. A \textit{monomial} shall always
mean a monomial (without coefficients) in the variables
$x_1, x_2, x_3, \ldots$.\ \ \ \ \footnote{For the sake of completeness,
let us give a detailed definition of monomials and of the topology
on $\Powser$. (This definition has been copied from
\cite[\S 2]{Gri-dimm}, essentially unchanged.)

Let $x_{1},x_{2},x_{3},\ldots$ be countably many distinct symbols. We
let $\Mon$ be the free abelian monoid on the set $\left\{
x_{1},x_{2},x_{3},\ldots\right\}  $ (written multiplicatively); it consists of
elements of the form $x_{1}^{a_{1}}x_{2}^{a_{2}}x_{3}^{a_{3}}\cdots$ for
finitely supported $\left(  a_{1},a_{2},a_{3},\ldots\right)  \in
 \NN ^{\infty}$ (where \textquotedblleft finitely
supported\textquotedblright\ means that all but finitely many positive
integers $i$ satisfy $a_{i}=0$). A \textit{monomial} will mean an element of
$\Mon$. Thus, a monomial is a combinatorial
object, independent of $\kk$; it does not carry a coefficient.

We consider the $\kk$-algebra $\Powser$ of (commutative) power
series in
countably many distinct indeterminates $x_{1},x_{2},x_{3},\ldots$ over
$\kk$. By abuse of notation, we shall identify every monomial
$x_{1}^{a_{1}} x_{2}^{a_{2}} x_{3}^{a_{3}} \cdots \in \Mon$ with the
corresponding element
$x_{1}^{a_{1}} \cdot x_{2}^{a_{2}} \cdot x_{3}^{a_{3}} \cdot \cdots$
of $\Powser$ when necessary
(e.g., when we speak of the sum of two monomials or
when we multiply a monomial with an element of $\kk$). (To be
very pedantic, this identification is slightly dangerous, because
it can happen that two distinct monomials in $\Mon$ get
identified with two identical elements of $\Powser$. However, this
can only happen when the ring $\kk$ is trivial, and even then it is
not a real problem unless we infer the equality of monomials from the
equality of their counterparts in $\Powser$, which we are not going to
do.)

We furthermore endow the ring $\Powser$ with the following topology
(as in \cite[Proof of Corollary 2.6.11]{Reiner}):

We endow the ring $\kk$ with the discrete topology. To define a
topology on the $\kk$-algebra $\Powser$, we (temporarily) regard every power
series in $\Powser$ as the family of its coefficients (indexed by the
set $\Mon$). More precisely, we have a
$\kk$-module isomorphism
\[
\prod_{\mathfrak{m} \in  \Mon } \kk \to \Powser,
\qquad \qquad \left(\lambda_{\mathfrak{m}}\right)_{\mathfrak{m} \in  \Mon }
\mapsto \sum_{\mathfrak{m} \in  \Mon } \lambda_{\mathfrak{m}} \mathfrak{m} .
\]
We use this isomorphism to transport the product topology on
$\prod_{\mathfrak{m} \in  \Mon } \kk$ to $\Powser$. The
resulting topology on $\Powser$ turns $\Powser$ into a topological
$\kk$-algebra; this is the topology that we will
be using whenever we make statements about convergence in $\Powser$
or write down infinite sums of power series.
A sequence $\left( a_n \right)_{n \in \NN}$ of power series converges
to a power series $a$ with respect to this topology if
and only if for every monomial $\mathfrak{m}$, all sufficiently high
$n \in \NN$ satisfy
\[
\left(  \text{the coefficient of } \mathfrak{m}\text{ in }a_{n}\right)
=\left(  \text{the coefficient of } \mathfrak{m}\text{ in }a\right)  .
\]

Note that this topological $\kk$-algebra $\Powser$ is \textbf{not}
the completion of the polynomial ring
$\kk \left[  x_{1},x_{2},x_{3},\ldots\right]$
with respect to the standard grading (in which all $x_{i}$ have degree $1$).
(They are distinct even as sets.)
}

Inside the $\kk$-algebra $\Powser$ is a
subalgebra $\Powser_{\bdd}$ consisting of the \textit{bounded-degree}
formal power series; these are the power series $f$ for which there
exists a $d \in \NN$ such that no monomial of degree $> d$ appears in
$f$\ \ \ \ \footnote{The \textit{degree} of a monomial
$x_1^{a_1} x_2^{a_2} x_3^{a_3} \cdots$ is defined to be the nonnegative
integer $a_1 + a_2 + a_3 + \cdots$. A monomial $\mathfrak{m}$ is said
to \textit{appear} in a power series $f \in \Powser$ if and only if
the coefficient of $\mathfrak{m}$ in $f$ is nonzero.}.
This $\kk$-subalgebra $\Powser_{\bdd}$ becomes a topological
$\kk$-algebra, by inheriting the topology from $\Powser$.

Two monomials $\mathfrak{m}$ and $\mathfrak{n}$ are said to be
\textit{pack-equivalent}\footnote{Pack-equivalence and the related
notions of packed combinatorial objects that we will encounter below
originate in work of Hivert, Novelli and Thibon
\cite{Nov-Thi}. Simple as they are, they are of great help in dealing
with quasisymmetric functions.} if they have the forms
$x_{i_1}^{a_1} x_{i_2}^{a_2} \cdots x_{i_\ell}^{a_\ell}$ and
$x_{j_1}^{a_1} x_{j_2}^{a_2} \cdots x_{j_\ell}^{a_\ell}$ for two
strictly increasing sequences
$\left(i_1 < i_2 < \cdots < i_\ell\right)$
and $\left(j_1 < j_2 < \cdots < j_\ell\right)$ of positive integers and
one (common) sequence $\left(a_1, a_2, \ldots, a_\ell\right)$ of
positive integers.\footnote{For instance, $x_2^2 x_3 x_4^2$ is
pack-equivalent to $x_1^2 x_4 x_8^2$ but not to $x_2 x_3^2 x_4^2$.}
A power series $f \in \Powser$ is said to be \textit{quasisymmetric}
if it satisfies the following condition:
If $\mathfrak{m}$ and $\mathfrak{n}$ are two pack-equivalent
monomials, then the coefficient of $\mathfrak{m}$ in $f$ equals
the coefficient of $\mathfrak{n}$ in $f$.

It is easy to see that the quasisymmetric
power series form a $\kk$-subalgebra of $\Powser$. But usually one
is interested in a subset of this $\kk$-subalgebra: namely,
the set of quasisymmetric bounded-degree power
series in $\Powser$.
This latter set is a $\kk$-subalgebra of $\Powser_{\bdd}$, and
is known as the \textit{$\kk$-algebra of quasisymmetric functions
over $\kk$}. It is denoted by $\QSym$.

The symmetric functions (in the
usual sense of this word in combinatorics -- so, really, symmetric
bounded-degree power series in $\Powser$) form a $\kk$-subalgebra
of $\QSym$. The quasisymmetric functions have a rich theory which
is related to, and often sheds new light on, the classical theory of
symmetric functions; expositions can be found in
\cite[\S\S 7.19, 7.23]{Stanley-EC2} and \cite[\S\S 5-6]{Reiner} and
other sources.\footnote{The notion of quasisymmetric functions goes
back to Gessel in 1984 \cite{Gessel}; they have been studied by many
authors, most significantly Malvenuto and Reutenauer
\cite{Mal-Reu-dua}.}

As a $\kk$-module, $\QSym$ has a basis
$\left(M_\alpha\right)_{\alpha \in \Comp}$ indexed by all
compositions, where the quasisymmetric function $M_\alpha$ for a
given composition $\alpha$ is defined as follows: Writing $\alpha$
as $\left(\alpha_1, \alpha_2, \ldots, \alpha_\ell\right)$, we set
\[
M_\alpha
= \sum_{i_1 < i_2 < \cdots < i_\ell}
 x_{i_1}^{\alpha_1} x_{i_2}^{\alpha_2} \cdots x_{i_\ell}^{\alpha_\ell}
= \sum_{\substack{\mathfrak{m}\text{ is a monomial pack-equivalent} \\
                  \text{to }
                  x_1^{\alpha_1} x_2^{\alpha_2} \cdots x_\ell^{\alpha_\ell}}}
  \mathfrak{m}
\]
(where the $i_k$ in the first sum are positive integers).
\begin{verlong}
\footnote{
The second equality sign in this equality is proven
in the Appendix (see
Proposition~\ref{prop.Malpha.equivalent}).
}
\end{verlong}
This
basis $\left(M_\alpha\right)_{\alpha \in \Comp}$ is known as the
\textit{monomial basis} of $\QSym$, and is the simplest to
define among many. (We shall briefly encounter another basis in
Example~\ref{exam.Gamma}.)

The $\kk$-algebra $\QSym$ can be endowed with a structure of a
$\kk$-coalgebra which, combined with its $\kk$-algebra structure,
turns it into a Hopf algebra. We refer to the literature both for
the theory of coalgebras and Hopf algebras
(see
\cite{Montg-Hopf}, \cite[\S 1]{Reiner}, \cite[\S 1-\S 2]{Manchon-HA},
\cite{Abe-HA}, \cite{Sweedler-HA}, \cite{Dasca-HA} or
\cite[Chapter 7]{Fresse-Op})
and for a deeper study of the Hopf algebra $\QSym$ (see
\cite{Malve-Thesis}, \cite[Chapter 6]{HGK} or
\cite[\S 5]{Reiner}); in this note we shall need but the very
basics of this structure, and so it is only them that we introduce.

In the following, all tensor products are over $\kk$ by
default (i.e., the sign $\otimes$ stands for $\otimes_{\kk}$ unless
it comes with a subscript).

Now, we define two $\kk$-linear maps $\Delta$ and $\varepsilon$
as follows\footnote{Both of their definitions rely on the fact
that
$\left( M_{\left(\alpha_1, \alpha_2, \ldots, \alpha_\ell\right)} \right)_{
   \left(\alpha_1, \alpha_2, \ldots, \alpha_\ell\right) \in \Comp}
= \left( M_\alpha \right)_{\alpha \in \Comp}$
is a basis of the $\kk$-module $\QSym$.}:
\begin{itemize}
\item
We define a $\kk$-linear map $\Delta : \QSym \to \QSym \otimes \QSym$
by requiring that
\begin{align}
\label{eq.coproduct.M}
\Delta \left( M_{\left( \alpha_1, \alpha_2, \ldots, \alpha_\ell
\right) }\right)
&= \sum_{k=0}^{\ell} M_{\left( \alpha_1, \alpha_2, \ldots,
\alpha_k \right) } \otimes M_{\left( \alpha_{k+1}, \alpha_{k+2},
\ldots, \alpha_\ell \right) } \\
& \qquad \text{ for every } \left(\alpha_1, \alpha_2,
\ldots, \alpha_\ell\right) \in \Comp . \nonumber
\end{align}
\item We define a $\kk$-linear map
$\varepsilon : \QSym \to \kk$ by requiring that
\[
\varepsilon\left(  M_{\left(
\alpha_1, \alpha_2, \ldots, \alpha_\ell \right) }\right)
= \delta_{\ell, 0}
\qquad \text{ for every } \left(\alpha_1, \alpha_2,
\ldots, \alpha_\ell\right) \in \Comp .
\]
(Here, $\delta_{u,v}$ is defined to be
$\begin{cases}
1, & \text{if }u = v \text{;}\\
0, & \text{if }u \neq v
\end{cases}$
whenever $u$ and $v$ are two objects.)
\end{itemize}

The map $\varepsilon$ can also be defined in a simpler
(equivalent) way: Namely, $\varepsilon$ sends every power series
$f \in \QSym$
to the result $f\left(0,0,0,\ldots\right)$ of substituting zeroes
for the variables $x_1, x_2, x_3, \ldots$ in $f$. The map $\Delta$
can also be described in such terms, but with greater
difficulty\footnote{See \cite[(5.1.7)]{Reiner} for the details.}.

It is well-known that these maps $\Delta$ and
$\varepsilon$ make the three diagrams
\begin{align*}
& \xymatrixcolsep{5pc}\xymatrix{
\QSym\ar[r]^-{\Delta} \ar[d]_{\Delta} & \QSym\otimes\QSym\ar[d]^{\Delta
\otimes\id} \\
\QSym\otimes\QSym\ar[r]_-{\id\otimes\Delta} & \QSym\otimes\QSym\otimes\QSym}
, \\
& \xymatrixcolsep{3pc}
\xymatrix{
\QSym\ar[dr]_{\cong} \ar[r]^-{\Delta} & \QSym\otimes\QSym\ar[d]^-{\varepsilon
\otimes\id} \\
& \bk\otimes\QSym}
,
\qquad
\xymatrixcolsep{3pc}
\xymatrix{
\QSym\ar[dr]_{\cong} \ar[r]^-{\Delta} & \QSym\otimes\QSym\ar[d]^-{\id
\otimes\varepsilon} \\
& \QSym\otimes\bk}
\end{align*}
(where the $\cong$ arrows are the canonical isomorphisms)
commutative, and so $\left(\QSym, \Delta, \varepsilon\right)$ is what
is commonly called a \textit{$\kk$-coalgebra}. Furthermore, $\Delta$
and $\varepsilon$ are $\kk$-algebra homomorphisms, which is what makes
this $\kk$-coalgebra $\QSym$ into a \textit{$\kk$-bialgebra}. Finally,
let $m : \QSym \otimes \QSym \to \QSym$ be the $\kk$-linear map sending
every pure tensor $a \otimes b$ to $ab$, and let $u : \kk \to \QSym$ be
the $\kk$-linear map sending $1 \in \kk$ to $1 \in \QSym$. Then, there
exists a unique $\kk$-linear map $S : \QSym \to \QSym$ making the
diagram
\begin{equation}
\xymatrix{
& \QSym \otimes \QSym \ar[rr]^{S \otimes \id} & & \QSym \otimes \QSym \ar[dr]^{m} & \\
\QSym \ar[ur]^{\Delta} \ar[rr]^{\varepsilon} \ar[dr]_{\Delta} & & \kk \ar[rr]^{u} & & \QSym \\
& \QSym \otimes \QSym \ar[rr]^{\id \otimes S} & & \QSym \otimes \QSym \ar[ur]^{m} & 
}
\label{eq.antipode}
\end{equation}
commutative. This map $S$ is known as the \textit{antipode} of $\QSym$.
It is known to be an involution and an algebra automorphism of $\QSym$,
and its action on the various quasisymmetric functions defined
combinatorially is the main topic of this note.
The existence of the antipode $S$ makes $\QSym$ into a
\textit{Hopf algebra}.
\begin{verlong}
\footnote{See the Appendix
(specifically, Proposition~\ref{prop.QSym.hopfalg})
for a proof of the fact that $\QSym$ is a Hopf algebra.}
\end{verlong}

\section{Double posets}
\label{sect.double-posets}

Next, we shall introduce the notion of a double poset, following
Malvenuto and Reutenauer \cite{Mal-Reu-DP}.

\begin{definition}
\label{def.double-poset}
\begin{itemize}

\item[(a)] We shall encode posets as pairs $\left(E, <\right)$,
where $E$ is a set and $<$ is a strict partial order
(i.e., an irreflexive, transitive and antisymmetric
binary relation) on the set $E$; this relation $<$ will be regarded
as the smaller relation of the poset. All binary relations will be
written in infix notation: i.e., we write ``$a < b$'' for ``$a$ is
related to $b$ by the relation $<$''. (If you define binary relations
as sets of pairs, then ``$a$ is related to $b$ by the relation $<$''
means that $\left(a,b\right)$ is an element of the set $<$.)

\item[(b)] If $<$ is a strict partial order on a set $E$,
and if $a$ and $b$ are two elements of $E$, then we say that
$a$ and $b$ are \textit{$<$-comparable} if we have either $a < b$
or $a = b$ or $b < a$. A strict partial order $<$ on a
set $E$ is said to be a \textit{total order} if and only
if every two elements of $E$ are $<$-comparable.

\item[(c)] If $<$ is a strict partial order on a set $E$,
and if $a$ and $b$ are two elements of $E$, then we say that
$a$ is \textit{$<$-covered by $b$} if we have $a < b$ and there
exists no $c \in E$ satisfying $a < c < b$. (For instance, if $<$
is the standard smaller relation on $\ZZ$, then each
$i \in \ZZ$ is $<$-covered by $i+1$.)

\item[(d)] A \textit{double poset} is defined as a triple
$\left(E, <_1, <_2\right)$ where $E$ is a finite set and $<_1$ and
$<_2$ are two strict partial orders on $E$.

\item[(e)] A double poset
$\left(E, <_1, <_2\right)$ is said to be \textit{special} if
the relation $<_2$ is a total order.

\item[(f)] A double poset
$\left(E, <_1, <_2\right)$ is said to be \textit{semispecial} if
every two $<_1$-comparable elements of $E$ are $<_2$-comparable.

\item[(g)] A double poset
$\left(E, <_1, <_2\right)$ is said to be \textit{tertispecial} if
it satisfies the following condition: If $a$ and $b$ are two
elements of $E$ such that $a$ is $<_1$-covered by $b$, then $a$
and $b$ are $<_2$-comparable.

\item[(h)] If $<$ is a binary relation on a set $E$, then the
\textit{opposite relation} of $<$ is defined to be the binary
relation $>$ on the set $E$ that is defined as follows: For any
$e \in E$ and $f \in E$, we have $e > f$ if and only if $f < e$.
Notice that if $<$ is a strict partial order, then so
is the opposite relation $>$ of $<$.
\end{itemize}
\end{definition}

Clearly, every special double poset is semispecial, and every
semispecial double poset is tertispecial.\footnote{The notions of a
double poset and of a special double poset come from
\cite{Mal-Reu-DP}. See \cite{Foissy13} for further results on
special double posets.
The notion of a ``tertispecial double poset''
(Dog Latin for ``slightly less special than semispecial''; in
hindsight, ``locally special'' would have been better terminology)
appears to be new and arguably sounds artificial, but is the
most suitable setting for some of the results below (see, e.g.,
Remark~\ref{rmk.antipode.Gamma.converse} below); moreover,
it appears in nature, beyond the particular case of special
double posets (see Example~\ref{exam.dp}).
We shall not use semispecial
double posets in the following; they were only introduced as a
middle-ground notion between special and tertispecial double posets
having a less daunting definition.}

\begin{definition}
\label{def.E-partition}
If $\EE = \left(E, <_1, <_2\right)$ is a double poset, then
an \textit{$\EE$-partition} shall mean a map
$\phi : E \to \left\{ 1,2,3,\ldots\right\}$ such that:
\begin{itemize}
\item every $e \in E$ and $f \in E$ satisfying $e <_1 f$ satisfy
$\phi\left(e\right) \leq \phi\left(f\right)$;
\item every $e \in E$ and $f \in E$ satisfying $e <_1 f$ and
$f <_2 e$ satisfy $\phi\left(e\right) < \phi\left(f\right)$.
\end{itemize}
\end{definition}

\begin{example}
\label{exam.dp}
The notion of an $\EE$-partition (which was inspired by the earlier
notions of $P$-partitions and $\left(P,\omega\right)$-partitions
as studied by Gessel and Stanley\footnotemark)
generalizes various well-known
combinatorial concepts. For example:
\begin{itemize}
\item If $<_2$ is the same order
as $<_1$ (or any extension of this order), then the
$\EE$-partitions are the weakly increasing maps from the poset
$\left(E, <_1\right)$ to the totally ordered set
$\left\{1, 2, 3, \ldots\right\}$.
\item If $<_2$ is the opposite relation of
$<_1$ (or any extension of this opposite relation), then the
$\EE$-partitions are the strictly increasing maps from the
poset $\left(E, <_1\right)$ to the totally ordered set
$\left\{1, 2, 3, \ldots\right\}$.
\end{itemize}
\begin{verlong}
(See the Appendix
(specifically, Proposition~\ref{prop.example.weaklinc}
and Proposition~\ref{prop.example.strictinc})
for the proofs of these two facts.)
\end{verlong}

For a more interesting example,
let $\mu = \left(\mu_1, \mu_2, \mu_3, \ldots\right)$ and
$\lambda = \left(\lambda_1, \lambda_2, \lambda_3, \ldots\right)$ be
two partitions such that $\mu \subseteq \lambda$.
(See \cite[\S 2]{Reiner} for the notations we are using
here.)
The skew Young
diagram $Y\left(\lambda / \mu\right)$ is then defined as the set of all
$\left(i, j\right) \in \left\{ 1, 2, 3, \ldots \right\}^2$ satisfying
$\mu_i < j \leq \lambda_i$. On this set $Y\left(\lambda / \mu\right)$,
we define two strict partial orders $<_1$ and $<_2$ by
\[
\left(i,j\right) <_1 \left(i',j'\right) \Longleftrightarrow
\left( i \leq i' \text{ and } j \leq j' \text{ and }
\left(i,j\right) \neq \left(i',j'\right) \right)
\]
and
\[
\left(i,j\right) <_2 \left(i',j'\right) \Longleftrightarrow
\left( i \geq i' \text{ and } j \leq j' \text{ and }
\left(i,j\right) \neq \left(i',j'\right) \right) .
\]
The resulting double poset
$\mathbf{Y}\left(\lambda / \mu\right)
= \left(Y\left(\lambda / \mu\right), <_1, <_2\right)$ has the
property that the $\mathbf{Y}\left(\lambda / \mu\right)$-partitions
are precisely the semistandard tableaux of shape
$\lambda / \mu$.
\begin{vershort}
(Again, see \cite[\S 2]{Reiner} for the meaning
of these words.)
\end{vershort}
\begin{verlong}
(Again, see \cite[\S 2]{Reiner} for the meaning
of these words.
Also, see the Appendix
(specifically, Proposition~\ref{prop.example.sst} (a))
for a proof of our claim that the
$\mathbf{Y}\left(\lambda / \mu\right)$-partitions
are precisely the semistandard tableaux of shape
$\lambda / \mu$.)
\end{verlong}

This double poset $\mathbf{Y}\left(\lambda / \mu\right)$
is not special (in general), but it is tertispecial. (Indeed,
if $a$ and $b$ are two elements of $Y\left(\lambda / \mu\right)$
such that $a$ is $<_1$-covered by $b$, then $a$ is either the left
neighbor of $b$ or the top neighbor of $b$, and thus we have
either $a <_2 b$ (in the former case) or $b <_2 a$ (in the latter
case).) Some authors prefer to use a special double poset instead,
which is defined as follows: We define a total
order $<_h$ on $Y\left(\lambda / \mu\right)$ by
\[
\left(i,j\right) <_h \left(i',j'\right) \Longleftrightarrow
\left( i > i' \text{ or } \left( i = i' \text{ and }
j < j' \right) \right) .
\]
Then, $\mathbf{Y}_h\left(\lambda / \mu\right)
= \left(Y\left(\lambda / \mu\right), <_1, <_h\right)$ is a special
double poset, and the
$\mathbf{Y}_h\left(\lambda / \mu\right)$-partitions
are precisely the semistandard tableaux of shape
$\lambda / \mu$.
\begin{verlong}
(See the Appendix
(specifically, Proposition~\ref{prop.example.sst} (b))
for a proof of the latter claim.)
\end{verlong}
\end{example}
\footnotetext{See \cite{Gessel-Ppar}
for the history of these notions, and see \cite{Gessel},
\cite{Stanley-Thes}, \cite[\S 3.15]{Stanley-EC1} and
\cite[\S 7.19]{Stanley-EC2} for
some of their theory. Mind that these sources use different and
sometimes incompatible notations -- e.g., the $P$-partitions of
\cite[\S 3.15]{Stanley-EC1} and \cite{Gessel-Ppar} differ from
those of \cite{Gessel} by a sign reversal.}

We now assign a certain formal power series to every double poset:

\begin{definition}
\label{def.Gammaw}
If $\EE = \left(E, <_1, <_2\right)$ is a double poset, and
$w : E \to \left\{1, 2, 3, \ldots\right\}$ is a map, then we define
a power series $\Gamma\left(\EE , w\right) \in \Powser$ by
\[
\Gamma\left(\EE , w\right)
= \sum_{\pi\text{ is an }\EE\text{-partition}}
  \xx_{\pi, w} ,
\qquad
\text{where } \xx_{\pi, w}
= \prod_{e \in E} x_{\pi\left(e\right)}^{w\left(e\right)} .
\]
\begin{verlong}
(See the Appendix
(specifically, Proposition~\ref{prop.Gammaw.welldef})
for a proof that this sum is well-defined.)
\end{verlong}
\end{definition}

The following fact is easy to see (but will be reproven below):

\begin{proposition}
\label{prop.Gammaw.qsym}
Let $\EE = \left(E, <_1, <_2\right)$ be a double poset, and
$w : E \to \left\{1, 2, 3, \ldots\right\}$ be a map. Then,
$\Gamma\left(\EE , w\right) \in \QSym$.
\end{proposition}

\begin{example}
\label{exam.Gamma}
The power series $\Gamma\left(\EE , w\right)$ generalize various
well-known quasisymmetric functions.

\begin{enumerate}
\item[(a)] If $\EE = \left(E, <_1, <_2\right)$ is a double poset, and
$w : E \to \left\{1, 2, 3, \ldots\right\}$ is the constant
function sending everything to $1$, then
$\Gamma\left(\EE , w\right)
= \sum_{\pi\text{ is an }\EE\text{-partition}} \xx_{\pi}$,
where $\xx_{\pi} = \prod_{e \in E} x_{\pi\left(e\right)}$.
We shall denote this power series $\Gamma\left(\EE , w\right)$
by $\Gamma\left(\EE\right)$; it is exactly what has been called
$\Gamma\left(\EE\right)$ in \cite[\S 2.2]{Mal-Reu-DP}. All results
proven below for $\Gamma\left(\EE , w\right)$ can be applied to
$\Gamma\left(\EE\right)$, yielding simpler (but less general)
statements.

\item[(b)] If $E = \left\{1, 2, \ldots, \ell\right\}$ for some
$\ell \in \NN$, if $<_1$ is the usual total order inherited from
$\ZZ$, and if $<_2$ is the opposite relation of $<_1$, then the
special double poset $\EE = \left(E, <_1, <_2\right)$ satisfies
$\Gamma\left(\EE, w\right) = M_\alpha$, where $\alpha$ is the
composition $\left(w\left(1\right), w\left(2\right), \ldots,
w\left(\ell\right)\right)$.
\begin{verlong}
(See the Appendix
(specifically, Proposition~\ref{prop.example.Gamma.b})
for a proof of this.)
\end{verlong}

Note that every $M_\alpha$ can be obtained
this way (by choosing $\ell$ and $w$ appropriately).
Thus, the elements of the monomial
basis $\left(M_\alpha\right)_{\alpha \in \Comp}$ are special
cases of the functions $\Gamma\left(\EE, w\right)$.
This shows that
the $\Gamma\left(\EE, w\right)$ for varying $\EE$ and $w$
span the $\kk$-module $\QSym$.

\item[(c)] Let
$\alpha = \left(\alpha_1, \alpha_2, \ldots, \alpha_\ell\right)$
be a composition of a nonnegative integer $n$. Let
$D\left(\alpha\right)$ be the set
$\left\{\alpha_1, \alpha_1 + \alpha_2, \alpha_1 + \alpha_2
+ \alpha_3, \ldots, \alpha_1 + \alpha_2 + \cdots + \alpha_{\ell-1}
\right\}$.
Let $E$ be the set $\left\{1, 2, \ldots, n\right\}$, and let
$<_1$ be the total order inherited on $E$ from $\ZZ$. Let $<_2$ be some
partial order on $E$ with the property that
\[
i+1 <_2 i \qquad \text{ for every } i \in D\left(\alpha\right)
\]
and
\[
i <_2 i+1 \qquad \text{ for every }
i \in \left\{1, 2, \ldots, n-1\right\}
\setminus D\left(\alpha\right) .
\]
(There are several choices for such an order; in particular, we
can find one which is a total order.%
\begin{verlong}
\ Indeed, this is proven in the Appendix
(specifically, Proposition~\ref{prop.example.Gamma.c1}).%
\end{verlong}
) Then,
\begin{verlong}
a simple argument (explained in detail in the Appendix, in the
proof of Proposition~\ref{prop.example.Gamma.c3}) shows that
\end{verlong}
\begin{align*}
\Gamma\left(\left(E, <_1, <_2\right)\right)
&= \sum_{\substack{i_1 \leq i_2 \leq \cdots \leq i_n; \\
                  i_j < i_{j+1} \text{ whenever } j \in D\left(\alpha\right)}}
  x_{i_1} x_{i_2} \cdots x_{i_n} \\
&= \sum_{\beta\text{ is a composition of }n;
        \ D\left(\beta\right) \supseteq D\left(\alpha\right)}
  M_\beta .
\end{align*}
This power series is known as the $\alpha$-th
\textit{fundamental quasisymmetric function}, usually called
$F_\alpha$ (in \cite{Gessel}, \cite[\S 2]{Mal-Reu-dua},
\cite[\S 2.4]{BBSSZ} and \cite[\S 2]{Gri-dimm})
or $L_\alpha$ (in \cite[\S 7.19]{Stanley-EC2} or
\cite[Definition 5.2.4]{Reiner}).

\item[(d)] Let $\EE$ be one of the two double posets
$\mathbf{Y}\left(\lambda / \mu\right)$ and
$\mathbf{Y}_h\left(\lambda / \mu\right)$
defined as in Example \ref{exam.dp} for two partitions $\mu$
and $\lambda$. Then, $\Gamma\left(\EE\right)$ is the skew
Schur function $s_{\lambda / \mu}$.

\item[(e)] Similarly, \textit{dual immaculate functions} as defined in
\cite[\S 3.7]{BBSSZ} can be realized as $\Gamma\left(\EE\right)$
for conveniently chosen $\EE$ (see \cite[Proposition 4.4]{Gri-dimm}), which
helped the author to prove one of their properties \cite{Gri-dimm}.
(The $\EE$-partitions here are the so-called
\textit{immaculate tableaux}.)

\item[(f)] When the relation $<_2$ of a double poset
$\EE = \left(E, <_1, <_2\right)$ is a total order (i.e.,
when the double poset $\EE$ is special), the
$\EE$-partitions are precisely the
reverse $\left(P, \omega\right)$-partitions (for
$P = \left(E, <_1\right)$ and $\omega$ being the unique
bijection $E \to \left\{1,2,\ldots,\left|E\right|\right\}$
satisfying
$\omega^{-1}\left(1\right) <_2 \omega^{-1}\left(2\right) <_2 \cdots
<_2 \omega^{-1}\left(\left|E\right|\right)$)
in the terminology
of \cite[\S 7.19]{Stanley-EC2}, and the power series
$\Gamma\left(\EE\right)$ is the $K_{P, \omega}$ of
\cite[\S 7.19]{Stanley-EC2}.
This can also be rephrased using
the notations of \cite[\S 5.2]{Reiner}: When the relation $<_2$ of a
double poset $\EE = \left(E, <_1, <_2\right)$ is a total order, we can
relabel the elements of $E$ by the integers $1, 2, \ldots, n$
(where $n = \left|E\right|$) in such
a way that $1 <_2 2 <_2 \cdots <_2 n$; then, the $\EE$-partitions are
the $P$-partitions in the terminology of \cite[Definition 5.2.1]{Reiner},
where $P$ is the labelled poset $\left(E, <_1\right)$; and furthermore,
our $\Gamma\left(\EE\right)$ is the $F_P\left(\xx\right)$ of
\cite[Definition 5.2.1]{Reiner}. Conversely, if $P$ is a labelled poset, then
the $F_P\left(\xx\right)$ of \cite[Definition 5.2.1]{Reiner} is our
$\Gamma\left(\left(P, <_P, <_{\ZZ}\right)\right)$.

\end{enumerate}

\end{example}

\section{The antipode theorem}
\label{sect.antipode}

We now come to the main results of this note. We first state a
theorem and a corollary which are not new, but will be reproven in
a more self-contained way which allows them to take their
(well-deserved) place as fundamental results rather than
afterthoughts in the theory of $\QSym$.

\begin{definition}
We let $S$ denote the antipode of $\QSym$.
\end{definition}

\begin{theorem}
\label{thm.antipode.Gammaw}
Let $\left(E, <_1, <_2\right)$ be a tertispecial double poset.
Let $w : E \to \left\{1, 2, 3, \ldots\right\}$. Then,
$S\left(\Gamma\left(\left(E, <_1, <_2\right), w\right)\right)
= \left(-1\right)^{\left|E\right|}
\Gamma\left(\left(E, >_1, <_2\right), w\right)$,
where $>_1$ denotes the opposite relation of $<_1$.
\end{theorem}

\begin{corollary}
\label{cor.antipode.Gamma}
Let $\left(E, <_1, <_2\right)$ be a tertispecial double poset.
Then, $S\left(\Gamma\left(\left(E, <_1, <_2\right)\right)\right)
= \left(-1\right)^{\left|E\right|}
\Gamma\left(\left(E, >_1, <_2\right)\right)$,
where $>_1$ denotes the opposite relation of $<_1$.
\end{corollary}

We shall give examples for consequences of these facts shortly
(Example~\ref{exam.antipode.Gammaw}), but
let us first explain where they have already appeared.
Corollary~\ref{cor.antipode.Gamma} is equivalent to
\cite[Corollary 5.2.20]{Reiner}\footnote{It is easiest to derive
\cite[Corollary 5.2.20]{Reiner} from our Corollary~\ref{cor.antipode.Gamma},
as this only requires setting $\EE = \left(P, <_P, <_{\ZZ}\right)$
(this is a special double poset, thus in particular a tertispecial
one) and noticing that
$\Gamma\left(\left(P, <_P, <_{\ZZ}\right)\right)
= F_P\left(\xx\right)$ and
$\Gamma\left(\left(P, >_P, <_{\ZZ}\right)\right)
= F_{P^{\operatorname{opp}}}\left(\xx\right)$, where all unexplained
notations are defined in \cite[Chapter 5]{Reiner}. But one can also
proceed in the opposite direction (hint: replace the partial order
$<_2$ by a linear extension, thus turning the tertispecial double
poset $\left(E, <_1, <_2\right)$ into a special one; argue that
this does not change $\Gamma\left(\left(E, <_1, <_2\right)\right)$
and $\Gamma\left(\left(E, >_1, <_2\right)\right)$).}
(a result found by Malvenuto and Reutenauer
\cite[Lemma 3.2]{Mal-Reu}).
Theorem~\ref{thm.antipode.Gammaw} is equivalent to Malvenuto's
and Reutenauer's \cite[Theorem 3.1]{Mal-Reu}\footnote{This equivalence
requires some work to set up. First of all, Malvenuto and
Reutenauer, in \cite{Mal-Reu}, do not work with the antipode $S$
of $\QSym$, but instead study a certain automorphism of
$\QSym$ called $\omega$. However, this automorphism is closely
related to $S$ (namely, for each $n \in \NN$ and each homogeneous
element $f \in \QSym$ of degree $n$, we have
$\omega\left(f\right) = \left(-1\right)^n S\left(f\right)$);
therefore, any statements about $\omega$ can be translated into
statements about $S$ and vice versa.
\par
Let me sketch how to derive \cite[Theorem 3.1]{Mal-Reu}
from our Theorem~\ref{thm.antipode.Gammaw}. Indeed, contract
all undirected edges in $G$ and $G'$,
denoting the (common) vertex set of the new graphs by $E$.
Then, define two strict partial orders $<_1$ and
$<_2$ on $E$ by
\[
\left(a <_1 b\right) \Longleftrightarrow \left(a \neq b,
\text{ and there exists a path from } a \text{ to } b \text{ in }
G \right)
\]
and
\[
\left(a <_2 b\right) \Longleftrightarrow \left(a \neq b,
\text{ and there exists a path from } a \text{ to } b \text{ in }
G' \right) .
\]
The map $w$ sends every $e \in E$ to the number of vertices
of $G$ that became $e$ when the edges were contracted. To show that
the resulting double poset $\left(E, <_1, <_2\right)$ is
tertispecial, we must notice that if $a$ is $<_1$-covered by $b$,
then $G$ had an edge from one of the vertices that became $a$ to
one of the vertices that became $b$. The ``$x_i$'s in $X$
satisfying a set of conditions'' (in the language of
\cite[Section 3]{Mal-Reu}) are in 1-to-1 correspondence with
$\left(E, <_1, <_2\right)$-partitions (at least when
$X = \left\{1, 2, 3, \ldots\right\}$); this is not immediately
obvious but not hard to check either (the acyclicity of $G$ and
$G^\prime$ is used in the proof). As a result,
\cite[Theorem 3.1]{Mal-Reu} follows from
Theorem~\ref{thm.antipode.Gammaw} above.
With some harder work, one can conversely derive
our Theorem~\ref{thm.antipode.Gammaw} from
\cite[Theorem 3.1]{Mal-Reu}.}. We
nevertheless believe that our versions of these facts are
slicker and simpler than the ones appearing in existing
literature\footnote{That said, we would not be surprised if
Malvenuto and Reutenauer are aware of them; after all, they have
discovered both the original
version of Theorem~\ref{thm.antipode.Gammaw} in
\cite{Mal-Reu} and the notion of double posets in \cite{Mal-Reu-DP}.},
and if not,
then at least our proofs below are more natural.

To these known results, we add another, which seems to be unknown so
far (probably because it is far harder to state in the terminologies
of $\left(P, \omega\right)$-partitions or
equality-and-inequality conditions appearing in literature). First,
we need to introduce some notation:

\begin{definition}
\label{def.G-sets.terminology}
Let $G$ be a group, and let $E$ be a $G$-set.

\begin{itemize}

\item[(a)] Let $<$ be a
strict partial order on $E$. We say that $G$
\textit{preserves the relation $<$} if the following holds:
For every $g \in G$, $a \in E$ and $b \in E$ satisfying $a < b$,
we have $ga < gb$.

\item[(b)] Let $w : E \to \left\{1, 2, 3, \ldots\right\}$. We
say that $G$ \textit{preserves $w$} if every $g \in G$ and
$e \in E$ satisfy $w\left(ge\right) = w\left(e\right)$.

\item[(c)] Let $g \in G$. Assume that the set $E$ is finite.
We say that $g$ is \textit{$E$-even}
if the action of $g$ on $E$ (that is, the permutation of $E$
that sends every $e \in E$ to $ge$) is an even permutation
of $E$.

\item[(d)] If $X$ is any set, then the set $X^E$ of all maps
$E \to X$ becomes a $G$-set in the following way: For any
$\pi \in X^E$ and $g \in G$, we define the element $g\pi \in X^E$
to be the map sending each $e \in E$ to $\pi\left(g^{-1}e\right)$.

\item[(e)] Let $F$ be a further $G$-set. Assume that the set
$E$ is finite. An element $\pi \in F$
is said to be \textit{$E$-coeven} if every $g \in G$
satisfying $g\pi = \pi$ is $E$-even. A $G$-orbit $O$ on $F$ is said
to be \textit{$E$-coeven} if all elements of $O$ are $E$-coeven.

\end{itemize}
\end{definition}

Before we come to the promised result, let us state two simple facts:

\begin{lemma}
\label{lem.coeven.all-one}
Let $G$ be a group. Let $F$ and $E$ be $G$-sets such that $E$ is
finite. Let $O$ be a
$G$-orbit on $F$. Then, $O$ is $E$-coeven if and only if at least
one element of $O$ is $E$-coeven.
\end{lemma}

\begin{proposition}
\label{prop.xxOw}
Let $\EE = \left(E, <_1, <_2\right)$ be a double poset.
Let $\Par \EE$ denote the set of all $\EE$-partitions.
Let $G$ be a finite
group which acts on $E$. Assume that $G$ preserves both
relations $<_1$ and $<_2$.

\begin{enumerate}

\item[(a)]
Then, $\Par \EE$ is a $G$-subset of the $G$-set
$\left\{1, 2, 3, \ldots\right\}^E$ (see
Definition~\ref{def.G-sets.terminology} (d) for the definition of
the latter).

\item[(b)] Let $w : E \to \left\{1, 2, 3, \ldots\right\}$. Assume that
$G$ preserves $w$.
Let $O$ be a $G$-orbit on $\Par \EE$. Then, the values of $\xx_{\pi, w}$
for all $\pi \in O$ are equal.

\end{enumerate}

\end{proposition}

\begin{theorem}
\label{thm.antipode.GammawG}
Let $\EE = \left(E, <_1, <_2\right)$ be a tertispecial double poset.
Let $\Par \EE$ denote the set of all $\EE$-partitions.
Let $w : E \to \left\{1, 2, 3, \ldots\right\}$. Let $G$ be a finite
group which acts on $E$. Assume that $G$ preserves both
relations $<_1$ and $<_2$, and also preserves $w$.
Then, $G$ acts also on the set $\Par \EE$ of all
$\EE$-partitions; namely, $\Par \EE$ is a $G$-subset of the $G$-set
$\left\{1, 2, 3, \ldots\right\}^E$ (according to
Proposition~\ref{prop.xxOw} (a)).
For any $G$-orbit $O$ on $\Par \EE$, we define a monomial $\xx_{O, w}$
by
\[ \xx_{O, w} = \xx_{\pi, w} \qquad \text{ for some element }
\pi \text{ of } O .
\]
(This is well-defined, since Proposition~\ref{prop.xxOw} (b) shows
that $\xx_{\pi, w}$ does not depend on the choice of $\pi \in O$.)

Let
\[
\Gamma\left(\EE, w, G\right) = \sum_{O\text{ is a }
G\text{-orbit on } \Par \EE} \xx_{O, w}
\]
and
\[
\Gamma^+\left(\EE, w, G\right) = \sum_{O\text{ is an }E\text{-coeven }
G\text{-orbit on } \Par \EE} \xx_{O, w} .
\]
Then, $\Gamma\left(\EE, w, G\right)$ and
$\Gamma^+\left(\EE, w, G\right)$ belong to $\QSym$ and satisfy
\[
S\left(\Gamma\left(\EE, w, G\right)\right)
= \left(-1\right)^{\left|E\right|}
\Gamma^+\left(\left(E, >_1, <_2\right), w, G\right) .
\]
Here, $>_1$ denotes the opposite relation of $<_1$.
\end{theorem}

This theorem, which combines Theorem~\ref{thm.antipode.Gammaw} with the
ideas of P\'olya enumeration, is inspired by Jochemko's reciprocity
result for order polynomials \cite[Theorem 2.8]{Joch}, which can be
obtained from it by specializations (see Section~\ref{sect.jochemko}
for the details of how Jochemko's result follows from ours).

We shall now briefly review a number of particular cases of
Theorem~\ref{thm.antipode.Gammaw}.

\begin{example}
\label{exam.antipode.Gammaw}

\begin{enumerate}

\item[(a)] Corollary~\ref{cor.antipode.Gamma}
follows from Theorem~\ref{thm.antipode.Gammaw} by letting $w$
be the function which is constantly $1$.

\item[(b)] Let
$\alpha = \left(\alpha_1, \alpha_2, \ldots, \alpha_\ell\right)$
be a composition of a nonnegative integer $n$, and let
$\EE = \left(E, <_1, <_2\right)$ be the double poset defined
in Example~\ref{exam.Gamma} (b). Let
$w : \left\{1, 2, \ldots, \ell\right\} \to \left\{ 1, 2, 3,
\ldots \right\}$ be the map sending every $i$ to $\alpha_i$.
As Example~\ref{exam.Gamma} (b) shows, we have
$\Gamma\left(\EE, w\right) = M_\alpha$.
Thus, applying Theorem~\ref{thm.antipode.Gammaw} to these $\EE$
and $w$
\begin{verlong}
and performing some manipulations
(see Proposition~\ref{prop.exam.antipode.Gammaw.b} in the
Appendix for the details)
\end{verlong}
yields
\begin{align*}
S\left(M_\alpha\right)
&= \left(-1\right)^\ell
   \Gamma\left(\left(E, >_1, <_2\right), w\right)
= \left(-1\right)^\ell
  \sum_{i_1 \geq i_2 \geq \cdots \geq i_\ell}
  x_{i_1}^{\alpha_1} x_{i_2}^{\alpha_2} \cdots
      x_{i_\ell}^{\alpha_\ell} \\
&= \left(-1\right)^\ell
  \sum_{i_1 \leq i_2 \leq \cdots \leq i_\ell}
  x_{i_1}^{\alpha_\ell} x_{i_2}^{\alpha_{\ell-1}} \cdots
      x_{i_\ell}^{\alpha_1}
= \left(-1\right)^\ell
  \sum_{\substack{\gamma \text{ is a composition of } n ; \\
        D\left(\gamma\right) \subseteq
        D\left(\left(\alpha_\ell, \alpha_{\ell-1},
                     \ldots, \alpha_1\right)\right)}}
  M_\gamma .
\end{align*}
This is the formula for $S\left(M_\alpha\right)$
given in \cite[Proposition 3.4]{Ehrenb96},
in \cite[(4.26)]{Malve-Thesis}, in
\cite[Theorem 5.1.11]{Reiner}, and in
\cite[Theorem 4.1]{BenSag} (originally due to Ehrenborg
and to Malvenuto and Reutenauer).

\item[(c)] Applying Corollary~\ref{cor.antipode.Gamma} to the
double poset of Example~\ref{exam.Gamma} (c) (where the relation
$<_2$ is chosen to be a total order) yields the formula for the
antipode of a fundamental quasisymmetric function
(\cite[(4.27)]{Malve-Thesis}, \cite[(5.2.7)]{Reiner}, \cite[Theorem 5.1]{BenSag}).

\item[(d)] Let us use the notations of Example~\ref{exam.dp}.
For any partition $\lambda$, let $\lambda^t$ denote the
conjugate partition of $\lambda$.
Let $\mu$ and $\lambda$ be two partitions satisfying
$\mu \subseteq \lambda$. Let $>_1$ and $>_2$ be the opposite
relations of $<_1$ and $<_2$.
Then, there is a bijection
$\tau : Y\left(\lambda / \mu\right) \to
Y\left(\lambda^t / \mu^t\right)$ sending each
$\left(i, j\right) \in Y\left(\lambda / \mu\right)$
to $\left(j, i\right)$. This bijection is an isomorphism
of double posets from
$\left(Y\left(\lambda / \mu\right), >_1, <_2\right)$
to
$\left(Y\left(\lambda^t / \mu^t\right), >_1, >_2\right)$
(where the notion of an ``isomorphism of double posets'' is defined
in the natural way -- i.e., an isomorphism of double posets is a
bijection $\phi$ between their ground sets such that each of the
two maps $\phi$ and $\phi^{-1}$ preserves each of the two orders).
Hence,
\begin{align}
\Gamma\left(\left(Y\left(\lambda / \mu\right), >_1, <_2\right)\right)
&=
\Gamma\left(\left(Y\left(\lambda^t / \mu^t\right), >_1, >_2\right)\right) .
\label{eq.exam.antipode.Gammaw.schur.0}
\end{align}
But applying Corollary~\ref{cor.antipode.Gamma} to the
tertispecial double poset $\mathbf{Y}\left(\lambda / \mu\right)$,
we obtain
\begin{align}
S\left(\Gamma\left(\mathbf{Y}\left(\lambda / \mu\right)\right)\right)
&= \left(-1\right)^{\left|\lambda / \mu\right|}
\Gamma\left(\left(Y\left(\lambda / \mu\right), >_1, <_2\right)\right)
\nonumber \\
&= \left(-1\right)^{\left|\lambda / \mu\right|}
\Gamma\left(\left(Y\left(\lambda^t / \mu^t\right), >_1, >_2\right)\right)
\label{eq.exam.antipode.Gammaw.schur.1}
\end{align}
(by (\ref{eq.exam.antipode.Gammaw.schur.0})).
But from Example~\ref{exam.Gamma} (d), we know that
$\Gamma\left(\mathbf{Y}\left(\lambda / \mu\right)\right)
= s_{\lambda / \mu}$. Moreover, a similar argument using
\cite[Remark 2.2.5]{Reiner} shows that
$\Gamma\left(\left(Y\left(\lambda / \mu\right), >_1, >_2\right)\right)
= s_{\lambda / \mu}$. Applying this to $\lambda^t$ and $\mu^t$
instead of $\lambda$ and $\mu$, we obtain
$\Gamma\left(\left(Y\left(\lambda^t / \mu^t\right), >_1, >_2\right)\right)
= s_{\lambda^t / \mu^t}$.
Now,
\eqref{eq.exam.antipode.Gammaw.schur.1} rewrites as
\begin{equation}
S\left(s_{\lambda / \mu}\right)
= \left(-1\right)^{\left|\lambda / \mu\right|}
s_{\lambda^t / \mu^t}
\label{eq.exam.antipode.Gammaw.schur.2}
\end{equation}
(since
$\Gamma\left(\mathbf{Y}\left(\lambda / \mu\right)\right)
= s_{\lambda / \mu}$ and
$\Gamma\left(\left(Y\left(\lambda^t / \mu^t\right), >_1, >_2\right)\right)
= s_{\lambda^t / \mu^t}$).
This is a well-known formula, and is usually stated
for $S$ being the antipode of the Hopf algebra of symmetric
(rather than quasisymmetric) functions; but this is an
equivalent statement, since the latter antipode
is a restriction of the antipode of $\QSym$.

It is also possible (but more difficult) to derive
\eqref{eq.exam.antipode.Gammaw.schur.2} by using the double
poset $\mathbf{Y}_h\left(\lambda / \mu\right)$ instead of
$\mathbf{Y}\left(\lambda / \mu\right)$. (This boils down to what
was done in \cite[proof of Corollary 5.2.22]{Reiner}.)

\item[(e)] A result of Benedetti and Sagan
\cite[Theorem 8.2]{BenSag} on the antipodes of immaculate
functions can be obtained from Corollary~\ref{cor.antipode.Gamma}
using dualization.

\end{enumerate}

\end{example}

\begin{remark}
\label{rmk.antipode.Gamma.converse}
Corollary \ref{cor.antipode.Gamma} has a
sort of converse. Namely, let us assume that $\kk = \ZZ$. If
$\left(  E,<_{1},<_{2}\right)  $ is a double poset satisfying $S\left(
\Gamma\left(  \left(  E,<_{1},<_{2}\right)  \right)  \right)  =\left(
-1\right)  ^{\left\vert E\right\vert }\Gamma\left(  \left(  E,>_{1}%
,<_{2}\right)  \right)  $, then $\left(  E,<_{1},<_{2}\right)  $ is tertispecial.

More precisely, the following holds: Define the \textit{length} $\ell\left(
\alpha\right)  $ of a composition $\alpha$ to be the number of entries of
$\alpha$. Define the \textit{size} $\left\vert \alpha\right\vert $ of a
composition $\alpha$ to be the sum of the entries of $\alpha$. Let
$\eta : \QSym \to \QSym$ be the $\kk$-linear map defined by
\[
\eta \left( M_{\alpha} \right)
=
\begin{cases}
M_{\alpha}, & \text{if }\ell\left(  \alpha\right)  \geq\left\vert
\alpha\right\vert -1;\\
0, & \text{if }\ell\left(  \alpha\right)  <\left\vert \alpha\right\vert -1
\end{cases}
\qquad \text{for every } \alpha \in \Comp .
\]
Thus, $\eta$ transforms a quasisymmetric function by removing all monomials
$\mathfrak{m}$ for which the number of indeterminates appearing in
$\mathfrak{m}$ is $<\deg\mathfrak{m}-1$. We partially order the ring
$\kk \left[  \left[ x_1, x_2, x_3, \ldots \right]  \right]  $
by a coefficientwise
order (i.e., two power series $a$ and $b$ satisfy $a\leq b$ if and only if
each coefficient of $a$ is $\leq$ to the corresponding coefficient of $b$).
Now, every double poset $\left(  E,<_{1},<_{2}\right)  $ satisfies
\begin{equation}
\eta\left(  \left(  -1\right)  ^{\left\vert E\right\vert }S\left(
\Gamma\left(  \left(  E,<_{1},<_{2}\right)  \right)  \right)  \right)
\leq\eta\left(  \Gamma\left(  \left(  E,>_{1},<_{2}\right)  \right)  \right)
,\label{eq.rmk.antipode.Gamma.converse.ineq}
\end{equation}
and equality holds if and only if the double poset
$\left( E, <_{1}, <_{2} \right)  $ is tertispecial.
(If we omit $\eta$, then the inequality fails in general.)

\begin{vershort}
The proof of (\ref{eq.rmk.antipode.Gamma.converse.ineq}) is somewhat
technical, but not too hard. A rough outline is given in the detailed
version of this paper.
\end{vershort}

\begin{verlong}
The proof of (\ref{eq.rmk.antipode.Gamma.converse.ineq}) is somewhat
technical, but not too difficult. I shall only give a rough outline, as the
result is tangential to this paper. Fix a double poset $\left(  E,<_{1}%
,<_{2}\right)  $, and set $n=\left\vert E\right\vert $ and $\left[  n\right]
=\left\{  1,2,\ldots,n\right\}  $. A \textit{costrictor} will mean a bijection
$\phi:\left[  n\right]  \rightarrow E$ whose inverse $\phi^{-1}:E\rightarrow
\left[  n\right]  $ is a strictly increasing map from the poset $\left(
E,<_{1}\right)  $ to $\left(  \left[  n\right]  ,<_{\mathbb{Z}}\right)  $.
(The costrictors are in 1-to-1 correspondence with the linear extensions of
$\left(  E,<_{1}\right)  $.) For two elements $e$ and $f$ of $E$, we write
$e\parallel_{1}f$ if and only if $e$ and $f$ are not $<_{1}$-comparable.
Whenever $k \in \NN$, we shall use the notation $1^k$ for ``$k$ ones,
written in a row''; thus, for example, $\left(3,1^5,4\right)$ is the
composition $\left(3,1,1,1,1,1,4\right)$.
Then, it is not hard to see that
\begin{align*}
& \Gamma\left(  \left(  E,<_{1},<_{2}\right)  \right)  \\
& =\left(  \text{the number of all costrictors}\right)  M_{\left(
1^{n}\right)  }+\sum_{k=1}^{n-1}\gamma_{\left(  E,<_{1},<_{2}\right)
,k}M_{\left(  1^{k-1},2,1^{n-k-1}\right)  }\\
& \ \ \ \ \ \ \ \ \ \ +\left(  \text{a linear combination of }M_{\alpha}\text{
with }\left|\alpha\right| = n \text{ and } \ell\left(  \alpha\right)  <n-1\right)  ,
\end{align*}
where
\begin{align*}
& \gamma_{\left(  E,<_{1},<_{2}\right)  ,k}\\
& =\left(  \text{the number of all costrictors}\right)  \\
& \ \ \ \ \ \ \ \ \ \ -\dfrac{1}{2}\left(  \text{the number of all costrictors
}\phi\text{ satisfying }\phi\left(  k\right)  \parallel_{1}\phi\left(
k+1\right)  \right)  \\
& \ \ \ \ \ \ \ \ \ \ -\left(  \text{the number of all costrictors }\phi\text{
satisfying }\phi\left(  k\right)  <_{1}\phi\left(  k+1\right) \right. \\
& \ \ \ \ \ \ \ \ \ \ \qquad \left. \text{ and
}\phi\left(  k\right)  >_{2}\phi\left(  k+1\right)  \right)  .
\end{align*}
Hence,
\begin{align*}
& \eta\left(  \Gamma\left(  \left(  E,<_{1},<_{2}\right)  \right)  \right)
\\
& =\left(  \text{the number of all costrictors}\right)  M_{\left(
1^{n}\right)  }+\sum_{k=1}^{n-1}\gamma_{\left(  E,<_{1},<_{2}\right)
,k}M_{\left(  1^{k-1},2,1^{n-k-1}\right)  }.
\end{align*}
Using this (and the formula for $S\left(  M_{\alpha}\right)  $ in Example
\ref{exam.antipode.Gammaw} (b)), it is easy to show that
\begin{align*}
& \eta\left(  \left(  -1\right)  ^{\left\vert E\right\vert }S\left(
\Gamma\left(  \left(  E,<_{1},<_{2}\right)  \right)  \right)  \right)  \\
& =\left(  \text{the number of all costrictors}\right)  M_{\left(
1^{n}\right)  }\\
& \ \ \ \ \ \ \ \ \ \ +\sum_{k=1}^{n-1}\left(  \left(  \text{the number of all
costrictors}\right)  -\gamma_{\left(  E,<_{1},<_{2}\right)  ,k}\right)
M_{\left(  1^{n-k-1},2,1^{k-1}\right)  }.
\end{align*}

But we can also define an \textit{anticostrictor} as a bijection $\phi:\left[
n\right]  \rightarrow E$ whose inverse $\phi^{-1}:E\rightarrow\left[
n\right]  $ is a strictly decreasing map from the poset $\left(
E,<_{1}\right)  $ to $\left(  \left[  n\right]  ,<_{\mathbb{Z}}\right)  $.
Then, similarly to our formula for $\eta\left(  \Gamma\left(  \left(
E,<_{1},<_{2}\right)  \right)  \right)  $, we can derive the formula
\begin{align*}
& \eta\left(  \Gamma\left(  \left(  E,>_{1},<_{2}\right)  \right)  \right)
\\
& =\left(  \text{the number of all anticostrictors}\right)  M_{\left(
1^{n}\right)  }+\sum_{k=1}^{n-1}\widetilde{\gamma}_{\left(  E,<_{1}%
,<_{2}\right)  ,k}M_{\left(  1^{k-1},2,1^{n-k-1}\right)  },
\end{align*}
where
\begin{align*}
& \widetilde{\gamma}_{\left(  E,<_{1},<_{2}\right)  ,k}\\
& =\left(  \text{the number of all anticostrictors}\right)  \\
& \ \ \ \ \ \ \ \ \ \ -\dfrac{1}{2}\left(  \text{the number of all
anticostrictors }\phi\text{ satisfying }\phi\left(  k\right)  \parallel
_{1}\phi\left(  k+1\right)  \right)  \\
& \ \ \ \ \ \ \ \ \ \ -\left(  \text{the number of all anticostrictors }%
\phi\text{ satisfying }\phi\left(  k\right)  >_{1}\phi\left(  k+1\right)
\right. \\
& \ \ \ \ \ \ \ \ \ \ \qquad \left.
\text{ and }\phi\left(  k\right)  >_{2}\phi\left(  k+1\right)  \right)  .
\end{align*}

Recall that we want to prove (\ref{eq.rmk.antipode.Gamma.converse.ineq}). In
light of our formulas for $\eta\left(  \left(  -1\right)  ^{\left\vert
E\right\vert }S\left(  \Gamma\left(  \left(  E,<_{1},<_{2}\right)  \right)
\right)  \right)  $ and $\eta\left(  \Gamma\left(  \left(  E,>_{1}%
,<_{2}\right)  \right)  \right)  $, this boils down to proving the following
two facts:

\begin{enumerate}
\item The number of all costrictors is $\leq$ to the number of all anticostrictors.

\item For each $k\in\left\{  1,2,\ldots,n-1\right\}  $, we have
\[
\left(  \text{the number of all costrictors}\right)  -\gamma_{\left(
E,<_{1},<_{2}\right)  ,k}\leq\widetilde{\gamma}_{\left(  E,<_{1},<_{2}\right)
,n-k}.
\]

\end{enumerate}

But the first of these two facts is easy to see: Let $w_{0}$ be the involution
$\left[  n\right]  \rightarrow\left[  n\right]  ,\ i\mapsto n+1-i$. Then,
$w_{0}$ is a poset isomorphism $\left(  \left[  n\right]  ,<_{\mathbb{Z}%
}\right)  \rightarrow\left(  \left[  n\right]  ,>_{\mathbb{Z}}\right)  $.
Hence, there is a 1-to-1 correspondence between the costrictors and the
anticostrictors, given by $\phi\mapsto\phi\circ w_{0}$. Thus, the number of
all costrictors equals the number of all anticostrictors. This proves Fact 1.

Proving Fact 2 is harder. Fix $k\in\left\{  1,2,\ldots,n-1\right\}  $.
Recalling our definition of $\gamma_{\left(  E,<_{1},<_{2}\right)  ,k}$ and
$\widetilde{\gamma}_{\left(  E,<_{1},<_{2}\right)  ,n-k}$, we notice that we
must show%
\begin{align*}
& \dfrac{1}{2}\left(  \text{the number of all costrictors }\phi\text{
satisfying }\phi\left(  k\right)  \parallel_{1}\phi\left(  k+1\right)
\right)  \\
& \ \ \ \ \ \ \ \ \ \ +\left(  \text{the number of all costrictors }\phi\text{
satisfying }\phi\left(  k\right)  <_{1}\phi\left(  k+1\right)  \right. \\
& \ \ \ \ \ \ \ \ \ \ \qquad \left. \text{ and
}\phi\left(  k\right)  >_{2}\phi\left(  k+1\right)  \right)  \\
& \leq\left(  \text{the number of all anticostrictors}\right)  \\
& \ \ \ \ \ \ \ \ \ \ -\dfrac{1}{2}\left(  \text{the number of all
anticostrictors }\phi\text{ satisfying }\phi\left(  n-k\right)  \parallel
_{1}\phi\left(  n-k+1\right)  \right)  \\
& \ \ \ \ \ \ \ \ \ \ -\left(  \text{the number of all anticostrictors }%
\phi\text{ satisfying }\phi\left(  n-k\right)  >_{1}\phi\left(  n-k+1\right)
\right. \\
& \ \ \ \ \ \ \ \ \ \ \qquad \left.
\text{ and }\phi\left(  n-k\right)  >_{2}\phi\left(  n-k+1\right)  \right)  .
\end{align*}
Using the 1-to-1 correspondence between the costrictors and the
anticostrictors (which we already used in the proof of Fact 1), we can rewrite
this as%
\begin{align*}
& \dfrac{1}{2}\left(  \text{the number of all costrictors }\phi\text{
satisfying }\phi\left(  k\right)  \parallel_{1}\phi\left(  k+1\right)
\right)  \\
& \ \ \ \ \ \ \ \ \ \ +\left(  \text{the number of all costrictors }\phi\text{
satisfying }\phi\left(  k\right)  <_{1}\phi\left(  k+1\right) \right. \\
& \ \ \ \ \ \ \ \ \ \ \qquad \left. \text{ and
}\phi\left(  k\right)  >_{2}\phi\left(  k+1\right)  \right)  \\
& \leq\left(  \text{the number of all costrictors}\right)  \\
& \ \ \ \ \ \ \ \ \ \ -\dfrac{1}{2}\left(  \text{the number of all costrictors
}\phi\text{ satisfying }\phi\left(  k\right)  \parallel_{1}\phi\left(
k+1\right)  \right)  \\
& \ \ \ \ \ \ \ \ \ \ -\left(  \text{the number of all costrictors }\phi\text{
satisfying }\phi\left(  k\right)  <_{1}\phi\left(  k+1\right) \right. \\
& \ \ \ \ \ \ \ \ \ \ \qquad \left. \text{ and
}\phi\left(  k\right)  <_{2}\phi\left(  k+1\right)  \right)
\end{align*}
(here, we have used the fact that $\left(  \phi\circ w_{0}\right)  \left(
n-k\right)  =\phi\left(  \underbrace{w_{0}\left(  n-k\right)  }_{=k+1}\right)
=\phi\left(  k+1\right)  $ and $\left(  \phi\circ w_{0}\right)  \left(
n-k+1\right)  =\phi\left(  \underbrace{w_{0}\left(  n-k+1\right)  }%
_{=k}\right)  =\phi\left(  k\right)  $). This simplifies to%
\begin{align*}
& \left(  \text{the number of all costrictors }\phi\text{ satisfying }%
\phi\left(  k\right)  \parallel_{1}\phi\left(  k+1\right)  \right)  \\
& \ \ \ \ \ \ \ \ \ \ +\left(  \text{the number of all costrictors }\phi\text{
satisfying }\phi\left(  k\right)  <_{1}\phi\left(  k+1\right)  \right. \\
& \ \ \ \ \ \ \ \ \ \ \qquad \left.\text{ and
}\phi\left(  k\right)  >_{2}\phi\left(  k+1\right)  \right)  \\
& \ \ \ \ \ \ \ \ \ \ +\left(  \text{the number of all costrictors }\phi\text{
satisfying }\phi\left(  k\right)  <_{1}\phi\left(  k+1\right)  \right. \\
& \ \ \ \ \ \ \ \ \ \ \qquad \left.\text{ and
}\phi\left(  k\right)  <_{2}\phi\left(  k+1\right)  \right)  \\
& \leq\left(  \text{the number of all costrictors}\right)  .
\end{align*}
This inequality is clearly satisfied (since each costrictor $\phi$ satisfies
at most one of the relations $\phi\left(  k\right)  \parallel_{1}\phi\left(
k+1\right)  $, $\left(  \phi\left(  k\right)  <_{1}\phi\left(  k+1\right)
\text{ and }\phi\left(  k\right)  >_{2}\phi\left(  k+1\right)  \right)  $ and
$\left(  \phi\left(  k\right)  <_{1}\phi\left(  k+1\right)  \text{ and }%
\phi\left(  k\right)  <_{2}\phi\left(  k+1\right)  \right)  $). Thus, the
inequality (\ref{eq.rmk.antipode.Gamma.converse.ineq}) is proven. It now
remains to show that equality holds only when the double poset $\left(
E,<_{1},<_{2}\right)  $ is tertispecial.

Indeed, assume that $\left(  E,<_{1},<_{2}\right)  $ is not tertispecial.
Then, there exist two elements $a$ and $b$ of $E$ such that $a$ is $<_{1}%
$-covered by $b$ but $a$ and $b$ are not $<_{2}$-comparable. Consider such $a$
and $b$. There exists at least one pair $\left(  \phi,k\right)  $ of a
costrictor $\phi$ and an element $k\in\left\{  1,2,\ldots,n-1\right\}  $
satisfying $\phi\left(  k\right)  =a$ and $\phi\left(  k+1\right)  =b$. (In
fact, in order to construct such a pair, we write our set $E$ as the
disjoint union $E=E_{1}\cup\left\{  a\right\}  \cup\left\{  b\right\}  \cup
E_{2}$, where $E_{1}=\left\{  e\in E\ \mid\ e<_{1}b\text{ and }e\neq
a\right\}  $ and $E_{2}=\left\{  e\in E\ \mid\ \text{neither }e<_{1}b\text{
nor }e=b\right\}  $. Then, we set $k=\left\vert E_{1}\right\vert +1$, and
choose strictly increasing bijections $\alpha:E_{1}\rightarrow\left[
k-1\right]  $ and $\beta:E_{2}\rightarrow\left[  n-k-1\right]  $. Finally, we
define a map $\gamma:E\rightarrow\left[  n\right]  $ by $\gamma\left(
e\right)  =
\begin{cases}
\alpha\left(  e\right)  , & \text{if }e\in E_{1};\\
k, & \text{if }e=a;\\
k+1, & \text{if }e=b;\\
k+1+\beta\left(  e\right)  , & \text{if }e\in E_{2}
\end{cases}
$, and define $\phi$ to be $\gamma^{-1}$. It is not hard to check that $\phi$
is a costrictor.) This causes one of the inequalities from which we obtained
(\ref{eq.rmk.antipode.Gamma.converse.ineq}) to be strict. This completes the
(outline of the) proof.
\end{verlong}
\end{remark}

\section{Lemmas: packed $\EE$-partitions and comultiplications}
\label{sect.lemmas}

We shall now prepare for the proofs of our results. To this end,
we introduce the notion of a \textit{packed map}.

\begin{definition}
\begin{itemize}

\item[(a)]
An \textit{initial interval} will mean a set of the form
$\left\{1, 2, \ldots, \ell\right\}$ for some $\ell \in \NN$.

\item[(b)]
If $E$ is a set and $\pi : E \to \left\{1, 2, 3, \ldots\right\}$ is
a map, then $\pi$ is said to be \textit{packed} if $\pi\left(E\right)$
is an initial interval. Clearly, this initial interval must be
$\left\{1, 2, \ldots, \left|\pi\left(E\right)\right|\right\}$.
\begin{verlong}
(Indeed, this follows from
Proposition~\ref{prop.ev.comp} (a), applied to
$\ell = \left|\pi\left(E\right)\right|$.)
\end{verlong}

\end{itemize}
\end{definition}

\begin{proposition}
\label{prop.ev.comp}
Let $E$ be a finite set.
Let $\pi : E \to \left\{1, 2, 3, \ldots\right\}$ be a packed map.
Let $\ell = \left|\pi\left(E\right)\right|$.

\begin{itemize}
\item[(a)] We have $\pi\left(E\right) = \left\{1, 2, \ldots, \ell\right\}$.

\item[(b)] Let $w : E \to \left\{1, 2, 3, \ldots\right\}$ be a
map. For each $i \in \left\{1, 2, \ldots, \ell\right\}$, define an
integer $\alpha_i$ by $\alpha_i = \sum_{e \in \pi^{-1}\left(i\right)}
w\left(e\right)$. Then,
$\left(\alpha_1, \alpha_2, \ldots, \alpha_\ell\right)$ is a
composition.
\end{itemize}
\end{proposition}

\begin{vershort}
\begin{proof}[Proof of Proposition \ref{prop.ev.comp}.]
This follows from the assumption that $\pi$ be packed. (Details
are left to the reader.)
\end{proof}
\end{vershort}

\begin{verlong}
\begin{proof}
[Proof of Proposition \ref{prop.ev.comp}.] The map $\pi$ is packed. In other
words, $\pi\left(  E\right)  $ is an initial interval (by the definition of
\textquotedblleft packed\textquotedblright). In other words, $\pi\left(
E\right)  =\left\{  1,2,\ldots,k\right\}  $ for some $k\in \NN $.
Consider this $k$. From $\pi\left(  E\right)  =\left\{  1,2,\ldots,k\right\}
$, we obtain $\left\vert \pi\left(  E\right)  \right\vert =\left\vert \left\{
1,2,\ldots,k\right\}  \right\vert =k$, so that $k=\left\vert \pi\left(
E\right)  \right\vert =\ell$. Now, $\pi\left(  E\right)  =\left\{
1,2,\ldots,k\right\}  =\left\{  1,2,\ldots,\ell\right\}  $ (since $k=\ell$).
This proves Proposition \ref{prop.ev.comp} (a).

(b) Let $i\in\left\{  1,2,\ldots,\ell\right\}  $. Then, $i\in\left\{
1,2,\ldots,\ell\right\}  =\pi\left(  E\right)  $. Hence, there exists some
$f\in E$ such that $i=\pi\left(  f\right)  $. Consider this $f$. We have
$f\in\pi^{-1}\left(  i\right)  $ (since $\pi\left(  f\right)  =i$). Also,
$w\left(  f\right)  \in\left\{  1,2,3,\ldots\right\}  $ (since $w$ is a map
$E\rightarrow\left\{  1,2,3,\ldots\right\}  $). Now,
\begin{align*}
\alpha_{i}  & =\sum_{e\in\pi^{-1}\left(  i\right)  }w\left(  e\right)
=\underbrace{w\left(  f\right)  }_{>0}+\sum_{\substack{e\in\pi^{-1}\left(
i\right)  ;\\e\neq f}}\underbrace{w\left(  e\right)  }_{\substack{\geq
0\\\text{(since }w\left(  e\right)  \in\left\{  1,2,3,\ldots\right\}
\text{)}}}\\
& \ \ \ \ \ \ \ \ \ \ \left(
\begin{array}[c]{c}
\text{here, we have split off the addend for }e=f\text{ from the sum}\\
\text{(since }f\in\pi^{-1}\left(  i\right)  \text{)}%
\end{array}
\right)  \\
& >0+\sum_{\substack{e\in\pi^{-1}\left(  i\right)  ;\\e\neq f}}0=0.
\end{align*}
Thus, $\alpha_{i}$ is a positive integer.

Now, forget that we have fixed $i$. We thus have shown that $\alpha_{i}$ is a
positive integer for each $i\in\left\{  1,2,\ldots,\ell\right\}  $. In other
words, $\left(  \alpha_{1},\alpha_{2},\ldots,\alpha_{\ell}\right)  $ is a
finite list of positive integers, i.e., a composition. This proves Proposition
\ref{prop.ev.comp} (b).
\end{proof}
\end{verlong}

\begin{definition}
Let $E$ be a finite set.
Let $\pi : E \to \left\{1, 2, 3, \ldots\right\}$ be a packed map.
Let $w : E \to \left\{1, 2, 3, \ldots\right\}$ be a map. Then, the
composition $\left(\alpha_1, \alpha_2, \ldots, \alpha_\ell\right)$
defined in Proposition \ref{prop.ev.comp} (b) will be denoted by
$\ev_w \pi$.
\end{definition}

\begin{proposition}
\label{prop.Gammaw.packed}
Let $\EE = \left(E, <_1, <_2\right)$ be a double poset. Let
$w : E \to \left\{1, 2, 3, \ldots\right\}$ be a map.
Then,
\begin{equation}
\label{eq.prop.Gammaw.packed}
\Gamma\left(\EE , w\right)
= \sum_{\varphi \text{ is a packed } \EE\text{-partition}}
M_{\ev_w \varphi} .
\end{equation}
\end{proposition}

\begin{proof}[Proof of Proposition~\ref{prop.Gammaw.packed}.]
For every finite subset $T$ of
$\left\{1, 2, 3, \ldots\right\}$, there exists a unique strictly
increasing bijection $\left\{1, 2, \ldots,
\left|T\right|\right\} \to T$. We shall denote this bijection by
$r_T$.
For every map $\pi : E \to \left\{1, 2, 3, \ldots\right\}$, we
define the \textit{packing of $\pi$} as the map
$r_{\pi\left(E\right)}^{-1} \circ \pi : E \to
\left\{1, 2, 3, \ldots\right\}$; this is a packed map (indeed,
its image is
$\left\{1, 2, \ldots, \left|\pi\left(E\right)\right|\right\}$),
and will be
denoted by $\pack \pi$. This map $\pack \pi$ is an $\EE$-partition
if and only if $\pi$ is an $\EE$-partition\footnote{Indeed,
$\pack \pi = r_{\pi\left(E\right)}^{-1} \circ \pi$. Since
$r_{\pi\left(E\right)}$ is strictly increasing, we thus see that,
for any given $e \in E$ and $f \in E$, the equivalences
\[
\left(\left(\pack \pi\right)\left(e\right)
      \leq \left(\pack \pi\right)\left(f\right)\right)
\Longleftrightarrow
\left( \pi\left(e\right) \leq \pi\left(f\right) \right)
\]
and
\[
\left(\left(\pack \pi\right)\left(e\right)
      < \left(\pack \pi\right)\left(f\right)\right)
\Longleftrightarrow
\left( \pi\left(e\right) < \pi\left(f\right) \right)
\]
hold. Hence, $\pack \pi$ is an $\EE$-partition
if and only if $\pi$ is an $\EE$-partition.}.
Hence, $\pack \pi$ is a packed $\EE$-partition for every
$\EE$-partition $\pi$.

We shall show that for every packed $\EE$-partition $\varphi$, we
have
\begin{equation}
\sum_{\pi\text{ is an }\EE\text{-partition; } \pack \pi = \varphi}
\xx_{\pi, w} = M_{\ev_w \varphi} .
\label{pf.prop.Gammaw.packed.1}
\end{equation}
Once this is proven, it will follow that
\begin{align*}
\Gamma\left(\EE , w\right)
&= \sum_{\pi\text{ is an }\EE\text{-partition}}
  \xx_{\pi, w}
= \sum_{\varphi \text{ is a packed } \EE\text{-partition}}
  \underbrace{\sum_{\pi\text{ is an }\EE\text{-partition; } \pack \pi = \varphi}
              \xx_{\pi, w}}_{\substack{ = M_{\ev_w \varphi} \\
                             \text{(by \eqref{pf.prop.Gammaw.packed.1})}
                             }} \\
& \qquad \left(\text{since } \pack \pi \text{ is a packed }
             \EE\text{-partition for every }
             \EE\text{-partition } \pi\right) \\
&= \sum_{\varphi \text{ is a packed } \EE\text{-partition}}
M_{\ev_w \varphi} ,
\end{align*}
and Proposition~\ref{prop.Gammaw.packed} will be proven.

So it remains to prove \eqref{pf.prop.Gammaw.packed.1}. Let $\varphi$
be a packed $\EE$-partition.
\begin{verlong}
Hence, $\varphi$ is a packed map
$E \to \left\{ 1, 2, 3, \ldots \right\}$.
\end{verlong}
Let
$\ell = \left|\varphi\left(E\right)\right|$; thus
$\varphi\left(E\right) = \left\{1, 2, \ldots, \ell\right\}$
\begin{vershort}
(since $\varphi$ is packed).
\end{vershort}
\begin{verlong}
(by Proposition~\ref{prop.ev.comp} (a) (applied to $\varphi$
instead of $\pi$)).
\end{verlong}
Let $\alpha_i = \sum_{e \in \varphi^{-1}\left(i\right)}
w\left(e\right)$ for every $i \in \left\{ 1, 2, \ldots, \ell \right\}$;
thus,
$\ev_w \varphi = \left(\alpha_1, \alpha_2, \ldots, \alpha_\ell\right)$
(by the definition of $\ev_w \varphi$).
Hence, the definition of $M_{\ev_w \varphi}$ yields
\begin{align}
M_{\ev_w \varphi}
&= \sum_{i_1 < i_2 < \cdots < i_\ell}
\underbrace{x_{i_1}^{\alpha_1} x_{i_2}^{\alpha_2} \cdots x_{i_\ell}^{\alpha_\ell}}
_{= \prod_{k = 1}^\ell x_{i_k}^{\alpha_k}}
= \sum_{i_1 < i_2 < \cdots < i_\ell}
\prod_{k = 1}^\ell \underbrace{x_{i_k}^{\alpha_k}}
                              _{\substack{= x_{i_k}^{\sum_{e \in \varphi^{-1}\left(k\right)}
                                w\left(e\right)} \\
                                \text{(since } \alpha_k = \sum_{e \in \varphi^{-1}\left(k\right)}
                                w\left(e\right) \text{)}}}
\nonumber \\
&= \sum_{i_1 < i_2 < \cdots < i_\ell}
\prod_{k = 1}^\ell \underbrace{x_{i_k}^{\sum_{e \in \varphi^{-1}\left(k\right)}
                                w\left(e\right)}}
                              _{\substack{= \prod_{e \in \varphi^{-1}\left(k\right)}
                                x_{i_k}^{w\left(e\right)} \\
                                = \prod_{e \in E;\ \varphi\left(e\right) = k}
                                x_{i_k}^{w\left(e\right)} }}
= \sum_{i_1 < i_2 < \cdots < i_\ell}
\prod_{k = 1}^\ell \prod_{e \in E;\ \varphi\left(e\right) = k}
\underbrace{x_{i_k}^{w\left(e\right)}}_{\substack{
         =x_{i_{\varphi\left(e\right)}}^{w\left(e\right)} \\
         \text{(since } k = \varphi\left(e\right) \text{)}
         }}
\nonumber \\
&= \sum_{i_1 < i_2 < \cdots < i_\ell}
\underbrace{\prod_{k = 1}^\ell \prod_{e \in E;\ \varphi\left(e\right) = k}
x_{i_{\varphi\left(e\right)}}^{w\left(e\right)}}
_{ = \prod_{e \in E} x_{i_{\varphi\left(e\right)}}^{w\left(e\right)}}
= \sum_{i_1 < i_2 < \cdots < i_\ell}
\prod_{e \in E} x_{i_{\varphi\left(e\right)}}^{w\left(e\right)}
\nonumber \\
&= \sum_{T \subseteq \left\{1, 2, 3, \ldots\right\} ; \ \left|T\right| = \ell}
\prod_{e \in E} x_{r_T\left(\varphi\left(e\right)\right)}^{w\left(e\right)}
\nonumber
\end{align}
\footnote{In the last equality, we have used the fact that
the strictly increasing sequences
$\left(i_1 < i_2 < \cdots < i_\ell\right)$ of positive integers are
in bijection with the subsets
$T \subseteq \left\{1, 2, 3, \ldots\right\}$
such that $\left|T\right| = \ell$. The bijection sends a sequence
$\left(i_1 < i_2 < \cdots < i_\ell\right)$ to the set of its entries;
its inverse map sends every $T$ to the sequence
$\left(r_T\left(1\right), r_T\left(2\right), \ldots,
r_T\left(\left|T\right|\right)\right)$.}. Hence,
\begin{align}
M_{\ev_w \varphi}
&= \sum_{T \subseteq \left\{1, 2, 3, \ldots\right\} ; \ \left|T\right| = \ell}
\underbrace{\prod_{e \in E} x_{r_T\left(\varphi\left(e\right)\right)}^{w\left(e\right)}}
           _{\substack{
              = \prod_{e \in E} x_{\left(r_T\circ \varphi\right)\left(e\right)}^{w\left(e\right)}
              = \xx_{r_T\circ\varphi,w} \\
              \text{(by the definition of } \xx_{r_T\circ\varphi,w} \text{)}
             }}
= \sum_{T \subseteq \left\{1, 2, 3, \ldots\right\} ; \ \left|T\right| = \ell}
\xx_{r_T\circ\varphi,w} .
\label{pf.prop.Gammaw.packed.1.pf.1}
\end{align}

On the other hand, recall that $\varphi$ is an $\EE$-partition.
Hence, every map $\pi$ satisfying $\pack \pi = \varphi$
is an $\EE$-partition (because, as we know, $\pack \pi$ is an
$\EE$-partition if and only if $\pi$ is an $\EE$-partition).
Thus, the $\EE$-partitions $\pi$ satisfying
$\pack \pi = \varphi$ are precisely the maps
$\pi : E \to \left\{1, 2, 3, \ldots\right\}$ satisfying
$\pack \pi = \varphi$. Hence,
\begin{align*}
\sum_{\pi\text{ is an }\EE\text{-partition; } \pack \pi = \varphi}
\xx_{\pi, w}
&= \sum_{\pi : E \to \left\{1, 2, 3, \ldots\right\} \text{; } \pack \pi = \varphi}
\xx_{\pi, w} \\
&= \sum_{T \subseteq \left\{1, 2, 3, \ldots\right\} ; \ \left|T\right| = \ell}
\sum_{\pi : E \to \left\{1, 2, 3, \ldots\right\} \text{; } \pack \pi = \varphi
\text{; } \pi\left(E\right) = T}
\xx_{\pi, w}
\end{align*}
(because if $\pi : E \to \left\{1, 2, 3, \ldots\right\}$ is a map
satisfying $\pack \pi = \varphi$, then
$\left|\pi\left(E\right)\right| = \ell$\ \ \ \ \footnote{\textit{Proof.}
Let $\pi : E \to \left\{1, 2, 3, \ldots\right\}$ be a map
satisfying $\pack \pi = \varphi$. The definition of $\pack \pi$
yields $\pack \pi = r_{\pi\left(E\right)}^{-1} \circ \pi$. Hence,
$\left|\left(\pack \pi\right)\left(E\right)\right|
= \left|\left(r_{\pi\left(E\right)}^{-1} \circ \pi\right)\left(E\right)\right|
= \left|r_{\pi\left(E\right)}^{-1} \left(\pi\left(E\right)\right)\right|
= \left|\pi\left(E\right)\right|$
(since $r_{\pi\left(E\right)}^{-1}$ is a bijection). Since
$\pack \pi = \varphi$, this rewrites as
$\left|\varphi\left(E\right)\right| = \left|\pi\left(E\right)\right|$.
Hence, $ \left|\pi\left(E\right)\right|
= \left|\varphi\left(E\right)\right| = \ell$, qed.}). But for every
$\ell$-element subset $T$ of
$\left\{1, 2, 3, \ldots\right\}$, there exists exactly
one $\pi : E \to \left\{1, 2, 3, \ldots\right\}$ satisfying
$\pack \pi = \varphi$ and $\pi\left(E\right) = T$: namely,
$\pi = r_T \circ \varphi$\ \ \ \ \footnote{\textit{Proof.}
Let $T$ be an $\ell$-element subset of $\left\{
1,2,3,\ldots\right\}  $. We need to show that there exists exactly one
$\pi:E\rightarrow\left\{  1,2,3,\ldots\right\}  $ satisfying
$ \pack \pi=\varphi$ and $\pi\left(  E\right)  =T$: namely,
$\pi=r_{T}\circ\varphi$. In other words, we need to prove the following two claims:

\begin{statement}
\textit{Claim 1:} The map $r_{T}\circ\varphi$ is a map $\pi:E\rightarrow
\left\{  1,2,3,\ldots\right\}  $ satisfying $ \pack \pi=\varphi$
and $\pi\left(  E\right)  =T$.
\end{statement}

\begin{statement}
\textit{Claim 2:} If $\pi:E\rightarrow\left\{  1,2,3,\ldots\right\}  $ is a
map satisfying $ \pack \pi=\varphi$ and $\pi\left(  E\right)  =T$,
then $\pi=r_{T}\circ\varphi$.
\end{statement}

\textit{Proof of Claim 1.} We have $\left|T\right| = \ell$ (since the set
$T$ is $\ell$-element), thus $\ell = \left|T\right|$.
We have $\left(  r_{T}\circ\varphi\right)  \left(
E\right)  =r_{T}\left(  \underbrace{\varphi\left(  E\right)  }_{=\left\{
1,2,\ldots,\ell\right\}  }\right)  =r_{T}\left(  \left\{  1,2,\ldots
,\underbrace{\ell}_{=\left\vert T\right\vert }\right\}  \right)
=r_{T}\left(  \left\{
1,2,\ldots,\left\vert T\right\vert \right\}  \right)  =T$ (by the definition
of $r_{T}$). Now, the definition of $ \pack \left(  r_{T}%
\circ\varphi\right)  $ shows that
\begin{align*}
 \pack \left(  r_{T}\circ\varphi\right)   & =r_{\left(  r_{T}%
\circ\varphi\right)  \left(  E\right)  }^{-1}\circ\left(  r_{T}\circ
\varphi\right)  =r_{T}^{-1}\circ\left(  r_{T}\circ\varphi\right)
\ \ \ \ \ \ \ \ \ \ \left(  \text{since }\left(  r_{T}\circ\varphi\right)
\left(  E\right)  =T\right) \\
& =\varphi.
\end{align*}
Thus, the map $r_{T}\circ\varphi:E\rightarrow\left\{  1,2,3,\ldots\right\}  $
satisfies $ \pack \left(  r_{T}\circ\varphi\right)  =\varphi$ and
$\left(  r_{T}\circ\varphi\right)  \left(  E\right)  =T$. In other words,
$r_{T}\circ\varphi$ is a map $\pi:E\rightarrow\left\{  1,2,3,\ldots\right\}  $
satisfying $ \pack \pi=\varphi$ and $\pi\left(  E\right)  =T$.
This proves Claim 1.

\textit{Proof of Claim 2.} Let $\pi:E\rightarrow\left\{  1,2,3,\ldots\right\}
$ be a map satisfying $ \pack \pi=\varphi$ and $\pi\left(
E\right)  =T$. The definition of $ \pack \pi$ shows that
$ \pack \pi=r_{\pi\left(  E\right)  }^{-1}\circ\pi=r_{T}^{-1}%
\circ\pi$ (since $\pi\left(  E\right)  =T$). Hence, $r_{T}^{-1}\circ
\pi= \pack \pi=\varphi$, so that $\pi=r_{T}\circ\varphi$. This
proves Claim 2.

Now, both Claims 1 and 2 are proven; hence, our proof is complete.
}. Therefore, for every $\ell$-element subset $T$ of
$\left\{1, 2, 3, \ldots\right\}$, we have
\[
\sum_{\pi : E \to \left\{1, 2, 3, \ldots\right\} \text{; } \pack \pi = \varphi
\text{; } \pi\left(E\right) = T}
\xx_{\pi, w}
= \xx_{r_T\circ\varphi,w} .
\]
Hence,
\begin{align*}
\sum_{\pi\text{ is an }\EE\text{-partition; } \pack \pi = \varphi}
\xx_{\pi, w}
&= \sum_{T \subseteq \left\{1, 2, 3, \ldots\right\} ; \ \left|T\right| = \ell}
\underbrace{\sum_{\pi : E \to \left\{1, 2, 3, \ldots\right\} \text{; } \pack \pi = \varphi
\text{; } \pi\left(E\right) = T}
\xx_{\pi, w}}_{= \xx_{r_T\circ\varphi,w}} \\
&= \sum_{T \subseteq \left\{1, 2, 3, \ldots\right\} ; \ \left|T\right| = \ell}
\xx_{r_T\circ\varphi,w}
= M_{\ev_w \varphi}
\end{align*}
(by \eqref{pf.prop.Gammaw.packed.1.pf.1}).
Thus, \eqref{pf.prop.Gammaw.packed.1}
is proven, and with it Proposition~\ref{prop.Gammaw.packed}.
\end{proof}

\begin{proof}[Proof of Proposition~\ref{prop.Gammaw.qsym}.]
Proposition~\ref{prop.Gammaw.qsym} follows immediately from
Proposition~\ref{prop.Gammaw.packed}
(since $M_\alpha \in \QSym$ for every composition $\alpha$).
\end{proof}

We shall now describe the coproduct of $\Gamma\left(\EE, w\right)$,
essentially giving the proof that is left to the reader in
\cite[Theorem 2.2]{Mal-Reu-DP}.

\begin{definition}
Let $\EE = \left(E, <_1, <_2\right)$ be a double poset.

\begin{itemize}

\item[(a)]
Then, $\Adm \EE$ will mean the set of all pairs
$\left(P, Q\right)$, where $P$ and $Q$ are subsets of $E$ satisfying
$P \cap Q = \varnothing$ and $P \cup Q = E$ and having the property
that no $p \in P$ and $q \in Q$ satisfy $q <_1 p$. These pairs
$\left(P, Q\right)$ are
called the \textit{admissible partitions} of $\EE$. (In the
terminology of \cite{Mal-Reu-DP}, they are the
\textit{decompositions} of $\left(E, <_1\right)$.)

\item[(b)] For
any subset $T$ of $E$, we let $\EE\mid_T$ denote the double poset
$\left(T, <_1, <_2\right)$, where $<_1$ and $<_2$ (by abuse of
notation) denote the restrictions of the relations $<_1$ and $<_2$
to $T$.

\end{itemize}
\end{definition}

\begin{proposition}
\label{prop.Gammaw.coprod}
Let $\EE = \left(E, <_1, <_2\right)$ be a double poset. Let
$w : E \to \left\{1, 2, 3, \ldots\right\}$ be a map.
Then,
\begin{equation}
\label{eq.prop.Gammaw.coprod}
\Delta\left(\Gamma\left(\EE, w\right)\right)
= \sum_{\left(P, Q\right) \in \Adm \EE}
\Gamma\left(\EE\mid_P, w\mid_P\right)
\otimes \Gamma\left(\EE\mid_Q, w\mid_Q\right) .
\end{equation}
\end{proposition}

A particular case of Proposition~\ref{prop.Gammaw.coprod} (namely,
the case when $w\left(e\right) = 1$ for each $e \in E$) appears in
\cite[Th\'eor\`eme 4.16]{Malve-Thesis}.

\begin{vershort}
The proof of Proposition~\ref{prop.Gammaw.coprod} relies on
a simple bijection that an experienced combinatorialist will
have no trouble finding (and proving even less); let us just
give a brief outline of the argument\footnote{See the detailed
version of this note for an (almost) completely written-out
proof; I am afraid that the additional level of detail is of
no help to the understanding.}:

\begin{proof}[Proof of Proposition~\ref{prop.Gammaw.coprod}.]
Whenever
$\alpha = \left(\alpha_1, \alpha_2, \ldots, \alpha_\ell\right)$
is a composition and $k \in \left\{0, 1, \ldots, \ell\right\}$,
we introduce the notation
$\alpha\left[:k\right]$ for the composition
$\left(\alpha_1, \alpha_2, \ldots, \alpha_k\right)$, and the
notation $\alpha\left[k:\right]$ for the composition
$\left(\alpha_{k+1}, \alpha_{k+2}, \ldots, \alpha_\ell\right)$.
Now, the formula \eqref{eq.coproduct.M} can be rewritten as
follows:
\begin{align}
\label{pf.prop.Gammaw.coprod.DeltaM}
\Delta \left( M_\alpha \right)
&= \sum_{k=0}^{\ell} M_{\alpha\left[:k\right]}
\otimes M_{\alpha\left[k:\right]} \\
& \qquad \text{ for every } \ell \in \NN \text{ and every composition } \alpha
\text{ with } \ell \text{ entries.} \nonumber
\end{align}

Now, applying $\Delta$ to the equality
\eqref{eq.prop.Gammaw.packed} yields
\begin{align}
\Delta\left(\Gamma\left(\EE , w\right)\right)
&= \sum_{\varphi \text{ is a packed } \EE\text{-partition}}
\underbrace{\Delta\left(M_{\ev_w \varphi}\right)}_{
  \substack{ = \sum_{k=0}^{\left|\varphi\left(E\right)\right|}
  M_{\left(\ev_w \varphi\right)\left[:k\right]} \otimes
  M_{\left(\ev_w \varphi\right)\left[k:\right]} \\
  \text{(by \eqref{pf.prop.Gammaw.coprod.DeltaM})} }}
 \nonumber \\
&= \sum_{\varphi \text{ is a packed } \EE\text{-partition}}
\sum_{k=0}^{\left|\varphi\left(E\right)\right|}
M_{\left(\ev_w \varphi\right)\left[:k\right]} \otimes
M_{\left(\ev_w \varphi\right)\left[k:\right]} .
\label{pf.Gammaw.coprod.lhs}
\end{align}

On the other hand, rewriting each of the tensorands on the right
hand side of \eqref{eq.prop.Gammaw.coprod} using
\eqref{eq.prop.Gammaw.packed}, we obtain
\begin{align*}
& \sum_{\left(P, Q\right) \in \Adm \EE}
\Gamma\left(\EE\mid_P, w\mid_P\right)
\otimes \Gamma\left(\EE\mid_Q, w\mid_Q\right) \\
&= \sum_{\left(P, Q\right) \in \Adm \EE}
\left(\sum_{\varphi \text{ is a packed } \EE\mid_P\text{-partition}}
M_{\ev_{w\mid_P} \varphi}\right)
\otimes
\left(\sum_{\varphi \text{ is a packed } \EE\mid_Q\text{-partition}}
M_{\ev_{w\mid_Q} \varphi}\right) \\
& = \sum_{\left(P, Q\right) \in \Adm \EE}
\left(\sum_{\sigma \text{ is a packed } \EE\mid_P\text{-partition}}
M_{\ev_{w\mid_P} \sigma}\right)
\otimes
\left(\sum_{\tau \text{ is a packed } \EE\mid_Q\text{-partition}}
M_{\ev_{w\mid_Q} \tau}\right) \\
& = \sum_{\left(P, Q\right) \in \Adm \EE}
\sum_{\sigma \text{ is a packed } \EE\mid_P\text{-partition}}
\sum_{\tau \text{ is a packed } \EE\mid_Q\text{-partition}}
M_{\ev_{w\mid_P} \sigma}
\otimes
M_{\ev_{w\mid_Q} \tau} .
\end{align*}
We need to prove that the right hand sides of this equality and of
\eqref{pf.Gammaw.coprod.lhs} are equal (because then, it will follow
that so are the left hand sides, and thus
Proposition~\ref{prop.Gammaw.coprod} will be proven). For this, it
is clearly enough to exhibit a bijection between
\begin{itemize}
\item the pairs
$\left(\varphi, k\right)$ consisting of a packed $\EE$-partition
$\varphi$ and a
$k \in \left\{0, 1, \ldots, \left|\varphi\left(E\right)\right|
\right\}$
\end{itemize}
and
\begin{itemize}
\item the triples $\left(\left(P, Q\right), \sigma, \tau
\right)$ consisting of a $\left(P, Q\right) \in \Adm \EE$, a packed
$\EE\mid_P$-partition $\sigma$ and a packed $\EE\mid_Q$-partition
$\tau$
\end{itemize}
which bijection has the property that
whenever it maps $\left(\varphi, k\right)$ to
$\left(\left(P, Q\right), \sigma, \tau\right)$,
we have the equalities
$\left(\ev_w \varphi\right)\left[:k\right]
= \ev_{w\mid_P}\sigma$
and
$\left(\ev_w \varphi\right)\left[k:\right]
= \ev_{w\mid_Q}\tau$.
Such a bijection is easy to construct: Given
$\left(\varphi, k\right)$, it sets
$P = \varphi^{-1}\left(\left\{1, 2, \ldots, k\right\}\right)$,
$Q = \varphi^{-1}\left(\left\{k+1, k+2, \ldots, \left|\varphi\left(E\right)\right|\right\}\right)$,
$\sigma = \varphi\mid_P$ and
$\tau = \pack \left(\varphi\mid_Q\right)$\ \ \ \ \footnote{We
notice that these $P$, $Q$, $\sigma$ and $\tau$ satisfy
$\sigma\left(e\right) = \varphi\left(e\right)$ for every
$e \in P$, and $\tau\left(e\right) = \varphi\left(e\right) - k$
for every $e \in Q$.}. Conversely, given
$\left(\left(P, Q\right), \sigma, \tau\right)$,
the inverse bijection
sets $k = \left|\sigma\left(P\right)\right|$ and constructs
$\varphi$ as the map $E \to \left\{1, 2, 3, \ldots\right\}$
which sends every $e \in E$ to
$\begin{cases} \sigma\left(e\right), &\mbox{if } e \in P; \\
\tau\left(e\right) + k, &\mbox{if } e \in Q \end{cases}$.
Proving that this alleged bijection and its alleged inverse
bijection are well-defined and actually mutually inverse is
straightforward and left to the reader\footnote{The only
part of the argument that is a bit trickier is proving the
well-definedness of the inverse bijection: We need to show
that if $\left(\left(P, Q\right), \sigma, \tau
\right)$ is a triple consisting of a
$\left(P, Q\right) \in \Adm \EE$, a packed
$\EE\mid_P$-partition $\sigma$ and a packed $\EE\mid_Q$-partition
$\tau$, and if we set $k = \left|\sigma\left(P\right)\right|$,
then the map $\varphi : E \to \left\{1, 2, 3, \ldots\right\}$
which sends every $e \in E$ to
$\begin{cases} \sigma\left(e\right), &\mbox{if } e \in P; \\
\tau\left(e\right) + k, &\mbox{if } e \in Q \end{cases}$
is actually a packed $\EE$-partition.

Indeed, it is clear that this map $\varphi$ is packed. It remains
to show that it is an $\EE$-partition. To do so, we must prove
the following two claims:

\begin{statement}
\textit{Claim 1:} Every $e \in E$ and $f \in E$ satisfying
$e <_1 f$ satisfy
$\varphi\left(e\right) \leq \varphi\left(f\right)$.
\end{statement}

\begin{statement}
\textit{Claim 2:} Every $e \in E$ and $f \in E$ satisfying
$e <_1 f$ and $f <_2 e$ satisfy
$\varphi\left(e\right) < \varphi\left(f\right)$.
\end{statement}

We shall only prove Claim 1 (as the proof of Claim 2 is
similar). So let $e \in E$ and $f \in E$ be such that
$e <_1 f$. We need to show that
$\varphi\left(e\right) \leq \varphi\left(f\right)$.
We are in one of the following four cases:

\textit{Case 1:} We have $e \in P$ and $f \in P$.

\textit{Case 2:} We have $e \in P$ and $f \in Q$.

\textit{Case 3:} We have $e \in Q$ and $f \in P$.

\textit{Case 4:} We have $e \in Q$ and $f \in Q$.

In Case 1, our claim
$\varphi\left(e\right) \leq \varphi\left(f\right)$ follows
from the assumption that $\sigma$ is an
$\EE\mid_P$-partition (because in Case 1,
we have $\varphi\left(e\right) = \sigma\left(e\right)$
and $\varphi\left(f\right) = \sigma\left(f\right)$).
In Case 4, it follows from the
assumption that $\tau$ is an $\EE\mid_Q$-partition
(since in Case 4, we have
$\varphi\left(e\right) = \tau\left(e\right) + k$ and
$\varphi\left(f\right) = \tau\left(f\right) + k$). In
Case 2, it clearly holds (indeed,
if $e \in P$, then the definition of $\varphi$ yields
$\varphi\left(e\right) = \sigma\left(e\right) \leq k$,
and if $f \in Q$, then
the definition of $\varphi$ yields
$\varphi\left(f\right) = \tau\left(f\right) + k > k$;
therefore, in Case 2, we have
$\varphi\left(e\right) \leq k < \varphi\left(f\right)$).
Finally, Case 3 is impossible (because having $e \in Q$
and $f \in P$ and $e <_1 f$ would contradict
$\left(P, Q\right) \in \Adm \EE$). Thus, we have proven the
claim in each of the four cases, and consequently Claim 1 is
proven. As we have said above, Claim 2 is proven similarly.}.
\end{proof}
\end{vershort}

\begin{verlong}
The proof of Proposition \ref{prop.Gammaw.coprod} is based upon a simple
bijection. We shall introduce it after some preparations.

\begin{lemma}
\label{lem.Gammaw.coprod.bij.a}Let $ \EE =\left(  E,<_{1},<_{2}\right)
$ be a double poset.

Let $\mathcal{S}$ be the set of all pairs $\left(  \varphi,k\right)  $
consisting of a packed $\EE$-partition $\varphi$ and a $k\in\left\{
0,1,\ldots,\left\vert \varphi\left(  E\right)  \right\vert \right\}  $.

Let $\mathcal{T}$ be the set of all triples $\left(  \left(  P,Q\right)
,\sigma,\tau\right)  $ consisting of a $\left(  P,Q\right)  \in
\Adm \EE $, a packed $ \EE \mid_{P}$-partition
$\sigma$ and a packed $ \EE \mid_{Q}$-partition $\tau$.

For every $\ell\in\mathbb{Z}$, we let $\operatorname{add}_{\ell}$
denote the bijective map $\mathbb{Z}\rightarrow\mathbb{Z},\ z\mapsto z+\ell$.

Fix $\left(  \varphi,k\right)  \in\mathcal{S}$. Set
\begin{align}
P  &  =\varphi^{-1}\left(  \left\{  1,2,\ldots,k\right\}  \right)
,\ \ \ \ \ \ \ \ \ \ Q=\varphi^{-1}\left(  \left\{  k+1,k+2,\ldots,\left\vert
\varphi\left(  E\right)  \right\vert \right\}  \right)
,\label{eq.lem.Gammaw.coprod.bij.def1}\\
\sigma &  =\varphi\mid_{P}\ \ \ \ \ \ \ \ \ \ \text{and}
\ \ \ \ \ \ \ \ \ \ \tau=\operatorname{add}_{-k}\circ\left(
\varphi\mid_{Q}\right)  .
\label{eq.lem.Gammaw.coprod.bij.def2}
\end{align}
Then, $\left(  \left(  P,Q\right)  ,\sigma,\tau\right)  \in\mathcal{T}$.
\end{lemma}

\begin{lemma}
\label{lem.Gammaw.coprod.bij.b}
Let $\EE = \left(  E,<_{1},<_{2}\right)  $
be a double poset. Let $\mathcal{S}$ and $\mathcal{T}$ be defined as in Lemma
\ref{lem.Gammaw.coprod.bij.a}.

Fix $\left(  \left(  P,Q\right)  ,\sigma,\tau\right)  \in\mathcal{T}$. Set
$k=\left\vert \sigma\left(  P\right)  \right\vert $, and let $\varphi$ be the
map $E\rightarrow\left\{  1,2,3,\ldots\right\}  $ which sends every $e\in E$
to $%
\begin{cases}
\sigma\left(  e\right)  , & \text{if }e\in P;\\
\tau\left(  e\right)  +k, & \text{if }e\in Q
\end{cases}
$. Then, $\left(  \varphi,k\right)  \in\mathcal{S}$.
\end{lemma}

\begin{lemma}
\label{lem.Gammaw.coprod.bij}
Let $\EE = \left(  E,<_{1},<_{2}\right)  $
be a double poset. Let $\mathcal{S}$, $\mathcal{T}$ and
$\operatorname{add}_{\ell}$ be defined as in
Lemma \ref{lem.Gammaw.coprod.bij.a}.

Define a map $\Phi:\mathcal{S}\rightarrow\mathcal{T}$ as follows: Let $\left(
\varphi,k\right)  \in\mathcal{S}$. Then, define $P$, $Q$, $\sigma$ and $\tau$
by (\ref{eq.lem.Gammaw.coprod.bij.def1}) and
(\ref{eq.lem.Gammaw.coprod.bij.def2}). From Lemma
\ref{lem.Gammaw.coprod.bij.a}, we know that $\left(  \left(  P,Q\right)
,\sigma,\tau\right)  \in\mathcal{T}$. Define $\Phi\left(  \varphi,k\right)  $
to be $\left(  \left(  P,Q\right)  ,\sigma,\tau\right)  $. Thus, a map
$\Phi:\mathcal{S}\rightarrow\mathcal{T}$ is defined.

Define a map $\Psi:\mathcal{T}\rightarrow\mathcal{S}$ as follows: Let $\left(
\left(  P,Q\right)  ,\sigma,\tau\right)  \in\mathcal{T}$. Set $k=\left\vert
\sigma\left(  P\right)  \right\vert $, and let $\varphi$ be the map
$E\rightarrow\left\{  1,2,3,\ldots\right\}  $ which sends every $e\in E$ to $%
\begin{cases}
\sigma\left(  e\right)  , & \text{if }e\in P;\\
\tau\left(  e\right)  +k, & \text{if }e\in Q
\end{cases}
$. From Lemma \ref{lem.Gammaw.coprod.bij.b}, we know that $\left(
\varphi,k\right)  \in\mathcal{S}$. Set $\Psi\left(  \left(  P,Q\right)
,\sigma,\tau\right)  =\left(  \varphi,k\right)  $. Thus, a map $\Psi
:\mathcal{T}\rightarrow\mathcal{S}$ is defined.

The maps $\Phi : \mathcal{S} \rightarrow \mathcal{T}$ and
$\Psi : \mathcal{T} \rightarrow \mathcal{S}$ are mutually inverse.
\end{lemma}

The preceding three lemmas should be obvious if the reader has
\textquotedblleft the right picture in their mind\textquotedblright. The
following proof is merely a formalization of the argument that such a picture
would straightforwardly produce; we are not sure whether it is actually worth
reading (as opposed to trying to conjure \textquotedblleft the right
picture\textquotedblright).

\begin{proof}
[Proof of Lemma \ref{lem.Gammaw.coprod.bij.a}.]
We have $\left(  \varphi ,k\right)  \in\mathcal{S}$. Thus,
$\varphi$ is a packed $\EE$-partition, and $k$ is an element
of $\left\{  0,1,\ldots,\left\vert
\varphi\left(  E\right)  \right\vert \right\}  $ (by the definition of
$\mathcal{S}$).

The map $\varphi : E \to \left\{ 1, 2, 3, \ldots \right\}$ is packed
and satisfies
$\left| \varphi\left(E\right) \right| = \left| \varphi\left(E\right) \right|$.
Hence, $\varphi\left(  E\right)  =\left\{  1,2,\ldots,\left\vert
\varphi\left(  E\right)  \right\vert \right\}  $
(by Proposition~\ref{prop.ev.comp} (a) (applied to $\varphi$ and
$\left| \varphi\left(E\right) \right|$
instead of $\pi$ and $\ell$)).

Now, $\left(  P,Q\right)  \in \Adm \EE$%
\ \ \ \ \footnote{\textit{Proof.} It is clear that $P$ and $Q$ are subsets of
$E$. Also, from (\ref{eq.lem.Gammaw.coprod.bij.def1}), we obtain
\begin{align*}
P\cap Q  &  =\varphi^{-1}\left(  \left\{  1,2,\ldots,k\right\}  \right)
\cap\varphi^{-1}\left(  \left\{  k+1,k+2,\ldots,\left\vert \varphi\left(
E\right)  \right\vert \right\}  \right) \\
&  =\varphi^{-1}\left(  \underbrace{\left\{  1,2,\ldots,k\right\}
\cap\left\{  k+1,k+2,\ldots,\left\vert \varphi\left(  E\right)  \right\vert
\right\}  }_{=\varnothing}\right)  =\varphi^{-1}\left(  \varnothing\right)
=\varnothing
\end{align*}
and
\begin{align*}
P\cup Q  &  =\varphi^{-1}\left(  \left\{  1,2,\ldots,k\right\}  \right)
\cup\varphi^{-1}\left(  \left\{  k+1,k+2,\ldots,\left\vert \varphi\left(
E\right)  \right\vert \right\}  \right) \\
&  =\varphi^{-1}\left(  \underbrace{\left\{  1,2,\ldots,k\right\}
\cup\left\{  k+1,k+2,\ldots,\left\vert \varphi\left(  E\right)  \right\vert
\right\}  }_{\substack{=\left\{  1,2,\ldots,\left\vert \varphi\left(
E\right)  \right\vert \right\}  =\varphi\left(  E\right)  \\\text{(since
}\varphi\text{ is packed)}}}\right)  =\varphi^{-1}\left(  \varphi\left(
E\right)  \right)  =E.
\end{align*}
Hence, in order to prove that $\left( P, Q \right) \in \Adm
\EE$, it remains to show that no $p\in P$ and $q\in Q$ satisfy
$q<_{1}p$.
\par
Let us assume the contrary (for the sake of contradiction). Thus, let $p\in P$
and $q\in Q$ be such that $q<_{1}p$. Since $\varphi$ is an
$\EE$-partition, we have
$\varphi\left(  q\right)  \leq\varphi\left(  p\right)  $
(because $q<_{1}p$). But $p\in P=\varphi^{-1}\left(  \left\{  1,2,\ldots
,k\right\}  \right)  $, so that $\varphi\left(  p\right)  \leq k$. On the
other hand, $q\in Q=\varphi^{-1}\left(  \left\{  k+1,k+2,\ldots,\left\vert
\varphi\left(  E\right)  \right\vert \right\}  \right)  $, so that
$\varphi\left(  q\right)  >k$. This contradicts $\varphi\left(  q\right)
\leq\varphi\left(  p\right)  \leq k$. This contradiction shows that our
assumption was false. Hence, the proof of $\left(  P,Q\right)  \in
\Adm \EE$ is complete.}. Furthermore, it is
straightforward to see that for every subset $T$ of $E$,
\begin{equation}
\text{the map }\varphi\mid_{T}\text{ is an }
\EE\mid_{T}\text{-partition}.
\label{pf.lem.Gammaw.coprod.bij.a.1}
\end{equation}
Applying this to $T=P$, we conclude that $\varphi\mid_{P}$ is an
$\EE\mid_{P}$-partition.

Since $P=\varphi^{-1}\left(  \left\{  1,2,\ldots,k\right\}  \right)  $, we
have $\varphi\left(  P\right)  \subseteq\left\{  1,2,\ldots,k\right\}  $.
Moreover, this inclusion is actually an equality (since $\varphi\left(
E\right)  =\left\{  1,2,\ldots,\left\vert \varphi\left(  E\right)  \right\vert
\right\}  $)\ \ \ \ \footnote{The proof in more detail: Let $g\in\left\{
1,2,\ldots,k\right\}  $. Then, $g\in\left\{  1,2,\ldots,k\right\}
\subseteq\left\{  1,2,\ldots,\left\vert \varphi\left(  E\right)  \right\vert
\right\}  =\varphi\left(  E\right)  $. Thus, there exists some $e\in E$ such
that $g=\varphi\left(  e\right)  $. Consider this $e$. From $\varphi\left(
e\right)  =g\in\left\{  1,2,\ldots,k\right\}  $, we obtain $e\in\varphi
^{-1}\left(  \left\{  1,2,\ldots,k\right\}  \right)  =P$. Thus, $\varphi
\left(  e\right)  \in\varphi\left(  P\right)  $, so that $g=\varphi\left(
e\right)  \in\varphi\left(  P\right)  $. Now, let us forget that we fixed $g$.
We thus have proven that $g\in\varphi\left(  P\right)  $ for every
$g\in\left\{  1,2,\ldots,k\right\}  $. In other words, $\left\{
1,2,\ldots,k\right\}  \subseteq\varphi\left(  P\right)  $. Combining this with
$\varphi\left(  P\right)  \subseteq\left\{  1,2,\ldots,k\right\}  $, we obtain
$\varphi\left(  P\right)  =\left\{  1,2,\ldots,k\right\}  $, qed.}. In other
words, we have
\begin{equation}
\varphi\left(  P\right)  =\left\{  1,2,\ldots,k\right\}  .
\label{pf.lem.Gammaw.coprod.bij.a.3}
\end{equation}
Similarly,
\begin{equation}
\varphi\left(  Q\right)  =\left\{  k+1,k+2,\ldots,\left\vert \varphi\left(
E\right)  \right\vert \right\}  .
\label{pf.lem.Gammaw.coprod.bij.a.4}
\end{equation}
Hence,
\begin{align*}
\left(  \operatorname{add}_{-k}\circ\left(  \varphi\mid_{Q}\right)
\right)  \left(  Q\right)   &  =\operatorname{add}_{-k}\left(
\underbrace{\left(  \varphi\mid_{Q}\right)  \left(  Q\right)  }_{=\varphi
\left(  Q\right)  =\left\{  k+1,k+2,\ldots,\left\vert \varphi\left(  E\right)
\right\vert \right\}  }\right) \\
&=  \operatorname{add}_{-k}\left(
\left\{  k+1,k+2,\ldots,\left\vert \varphi\left(  E\right)  \right\vert
\right\}  \right) \\
&  =\left\{  1,2,\ldots,\left\vert \varphi\left(  E\right)  \right\vert
-k\right\}
\end{align*}
(by the definition of $\operatorname{add}_{-k}$). Since $\left(
\varphi\mid_{P}\right)  \left(  P\right)  =\varphi\left(  P\right)  =\left\{
1,2,\ldots,k\right\}  $ is an initial interval, we deduce that the
$\EE\mid_{P}$-partition $\varphi\mid_{P}$ is packed. Thus,
$\sigma=\varphi\mid_{P}$ is a packed $\EE\mid_{P}$-partition.

On the other hand, (\ref{pf.lem.Gammaw.coprod.bij.a.1}) (applied to $T=Q$)
shows that $\varphi\mid_{Q}$ is an $\EE\mid_{Q}$-partition. Hence, the
map $\operatorname{add}_{-k}\circ\left(  \varphi\mid_{Q}\right)  $
is an $\EE\mid_{Q}$-partition (since the map
$\operatorname{add}_{-k}$ is strictly increasing, and since
$\left(  \operatorname{add}_{-k}\circ\left(  \varphi\mid_{Q}\right)
 \right)  \left(  Q\right)
=\left\{  1,2,\ldots,\left\vert \varphi\left(  E\right)  \right\vert
-k\right\}  \subseteq\left\{  1,2,3,\ldots\right\}  $). This $\EE%
\mid_{Q}$-partition $\operatorname{add}_{-k}\circ\left(  \varphi
\mid_{Q}\right)  $ is packed \newline (since
$\left(  \operatorname{add}_{-k}\circ
\left(  \varphi\mid_{Q}\right)  \right)  \left(  Q\right)
=\left\{  1,2,\ldots,\left\vert \varphi\left(  E\right)  \right\vert
-k\right\}  $ is an initial interval). Thus,
$\tau = \operatorname{add}_{-k} \circ\left(  \varphi\mid_{Q}\right)  $
is a packed $\EE\mid_{Q}$-partition.

We now know that $\left(  P,Q\right)  \in \Adm \EE$, that
$\sigma$ is a packed $\EE\mid_{P}$-partition, and that $\tau$ is a
packed $\EE\mid_{Q}$-partition. In other words, we know that $\left(
\left(  P,Q\right)  ,\sigma,\tau\right)  \in\mathcal{T}$. This proves Lemma
\ref{lem.Gammaw.coprod.bij.a}.
\end{proof}

\begin{proof}
[Proof of Lemma \ref{lem.Gammaw.coprod.bij.b}.]
We have $\left( \left( P, Q \right) , \sigma, \tau \right)
\in \mathcal{T}$. According to the definition of $\mathcal{T}$, this
means that $\left( P, Q \right) \in \Adm \EE $, that $\sigma$ is a
packed $\EE \mid_{P}$-partition, and that $\tau$ is a packed
$\EE \mid_{Q}$-partition.

From $\left(  P,Q\right)  \in \Adm \EE$, we conclude that
$P$ and $Q$ are subsets of $E$ satisfying $P\cap Q=\varnothing$ and $P\cup
Q=E$ and having the property that
\begin{equation}
\text{no }p\in P\text{ and }q\in Q\text{ satisfy }q<_{1}p.
\label{pf.lem.Gammaw.coprod.bij.b.Adm}
\end{equation}
Using $P \cap Q = \varnothing$ and $P \cup Q = E$, we see that the
map $\varphi$ is well-defined. (Indeed, we defined it as the map
$E \to \left\{1,2,3,\ldots\right\}$ which sends every $e\in E$ to $%
\begin{cases}
\sigma\left(  e\right)  , & \text{if }e\in P;\\
\tau\left(  e\right)  +k, & \text{if }e\in Q
\end{cases}
$.)

The map $\sigma : P \to \left\{ 1, 2, 3, \ldots \right\}$ is packed,
and we have
$k = \left\vert \sigma\left(  P\right)  \right\vert$.
Hence, Proposition~\ref{prop.ev.comp} (a) (applied to $P$, $\sigma$ and
$k$ instead of $E$, $\pi$ and $\ell$)
yields $\sigma\left(  P\right)  = \left\{ 1,2,\ldots,k\right\}  $.

The map $\tau : Q \to \left\{ 1, 2, 3, \ldots \right\}$ is packed,
and we have
$\left\vert \tau\left(  Q\right)  \right\vert
= \left\vert \tau\left(  Q\right)  \right\vert$.
Hence, Proposition~\ref{prop.ev.comp} (a) (applied to $Q$, $\tau$ and
$\left\vert \tau\left(  Q\right)  \right\vert$ instead of
$E$, $\pi$ and $\ell$)
yields $\tau\left(  Q\right)  =\left\{
1,2,\ldots,\left\vert \tau\left(  Q\right)  \right\vert \right\}  $.

The definition of $\varphi$ shows that%
\begin{equation}
\varphi\left(  e\right)  =\sigma\left(  e\right)
\ \ \ \ \ \ \ \ \ \ \text{for every }e\in P.
\label{pf.lem.Gammaw.coprod.bij.b.1o}%
\end{equation}
Hence, $\varphi\left(  P\right)  =\sigma\left(  P\right)
=\left\{  1,2,\ldots,k\right\}  $.

Also, the definition of $\varphi$ shows that%
\begin{equation}
\varphi\left(  e\right)  =\tau\left(  e\right)
+k\ \ \ \ \ \ \ \ \ \ \text{for every }e\in Q.
\label{pf.lem.Gammaw.coprod.bij.b.2o}
\end{equation}
Thus,
\begin{align*}
\varphi\left(  Q\right)
& = \left\{ \underbrace{\varphi\left(e\right)}_{= \tau\left(e\right) + k}
\mid e \in Q \right\} = \left\{ \tau\left(e\right) + k \mid e \in Q \right\} \\
&  =\left\{  u+k\ \mid\ u\in\underbrace{\tau\left(
Q\right)  }_{=\left\{  1,2,\ldots,\left\vert \tau\left( Q \right)  \right\vert
\right\}  }\right\}  =\left\{  u+k\ \mid\ u\in\left\{  1,2,\ldots,\left\vert
\tau\left(  Q\right)  \right\vert \right\}  \right\} \\
&  =\left\{  k+1,k+2,\ldots,k+\left\vert \tau\left( Q \right)  \right\vert
\right\}  .
\end{align*}
Now, $E=P\cup Q$, so that%
\begin{align}
\varphi\left(  E\right)   &  =\varphi\left(  P\cup Q\right)
=\underbrace{\varphi\left(  P\right)  }_{=\left\{  1,2,\ldots,k\right\}  }%
\cup\underbrace{\varphi\left(  Q\right)  }_{=\left\{  k+1,k+2,\ldots
,k+\left\vert \tau\left( Q \right)  \right\vert \right\}  }\nonumber\\
&  =\left\{  1,2,\ldots,k\right\}  \cup\left\{  k+1,k+2,\ldots,k+\left\vert
\tau\left( Q \right)  \right\vert \right\}  \nonumber \\
&=  \left\{  1,2,\ldots,k+\left\vert
\tau\left( Q \right)  \right\vert \right\}  .
\label{pf.lem.Gammaw.coprod.bij.b.3o}%
\end{align}
Thus, $\varphi\left(  E\right)  $ is an initial interval; in other words, the
map $\varphi$ is packed. Furthermore, (\ref{pf.lem.Gammaw.coprod.bij.b.3o})
shows that $\left\vert \varphi\left(  E\right)  \right\vert
=k+\underbrace{\left\vert \tau\left( Q \right)  \right\vert }_{\geq0}\geq k$,
so that $k\in\left\{  0,1,\ldots,\left\vert \varphi\left(  E\right)
\right\vert \right\}  $.

We shall now show that $\varphi$ is an $\EE$-partition. To do so, we
must prove the following two claims:

\begin{statement}
\textit{Claim 1:} Every $e\in E$ and $f\in E$ satisfying $e<_{1}f$ satisfy
$\varphi\left(  e\right)  \leq\varphi\left(  f\right)  $.
\end{statement}

\begin{statement}
\textit{Claim 2:} Every $e\in E$ and $f\in E$ satisfying $e<_{1}f$ and
$f<_{2}e$ satisfy $\varphi\left(  e\right)  <\varphi\left(  f\right)  $.
\end{statement}

We shall only prove Claim 1 (as the proof of Claim 2 is similar). So let $e\in
E$ and $f\in E$ be such that $e<_{1}f$. We need to show that $\varphi\left(
e\right)  \leq\varphi\left(  f\right)  $. We are in one of the following four cases:

\textit{Case 1:} We have $e\in P$ and $f\in P$.

\textit{Case 2:} We have $e\in P$ and $f\in Q$.

\textit{Case 3:} We have $e\in Q$ and $f\in P$.

\textit{Case 4:} We have $e\in Q$ and $f\in Q$.

In Case 1, our claim $\varphi\left(  e\right)  \leq\varphi\left(  f\right)  $
follows from the assumption that $\sigma$ is an $ \EE \mid_{P}%
$-partition\footnote{\textit{Proof.} Assume that we are in Case 1. Thus, we
have $e\in P$ and $f\in P$. Thus, $e$ and $f$ are elements of $P$ satisfying
$e<_{1}f$. Hence, $\sigma\left(  e\right)  \leq\sigma\left(  f\right)  $
(since $\sigma$ is an $ \EE \mid_{P}$-partition). But the definition of
$\varphi$ yields $\varphi\left(  e\right)  =%
\begin{cases}
\sigma\left(  e\right)  , & \text{if }e\in P;\\
\tau\left(  e\right)  +k, & \text{if }e\in Q
\end{cases}
=\sigma\left(  e\right)  $ (since $e\in P$) and $\varphi\left(  f\right)
=\sigma\left(  f\right)  $ (similarly). Hence, $\varphi\left(  e\right)
=\sigma\left(  e\right)  \leq\sigma\left(  f\right)  =\varphi\left(  f\right)
$. Qed.}. In Case 4, it follows from the assumption that $\tau$ is an
$ \EE \mid_{Q}$-partition\footnote{\textit{Proof.} Assume that we are
in Case 4. Thus, we have $e\in Q$ and $f\in Q$. Thus, $e$ and $f$ are elements
of $Q$ satisfying $e<_{1}f$. Hence, $\tau\left(  e\right)  \leq\tau\left(
f\right)  $ (since $\tau$ is an $ \EE \mid_{Q}$-partition). But the
definition of $\varphi$ yields $\varphi\left(  e\right)  =%
\begin{cases}
\sigma\left(  e\right)  , & \text{if }e\in P;\\
\tau\left(  e\right)  +k, & \text{if }e\in Q
\end{cases}
=\tau\left(  e\right)  +k$ (since $e\in Q$) and $\varphi\left(  f\right)
=\tau\left(  f\right)  +k$ (similarly). Hence, $\varphi\left(  e\right)
=\underbrace{\tau\left(  e\right)  }_{\leq\tau\left(  f\right)  }+k\leq
\tau\left(  f\right)  +k=\varphi\left(  f\right)  $. Qed.}. In Case 2, it
clearly holds\footnote{\textit{Proof.} Assume that we are in Case 2. Thus, we
have $e\in P$ and $f\in Q$. The definition of $\varphi$ yields $\varphi\left(
e\right)  =
\begin{cases}
\sigma\left(  e\right)  , & \text{if }e\in P;\\
\tau\left(  e\right)  +k, & \text{if }e\in Q
\end{cases}
=\sigma\left(  e\right)  $ (since $e\in P$) and $\varphi\left(  f\right)  =%
\begin{cases}
\sigma\left(  f\right)  , & \text{if }f\in P;\\
\tau\left(  f\right)  +k, & \text{if }f\in Q
\end{cases}
=\tau\left(  f\right)  +k$ (since $f\in Q$). But we have $\varphi\left(
e\right)  =\sigma\left(  \underbrace{e}_{\in P}\right)  \in\sigma\left(
P\right)  =\left\{  1,2,\ldots,k\right\}  $, so that $\varphi\left(  e\right)
\leq k$. Meanwhile, $\varphi\left(  f\right)  =\underbrace{\tau\left(
f\right)  }_{>0}+k>k$. Thus, $\varphi\left(  e\right)  \leq k<\varphi\left(
f\right)  $, and therefore $\varphi\left(  e\right)  \leq\varphi\left(
f\right)  $. Qed.}. Finally, Case 3 is impossible\footnote{\textit{Proof.}
Assume that we are in Case 3. Thus, $e\in Q$ and $f\in P$. The elements $f\in
P$ and $e\in Q$ satisfy $e<_{1}f$. This contradicts
(\ref{pf.lem.Gammaw.coprod.bij.b.Adm}) (applied to $p=f$ and $q=e$). Thus, we
have obtained a contradiction; hence, our assumption (that we are in Case 3)
was wrong. Therefore, Case 3 is impossible.}. Thus, we have proven the claim
in each of the four cases, and consequently Claim 1 is proven. As we have said
above, Claim 2 is proven similarly. Thus, we have proven that $\varphi$ is an
$\EE$-partition.

Altogether, we now know that $\varphi$ is a packed $\EE$-partition,
and that \newline
$k\in\left\{  0,1,\ldots,\left\vert \varphi\left(  E\right)
\right\vert \right\}  $. In other words, $\left(  \varphi,k\right)
\in\mathcal{S}$. This proves Lemma \ref{lem.Gammaw.coprod.bij.b}.
\end{proof}

\begin{proof}
[Proof of Lemma \ref{lem.Gammaw.coprod.bij}.]We need to prove the following
two claims:

\begin{statement}
\textit{Claim 1:} We have $\Phi\circ\Psi=\id$.
\end{statement}

\begin{statement}
\textit{Claim 2:} We have $\Psi\circ\Phi=\id$.
\end{statement}

\textit{Proof of Claim 1:} Fix $\left(  \left(  P,Q\right)  ,\sigma
,\tau\right)  \in\mathcal{T}$. Set $k=\left\vert \sigma\left(  P\right)
\right\vert $, and let $\varphi$ be the map $E\rightarrow\left\{
1,2,3,\ldots\right\}  $ which sends every $e\in E$ to $%
\begin{cases}
\sigma\left(  e\right)  , & \text{if }e\in P;\\
\tau\left(  e\right)  +k, & \text{if }e\in Q
\end{cases}
$. The definition of $\Psi$ thus yields $\Psi\left(  \left(  P,Q\right)
,\sigma,\tau\right)  =\left(  \varphi,k\right)  $. We shall now show that
$\Phi\left(  \varphi,k\right)  =\left(  \left(  P,Q\right)  ,\sigma
,\tau\right)  $.

Lemma \ref{lem.Gammaw.coprod.bij.b} shows that $\left(  \varphi,k\right)
\in\mathcal{S}$. In other words, $\varphi$ is a packed $ \EE %
$-partition, and we have $k\in\left\{  0,1,\ldots,\left\vert \varphi\left(
E\right)  \right\vert \right\}  $.
The map $\varphi : E \to \left\{ 1, 2, 3, \ldots \right\}$ is packed
and satisfies
$\left\vert \varphi\left( E\right) \right\vert
= \left\vert \varphi\left( E\right) \right\vert $.
Hence, we have
$\varphi \left( E \right)
= \left\{ 1, 2, \ldots, \left\vert \varphi\left( E\right)
            \right\vert \right\}$
(by Proposition~\ref{prop.ev.comp} (a) (applied to $\varphi$ and
$\left| \varphi\left(E\right) \right|$
instead of $\pi$ and $\ell$)).

The map $\operatorname{add}_{k}:\mathbb{Z}\rightarrow\mathbb{Z}$ is
a bijection, and its inverse is $\left(  \operatorname{add}_{k}%
\right)  ^{-1}=\operatorname{add}_{-k}$.

The map $\sigma : P \to \left\{ 1, 2, 3, \ldots \right\}$ is packed,
and we have
$k = \left\vert \sigma\left(  P\right)  \right\vert$.
Hence, Proposition~\ref{prop.ev.comp} (a) (applied to $P$, $\sigma$ and
$k$ instead of $E$, $\pi$ and $\ell$)
yields $\sigma\left(  P\right) = \left\{ 1,2,\ldots,k\right\}  $.

From $\left(  P,Q\right)  \in \Adm \EE$, we conclude that
$P$ and $Q$ are subsets of $E$ satisfying $P\cap Q=\varnothing$ and $P\cup
Q=E$. Hence, $Q=E\setminus P$ and $P=E\setminus Q$.

The definition of $\varphi$ shows that%
\begin{equation}
\varphi\left(  e\right)  =\sigma\left(  e\right)
\ \ \ \ \ \ \ \ \ \ \text{for every }e\in P.
\label{pf.lem.Gammaw.coprod.bij.c.1}%
\end{equation}
Hence, $\varphi\mid_{P}=\sigma$. Also, the definition of $\varphi$ shows that%
\begin{equation}
\varphi\left(  e\right)  =\tau\left(  e\right)
+k\ \ \ \ \ \ \ \ \ \ \text{for every }e\in Q.
\label{pf.lem.Gammaw.coprod.bij.c.2}%
\end{equation}
Thus, every $e\in Q$ satisfies%
\begin{align}
\varphi\left(  e\right)   &  =\tau\left(  e\right)  +k
= \operatorname{add}_{k}\left(  \tau\left(  e\right)  \right)
\ \ \ \ \ \ \ \ \ \ \left(  \text{since }\operatorname{add}_{k}%
\left(  \tau\left(  e\right)  \right)  \text{ is defined to be }\tau\left(
e\right)  +k\right) \nonumber\\
&  =\left(  \operatorname{add}_{k}\circ\tau\right)  \left(
e\right)  . \label{pf.lem.Gammaw.coprod.bij.c.2b}%
\end{align}
Hence, $\varphi\mid_{Q}=\operatorname{add}_{k}\circ\tau$, so that
$\tau=\underbrace{\left(  \operatorname{add}_{k}\right)  ^{-1}%
}_{=\operatorname{add}_{-k}}\circ\left(  \varphi\mid_{Q}\right)
=\operatorname{add}_{-k}\circ\left(  \varphi\mid_{Q}\right)  $.

Furthermore, $P=\varphi^{-1}\left(  \left\{  1,2,\ldots,k\right\}  \right)
$\ \ \ \ \footnote{\textit{Proof.} Let $e\in\varphi^{-1}\left(  \left\{
1,2,\ldots,k\right\}  \right)  $. Thus, $e\in E$ and $\varphi\left(  e\right)
\in\left\{  1,2,\ldots,k\right\}  $. If we had $e\in Q$, then we would have
\begin{align*}
\varphi\left(  e\right)   &  =\underbrace{\tau\left(  e\right)  }%
_{>0}+k\ \ \ \ \ \ \ \ \ \ \left(  \text{by
(\ref{pf.lem.Gammaw.coprod.bij.c.2})}\right) \\
&  >k,
\end{align*}
which would contradict $\varphi\left(  e\right)  \in\left\{  1,2,\ldots
,k\right\}  $. Hence, we cannot have $e\in Q$. Thus, $e\in E\setminus Q=P$.
\par
Now, let us forget that we fixed $e$. Thus we have proven that $e\in P$ for
every $e\in\varphi^{-1}\left(  \left\{  1,2,\ldots,k\right\}  \right)  $. In
other words, $\varphi^{-1}\left(  \left\{  1,2,\ldots,k\right\}  \right)
\subseteq P$.
\par
On the other hand, fix $p\in P$. Then, $\varphi\left(  p\right)
=\sigma\left(  p\right)  $ (by (\ref{pf.lem.Gammaw.coprod.bij.c.1})). Hence,
$\varphi\left(  p\right)  =\sigma\left(  p\right)  \in\sigma\left(  P\right)
=\left\{  1,2,\ldots,k\right\}  $, so that $p\in\varphi^{-1}\left(  \left\{
1,2,\ldots,k\right\}  \right)  $.
\par
Now, let us forget that we fixed $p$. Thus we have proven that $p\in
\varphi^{-1}\left(  \left\{  1,2,\ldots,k\right\}  \right)  $ for every $p\in
P$. In other words, $P\subseteq\varphi^{-1}\left(  \left\{  1,2,\ldots
,k\right\}  \right)  $. Combining this with $\varphi^{-1}\left(  \left\{
1,2,\ldots,k\right\}  \right)  \subseteq P$, we obtain $P=\varphi^{-1}\left(
\left\{  1,2,\ldots,k\right\}  \right)  $, qed.} and $Q=\varphi^{-1}\left(
\left\{  k+1,k+2,\ldots,\left\vert \varphi\left(  E\right)  \right\vert
\right\}  \right)  $\ \ \ \ \footnote{\textit{Proof.} Let $e\in\varphi
^{-1}\left(  \left\{  k+1,k+2,\ldots,\left\vert \varphi\left(  E\right)
\right\vert \right\}  \right)  $. Thus, $e\in E$ and $\varphi\left(  e\right)
\in\left\{  k+1,k+2,\ldots,\left\vert \varphi\left(  E\right)  \right\vert
\right\}  $. If we had $e\in P$, then we would have
\begin{align*}
\varphi\left(  e\right)   &  =\sigma\left(  e\right)
\ \ \ \ \ \ \ \ \ \ \left(  \text{by (\ref{pf.lem.Gammaw.coprod.bij.c.1}%
)}\right) \\
&  \in\sigma\left(  P\right)  =\left\{  1,2,\ldots,k\right\}
\end{align*}
and therefore $\varphi\left(  e\right)  \leq k$, which would contradict
$\varphi\left(  e\right)  \in\left\{  k+1,k+2,\ldots,\left\vert \varphi\left(
E\right)  \right\vert \right\}  $. Hence, we cannot have $e\in P$. Thus, $e\in
E\setminus P=Q$.
\par
Now, let us forget that we fixed $e$. Thus we have proven that $e\in Q$ for
every $e\in\varphi^{-1}\left(  \left\{  k+1,k+2,\ldots,\left\vert
\varphi\left(  E\right)  \right\vert \right\}  \right)  $. In other words,
$\varphi^{-1}\left(  \left\{  k+1,k+2,\ldots,\left\vert \varphi\left(
E\right)  \right\vert \right\}  \right)  \subseteq Q$.
\par
On the other hand, fix $q\in Q$. Then, $\varphi\left(  q\right)  =\tau\left(
q\right)  +k$ (by (\ref{pf.lem.Gammaw.coprod.bij.c.2})). Hence, $\varphi
\left(  q\right)  =\underbrace{\tau\left(  q\right)  }_{>0}+k>k$. Combining
this with $\varphi\left(  q\right)  \in\varphi\left(  E\right)  =\left\{
1,2,\ldots,\left\vert \varphi\left(  E\right)  \right\vert \right\}  $, we
obtain $\varphi\left(  q\right)  \in\left\{  k+1,k+2,\ldots,\left\vert
\varphi\left(  E\right)  \right\vert \right\}  $. Hence, $q\in\varphi
^{-1}\left(  \left\{  k+1,k+2,\ldots,\left\vert \varphi\left(  E\right)
\right\vert \right\}  \right)  $.
\par
Now, let us forget that we fixed $q$. Thus we have proven that $q\in
\varphi^{-1}\left(  \left\{  k+1,k+2,\ldots,\left\vert \varphi\left(
E\right)  \right\vert \right\}  \right)  $ for every $q\in Q$. In other words,
$Q\subseteq\varphi^{-1}\left(  \left\{  k+1,k+2,\ldots,\left\vert
\varphi\left(  E\right)  \right\vert \right\}  \right)  $. Combining this with
$\varphi^{-1}\left(  \left\{  k+1,k+2,\ldots,\left\vert \varphi\left(
E\right)  \right\vert \right\}  \right)  \subseteq Q$, we obtain
$Q=\varphi^{-1}\left(  \left\{  k+1,k+2,\ldots,\left\vert \varphi\left(
E\right)  \right\vert \right\}  \right)  $, qed.}. Altogether, we thus know
that
\begin{align*}
P  &  =\varphi^{-1}\left(  \left\{  1,2,\ldots,k\right\}  \right)
,\ \ \ \ \ \ \ \ \ \ Q=\varphi^{-1}\left(  \left\{  k+1,k+2,\ldots,\left\vert
\varphi\left(  E\right)  \right\vert \right\}  \right)  ,\\
\sigma &  =\varphi\mid_{P}\ \ \ \ \ \ \ \ \ \ \text{and}%
\ \ \ \ \ \ \ \ \ \ \tau=\operatorname{add}_{-k}\circ\left(
\varphi\mid_{Q}\right)  .
\end{align*}
These equations are identical with the equations
(\ref{eq.lem.Gammaw.coprod.bij.def1}) and (\ref{eq.lem.Gammaw.coprod.bij.def2}%
) that were used in the definition of $\Phi\left(  \varphi,k\right)  $. Hence,
the definition of $\Phi$ shows that $\Phi\left(  \varphi,k\right)  =\left(
\left(  P,Q\right)  ,\sigma,\tau\right)  $. Thus, $\left(  \left(  P,Q\right)
,\sigma,\tau\right)  =\Phi\underbrace{\left(  \varphi,k\right)  }%
_{=\Psi\left(  \left(  P,Q\right)  ,\sigma,\tau\right)  }=\Phi\left(
\Psi\left(  \left(  P,Q\right)  ,\sigma,\tau\right)  \right)  $.

Now, let us forget that we fixed $\left(  \left(  P,Q\right)  ,\sigma
,\tau\right)  $. We thus have shown that \newline
$\Phi\left(  \Psi\left(  \left(
P,Q\right)  ,\sigma,\tau\right)  \right)  =\left(  \left(  P,Q\right)
,\sigma,\tau\right)  $ for every $\left(  \left(  P,Q\right)  ,\sigma
,\tau\right)  \in\mathcal{T}$. In other words, $\Phi\circ\Psi
=\id$. This proves Claim 1.

\textit{Proof of Claim 2:} Fix $\left(  \varphi,k\right)  \in\mathcal{S}$.
Define $P$, $Q$, $\sigma$ and $\tau$ by (\ref{eq.lem.Gammaw.coprod.bij.def1})
and (\ref{eq.lem.Gammaw.coprod.bij.def2}). The definition of $\Phi$ shows that
$\Phi\left(  \varphi,k\right)  =\left( \left(P,Q\right),\sigma,\tau\right)  $. From Lemma
\ref{lem.Gammaw.coprod.bij.a}, we know that $\left(  \left(  P,Q\right)
,\sigma,\tau\right)  \in\mathcal{T}$. In other words, we know that $\left(
P,Q\right)  \in \Adm \EE $, that $\sigma$ is a packed
$ \EE \mid_{P}$-partition, and that $\tau$ is a packed
$\EE\mid_{Q}$-partition.

From $\left(  P,Q\right)  \in \Adm \EE$, we conclude that
$P$ and $Q$ are subsets of $E$ satisfying $P\cap Q=\varnothing$ and $P\cup
Q=E$.

We have $\varphi\left(  P\right)  =\left\{  1,2,\ldots,k\right\}  $. (This was
proven in our proof of Lemma \ref{lem.Gammaw.coprod.bij.a} above; see the
equality (\ref{pf.lem.Gammaw.coprod.bij.a.3}).)

We have $\sigma=\varphi\mid_{P}$. Thus, for every $e\in P$, we have
$\sigma\left(  e\right)  =\left(  \varphi\mid_{P}\right)  \left(  e\right)
=\varphi\left(  e\right)  $. In other words, for every $e\in P$, we have
\begin{equation}
\varphi\left(  e\right)  =\sigma\left(  e\right)  .
\label{pf.lem.Gammaw.coprod.bij.c.11}%
\end{equation}
Also, $\tau=\operatorname{add}_{-k}\circ\left(  \varphi\mid
_{Q}\right)  $. Hence, for every $e\in Q$, we have%
\begin{align*}
\tau\left(  e\right)   &  =\left(  \operatorname{add}_{-k}%
\circ\left(  \varphi\mid_{Q}\right)  \right)  \left(  e\right)
=\operatorname{add}_{-k}\left(  \underbrace{\left(  \varphi\mid
_{Q}\right)  \left(  e\right)  }_{=\varphi\left(  e\right)  }\right) \\
&  =\operatorname{add}_{-k}\left(  \varphi\left(  e\right)  \right)
=\varphi\left(  e\right)  +\left(  -k\right)  \ \ \ \ \ \ \ \ \ \ \left(
\text{by the definition of }\operatorname{add}_{-k}\right) \\
&  =\varphi\left(  e\right)  -k.
\end{align*}
Thus, for every $e\in Q$, we have%
\[
\varphi\left(  e\right)  =\tau\left(  e\right)  +k.
\]
Combining this with (\ref{pf.lem.Gammaw.coprod.bij.c.11}), we conclude that%
\begin{equation}
\varphi\left(  e\right)  =%
\begin{cases}
\sigma\left(  e\right)  , & \text{if }e\in P;\\
\tau\left(  e\right)  +k, & \text{if }e\in Q
\end{cases}
\ \ \ \ \ \ \ \ \ \ \text{for every }e\in E
\label{pf.lem.Gammaw.coprod.bij.c.12}%
\end{equation}
\footnote{The right hand side of this equality makes sense because $P\cap
Q=\varnothing$ and $P\cup Q=E$.}.
Moreover, $\sigma : P \to \left\{ 1, 2, 3, \ldots \right\}$ is a
packed map and satisfies
$\left\vert \sigma\left(  P\right)  \right\vert
= \left\vert \sigma\left(  P\right)  \right\vert$. Thus,
Proposition~\ref{prop.ev.comp} (a) (applied to $P$, $\sigma$ and
$\left\vert \sigma\left(  P\right)  \right\vert$ instead of $E$, $\pi$
and $\ell$) yields
$\sigma\left(  P\right)  =\left\{  1,2,\ldots,\left\vert \sigma\left(
P\right)  \right\vert \right\}  $. Hence,%
\[
\left\{  1,2,\ldots,\left\vert \sigma\left(  P\right)  \right\vert \right\}
=\underbrace{\sigma}_{=\varphi\mid_{P}}\left(  P\right)  =\left(  \varphi
\mid_{P}\right)  \left(  P\right)  =\varphi\left(  P\right)  =\left\{
1,2,\ldots,k\right\}  .
\]
Thus, $\left\vert \sigma\left(  P\right)  \right\vert =k$.

So we know that $k=\left\vert \sigma\left(  P\right)  \right\vert $, and that
$\varphi$ is the map $E\rightarrow\left\{  1,2,3,\ldots\right\}  $ which sends
every $e\in E$ to $%
\begin{cases}
\sigma\left(  e\right)  , & \text{if }e\in P;\\
\tau\left(  e\right)  +k, & \text{if }e\in Q
\end{cases}
$ (because of (\ref{pf.lem.Gammaw.coprod.bij.c.12})). Thus, our $k$ and our
$\varphi$ are precisely the $k$ and the $\varphi$ in the definition of
$\Psi\left(  \left(  P,Q\right)  ,\sigma,\tau\right)  $. Hence, $\Psi\left(
\left(  P,Q\right)  ,\sigma,\tau\right)  =\left(  \varphi,k\right)  $. Thus,
$\left(  \varphi,k\right)  =\Psi\underbrace{\left(  \left(  P,Q\right)
,\sigma,\tau\right)  }_{=\Phi\left(  \varphi,k\right)  }=\Psi\left(
\Phi\left(  \varphi,k\right)  \right)  $.

Now, let us forget that we fixed $\left(  \varphi,k\right)  $. We thus have
shown that $\Psi\left(  \Phi\left(  \varphi,k\right)  \right)  =\left(
\varphi,k\right)  $ for every $\left(  \varphi,k\right)  \in\mathcal{S}$. In
other words, $\Psi\circ\Phi=\id$. This proves Claim 2.

Now, both Claims 1 and 2 are proven. Thus, the maps $\Phi$ and $\Psi$ are
mutually inverse. This proves Lemma \ref{lem.Gammaw.coprod.bij}.
\end{proof}

\begin{proof}
[Proof of Proposition~\ref{prop.Gammaw.coprod}.]Define $\mathcal{S}$,
$\mathcal{T}$, $\Phi$ and $\Psi$ as in Lemma \ref{lem.Gammaw.coprod.bij}. From
Lemma \ref{lem.Gammaw.coprod.bij}, we know that the maps $\Phi$ and $\Psi$ are
mutually inverse. Hence, $\Phi$ is a bijection from $\mathcal{S}$ to
$\mathcal{T}$.

Whenever $\alpha=\left(  \alpha_{1},\alpha_{2},\ldots,\alpha_{\ell}\right)  $
is a composition and $k\in\left\{  0,1,\ldots,\ell\right\}  $, we introduce
the notation $\alpha\left[  :k\right]  $ for the composition $\left(
\alpha_{1},\alpha_{2},\ldots,\alpha_{k}\right)  $, and the notation
$\alpha\left[  k:\right]  $ for the composition $\left(  \alpha_{k+1}%
,\alpha_{k+2},\ldots,\alpha_{\ell}\right)  $. Now, the formula
\eqref{eq.coproduct.M} can be rewritten as follows:
\begin{align}
\Delta\left(  M_{\alpha}\right)   &  =\sum_{k=0}^{\ell}M_{\alpha\left[
:k\right]  }\otimes M_{\alpha\left[  k:\right]  }%
\label{pf.prop.Gammaw.coprod.long.DeltaM}\\
&  \qquad\text{ for every }\ell\in \NN \text{ and every composition
}\alpha\text{ with }\ell\text{ entries.}\nonumber
\end{align}

Let us observe a simple fact: For any $\left(  \varphi,k\right)
\in\mathcal{S}$, we have%
\begin{equation}
\left(  \operatorname{ev}_{w}\varphi\right)  \left[  :k\right]
=\operatorname{ev}_{w\mid_{P}}\sigma\ \ \ \ \ \ \ \ \ \ \text{and}%
\ \ \ \ \ \ \ \ \ \ \left(  \operatorname{ev}_{w}\varphi\right)  \left[
k:\right]  =\operatorname{ev}_{w\mid_{Q}}\tau,
\label{pf.prop.Gammaw.coprod.evs}%
\end{equation}
where $\left(  \left(  P,Q\right)  ,\sigma,\tau\right)  =\Phi\left(
\varphi,k\right)  $\ \ \ \ \footnote{\textit{Proof of
(\ref{pf.prop.Gammaw.coprod.evs}):} Let $\left(  \varphi,k\right)
\in\mathcal{S}$. Let $\left(  \left(  P,Q\right)  ,\sigma,\tau\right)
=\Phi\left(  \varphi,k\right)  $. We must prove the equalities
(\ref{pf.prop.Gammaw.coprod.evs}).
\par
For every $\ell\in\mathbb{Z}$, define the map
$\operatorname{add}_{\ell}:\mathbb{Z}\rightarrow\mathbb{Z}$ as in
Lemma \ref{lem.Gammaw.coprod.bij.a}.
\par
Let $\ell=\left\vert \varphi\left(  E\right)  \right\vert $.
Thus, Proposition~\ref{prop.ev.comp} (a) (applied to $\varphi$
instead of $\pi$) yields
$\varphi\left(  E\right)  =\left\{  1,2,\ldots,\ell\right\}  $ (since
the map $\varphi : E \to \left\{ 1, 2, 3, \ldots \right\}$ is packed).
For each $i\in\left\{  1,2,\ldots,\ell\right\}  $,
define $\alpha_{i}\in \NN $ by
\begin{equation}
\alpha_{i}=\sum_{e\in\varphi^{-1}\left(  i\right)  }w\left(  e\right)  .
\label{pf.prop.Gammaw.coprod.evs.pf.alphai}%
\end{equation}
Then, $\operatorname{ev}_{w}\varphi=\left(  \alpha_{1},\alpha
_{2},\ldots,\alpha_{\ell}\right)  $ (by the definition of
$\operatorname{ev}_{w}\varphi$).
\par
We have $\left(  \varphi,k\right)  \in\mathcal{S}$. Thus, $\varphi$ is a
packed $\EE$-partition, and $k$ is an element of $\left\{
0,1,\ldots,\left\vert \varphi\left(  E\right)  \right\vert \right\}  $ (by the
definition of $\mathcal{S}$). Thus, $k\in\left\{  0,1,\ldots,\left\vert
\varphi\left(  E\right)  \right\vert \right\}  =\left\{  0,1,\ldots
,\ell\right\}  $ (since $\left\vert \varphi\left(  E\right)  \right\vert
=\ell$).
\par
Now, from $\operatorname{ev}_{w}\varphi=\left(  \alpha_{1}%
,\alpha_{2},\ldots,\alpha_{\ell}\right)  $, we obtain%
\[
\left(  \operatorname{ev}_{w}\varphi\right)  \left[  :k\right]  =\left(
\alpha_{1},\alpha_{2},\ldots,\alpha_{k}\right)  \ \ \ \ \ \ \ \ \ \ \text{and}%
\ \ \ \ \ \ \ \ \ \ \left(  \operatorname{ev}_{w}\varphi\right)  \left[
k:\right]  =\left(  \alpha_{k+1},\alpha_{k+2},\ldots,\alpha_{\ell}\right)  .
\]
\par
Recall that $\left(  \left(  P,Q\right)  ,\sigma,\tau\right)  =\Phi\left(
\varphi,k\right)  $. Hence, $P$, $Q$, $\sigma$ and $\tau$ are defined by
(\ref{eq.lem.Gammaw.coprod.bij.def1}) and (\ref{eq.lem.Gammaw.coprod.bij.def2}%
) (according to the definition of $\Phi$). We know (from Lemma
\ref{lem.Gammaw.coprod.bij.a}) that $\left(  \left(  P,Q\right)  ,\sigma
,\tau\right)  \in\mathcal{T}$. In other words, we know that $\left(
P,Q\right)  \in \Adm \EE $, that $\sigma$ is a packed
$ \EE \mid_{P}$-partition, and that $\tau$ is a packed
$\EE\mid_{Q}$-partition.
\par
For every $e\in P$, we have%
\begin{equation}
\underbrace{\sigma}_{\substack{=\varphi\mid_{P}\\\text{(by
(\ref{eq.lem.Gammaw.coprod.bij.def2}))}}}\left(  e\right)  =\left(
\varphi\mid_{P}\right)  \left(  e\right)  =\varphi\left(  e\right)  .
\label{pf.prop.Gammaw.coprod.evs.pf.1}%
\end{equation}
\par
For every $e\in Q$, we have%
\begin{align}
\underbrace{\tau}_{\substack{=\operatorname{add}_{-k}\circ\left(
\varphi\mid_{Q}\right)  \\\text{(by (\ref{eq.lem.Gammaw.coprod.bij.def2}))}%
}}\left(  e\right)   &  =\left(  \operatorname{add}_{-k}\circ\left(
\varphi\mid_{Q}\right)  \right)  \left(  e\right)
= \operatorname{add}_{-k}
\left(  \left(  \varphi\mid_{Q}\right)  \left(  e\right)
\right) \nonumber\\
&  =\underbrace{\left(  \varphi\mid_{Q}\right)  \left(  e\right)  }%
_{=\varphi\left(  e\right)  }+\left(  -k\right)  \ \ \ \ \ \ \ \ \ \ \left(
\text{by the definition of }\operatorname{add}_{-k}\right)
\nonumber\\
&  =\varphi\left(  e\right)  -k. \label{pf.prop.Gammaw.coprod.evs.pf.2}%
\end{align}
\par
For every $i\in\left\{  1,2,\ldots,k\right\}  $, we have%
\begin{align}
\sigma^{-1}\left(  i\right)   &  =\left\{  e\in P\ \mid\ \underbrace{\sigma
\left(  e\right)  }_{\substack{=\varphi\left(  e\right)  \\\text{(by
(\ref{pf.prop.Gammaw.coprod.evs.pf.1}))}}}=i\right\}  =\left\{  e\in
P\ \mid\ \varphi\left(  e\right)  =i\right\} \nonumber\\
&  =\underbrace{\left\{  e\in E\ \mid\ \varphi\left(  e\right)  =i\right\}
}_{=\varphi^{-1}\left(  i\right)  }\cap\underbrace{P}_{=\varphi^{-1}\left(
\left\{  1,2,\ldots,k\right\}  \right)  }\nonumber\\
&  =\varphi^{-1}\left(  i\right)  \cap\varphi^{-1}\left(  \left\{
1,2,\ldots,k\right\}  \right)  =\varphi^{-1}\left(  i\right)
\label{pf.prop.Gammaw.coprod.evs.pf.3}%
\end{align}
(since $\varphi^{-1}\left(  i\right)  \subseteq\varphi^{-1}\left(  \left\{
1,2,\ldots,k\right\}  \right)  $ (because $i\in\left\{  1,2,\ldots,k\right\}
$)).
\par
For every $i\in\left\{  1,2,\ldots,\ell-k\right\}  $, we have%
\begin{align}
\tau^{-1}\left(  i\right)   &  =\left\{  e\in Q\ \mid\ \underbrace{\tau\left(
e\right)  }_{\substack{=\varphi\left(  e\right)  -k\\\text{(by
(\ref{pf.prop.Gammaw.coprod.evs.pf.2}))}}}=i\right\}  =\left\{  e\in
Q\ \mid\ \underbrace{\varphi\left(  e\right)  -k=i}_{\Longleftrightarrow
\ \left(  \varphi\left(  e\right)  =k+i\right)  }\right\} \nonumber\\
&  =\left\{  e\in Q\ \mid\ \varphi\left(  e\right)  =k+i\right\}
=\underbrace{\left\{  e\in E\ \mid\ \varphi\left(  e\right)  =k+i\right\}
}_{=\varphi^{-1}\left(  k+i\right)  }\cap\underbrace{Q}_{\substack{=\varphi
^{-1}\left(  \left\{  k+1,k+2,\ldots,\left\vert \varphi\left(  E\right)
\right\vert \right\}  \right)  \\=\varphi^{-1}\left(  \left\{  k+1,k+2,\ldots
,\ell\right\}  \right)  \\\text{(since }\left\vert \varphi\left(  E\right)
\right\vert =\ell\text{)}}}\nonumber\\
&  =\varphi^{-1}\left(  k+i\right)  \cap\varphi^{-1}\left(  \left\{
k+1,k+2,\ldots,\ell\right\}  \right)  =\varphi^{-1}\left(  k+i\right)
\label{pf.prop.Gammaw.coprod.evs.pf.4}%
\end{align}
(since $\varphi^{-1}\left(  k+i\right)  \subseteq\varphi^{-1}\left(  \left\{
k+1,k+2,\ldots,\ell\right\}  \right)  $ (since $k+i\in\left\{  k+1,k+2,\ldots
,\ell\right\}  $ (since $i\in\left\{  1,2,\ldots,\ell-k\right\}  $))).
\par
We have $\varphi\left(  P\right)  =\left\{  1,2,\ldots,k\right\}  $. (This was
proven in our proof of Lemma \ref{lem.Gammaw.coprod.bij.a} above; see the
equality (\ref{pf.lem.Gammaw.coprod.bij.a.3}).) But $\sigma=\varphi\mid_{P}$,
so that $\sigma\left(  P\right)  =\left(  \varphi\mid_{P}\right)  \left(
P\right)  =\varphi\left(  P\right)  =\left\{  1,2,\ldots,k\right\}  $. Hence,
$\left\vert \sigma\left(  P\right)  \right\vert =\left\vert \left\{
1,2,\ldots,k\right\}  \right\vert =k$. Therefore, the definition of
$\operatorname{ev}_{w\mid_{P}}\sigma$ shows that $\operatorname{ev}_{w\mid
_{P}}\sigma=\left(  \beta_{1},\beta_{2},\ldots,\beta_{k}\right)  $, where each
$\beta_{i}$ is defined as $\sum_{e\in\sigma^{-1}\left(  i\right)  }\left(
w\mid_{P}\right)  \left(  e\right)  $. Thus, every $i\in\left\{
1,2,\ldots,k\right\}  $ satisfies%
\[
\beta_{i}=\underbrace{\sum_{e\in\sigma^{-1}\left(  i\right)  }}%
_{\substack{=\sum_{e\in\varphi^{-1}\left(  i\right)  }\\\text{(by
(\ref{pf.prop.Gammaw.coprod.evs.pf.3}))}}}\underbrace{\left(  w\mid
_{P}\right)  \left(  e\right)  }_{=w\left(  e\right)  }=\sum_{e\in\varphi
^{-1}\left(  i\right)  }w\left(  e\right)  =\alpha_{i}%
\ \ \ \ \ \ \ \ \ \ \left(  \text{by
(\ref{pf.prop.Gammaw.coprod.evs.pf.alphai})}\right)  .
\]
Hence, $\left(  \beta_{1},\beta_{2},\ldots,\beta_{k}\right)  =\left(
\alpha_{1},\alpha_{2},\ldots,\alpha_{k}\right)  =\left(  \operatorname{ev}%
_{w}\varphi\right)  \left[  :k\right]  $, so that $\left(  \operatorname{ev}%
_{w}\varphi\right)  \left[  :k\right]  =\left(  \beta_{1},\beta_{2}%
,\ldots,\beta_{k}\right)  =\operatorname{ev}_{w\mid_{P}}\sigma$.
\par
We have $\varphi\left(  Q\right)  =\left\{  k+1,k+2,\ldots,\left\vert
\varphi\left(  E\right)  \right\vert \right\}  $. (This was proven in our
proof of Lemma \ref{lem.Gammaw.coprod.bij.a} above; see the equality
(\ref{pf.lem.Gammaw.coprod.bij.a.4}).) But
\begin{align*}
\tau\left(  Q\right)   &  =\left\{  \underbrace{\tau\left(  e\right)
}_{\substack{=\varphi\left(  e\right)  -k\\\text{(by
(\ref{pf.prop.Gammaw.coprod.evs.pf.2}))}}}\ \mid\ e\in Q\right\}  =\left\{
\varphi\left(  e\right)  -k\ \mid\ e\in Q\right\}
 =\left\{  u-k\ \mid\ u\in\varphi\left(  Q\right)  \right\}  .
\end{align*}
Thus,
\begin{align*}
\left\vert \tau\left(  Q\right)  \right\vert  &  =\left\vert \varphi\left(
Q\right)  \right\vert =\left\vert \left\{  k+1,k+2,\ldots,\left\vert
\varphi\left(  E\right)  \right\vert \right\}  \right\vert
\ \ \ \ \ \ \ \ \ \ \left(  \text{since }\varphi\left(  Q\right)  =\left\{
k+1,k+2,\ldots,\left\vert \varphi\left(  E\right)  \right\vert \right\}
\right) \\
&  =\underbrace{\left\vert \varphi\left(  E\right)  \right\vert }_{=\ell
}-k=\ell-k.
\end{align*}
Therefore, the definition of $\operatorname{ev}_{w\mid_{Q}}\tau$ shows that
$\operatorname{ev}_{w\mid_{Q}}\tau=\left(  \gamma_{1},\gamma_{2},\ldots
,\gamma_{\ell-k}\right)  $, where each $\gamma_{i}$ is defined as $\sum
_{e\in\tau^{-1}\left(  i\right)  }\left(  w\mid_{Q}\right)  \left(  e\right)
$. Thus, every $i\in\left\{  1,2,\ldots,\ell-k\right\}  $ satisfies%
\begin{align*}
\gamma_{i}  &  =\underbrace{\sum_{e\in\tau^{-1}\left(  i\right)  }%
}_{\substack{=\sum_{e\in\varphi^{-1}\left(  k+i\right)  }\\\text{(by
(\ref{pf.prop.Gammaw.coprod.evs.pf.4}))}}}\underbrace{\left(  w\mid
_{Q}\right)  \left(  e\right)  }_{=w\left(  e\right)  }=\sum_{e\in\varphi
^{-1}\left(  k+i\right)  }w\left(  e\right)  =\alpha_{k+i}\\
&  \ \ \ \ \ \ \ \ \ \ \left(
\begin{array}
[c]{c}%
\text{since (\ref{pf.prop.Gammaw.coprod.evs.pf.alphai}) (applied to }k+i\text{
instead of }i\text{)}\\
\text{shows that }\alpha_{k+i}=\sum_{e\in\varphi^{-1}\left(  k+i\right)
}w\left(  e\right)
\end{array}
\right)  .
\end{align*}
Hence, $\left(  \gamma_{1},\gamma_{2},\ldots,\gamma_{\ell-k}\right)  =\left(
\alpha_{k+1},\alpha_{k+2},\ldots,\alpha_{\ell}\right)  =\left(
\operatorname{ev}_{w}\varphi\right)  \left[  k:\right]  $, so that $\left(
\operatorname{ev}_{w}\varphi\right)  \left[  k:\right]  =\left(  \gamma
_{1},\gamma_{2},\ldots,\gamma_{\ell-k}\right)  =\operatorname{ev}_{w\mid_{Q}%
}\tau$.
\par
Thus, both parts of (\ref{pf.prop.Gammaw.coprod.evs}) are proven.}.

If $\varphi : E \to \left\{1,2,3,\ldots\right\}$ is any packed map, then
$\ev_w \varphi$ is a composition with $\left| \varphi\left(E\right) \right|$
entries (by the definition of $\ev_w \varphi$), and thus it satisfies
\begin{equation}
\Delta \left( M_{\ev_w \varphi} \right)
= \sum_{k=0}^{\left\vert \varphi\left(  E\right)  \right\vert }
  M_{\left(  \operatorname{ev}_{w}\varphi\right)  \left[  :k\right] }
  \otimes
  M_{\left(  \operatorname{ev}_{w}\varphi\right)  \left[  k:\right] }
\label{pf.prop.Gammaw.coprod.long.lhs-addend}
\end{equation}
(by \eqref{pf.prop.Gammaw.coprod.long.DeltaM}, applied to
$\alpha = \ev_w \varphi$ and $\ell = \left| \varphi\left(E\right) \right|$).

Now, applying $\Delta$ to the equality \eqref{eq.prop.Gammaw.packed} yields
\begin{align}
\Delta \left( \Gamma \left( \EE, w \right) \right)
&= \Delta \left( \sum_{\varphi \text{ is a packed } \EE\text{-partition}}
  M_{\ev_w \varphi} \right) \nonumber \\
&= \sum_{\varphi \text{ is a packed } \EE\text{-partition}}
  \underbrace{\Delta \left( M_{\ev_w \varphi} \right)}_{\substack{
    =\sum_{k=0}^{\left\vert \varphi\left(  E\right)  \right\vert }M_{\left(
    \operatorname{ev}_{w}\varphi\right)  \left[  :k\right]  }\otimes M_{\left(
    \operatorname{ev}_{w}\varphi\right)  \left[  k:\right]  }\\
    \text{(by \eqref{pf.prop.Gammaw.coprod.long.lhs-addend})}}}\nonumber\\
&  =\underbrace{\sum_{\varphi\text{ is a packed } \EE \text{-partition}%
}\sum_{k=0}^{\left\vert \varphi\left(  E\right)  \right\vert }}%
_{\substack{=\sum_{\left(  \varphi,k\right)  \in\mathcal{S}}\\\text{(by the
definition of }\mathcal{S}\text{)}}}M_{\left(  \operatorname{ev}_{w}%
\varphi\right)  \left[  :k\right]  }\otimes M_{\left(  \operatorname{ev}%
_{w}\varphi\right)  \left[  k:\right]  }\nonumber\\
&  =\sum_{\left(  \varphi,k\right)  \in\mathcal{S}}M_{\left(
\operatorname{ev}_{w}\varphi\right)  \left[  :k\right]  }\otimes M_{\left(
\operatorname{ev}_{w}\varphi\right)  \left[  k:\right]  }%
.\label{pf.Gammaw.coprod.long.lhs}%
\end{align}

On the other hand, rewriting each of the tensorands on the right hand side of
\eqref{eq.prop.Gammaw.coprod} using \eqref{eq.prop.Gammaw.packed}, we obtain
\begin{align*}
&  \sum_{\left(  P,Q\right)  \in \Adm \EE }%
\underbrace{\Gamma\left(   \EE \mid_{P},w\mid_{P}\right)
}_{\substack{=\sum_{\varphi\text{ is a packed } \EE \mid_{P}%
\text{-partition}}M_{\operatorname{ev}_{w\mid_{P}}\varphi}\\\text{(by
\eqref{eq.prop.Gammaw.packed})}}}\otimes\underbrace{\Gamma\left(
\EE\mid_{Q},w\mid_{Q}\right)  }_{\substack{=\sum_{\varphi
\text{ is a packed } \EE \mid_{Q}\text{-partition}}
M_{\operatorname{ev}_{w\mid_{Q}} \varphi} \\
\text{(by \eqref{eq.prop.Gammaw.packed})}}}\\
&  =\sum_{\left(  P,Q\right)  \in \Adm \EE}
\underbrace{ \left(
 \sum_{\varphi\text{ is a packed } \EE \mid_{P}\text{-partition}%
  }M_{\operatorname{ev}_{w\mid_{P}}\varphi}\right)}
  _{\substack{
    = \sum_{\sigma\text{ is a packed } \EE \mid_{P}\text{-partition}%
    }M_{\operatorname{ev}_{w\mid_{P}}\sigma} \\
    \text{(here, we have renamed the } \\
    \text{summation index } \varphi \text{ as } \sigma \text{)}
    }}
\otimes
\underbrace{ \left(
 \sum_{\varphi\text{ is a packed } \EE \mid_{Q}\text{-partition}%
  }M_{\operatorname{ev}_{w\mid_{Q}}\varphi}\right)}
   _{\substack{
     = \sum_{\tau\text{ is a packed } \EE \mid_{Q}\text{-partition}}%
     M_{\operatorname{ev}_{w\mid_{Q}}\tau} \\
    \text{(here, we have renamed the } \\
    \text{summation index } \varphi \text{ as } \tau \text{)}
    }}
\\
&  =\sum_{\left(  P,Q\right)  \in \Adm \EE }\left(
\sum_{\sigma\text{ is a packed } \EE \mid_{P}\text{-partition}%
}M_{\operatorname{ev}_{w\mid_{P}}\sigma}\right)  \otimes\left(  \sum
_{\tau\text{ is a packed } \EE \mid_{Q}\text{-partition}}%
M_{\operatorname{ev}_{w\mid_{Q}}\tau}\right)  \\
&  =\underbrace{\sum_{\left(  P,Q\right)  \in \Adm \EE %
}\sum_{\sigma\text{ is a packed } \EE \mid_{P}\text{-partition}}%
\sum_{\tau\text{ is a packed } \EE \mid_{Q}\text{-partition}}%
}_{\substack{=\sum_{\left(  \left(  P,Q\right)  ,\sigma,\tau\right)
\in\mathcal{T}}\\\text{(by the definition of }\mathcal{T}\text{)}%
}}M_{\operatorname{ev}_{w\mid_{P}}\sigma}\otimes M_{\operatorname{ev}%
_{w\mid_{Q}}\tau}\\
&  =\sum_{\left(  \left(  P,Q\right)  ,\sigma,\tau\right)  \in\mathcal{T}%
}M_{\operatorname{ev}_{w\mid_{P}}\sigma}\otimes M_{\operatorname{ev}%
_{w\mid_{Q}}\tau}\\
&  =\sum_{\left(  \varphi,k\right)  \in\mathcal{S}}M_{\left(
\operatorname{ev}_{w}\varphi\right)  \left[  :k\right]  }\otimes M_{\left(
\operatorname{ev}_{w}\varphi\right)  \left[  k:\right]  }%
\end{align*}
(here, we have substituted $\Phi\left(  \varphi,k\right)  $ for $\left(
\left(  P,Q\right)  ,\sigma,\tau\right)  $ in the sum, using the fact that
$\Phi$ is a bijection from $\mathcal{S}$ to $\mathcal{T}$, and using the
equalities (\ref{pf.prop.Gammaw.coprod.evs}) to rewrite the addend
$M_{\operatorname{ev}_{w\mid_{P}}\sigma}\otimes M_{\operatorname{ev}%
_{w\mid_{Q}}\tau}$ as $M_{\left(  \operatorname{ev}_{w}\varphi\right)  \left[
:k\right]  }\otimes M_{\left(  \operatorname{ev}_{w}\varphi\right)  \left[
k:\right]  }$). Comparing this with (\ref{pf.Gammaw.coprod.long.lhs}), we
obtain%
\[
\Delta\left(  \Gamma\left(   \EE ,w\right)  \right)  =\sum_{\left(
P,Q\right)  \in \Adm \EE }\Gamma\left(   \EE %
\mid_{P},w\mid_{P}\right)  \otimes\Gamma\left(   \EE \mid_{Q},w\mid
_{Q}\right)  .
\]
This proves Proposition \ref{prop.Gammaw.coprod}.
\end{proof}
\end{verlong}

We note in passing that there is also a rule for multiplying
quasisymmetric functions of the form $\Gamma\left(\EE, w\right)$.
Namely, if $\EE$ and $\FF$ are two double posets and $u$ and $v$
are corresponding maps, then $\Gamma\left(\EE, u\right)
\Gamma\left(\FF, v\right) = \Gamma\left(\EE \FF, w\right)$ for a
map $w$ which is defined to be $u$ on the subset $\EE$ of
$\EE \FF$, and $v$ on the subset $\FF$ of $\EE \FF$. Here, $\EE \FF$
is a double poset defined as in \cite[\S 2.1]{Mal-Reu-DP}.
\begin{verlong}
\footnote{See the Appendix
(specifically, Corollary \ref{cor.djun.Gamma.EF}) for the exact
statement of this rule (and a proof).}
\end{verlong}
Combined with Proposition~\ref{prop.Gammaw.qsym}, this fact gives
a combinatorial proof for the fact that $\QSym$ is a $\kk$-algebra%
\begin{verlong}
\footnote{This combinatorial proof is shown in detail in the Appendix
(specifically, see the proof of Proposition~\ref{prop.QSym.alg}).
\par
Actually, we can do better: We can use these ideas to show that
$\QSym$ is a Hopf algebra.
See the proof of Proposition~\ref{prop.QSym.hopfalg} for how this
is done.}%
\end{verlong}
,
as well as for some standard formulas for multiplications of
quasisymmetric functions; similarly,
Proposition~\ref{prop.Gammaw.coprod} can be used to derive the
well-known formulas for $\Delta M_\alpha$, $\Delta L_\alpha$,
$\Delta s_{\lambda / \mu}$ etc. (although, of course, we have
already used the formula for $\Delta M_\alpha$ in our proof of
Proposition~\ref{prop.Gammaw.coprod}).

\begin{verlong}
Finally, let us state one more almost-trivial lemma that will
be used later:

\begin{lemma}
\label{lem.tertispecial.op}
Let $\left(  E,<_{1},<_{2}\right)  $ be a
tertispecial double poset. Let $>_{1}$ be the opposite relation of $<_{1}$.
Then, $\left(  E,>_{1},<_{2}\right)  $ is a tertispecial double poset.
\end{lemma}

\begin{proof}
[Proof of Lemma \ref{lem.tertispecial.op}.] The relations $<_{1}$ and $<_{2}$
are strict partial orders (since $\left(  E,<_{1},<_{2}\right)  $ is a double
poset). The relation $>_{1}$ is the opposite relation of $<_{1}$, and thus is
a strict partial order (since $<_{1}$ is a strict partial order). Now we know
that both relations $>_{1}$ and $<_{2}$ are strict partial orders on the set
$E$. Hence, $\left(  E,>_{1},<_{2}\right)  $ is a double poset. It remains to
prove that this double poset $\left(  E,>_{1},<_{2}\right)  $ is tertispecial.

We know that the double poset $\left(  E,<_{1},<_{2}\right)  $ is
tertispecial. In other words, the following statement holds:

\textit{Statement 1:} If $a$ and $b$ are two elements of $E$ such that $a$ is
$<_{1}$-covered by $b$, then $a$ and $b$ are $<_{2}$-comparable.

On the other hand, the following statement holds:

\textit{Statement 2:} Let $a$ and $b$ be two elements of $E$. Then, we have
the following logical equivalence:%
\[
\left(  a\text{ is }>_{1}\text{-covered by }b\right)  \ \Longleftrightarrow
\ \left(  b\text{ is }<_{1}\text{-covered by }a\right)  .
\]

[\textit{Proof of Statement 2:} We have the following chain of logical
equivalences:%
\begin{align}
& \ \left(  a\text{ is }>_{1}\text{-covered by }b\right)  \nonumber\\
& \Longleftrightarrow\ \left(  \text{we have }a>_{1}b\text{, and there exists
no }c\in E\text{ satisfying }\underbrace{a>_{1}c>_{1}b}_{\Longleftrightarrow
\ \left(  a>_{1}c\right)  \wedge\left(  c>_{1}b\right)  }\right)  \nonumber\\
& \ \ \ \ \ \ \ \ \ \ \left(  \text{by the definition of the notion
\textquotedblleft}>_{1}\text{-covered by\textquotedblright}\right)
\nonumber\\
& \Longleftrightarrow\ \left(  \text{we have }\underbrace{a>_{1}%
b}_{\substack{\Longleftrightarrow~\left(  b<_{1}a\right)  \\\text{(since
}>_{1}\text{ is}\\\text{the opposite}\\\text{relation}\\\text{of }<_{1}\text{)}%
}}\text{, and there exists no }c\in E\text{ satisfying }\underbrace{\left(
a>_{1}c\right)  }_{\substack{\Longleftrightarrow~\left(  c<_{1}a\right)
\\\text{(since }>_{1}\text{ is}\\\text{the opposite}\\\text{relation}\\\text{of }%
<_{1}\text{)}}}\wedge\underbrace{\left(  c>_{1}b\right)  }%
_{\substack{\Longleftrightarrow~\left(  b<_{1}c\right)  \\\text{(since }%
>_{1}\text{ is}\\\text{the opposite}\\\text{relation}\\\text{of }<_{1}\text{)}%
}}\right)  \nonumber\\
& \Longleftrightarrow\ \left(  \text{we have }b<_{1}a\text{, and there exists
no }c\in E\text{ satisfying }\underbrace{\left(  c<_{1}a\right)  \wedge\left(
b<_{1}c\right)  }_{\substack{\Longleftrightarrow\ \left(  b<_{1}c\right)
\wedge\left(  c<_{1}a\right)  \\\Longleftrightarrow\ \left(  b<_{1}%
c<_{1}a\right)  }}\right)  \nonumber\\
& \Longleftrightarrow\ \left(  \text{we have }b<_{1}a\text{, and there exists
no }c\in E\text{ satisfying }b<_{1}c<_{1}a\right)
.\label{pf.lem.tertispecial.op.5}%
\end{align}
On the other hand, we have the following chain of logical equivalences:%
\begin{align*}
& \ \left(  b\text{ is }<_{1}\text{-covered by }a\right)  \\
& \Longleftrightarrow\ \left(  \text{we have }b<_{1}a\text{, and there exists
no }c\in E\text{ satisfying }b<_{1}c<_{1}a\right)  \\
& \ \ \ \ \ \ \ \ \ \ \left(  \text{by the definition of the notion
\textquotedblleft}<_{1}\text{-covered by\textquotedblright}\right)  \\
& \Longleftrightarrow\ \left(  a\text{ is }>_{1}\text{-covered by }b\right)
\ \ \ \ \ \ \ \ \ \ \left(  \text{by (\ref{pf.lem.tertispecial.op.5})}\right)
.
\end{align*}
This proves Statement 2.]

Now, we shall prove the following statement:

\textit{Statement 3:} If $a$ and $b$ are two elements of $E$ such that $a$ is
$>_{1}$-covered by $b$, then $a$ and $b$ are $<_{2}$-comparable.

[\textit{Proof of Statement 3:} Let $a$ and $b$ be two elements of $E$ such
that $a$ is $>_{1}$-covered by $b$. We must show that $a$ and $b$ are $<_{2}$-comparable.

Statement 2 shows that we have the following logical equivalence:%
\[
\left(  a\text{ is }>_{1}\text{-covered by }b\right)  \ \Longleftrightarrow
\ \left(  b\text{ is }<_{1}\text{-covered by }a\right)  .
\]
Hence, $b$ is $<_{1}$-covered by $a$ (since $a$ is $>_{1}$-covered by $b$).
Thus, Statement 1 (applied to $b$ and $a$ instead of $a$ and $b$) yields that
$b$ and $a$ are $<_{2}$-comparable. In other words, either $b<_{2}a$ or $b=a$
or $a<_{2}b$. In other words, either $a<_{2}b$ or $b=a$ or $b<_{2}a$. In other
words, either $a<_{2}b$ or $a=b$ or $b<_{2}a$ (since $b=a$ is equivalent to
$a=b$). In other words, $a$ and $b$ are $<_{2}$-comparable. This proves
Statement 3.]

But the double poset $\left(  E,>_{1},<_{2}\right)  $ is tertispecial if and
only if Statement 3 holds (by the definition of \textquotedblleft
tertispecial\textquotedblright). Hence, the double poset $\left(
E,>_{1},<_{2}\right)  $ is tertispecial (since Statement 3 holds). This
completes the proof of Lemma~\ref{lem.tertispecial.op}.
\end{proof}
\end{verlong}

\section{Proof of Theorem~\ref{thm.antipode.Gammaw}}
\label{sect.proof}

Before we come to the proof of Theorem~\ref{thm.antipode.Gammaw},
let us state five lemmas:

\begin{lemma}
\label{lem.admissible.cover}
Let $\EE = \left(E, <_1, <_2\right)$ be a double poset.
Let $P$ and $Q$ be subsets of $E$ such that
$P \cap Q = \varnothing$ and $P \cup Q = E$.
Assume that there exist no $p \in P$ and $q \in Q$ such that
$q$ is $<_1$-covered by $p$. Then, $\left(P, Q\right) \in
\Adm \EE$.
\end{lemma}

\begin{proof}[Proof of Lemma~\ref{lem.admissible.cover}.]
For any $a \in E$ and $b \in E$, we let $\left[a, b\right]$
denote the subset \newline
$\left\{e \in E \mid a <_1 e <_1 b\right\}$ of $E$. It is
easy to see that if $a$, $b$ and $c$ are three elements of
$E$ satisfying $a <_1 c <_1 b$, then
both $\left[a, c\right]$ and $\left[c, b\right]$ are proper
subsets of $\left[a, b\right]$, and therefore
\begin{equation}
\text{both numbers }
\left|\left[a, c\right]\right| \text{ and }
\left|\left[c, b\right]\right| \text{ are smaller than }
\left|\left[a, b\right]\right|.
\label{pf.lem.admissible.cover.1}
\end{equation}

\begin{verlong}%
[\textit{Proof of \eqref{pf.lem.admissible.cover.1}:}
Let $a$, $b$ and $c$ be three elements of $E$
satisfying $a <_1 c <_1 b$.

The definition of $\left[a, b\right]$ yields
$\left[a, b\right]
= \left\{e \in E \mid a <_1 e <_1 b\right\}$.
Hence, $c \in \left[a, b\right]$ (since $a <_1 c <_1 b$).

The definition of $\left[a, c\right]$ yields
\[
\left[a, c\right]
= \left\{ e \in E \mid a <_1 e <_1 c \right\}
\subseteq \left\{ e \in E \mid a <_1 e <_1 b \right\}
\]
(because every $e \in E$ satisfying $a <_1 e <_1 c$
must also satisfy $e <_1 c <_1 b$ and therefore
$a <_1 e <_1 b$). Thus,
\[
\left[a, c\right]
\subseteq \left\{ e \in E \mid a <_1 e <_1 b \right\}
= \left[a, b\right] .
\]

If we had $\left[a, c\right] = \left[a, b\right]$,
then we would have $c \in \left[a, b\right]
= \left[a, c\right]
= \left\{ e \in E \mid a <_1 e <_1 c \right\}$
and therefore $a <_1 c <_1 c$; but this would contradict the
fact that we don't have $c <_1 c$. Thus, we cannot have
$\left[a, c\right] = \left[a, b\right]$. Thus, we have
$\left[a, c\right] \neq \left[a, b\right]$. Combining
this with $\left[a, c\right] \subseteq \left[a, b\right]$,
we conclude that $\left[a, c\right]$ is a proper
subset of $\left[a, b\right]$.

The definition of $\left[c, b\right]$ yields
\[
\left[c, b\right]
= \left\{ e \in E \mid c <_1 e <_1 b \right\}
\subseteq \left\{ e \in E \mid a <_1 e <_1 b \right\}
\]
(because every $e \in E$ satisfying $c <_1 e <_1 b$
must also satisfy $a <_1 c <_1 e$ and therefore
$a <_1 e <_1 b$). Thus,
\[
\left[c, b\right]
\subseteq \left\{ e \in E \mid a <_1 e <_1 b \right\}
= \left[a, b\right] .
\]

If we had $\left[c, b\right] = \left[a, b\right]$,
then we would have $c \in \left[a, b\right]
= \left[c, b\right]
= \left\{ e \in E \mid c <_1 e <_1 b \right\}$
and therefore $c <_1 c <_1 b$; but this would contradict the
fact that we don't have $c <_1 c$. Thus, we cannot have
$\left[c, b\right] = \left[a, b\right]$. Thus, we have
$\left[c, b\right] \neq \left[a, b\right]$. Combining
this with $\left[c, b\right] \subseteq \left[a, b\right]$,
we conclude that $\left[c, b\right]$ is a proper
subset of $\left[a, b\right]$.

Thus, we have shown that both $\left[a, c\right]$ and
$\left[c, b\right]$ are proper subsets of $\left[a, b\right]$.
Hence, \eqref{pf.lem.admissible.cover.1} follows
(since $\left[a, b\right]$ is a finite set). This completes
the proof of \eqref{pf.lem.admissible.cover.1}.]
\end{verlong}

A pair $\left(p, q\right) \in P \times Q$ is said to be a
\textit{malposition} if it satisfies $q <_1 p$. Now, let us
assume (for the sake of contradiction) that there exists a
malposition. Fix a malposition $\left(u, v\right)$ for which the
value $\left|\left[v, u\right]\right|$ is minimum. Thus,
$\left(u, v\right) \in P \times Q$ and $v <_1 u$.
From $\left(u, v\right) \in P \times Q$, we obtain
$u \in P$ and $v \in Q$. Hence, $v$ is not
$<_1$-covered by $u$ (since there exist no $p \in P$ and $q \in Q$
such that $q$ is $<_1$-covered by $p$). Hence, there exists a
$w \in E$ such that $v <_1 w <_1 u$ (since $v <_1 u$). Consider
this $w$. Applying \eqref{pf.lem.admissible.cover.1} to $a = v$,
$c = w$ and $b = u$, we see that both numbers
$\left|\left[v, w\right]\right|$ and
$\left|\left[w, u\right]\right|$ are smaller than
$\left|\left[v, u\right]\right|$, and therefore neither
$\left(w, v\right)$ nor $\left(u, w\right)$ is a malposition
(since we picked $\left(u, v\right)$ to be a malposition with
minimum $\left|\left[v, u\right]\right|$). But
$w \in E = P \cup Q$, so that either $w \in P$ or $w \in Q$.
If $w \in P$, then $\left(w, v\right)$ is a malposition;
if $w \in Q$, then $\left(u, w\right)$ is a malposition. In
either case, we obtain a contradiction to the fact that
neither $\left(w, v\right)$ nor $\left(u, w\right)$ is a malposition.
This contradiction shows that our assumption was wrong.
Hence, there exists no malposition. In other words, there
exists no $\left(p, q\right) \in P \times Q$ satisfying
$q <_1 p$ (since this is what ``malposition'' means). In
other words, no $p \in P$ and $q \in Q$ satisfy $q <_1 p$.
Consequently, $\left(P, Q\right) \in \Adm \EE$.
This proves Lemma~\ref{lem.admissible.cover}.
\end{proof}

\begin{lemma}
\label{lem.tertispecial.subset}
Let $\EE = \left(E, <_1, <_2\right)$ be a tertispecial
double poset.
Let $\left(P,Q\right) \in \Adm \EE$. Then, $\EE\mid_P$ is
a tertispecial double poset.
\end{lemma}

\begin{proof}[Proof of Lemma~\ref{lem.tertispecial.subset}.]
Recall that we are using the symbol $<_1$ to denote two
different relations:
a strict partial order on $E$, and its restriction to $P$.
This abuse of notation is usually harmless, but in the
current proof it is dangerous, because it causes the statement
``$a$ is $<_1$-covered by $b$'' (for two elements $a$ and
$b$ of $P$) to carry two meanings (depending on whether the
symbol $<_1$ is interpreted as the strict partial order on
$E$, or as its restriction to $P$). (These two meanings are
actually equivalent, but their equivalence is not immediately
obvious.)

Thus, for the duration of this proof, we shall revert to a
less ambiguous notation. Namely, the notation $<_1$ shall
only be used for the strict partial order on $E$ which
constitutes part of the double poset $\EE$. The restriction
of this partial order $<_1$ to the subset $P$ will be denoted
by $<_{1,P}$ (not by $<_1$). Similarly, the restriction of
the partial order $<_2$ to the subset $P$ will be denoted by
$<_{2,P}$ (not by $<_2$). Thus, the double poset
$\EE\mid_P$ is defined as
$\EE\mid_P = \left(P, <_{1,P}, <_{2,P}\right)$.

We need to show that the double poset
$\EE\mid_P = \left(P, <_{1,P}, <_{2,P}\right)$ is
tertispecial. In other words, we need to show that if $a$
and $b$ are two elements of $P$ such that $a$ is
$<_{1,P}$-covered by $b$, then $a$ and $b$
are $<_{2,P}$-comparable.

Let $a$ and $b$ be two elements of $P$ such that $a$ is
$<_{1,P}$-covered by $b$. Thus,
$a <_{1,P} b$, and
\begin{equation}
\text{there exists no } c \in P \text{ satisfying }
a <_{1,P} c <_{1,P} b .
\label{pf.lem.tertispecial.subset.1}
\end{equation}

We have $a <_{1,P} b$. In other words, $a <_1 b$ (since
$<_{1,P}$ is the restriction of the relation $<_1$ to $P$).

Now, if $c \in E$ is such that $a <_1 c <_1 b$, then $c$ must
belong to $P$\ \ \ \ \footnote{\textit{Proof.} Assume the
contrary. Thus, $c \notin P$. But
$\left(P, Q\right) \in \Adm \EE$. Thus, $P \cap Q = \varnothing$,
$P \cup Q = E$, and
\begin{equation}
\text{no } p \in P \text{ and } q \in Q \text{ satisfy }
q <_1 p .
\label{pf.lem.tertispecial.subset.2}
\end{equation}
From $c \in E$ and $c \notin P$, we obtain
$c \in E\setminus P \subseteq Q$ (since $P \cup Q = E$).
Applying \eqref{pf.lem.tertispecial.subset.2} to $p = b$ and
$q = c$, we thus conclude that we cannot have $c <_1 b$.
This contradicts $c <_1 b$. This contradiction shows that our
assumption was false, qed.}, and therefore satisfy
$a <_{1,P} c <_{1,P} b$\ \ \ \ \footnote{\textit{Proof.}
Let $c \in E$ be such that $a <_1 c <_1 b$. Then, $c$ must
belong to $P$ (as we have just proven). Now, $a <_1 c$. In
light of $a \in P$ and $c \in P$, this rewrites as
$a <_{1,P} c$ (since $<_{1,P}$ is the restriction of the
relation $<_1$ to $P$). Similarly, $c <_1 b$ rewrites as
$c <_{1,P} b$. Thus, $a <_{1,P} c <_{1,P} b$, qed.},
which entails a
contradiction to \eqref{pf.lem.tertispecial.subset.1}. Thus, there
is no $c \in E$ satisfying $a <_1 c <_1 b$. Therefore (and
because we have $a <_1 b$), we see that $a$ is $<_1$-covered
by $b$. Since $\EE$ is tertispecial,
this yields that $a$ and $b$ are $<_2$-comparable. In other
words, either $a <_2 b$ or $a = b$ or $b <_2 a$. Since $a$ and
$b$ both belong to $P$, we can rewrite this by replacing the
relation $<_2$ by its restriction $<_{2,P}$. We thus conclude
that either $a <_{2,P} b$ or $a = b$ or $b <_{2,P} a$. In other
words, $a$ and $b$ are $<_{2,P}$-comparable.

Now, forget that we fixed $a$ and $b$.
Thus, we have shown that if $a$ and $b$ are two elements
of $P$ such that $a$ is $<_{1,P}$-covered by $b$, then $a$ and
$b$ are $<_{2,P}$-comparable. This completes the proof of
Lemma~\ref{lem.tertispecial.subset}.

(We could similarly show that $\EE\mid_Q$ is a tertispecial
double poset; but we will not use this.)
\end{proof}

\begin{lemma}
\label{lem.Gammaw.empty}Let $\EE = \left( E, <_1, <_2 \right)$ be a
double poset.
Let $w : E \rightarrow \left\{ 1, 2, 3, \ldots \right\}$ be a map.

\begin{enumerate}
\item[(a)] If $E = \varnothing$, then $\Gamma \left( \EE , w \right) = 1$.

\item[(b)] If $E \neq \varnothing$, then
$\varepsilon \left( \Gamma \left( \EE, w \right) \right) = 0$.
\end{enumerate}
\end{lemma}

\begin{proof}
[Proof of Lemma \ref{lem.Gammaw.empty}.] (a) Part (a) is obvious (since there is
only one $\EE$-partition $\pi$ when $E=\varnothing$, and since
this $\EE$-partition $\pi$ satisfies $\xx_{\pi, w} = 1$).

(b) Observe that $\Gamma \left( \EE, w \right)$ is a homogeneous power
series of degree $\sum_{e\in E} w\left( e \right)$. When $E \neq \varnothing$,
this degree is $>0$ (since it is then a nonempty sum of positive integers),
and thus the power series $\Gamma \left( \EE, w \right)$
is annihilated by $\varepsilon$ (since $\varepsilon$ annihilates any
homogeneous power series in $\QSym$ whose degree is $> 0$).
\end{proof}

\begin{lemma}
\label{lem.Gammaw.toggle}
Let $\left(E, <_1, <_2\right)$ be a double poset. Let $>_1$ be the opposite
relation of $<_1$. Let $P$ and $Q$ be two
subsets of $E$ satisfying $P \cup Q = E$.
Let $\pi : E \to \left\{1, 2, 3, \ldots\right\}$ be a map such
that $\pi \mid_P$ is a $\left(P, >_1, <_2\right)$-partition.
Let $f \in P$. Assume that
\begin{equation}
\text{no } p \in P \text{ and } q \in Q \text{ satisfy }
q <_1 p.
\label{eq.lem.Gammaw.toggle.eq0}
\end{equation}
Also, assume that
\begin{equation}
\pi\left( f \right) \leq \pi\left( h \right)
\qquad \text{for every } h \in E .
\label{eq.lem.Gammaw.toggle.eq1}
\end{equation}
Furthermore, assume that
\begin{equation}
\pi \left( f \right) < \pi \left( h \right)
\qquad \text{for every } h \in E \text{ satisfying } h <_2 f .
\label{eq.lem.Gammaw.toggle.eq2}
\end{equation}

\begin{enumerate}
\item[(a)] If $p \in P \setminus \left\{f\right\}$
and $q \in Q \cup \left\{f\right\}$ are such that
$q <_1 p$, then we have neither $q <_2 p$ nor
$p <_2 q$.

\item[(b)] If $\pi \mid_Q$ is a $\left(Q, <_1, <_2\right)$-partition,
then
$\pi \mid_{Q \cup \left\{f\right\}}$ is a
$\left(Q \cup \left\{f\right\}, <_1, <_2\right)$-partition.
\end{enumerate}
\end{lemma}

\begin{proof}[Proof of Lemma~\ref{lem.Gammaw.toggle}.]From
$P \cup Q = E$, we obtain $\underbrace{E}_{= P \cup Q}
\setminus P = \left(P \cup Q\right) \setminus P \subseteq Q$.

(a) Let $p \in P \setminus \left\{f\right\}$
and $q \in Q \cup \left\{f\right\}$ be such that
$q <_1 p$. We must show that we have neither $q <_2 p$ nor $p <_2 q$.

Indeed, assume the contrary. Thus, we have
either $q <_2 p$ or $p <_2 q$.

We have $q <_{1} p$ and
$p\in P\setminus\left\{  f\right\}  \subseteq P$. Hence, if we had $q\in Q$,
then we would obtain a contradiction to
\eqref{eq.lem.Gammaw.toggle.eq0}. Hence, we cannot have $q\in Q$.
Therefore, $q=f$ (since $q\in Q\cup\left\{  f\right\}  $ but not $q\in Q$).
Hence, $f=q<_{1}p$, so that $p>_{1}f$. Therefore, $\pi\left(  p\right)
\leq\pi\left(  f\right)  $ (since $\pi\mid_{P}$ is a
$\left(  P,>_{1},<_{2}\right)  $-partition, and since both $f$
and $p$ belong to $P$).

Now, recall that we have either $q <_2 p$ or $p <_2 q$.
Since $q = f$, we can rewrite this as follows:
We have either $f <_2 p$ or $p <_2 f$.
But $p<_{2}f$ cannot hold (because if we had $p<_{2}f$, then
\eqref{eq.lem.Gammaw.toggle.eq2} (applied to $h=p$) would lead to
$\pi\left(  f\right)  <\pi\left(  p\right)  $, which would contradict
$\pi\left(  p\right)  \leq\pi\left(  f\right)  $).
Thus, we must have $f<_{2}p$.

But $\pi\mid_{P}$ is a $\left(  P,>_{1},<_{2}\right)  $-partition. Hence,
$\pi\left(  p\right)  <\pi\left(  f\right)  $ (since $p>_{1}f$ and $f<_{2}p$,
and since $p$ and $f$ both lie in $P$).
But \eqref{eq.lem.Gammaw.toggle.eq1} (applied to $h=p$) shows that
$\pi\left(  f\right)  \leq\pi\left(  p\right)  $. Hence, $\pi\left(  p\right)
<\pi\left(  f\right)  \leq\pi\left(  p\right)  $, a contradiction. Thus, our
assumption was wrong. This completes the proof of
Lemma~\ref{lem.Gammaw.toggle} (a).

(b) Assume that $\pi \mid_Q$ is a $\left(Q, <_1, <_2\right)$-partition. We need
to show that
$\pi \mid_{Q \cup \left\{f\right\}}$ is a
$\left(Q \cup \left\{f\right\}, <_1, <_2\right)$-partition.
In order
to prove this, we need to verify the following two claims:

\begin{statement}
\textit{Claim 1:} Every $a\in Q\cup\left\{  f\right\}  $ and $b\in
Q\cup\left\{  f\right\}  $ satisfying $a<_{1}b$ satisfy $\pi\left(  a\right)
\leq\pi\left(  b\right)  $.
\end{statement}

\begin{statement}
\textit{Claim 2:} Every $a\in Q\cup\left\{  f\right\}  $ and $b\in
Q\cup\left\{  f\right\}  $ satisfying $a<_{1}b$ and $b<_{2}a$ satisfy
$\pi\left(  a\right)  <\pi\left(  b\right)  $.
\end{statement}

\textit{Proof of Claim 1:} Let $a\in Q\cup\left\{  f\right\}  $ and $b\in
Q\cup\left\{  f\right\}  $ be such that $a<_{1}b$. We need to prove that
$\pi\left(  a\right)  \leq\pi\left(  b\right)  $. If $a=f$, then this follows
immediately from \eqref{eq.lem.Gammaw.toggle.eq1} (applied to $h=b$).
Hence, we WLOG assume that $a\neq f$. Thus, $a\in Q$ (since $a\in
Q\cup\left\{  f\right\}  $). Now, if $b\in P$, then $a<_{1}b$ contradicts
\eqref{eq.lem.Gammaw.toggle.eq0} (applied to $p=b$ and $q=a$). Hence,
we cannot have $b\in P$. Therefore, $b\in E\setminus P
\subseteq Q$. Thus, $\pi\left(  a\right)  \leq\pi\left(
b\right)  $ follows immediately from the fact that $\pi\mid_{Q}$ is a $\left(
Q,<_{1},<_{2}\right)  $-partition (since $a \in Q$ and $b \in Q$ and
$a <_1 b$).
This proves Claim 1.

\textit{Proof of Claim 2:} Let $a\in Q\cup\left\{  f\right\}  $ and $b\in
Q\cup\left\{  f\right\}  $ be such that $a<_{1}b$ and $b<_{2}a$. We need to
prove that $\pi\left(  a\right)  <\pi\left(  b\right)  $. If $a=f$, then this
follows immediately from \eqref{eq.lem.Gammaw.toggle.eq2} (applied
to $h=b$) (because if $a = f$, then $b <_2 a = f$).
Hence, we WLOG assume that $a\neq f$. Thus, $a\in Q$ (since $a\in
Q\cup\left\{  f\right\}  $). Now, if $b\in P$, then $a<_{1}b$ contradicts
\eqref{eq.lem.Gammaw.toggle.eq0} (applied to $p=b$ and $q=a$). Hence,
we cannot have $b\in P$. Therefore, $b\in E\setminus P \subseteq Q$.
Thus, $\pi\left(  a\right)  <\pi\left(
b\right)  $ follows immediately from the fact that $\pi\mid_{Q}$ is a $\left(
Q,<_{1},<_{2}\right)  $-partition (since $a \in Q$ and $b \in Q$ and
$a <_1 b$ and $b <_2 a$).
This proves Claim 2.

Now, both Claim 1 and Claim 2 are proven. As already said, this
completes the proof of Lemma~\ref{lem.Gammaw.toggle} (b).
\end{proof}

\begin{lemma}
\label{lem.Gammaw.altsum}
Let $\EE = \left(E, <_1, <_2\right)$ be a tertispecial double
poset satisfying $\left|E\right| > 0$.
Let $\pi : E \to \left\{ 1, 2, 3, \ldots \right\}$ be a map.
Let $>_1$ denote the opposite relation of $<_1$. Then,
\begin{equation}
\sum_{\substack{\left(P, Q\right) \in \Adm \EE ; \\
                \pi\mid_P \text{ is a }\left(P, >_1, <_2\right)\text{-partition;} \\
                \pi\mid_Q \text{ is a }\left(Q, <_1, <_2\right)\text{-partition}}}
\left(-1\right)^{\left|P\right|}
= 0 .
\label{pf.thm.antipode.Gammaw.signrev}
\end{equation}
\end{lemma}

\begin{proof}[Proof of Lemma~\ref{lem.Gammaw.altsum}.]
\begin{vershort}
Our goal is to prove \eqref{pf.thm.antipode.Gammaw.signrev}.
To do so, we denote by $Z$ the set of all
$\left(P, Q\right) \in \Adm \EE$ such that
$\pi\mid_P$ is a $\left(P, >_1, <_2\right)$-partition and
$\pi\mid_Q$ is a $\left(Q, <_1, <_2\right)$-partition. We are going
to define an involution $T : Z \to Z$ of the set $Z$ having the
following property:
\begin{statement}
\textit{Property P:} Let $\left(P, Q\right) \in Z$. If we write
$T\left(\left(P, Q\right)\right)$ in the form
$\left(P', Q'\right)$, then
$\left(-1\right)^{\left|P'\right|}
= - \left(-1\right)^{\left|P\right|}$.
\end{statement}
Once such an involution $T$
is found, it will be clear that it matches the addends on the left
hand side of \eqref{pf.thm.antipode.Gammaw.signrev} into pairs of
mutually cancelling addends\footnote{In fact, Property P
entails that $T$ has no fixed points. Therefore, to each addend
on the left
hand side of \eqref{pf.thm.antipode.Gammaw.signrev} corresponds an
addend with opposite sign, which cancels it: Namely, for each
$\left(A, B\right) \in Z$, the addend for
$\left(P, Q\right) = \left(A, B\right)$ is cancelled by the addend
for $\left(P, Q\right) = T\left(\left(A, B\right)\right)$.}, and so
\eqref{pf.thm.antipode.Gammaw.signrev}
will follow and we will be done. It thus remains to find $T$.
\end{vershort}

\begin{verlong}
Our goal is to prove \eqref{pf.thm.antipode.Gammaw.signrev}.
To do so, we denote by $Z$ the set of all
$\left(P, Q\right) \in \Adm \EE$ such that
$\pi\mid_P$ is a $\left(P, >_1, <_2\right)$-partition and
$\pi\mid_Q$ is a $\left(Q, <_1, <_2\right)$-partition. We are going
to define an involution $T : Z \to Z$ of the set $Z$ having the
following property:
\begin{statement}
\textit{Property P:} Let $\left(P, Q\right) \in Z$. If we write
$T\left(\left(P, Q\right)\right)$ in the form
$\left(P', Q'\right)$, then
$\left(-1\right)^{\left|P'\right|}
= - \left(-1\right)^{\left|P\right|}$.
\end{statement}
Once such an involution $T$
is found, the equality \eqref{pf.thm.antipode.Gammaw.signrev} will
follow\footnote{Here is the argument in detail:

Assume that we have found an involution $T:Z\rightarrow Z$ of the set $Z$
satisfying Property P. Consider this $T$. Then,
for any $\left(  P,Q\right)  \in Z$, if we write
$T\left(  \left(  P,Q\right)  \right)  $ in the form $\left(  P^{\prime
},Q^{\prime}\right)  $, then
\begin{equation}
\left(  -1\right)  ^{\left\vert P^{\prime}\right\vert }=-\left(  -1\right)
^{\left\vert P\right\vert }
\label{pf.thm.antipode.Gammaw.signrev.final.1}
\end{equation}
(because Property P is satisfied).
We now need to prove the equality \eqref{pf.thm.antipode.Gammaw.signrev}.

The map $T$ is an involution, and thus a bijection.

Now, let $Z_{0}$ be the subset $\left\{  \left(  P,Q\right)  \in
Z\ \mid\ \left\vert P\right\vert \text{ is even}\right\}  $ of $Z$. Thus, for
every $\left(  P,Q\right)  \in Z$, we have the following logical equivalence:%
\begin{equation}
\left(  \left(  P,Q\right)  \in Z_{0}\right)  \ \Longleftrightarrow\ \left(
\left\vert P\right\vert \text{ is even}\right)
.\label{pf.thm.antipode.Gammaw.signrev.final.defZ0}%
\end{equation}
Hence, for every $\left(  P,Q\right)  \in Z$, we have the following logical
equivalence:%
\begin{align}
\left(  \left(  P,Q\right)  \notin Z_{0}\right)  \   & \Longleftrightarrow
\ \left(  \left\vert P\right\vert \text{ is not even}\right)  \nonumber\\
& \ \ \ \ \ \ \ \ \ \ \left(
\begin{array}
[c]{c}%
\text{this equivalence is obtained from
(\ref{pf.thm.antipode.Gammaw.signrev.final.defZ0})}\\
\text{by replacing each part by its negation}%
\end{array}
\right)  \nonumber\\
& \Longleftrightarrow\ \left(  \left\vert P\right\vert \text{ is odd}\right)
.\label{pf.thm.antipode.Gammaw.signrev.final.defZ1}%
\end{align}

Now, for every $\left(  P,Q\right)  \in Z$, we have the following logical
equivalence:%
\begin{equation}
\left(  \left(  P,Q\right)  \in Z_{0}\right)  \ \Longleftrightarrow\ \left(
T\left(  \left(  P,Q\right)  \right)  \notin Z_{0}\right)
.\label{pf.thm.antipode.Gammaw.signrev.final.3}%
\end{equation}

\textit{Proof of (\ref{pf.thm.antipode.Gammaw.signrev.final.3}):} Let $\left(
P,Q\right)  \in Z$. Write $T\left(  \left(  P,Q\right)  \right)  \in Z$ in the
form $\left(  P^{\prime},Q^{\prime}\right)  $. Then,
(\ref{pf.thm.antipode.Gammaw.signrev.final.defZ1}) (applied to $\left(
P^{\prime},Q^{\prime}\right)  $ instead of $\left(  P,Q\right)  $) shows that
we have the following logical equivalence:%
\[
\left(  \left(  P^{\prime},Q^{\prime}\right)  \notin Z_{0}\right)
\ \Longleftrightarrow\ \left(  \left\vert P^{\prime}\right\vert \text{ is
odd}\right)  .
\]
Thus, we have the following logical equivalence:%
\begin{align*}
\left(  \left(  P^{\prime},Q^{\prime}\right)  \notin Z_{0}\right)  \   &
\Longleftrightarrow\ \left(  \left\vert P^{\prime}\right\vert \text{ is
odd}\right)  \ \Longleftrightarrow\ \left(  \underbrace{\left(  -1\right)
^{\left\vert P^{\prime}\right\vert }}_{\substack{=-\left(  -1\right)
^{\left\vert P\right\vert }\\\text{(by
(\ref{pf.thm.antipode.Gammaw.signrev.final.1}))}}}=-1\right)
\ \Longleftrightarrow\ \left(  -\left(  -1\right)  ^{\left\vert P\right\vert
}=-1\right)  \\
& \Longleftrightarrow\ \left(  \left(  -1\right)  ^{\left\vert P\right\vert
}=1\right)  \ \Longleftrightarrow\ \left(  \left\vert P\right\vert \text{ is
even}\right)  \ \Longleftrightarrow\ \left(  \left(  P,Q\right)  \in
Z_{0}\right)  \ \ \ \ \ \ \ \ \ \ \left(  \text{by
(\ref{pf.thm.antipode.Gammaw.signrev.final.defZ0})}\right)  .
\end{align*}
Hence, we have the following logical equivalence:%
\[
\left(  \left(  P,Q\right)  \in Z_{0}\right)  \ \Longleftrightarrow\ \left(
\underbrace{\left(  P^{\prime},Q^{\prime}\right)  }_{=T\left(  \left(
P,Q\right)  \right)  }\notin Z_{0}\right)  \ \Longleftrightarrow\ \left(
T\left(  \left(  P,Q\right)  \right)  \notin Z_{0}\right)  .
\]
This proves (\ref{pf.thm.antipode.Gammaw.signrev.final.3}).

Now,
\[
\sum_{\substack{\left(  P,Q\right)  \in Z;\\\left\vert P\right\vert \text{ is
even}}}\underbrace{\left(  -1\right)  ^{\left\vert P\right\vert }%
}_{\substack{=1\\\text{(since }\left\vert P\right\vert \text{ is even)}%
}}=\underbrace{\sum_{\substack{\left(  P,Q\right)  \in Z;\\\left\vert
P\right\vert \text{ is even}}}}_{\substack{=\sum_{\substack{\left(
P,Q\right)  \in Z;\\\left(  P,Q\right)  \in Z_{0}}}\\\text{(because for every
}\left(  P,Q\right)  \in Z\text{, the condition}\\\left(  \left\vert
P\right\vert \text{ is even}\right)  \text{ is equivalent to }\left(  \left(
P,Q\right)  \in Z_{0}\right)  \\\text{(by
(\ref{pf.thm.antipode.Gammaw.signrev.final.defZ0})))}}}1=\sum
_{\substack{\left(  P,Q\right)  \in Z;\\\left(  P,Q\right)  \in Z_{0}}}1
\]
and%
\begin{align*}
\sum_{\substack{\left(  P,Q\right)  \in Z;\\\left\vert P\right\vert \text{ is
odd}}}\underbrace{\left(  -1\right)  ^{\left\vert P\right\vert }%
}_{\substack{=-1\\\text{(since }\left\vert P\right\vert \text{ is odd)}}}  &
=\underbrace{\sum_{\substack{\left(  P,Q\right)  \in Z;\\\left\vert
P\right\vert \text{ is odd}}}}_{\substack{=\sum_{\substack{\left(  P,Q\right)
\in Z;\\\left(  P,Q\right)  \notin Z_{0}}}\\\text{(because for every }\left(
P,Q\right)  \in Z\text{, the condition}\\\left(  \left\vert P\right\vert
\text{ is odd}\right)  \text{ is equivalent to }\left(  \left(  P,Q\right)
\notin Z_{0}\right)  \\\text{(by
(\ref{pf.thm.antipode.Gammaw.signrev.final.defZ1})))}}}\left(  -1\right)
=\sum_{\substack{\left(  P,Q\right)  \in Z;\\\left(  P,Q\right)  \notin Z_{0}%
}}\left(  -1\right)  \\
& =\underbrace{\sum_{\substack{\left(  P,Q\right)  \in Z;\\T\left(  \left(
P,Q\right)  \right)  \notin Z_{0}}}}_{\substack{=\sum_{\substack{\left(
P,Q\right)  \in Z;\\\left(  P,Q\right)  \in Z_{0}}}\\\text{(because for every
}\left(  P,Q\right)  \in Z\text{, the condition}\\\left(  T\left(  \left(
P,Q\right)  \right)  \notin Z_{0}\right)  \text{ is equivalent to }\left(
\left(  P,Q\right)  \in Z_{0}\right)  \\\text{(by
(\ref{pf.thm.antipode.Gammaw.signrev.final.3})))}}}\left(  -1\right) \\
&\ \ \ \ \ \ \ \ \ \ \left(
\begin{array}
[c]{c}%
\text{here, we have substituted }T\left(  \left(  P,Q\right)  \right)  \text{
for }\left(  P,Q\right)  \\
\text{in the sum, since the map }T:Z\rightarrow Z\text{ is a bijection}%
\end{array}
\right)  \\
& =\sum_{\substack{\left(  P,Q\right)  \in Z;\\\left(  P,Q\right)  \in Z_{0}%
}}\left(  -1\right)  =-\sum_{\substack{\left(  P,Q\right)  \in Z;\\\left(
P,Q\right)  \in Z_{0}}}1.
\end{align*}
Finally,
\begin{align*}
\underbrace{\sum_{\substack{\left(  P,Q\right)  \in \Adm%
 \EE ;\\\pi\mid_{P}\text{ is a }\left(  P,>_{1},<_{2}\right)
\text{-partition;}\\\pi\mid_{Q}\text{ is a }\left(  Q,<_{1},<_{2}\right)
\text{-partition}}}}_{\substack{=\sum_{\left(  P,Q\right)  \in Z}\\\text{(by
the definition of }Z\text{)}}}\left(  -1\right)  ^{\left\vert P\right\vert
}
& =\sum_{\left(  P,Q\right)  \in Z}\left(  -1\right)  ^{\left\vert
P\right\vert }=\underbrace{\sum_{\substack{\left(  P,Q\right)  \in
Z;\\\left\vert P\right\vert \text{ is even}}}\left(  -1\right)  ^{\left\vert
P\right\vert }}_{=\sum_{\substack{\left(  P,Q\right)  \in Z;\\\left(
P,Q\right)  \in Z_{0}}}1}+\underbrace{\sum_{\substack{\left(  P,Q\right)  \in
Z;\\\left\vert P\right\vert \text{ is odd}}}\left(  -1\right)  ^{\left\vert
P\right\vert }}_{=-\sum_{\substack{\left(  P,Q\right)  \in Z;\\\left(
P,Q\right)  \in Z_{0}}}1}\\
& =\sum_{\substack{\left(  P,Q\right)  \in Z;\\\left(  P,Q\right)  \in Z_{0}%
}}1+\left(  -\sum_{\substack{\left(  P,Q\right)  \in Z;\\\left(  P,Q\right)
\in Z_{0}}}1\right)  =0.
\end{align*}
Thus, \eqref{pf.thm.antipode.Gammaw.signrev} is proven.
} and we will be done. It thus remains to find $T$.
\end{verlong}

The definition of the map $T : Z \to Z$
is simple (although it will take us a while to prove
that it is well-defined): Let $F$ be the subset of $E$ consisting of those
$e\in E$ for which the value $\pi \left( e \right)$ is minimum.
Then, $F$ is a nonempty
subposet\footnote{The nonemptiness of $F$ follows from the nonemptiness
of $E$ (which, in turn, follows from $\left|E\right| > 0$).}
of the poset $\left(  E,<_{2}\right)  $, and hence has a minimal
element\footnote{A \textit{minimal element} of a poset
$\left(P, \prec\right)$ is an element $p \in P$ such that no
$g \in P$ satisfies $g \prec p$. It is well-known that every nonempty
finite poset has at least one minimal element. We are using this fact
here.}
$f$ (that is, an element $f$ such that no $g\in F$ satisfies $g<_{2}
f$). Fix such an $f$. Now, the map $T$ sends a $\left(  P,Q\right)  \in Z$ to
$
\begin{cases}
\left(  P\cup\left\{  f\right\}  ,Q\setminus\left\{  f\right\}  \right)  , &
\text{if }f\notin P;\\
\left(  P\setminus\left\{  f\right\}  ,Q\cup\left\{  f\right\}  \right)  , &
\text{if }f\in P
\end{cases}
$.

In order to prove that the map $T$ is well-defined, we need to prove that its
output values all belong to $Z$. In other words, we need to prove that
\begin{equation}
\begin{cases}
\left(  P\cup\left\{  f\right\}  ,Q\setminus\left\{  f\right\}  \right)  , &
\text{if }f\notin P;\\
\left(  P\setminus\left\{  f\right\}  ,Q\cup\left\{  f\right\}  \right)  , &
\text{if }f\in P
\end{cases}
\in Z
\label{pf.thm.antipode.Gammaw.Zwd}
\end{equation}
for every $\left(  P,Q\right)  \in Z$.

\textit{Proof of \eqref{pf.thm.antipode.Gammaw.Zwd}:} Fix $\left(  P,Q\right)
\in Z$. Thus, $\left(  P,Q\right)  $ is an element of
$ \Adm \EE$ with the property that
$\pi\mid_{P}$ is a $\left(  P,>_{1},<_{2}\right)  $-partition
and $\pi\mid_{Q}$ is a $\left(  Q,<_{1},<_{2}\right)  $-partition
(by the definition of $Z$).

From $\left(  P,Q\right)  \in \Adm \EE$, we see that
$P\cap Q=\varnothing$ and $P\cup Q=E$, and furthermore that
\begin{equation}
\text{no }p\in P\text{ and }q\in Q\text{ satisfy }q<_{1}
p.
\label{pf.thm.antipode.Gammaw.Zwd.pf.Adm}
\end{equation}

We know that $f$ belongs to the set $F$, which is the subset of $E$ consisting
of those $e\in E$ for which the value $\pi \left( e \right)$
is minimum. Thus,
\begin{equation}
\pi\left(  f\right)  \leq\pi\left(  h\right)  \qquad\text{for every }h\in
E.
\label{pf.thm.antipode.Gammaw.Zwd.pf.min}
\end{equation}

Moreover,
\begin{equation}
\pi\left(  f\right)  <\pi\left(  h\right)
\qquad\text{for every } h\in E \text{ satisfying } h<_{2}f
\label{pf.thm.antipode.Gammaw.Zwd.pf.min2}
\end{equation}
\begin{vershort}
\footnote{\textit{Proof of \eqref{pf.thm.antipode.Gammaw.Zwd.pf.min2}:}
Let $h \in E$ be such that $h <_2 f$. We must prove
\eqref{pf.thm.antipode.Gammaw.Zwd.pf.min2}.
Indeed, assume the contrary. Thus,
$\pi \left( f \right) \geq \pi \left( h \right)$.
But every $g \in E$ satisfies
$\pi \left( f \right) \leq \pi \left( g \right)$
(by \eqref{pf.thm.antipode.Gammaw.Zwd.pf.min}, applied to
$g$ instead of $h$). Hence, every $g \in E$ satisfies
$\pi \left( g \right) \geq \pi \left( f \right)
\geq \pi \left( h \right)$. In other words,
$h$ is one of those $e\in E$ for which the value
$\pi \left( e \right)$ is minimum.
\par
But recall that $F$ is the subset of $E$ consisting
of those $e \in E$ for which the value $\pi \left( e \right)$
is minimum.
Since $h$ is one of these $e \in E$, we thus conclude that $h \in F$.
But $f$ is a minimal element
of the subposet $F$ of $\left(E, <_2\right)$.
In other words, no $g\in F$ satisfies $g<_{2}f$.
This contradicts the fact that $h\in F$ satisfies $h <_2 f$.
This contradiction proves that our assumption was wrong, qed.}.
\end{vershort}
\begin{verlong}
\footnote{\textit{Proof of \eqref{pf.thm.antipode.Gammaw.Zwd.pf.min2}:}
Let $h \in E$ be such that $h <_2 f$. We must prove
\eqref{pf.thm.antipode.Gammaw.Zwd.pf.min2}.
Indeed, assume the contrary. Thus,
$\pi \left( f \right) \geq \pi \left( h \right)$.
But every $g \in E$ satisfies
$\pi \left( f \right) \leq \pi \left( g \right)$
(by \eqref{pf.thm.antipode.Gammaw.Zwd.pf.min}, applied to
$g$ instead of $h$). Hence, every $g \in E$ satisfies
$\pi \left( g \right) \geq \pi \left( f \right)
\geq \pi \left( h \right)$. In other words,
$h$ is one of those $e\in E$ for which the value
$\pi \left( e \right)$ is minimum.
\par
But recall that $F$ is the subset of $E$ consisting
of those $e \in E$ for which the value $\pi \left( e \right)$
is minimum.
Since $h$ is one of these $e \in E$, we thus conclude that $h \in F$.
Recall also that $h <_2 f$.
Hence, there exists some $g \in F$ satisfying $g <_2 f$
(namely, $g = h$).
But $f$ is a minimal element
of the subposet $F$ of $\left(E, <_2\right)$.
In other words, no $g\in F$ satisfies $g<_{2}f$.
This contradicts the fact that there exists some $g \in F$ satisfying
$g <_2 f$.
This contradiction proves that our assumption was wrong, qed.}.
\end{verlong}

We need to prove \eqref{pf.thm.antipode.Gammaw.Zwd}. We are in one of the
following two cases:

\textit{Case 1:} We have $f\in P$.

\textit{Case 2:} We have $f\notin P$.

Let us first consider Case 1. In this case, we have $f\in P$.

Recall that $P\cap Q=\varnothing$ and $P\cup Q=E$. From this, we easily obtain
$\left(  P\setminus\left\{  f\right\}  \right)  \cap\left(  Q\cup\left\{
f\right\}  \right)  =\varnothing$ and $\left(  P\setminus\left\{  f\right\}
\right)  \cup\left(  Q\cup\left\{  f\right\}  \right)  =E$.

Furthermore, there exist no $p\in P\setminus\left\{  f\right\}  $ and $q\in
Q\cup\left\{  f\right\}  $ such that $q$ is $<_{1}$-covered by
$p$\ \ \ \ \footnote{\textit{Proof.}
Assume the contrary. Thus, there exist $p\in
P\setminus\left\{  f\right\}  $ and $q\in Q\cup\left\{  f\right\}  $ such that
$q$ is $<_{1}$-covered by $p$. Consider such $p$ and $q$.
\par
We know that $q$ is $<_{1}$-covered by $p$, and thus we have $q<_{1}p$.
Hence, Lemma~\ref{lem.Gammaw.toggle} (a) shows that
we have neither $q <_2 p$ nor $p <_2 q$.
On the other hand,
$q$ is $<_{1}$-covered by $p$. Hence, $q$ and $p$ are
$<_{2}$-comparable (since $\EE$ is tertispecial).
In other words, we have either $q <_2 p$ or $q = p$ or $p <_2 q$.
Hence, we must have $q = p$ (since we have neither $q <_2 p$
nor $p <_2 q$). But this contradicts $q <_1 p$.
This contradiction shows that our
assumption was wrong, qed.}. Hence, Lemma~\ref{lem.admissible.cover} (applied
to $P\setminus\left\{  f\right\}  $ and $Q\cup\left\{  f\right\}  $ instead of
$P$ and $Q$) shows that $\left(  P\setminus\left\{  f\right\}  ,Q\cup\left\{
f\right\}  \right)  \in \Adm \EE$.

Furthermore, $\pi\mid_{P}$ is a $\left(  P,>_{1},<_{2}\right)  $-partition.
Hence, $\pi\mid_{P\setminus\left\{  f\right\}  }$ is a $\left(
P\setminus\left\{  f\right\}  ,>_{1},<_{2}\right)  $-partition (since
$P\setminus\left\{  f\right\}  \subseteq P$).

Furthermore, $\pi\mid_{Q\cup\left\{  f\right\}  }$ is a $\left(  Q\cup\left\{
f\right\}  ,<_{1},<_{2}\right)  $-partition\footnote{This follows from
Lemma~\ref{lem.Gammaw.toggle} (b) (since $\pi\mid_Q$ is a
$\left(Q, <_1, <_2\right)$-partition).}.

Altogether, we now know that $\left(  P\setminus\left\{  f\right\}
,Q\cup\left\{  f\right\}  \right)  \in \Adm \EE$, that
$\pi\mid_{P\setminus\left\{  f\right\}  }$ is a $\left(  P\setminus\left\{
f\right\}  ,>_{1},<_{2}\right)  $-partition, and that $\pi\mid_{Q\cup\left\{
f\right\}  }$ is a $\left(  Q\cup\left\{  f\right\}  ,<_{1},<_{2}\right)
$-partition. In other words, $\left(  P\setminus\left\{  f\right\}
,Q\cup\left\{  f\right\}  \right)  \in Z$ (by the definition of $Z$). Thus,
\begin{align*}
\begin{cases}
\left(  P\cup\left\{  f\right\}  ,Q\setminus\left\{  f\right\}  \right)  , &
\text{if }f\notin P;\\
\left(  P\setminus\left\{  f\right\}  ,Q\cup\left\{  f\right\}  \right)  , &
\text{if }f\in P
\end{cases}
& =\left(  P\setminus\left\{  f\right\}  ,Q\cup\left\{  f\right\}  \right)
\ \ \ \ \ \ \ \ \ \ \left(  \text{since }f\in P\right)  \\
& \in Z.
\end{align*}
Hence, \eqref{pf.thm.antipode.Gammaw.Zwd} is proven in Case 1.

Let us next consider Case 2. In this case, we have $f\notin P$. Hence,
$f \in E \setminus P = Q$ (since
$P \cap Q = \varnothing$ and $P \cup Q = E$).

Recall that $P\cap Q=\varnothing$ and $P\cup Q=E$. From this, we easily obtain
$\left(  P\cup\left\{  f\right\}  \right)  \cap\left(  Q\setminus\left\{
f\right\}  \right)  =\varnothing$ and $\left(  P\cup\left\{  f\right\}
\right)  \cup\left(  Q\setminus\left\{  f\right\}  \right)  =E$.

We have $f \in Q$ and
$Q \cup P = P \cup Q = E$. Furthermore, $>_1$ is the
opposite relation of $<_1$, and thus is a strict partial
order (since $<_1$ is a strict partial order). Hence,
$\left(E, >_1, <_2\right)$ is a double poset. Furthermore,
the relation $<_1$ is the opposite relation of $>_1$ (since $>_1$
is the opposite relation of $<_1$).
The map
$\pi\mid_Q$ is a $\left(Q, <_1, <_2\right)$-partition. Moreover,
\begin{equation}
\text{no } p \in Q \text{ and } q \in P \text{ satisfy }
q >_1 p
\label{pf.thm.antipode.Gammaw.Zwd.pf.AdmQP}
\end{equation}
\footnote{\textit{Proof.} Let $a \in Q$ and $b \in P$ be such that
$b >_1 a$. We shall derive a contradiction.

We have $b >_1 a$. In other words, $a <_1 b$. Thus, $b \in P$ and
$a \in Q$ satisfy $a <_1 b$. This contradicts
\eqref{pf.thm.antipode.Gammaw.Zwd.pf.Adm} (applied to $p = b$ and
$q = a$).

Now, forget that we fixed $a$ and $b$. We thus have found a
contradiction for every $a \in Q$ and $b \in P$ satisfying
$b >_1 a$. Hence, no $a \in Q$ and $b \in P$ satisfy $b >_1 a$.
Renaming $a$ and $b$ as $p$ and $q$ in this statement, we obtain
the following: No $p \in Q$ and $q \in P$ satisfy $q >_1 p$.
This proves \eqref{pf.thm.antipode.Gammaw.Zwd.pf.AdmQP}.}.
Hence, we can apply Lemma~\ref{lem.Gammaw.toggle} to
$\left(E, >_1, <_2\right)$, $<_1$, $Q$ and $P$ instead of
$\left(E, <_1, <_2\right)$, $>_1$, $P$ and $Q$.

There exist no $p\in P\cup\left\{  f\right\}  $ and $q\in
Q\setminus\left\{  f\right\}  $ such that $q$ is $<_{1}$-covered by
$p$\ \ \ \ \footnote{\textit{Proof.} Assume the contrary. Thus, there exist
$p\in P\cup\left\{  f\right\}  $ and $q\in Q\setminus\left\{  f\right\}  $
such that $q$ is $<_{1}$-covered by $p$. Consider such $p$ and $q$.
\par
We know that $q$ is $<_1$-covered by $p$, and thus we have
$q <_1 p$. In other words, $p >_1 q$. Thus,
Lemma~\ref{lem.Gammaw.toggle} (a) (applied to
$\left(E, >_1, <_2\right)$, $<_1$, $Q$, $P$, $q$ and $p$ instead of
$\left(E, <_1, <_2\right)$, $>_1$, $P$, $Q$, $p$ and $q$) yields
that we have neither $p <_2 q$ nor $q <_2 p$.
On the other hand,
$q$ is $<_{1}$-covered by $p$. Hence, $q$ and $p$ are
$<_{2}$-comparable (since $\EE$ is tertispecial).
In other words, we have either $q <_2 p$ or $q = p$ or $p <_2 q$.
Hence, we must have $q = p$ (since we have neither $p <_2 q$
nor $q <_2 p$). But this contradicts $q <_1 p$.
This contradiction shows that our
assumption was wrong, qed.}. Hence, Lemma \ref{lem.admissible.cover} (applied
to $P\cup\left\{  f\right\}  $ and $Q\setminus\left\{  f\right\}  $ instead of
$P$ and $Q$) shows that $\left(  P\cup\left\{  f\right\}  ,Q\setminus\left\{
f\right\}  \right)  \in \Adm \EE$.

Furthermore, $\pi\mid_{Q}$ is a $\left(  Q,<_{1},<_{2}\right)  $-partition.
Hence, $\pi\mid_{Q\setminus\left\{  f\right\}  }$ is a $\left(
Q\setminus\left\{  f\right\}  ,<_{1},<_{2}\right)  $-partition (since
$Q\setminus\left\{  f\right\}  \subseteq Q$).

Furthermore, $\pi\mid_{P\cup\left\{  f\right\}  }$ is a $\left(  P\cup\left\{
f\right\}  ,>_{1},<_{2}\right)  $-partition\footnote{This follows from
Lemma~\ref{lem.Gammaw.toggle} (b) (applied to
$\left(E, >_1, <_2\right)$, $<_1$, $Q$ and $P$ instead of
$\left(E, <_1, <_2\right)$, $>_1$, $P$ and $Q$), since $\pi\mid_P$ is a
$\left(P, >_1, <_2\right)$-partition.}.

Altogether, we now know that $\left(  P\cup\left\{  f\right\}  ,Q\setminus
\left\{  f\right\}  \right)  \in \Adm \EE$, that $\pi
\mid_{P\cup\left\{  f\right\}  }$ is a $\left(  P\cup\left\{  f\right\}
,>_{1},<_{2}\right)  $-partition, and that $\pi\mid_{Q\setminus\left\{
f\right\}  }$ is a $\left(  Q\setminus\left\{  f\right\}  ,<_{1},<_{2}\right)
$-partition. In other words, $\left(  P\cup\left\{  f\right\}  ,Q\setminus
\left\{  f\right\}  \right)  \in Z$ (by the definition of $Z$). Thus,
\begin{align*}
\begin{cases}
\left(  P\cup\left\{  f\right\}  ,Q\setminus\left\{  f\right\}  \right)  , &
\text{if }f\notin P;\\
\left(  P\setminus\left\{  f\right\}  ,Q\cup\left\{  f\right\}  \right)  , &
\text{if }f\in P
\end{cases}
& =\left(  P\cup\left\{  f\right\}  ,Q\setminus\left\{  f\right\}  \right)
\ \ \ \ \ \ \ \ \ \ \left(  \text{since }f\notin P\right)  \\
& \in Z.
\end{align*}
Hence, \eqref{pf.thm.antipode.Gammaw.Zwd} is proven in Case 2.

We have now proven \eqref{pf.thm.antipode.Gammaw.Zwd} in both Cases 1 and 2.
Thus, \eqref{pf.thm.antipode.Gammaw.Zwd} always holds. In other words, the map
$T$ is well-defined.

\begin{vershort}
What the map $T$ does to a pair $\left(  P,Q\right)  \in Z$ can be described
as moving the element $f$ from the set where it resides (either $P$ or $Q$) to
the other set. Clearly, doing this twice gives us the original pair back.
Hence, the map $T$ is an involution. Furthermore, for any $\left(  P,Q\right)
\in Z$, if we write $T\left(  \left(  P,Q\right)  \right)  $ in the form
$\left(  P^{\prime},Q^{\prime}\right)  $, then $\left(  -1\right)
^{\left\vert P^{\prime}\right\vert }=-\left(  -1\right)  ^{\left\vert
P\right\vert }$ (because $P^{\prime}=
\begin{cases}
P\cup\left\{  f\right\}  , & \text{if }f\notin P;\\
P\setminus\left\{  f\right\}  , & \text{if }f\in P
\end{cases}
$ and thus
$\left|P^\prime\right| = \left|P\right| \pm 1$).
In other words, the involution $T$ satisfies Property P.
\end{vershort}
\begin{verlong}
Every $\alpha\in Z$ satisfies $\left(  T\circ T\right)  \left(  \alpha\right)
=\id\left(  \alpha\right)  $\ \ \ \ \footnote{\textit{Proof.}
Let $\alpha\in Z$. We want to show that $\left(  T\circ T\right)  \left(
\alpha\right)  = \id \left(  \alpha\right)  $.
\par
We have $\alpha\in Z$. In other words, $\alpha$ can be written in the form
$\alpha=\left(  P,Q\right)  $ for some $\left(  P,Q\right)  \in
\Adm \EE $ having the property that $\pi\mid_{P}$ is a
$\left(  P,>_{1},<_{2}\right)  $-partition and $\pi\mid_{Q}$ is a $\left(
Q,<_{1},<_{2}\right)  $-partition (by the definition of $Z$). Write $\alpha$
in this form.
\par
From $\left(  P,Q\right)  \in \Adm \EE $, we see that
$P\cap Q=\varnothing$ and $P\cup Q=E$, and furthermore that no $p\in P$ and
$q\in Q$ satisfy $q<_{1}p$. From $P\cap Q=\varnothing$ and $P\cup Q=E$, we
conclude that $P=E\setminus Q$ and $Q=E\setminus P$.
\par
We are in one of the following two cases:
\par
\textit{Case 1:} We have $f\in P$.
\par
\textit{Case 2:} We have $f\notin P$.
\par
Let us first consider Case 1. In this case, we have $f\in P$. Hence, $f\notin
E\setminus P=Q$. Clearly, $f\notin P\setminus\left\{  f\right\}  $ (since
$f\in\left\{  f\right\}  $) and $\left\{  f\right\}  \subseteq P$ (since $f\in
P$). Furthermore, the sets $Q$ and $\left\{  f\right\}  $ are disjoint (since
$f\notin Q$). Now,%
\begin{align*}
T\left(  \underbrace{\alpha}_{=\left(  P,Q\right)  }\right)    & =T\left(
\left(  P,Q\right)  \right)  =%
\begin{cases}
\left(  P\cup\left\{  f\right\}  ,Q\setminus\left\{  f\right\}  \right)  , &
\text{if }f\notin P;\\
\left(  P\setminus\left\{  f\right\}  ,Q\cup\left\{  f\right\}  \right)  , &
\text{if }f\in P
\end{cases}
\ \ \ \ \ \ \ \ \ \ \left(  \text{by the definition of }T\right)  \\
& =\left(  P\setminus\left\{  f\right\}  ,Q\cup\left\{  f\right\}  \right)
\ \ \ \ \ \ \ \ \ \ \left(  \text{since }f\in P\right)  .
\end{align*}
Now,%
\begin{align*}
\left(  T\circ T\right)  \left(  \alpha\right)    & =T\left(
\underbrace{T\left(  \alpha\right)  }_{=\left(  P\setminus\left\{  f\right\}
,Q\cup\left\{  f\right\}  \right)  }\right)  =T\left(  \left(  P\setminus
\left\{  f\right\}  ,Q\cup\left\{  f\right\}  \right)  \right)  \\
& =%
\begin{cases}
\left(  \left(  P\setminus\left\{  f\right\}  \right)  \cup\left\{  f\right\}
,\left(  Q\cup\left\{  f\right\}  \right)  \setminus\left\{  f\right\}
\right)  , & \text{if }f\notin P\setminus\left\{  f\right\}  ;\\
\left(  \left(  P\setminus\left\{  f\right\}  \right)  \setminus\left\{
f\right\}  ,\left(  Q\cup\left\{  f\right\}  \right)  \cup\left\{  f\right\}
\right)  , & \text{if }f\in P\setminus\left\{  f\right\}
\end{cases}
\ \ \ \ \ \ \ \ \ \ \left(  \text{by the definition of }T\right)  \\
& =\left(  \underbrace{\left(  P\setminus\left\{  f\right\}  \right)
\cup\left\{  f\right\}  }_{\substack{=P\\\text{(since }\left\{  f\right\}
\subseteq P\text{)}}},\underbrace{\left(  Q\cup\left\{  f\right\}  \right)
\setminus\left\{  f\right\}  }_{\substack{=Q\\\text{(since the sets }Q\text{
and }\left\{  f\right\}  \text{ are disjoint)}}}\right)
\ \ \ \ \ \ \ \ \ \ \left(  \text{since }f\notin P\setminus\left\{  f\right\}
\right)  \\
& =\left(  P,Q\right)  =\alpha = \id \left(  \alpha\right)  .
\end{align*}
Hence, $\left(  T\circ T\right)  \left(  \alpha\right)
= \id \left(  \alpha\right)  $ is proven in Case 1.
\par
Let us now consider Case 2. In this case, we have $f\notin P$. Hence, $f\in
E\setminus P=Q$. Clearly, $f\in\left\{  f\right\}  \subseteq P\cup\left\{
f\right\}  $. Also, $\left\{  f\right\}  \subseteq Q$ (since $f\in Q$).
Furthermore, the sets $P$ and $\left\{  f\right\}  $ are disjoint (since
$f\notin P$). Now,%
\begin{align*}
T\left(  \underbrace{\alpha}_{=\left(  P,Q\right)  }\right)    & =T\left(
\left(  P,Q\right)  \right)  =%
\begin{cases}
\left(  P\cup\left\{  f\right\}  ,Q\setminus\left\{  f\right\}  \right)  , &
\text{if }f\notin P;\\
\left(  P\setminus\left\{  f\right\}  ,Q\cup\left\{  f\right\}  \right)  , &
\text{if }f\in P
\end{cases}
\ \ \ \ \ \ \ \ \ \ \left(  \text{by the definition of }T\right)  \\
& =\left(  P\cup\left\{  f\right\}  ,Q\setminus\left\{  f\right\}  \right)
\ \ \ \ \ \ \ \ \ \ \left(  \text{since }f\notin P\right)  .
\end{align*}
Now,%
\begin{align*}
\left(  T\circ T\right)  \left(  \alpha\right)    & =T\left(
\underbrace{T\left(  \alpha\right)  }_{=\left(  P\cup\left\{  f\right\}
,Q\setminus\left\{  f\right\}  \right)  }\right)  =T\left(  \left(
P\cup\left\{  f\right\}  ,Q\setminus\left\{  f\right\}  \right)  \right)  \\
& =%
\begin{cases}
\left(  \left(  P\cup\left\{  f\right\}  \right)  \cup\left\{  f\right\}
,\left(  Q\setminus\left\{  f\right\}  \right)  \setminus\left\{  f\right\}
\right)  , & \text{if }f\notin P\cup\left\{  f\right\}  ;\\
\left(  \left(  P\cup\left\{  f\right\}  \right)  \setminus\left\{  f\right\}
,\left(  Q\setminus\left\{  f\right\}  \right)  \cup\left\{  f\right\}
\right)  , & \text{if }f\in P\cup\left\{  f\right\}
\end{cases}
\ \ \ \ \ \ \ \ \ \ \left(  \text{by the definition of }T\right)  \\
& =\left(  \underbrace{\left(  P\cup\left\{  f\right\}  \right)
\setminus\left\{  f\right\}  }_{\substack{=P\\\text{(since the sets }P\text{
and }\left\{  f\right\}  \text{ are disjoint)}}},\underbrace{\left(
Q\setminus\left\{  f\right\}  \right)  \cup\left\{  f\right\}  }%
_{\substack{=Q\\\text{(since }\left\{  f\right\}  \subseteq Q\text{)}%
}}\right)  \ \ \ \ \ \ \ \ \ \ \left(  \text{since }f\in P\cup\left\{
f\right\}  \right)  \\
& =\left(  P,Q\right)  =\alpha = \id \left(  \alpha\right)  .
\end{align*}
Hence, $\left(  T\circ T\right)  \left(  \alpha\right)
= \id \left(  \alpha\right)  $ is proven in Case 2.
\par
We have now proven $\left(  T\circ T\right)  \left(  \alpha\right)
= \id \left(  \alpha\right)  $ in both Cases 1 and 2. Thus,
$\left(  T\circ T\right)  \left(  \alpha\right)  = \id \left(
\alpha\right)  $ always holds. Qed.}. In other words, $T\circ
T = \id$. In other words, the map $T$ is an involution.
Furthermore, this involution $T$ satisfies Property
P\ \ \ \ \footnote{\textit{Proof.} Let $\left(  P,Q\right)  \in Z$. Write
$T\left(  \left(  P,Q\right)  \right)  $ in the form $\left(  P^{\prime
},Q^{\prime}\right)  $. Then, we must prove that $\left(  -1\right)
^{\left\vert P^{\prime}\right\vert }=-\left(  -1\right)  ^{\left\vert
P\right\vert }$.
\par
We are in one of the following two cases:
\par
\textit{Case 1:} We have $f\in P$.
\par
\textit{Case 2:} We have $f\notin P$.
\par
Let us first consider Case 1. In this case, we have $f\in P$. Now,%
\begin{align*}
\left(  P^{\prime},Q^{\prime}\right)    & =T\left(  \left(  P,Q\right)
\right)  =%
\begin{cases}
\left(  P\cup\left\{  f\right\}  ,Q\setminus\left\{  f\right\}  \right)  , &
\text{if }f\notin P;\\
\left(  P\setminus\left\{  f\right\}  ,Q\cup\left\{  f\right\}  \right)  , &
\text{if }f\in P
\end{cases}
\ \ \ \ \ \ \ \ \ \ \left(  \text{by the definition of }T\right)  \\
& =\left(  P\setminus\left\{  f\right\}  ,Q\cup\left\{  f\right\}  \right)
\ \ \ \ \ \ \ \ \ \ \left(  \text{since }f\in P\right)  .
\end{align*}
In other words, $P^{\prime}=P\setminus\left\{  f\right\}  $ and $Q^{\prime
}=Q\cup\left\{  f\right\}  $. Now, $\left\vert \underbrace{P^{\prime}%
}_{=P\setminus\left\{  f\right\}  }\right\vert =\left\vert P\setminus\left\{
f\right\}  \right\vert =\left\vert P\right\vert -1$ (since $f\in P$), and thus
$\left(  -1\right)  ^{\left\vert P^{\prime}\right\vert }=\left(  -1\right)
^{\left\vert P\right\vert -1}=-\left(  -1\right)  ^{\left\vert P\right\vert }%
$. Hence, $\left(  -1\right)  ^{\left\vert P^{\prime}\right\vert }=-\left(
-1\right)  ^{\left\vert P\right\vert }$ is proven in Case 1.
\par
Let us now consider Case 2. In this case, we have $f\notin P$. Now,%
\begin{align*}
\left(  P^{\prime},Q^{\prime}\right)    & =T\left(  \left(  P,Q\right)
\right)  =%
\begin{cases}
\left(  P\cup\left\{  f\right\}  ,Q\setminus\left\{  f\right\}  \right)  , &
\text{if }f\notin P;\\
\left(  P\setminus\left\{  f\right\}  ,Q\cup\left\{  f\right\}  \right)  , &
\text{if }f\in P
\end{cases}
\ \ \ \ \ \ \ \ \ \ \left(  \text{by the definition of }T\right)  \\
& =\left(  P\cup\left\{  f\right\}  ,Q\setminus\left\{  f\right\}  \right)
\ \ \ \ \ \ \ \ \ \ \left(  \text{since }f\notin P\right)  .
\end{align*}
In other words, $P^{\prime}=P\cup\left\{  f\right\}  $ and $Q^{\prime
}=Q\setminus\left\{  f\right\}  $. Now, $\left\vert \underbrace{P^{\prime}%
}_{=P\cup\left\{  f\right\}  }\right\vert =\left\vert P\cup\left\{  f\right\}
\right\vert =\left\vert P\right\vert +1$ (since $f\notin P$), and thus
$\left(  -1\right)  ^{\left\vert P^{\prime}\right\vert }=\left(  -1\right)
^{\left\vert P\right\vert +1}=-\left(  -1\right)  ^{\left\vert P\right\vert }%
$. Hence, $\left(  -1\right)  ^{\left\vert P^{\prime}\right\vert }=-\left(
-1\right)  ^{\left\vert P\right\vert }$ is proven in Case 2.
\par
We have now proven $\left(  -1\right)  ^{\left\vert P^{\prime}\right\vert
}=-\left(  -1\right)  ^{\left\vert P\right\vert }$ in both Cases 1 and 2.
Thus, $\left(  -1\right)  ^{\left\vert P^{\prime}\right\vert }=-\left(
-1\right)  ^{\left\vert P\right\vert }$ always holds. This completes the proof
of Property P.}. We thus have defined an involution $T:Z\rightarrow Z$ of the
set $Z$ satisfying Property P. This was precisely our goal.
\end{verlong}
As we have already
explained, this proves \eqref{pf.thm.antipode.Gammaw.signrev}. Hence,
Lemma~\ref{lem.Gammaw.altsum} is proven.
\end{proof}

\begin{proof}[Proof of Theorem~\ref{thm.antipode.Gammaw}.]
We shall
prove Theorem~\ref{thm.antipode.Gammaw} by strong induction over
$\left|E\right|$. The induction step proceeds as follows: Consider a
tertispecial double poset $\EE = \left(E, <_1, <_2\right)$ and
a map $w : E \to \left\{1, 2, 3, \ldots\right\}$, and
assume (as the induction hypothesis)
that Theorem~\ref{thm.antipode.Gammaw} is proven for all
tertispecial double posets of smaller size\footnote{The
\textit{size} of a double poset $\left(P, <_1, <_2\right)$
means the nonnegative integer $\left|P\right|$.}.
\begin{verlong}
More precisely:
Assume (as the induction hypothesis) that every tertispecial
double poset $\left(P, \prec_1, \prec_2\right)$ satisfying
$\left|P\right| < \left|E\right|$ and every map
$x : P \to \left\{1, 2, 3, \ldots\right\}$ satisfy
\begin{equation}
S\left(\Gamma\left(\left(P, \prec_1, \prec_2\right), x\right)\right)
= \left(-1\right)^{\left|P\right|}
\Gamma\left(\left(P, \succ_1, \prec_2\right), x\right) ,
\label{pf.thm.antipode.Gammaw.indhyp}
\end{equation}
where $\succ_1$ denotes the opposite relation of $\prec_1$.
\end{verlong}
Our goal is to show that \newline
$S\left(\Gamma\left(\left(E, <_1, <_2\right), w\right) \right)
= \left(-1\right)^{\left|E\right|}
\Gamma\left(\left(E, >_1, <_2\right), w\right)$.
Here, as usual, $>_1$ denotes the opposite relation of $<_1$.

If $E = \varnothing$, then this is easy\footnote{Hint:
If $E = \varnothing$, then both
$\Gamma\left(\left(E, <_1, <_2\right), w\right)$ and
$\Gamma\left(\left(E, >_1, <_2\right), w\right)$ are equal
to $1$ (by Lemma~\ref{lem.Gammaw.empty} (a)),
but the antipode $S$ satisfies $S\left(1\right) = 1$
and $\left(-1\right)^{\left|\varnothing\right|} = 1$.}.
Thus, we WLOG assume that $E \neq \varnothing$. Hence,
$\left| E \right| > 0$. Moreover,
Lemma~\ref{lem.Gammaw.empty} (b) shows that
$\varepsilon \left( \Gamma\left(\EE, w\right) \right) = 0$. Thus,
$\left( u \circ \varepsilon \right) \left( \Gamma \left( \EE, w \right) \right)
= u \left( \underbrace{\varepsilon \left( \Gamma \left( \EE, w \right) \right)}_{= 0} \right)
= u \left( 0 \right) = 0$.

\begin{verlong}
Using the induction hypothesis, we can see the following:
If $\left(P, Q\right) \in \Adm \EE$ is such that
$\left(P, Q\right) \neq \left(E, \varnothing\right)$,
then
\begin{equation}
S\left(\Gamma\left(\EE\mid_P, w\mid_P\right)\right)
= \left(-1\right)^{\left|P\right|}
\Gamma\left(\left(P, >_1, <_2\right), w\mid_P\right)
\label{pf.thm.antipode.Gammaw.indhyp-used}
\end{equation}
\footnote{\textit{Proof of
\eqref{pf.thm.antipode.Gammaw.indhyp-used}:}
Let $\left(P, Q\right) \in \Adm \EE$ be such that
$\left(P, Q\right) \neq \left(E, \varnothing\right)$.
From $\left(P, Q\right) \in \Adm \EE$, we conclude
that $P$ and $Q$ are subsets of $E$ satisfying
$P \cap Q = \varnothing$ and $P \cup Q = E$.
Hence, $Q = E \setminus P$.

The double poset $\EE\mid_P = \left(P, <_1, <_2\right)$
is tertispecial (by Lemma~\ref{lem.tertispecial.subset}).

If we had $P = E$, then we would have
$\left(\underbrace{P}_{= E}, \underbrace{Q}_{= E \setminus P}\right)
= \left(E, E \setminus \underbrace{P}_{= E}\right)
= \left(E, \underbrace{E \setminus E}_{= \varnothing}\right)
= \left(E, \varnothing\right)$, which would contradict
$\left(P, Q\right) \neq \left(E, \varnothing\right)$.
Hence, we cannot have $P = E$. Thus, $P$ is a proper
subset of $E$ (since $P$ is a subset of $E$). Hence,
$\left|P\right| < \left|E\right|$. Therefore,
\eqref{pf.thm.antipode.Gammaw.indhyp} (applied to
$\left(\prec_1\right) = \left(<_1\right)$,
$\left(\prec_2\right) = \left(<_2\right)$,
$\left(\succ_1\right) = \left(>_1\right)$,
and $x = w\mid_P$) yields
$S\left(\Gamma\left(\left(P, <_1, <_2\right), w\mid_P\right)\right)
= \left(-1\right)^{\left|P\right|}
\Gamma\left(\left(P, >_1, <_2\right), w\mid_P\right)$.

Now,
\[
S\left(\Gamma\left(\underbrace{\EE\mid_P}_{= \left(P, <_1, <_2\right)}, w\mid_P\right)\right)
= S\left(\Gamma\left(\left(P, <_1, <_2\right), w\mid_P\right)\right)
= \left(-1\right)^{\left|P\right|}
\Gamma\left(\left(P, >_1, <_2\right), w\mid_P\right) .
\]
This proves \eqref{pf.thm.antipode.Gammaw.indhyp-used}.}.
Furthermore, it is straightforward to see that
$\left(E, \varnothing\right) \in \Adm \EE$.
Notice that
\[
\Gamma\left(\underbrace{\EE\mid_\varnothing}_{= \left(\varnothing, <_1, <_2\right)},
            w\mid_\varnothing\right)
=\Gamma\left(\left(\varnothing, <_1, <_2\right), w\mid_\varnothing\right)
=1
\]
(by Lemma~\ref{lem.Gammaw.empty} (a)).
\end{verlong}

\begin{vershort}
The upper commutative pentagon of \eqref{eq.antipode} shows that
$u \circ \varepsilon = m \circ \left(S \otimes \id\right) \circ
\Delta$. Applying both sides of this equality to
$\Gamma\left(\EE, w\right)$, we obtain
$\left(u \circ \varepsilon\right)
\left(\Gamma\left(\EE, w\right)\right)
= \left(m \circ \left(S \otimes \id\right) \circ
\Delta\right) \left(\Gamma\left(\EE, w\right)\right)$.
Since
$\left(u \circ \varepsilon\right)
\left(\Gamma\left(\EE, w\right)\right) = 0$, this
becomes
\begin{align}
0
&= \left(m \circ \left(S \otimes \id\right) \circ
\Delta\right) \left(\Gamma\left(\EE, w\right)\right)
= m \left(\left(S \otimes \id\right) \left(
\Delta \left(\Gamma\left(\EE, w\right)\right)\right)\right)
\nonumber\\
&= m \left(\left(S \otimes \id\right) \left(
\sum_{\left(P, Q\right) \in \Adm \EE}
\Gamma\left(\EE\mid_P, w\mid_P\right)
\otimes \Gamma\left(\EE\mid_Q, w\mid_Q\right) \right) \right)
\qquad \left(\text{by \eqref{eq.prop.Gammaw.coprod}}\right)
\nonumber\\
&= m \left(\sum_{\left(P, Q\right) \in \Adm \EE}
S\left(\Gamma\left(\EE\mid_P, w\mid_P\right)\right)
\otimes
\Gamma\left(\EE\mid_Q, w\mid_Q\right)\right)
\nonumber\\
&= \sum_{\left(P, Q\right) \in \Adm \EE}
S\left(\Gamma\left(\EE\mid_P, w\mid_P\right)\right)
\Gamma\left(\EE\mid_Q, w\mid_Q\right)
\nonumber \\
&= S \left(\Gamma\left(\EE\mid_E, w\mid_E\right)\right)
\Gamma\left(\EE\mid_\varnothing, w\mid_\varnothing\right)
+ \sum_{\substack{\left(P, Q\right) \in \Adm \EE ; \\
                  \left|P\right| < \left|E\right|}}
S\left(\Gamma\left(\EE\mid_P, w\mid_P\right)\right)
\Gamma\left(\EE\mid_Q, w\mid_Q\right)
\label{pf.thm.antipode.Gammaw.Req.1}
\end{align}
(since the only pair $\left(P, Q\right) \in \Adm \EE$ satisfying
$\left|P\right| \geq \left|E\right|$ is $\left(E, \varnothing\right)$,
whereas all other pairs $\left(P, Q\right) \in \Adm \EE$
satisfy $\left|P\right| < \left|E\right|$).

But whenever $\left(P, Q\right) \in \Adm \EE$ is such that
$\left|P\right| < \left|E\right|$, the double poset
$\EE\mid_P = \left(P, <_1, <_2\right)$ is tertispecial
(by Lemma~\ref{lem.tertispecial.subset}), and
therefore we have
$S\left(\Gamma\left(\EE\mid_P, w\mid_P\right)\right)
= S\left(\Gamma\left(\left(P, <_1, <_2\right), w\mid_P\right)\right)
= \left(-1\right)^{\left|P\right|}
\Gamma\left(\left(P, >_1, <_2\right), w\mid_P\right)$
(by the induction hypothesis).
Hence,
\eqref{pf.thm.antipode.Gammaw.Req.1} becomes
\begin{align}
0
&= S \left(\Gamma\left(\underbrace{\EE\mid_E}_{=\EE},
           \underbrace{w\mid_E}_{=w}\right)\right)
  \underbrace{\Gamma\left(\EE\mid_\varnothing, w\mid_\varnothing\right)
             }_{\substack{
             =\Gamma\left(\left(\varnothing, <_1, <_2\right), w\mid_\varnothing\right)
             =1 \\ \text{ (by Lemma \ref{lem.Gammaw.empty} (a))}}}
     \nonumber \\
& \qquad + \sum_{\substack{\left(P, Q\right) \in \Adm \EE ; \\
                  \left|P\right| < \left|E\right|}}
  \underbrace{S\left(\Gamma\left(\EE\mid_P, w\mid_P\right)\right)
             }_{= \left(-1\right)^{\left|P\right|}
                  \Gamma\left(\left(P, >_1, <_2\right), w\mid_P\right)}
  \Gamma\left(\EE\mid_Q, w\mid_Q\right)
     \nonumber \\
& = S\left(\Gamma\left(\EE, w\right)\right)
  + \sum_{\substack{\left(P, Q\right) \in \Adm \EE ; \\
                  \left|P\right| < \left|E\right|}}
  \left(-1\right)^{\left|P\right|}
  \Gamma\left(\left(P, >_1, <_2\right), w\mid_P\right)
  \Gamma\left(\EE\mid_Q, w\mid_Q\right) .
\nonumber
\end{align}
Thus,
\begin{equation}
S\left(\Gamma\left(\EE, w\right)\right)
= - \sum_{\substack{\left(P, Q\right) \in \Adm \EE ; \\
                  \left|P\right| < \left|E\right|}}
\left(-1\right)^{\left|P\right|}
\Gamma\left(\left(P, >_1, <_2\right), w\mid_P\right)
\Gamma\left(\EE\mid_Q, w\mid_Q\right) .
\label{pf.thm.antipode.Gammaw.Seq}
\end{equation}
\end{vershort}

\begin{verlong}
The upper commutative pentagon of \eqref{eq.antipode} shows that
$u \circ \varepsilon = m \circ \left(S \otimes \id\right) \circ
\Delta$. Applying both sides of this equality to
$\Gamma\left(\EE, w\right)$, we obtain
$\left(u \circ \varepsilon\right)
\left(\Gamma\left(\EE, w\right)\right)
= \left(m \circ \left(S \otimes \id\right) \circ
\Delta\right) \left(\Gamma\left(\EE, w\right)\right)$.
Since
$\left(u \circ \varepsilon\right)
\left(\Gamma\left(\EE, w\right)\right) = 0$, this
becomes
\begin{align}
0
&= \left(m \circ \left(S \otimes \id\right) \circ
\Delta\right) \left(\Gamma\left(\EE, w\right)\right)
= m \left(\left(S \otimes \id\right) \left(
\Delta \left(\Gamma\left(\EE, w\right)\right)\right)\right)
\nonumber\\
&= m \left(\left(S \otimes \id\right) \left(
\sum_{\left(P, Q\right) \in \Adm \EE}
\Gamma\left(\EE\mid_P, w\mid_P\right)
\otimes \Gamma\left(\EE\mid_Q, w\mid_Q\right) \right) \right)
\qquad \left(\text{by \eqref{eq.prop.Gammaw.coprod}}\right)
\nonumber\\
&= m \left(\sum_{\left(P, Q\right) \in \Adm \EE}
S\left(\Gamma\left(\EE\mid_P, w\mid_P\right)\right)
\otimes
\Gamma\left(\EE\mid_Q, w\mid_Q\right)\right)
\nonumber\\
& \ \ \ \ \ \ \ \ \ \ \left(
  \text{by the definition of the map } S \otimes \id \right)
\nonumber \\
&= \sum_{\left(P, Q\right) \in \Adm \EE}
S\left(\Gamma\left(\EE\mid_P, w\mid_P\right)\right)
\Gamma\left(\EE\mid_Q, w\mid_Q\right)
\nonumber \\
& \ \ \ \ \ \ \ \ \ \ \left(
  \text{by the definition of the map } m \right)
\nonumber \\
&= S \left(\Gamma\left(\underbrace{\EE\mid_E}_{=\EE},
           \underbrace{w\mid_E}_{=w}\right)\right)
  \underbrace{\Gamma\left(\EE\mid_\varnothing, w\mid_\varnothing\right)
             }_{= 1}
     \nonumber \\
& \qquad
  + \sum_{\substack{\left(P, Q\right) \in \Adm \EE ; \\
                    \left(P, Q\right) \neq \left(E, \varnothing\right)}}
  \underbrace{S\left(\Gamma\left(\EE\mid_P, w\mid_P\right)\right)
             }_{\substack{
                = \left(-1\right)^{\left|P\right|}
                  \Gamma\left(\left(P, >_1, <_2\right), w\mid_P\right) \\
                \text{(by \eqref{pf.thm.antipode.Gammaw.indhyp-used})}}}
  \Gamma\left(\EE\mid_Q, w\mid_Q\right)
     \nonumber \\
& \ \ \ \ \ \ \ \ \ \ \left( \begin{array}{c}
  \text{here, we have split off the addend } \\
  \text{for } \left(P, Q\right) = \left(E, \varnothing\right)
  \text{ from the sum}
  \end{array} \right) \nonumber \\
& = S\left(\Gamma\left(\EE, w\right)\right)
  + \sum_{\substack{\left(P, Q\right) \in \Adm \EE ; \\
                    \left(P, Q\right) \neq \left(E, \varnothing\right)}}
  \left(-1\right)^{\left|P\right|}
  \Gamma\left(\left(P, >_1, <_2\right), w\mid_P\right)
  \Gamma\left(\EE\mid_Q, w\mid_Q\right) .
\nonumber
\end{align}
Thus,
\begin{equation}
S\left(\Gamma\left(\EE, w\right)\right)
= - \sum_{\substack{\left(P, Q\right) \in \Adm \EE ; \\
                    \left(P, Q\right) \neq \left(E, \varnothing\right)}}
\left(-1\right)^{\left|P\right|}
\Gamma\left(\left(P, >_1, <_2\right), w\mid_P\right)
\Gamma\left(\EE\mid_Q, w\mid_Q\right) .
\label{pf.thm.antipode.Gammaw.Seq}
\end{equation}
\end{verlong}

For every subset $P$ of $E$, we have

\begin{align}
\Gamma\left(\left(P, >_1, <_2\right), w\mid_P\right)
&= \sum_{\pi \text{ is a }\left(P, >_1, <_2\right)\text{-partition}}
\xx_{\pi, w\mid_P}
\nonumber \\
& \qquad \left(\text{by the definition of }
 \Gamma\left(\left(P, >_1, <_2\right), w\mid_P\right) \right)
\nonumber \\
& = \sum_{\sigma \text{ is a }\left(P, >_1, <_2\right)\text{-partition}}
\xx_{\sigma, w\mid_P}
\label{pf.thm.antipode.Gammaw.Req.pf.Gamma1}
\end{align}
(here, we have renamed the summation index $\pi$ as
$\sigma$).

For every subset $Q$ of $E$, we have

\begin{align}
\Gamma\left(\underbrace{\EE\mid_Q}_{=\left(Q, <_1, <_2\right)}, w\mid_Q\right)
&= \Gamma\left(\left(Q, <_1, <_2\right), w\mid_Q\right)
\nonumber 
= \sum_{\pi \text{ is a }\left(Q, <_1, <_2\right)\text{-partition}}
\xx_{\pi, w\mid_Q}
\nonumber \\
& \qquad \left(\text{by the definition of }
 \Gamma\left(\left(Q, <_1, <_2\right), w\mid_Q\right) \right)
\nonumber \\
& = \sum_{\tau \text{ is a }\left(Q, <_1, <_2\right)\text{-partition}}
\xx_{\tau, w\mid_Q}
\label{pf.thm.antipode.Gammaw.Req.pf.Gamma2}
\end{align}
(here, we have renamed the summation index $\pi$ as
$\tau$).

\begin{vershort}
Now,
\begin{align*}
& \sum_{\left(P, Q\right) \in \Adm \EE}
\left(-1\right)^{\left|P\right|}
\underbrace{\Gamma\left(\left(P, >_1, <_2\right), w\mid_P\right)}_{
 \substack{ = \sum_{\sigma \text{ is a }\left(P, >_1, <_2\right)\text{-partition}}
            \xx_{\sigma, w\mid_P} \\
            \text{(by \eqref{pf.thm.antipode.Gammaw.Req.pf.Gamma1})}}}
\underbrace{\Gamma\left(\EE\mid_Q, w\mid_Q\right)}_{
 \substack{ = \sum_{\tau \text{ is a }\left(Q, <_1, <_2\right)\text{-partition}}
            \xx_{\tau, w\mid_Q} \\
            \text{(by \eqref{pf.thm.antipode.Gammaw.Req.pf.Gamma2})}}}
\\
&= \sum_{\left(P, Q\right) \in \Adm \EE}
\left(-1\right)^{\left|P\right|}
\left(\sum_{\sigma \text{ is a }\left(P, >_1, <_2\right)\text{-partition}}
\xx_{\sigma, w\mid_P}\right)
\left(\sum_{\tau \text{ is a }\left(Q, <_1, <_2\right)\text{-partition}}
\xx_{\tau, w\mid_Q}\right) \\
&= \sum_{\left(P, Q\right) \in \Adm \EE}
\left(-1\right)^{\left|P\right|}
\sum_{\sigma \text{ is a }\left(P, >_1, <_2\right)\text{-partition}}
\sum_{\tau \text{ is a }\left(Q, <_1, <_2\right)\text{-partition}}
\xx_{\sigma, w\mid_P} \xx_{\tau, w\mid_Q} \\
&= \sum_{\left(P, Q\right) \in \Adm \EE}
\left(-1\right)^{\left|P\right|}
\sum_{\substack{\left(\sigma, \tau\right); \\
                \sigma : P \to \left\{1, 2, 3, \ldots\right\}; \\
                \tau : Q \to \left\{1, 2, 3, \ldots\right\}; \\
                \sigma \text{ is a }\left(P, >_1, <_2\right)\text{-partition;} \\
                \tau \text{ is a }\left(Q, <_1, <_2\right)\text{-partition}}}
\xx_{\sigma, w\mid_P} \xx_{\tau, w\mid_Q} \\
&= \sum_{\left(P, Q\right) \in \Adm \EE}
\left(-1\right)^{\left|P\right|}
\sum_{\substack{\pi : E \to \left\{1, 2, 3, \ldots\right\}; \\
                \pi\mid_P \text{ is a }\left(P, >_1, <_2\right)\text{-partition;} \\
                \pi\mid_Q \text{ is a }\left(Q, <_1, <_2\right)\text{-partition}}}
\underbrace{\xx_{\pi\mid_P, w\mid_P} \xx_{\pi\mid_Q, w\mid_Q}}_{=\xx_{\pi, w}} \\
& \qquad \left(
 \begin{array}{c}
 \text{here, we have substituted } \left(\pi\mid_P, \pi\mid_Q\right)
 \text{ for } \left(\sigma, \tau\right) \text{ in the inner sum,} \\
 \text{ since every pair } \left(\sigma, \tau\right)
 \text{ consisting of a map }
 \sigma : P \to \left\{1, 2, 3, \ldots\right\} \\
 \text{ and a map } \tau : Q \to \left\{1, 2, 3, \ldots\right\} \\
 \text{ can be written as } \left(\pi\mid_P, \pi\mid_Q\right)
 \text{ for a unique }
 \pi : E \to \left\{1, 2, 3, \ldots\right\} \\
 \text{(namely, for the }
 \pi : E \to \left\{1, 2, 3, \ldots\right\}
 \text{ that is defined to send every } \\
 e \in P \text{ to }
 \sigma\left(e\right) \text{ and to send every } e \in Q
 \text{ to } \tau\left(e\right) \text{)}
 \end{array}
 \right) \\
& = \sum_{\left(P, Q\right) \in \Adm \EE}
\left(-1\right)^{\left|P\right|}
\sum_{\substack{\pi : E \to \left\{1, 2, 3, \ldots\right\}; \\
                \pi\mid_P \text{ is a }\left(P, >_1, <_2\right)\text{-partition;} \\
                \pi\mid_Q \text{ is a }\left(Q, <_1, <_2\right)\text{-partition}}}
\xx_{\pi, w} \\
& = \sum_{\pi : E \to \left\{1, 2, 3, \ldots\right\}}
\underbrace{
\sum_{\substack{\left(P, Q\right) \in \Adm \EE ; \\
                \pi\mid_P \text{ is a }\left(P, >_1, <_2\right)\text{-partition;} \\
                \pi\mid_Q \text{ is a }\left(Q, <_1, <_2\right)\text{-partition}}}
\left(-1\right)^{\left|P\right|}
}
_{\substack{=0 \\
            \text{(by \eqref{pf.thm.antipode.Gammaw.signrev})}}}
\xx_{\pi, w}
= \sum_{\pi : E \to \left\{1, 2, 3, \ldots\right\}} 0 \xx_{\pi, w}
= 0.
\end{align*}
\end{vershort}
\begin{verlong}
Now, for each $\left(P, Q\right) \in \Adm \EE$, we have
\begin{equation}
\Gamma\left(\left(P, >_1, <_2\right), w\mid_P\right)
\Gamma\left(\EE\mid_Q, w\mid_Q\right)
= \sum_{\substack{\pi : E \to \left\{1, 2, 3, \ldots\right\}; \\
                \pi\mid_P \text{ is a }\left(P, >_1, <_2\right)\text{-partition;} \\
                \pi\mid_Q \text{ is a }\left(Q, <_1, <_2\right)\text{-partition}}}
\xx_{\pi, w}
\label{pf.thm.antipode.Gammaw.long.GG=sum}
\end{equation}
\footnote{
\textit{Proof of \eqref{pf.thm.antipode.Gammaw.long.GG=sum}.}
Let $\left(P, Q\right) \in \Adm \EE$. Thus, $P$ and $Q$ are
two subsets of $E$ satisfying $P \cap Q = \varnothing$ and
$P \cup Q = E$. Thus, the set $E$ is the union of its two
disjoint subsets $P$ and $Q$.

If $\pi:E\rightarrow\left\{  1,2,3,\ldots\right\}  $ is a map, then
\begin{align*}
\underbrace{\xx_{\pi\mid_{P},w\mid_{P}}}_{\substack{
=\prod_{e\in P}x_{\left(
\pi\mid_{P}\right)  \left(  e\right)  }^{\left(  w\mid_{P} \right)
\left(  e\right) }  \\
\text{(by the definition of } \xx_{\pi\mid_P, w\mid_P} \text{)}
}}
\underbrace{\xx_{\pi\mid_{Q},w\mid_{Q}}}_{\substack{
=\prod_{e\in Q}x_{\left(  \pi\mid_{Q}\right)  \left(  e\right)  }^{
\left(  w\mid_{Q} \right) \left(e\right) } \\
\text{(by the definition of } \xx_{\pi\mid_Q, w\mid_Q} \text{)}
}}
& =\left(  \prod_{e\in P}\underbrace{x_{\left(  \pi\mid_{P}\right)  \left(
e\right)  }^{\left(  w\mid_{P}  \right) \left(  e\right) }}_{\substack{=x_{\pi
\left(  e\right)  }^{w\left(  e\right)  }\\\text{(since }\left(  \pi\mid
_{P}\right)  \left(  e\right)  =\pi\left(  e\right)  \\\text{and }\left(
w\mid_{P} \right) \left(  e\right) =w\left(  e\right)  \text{)}}}\right)
\left(  \prod_{e\in Q}\underbrace{x_{\left(  \pi\mid_{Q}\right)  \left(
e\right)  }^{\left(  w\mid_{Q} \right) \left(  e\right) }}_{\substack{=x_{\pi
\left(  e\right)  }^{w\left(  e\right)  }\\\text{(since }\left(  \pi\mid
_{Q}\right)  \left(  e\right)  =\pi\left(  e\right)  \\\text{and }\left(
w\mid_{Q} \right) \left(  e\right)  =w\left(  e\right)  \text{)}}}\right)  \\
& =\left(  \prod_{e\in P}x_{\pi\left(  e\right)  }^{w\left(  e\right)
}\right)  \left(  \prod_{e\in Q}x_{\pi\left(  e\right)  }^{w\left(  e\right)
}\right)  =\prod_{e\in E}x_{\pi\left(  e\right)  }^{w\left(  e\right)  }%
\end{align*}
(here, we have merged the two products, since the set $E$ is the union of its
two disjoint subsets $P$ and $Q$).

But the set $E$ is the union of its two disjoint subsets $P$ and $Q$. Hence,
every pair $\left(  \sigma,\tau\right)  $ consisting of a map $\sigma
:P\rightarrow\left\{  1,2,3,\ldots\right\}  $ and a map $\tau:Q\rightarrow
\left\{  1,2,3,\ldots\right\}  $ can be written as $\left(  \pi\mid_{P}%
,\pi\mid_{Q}\right)  $ for a unique $\pi:E\rightarrow\left\{  1,2,3,\ldots
\right\}  $ (namely, for the $\pi:E\rightarrow\left\{  1,2,3,\ldots\right\}  $
that sends every $e \in E$ to
$\begin{cases}
\sigma\left(e\right), & \text{if } e \in P \text{;}\\
\tau\left(e\right), & \text{if } e \in Q
\end{cases}$). Hence, we can substitute
$\left(  \pi\mid_{P},\pi\mid_{Q}\right)  $ for $\left(  \sigma,\tau\right)  $
in the sum $\sum_{\substack{\left(\sigma, \tau\right); \\
                \sigma : P \to \left\{1, 2, 3, \ldots\right\}; \\
                \tau : Q \to \left\{1, 2, 3, \ldots\right\}; \\
                \sigma \text{ is a }\left(P, >_1, <_2\right)\text{-partition;} \\
                \tau \text{ is a }\left(Q, <_1, <_2\right)\text{-partition}}}
\xx_{\sigma, w\mid_P} \xx_{\tau, w\mid_Q}
$. We thus obtain
\begin{align*}
\sum_{\substack{\left(\sigma, \tau\right); \\
                \sigma : P \to \left\{1, 2, 3, \ldots\right\}; \\
                \tau : Q \to \left\{1, 2, 3, \ldots\right\}; \\
                \sigma \text{ is a }\left(P, >_1, <_2\right)\text{-partition;} \\
                \tau \text{ is a }\left(Q, <_1, <_2\right)\text{-partition}}}
\xx_{\sigma, w\mid_P} \xx_{\tau, w\mid_Q}
& =\sum_{\substack{\pi:E\rightarrow\left\{  1,2,3,\ldots\right\} ; \\\pi
\mid_{P}\text{ is a }\left(  P,>_{1},<_{2}\right)  \text{-partition;}\\\pi
\mid_{Q}\text{ is a }\left(  Q,<_{1},<_{2}\right)  \text{-partition}%
}}\underbrace{\xx_{\pi\mid_{P},w\mid_{P}}{\xx}_{\pi\mid
_{Q},w\mid_{Q}}}_{\substack{=\prod_{e\in E}x_{\pi\left(  e\right)  }^{w\left(
e\right)  } = \xx_{\pi,w}\\\text{(since }\xx_{\pi,w}=\prod_{e\in
E}x_{\pi\left(  e\right)  }^{w\left(  e\right)  }\\\text{(by the definition of
}\xx_{\pi,w}\text{))}}}\\
& =\sum_{\substack{\pi:E\rightarrow\left\{  1,2,3,\ldots\right\} ; \\\pi
\mid_{P}\text{ is a }\left(  P,>_{1},<_{2}\right)  \text{-partition;}\\\pi
\mid_{Q}\text{ is a }\left(  Q,<_{1},<_{2}\right)  \text{-partition}%
}}\xx_{\pi,w}.
\end{align*}

Now,
\begin{align*}
&
\underbrace{\Gamma\left(\left(P, >_1, <_2\right), w\mid_P\right)}_{
 \substack{ = \sum_{\sigma \text{ is a }\left(P, >_1, <_2\right)\text{-partition}}
            \xx_{\sigma, w\mid_P} \\
            \text{(by \eqref{pf.thm.antipode.Gammaw.Req.pf.Gamma1})}}}
\underbrace{\Gamma\left(\EE\mid_Q, w\mid_Q\right)}_{
 \substack{ = \sum_{\tau \text{ is a }\left(Q, <_1, <_2\right)\text{-partition}}
            \xx_{\tau, w\mid_Q} \\
            \text{(by \eqref{pf.thm.antipode.Gammaw.Req.pf.Gamma2})}}}
\\
&=
\left(\sum_{\sigma \text{ is a }\left(P, >_1, <_2\right)\text{-partition}}
\xx_{\sigma, w\mid_P}\right)
\left(\sum_{\tau \text{ is a }\left(Q, <_1, <_2\right)\text{-partition}}
\xx_{\tau, w\mid_Q}\right) \\
&=
\underbrace{
 \sum_{\sigma \text{ is a }\left(P, >_1, <_2\right)\text{-partition}}
 \sum_{\tau \text{ is a }\left(Q, <_1, <_2\right)\text{-partition}}
}_{=
   \sum_{\substack{\left(\sigma, \tau\right); \\
                   \sigma : P \to \left\{1, 2, 3, \ldots\right\}; \\
                   \tau : Q \to \left\{1, 2, 3, \ldots\right\}; \\
                   \sigma \text{ is a }\left(P, >_1, <_2\right)\text{-partition;} \\
                   \tau \text{ is a }\left(Q, <_1, <_2\right)\text{-partition}}}
   }
\xx_{\sigma, w\mid_P} \xx_{\tau, w\mid_Q}
&=
\sum_{\substack{\left(\sigma, \tau\right); \\
                \sigma : P \to \left\{1, 2, 3, \ldots\right\}; \\
                \tau : Q \to \left\{1, 2, 3, \ldots\right\}; \\
                \sigma \text{ is a }\left(P, >_1, <_2\right)\text{-partition;} \\
                \tau \text{ is a }\left(Q, <_1, <_2\right)\text{-partition}}}
\xx_{\sigma, w\mid_P} \xx_{\tau, w\mid_Q} \\
&= \sum_{\substack{\pi : E \to \left\{1, 2, 3, \ldots\right\}; \\
                \pi\mid_P \text{ is a }\left(P, >_1, <_2\right)\text{-partition;} \\
                \pi\mid_Q \text{ is a }\left(Q, <_1, <_2\right)\text{-partition}}}
\xx_{\pi, w}.
\end{align*}
This proves \eqref{pf.thm.antipode.Gammaw.long.GG=sum}.}.

Now,
\begin{align*}
& \sum_{\left(P, Q\right) \in \Adm \EE}
\left(-1\right)^{\left|P\right|}
\underbrace{\Gamma\left(\left(P, >_1, <_2\right), w\mid_P\right)
            \Gamma\left(\EE\mid_Q, w\mid_Q\right)}_{
 \substack{ = \sum_{\substack{\pi : E \to \left\{1, 2, 3, \ldots\right\}; \\
                \pi\mid_P \text{ is a }\left(P, >_1, <_2\right)\text{-partition;} \\
                \pi\mid_Q \text{ is a }\left(Q, <_1, <_2\right)\text{-partition}}}
              \xx_{\pi, w} \\
            \text{(by \eqref{pf.thm.antipode.Gammaw.long.GG=sum})}}}
\\
& = \sum_{\left(P, Q\right) \in \Adm \EE}
\left(-1\right)^{\left|P\right|}
\sum_{\substack{\pi : E \to \left\{1, 2, 3, \ldots\right\}; \\
                \pi\mid_P \text{ is a }\left(P, >_1, <_2\right)\text{-partition;} \\
                \pi\mid_Q \text{ is a }\left(Q, <_1, <_2\right)\text{-partition}}}
\xx_{\pi, w} \\
& =
\underbrace{\sum_{\left(P, Q\right) \in \Adm \EE}
\sum_{\substack{\pi : E \to \left\{1, 2, 3, \ldots\right\}; \\
                \pi\mid_P \text{ is a }\left(P, >_1, <_2\right)\text{-partition;} \\
                \pi\mid_Q \text{ is a }\left(Q, <_1, <_2\right)\text{-partition}}}
}_{ =
\sum_{\pi : E \to \left\{1, 2, 3, \ldots\right\}}
\sum_{\substack{\left(P, Q\right) \in \Adm \EE ; \\
                \pi\mid_P \text{ is a }\left(P, >_1, <_2\right)\text{-partition;} \\
                \pi\mid_Q \text{ is a }\left(Q, <_1, <_2\right)\text{-partition}}}
}
\left(-1\right)^{\left|P\right|}
\xx_{\pi, w} \\
& = \sum_{\pi : E \to \left\{1, 2, 3, \ldots\right\}}
\underbrace{
\sum_{\substack{\left(P, Q\right) \in \Adm \EE ; \\
                \pi\mid_P \text{ is a }\left(P, >_1, <_2\right)\text{-partition;} \\
                \pi\mid_Q \text{ is a }\left(Q, <_1, <_2\right)\text{-partition}}}
\left(-1\right)^{\left|P\right|}
}
_{\substack{=0 \\
            \text{(by \eqref{pf.thm.antipode.Gammaw.signrev})}}}
\xx_{\pi, w}
= \sum_{\pi : E \to \left\{1, 2, 3, \ldots\right\}} 0 \xx_{\pi, w}
= 0.
\end{align*}
\end{verlong}
\begin{vershort}
Thus,
\begin{align*}
0 &= \sum_{\left(P, Q\right) \in \Adm \EE}
\left(-1\right)^{\left|P\right|}
\Gamma\left(\left(P, >_1, <_2\right), w\mid_P\right)
\Gamma\left(\EE\mid_Q, w\mid_Q\right) \\
&= \left(-1\right)^{\left|E\right|}
\Gamma\left(\left(E, >_1, <_2\right), \underbrace{w\mid_E}_{=w}\right)
\underbrace{\Gamma\left(\EE\mid_\varnothing, w\mid_\varnothing\right)
           }_{
           =\Gamma\left(\left(\varnothing, <_1, <_2\right), w\mid_\varnothing\right)
           =1}
\\
&\qquad + \sum_{\substack{\left(P, Q\right) \in \Adm \EE ; \\
                  \left|P\right| < \left|E\right|}}
\left(-1\right)^{\left|P\right|}
\Gamma\left(\left(P, >_1, <_2\right), w\mid_P\right)
\Gamma\left(\EE\mid_Q, w\mid_Q\right)
\\
& \ \ \ \ \ \ \ \ \ \ \left(
 \begin{array}{c}
  \text{because the only pair } \left(P, Q\right) \in \Adm \EE
  \text{ satisfying } \left|P\right| \geq \left|E\right| \\
  \text{ is } \left(P, Q\right) = \left(E, \varnothing\right)
  \text{,} \\
  \text{whereas all other pairs } \left(P, Q\right) \in \Adm \EE
  \text{ satisfy } \left|P\right| < \left|E\right|
 \end{array}
\right) \\
&= \left(-1\right)^{\left|E\right|} \Gamma\left(\left(E, >_1, <_2\right), w\right) \\
& \qquad
+ \sum_{\substack{\left(P, Q\right) \in \Adm \EE ; \\
                  \left|P\right| < \left|E\right|}}
\left(-1\right)^{\left|P\right|}
\Gamma\left(\left(P, >_1, <_2\right), w\mid_P\right)
\Gamma\left(\EE\mid_Q, w\mid_Q\right) ,
\end{align*}
so that
\begin{align*}
\left(-1\right)^{\left|E\right|} \Gamma\left(\left(E, >_1, <_2\right), w\right)
&= - \sum_{\substack{\left(P, Q\right) \in \Adm \EE ; \\
                  \left|P\right| < \left|E\right|}}
\left(-1\right)^{\left|P\right|}
\Gamma\left(\left(P, >_1, <_2\right), w\mid_P\right)
\Gamma\left(\EE\mid_Q, w\mid_Q\right) \\
&= S\left(\Gamma\left(\underbrace{\EE}_{=\left(E, <_1, <_2\right)}, w\right)\right)
\qquad \left(\text{by \eqref{pf.thm.antipode.Gammaw.Seq}}\right) \\
&= S\left(\Gamma\left(\left(E, <_1, <_2\right), w\right)\right) ,
\end{align*}
and thus
$S\left(\Gamma\left(\left(E, <_1, <_2\right), w\right) \right)
= \left(-1\right)^{\left|E\right|}
\Gamma\left(\left(E, >_1, <_2\right), w\right)$.
\end{vershort}
\begin{verlong}
Thus,
\begin{align*}
0 &= \sum_{\left(P, Q\right) \in \Adm \EE}
\left(-1\right)^{\left|P\right|}
\Gamma\left(\left(P, >_1, <_2\right), w\mid_P\right)
\Gamma\left(\EE\mid_Q, w\mid_Q\right) \\
&= \left(-1\right)^{\left|E\right|}
\Gamma\left(\left(E, >_1, <_2\right), \underbrace{w\mid_E}_{=w}\right)
\underbrace{\Gamma\left(\EE\mid_\varnothing, w\mid_\varnothing\right)
           }_{=1}
\\
& \qquad
 + \sum_{\substack{\left(P, Q\right) \in \Adm \EE ; \\
                   \left(P, Q\right) \neq \left(E, \varnothing\right)}}
\left(-1\right)^{\left|P\right|}
\Gamma\left(\left(P, >_1, <_2\right), w\mid_P\right)
\Gamma\left(\EE\mid_Q, w\mid_Q\right)
\\
& \ \ \ \ \ \ \ \ \ \ \left( \begin{array}{c}
  \text{here, we have split off the addend } \\
  \text{for } \left(P, Q\right) = \left(E, \varnothing\right)
  \text{ from the sum}
  \end{array} \right) \\
&= \left(-1\right)^{\left|E\right|} \Gamma\left(\left(E, >_1, <_2\right), w\right) \\
& \qquad
+ \sum_{\substack{\left(P, Q\right) \in \Adm \EE ; \\
                  \left(P, Q\right) \neq \left(E, \varnothing\right)}}
\left(-1\right)^{\left|P\right|}
\Gamma\left(\left(P, >_1, <_2\right), w\mid_P\right)
\Gamma\left(\EE\mid_Q, w\mid_Q\right) ,
\end{align*}
so that
\begin{align*}
\left(-1\right)^{\left|E\right|} \Gamma\left(\left(E, >_1, <_2\right), w\right)
&= - \sum_{\substack{\left(P, Q\right) \in \Adm \EE ; \\
                     \left(P, Q\right) \neq \left(E, \varnothing\right)}}
\left(-1\right)^{\left|P\right|}
\Gamma\left(\left(P, >_1, <_2\right), w\mid_P\right)
\Gamma\left(\EE\mid_Q, w\mid_Q\right) \\
&= S\left(\Gamma\left(\underbrace{\EE}_{=\left(E, <_1, <_2\right)}, w\right)\right)
\qquad \left(\text{by \eqref{pf.thm.antipode.Gammaw.Seq}}\right) \\
&= S\left(\Gamma\left(\left(E, <_1, <_2\right), w\right)\right) ,
\end{align*}
and thus
$S\left(\Gamma\left(\left(E, <_1, <_2\right), w\right)\right)
= \left(-1\right)^{\left|E\right|}
\Gamma\left(\left(E, >_1, <_2\right), w\right)$.
\end{verlong}
This completes the
induction step and thus the proof of Theorem~\ref{thm.antipode.Gammaw}.
\end{proof}

\section{Proof of Theorem~\ref{thm.antipode.GammawG}}
\label{sect.proofG}

Before we begin proving Theorem~\ref{thm.antipode.GammawG}, we state a
criterion for $\EE$-partitions that is less wasteful (in the sense that
it requires fewer verifications) than the definition:

\begin{lemma}
\label{lem.Epartition.cover}
Let $\EE = \left(E, <_1, <_2\right)$ be a tertispecial double poset.
Let $\phi : E \to \left\{1, 2, 3, \ldots\right\}$ be a map. Assume
that the following two conditions hold:

\begin{itemize}

\item \textit{Condition 1:} If $e \in E$ and $f \in E$ are such that
$e$ is $<_1$-covered by $f$, and if we have $e <_2 f$, then
$\phi\left(e\right) \leq \phi\left(f\right)$.

\item \textit{Condition 2:} If $e \in E$ and $f \in E$ are such that
$e$ is $<_1$-covered by $f$, and if we have $f <_2 e$, then
$\phi\left(e\right) < \phi\left(f\right)$.

\end{itemize}

Then, $\phi$ is an $\EE$-partition.
\end{lemma}

\begin{proof}[Proof of Lemma~\ref{lem.Epartition.cover}.]
For any $a \in E$ and $b \in E$, we define a subset
$\left[a, b\right]$ of $E$ as in the proof of
Lemma~\ref{lem.admissible.cover}.

We need to show that $\phi$ is an $\EE$-partition. In other words,
we need to prove the following two claims:

\begin{statement}
\textit{Claim 1:} Every $e \in E$ and $f \in E$ satisfying
$e <_1 f$ satisfy $\phi\left(e\right) \leq \phi\left(f\right)$.
\end{statement}

\begin{statement}
\textit{Claim 2:} Every $e \in E$ and $f \in E$ satisfying
$e <_1 f$ and $f <_2 e$ satisfy
$\phi\left(e\right) < \phi\left(f\right)$.
\end{statement}

\textit{Proof of Claim 1:} Assume the contrary. Thus, there
exists a pair $\left(e, f\right) \in E \times E$ satisfying
$e <_1 f$ but not $\phi\left(e\right) \leq \phi\left(f\right)$.
Such a pair will be called a \textit{malrelation}. Fix a
malrelation $\left(u, v\right)$ for which the value
$\left|\left[u, v\right]\right|$ is minimum (such a
$\left(u, v\right)$ exists, since there exists a malrelation).
Thus, $u \in E$ and $v \in E$ and $u <_1 v$ but not
$\phi\left(u\right) \leq \phi\left(v\right)$.

If $u$ was $<_1$-covered by $v$, then we would obtain
$\phi\left(u\right) \leq \phi\left(v\right)$
\ \ \ \ \footnote{\textit{Proof.} Assume that $u$ is
$<_1$-covered by $v$. Thus, $u$ and $v$ are $<_2$-comparable
(since the double poset $\EE$ is tertispecial). In other words,
we have either $u <_2 v$ or $u = v$ or $v <_2 u$. In the
first of these three cases, we obtain
$\phi\left(u\right) \leq \phi\left(v\right)$ by applying
Condition 1 to $e = u$ and $f = v$. In the third of these
cases, we obtain
$\phi\left(u\right) < \phi\left(v\right)$ (and thus
$\phi\left(u\right) \leq \phi\left(v\right)$) 
by applying Condition 2 to $e = u$ and $f = v$. The second
of these cases cannot happen because $u <_1 v$. Thus, we
always have $\phi\left(u\right) \leq \phi\left(v\right)$,
qed.}, which would
contradict the fact that we do not have
$\phi\left(u\right) \leq \phi\left(v\right)$. Hence, $u$ is not
$<_1$-covered by $v$. Consequently, there exists a $w \in E$
such that $u <_1 w <_1 v$ (since $u <_1 v$). Consider this
$w$. Applying \eqref{pf.lem.admissible.cover.1} to $a = u$,
$c = w$ and $b = v$, we see that both numbers
$\left|\left[u, w\right]\right|$ and
$\left|\left[w, v\right]\right|$ are smaller than
$\left|\left[u, v\right]\right|$. Hence, neither
$\left(u, w\right)$ nor $\left(w, v\right)$ is a malrelation
(since we picked $\left(u, v\right)$ to be a malrelation with
minimum $\left|\left[u, v\right]\right|$). Therefore, we have
$\phi\left(u\right) \leq \phi\left(w\right)$ (since $u <_1 w$,
but $\left(u, w\right)$ is not a malrelation) and
$\phi\left(w\right) \leq \phi\left(v\right)$ (since $w <_1 v$,
but $\left(w, v\right)$ is not a malrelation).
Combining these two inequalities, we obtain
$\phi\left(u\right) \leq \phi\left(w\right)
\leq \phi\left(v\right)$. This contradicts
the fact that we do not have
$\phi\left(u\right) \leq \phi\left(v\right)$. This contradiction
concludes the proof of Claim 1.

Instead of Claim 2, we shall prove the following stronger claim:

\begin{statement}
\textit{Claim 3:} Every $e \in E$ and $f \in E$ satisfying
$e <_1 f$ and not $e <_2 f$ satisfy
$\phi\left(e\right) < \phi\left(f\right)$.
\end{statement}

\textit{Proof of Claim 3:} Assume the contrary. Thus, there
exists a pair $\left(e, f\right) \in E \times E$ satisfying
$e <_1 f$ and not $e <_2 f$ but not
$\phi\left(e\right) < \phi\left(f\right)$.
Such a pair will be called a \textit{malrelation}. Fix a
malrelation $\left(u, v\right)$ for which the value
$\left|\left[u, v\right]\right|$ is minimum (such a
$\left(u, v\right)$ exists, since there exists a malrelation).
Thus, $u \in E$ and $v \in E$ and $u <_1 v$ and not $u <_2 v$
but not $\phi\left(u\right) < \phi\left(v\right)$.

If $u$ was $<_1$-covered by $v$, then we would obtain
$\phi\left(u\right) < \phi\left(v\right)$
easily\footnote{\textit{Proof.} Assume that $u$ is
$<_1$-covered by $v$. Thus, $u$ and $v$ are $<_2$-comparable
(since the double poset $\EE$ is tertispecial). In other words,
we have either $u <_2 v$ or $u = v$ or $v <_2 u$. Since
neither $u <_2 v$ nor $u = v$ can hold (indeed, $u <_2 v$
is ruled out by assumption, whereas $u = v$ is ruled out by
$u <_1 v$), we thus have $v <_2 u$. Therefore,
$\phi\left(u\right) < \phi\left(v\right)$
by Condition 2 (applied to $e = u$ and $f = v$), qed.}, which
would contradict the fact that we do not have
$\phi\left(u\right) < \phi\left(v\right)$. Hence, $u$ is not
$<_1$-covered by $v$. Consequently, there exists a $w \in E$
such that $u <_1 w <_1 v$ (since $u <_1 v$). Consider this
$w$. Applying \eqref{pf.lem.admissible.cover.1} to $a = u$,
$c = w$ and $b = v$, we see that both numbers
$\left|\left[u, w\right]\right|$ and
$\left|\left[w, v\right]\right|$ are smaller than
$\left|\left[u, v\right]\right|$. Hence, neither
$\left(u, w\right)$ nor $\left(w, v\right)$ is a malrelation
(since we picked $\left(u, v\right)$ to be a malrelation with
minimum $\left|\left[u, v\right]\right|$).

But $\phi\left(v\right) \leq \phi\left(u\right)$ (since we do
not have $\phi\left(u\right) < \phi\left(v\right)$). On the
other hand, $u <_1 w$ and therefore $\phi\left(u\right) \leq
\phi\left(w\right)$ (by Claim 1, applied to $e = u$ and
$f = w$). Furthermore, $w <_1 v$ and
thus $\phi\left(w\right) \leq \phi\left(v\right)$ (by Claim 1,
applied to $e = w$ and $f = v$).
The chain of inequalities
$\phi\left(v\right) \leq \phi\left(u\right)
\leq \phi\left(w\right) \leq \phi\left(v\right)$ ends with
the same term that it begins with; therefore, it must be a chain
of equalities. In other words, we have
$\phi\left(v\right) = \phi\left(u\right)
= \phi\left(w\right) = \phi\left(v\right)$.

Now, using $\phi\left(w\right) = \phi\left(v\right)$, we can
see that $w <_2 v$\ \ \ \ \footnote{\textit{Proof.} Assume
the contrary. Thus, we do not have $w <_2 v$. But
$\phi\left(w\right) = \phi\left(v\right)$ shows that we do not
have $\phi\left(w\right) < \phi\left(v\right)$. Hence,
$\left(w, v\right)$ is a malrelation (since $w <_1 v$ and not
$w <_2 v$ but not $\phi\left(w\right) < \phi\left(v\right)$).
This contradicts the fact that $\left(w, v\right)$ is not
a malrelation. This contradiction completes the proof.}.
The same argument (applied to $u$ and $w$ instead of $w$ and
$v$) shows that $u <_2 w$. Thus, $u <_2 w <_2 v$, which
contradicts the fact that we do not have $u <_2 v$. This
contradiction proves Claim 3.

\textit{Proof of Claim 2:} The condition ``$f <_2 e$'' is stronger
than ``not $e <_2 f$''. Thus, Claim 2 follows from Claim 3.

Claims 1 and 2 are now both proven, and so
Lemma~\ref{lem.Epartition.cover} follows.
\end{proof}

\begin{proof}
[Proof of Lemma \ref{lem.coeven.all-one}.] Consider the following three
logical statements:

\textit{Statement 1:} The $G$-orbit $O$ is $E$-coeven.

\textit{Statement 2:} All elements of $O$ are $E$-coeven.

\textit{Statement 3:} At least one element of $O$ is $E$-coeven.

Statements 1 and 2 are equivalent (according to the definition of
when a $G$-orbit is $E$-coeven).
Our goal is to prove that Statements 1 and 3 are
equivalent (because this is precisely what Lemma \ref{lem.coeven.all-one}
says). Thus, it suffices to show that Statements 2 and 3 are
equivalent (because we already know that Statements 1 and 2 are
equivalent). Since Statement 2 obviously implies Statement 3 (in fact,
the $G$-orbit $O$ contains at least one element), we therefore only
need to show that Statement 3 implies Statement 2. Thus, assume that Statement
3 holds. We need to prove that Statement 2 holds.

There exists at least one $E$-coeven $\phi\in O$ (because we assumed that
Statement 3 holds). Consider this $\phi$. Now, let $\pi\in O$ be arbitrary. We
shall show that $\pi$ is $E$-coeven.

We know that $\phi$ is $E$-coeven. In other words,
\begin{equation}
\text{every }g\in G\text{ satisfying }g\phi=\phi\text{ is }E\text{-even.}%
\label{pf.lem.coeven-all-one.1}%
\end{equation}

Now, let $g\in G$ be such that $g\pi=\pi$. Since $\phi$ belongs to the
$G$-orbit $O$, we have $O=G\phi$. Now, $\pi\in O=G\phi$. In other words, there
exists some $h\in G$ such that $\pi=h\phi$. Consider this $h$. We have
$g\pi=\pi$. Since $\pi=h\phi$, this rewrites as $gh\phi=h\phi$. In other
words, $h^{-1}gh\phi=\phi$. Thus, (\ref{pf.lem.coeven-all-one.1}) (applied to
$h^{-1}gh$ instead of $g$) shows that $h^{-1}gh$ is $E$-even. In other words,%
\begin{equation}
\text{the action of }h^{-1}gh\text{ on }E\text{ is an even permutation of
}E\text{.}\label{pf.lem.coeven-all-one.2}%
\end{equation}

Now, let $\varepsilon$ be the group homomorphism from $G$ to
$\operatorname{Aut}E$\ \ \ \ \footnote{We use the notation
$\operatorname{Aut}E$ for the group of all permutations
of the set $E$.} which describes the $G$-action on $E$. Then,
$\varepsilon\left(  h^{-1}gh\right)  $ is the action of $h^{-1}gh$ on $E$, and
thus is an even permutation of $E$ (by (\ref{pf.lem.coeven-all-one.2})).

But since $\varepsilon$ is a group homomorphism, we have $\varepsilon\left(
h^{-1}gh\right)  =\varepsilon\left(  h\right)  ^{-1}\varepsilon\left(
g\right)  \varepsilon\left(  h\right)  $. Thus, the permutations
$\varepsilon\left(  h^{-1}gh\right)  $ and $\varepsilon\left(  g\right)  $ of
$E$ are conjugate. Since the permutation $\varepsilon\left(  h^{-1}gh\right)
$ is even, this shows that the permutation $\varepsilon\left(  g\right)  $ is
even. In other words, the action of $g$ on $E$ is an even permutation of $E$.
In other words, $g$ is $E$-even.

Now, let us forget that we fixed $g$. We thus have shown that every $g\in G$
satisfying $g\pi=\pi$ is $E$-even. In other words, $\pi$ is $E$-coeven.

Let us now forget that we fixed $\pi$. Thus, we have proven that every $\pi\in
O$ is $E$-coeven. In other words, Statement 2 holds. We have thus shown that
Statement 3 implies Statement 2. Consequently, Statements 2 and 3 are
equivalent, and so the proof of Lemma \ref{lem.coeven.all-one} is complete.
\end{proof}

\begin{vershort}
We leave the fairly straightforward proof of Proposition~\ref{prop.xxOw}
to the reader.
\end{vershort}

\begin{verlong}
\begin{proof}[Proof of Proposition~\ref{prop.xxOw}.]

(a) We need to prove that $g \pi \in \Par \EE$ for each $g \in G$ and
each $\pi \in \Par \EE$. So let us fix $g \in G$ and $\pi \in \Par \EE$.
We must prove that $g \pi \in \Par \EE$.

We know that $\pi$ is an $\EE$-partition (since $\pi \in \Par \EE$).
In other words, the following two claims hold:

\begin{statement}
\textit{Claim 1:} Every $e \in E$ and $f \in E$ satisfying $e <_1 f$
satisfy $\pi\left(e\right) \leq \pi\left(f\right)$.
\end{statement}

\begin{statement}
\textit{Claim 2:} Every $e \in E$ and $f \in E$ satisfying $e <_1 f$
and $f <_2 e$ satisfy $\pi\left(e\right) < \pi\left(f\right)$.
\end{statement}

But our goal is to prove that $g \pi \in \Par \EE$. In other words,
our goal is to prove that $g \pi$ is an $\EE$-partition. In other
words, we must show that the following two claims hold:

\begin{statement}
\textit{Claim 3:} Every $e \in E$ and $f \in E$ satisfying $e <_1 f$
satisfy
$\left(g \pi\right)\left(e\right) \leq \left(g \pi\right)\left(f\right)$.
\end{statement}

\begin{statement}
\textit{Claim 4:} Every $e \in E$ and $f \in E$ satisfying $e <_1 f$
and $f <_2 e$ satisfy
$\left(g \pi\right)\left(e\right) < \left(g \pi\right)\left(f\right)$.
\end{statement}

Let us first prove Claim 4:

\textit{Proof of Claim 4:} Let $e \in E$ and $f \in E$ be such that
$e <_1 f$ and $f <_2 e$. The definition of the $G$-action on the set
$\left\{1, 2, 3, \ldots\right\}^E$ shows that
$\left(g \pi\right) \left(e\right) = \pi \left(g^{-1} e\right)$ and
$\left(g \pi\right) \left(f\right) = \pi \left(g^{-1} f\right)$.

Now, from $e <_1 f$, we obtain
$g^{-1} e <_1 g^{-1} f$ (since $G$ preserves
the relation $<_1$). Also, from $f <_2 e$, we obtain
$g^{-1} f <_2 g^{-1} e$ (since $G$ preserves
the relation $<_2$). Thus, we can apply Claim 2 to
$g^{-1} e$ and $g^{-1} f$ instead of $e$ and $f$. As a result, we
conclude that
$\pi \left(g^{-1} e\right) < \pi \left(g^{-1} f\right)$. In view of
$\left(g \pi\right) \left(e\right) = \pi \left(g^{-1} e\right)$ and
$\left(g \pi\right) \left(f\right) = \pi \left(g^{-1} f\right)$, this
rewrites as
$\left(g \pi\right)\left(e\right) < \left(g \pi\right)\left(f\right)$.
Thus, Claim 4 is proven.

We thus have derived Claim 4 from Claim 2. Similarly, Claim 3 can be
derived from Claim 1. Thus, Claim 3 and Claim 4 are proven. As
explained above, this shows that $g \pi$ is an $\EE$-partition. In
other words, $g \pi \in \Par \EE$. This completes our proof of
Proposition~\ref{prop.xxOw} (a).

(b) We need to show that $\xx_{\pi, w} = \xx_{\psi, w}$ for any two
elements $\pi$ and $\psi$ of $O$. So let $\pi$ and $\psi$ be two
elements of $O$.

We know that $G$ preserves $w$. In other words, every $g \in G$ and
$e \in E$ satisfy
\begin{equation}
w \left( ge \right) = w \left( e \right) .
\label{pf.prop.xxOw.b.1}
\end{equation}

The two elements $\pi$ and $\psi$ belong to one and
the same $G$-orbit (namely, to $O$). Thus, there exists a $g \in G$
satisfying $g \pi = \psi$. Consider this $g$. For each $e \in E$,
we have
\begin{equation}
\underbrace{\psi}_{= g \pi} \left(e\right)
= \left(g \pi\right) \left(e\right)
= \pi \left(g^{-1} e\right)
\label{pf.prop.xxOw.b.3}
\end{equation}
(by the definition of the $G$-action on the set
$\left\{1, 2, 3, \ldots\right\}^E$).

The definition of $\xx_{\pi, w}$ yields
$\xx_{\pi, w}
= \prod_{e \in E} x_{\pi\left(e\right)}^{w\left(e\right)}$.
The definition of $\xx_{\psi, w}$ yields
\begin{align*}
\xx_{\psi, w}
&= \prod_{e \in E} \underbrace{x_{\psi\left(e\right)}^{w\left(e\right)}}_{\substack{= x_{\pi \left(g^{-1} e\right)}^{w\left(e\right)} \\ \text{(by (\ref{pf.prop.xxOw.b.3}))}}}
= \prod_{e \in E} x_{\pi \left(g^{-1} e\right)}^{w\left(e\right)}
= \prod_{e \in E}
     \underbrace{x_{\pi \left(g^{-1} g e\right)}^{w\left(g e\right)}}_{
                \substack{= x_{\pi\left(e\right)}^{w\left(g e\right)} \\
                           \text{(since } g^{-1} g e = e \text{)}}} \\
& \ \ \ \ \ \ \ \ \ \ \left( \begin{array}{c}
\text{here, we have substituted } g e \text{ for } e \text{ in the product,} \\
\text{since the map } E \to E, \ e \mapsto g e \text{ is a bijection}
\end{array} \right) \\
&= \prod_{e \in E} \underbrace{x_{\pi\left(e\right)}^{w\left(g e\right)}}_{\substack{= x_{\pi\left(e\right)}^{w\left(e\right)} \\ \text{(by (\ref{pf.prop.xxOw.b.1}))}}}
= \prod_{e \in E} x_{\pi\left(e\right)}^{w\left(e\right)}
= \xx_{\pi, w} .
\end{align*}
Thus, $\xx_{\pi, w} = \xx_{\psi, w}$ is proven. This completes the
proof of Proposition~\ref{prop.xxOw} (b).
\end{proof}
\end{verlong}

Next, we will show three simple properties of posets on which
groups act. First, we introduce a notation:

\begin{definition}
\label{def.G-poset.g-orbit}
Let $G$ be a group. Let $g \in G$. Let $E$ be a $G$-set.
Then, the subgroup $\left< g \right>$ of $G$ (this is the
subgroup of $G$ generated by $g$) also acts on $E$. The
$\left< g \right>$-orbits on $E$ will be called the
\textit{$g$-orbits} on $E$. When $E$ is clear from the
context, we shall simply call them the
\textit{$g$-orbits}.

We can also describe these $g$-orbits explicitly: For any given
$e \in E$, the $g$-orbit of $e$ (that is, the unique
$g$-orbit that contains $e$) is
$\left< g \right> e = \left\{ g^k e \mid k \in \ZZ \right\}$.

Equivalently, the $g$-orbits on $E$ can be characterized as
follows: The action of $g$ on $E$ is a permutation of $E$.
The cycles of this permutation are the $g$-orbits on $E$
(at least when $E$ is finite).
\end{definition}

\begin{proposition}
\label{prop.G-poset.quot.poset}
Let $E$ be a set. Let $<_1$ be a strict partial order on $E$.
Let $G$ be a finite group which acts on $E$. Assume that $G$ preserves
the relation $<_1$.

Let $g \in G$. Let $E^g$ be the set of all $g$-orbits
on $E$. Define a binary relation $<_1^g$ on $E^g$ by
\[
\left(  u<_{1}^{g}v\right)  \Longleftrightarrow\left(  \text{there exist }a\in
u\text{ and }b\in v\text{ with }a<_{1}b\right)  .
\]
Then, $<_1^g$ is a strict partial order.
\end{proposition}

Proposition~\ref{prop.G-poset.quot.poset} is precisely \cite[Lemma 2.4]{Joch},
but let us outline the proof for the sake of completeness:

\begin{proof}[Proof of Proposition~\ref{prop.G-poset.quot.poset}.]
Let us first show that the relation $<_{1}^{g}$ is irreflexive.
Indeed, assume the contrary. Thus, there exists a $u \in E^g$
such that $u <_1^g u$. Consider this $u$.

We have $u \in E^g$. In other words, $u$ is a $g$-orbit on $E$.

Since $u<_{1}^{g}u$, there exist $a\in u$ and $b\in u$ with
$a<_{1}b$ (by the definition of the relation $<_{1}^{g}$).
Consider these $a$ and $b$. There exists a $k\in\mathbb{Z}$ such
that $b=g^{k}a$ (since $a$ and $b$ both lie in one and the same $g$-orbit
$u$). Consider this $k$.

Each element of $G$ has finite order (since $G$ is a finite group).
In particular, the element $g$ of $G$ has finite order. In other words,
there exists a positive integer $n$ such that $g^{n}=1_G$.
Consider this $n$. Every $p \in \ZZ$ satisfies $g^{np}=\left(
g^{n}\right)  ^{p}=1_G$ (since $g^{n}=1_G$).
Applying this to $p=k$, we obtain $g^{nk}=1_G$.

Now, $a<_{1}b=g^{k}a$. Since $G$ preserves the relation $<_{1}$, this shows
that $ha<_{1}hg^{k}a$ for every $h\in G$.
\begin{vershort}
Thus, $g^{\ell k}a<_{1}g^{\ell k}g^{k}a$ for every $\ell\in \NN $.
\end{vershort}
\begin{verlong}
Thus, for every $\ell \in \NN$, we have
$g^{\ell k}a<_{1}g^{\ell k}g^{k}a$ (by the inequality $ha <_1 h g^k a$,
applied to $h = g^{\ell k}$).
\end{verlong}
Hence,
$g^{\ell k} a <_{1} g^{\ell k} g^{k} a
= g^{\ell k + k} a = g^{\left(  \ell+1\right)  k} a$
for every $\ell\in \NN $.
Consequently, $g^{0k}a<_{1}g^{1k}a<_{1}g^{2k}a<_{1}\cdots<_{1}g^{nk}a$. Thus,
$g^{0k}a<_{1}\underbrace{g^{nk}}_{=1_G}a=a$, which
contradicts $\underbrace{g^{0k}}_{=g^0=1_G}a=1_Ga=a$.
This contradiction proves that our assumption was
wrong. Hence, the relation $<_{1}^{g}$ is irreflexive.

Let us next show that the relation $<_{1}^{g}$ is transitive. Indeed, let $u$,
$v$ and $w$ be three elements of $E^{g}$ such that $u<_{1}^{g}v$ and
$v<_{1}^{g}w$. We must prove that $u<_{1}^{g}w$.

There exist $a\in u$ and $b\in v$ with $a<_{1}b$ (since $u<_{1}^{g}v$).
Consider these $a$ and $b$.

There exist $a^{\prime}\in v$ and $b^{\prime}\in w$ with $a^{\prime}
<_{1}b^{\prime}$ (since $v<_{1}^{g}w$). Consider these $a^{\prime}$ and
$b^{\prime}$.

The set $v$ is a $g$-orbit (since $v \in E^g$).
The elements $b$ and $a^{\prime}$ lie in one and the same $g$-orbit (namely,
in $v$). Hence, there exists some $k\in\mathbb{Z}$ such that $a^{\prime}
=g^{k}b$. Consider this $k$. We have $a<_{1}b$ and thus $g^{k}a<_{1}g^{k}b$
(since $G$ preserves the relation $<_{1}$). Hence,
$g^{k}a<_{1}g^{k}b=a^{\prime}<_{1}b^{\prime}$.
Since $g^{k}a\in u$ (because $a\in u$, and because $u$ is a
$g$-orbit) and $b^{\prime}\in w$, this shows that $u<_{1}^{g}w$
(by the definition of the relation $<_{1}^{g}$).
We thus have proven that the
relation $<_{1}^{g}$ is transitive.

Now, we know that the relation $<_{1}^{g}$ is irreflexive and transitive, and
thus also antisymmetric (since every irreflexive and transitive binary
relation is antisymmetric). In other words, $<_{1}^{g}$ is a strict partial
order. This proves Proposition~\ref{prop.G-poset.quot.poset}.
\end{proof}

\begin{remark}
Proposition~\ref{prop.G-poset.quot.poset} can be generalized:
Let $E$ be a set. Let $<_1$ be a strict partial order on $E$.
Let $G$ be a finite group which acts on $E$. Assume that $G$ preserves
the relation $<_1$. Let $H$ be a subgroup of $G$.
Let $E^H$ be the set of all $H$-orbits on $E$.
Define a binary relation $<_{1}^H$ on $E^H$ by
\[
\left(  u<_{1}^H v\right)  \Longleftrightarrow\left(  \text{there exist }a\in
u\text{ and }b\in v\text{ with }a<_{1}b\right)  .
\]
Then, $<_1^H$ is a strict partial order.

This result (whose proof is quite similar to that of
Proposition~\ref{prop.G-poset.quot.poset}) implicitly appears in
\cite[p. 30]{Stanley-Peck}.
\end{remark}

\begin{proposition}
\label{prop.G-poset.quot.double}
Let $\EE = \left(E, <_1, <_2\right)$ be a tertispecial double poset.
Let $G$ be a finite group which acts on $E$. Assume that $G$ preserves
both relations $<_1$ and $<_2$.

Let $g \in G$. Let $E^g$ be the set of all $g$-orbits
on $E$. Define a binary relation $<_{1}^{g}$ on $E^{g}$ by
\[
\left(  u<_{1}^{g}v\right)  \Longleftrightarrow\left(  \text{there exist }a\in
u\text{ and }b\in v\text{ with }a<_{1}b\right)  .
\]
Define a binary relation $<_{2}^{g}$ on $E^{g}$ by
\[
\left(  u<_{2}^{g}v\right)  \Longleftrightarrow\left(  \text{there exist }a\in
u\text{ and }b\in v\text{ with }a<_{2}b\right)  .
\]
Let $\EE^g$ be the triple $\left(E^g, <_1^g, <_2^g\right)$. Then,
$\EE^g$ is a tertispecial double poset.
\end{proposition}

\begin{proof}[Proof of Proposition~\ref{prop.G-poset.quot.double}.]
\begin{vershort}
Proposition~\ref{prop.G-poset.quot.poset} shows that $<_{1}^{g}$ is a
strict partial order. Similarly, $<_{2}^{g}$ is a
strict partial order.
\end{vershort}
\begin{verlong}
Both relations $<_1$ and $<_2$ are strict partial orders
(since $\EE$ is a double poset).
Proposition~\ref{prop.G-poset.quot.poset} shows that $<_{1}^{g}$ is a
strict partial order.
Proposition~\ref{prop.G-poset.quot.poset} (applied to $<_2$ and
$<_2^g$ instead of $<_1$ and $<_1^g$) shows that $<_{2}^{g}$ is a
strict partial order.
\end{verlong}
Thus, $\EE^g
= \left(E^g, <_1^g, <_2^g\right)$ is a double poset. It remains
to show that this double poset $\EE^g$ is tertispecial.

Let $u$ and $v$ be two elements of
$E^{g}$ such that $u$ is $<_{1}^{g}$-covered by $v$. We shall prove that $u$
and $v$ are $<_{2}^{g}$-comparable.

We have $u<_{1}^{g}v$ (since $u$ is $<_{1}^{g}$-covered by $v$). In other
words, there exist $a\in u$ and $b\in v$ with $a<_{1}b$ (by the
definition of the relation $<_{1}^{g}$). Consider these $a$ and $b$.

\begin{vershort}
If there was a $c \in E$ satisfying $a <_1 c <_1 b$, then we
would have $u <_1^g w <_1^g v$ with $w$ being the $g$-orbit of
$c$, and this would contradict the condition that $u$ is
$<_1^g$-covered by $v$. Hence, no such $c$ can exist.
\end{vershort}
\begin{verlong}
There exists no $c \in E$ satisfying $a <_1 c <_1 b$
\ \ \ \ \footnote{\textit{Proof.} Assume the contrary. Thus,
there exists some $c \in E$ satisfying $a <_1 c <_1 b$. Consider
this $c$. Let $w$ be the $g$-orbit of $c$. Thus, $w \in E^g$
and $c \in w$.
\par
Now, the elements $a \in u$ and $c \in w$ satisfy
$a <_1 c$. Hence, $u <_1^g w$ (by the definition of the
relation $<_1^g$).
\par
Also, the elements $c \in w$ and $b \in v$ satisfy
$c <_1 b$. Hence, $w <_1^g v$ (by the definition of the
relation $<_1^g$).
\par
Now, we have $u <_1^g w <_1^g v$. This contradicts the
fact that $u$ is $<_1^g$-covered by $v$. Thus, we have
obtained a contradiction; hence, our assumption was wrong.
Qed.}.
\end{verlong}
In other words, $a$ is $<_1$-covered by $b$ (since we know that
$a <_1 b$). Thus, $a$ and $b$ are $<_2$-comparable (since the
double poset $\EE$ is tertispecial).
\begin{vershort}
Consequently, $u$ and $v$ are $<_2^g$-comparable.
\end{vershort}
\begin{verlong}
In other words, either $a <_2 b$ or $a = b$ or $b <_2 a$.
Therefore, either $u <_2^g v$ or $u = v$ or $v <_2^g u$
\ \ \ \ \footnote{Here, we have used the following facts:
\begin{itemize}
\item If $a <_2 b$, then $u <_2^g v$. (This follows from the
definition of the relation $<_2^g$, since $a \in u$ and
$b \in v$.)
\item If $a = b$, then $u = v$. (This follows from the
fact that $u$ is the $g$-orbit of $a$ (since $u$ is a $g$-orbit
and contains $a$) whereas $v$ is the $g$-orbit of $b$
(since $v$ is a $g$-orbit and contains $b$).)
\item If $b <_2 a$, then $v <_2^g u$. (This follows from the
definition of the relation $<_2^g$, since $b \in v$ and
$a \in u$.)
\end{itemize}
}. In other words, $u$ and $v$ are $<_2^g$-comparable.
\end{verlong}

Now, let us forget that we fixed $u$ and $v$. We thus have shown that if
$u$ and $v$ are two elements of $E^{g}$ such that $u$ is $<_{1}^{g}$-covered
by $v$, then $u$ and $v$ are $<_{2}^{g}$-comparable. In other words, the
double poset $\EE^g = \left(E^g, <_1^g, <_2^g\right)$ is tertispecial. This
proves Proposition~\ref{prop.G-poset.quot.double}.
\end{proof}

\begin{verlong}
Next, let us state a basic fact about $G$-sets:

\begin{proposition}
\label{prop.G-set.quot.Phi}Let $G$ be a finite group. Let $E$ be a $G$-set.
Let $g \in G$.
Let $X$ be any set. Recall that the set $X^{E}$ of all maps $E\rightarrow X$
becomes a $G$-set (according to Definition \ref{def.G-sets.terminology} (d)).

There is a bijection $\Phi$ between

\begin{itemize}
\item the maps $\pi:E\rightarrow X$ satisfying $g\pi=\pi$
\end{itemize}

and

\begin{itemize}
\item the maps $\overline{\pi}:E^{g}\rightarrow X$.
\end{itemize}

Namely, this bijection $\Phi$ sends any map $\pi:E\rightarrow X$ satisfying
$g\pi=\pi$ to the map $\overline{\pi}:E^{g}\rightarrow X$ defined by
\[
\overline{\pi}\left(  u\right)  =\pi\left(  a\right)  \qquad\text{for every
}u\in E^{g}\text{ and }a\in u.
\]

\end{proposition}

\begin{proof}
[Proof of Proposition \ref{prop.G-set.quot.Phi}.]Let $\mathfrak{A}$ be the set
of all maps $\pi:E\rightarrow X$ satisfying $g\pi=\pi$. Thus,%
\[
\mathfrak{A}=\left\{  \pi:E\rightarrow X\ \mid\ g\pi=\pi\right\}  =\left\{
\psi:E\rightarrow X\ \mid\ g\psi=\psi\right\}
\]
(here, we have renamed the index $\pi$ as $\psi$).

Let $\mathfrak{B}$ be the set of all maps $\overline{\pi}:E^{g}\rightarrow X$.
Thus,%
\[
\mathfrak{B}=\left\{  \overline{\pi}:E^{g}\rightarrow X\right\}  =X^{E^{g}}.
\]

Let $\pi\in\mathfrak{A}$. Thus, $\pi\in\mathfrak{A}=\left\{  \psi:E\rightarrow
X\ \mid\ g\psi=\psi\right\}  $. In other words, $\pi$ is a map $\psi
:E\rightarrow X$ satisfying $g\psi=\psi$. In other words, $\pi$ is a map
$E\rightarrow X$ and satisfies $g\pi=\pi$. Now, we define a map $\pi^{\circ
}:E^{g}\rightarrow X$ as follows: Let $u\in E^{g}$. Thus, $u$ is a $g$-orbit.
In particular, $u$ is a nonempty set. Hence, we can pick some $a\in u$. Now,
the element $\pi\left(  a\right)  \in X$ is independent on the choice of
$a$\ \ \ \ \footnote{\textit{Proof.} We must show that if $a_{1}$ and $a_{2}$
are two elements of $u$, then $\pi\left(  a_{1}\right)  =\pi\left(
a_{2}\right)  $.
\par
Indeed, let $a_{1}$ and $a_{2}$ be two elements of $u$. We must show that
$\pi\left(  a_{1}\right)  =\pi\left(  a_{2}\right)  $.
\par
The elements $a_{1}$ and $a_{2}$ lie in one and the same $g$-orbit (namely, in
$u$). In other words, there exists some $k\in\mathbb{Z}$ such that
$a_{2}=g^{k}a_{1}$. Consider this $k$.
\par
Recall that $g\pi=\pi$. Thus, $g$ lies in the stabilizer of $\pi$. Since the
stabilizer of $\pi$ is a subgroup of $G$, we therefore conclude that every
power of $g$ must also lie in the stabilizer of $\pi$. In particular, $g^{k}$
lies in the stabilizer of $\pi$ (since $g^{k}$ is a power of $g$). In other
words, $g^{k}\pi=\pi$. But $a_{2}=g^{k}a_{1}$, so that $a_{1}=\left(
g^{k}\right)  ^{-1}a_{2}$. Applying the map $\pi$ to both sides of this
equality, we obtain
\[
\pi\left(  a_{1}\right)  =\pi\left(  \left(  g^{k}\right)  ^{-1}a_{2}\right)
.
\]
Comparing this with%
\[
\left(  g^{k}\pi\right)  \left(  a_{2}\right)  =\pi\left(  \left(
g^{k}\right)  ^{-1}a_{2}\right)  \ \ \ \ \ \ \ \ \ \ \left(  \text{by the
definition of the }G\text{-action on }X^{E}\right)  ,
\]
we obtain $\pi\left(  a_{1}\right)  =\underbrace{\left(  g^{k}\pi\right)
}_{=\pi}\left(  a_{2}\right)  =\pi\left(  a_{2}\right)  $. This completes our
proof.}. Thus, we can define $\pi^{\circ}\left(  u\right)  $ as $\pi\left(
a\right)  $.

Thus, we have defined an element $\pi^{\circ}\left(  u\right)  $ for each
$u\in E^{g}$. In other words, we have defined a map $\pi^{\circ}%
:E^{g}\rightarrow X$. Furthermore, this map $\pi^{\circ}$ has the following
property: If $u\in E^{g}$, then%
\begin{equation}
\pi^{\circ}\left(  u\right)  =\pi\left(  a\right)
\ \ \ \ \ \ \ \ \ \ \text{for every }a\in u.\label{pf.prop.G-set.quot.Phi.1}%
\end{equation}

The element $\pi^{\circ}$ is a map $E^{g}\rightarrow X$, thus an element of
$X^{E^{g}}$. In other words, $\pi^{\circ}\in X^{E^{g}}=\mathfrak{B}$.

Now, forget that we fixed $\pi$. Thus, for each $\pi\in\mathfrak{A}$, we have
defined a $\pi^{\circ}\in\mathfrak{B}$, and this $\pi^{\circ}$ satisfies
(\ref{pf.prop.G-set.quot.Phi.1}) for each $u\in E^{g}$.

Now, let $\Phi$ be the map%
\[
\mathfrak{A}\rightarrow\mathfrak{B},\ \ \ \ \ \ \ \ \ \ \pi\mapsto\pi^{\circ}.
\]
(This is well-defined, since $\pi^{\circ}\in\mathfrak{B}$ for each $\pi
\in\mathfrak{A}$.)

Next, let us introduce one more notation: If $e$ is an element of $E$, then
$\left[  e\right]  $ shall mean the $g$-orbit of $e$. Furthermore, let
$\mathfrak{o}$ be the map $E\rightarrow E^{g}$ that sends each element $e\in
E$ to its $g$-orbit $\left[  e\right]  $.

For every $\overline{\pi}\in\mathfrak{B}$, we have $\overline{\pi}%
\circ\mathfrak{o}\in\mathfrak{A}$\ \ \ \ \footnote{\textit{Proof:} Let
$\overline{\pi}\in\mathfrak{B}$. Thus, $\overline{\pi}\in\mathfrak{B}%
=X^{E^{g}}$. In other words, $\overline{\pi}$ is a map $E^{g}\rightarrow X$.
\par
Let $e\in E$. Then, $\mathfrak{o}\left(  e\right)  =\left[  e\right]  $ (by
the definition of $\mathfrak{o}$). Moreover, $\left(  g\left(  \overline{\pi
}\circ\mathfrak{o}\right)  \right)  \left(  e\right)  =\left(  \overline{\pi
}\circ\mathfrak{o}\right)  \left(  g^{-1}e\right)  $ (by the definition of the
$G$-action on $X^{E}$). But the elements $e$ and $g^{-1}e$ of $E$ lie in one
and the same $g$-orbit. In other words, $\left[  e\right]  =\left[
g^{-1}e\right]  $. Now,%
\begin{align*}
\left(  g\left(  \overline{\pi}\circ\mathfrak{o}\right)  \right)  \left(
e\right)   &  =\left(  \overline{\pi}\circ\mathfrak{o}\right)  \left(
g^{-1}e\right)  =\overline{\pi}\left(  \underbrace{\mathfrak{o}\left(
g^{-1}e\right)  }_{\substack{=\left[  g^{-1}e\right]  \\\text{(by the
definition of }\mathfrak{o}\text{)}}}\right)  =\overline{\pi}\left(
\underbrace{\left[  g^{-1}e\right]  }_{=\left[  e\right]  }\right)  \\
&  =\overline{\pi}\left(  \underbrace{\left[  e\right]  }_{=\mathfrak{o}%
\left(  e\right)  }\right)  =\overline{\pi}\left(  \mathfrak{o}\left(
e\right)  \right)  =\left(  \overline{\pi}\circ\mathfrak{o}\right)  \left(
e\right)  .
\end{align*}
\par
Now, forget that we fixed $e$. We thus have proven that $\left(  g\left(
\overline{\pi}\circ\mathfrak{o}\right)  \right)  \left(  e\right)  =\left(
\overline{\pi}\circ\mathfrak{o}\right)  \left(  e\right)  $ for each $e\in E$.
In other words, $g\left(  \overline{\pi}\circ\mathfrak{o}\right)
=\overline{\pi}\circ\mathfrak{o}$.
\par
Now, $\overline{\pi}\circ\mathfrak{o}$ is a map $E\rightarrow X$ and satisfies
$g\left(  \overline{\pi}\circ\mathfrak{o}\right)  =\overline{\pi}%
\circ\mathfrak{o}$. In other words, $\overline{\pi}\circ\mathfrak{o}$ is a map
$\psi:E\rightarrow X$ satisfying $g\psi=\psi$. In other words,%
\[
\overline{\pi}\circ\mathfrak{o}\in\left\{  \psi:E\rightarrow X\ \mid
\ g\psi=\psi\right\}  =\mathfrak{A}.
\]
Qed.}.

Now, let $\Psi$ be the map%
\[
\mathfrak{B}\rightarrow\mathfrak{A},\ \ \ \ \ \ \ \ \ \ \overline{\pi}%
\mapsto\overline{\pi}\circ\mathfrak{o}.
\]
(This is well-defined, since $\overline{\pi}\circ\mathfrak{o}\in\mathfrak{A}$
for every $\overline{\pi}\in\mathfrak{B}$.)

Now, we have the equalities $\Phi\circ\Psi=\id$%
\ \ \ \ \footnote{\textit{Proof.} Let $\beta\in\mathfrak{B}$. Then,
$\Psi\left(  \beta\right)  =\beta\circ\mathfrak{o}$ (by the definition of
$\Psi$). Let $\pi=\Psi\left(  \beta\right)  $. Thus, $\pi=\Psi\left(
\beta\right)  \in\mathfrak{A}$.
\par
Now, let $u\in E^{g}$. Thus, $u$ is a $g$-orbit; hence, $u$ is nonempty. Thus,
there exists some $a\in u$. Consider such an $a$. Now,
(\ref{pf.prop.G-set.quot.Phi.1}) yields $\pi^{\circ}\left(  u\right)
=\pi\left(  a\right)  $. But $a$ is an element of the $g$-orbit $u$; thus, the
$g$-orbit of $a$ is $u$. In other words, $\left[  a\right]  =u$. The
definition of $\mathfrak{o}$ yields $\mathfrak{o}\left(  a\right)  =\left[
a\right]  =u$. Now, $\pi=\Psi\left(  \beta\right)  =\beta\circ\mathfrak{o}$,
so that $\pi\left(  a\right)  =\left(  \beta\circ\mathfrak{o}\right)  \left(
a\right)  =\beta\left(  \underbrace{\mathfrak{o}\left(  a\right)  }%
_{=u}\right)  =\beta\left(  u\right)  $. Now, $\pi^{\circ}\left(  u\right)
=\pi\left(  a\right)  =\beta\left(  u\right)  $.
\par
Now, forget that we fixed $u$. We thus have proven that $\pi^{\circ}\left(
u\right)  =\beta\left(  u\right)  $ for each $u\in E^{g}$. In other words,
$\pi^{\circ}=\beta$. But the definition of $\Phi$ yields $\Phi\left(
\pi\right)  =\pi^{\circ}=\beta$. Now, $\left(  \Phi\circ\Psi\right)  \left(
\beta\right)  =\Phi\left(  \underbrace{\Psi\left(  \beta\right)  }_{=\pi
}\right)  =\Phi\left(  \pi\right)  =\beta = \id\left(
\beta\right)  $.
\par
Now, forget that we fixed $\beta$. We thus have proven that $\left(  \Phi
\circ\Psi\right)  \left(  \beta\right)  = \id \left(
\beta\right)  $ for each $\beta\in\mathfrak{B}$. In other words, $\Phi
\circ\Psi = \id$, qed.} and $\Psi\circ\Phi = \id%
$\ \ \ \ \footnote{\textit{Proof.} Let $\alpha\in\mathfrak{A}$. Then, the
definition of $\Phi$ yields $\Phi\left(  \alpha\right)  =\alpha^{\circ}$.
\par
Now, let $a\in E$. Let $u=\mathfrak{o}\left(  a\right)  $. Thus,
$u=\mathfrak{o}\left(  a\right)  =\left[  a\right]  $ (by the definition of
$\mathfrak{o}$). In other words, $u$ is the $g$-orbit of $a$. Thus, $u$ is a
$g$-orbit and contains $a$. So we know that $u\in E^{g}$ (since $u$ is a
$g$-orbit) and that $a\in u$ (since $u$ contains $a$).
\par
The equality (\ref{pf.prop.G-set.quot.Phi.1}) (applied to $\pi=\alpha$) yields
$\alpha^{\circ}\left(  u\right)  =\alpha\left(  a\right)  $. We have $\left(
\alpha^{\circ}\circ\mathfrak{o}\right)  \left(  a\right)  =\alpha^{\circ
}\left(  \underbrace{\mathfrak{o}\left(  a\right)  }_{=u}\right)
=\alpha^{\circ}\left(  u\right)  =\alpha\left(  a\right)  $.
\par
Now, forget that we fixed $a$. We thus have shown that $\left(  \alpha^{\circ
}\circ\mathfrak{o}\right)  \left(  a\right)  =\alpha\left(  a\right)  $ for
each $a\in E$. In other words, $\alpha^{\circ}\circ\mathfrak{o}=\alpha$.
\par
But $\left(  \Psi\circ\Phi\right)  \left(  \alpha\right)  =\Psi\left(
\underbrace{\Phi\left(  \alpha\right)  }_{=\alpha^{\circ}}\right)
=\Psi\left(  \alpha^{\circ}\right)  =\alpha^{\circ}\circ\mathfrak{o}$ (by the
definition of $\Psi$). Hence, $\left(  \Psi\circ\Phi\right)  \left(
\alpha\right)  =\alpha^{\circ}\circ\mathfrak{o}=\alpha = \id
\left(  \alpha\right)  $.
\par
Now, forget that we fixed $\alpha$. We thus have proven that $\left(
\Psi\circ\Phi\right)  \left(  \alpha\right)  = \id \left(
\alpha\right)  $ for each $\alpha\in\mathfrak{A}$. In other words, $\Psi
\circ\Phi = \id$. Qed.}. These two equalities show that the maps
$\Phi$ and $\Psi$ are mutually inverse. Hence, the map $\Phi$ is invertible.
In other words, $\Phi$ is a bijection.

The map $\Phi$ is a bijection from $\mathfrak{A}$ to $\mathfrak{B}$. In other
words, the map $\Phi$ is a bijection between

\begin{itemize}
\item the maps $\pi:E\rightarrow X$ satisfying $g\pi=\pi$
\end{itemize}

and

\begin{itemize}
\item the maps $\overline{\pi}:E^{g}\rightarrow X$
\end{itemize}

(because $\mathfrak{A}$ is the set of all maps $\pi:E\rightarrow X$ satisfying
$g\pi=\pi$, whereas $\mathfrak{B}$ is the set of all maps $\overline{\pi
}:E^{g}\rightarrow X$). Furthermore, this bijection $\Phi$ sends any map
$\pi:E\rightarrow X$ satisfying $g\pi=\pi$ to the map $\overline{\pi}%
:E^{g}\rightarrow X$ defined by
\[
\overline{\pi}\left(  u\right)  =\pi\left(  a\right)  \qquad\text{for every
}u\in E^{g}\text{ and }a\in u
\]
\footnote{\textit{Proof.} Let $\pi:E\rightarrow X$ be a map satisfying
$g\pi=\pi$. We must prove that the bijection $\Phi$ sends $\pi$ to the map
$\overline{\pi}:E^{g}\rightarrow X$ defined by
\begin{equation}
\overline{\pi}\left(  u\right)  =\pi\left(  a\right)  \qquad\text{for every
}u\in E^{g}\text{ and }a\in u.\label{pf.prop.G-set.quot.Phi.fn6.1}%
\end{equation}
\par
In fact, (\ref{pf.prop.G-set.quot.Phi.1}) shows that $\pi^{\circ}\left(
u\right)  =\pi\left(  a\right)  $ for every $u\in E^{g}$ and $a\in u$. Thus,
the map $\pi^{\circ}$ is the map $\overline{\pi}:E^{g}\rightarrow X$ defined
by (\ref{pf.prop.G-set.quot.Phi.fn6.1}). Now, the bijection $\Phi$ sends $\pi$
to $\pi^{\circ}$ (by the definition of $\Phi$). In other words, the bijection
$\Phi$ sends $\pi$ to the map $\overline{\pi}:E^{g}\rightarrow X$ defined by
(\ref{pf.prop.G-set.quot.Phi.fn6.1}) (since $\pi^{\circ}$ is the map
$\overline{\pi}:E^{g}\rightarrow X$ defined by
(\ref{pf.prop.G-set.quot.Phi.fn6.1})). Qed.}. Hence, we have constructed the
bijection $\Phi$ whose existence was claimed in Proposition
\ref{prop.G-set.quot.Phi}. Thus, Proposition \ref{prop.G-set.quot.Phi} is proven.
\end{proof}
\end{verlong}

\begin{proposition}
\label{prop.G-poset.quot.Phi}
Let $\EE = \left(E, <_1, <_2\right)$ be a tertispecial double poset.
Let $G$ be a finite group which acts on $E$. Assume that $G$ preserves
both relations $<_1$ and $<_2$.

Let $g \in G$. Define the set $E^g$, the relations $<_1^g$ and $<_2^g$
and the triple $\EE^g$ as in Proposition~\ref{prop.G-poset.quot.double}.
Thus, $\EE^g$ is a tertispecial double poset (by
Proposition~\ref{prop.G-poset.quot.double}).

\begin{vershort}
There is a bijection $\Phi$ between

\begin{itemize}
\item the maps $\pi:E\rightarrow\left\{  1,2,3,\ldots\right\}  $ satisfying
$g\pi=\pi$
\end{itemize}

and

\begin{itemize}
\item the maps $\overline{\pi}:E^{g}\rightarrow\left\{  1,2,3,\ldots\right\}
$.
\end{itemize}

Namely, this bijection $\Phi$ sends any map
$\pi:E\rightarrow\left\{  1,2,3,\ldots \right\}  $ satisfying $g\pi=\pi$
to the map $\overline{\pi}:E^{g}
\rightarrow\left\{  1,2,3,\ldots\right\}  $ defined by
\[
\overline{\pi}\left( u \right)
= \pi\left( a \right)
\qquad\text{for every } u \in E^{g} \text{ and } a \in u.
\]
(The well-definedness of this map $\overline{\pi}$ is easy to see:
Indeed, from $g\pi=\pi$, we can conclude that any two elements
$a_1$ and $a_2$ of a given $g$-orbit $u$ satisfy
$\pi\left(a_1\right) = \pi\left(a_2\right)$.)
\end{vershort}

\begin{verlong}
Proposition~\ref{prop.G-set.quot.Phi} (applied to
$X = \left\{1,2,3,\ldots\right\}$) shows the following:

There is a bijection $\Phi$ between

\begin{itemize}
\item the maps $\pi:E\rightarrow\left\{  1,2,3,\ldots\right\}  $ satisfying
$g\pi=\pi$
\end{itemize}

and

\begin{itemize}
\item the maps $\overline{\pi}:E^{g}\rightarrow\left\{  1,2,3,\ldots\right\}
$.
\end{itemize}

Namely, this bijection $\Phi$ sends any map
$\pi:E\rightarrow\left\{  1,2,3,\ldots \right\}  $ satisfying $g\pi=\pi$
to the map $\overline{\pi}:E^{g}
\rightarrow\left\{  1,2,3,\ldots\right\}  $ defined by
\[
\overline{\pi}\left( u \right)
= \pi\left( a \right)
\qquad\text{for every } u \in E^{g} \text{ and } a \in u.
\]
\end{verlong}

Consider this bijection $\Phi$. Let
$\pi:E\rightarrow\left\{  1,2,3,\ldots\right\}  $ be a
map satisfying $g\pi=\pi$.

\begin{enumerate}
\item[(a)] If $\pi$ is an $\EE$-partition, then $\Phi\left(
\pi\right)  $ is an $ \EE ^{g}$-partition.

\item[(b)] If $\Phi\left(  \pi\right)  $ is an
$\EE^{g}$-partition, then $\pi$ is an $\EE$-partition.

\item[(c)] Let $w : E \to \left\{1,2,3,\ldots\right\}$ be
a map. Define a map $w^{g}:E^{g}\rightarrow
\left\{1,2,3,\ldots\right\}$
by
\[
w^{g}\left(  u\right)  =\sum\limits_{a\in u}w\left(
a\right)  \qquad \qquad \text{ for every } u \in E^g .
\]
Then, $\xx_{\Phi\left(  \pi\right)  ,w^{g}}
=\xx_{\pi,w}$.

\end{enumerate}
\end{proposition}

\begin{proof}[Proof of Proposition~\ref{prop.G-poset.quot.Phi} (sketched).]
The definition of $\Phi$ shows that
\begin{equation}
\left(\Phi\left(\pi\right)\right) \left( u \right)
= \pi\left( a \right)
\qquad\text{for every } u \in E^{g} \text{ and } a \in u.
\label{pf.prop.G-poset.quot.phi.main}
\end{equation}

(a) Assume that $\pi$ is an $\EE$-partition. We
want to show that $\Phi\left(  \pi\right)  $ is an
$ \EE ^{g}$-partition. In order to do so, we can
use Lemma \ref{lem.Epartition.cover}
(applied to $\EE^{g}$, $\left(  E^{g},<_{1}^{g},<_{2}^{g}\right)  $ and
$\Phi\left(  \pi\right)  $ instead of $\EE$,
$\left(  E,<_{1},<_{2}\right)  $ and $\phi$); we only need to
check the following two conditions:

\textit{Condition 1:} If $e\in E^{g}$ and $f\in E^{g}$ are such that $e$ is
$<_{1}^{g}$-covered by $f$, and if we have $e<_{2}^{g}f$, then $\left(
\Phi\left(  \pi\right)  \right)  \left(  e\right)  \leq\left(  \Phi\left(
\pi\right)  \right)  \left(  f\right)  $.

\textit{Condition 2:} If $e\in E^{g}$ and $f\in E^{g}$ are such that $e$ is
$<_{1}^{g}$-covered by $f$, and if we have $f<_{2}^{g}e$, then $\left(
\Phi\left(  \pi\right)  \right)  \left(  e\right)  <\left(  \Phi\left(
\pi\right)  \right)  \left(  f\right)  $.

\textit{Proof of Condition 1:} Let $e\in E^{g}$ and $f\in E^{g}$ be such that
$e$ is $<_{1}^{g}$-covered by $f$. Assume that we have $e<_{2}^{g}f$.

We have $e<_{1}^{g}f$ (because $e$ is $<_{1}^{g}$-covered by $f$). In other
words, there exist $a\in e$ and $b\in f$ satisfying $a<_{1}b$. Consider these
$a$ and $b$. Since $\pi$ is an $\EE$-partition, we have $\pi\left(
a\right)  \leq\pi\left(  b\right)  $ (since $a<_{1}b$). But the definition of
$\Phi\left(  \pi\right)  $ shows that $\left(  \Phi\left(  \pi\right)
\right)  \left(  e\right)  =\pi\left(  a\right)  $ (since $a\in e$) and
$\left(  \Phi\left(  \pi\right)  \right)  \left(  f\right)  =\pi\left(
b\right)  $ (since $b\in f$). Thus, $\left(  \Phi\left(  \pi\right)  \right)
\left(  e\right)  =\pi\left(  a\right)  \leq\pi\left(  b\right)  =\left(
\Phi\left(  \pi\right)  \right)  \left(  f\right)  $. Hence, Condition 1 is proven.

\textit{Proof of Condition 2:} Let $e\in E^{g}$ and $f\in E^{g}$ be such that
$e$ is $<_{1}^{g}$-covered by $f$. Assume that we have $f<_{2}^{g}e$.

We have $e<_{1}^{g}f$ (because $e$ is $<_{1}^{g}$-covered by $f$). In other
words, there exist $a\in e$ and $b\in f$ satisfying $a<_{1}b$. Consider these
$a$ and $b$.

\begin{vershort}
If there was a $c \in E$ satisfying $a <_1 c <_1 b$, then
the $g$-orbit $w$ of this $c$ would satisfy
$e <_1^g w <_1^g f$, which would contradict the fact that $e$ is
$<_1^g$-covered by $f$. Hence, there exists no such $c$.
\end{vershort}
\begin{verlong}
There exists no $c \in E$ satisfying $a <_1 c <_1 b$
\ \ \ \ \footnote{\textit{Proof.} Assume the contrary. Thus,
there exists some $c \in E$ satisfying $a <_1 c <_1 b$. Consider
this $c$. Let $w$ be the $g$-orbit of $c$. Thus, $w \in E^g$
and $c \in w$.
\par
Now, the elements $a \in e$ and $c \in w$ satisfy
$a <_1 c$. Hence, $e <_1^g w$ (by the definition of the
relation $<_1^g$).
\par
Also, the elements $c \in w$ and $b \in f$ satisfy
$c <_1 b$. Hence, $w <_1^g f$ (by the definition of the
relation $<_1^g$).
\par
Now, we have $e <_1^g w <_1^g f$. This contradicts the
fact that $e$ is $<_1^g$-covered by $f$. Thus, we have
obtained a contradiction; hence, our assumption was wrong.
Qed.}.
\end{verlong}
In other words, $a$ is $<_{1}$-covered by $b$ (since $a<_{1}b$). Therefore,
$a$ and $b$ are $<_{2}$-comparable (since $\EE$ is tertispecial). In
other words, we have either $a<_{2}b$ or $a=b$ or $b<_{2}a$. Since $a<_{2}b$
is impossible (because if we had $a<_{2}b$, then we would have $e<_{2}^{g}f$
(since $a\in e$ and $b\in f$), which would contradict $f<_{2}^{g}e$ (since
$<_{2}^{g}$ is a strict partial order)), and since $a=b$ is
impossible (because $a<_{1}b$), we therefore must have $b<_{2}a$. But since
$\pi$ is an $\EE$-partition, we have $\pi\left(  a\right)  <\pi\left(
b\right)  $ (since $a<_{1}b$ and $b<_{2}a$). But the definition of
$\Phi\left(  \pi\right)  $ shows that $\left(  \Phi\left(  \pi\right)
\right)  \left(  e\right)  =\pi\left(  a\right)  $ (since $a\in e$) and
$\left(  \Phi\left(  \pi\right)  \right)  \left(  f\right)  =\pi\left(
b\right)  $ (since $b\in f$). Thus, $\left(  \Phi\left(  \pi\right)  \right)
\left(  e\right)  =\pi\left(  a\right)  <\pi\left(  b\right)  =\left(
\Phi\left(  \pi\right)  \right)  \left(  f\right)  $. Hence, Condition 2 is proven.

Thus, Condition 1 and Condition 2 are proven. Hence,
Proposition~\ref{prop.G-poset.quot.Phi} (a) is proven.

(b) Assume that $\Phi\left(  \pi\right)  $ is an
$\EE^{g}$-partition. We want to show that $\pi$ is an
$\EE$-partition. In order to do so, we can use
Lemma \ref{lem.Epartition.cover}
(applied to $\phi=\pi$); we only need to check the following two conditions:

\textit{Condition 1:} If $e\in E$ and $f\in E$ are such that $e$ is
$<_{1}$-covered by $f$, and if we have $e<_{2}f$, then
$\pi\left(  e\right)  \leq \pi\left(  f\right)  $.

\textit{Condition 2:} If $e\in E$ and $f\in E$ are such that $e$ is
$<_{1}$-covered by $f$, and if we have $f<_{2}e$, then
$\pi\left(  e\right) <\pi\left(  f\right)  $.

\textit{Proof of Condition 1:} Let $e\in E$ and $f\in E$ be such that $e$ is
$<_{1}$-covered by $f$. Assume that we have $e<_{2}f$.

We have $e<_{1}f$ (since $e$ is $<_{1}$-covered by $f$). Let $u$ and $v$ be
the $g$-orbits of $e$ and $f$, respectively. Thus, $u$ and $v$ belong to
$E^{g}$, and satisfy $e \in u$ and $f \in v$. Hence,
$u<_{1}^g v$ (since $e<_{1}f$). Hence, $\left(  \Phi\left(
\pi\right)  \right)  \left(  u\right)  \leq\left(  \Phi\left(  \pi\right)
\right)  \left(  v\right)  $ (since $\Phi\left(  \pi\right)  $ is an
$\EE^{g}$-partition). But the definition of $\Phi\left(  \pi\right)  $
shows that $\left(  \Phi\left(  \pi\right)  \right)  \left(  u\right)
=\pi\left(  e\right)  $ (since $e\in u$) and $\left(  \Phi\left(  \pi\right)
\right)  \left(  v\right)  =\pi\left(  f\right)  $ (since $f\in v$). Thus,
$\pi\left(  e\right)  =\left(  \Phi\left(  \pi\right)  \right)  \left(
u\right)  \leq\left(  \Phi\left(  \pi\right)  \right)  \left(  v\right)
=\pi\left(  f\right)  $. Hence, Condition 1 is proven.

\textit{Proof of Condition 2:} Let $e\in E$ and $f\in E$ be such that $e$ is
$<_{1}$-covered by $f$. Assume that we have $f<_{2}e$.

We have $e<_{1}f$ (since $e$ is $<_{1}$-covered by $f$). Let $u$ and $v$ be
the $g$-orbits of $e$ and $f$, respectively. Thus, $u$ and $v$ belong to
$E^{g}$, and satisfy $e \in u$ and $f \in v$. Hence,
$u<_{1}^g v$ (since $e<_{1}f$) and $v<_{2}^g u$ (since
$f<_{2}e$). Hence, $\left(  \Phi\left(  \pi\right)  \right)  \left(  u\right)
<\left(  \Phi\left(  \pi\right)  \right)  \left(  v\right)  $ (since
$\Phi\left(  \pi\right)  $ is an $\EE^{g}$-partition). But the
definition of $\Phi\left(  \pi\right)  $ shows that $\left(  \Phi\left(
\pi\right)  \right)  \left(  u\right)  =\pi\left(  e\right)  $ (since $e\in
u$) and $\left(  \Phi\left(  \pi\right)  \right)  \left(  v\right)
=\pi\left(  f\right)  $ (since $f\in v$). Thus, $\pi\left(  e\right)  =\left(
\Phi\left(  \pi\right)  \right)  \left(  u\right)  <\left(  \Phi\left(
\pi\right)  \right)  \left(  v\right)  =\pi\left(  f\right)  $. Hence,
Condition 2 is proven.

Thus, Condition 1 and Condition 2 are proven. Hence,
Proposition~\ref{prop.G-poset.quot.Phi} (b) is proven.

\begin{vershort}
(c) 
The definition of $\xx_{\Phi\left(  \pi\right)  ,w^{g}}$ shows that
\begin{align*}
\xx_{\Phi\left(  \pi\right)  ,w^{g}}
&= \prod_{e\in E^{g}}
 x_{\left( \Phi\left(  \pi\right)  \right)  \left(  e\right)  }^{
    w^{g}\left(  e\right) }
= \prod_{u\in E^{g}}
 \underbrace{x_{\left(  \Phi\left(  \pi\right)  \right)
    \left(  u\right)  }^{w^{g}\left(  u\right)  }}_{
    \substack{=\prod_{a\in u}x_{\left(  \Phi\left(  \pi\right) 
    \right)  \left(  u\right)  }^{w\left( a\right)  }\\\text{(since }
    w^{g}\left(  u\right)  =\sum\limits_{a\in u}w\left(  a\right)
    \text{)}}}
= \prod_{u\in E^{g}} \prod_{a\in u}
  \underbrace{x_{\left(  \Phi\left(  \pi\right)  \right) 
    \left(  u\right) }^{w\left(  a\right)  }}_{
    \substack{=x_{\pi\left(  a\right)  }^{w\left( a\right)  }\\
    \text{(by (\ref{pf.prop.G-poset.quot.phi.main}))}}}\\
&= \underbrace{\prod_{u\in E^{g}}\prod_{a\in u}}_{=\prod_{a\in E}}
  x_{\pi\left(  a\right)  }^{w\left(  a\right)  }
=\prod_{a\in E}x_{\pi\left( a\right)  }^{w\left(  a\right)  }
=\prod_{e\in E}x_{\pi\left(  e\right) }^{w\left(  e\right)  }
=\xx_{\pi,w}%
\end{align*}
(by the definition of $\xx_{\pi,w}$). This proves
Proposition~\ref{prop.G-poset.quot.Phi} (c).
\end{vershort}
\begin{verlong}
(c) The elements of $E^g$ are the $g$-orbits on $E$. Hence, the
elements of $E^g$ are pairwise disjoint subsets of $E$, and their
union is $E$. In other words, the set $E$ is the union of its
disjoint subsets $u \in E^g$. Therefore,
$\prod_{a\in E} = \prod_{u\in E^{g}} \prod_{a\in u}$ (an equality
between product signs).

The definition of $\xx_{\Phi\left(  \pi\right)  ,w^{g}}$ shows that
\begin{align*}
\xx_{\Phi\left(  \pi\right)  ,w^{g}}
&= \prod_{e\in E^{g}}
 x_{\left( \Phi\left(  \pi\right)  \right)  \left(  e\right)  }^{
    w^{g}\left(  e\right) }
= \prod_{u\in E^{g}}
 \underbrace{x_{\left(  \Phi\left(  \pi\right)  \right)
    \left(  u\right)  }^{w^{g}\left(  u\right)  }}_{
    \substack{= x_{\left(  \Phi\left(  \pi\right) \right) \left( u
    \right) }^{\sum_{a \in u} w \left( a \right) } \\
    \text{(since }
    w^{g}\left(  u\right) = \sum\limits_{a\in u}w\left(  a\right)
    \text{)}}}
\\
&\ \ \ \ \ \ \ \ \ 
\left(\text{here, we have renamed the index } e \text{ as } u
\text{ in the product}\right) \\
&= \prod_{u\in E^{g}}
 \underbrace{x_{\left(  \Phi\left(  \pi\right)  \right)
    \left(  u\right)  }^{ \sum_{a \in u} w \left( a \right) }}_{
    \substack{=\prod_{a\in u}x_{\left(  \Phi\left(  \pi\right) 
    \right)  \left(  u\right)  }^{w\left( a\right)  }}}
= \prod_{u\in E^{g}} \prod_{a\in u}
  \underbrace{x_{\left(  \Phi\left(  \pi\right)  \right) 
    \left(  u\right) }^{w\left(  a\right)  }}_{
    \substack{=x_{\pi\left(  a\right)  }^{w\left( a\right)  }\\
    \text{(since } \left( \Phi \left( \pi \right) \right) \left(
    u \right) = \pi \left( a \right) \\
    \text{(by (\ref{pf.prop.G-poset.quot.phi.main})))}}}\\
&= \underbrace{\prod_{u\in E^{g}}\prod_{a\in u}}_{=\prod_{a\in E}}
  x_{\pi\left(  a\right)  }^{w\left(  a\right)  }
= \prod_{a\in E}x_{\pi\left( a\right)  }^{w\left(  a\right)  }
= \prod_{e\in E}x_{\pi\left(  e\right) }^{w\left(  e\right)  } \\
&\ \ \ \ \ \ \ \ \ 
\left(\text{here, we have renamed the index } a \text{ as } e
\text{ in the product}\right) \\
&= \xx_{\pi,w}%
\end{align*}
(since the definition of $\xx_{\pi,w}$ yields
$\xx_{\pi, w}
= \prod_{e\in E}x_{\pi\left(  e\right) }^{w\left(  e\right)  }$).
This proves Proposition~\ref{prop.G-poset.quot.Phi} (c).
\end{verlong}
\end{proof}

Our next lemma is a standard argument in P\'olya enumeration theory (compare
it with the proof of Burnside's lemma):

\begin{lemma}
\label{lem.burnside.sums} Let $G$ be a finite group. Let $F$ be a
$G$-set. Let $O$ be a $G$-orbit on $F$, and let $\pi\in O$.

\begin{enumerate}
\item[(a)] We have
\begin{equation}
\dfrac{1}{\left\vert O\right\vert }=\dfrac{1}{\left\vert G\right\vert }%
\sum_{\substack{g\in G;\\g\pi=\pi}}1.\label{eq.lem.burnside.sums.a}%
\end{equation}

\item[(b)] Let $E$ be a finite $G$-set. For every $g\in G$, let
$\sign_E g$ denote the sign of the permutation of $E$ that
sends every $e\in E$ to $ge$. (Thus, $g\in G$ is $E$-even if and only if
$\sign_E g = 1$.) Then,
\begin{equation}
\begin{cases}
\dfrac{1}{\left\vert O\right\vert }, & \text{if }O\text{ is }E\text{-coeven}%
;\\
0, & \text{if }O\text{ is not }E\text{-coeven}%
\end{cases}
=
\dfrac{1}{\left\vert G\right\vert }
\sum_{\substack{g\in G;\\ g\pi=\pi}} \sign_E g .
\label{eq.lem.burnside.sums.b}
\end{equation}

\end{enumerate}
\end{lemma}

\begin{proof}
[Proof of Lemma \ref{lem.burnside.sums}.] Let $\Stab_{G}\pi$
denote the stabilizer of $\pi$; this is the subgroup $\left\{  g\in
G\ \mid\ g\pi=\pi\right\}  $ of $G$. (This subgroup is also known as
the \textit{stabilizer subgroup} or the \textit{isotropy group} of
$\pi$.) The $G$-orbit of $\pi$ is $O$ (since $O$
is a $G$-orbit on $F$, and since $\pi\in O$). In other words, $O = G\pi$.
Therefore,
$\left\vert O\right\vert =\left\vert G\pi\right\vert
= \left\vert G\right\vert / \left\vert \Stab_G \pi \right\vert$
(by the orbit-stabilizer theorem). Hence,
\begin{equation}
\dfrac{1}{\left\vert O\right\vert }
= \dfrac{1}{\left\vert G\right\vert / \left\vert \Stab_G \pi \right\vert}
= \dfrac{\left\vert \Stab_G \pi \right\vert}{\left\vert G\right\vert}
.
\label{pf.thm.antipode.GammawG.os2}
\end{equation}

(a) We have%
\[
\sum_{\substack{g\in G;\\g\pi=\pi}}1=\left\vert \underbrace{\left\{  g\in
G\ \mid\ g\pi=\pi\right\}  }_{= \Stab_G \pi
}\right\vert =\left\vert  \Stab_G \pi\right\vert .
\]
Hence,
\[
\dfrac{1}{\left\vert G\right\vert }
\underbrace{\sum_{\substack{g \in G;\\ g \pi = \pi}} 1}_{
 = \left\vert \Stab_G \pi \right\vert}
= \dfrac{1}{\left\vert G\right\vert }
\left\vert \Stab_G \pi \right\vert
= \dfrac{\left\vert \Stab_G \pi \right\vert }{
\left\vert G\right\vert }
= \dfrac{1}{\left\vert O\right\vert }
\]
(by (\ref{pf.thm.antipode.GammawG.os2})). This proves
Lemma~\ref{lem.burnside.sums} (a).

(b) We need to prove (\ref{eq.lem.burnside.sums.b}). Assume first that $O$ is
$E$-coeven. Thus, all elements of $O$ are $E$-coeven (by the
definition of what it means for
$O$ to be $E$-coeven). Hence, $\pi$ is $E$-coeven (since $\pi \in O$).
This means that every $g\in G$ satisfying $g\pi=\pi$ is
$E$-even. Hence, every $g\in G$ satisfying $g\pi=\pi$ satisfies
$\sign_E g = 1$ (since $g$ is $E$-even if and only if
$\sign_E g = 1$). Thus,
\begin{align*}
\dfrac{1}{\left\vert G\right\vert }\sum_{\substack{g\in G;\\g\pi=\pi
}}\underbrace{\sign_E g}_{=1} &  =\dfrac
{1}{\left\vert G\right\vert }\sum_{\substack{g\in G;\\g\pi=\pi}}1=\dfrac
{1}{\left\vert O\right\vert }\ \ \ \ \ \ \ \ \ \ \left(  \text{by
\eqref{eq.lem.burnside.sums.a}}\right)  \\
&  =%
\begin{cases}
\dfrac{1}{\left\vert O\right\vert }, & \text{if }O\text{ is }E\text{-coeven}%
;\\
0, & \text{if }O\text{ is not }E\text{-coeven}%
\end{cases}
\ \ \ \ \ \ \ \ \ \ \left(  \text{since }O\text{ is }E\text{-coeven}\right)  .
\end{align*}

Thus, we have proven (\ref{eq.lem.burnside.sums.b}) under the assumption that
$O$ is $E$-coeven. We can therefore WLOG assume the opposite now. Thus, assume
WLOG that $O$ is not $E$-coeven. Hence, no element of $O$ is
$E$-coeven (due to the contrapositive of Lemma~\ref{lem.coeven.all-one}).
In particular, $\pi$ is not $E$-coeven (since $\pi \in O$).
In other words, not every $g\in G$ satisfying $g\pi=\pi$ is
$E$-even. In other words, not every $g\in \Stab_G \pi$
is $E$-even (since the elements $g\in G$ satisfying $g\pi=\pi$ are exactly the
elements $g\in \Stab_G \pi$). In other words, not
every $g\in \Stab_G \pi$ satisfies
$\sign_E g = 1$ (since $g$ is $E$-even if and only if
$\sign_E g = 1$).

Now, the map
\[
 \Stab_G \pi\rightarrow\left\{  1,-1\right\}
,\ \ \ \ \ \ \ \ \ \ g \mapsto \sign_E g
\]
is a group homomorphism (since the action of $G$ on $E$ is a group
homomorphism $G \to \operatorname{Aut} E$, and since
the sign of a permutation is multiplicative)
and is not the trivial homomorphism (since not every
$g \in \Stab_G \pi$ satisfies $\sign_E g = 1$). Hence, it
must send exactly half the elements of $ \Stab_G \pi$
to $1$ and the other half to $-1$. Therefore, the addends in the sum
$\sum_{g\in \Stab_G \pi}
\sign_E g$ cancel each other out (one half of them are $1$, and the
others are $-1$). Therefore,
$\sum_{g\in \Stab_G \pi} \sign_E g = 0$. Now,
\[
\dfrac{1}{\left\vert G\right\vert }\underbrace{\sum_{\substack{g\in
G;\\g\pi=\pi}}}_{=\sum_{g\in \Stab_G \pi}}
\sign_E g
= \dfrac{1}{\left\vert G\right\vert }
\underbrace{\sum_{g\in \Stab_G \pi} \sign_E g}_{=0}
= 0
=
\begin{cases}
\dfrac{1}{\left\vert O\right\vert }, & \text{if }O\text{ is }E\text{-coeven}%
;\\
0, & \text{if }O\text{ is not }E\text{-coeven}%
\end{cases}
\]
(since $O$ is not $E$-coeven).
This proves (\ref{eq.lem.burnside.sums.b}). Lemma \ref{lem.burnside.sums} (b)
is thus proven.
\end{proof}

\begin{proof}
[Proof of Theorem~\ref{thm.antipode.GammawG} (sketched).]Let $g\in G$.
Define the set $E^g$, the relations $<_1^g$ and $<_2^g$
and the triple $\EE^g$ as in Proposition~\ref{prop.G-poset.quot.double}.
Thus, $\EE^g$ is a tertispecial double poset (by
Proposition~\ref{prop.G-poset.quot.double}). In other words,
$\left( E^g, <_1^g, <_2^g \right)$ is a tertispecial double poset
(since $\EE^g = \left( E^g, <_1^g, <_2^g \right)$).

Now, forget that we fixed $g$. We thus have constructed a
tertispecial double poset
$\EE^g = \left( E^g, <_1^g, <_2^g \right)$ for every $g \in G$.

Moreover, for every $g \in G$, let us define $>_1^g$ to be the
opposite relation of $<_1^g$.

Furthermore, for every $g\in G$, define a map $w^{g}:E^{g}\rightarrow
\left\{1,2,3,\ldots\right\}$
by $w^{g}\left(  u\right)  =\sum\limits_{a\in u}w\left(
a\right)  $. (Since $G$ preserves $w$, the numbers $w\left(  a\right)  $ for
all $a\in u$ are equal (for given $u$), and thus $\sum\limits_{a\in u}w\left(
a\right)  $ can be rewritten as $\left\vert u\right\vert \cdot w\left(
b\right)  $ for any particular $b\in u$. But we shall not use
this observation.)
Now, every $g \in G$ satisfies
\begin{equation}
S\left(  \Gamma\left(  \left(  E^{g},<_{1}^{g},<_{2}^{g}\right)
,w^{g}\right)  \right)  =\left(  -1\right)  ^{\left\vert E^{g}\right\vert
}\Gamma\left(  \left(  E^{g},>_{1}^{g},<_{2}^{g}\right)  ,w^{g}\right) .
\label{pf.thm.antipode.GammawG.S1}
\end{equation}
(Indeed, this follows from Theorem~\ref{thm.antipode.Gammaw}
(applied to $\left( E^g, <_1^g, <_2^g \right)$ and $w^g$
instead of $\left( E, <_1, <_2 \right)$ and $w$)
since the double poset
$\left( E^g, <_1^g, <_2^g \right)$ is tertispecial.)

For every $g\in G$, we have
\begin{equation}
\sum_{\substack{\pi\text{ is an } \EE \text{-partition;}\\g\pi=\pi
}}\xx_{\pi,w}
= \Gamma \left( \EE^g , w^g \right)
\label{pf.thm.antipode.GammawG.red}
\end{equation}
\footnote{\textit{Proof of \eqref{pf.thm.antipode.GammawG.red}:} Let $g\in G$.
The definition of $\Gamma \left( \EE^g , w^g \right)$ yields
\begin{equation}
\Gamma\left(  \EE^{g},w^{g}\right)
= \sum_{\pi\text{ is an }
\EE ^{g}\text{-partition}}\xx_{\pi,w^{g}}
= \sum_{\overline{\pi}\text{ is an } \EE ^{g}\text{-partition}}
\xx_{\overline{\pi},w^{g}}
\label{pf.thm.antipode.GammawG.red.pf.1}
\end{equation}
(here, we have renamed the summation index $\pi$ as
$\overline{\pi}$).
\par
In Proposition~\ref{prop.G-poset.quot.Phi}, we have introduced a
bijection $\Phi$ between
\par
\begin{itemize}
\item the maps $\pi:E\rightarrow\left\{  1,2,3,\ldots\right\}  $ satisfying
$g\pi=\pi$
\end{itemize}
\par
and
\par
\begin{itemize}
\item the maps $\overline{\pi}:E^{g}\rightarrow\left\{  1,2,3,\ldots\right\}
$.
\end{itemize}
\par
Parts (a) and (b) of Proposition~\ref{prop.G-poset.quot.Phi} show that
this bijection $\Phi$ restricts to a bijection between
\par
\begin{itemize}
\item the $\EE$-partitions $\pi:E\rightarrow\left\{  1,2,3,\ldots
\right\}  $ satisfying $g\pi=\pi$
\end{itemize}
\par
and
\par
\begin{itemize}
\item the $\EE ^g$-partitions $\overline{\pi}:E^{g}\rightarrow\left\{
1,2,3,\ldots\right\}  $.
\end{itemize}
\par
Hence, we can substitute $\Phi\left(\pi\right)$ for
$\overline{\pi}$ in the sum
$\sum_{\overline{\pi} \text{ is an } \EE ^g\text{-partition}}
\xx_{\overline{\pi},w^{g}}$. We thus obtain
\[
\sum_{\overline{\pi}\text{ is an } \EE ^g\text{-partition}}
\xx_{\overline{\pi},w^{g}}
=\sum_{\substack{\pi\text{ is an } \EE \text{-partition;}\\g\pi=\pi
}}
\underbrace{\xx_{\Phi\left(  \pi\right)  ,w^{g}}}_{
\substack{=\xx_{\pi,w} \\
\text{(by Proposition~\ref{prop.G-poset.quot.Phi} (c))}
}}
= \sum_{\substack{\pi\text{ is an } \EE \text{-partition;}
\\ g\pi=\pi}}\xx_{\pi,w},
\]
whence $\sum_{\substack{\pi\text{ is an } \EE \text{-partition;}
\\g\pi=\pi}}\xx_{\pi,w}
=\sum_{\overline{\pi}\text{ is an }
\EE ^{g}\text{-partition}}\xx_{\overline{\pi},w^{g}}
=\Gamma\left(  \EE^{g},w^{g}\right)  $
(by \eqref{pf.thm.antipode.GammawG.red.pf.1}).
This proves \eqref{pf.thm.antipode.GammawG.red}.}.

It is clearly sufficient to prove
Theorem~\ref{thm.antipode.GammawG} for $\kk = \ZZ$ (since all
the power series that we are discussing are defined functorially
in $\kk$ (and so are the Hopf algebra $\QSym$ and its antipode
$S$), and thus any identity between these series that holds
over $\ZZ$ must hold over any $\kk$). Therefore, it is sufficient
to prove Theorem~\ref{thm.antipode.GammawG} for $\kk = \QQ$ (since
$\ZZ\left[\left[x_1,x_2,x_3,\ldots\right]\right]$ embeds into
$\QQ\left[\left[x_1,x_2,x_3,\ldots\right]\right]$, and using
this embedding we have
$\QSym_{\ZZ} = \QSym_{\QQ} \cap
\ZZ\left[\left[x_1,x_2,x_3,\ldots\right]\right]$
\ \ \ \ \footnote{Here, we
are using the notation $\QSym_{\kk}$ for the Hopf algebra $\QSym$
defined over a base ring $\kk$.}).
Thus, we WLOG assume that $\kk = \QQ$.
This will allow us to divide by positive integers.

Every $G$-orbit $O$ on $\Par \EE$ satisfies
\begin{equation}
\dfrac{1}{\left|O\right|} \sum_{\pi \in O}
\underbrace{\xx_{\pi, w}}_{\substack{= \xx_{O, w} \\
                           \text{(since } \xx_{O, w}
                           \text{ is defined} \\
                           \text{to be } \xx_{\pi, w}
                           \text{)}}}
= \dfrac{1}{\left|O\right|} \underbrace{\sum_{\pi \in O} \xx_{O, w}}_{
                                        = \left|O\right| \xx_{O, w}}
= \dfrac{1}{\left|O\right|} \left|O\right| \xx_{O, w}
= \xx_{O, w} .
\label{pf.thm.antipode.GammawG.averaging}
\end{equation}

Now,
\begin{align}
\Gamma\left( \EE , w, G\right)   &  =\sum_{O\text{ is a
}G\text{-orbit on } \Par \EE }\underbrace{\xx_{O,w}}_{
\substack{=\dfrac{1}{\left\vert O\right\vert }\sum\limits_{\pi\in
O}\xx_{\pi,w}\\\text{(by \eqref{pf.thm.antipode.GammawG.averaging})}%
}}=\sum_{O\text{ is a }G\text{-orbit on } \Par \EE }
\dfrac{1}{\left\vert O\right\vert }\sum\limits_{\pi\in O}\xx_{\pi
,w}\nonumber\\
&  =\sum_{O\text{ is a }G\text{-orbit on } \Par \EE
}\sum\limits_{\pi\in O}\underbrace{\dfrac{1}{\left\vert O\right\vert }%
}_{\substack{=\dfrac{1}{\left\vert G\right\vert }\sum_{\substack{g\in
G;\\g\pi=\pi}}1\\\text{(by \eqref{eq.lem.burnside.sums.a}, applied to }
F = \Par\EE {)}}}\xx_{\pi,w}
\nonumber\\
& = \underbrace{\sum_{O \text{ is a } G\text{-orbit on } \Par \EE}
\sum\limits_{\pi\in O}}_{=\sum_{\pi\in \Par \EE }
=\sum_{\pi\text{ is an } \EE \text{-partition}}}
\left(
\dfrac{1}{\left\vert G\right\vert }\sum_{\substack{g\in G;\\g\pi=\pi
}}1\right)  \xx_{\pi,w}\nonumber\\
&  =\sum_{\pi\text{ is an } \EE \text{-partition}}\left(  \dfrac
{1}{\left\vert G\right\vert }\sum_{\substack{g\in G;\\g\pi=\pi}}1\right)
\xx_{\pi,w}=\dfrac{1}{\left\vert G\right\vert }\underbrace{\sum
_{\pi\text{ is an } \EE \text{-partition}}\sum_{\substack{g\in
G;\\g\pi=\pi}}}_{=\sum_{g\in G}\sum_{\substack{\pi\text{ is an }\EE
\text{-partition;}\\g\pi=\pi}}}\xx_{\pi,w}\nonumber\\
&  =\dfrac{1}{\left\vert G\right\vert }\sum_{g\in G}\underbrace{\sum
_{\substack{\pi\text{ is an } \EE \text{-partition;}\\g\pi=\pi
}}\xx_{\pi,w}}_{\substack{=\Gamma\left(   \EE ^{g},w^{g}\right)
\\\text{(by \eqref{pf.thm.antipode.GammawG.red})}}}\nonumber\\
&  =\dfrac{1}{\left\vert G\right\vert }\sum_{g\in G}\Gamma\left(
\underbrace{ \EE ^{g}}_{=\left(  E^{g},<_{1}^{g},<_{2}^{g}\right)
},w^{g}\right)  \nonumber\\
&=\dfrac{1}{\left\vert G\right\vert }\sum_{g\in G}\Gamma\left(
\left(  E^{g},<_{1}^{g},<_{2}^{g}\right)  ,w^{g}\right)
.
\label{pf.thm.antipode.GammawG.1}
\end{align}
Hence, $\Gamma\left(  { \EE },w,G\right) \in \QSym$
(by Proposition~\ref{prop.Gammaw.qsym}).

\begin{vershort}
Applying the map $S$ to both sides of the equality
\eqref{pf.thm.antipode.GammawG.1}, we obtain
\begin{align}
S\left(  \Gamma\left(  { \EE },w,G\right)  \right)
& = \dfrac{1}{
\left\vert G\right\vert }\sum_{g\in G}\underbrace{S\left(  \Gamma\left(
\left(  E^{g},<_{1}^{g},<_{2}^{g}\right)  ,w^{g}\right)  \right)
}_{\substack{=\left(  -1\right)  ^{\left\vert E^{g}\right\vert }\Gamma\left(
\left(  E^{g},>_{1}^{g},<_{2}^{g}\right)  ,w^{g}\right)  \\\text{(by
\eqref{pf.thm.antipode.GammawG.S1})}}}\nonumber\\
& =\dfrac{1}{\left\vert G\right\vert }\sum_{g\in G}\left(  -1\right)
^{\left\vert E^{g}\right\vert }
\Gamma\left(  \left(  E^{g},>_{1}^{g},<_{2}^{g}\right)  ,w^{g}\right)
.
\label{pf.thm.antipode.GammawG.1S}
\end{align}
\end{vershort}
\begin{verlong}
Applying the map $S$ to both sides of the equality
\eqref{pf.thm.antipode.GammawG.1}, we obtain
\begin{align}
S\left(  \Gamma\left(  { \EE },w,G\right)  \right)
& = S \left( \dfrac{1}{
\left\vert G\right\vert }\sum_{g\in G} \Gamma\left(
\left(  E^{g},<_{1}^{g},<_{2}^{g}\right)  ,w^{g}\right)
\right)
= \dfrac{1}{
\left\vert G\right\vert }\sum_{g\in G}\underbrace{S\left(  \Gamma\left(
\left(  E^{g},<_{1}^{g},<_{2}^{g}\right)  ,w^{g}\right)  \right)
}_{\substack{=\left(  -1\right)  ^{\left\vert E^{g}\right\vert }\Gamma\left(
\left(  E^{g},>_{1}^{g},<_{2}^{g}\right)  ,w^{g}\right)  \\\text{(by
\eqref{pf.thm.antipode.GammawG.S1})}}}\nonumber\\
& =\dfrac{1}{\left\vert G\right\vert }\sum_{g\in G}\left(  -1\right)
^{\left\vert E^{g}\right\vert }
\Gamma\left(  \left(  E^{g},>_{1}^{g},<_{2}^{g}\right)  ,w^{g}\right)
.
\label{pf.thm.antipode.GammawG.1S}
\end{align}
\end{verlong}

On the other hand, for every $g\in G$, let $\sign_E g$
denote the sign of the permutation of $E$ that sends every
$e\in E$ to $ge$. Thus, $g\in G$ is $E$-even if and only if
$\sign_E g = 1$. Now, every $G$-orbit $O$ on $\Par \EE $
and every $\pi\in O$ satisfy%
\begin{equation}%
\begin{cases}
\dfrac{1}{\left\vert O\right\vert }, & \text{if }O\text{ is }E\text{-coeven};\\
0, & \text{if }O\text{ is not }E\text{-coeven}%
\end{cases}
=\dfrac{1}{\left\vert G\right\vert }\sum_{\substack{g\in G; \\ g\pi = \pi}}
\sign_E g
\label{pf.thm.antipode.GammawG.signed}
\end{equation}
(by \eqref{eq.lem.burnside.sums.b}, applied to $F = \Par\EE$). Furthermore,
\begin{equation}
\sign_E g = \left(  -1\right)  ^{\left\vert
E\right\vert -\left\vert E^{g}\right\vert }
\label{pf.thm.antipode.GammawG.sign}
\end{equation}
for every $g\in G$\ \ \ \ \footnote{\textit{Proof of
\eqref{pf.thm.antipode.GammawG.sign}:} Let $g\in G$. Recall that
$\sign_E g$ is the sign of the permutation of $E$
that sends every $e\in E$ to $ge$. Denote this permutation by $\zeta$.
Thus, $\sign_E g$ is the sign of $\zeta$.
\par
The permutation $\zeta$ is the permutation of $E$
that sends every $e \in E$ to $ge$. In other words,
$\zeta$ is the action of $g$ on $E$.
Hence, the cycles of $\zeta$ are the $g$-orbits on $E$.
Thus, the set of all cycles of $\zeta$ is the set
of all $g$-orbits on $E$; this latter set is $E^g$.
Hence, $E^g$ is the set of all cycles of $\zeta$.
\par
But if $\sigma$ is a permutation of a
finite set $X$, then the sign of $\sigma$ is $\left(  -1\right)  ^{\left\vert
X\right\vert -\left\vert X^{\sigma}\right\vert }$, where $X^{\sigma}$ is the
set of all cycles of $\sigma$. Applying this to $X=E$,
$\sigma=\zeta$ and $X^{\sigma}=E^{g}$, we see that the sign of $\zeta$
is $\left(  -1\right)  ^{\left\vert
E\right\vert -\left\vert E^{g}\right\vert }$
(because $E^g$ is the set of all cycles of $\zeta$).
In other words,
$\sign_E g = \left(  -1\right)  ^{\left\vert
E\right\vert -\left\vert E^{g}\right\vert }$ (since
$\sign_E g$ is the sign of $\zeta$), qed.}.

\begin{vershort}
Now,%
\begin{align}
&  \Gamma^{+}\left(   \EE ,w,G\right)  \nonumber\\
&  =\sum_{O\text{ is an }E\text{-coeven }G\text{-orbit on } \Par
 \EE }\underbrace{\xx_{O,w}}_{\substack{=\dfrac{1}{\left\vert
O\right\vert }\sum\limits_{\pi\in O}\xx_{\pi,w}\\\text{(by
\eqref{pf.thm.antipode.GammawG.averaging})}}}=\sum_{O\text{ is an
}E\text{-coeven }G\text{-orbit on } \Par \EE }\dfrac
{1}{\left\vert O\right\vert }\sum\limits_{\pi\in O}\xx_{\pi
,w}\nonumber\\
&  =\sum_{O\text{ is a }G\text{-orbit on } \Par \EE }
\begin{cases}
\dfrac{1}{\left\vert O\right\vert }, & \text{if }O\text{ is }E\text{-coeven}%
;\\
0, & \text{if }O\text{ is not }E\text{-coeven}%
\end{cases}
\sum\limits_{\pi\in O}\xx_{\pi,w}\nonumber\\
&  \qquad\left(
\begin{array}
[c]{c}%
\text{here, we have extended the sum to all }G\text{-orbits on } \Par \EE
\\
\text{ (not just the }E\text{-coeven ones); but all new addends are }0\\
\text{and therefore do not influence the value of the sum}%
\end{array}
\right)  \nonumber\\
&  =\sum_{O\text{ is a }G\text{-orbit on } \Par \EE }
\sum\limits_{\pi\in O}\underbrace{%
\begin{cases}
\dfrac{1}{\left\vert O\right\vert }, & \text{if }O\text{ is }E\text{-coeven}%
;\\
0, & \text{if }O\text{ is not }E\text{-coeven}%
\end{cases}
}_{\substack{=\dfrac{1}{\left\vert G\right\vert }\sum_{\substack{g\in
G;\\g\pi=\pi}} \sign_E g \\\text{(by
\eqref{pf.thm.antipode.GammawG.signed})}}}\xx_{\pi,w}\nonumber\\
&  =\underbrace{\sum_{O\text{ is a }G\text{-orbit on } \Par
\EE }\sum\limits_{\pi\in O}}_{=\sum_{\pi\in \Par
\EE }=\sum_{\pi\text{ is an } \EE \text{-partition}}}\left(
\dfrac{1}{\left\vert G\right\vert }\sum_{\substack{g\in G;\\g\pi=\pi
}} \sign_E g \right)  \xx_{\pi,w}\nonumber\\
&  =\sum_{\pi\text{ is an } \EE \text{-partition}}\left(  \dfrac
{1}{\left\vert G\right\vert }\sum_{\substack{g\in G;\\g\pi=\pi}%
} \sign_E g \right)  \xx_{\pi,w}=\dfrac
{1}{\left\vert G\right\vert }\underbrace{\sum_{\pi\text{ is an }
\EE\text{-partition}}
\sum_{\substack{g\in G;\\g\pi=\pi}}}_{=\sum_{g\in G}%
\sum_{\substack{\pi\text{ is an } \EE \text{-partition;}\\g\pi=\pi}%
}}\left(   \sign_E g \right)  \xx_{\pi,w}
\nonumber\\
&  =\dfrac{1}{\left\vert G\right\vert }\sum_{g\in G}%
\underbrace{ \sign_E g }_{\substack{=\left(  -1\right)
^{\left\vert E\right\vert -\left\vert E^{g}\right\vert }\\\text{(by
\eqref{pf.thm.antipode.GammawG.sign})}}}\underbrace{\sum_{\substack{\pi\text{
is an } \EE \text{-partition;}\\g\pi=\pi}}\xx_{\pi,w}%
}_{\substack{=\Gamma\left(   \EE ^{g},w^{g}\right)  \\\text{(by
\eqref{pf.thm.antipode.GammawG.red})}}}
=\dfrac{1}{\left\vert G\right\vert }\sum_{g\in G}\left(  -1\right)
^{\left\vert E\right\vert -\left\vert E^{g}\right\vert }\Gamma\left(
 \underbrace{\EE^{g}}_{=\left(E^g, <_1^g, <_2^g\right)},w^{g}\right)
\nonumber\\
&= \dfrac{1}{\left\vert G\right\vert }\sum_{g\in G}\left(  -1\right)
^{\left\vert E\right\vert -\left\vert E^{g}\right\vert }\Gamma\left(
 \left(E^g, <_1^g, <_2^g\right),w^{g}\right) .
\label{pf.thm.antipode.GammawG.Gammaplus}
\end{align}
Hence, $\Gamma^+\left( \EE ,w,G\right) \in \QSym$
(by Proposition~\ref{prop.Gammaw.qsym}).
\end{vershort}

\begin{verlong}
Now,%
\begin{align*}
&  \Gamma^{+}\left(   \EE ,w,G\right)  \\
&  =\sum_{O\text{ is an }E\text{-coeven }G\text{-orbit on } \Par
 \EE }\underbrace{\xx_{O,w}}_{\substack{=\dfrac{1}{\left\vert
O\right\vert }\sum\limits_{\pi\in O}\xx_{\pi,w}\\\text{(by
\eqref{pf.thm.antipode.GammawG.averaging})}}}
=\underbrace{\sum_{O\text{ is an
}E\text{-coeven }G\text{-orbit on } \Par \EE }}%
_{=\sum_{\substack{O\text{ is a }G\text{-orbit on } \Par
 \EE ;\\O\text{ is }E\text{-coeven}}}}\dfrac{1}{\left\vert O\right\vert
}\sum\limits_{\pi\in O} \xx_{\pi,w}\\
&  =\sum_{\substack{O\text{ is a }G\text{-orbit on } \Par
 \EE ;\\O\text{ is }E\text{-coeven}}}\dfrac{1}{\left\vert O\right\vert
}\sum\limits_{\pi\in O} \xx_{\pi,w}\\
&  =\sum_{\substack{O\text{ is a }G\text{-orbit on } \Par
 \EE ;\\O\text{ is }E\text{-coeven}}}\underbrace{\dfrac{1}{\left\vert
O\right\vert }}_{\substack{=%
\begin{cases}
\dfrac{1}{\left\vert O\right\vert }, & \text{if }O\text{ is }E\text{-coeven}%
;\\
0, & \text{if }O\text{ is not }E\text{-coeven}%
\end{cases}
\\\text{(since }O\text{ is }E\text{-coeven)}}}\sum\limits_{\pi\in
O} \xx_{\pi,w} \\
& \ \ \ \ \ \ \ \ \ \ +\sum_{\substack{O\text{ is a }G\text{-orbit on
} \Par \EE ;\\O\text{ is not }E\text{-coeven}%
}}\underbrace{0}_{\substack{=%
\begin{cases}
\dfrac{1}{\left\vert O\right\vert }, & \text{if }O\text{ is }E\text{-coeven}%
;\\
0, & \text{if }O\text{ is not }E\text{-coeven}%
\end{cases}
\\\text{(since }O\text{ is not }E\text{-coeven)}}}\sum\limits_{\pi\in
O} \xx_{\pi,w}\\
&  \ \ \ \ \ \ \ \ \ \ \left(
\begin{array}
[c]{c}%
\text{since }\sum_{\substack{O\text{ is a }G\text{-orbit on }%
 \Par \EE ;\\O\text{ is }E\text{-coeven}}}\dfrac
{1}{\left\vert O\right\vert }\sum\limits_{\pi\in O} \xx_{\pi
,w}+\underbrace{\sum_{\substack{O\text{ is a }G\text{-orbit on }%
 \Par \EE ;\\O\text{ is not }E\text{-coeven}}%
}0\sum\limits_{\pi\in O} \xx_{\pi,w}}_{=0}\\
=\sum_{\substack{O\text{ is a }G\text{-orbit on } \Par
 \EE ;\\O\text{ is }E\text{-coeven}}}\dfrac{1}{\left\vert O\right\vert
}\sum\limits_{\pi\in O} \xx_{\pi,w}%
\end{array}
\right)  \\
&  =\sum_{\substack{O\text{ is a }G\text{-orbit on } \Par
 \EE ;\\O\text{ is }E\text{-coeven}}}%
\begin{cases}
\dfrac{1}{\left\vert O\right\vert }, & \text{if }O\text{ is }E\text{-coeven}%
;\\
0, & \text{if }O\text{ is not }E\text{-coeven}%
\end{cases}
\sum\limits_{\pi\in O} \xx_{\pi,w} \\
& \ \ \ \ \ \ \ \ \ \ +\sum_{\substack{O\text{ is a
}G\text{-orbit on } \Par \EE ;\\O\text{ is not
}E\text{-coeven}}}%
\begin{cases}
\dfrac{1}{\left\vert O\right\vert }, & \text{if }O\text{ is }E\text{-coeven}%
;\\
0, & \text{if }O\text{ is not }E\text{-coeven}%
\end{cases}
\sum\limits_{\pi\in O} \xx_{\pi,w}\\
&  =\sum_{O\text{ is a }G\text{-orbit on } \Par \EE }%
\begin{cases}
\dfrac{1}{\left\vert O\right\vert }, & \text{if }O\text{ is }E\text{-coeven}%
;\\
0, & \text{if }O\text{ is not }E\text{-coeven}%
\end{cases}
\sum\limits_{\pi\in O} \xx_{\pi,w}
\end{align*}
\begin{align}
&  =\sum_{O\text{ is a }G\text{-orbit on } \Par \EE }
\sum\limits_{\pi\in O}\underbrace{%
\begin{cases}
\dfrac{1}{\left\vert O\right\vert }, & \text{if }O\text{ is }E\text{-coeven}%
;\\
0, & \text{if }O\text{ is not }E\text{-coeven}%
\end{cases}
}_{\substack{=\dfrac{1}{\left\vert G\right\vert }\sum_{\substack{g\in
G;\\g\pi=\pi}} \sign_E g \\\text{(by
\eqref{pf.thm.antipode.GammawG.signed})}}}\xx_{\pi,w}\nonumber\\
&  =\underbrace{\sum_{O\text{ is a }G\text{-orbit on } \Par
\EE }\sum\limits_{\pi\in O}}_{=\sum_{\pi\in \Par
\EE }=\sum_{\pi\text{ is an } \EE \text{-partition}}}\left(
\dfrac{1}{\left\vert G\right\vert }\sum_{\substack{g\in G;\\g\pi=\pi
}} \sign_E g \right)  \xx_{\pi,w}\nonumber\\
&  =\sum_{\pi\text{ is an } \EE \text{-partition}}\left(  \dfrac
{1}{\left\vert G\right\vert }\sum_{\substack{g\in G;\\g\pi=\pi}%
} \sign_E g \right)  \xx_{\pi,w}=\dfrac
{1}{\left\vert G\right\vert }\underbrace{\sum_{\pi\text{ is an }
\EE\text{-partition}}
\sum_{\substack{g\in G;\\g\pi=\pi}}}_{=\sum_{g\in G}%
\sum_{\substack{\pi\text{ is an } \EE \text{-partition;}\\g\pi=\pi}%
}}\left(   \sign_E g \right)  \xx_{\pi,w}
\nonumber\\
&  =\dfrac{1}{\left\vert G\right\vert }\sum_{g\in G}%
\underbrace{ \sign_E g }_{\substack{=\left(  -1\right)
^{\left\vert E\right\vert -\left\vert E^{g}\right\vert }\\\text{(by
\eqref{pf.thm.antipode.GammawG.sign})}}}\underbrace{\sum_{\substack{\pi\text{
is an } \EE \text{-partition;}\\g\pi=\pi}}\xx_{\pi,w}%
}_{\substack{=\Gamma\left(   \EE ^{g},w^{g}\right)  \\\text{(by
\eqref{pf.thm.antipode.GammawG.red})}}}
=\dfrac{1}{\left\vert G\right\vert }\sum_{g\in G}\left(  -1\right)
^{\left\vert E\right\vert -\left\vert E^{g}\right\vert }\Gamma\left(
 \underbrace{\EE^{g}}_{=\left(E^g, <_1^g, <_2^g\right)},w^{g}\right)
\nonumber\\
&= \dfrac{1}{\left\vert G\right\vert }\sum_{g\in G}\left(  -1\right)
^{\left\vert E\right\vert -\left\vert E^{g}\right\vert }\Gamma\left(
 \left(E^g, <_1^g, <_2^g\right),w^{g}\right) .
\label{pf.thm.antipode.GammawG.Gammaplus}
\end{align}
Hence, $\Gamma^+\left( \EE ,w,G\right) \in \QSym$
(by Proposition~\ref{prop.Gammaw.qsym}).
\end{verlong}

The group $G$ preserves the relation $>_1$ (since it preserves the
relation $<_1$).
\begin{vershort}
Furthermore, the double poset $\left( E, >_1, <_2 \right)$ is
tertispecial\footnote{This can be easily derived from the fact
that $\left( E, <_1, <_2 \right)$ is tertispecial. (Observe that
an $a \in E$ is $>_1$-covered by a $b \in E$ if and only if $b$
is $<_1$-covered by $a$.)}.
\end{vershort}
\begin{verlong}
Furthermore, Lemma~\ref{lem.tertispecial.op} shows that
$\left( E, >_1, <_2 \right)$ is a tertispecial double poset.
\end{verlong}
Hence, we can apply
\eqref{pf.thm.antipode.GammawG.Gammaplus}
to $\left(  E,>_{1},<_{2}\right)  $, $>_1$ and $>_1^g$
instead of $\EE$, $<_1$ and $<_1^g$. As a result, we obtain
\[
\Gamma^{+}\left(  \left(  E,>_{1},<_{2}\right)  ,w,G\right)  =\dfrac
{1}{\left\vert G\right\vert }\sum_{g\in G}\left(  -1\right)  ^{\left\vert
E\right\vert -\left\vert E^{g}\right\vert }\Gamma\left(  \left(  E^{g}%
,>_{1}^{g},<_{2}^{g}\right)  ,w^{g}\right)  .
\]
Multiplying both sides of this equality by $\left(  -1\right)  ^{\left\vert
E\right\vert }$, we transform it into
\begin{align*}
\left(  -1\right)  ^{\left\vert E\right\vert }\Gamma^{+}\left(  \left(
E,>_{1},<_{2}\right)  ,w,G\right)
& = \left(  -1\right)  ^{\left\vert E\right\vert}
\dfrac{1}{\left\vert G\right\vert }\sum_{g\in G}
\left(  -1\right)  ^{\left\vert E\right\vert -\left\vert E^{g}\right\vert }%
\Gamma\left(  \left(
E^{g},>_{1}^{g},<_{2}^{g}\right)  ,w^{g}\right)  \\
& =\dfrac{1}{\left\vert G\right\vert
}\sum_{g\in G}\underbrace{\left(  -1\right)  ^{\left\vert E\right\vert
}\left(  -1\right)  ^{\left\vert E\right\vert -\left\vert E^{g}\right\vert }%
}_{=\left(  -1\right)  ^{\left\vert E^{g}\right\vert }}\Gamma\left(  \left(
E^{g},>_{1}^{g},<_{2}^{g}\right)  ,w^{g}\right)  \\
& =\dfrac{1}{\left\vert G\right\vert }\sum_{g\in G}\left(  -1\right)
^{\left\vert E^{g}\right\vert }\Gamma\left(  \left(  E^{g},>_{1}^{g},<_{2}%
^{g}\right)  ,w^{g}\right)  \\
& =S\left(  \Gamma\left(  { \EE },w,G\right)  \right)
\ \ \ \ \ \ \ \ \ \ 
\left(  \text{by \eqref{pf.thm.antipode.GammawG.1S}}\right)  .
\end{align*}
This completes the proof of Theorem~\ref{thm.antipode.GammawG}.
\end{proof}

\section{Application: Jochemko's theorem}
\label{sect.jochemko}

We shall now demonstrate an application of Theorem \ref{thm.antipode.GammawG}:
namely, we will use it to provide an alternative proof of \cite[Theorem
2.13]{Joch}. The way we derive \cite[Theorem 2.13]{Joch} from Theorem
\ref{thm.antipode.GammawG} is classical, and in fact was what originally
motivated the discovery of Theorem \ref{thm.antipode.GammawG} (although, of
course, it cannot be conversely derived from \cite[Theorem 2.13]{Joch}, so it
is an actual generalization).

An intermediate step between \cite[Theorem 2.13]{Joch} and Theorem
\ref{thm.antipode.GammawG} will be the following fact:

\begin{corollary}
\label{cor.reciprocity.GammawG}Let $ \EE =\left(  E,<_{1},<_{2}\right)
$ be a tertispecial double poset. Let $w:E\rightarrow\left\{  1,2,3,\ldots
\right\}  $. Let $G$ be a finite group which acts on $E$. Assume that $G$
preserves both relations $<_{1}$ and $<_{2}$, and also preserves $w$. For
every $q\in \NN $, let $\Par_q \EE$
denote the set of all $\EE$-partitions whose image is contained in
$\left\{  1,2,\ldots,q\right\}  $. Then, the group $G$ also acts on
$\Par_q \EE$; namely,
$\Par_q \EE$ is a $G$-subset of the $G$-set $\left\{  1,2,\ldots
,q\right\}  ^{E}$ (see Definition~\ref{def.G-sets.terminology} (d) for the
definition of the latter).

\begin{enumerate}
\item[(a)] There exists a unique polynomial $\Omega_{\EE, G}
\in \QQ \left[  X\right]  $ such that every $q\in \NN $ satisfies%
\begin{equation}
\Omega_{\EE, G}\left(  q\right)  =\left(  \text{the number of all
}G\text{-orbits on } \Par_q \EE \right)  .
\label{eq.cor.reciprocity.GammawG.a.def}
\end{equation}

\item[(b)] This polynomial satisfies%
\begin{align}
&  \Omega_{\EE,G}\left(  -q\right) \nonumber\\
&  =\left(  -1\right)  ^{\left\vert E\right\vert }\left(  \text{the number of
all $E$-coeven }G\text{-orbits on } \Par_q \left(
E,>_{1},<_{2}\right)  \right) \nonumber\\
&  =\left(  -1\right)  ^{\left\vert E\right\vert }\left(  \text{the number of
all $E$-coeven }G\text{-orbits on } \Par_q \left(
E,<_{1},>_{2}\right)  \right)
\label{eq.cor.reciprocity.GammawG.b.2}
\end{align}
for all $q\in \NN $.
\end{enumerate}
\end{corollary}

\begin{proof}
[Proof of Corollary \ref{cor.reciprocity.GammawG} (sketched).]
Set $\kk = \QQ$. For any $f \in \QSym$ and any
$q\in \NN $, we define an element $\operatorname{ps}^{1}\left(
f\right)  \left(  q\right)  \in \QQ$ by
\[
\operatorname{ps}^{1}\left(  f\right)  \left(  q\right)  =f\left(
\underbrace{1,1,\ldots,1}_{q\text{ times}},0,0,0,\ldots\right)
\]
(that is, $\operatorname{ps}^{1}\left(  f\right)  \left(  q\right)
$ is the result of substituting $1$ for each of the variables
$x_{1},x_{2},\ldots,x_{q}$ and $0$ for each of the variables
$x_{q+1},x_{q+2},x_{q+3},\ldots$ in the power series $f$).

(a) Consider the elements $\Gamma\left(   \EE ,w,G\right)  $ and
$\Gamma^{+}\left( \EE ,w,G \right)  $ of $\QSym$
defined in Theorem~\ref{thm.antipode.GammawG}. Observe that
$\Par_q \EE$ is a $G$-subset of $\Par \EE$.

\begin{noncompile}
Clearly, there exists \textbf{at most} one polynomial $\Omega_{\EE,
G} \in \QQ \left[  X\right]  $ such that every $q\in \NN $ satisfies
(\ref{eq.cor.reciprocity.GammawG.a.def}) (because a polynomial in
$\QQ \left[  X\right]  $ is uniquely determined by its values at all
nonnegative integers). It remains to show that there exists \textbf{at least}
one such polynomial.
\end{noncompile}

Now, \cite[Proposition 7.1.7 (i)]{Reiner} shows that, for any given
$f\in \QSym $, there exists a unique polynomial in $\QQ
\left[  X\right]  $ whose value on each $q\in \NN $ equals
$\operatorname{ps}^{1}\left(  f\right)  \left(  q\right)  $.
Applying this to $f=\Gamma\left(   \EE ,w,G\right)  $, we conclude that
there exists a unique polynomial in $\QQ \left[  X\right]  $ whose value
on each $q\in \NN $ equals $\operatorname{ps}^{1}\left(
\Gamma\left(   \EE ,w,G\right)  \right)  \left(  q\right)  $. But since
every $q\in \NN $ satisfies
\begin{align}
\operatorname{ps}^{1}\left(  \Gamma\left(   \EE ,w,G\right)
\right)  \left(  q\right)   &  =\underbrace{\left(  \Gamma\left( \EE
,w,G\right)  \right)  }_{=\sum_{O\text{ is a }G\text{-orbit on }
\Par \EE } \xx_{O,w}}\left(
\underbrace{1,1,\ldots,1}_{q\text{ times}},0,0,0,\ldots\right) \nonumber\\
&  =\sum_{O\text{ is a }G\text{-orbit on } \Par \EE}
\underbrace{\xx_{O,w}\left(  \underbrace{1,1,\ldots,1}_{q\text{
times}},0,0,0,\ldots\right)  }_{=%
\begin{cases}
1, & \text{if }O\subseteq \Par_q \EE ;\\
0, & \text{if }O\not \subseteq \Par_q \EE
\end{cases}
}\nonumber\\
&  =\sum_{O\text{ is a }G\text{-orbit on } \Par \EE }
\begin{cases}
1, & \text{if }O\subseteq \Par_q \EE ;\\
0, & \text{if }O\not \subseteq \Par_q \EE
\end{cases}
\nonumber\\
& = \sum_{O\text{ is a }G\text{-orbit on } \Par_{q} \EE} 1
=\left(  \text{the number of all }G\text{-orbits on }
\Par_q \EE \right)  ,
\label{pf.cor.reciprocity.GammawG.a.1}
\end{align}
this rewrites as follows: There exists a unique polynomial in $\QQ
\left[  X\right]  $ whose value on each $q\in \NN $ equals $\left(
\text{the number of all }G\text{-orbits on } \Par_q \EE
\right)  $. This proves Corollary~\ref{cor.reciprocity.GammawG} (a).

(b) \cite[Proposition 7.1.7 (i)]{Reiner} shows that, for any given
$f\in \QSym $, there exists a unique polynomial in $\QQ
\left[  X\right]  $ whose value on each $q\in \NN $ equals
$\operatorname{ps}^{1}\left(  f\right)  \left(  q\right)  $. This
polynomial is denoted by $\operatorname{ps}^{1}\left(  f\right)  $
in \cite[Proposition 7.1.7]{Reiner}. From our above proof of Corollary
\ref{cor.reciprocity.GammawG} (a), we see that
\[
\Omega_{\mathbf{E},G}=\operatorname{ps}^{1}\left(  \Gamma\left(
 \EE ,w,G\right)  \right)  .
\]

But \cite[Proposition 7.1.7 (iii)]{Reiner} shows that, for any
$f \in \QSym$ and $m \in \NN$, we have
$\operatorname{ps}^{1}\left(  S\left(  f\right)  \right)
\left(  m\right)
=\operatorname{ps}^{1}\left(  f\right)  \left(  -m\right)  $.
Applying this to $f=\Gamma\left(   \EE ,w,G\right)  $, we obtain
\[
\operatorname{ps}^{1}\left(  S\left(  \Gamma\left(  {\mathbf{E}%
},w,G\right)  \right)  \right)  \left(  m\right)
=\underbrace{\operatorname{ps}^{1}\left(  \Gamma\left(  {\mathbf{E}%
},w,G\right)  \right)  }_{=\Omega_{\mathbf{E},G}}\left(  -m\right)
=\Omega_{\mathbf{E},G}\left(  -m\right)
\]
for any $m\in \NN $. Thus, any $m\in \NN $ satisfies
\begin{align*}
\Omega_{\mathbf{E},G}\left(  -m\right)   &  =\operatorname{ps}^{1}%
\left(  \underbrace{S\left(  \Gamma\left(  \EE ,w,G\right)  \right)
}_{\substack{=\left(  -1\right)  ^{\left\vert E\right\vert }\Gamma^{+}\left(
\left(  E,>_{1},<_{2}\right)  ,w,G\right)  \\\text{(by Theorem
\ref{thm.antipode.GammawG})}}}\right)  \left(  m\right) \\
&  =\operatorname{ps}^{1}\left(  \left(  -1\right)  ^{\left\vert
E\right\vert }\Gamma^{+}\left(  \left(  E,>_{1},<_{2}\right)  ,w,G\right)
\right)  \left(  m\right) \\
&  =\left(  -1\right)  ^{\left\vert E\right\vert }
\operatorname{ps}^{1}
\left(  \Gamma^{+}\left(  \left(  E,>_{1},<_{2}\right)
,w,G\right)  \right)  \left(  m\right)  .
\end{align*}
Renaming $m$ as $q$ in this equality, we see that every $q\in \NN $
satisfies
\begin{equation}
\Omega_{\mathbf{E},G}\left(  -q\right)  =\left(  -1\right)  ^{\left\vert
E\right\vert }\operatorname{ps}^{1}\left(  \Gamma^{+}\left(  \left(
E,>_{1},<_{2}\right)  ,w,G\right)  \right)  \left(  q\right)  .
\label{pf.cor.reciprocity.GammawG.b.2}
\end{equation}

But just as we proved (\ref{pf.cor.reciprocity.GammawG.a.1}), we can show that
every $q\in \NN $ satisfies
\[
\operatorname{ps}^{1}\left(  \Gamma^{+}\left( \EE ,
w,G\right)  \right)  \left(  q\right)  =\left(  \text{the number of all $E$-coeven
}G\text{-orbits on } \Par_q \EE \right)  .
\]
Applying this to $\left(  E,>_{1},<_{2}\right)  $ instead of $\mathbf{E}$, we
obtain%
\begin{align*}
&  \operatorname{ps}^{1}\left(  \Gamma^{+}\left(  \left(
E,>_{1},<_{2}\right)  ,w,G\right)  \right)  \left(  q\right) \\
&  =\left(  \text{the number of all $E$-coeven }G\text{-orbits on }%
\Par_q \left(  E,>_{1},<_{2}\right)  \right)  .
\end{align*}
Now, (\ref{pf.cor.reciprocity.GammawG.b.2}) becomes%
\begin{align*}
\Omega_{\EE, G}\left(  -q\right)   &  =\left(  -1\right)  ^{\left\vert
E\right\vert }\underbrace{\operatorname{ps}^{1}\left(  \Gamma
^{+}\left(  \left(  E,>_{1},<_{2}\right)  ,w,G\right)  \right)  \left(
q\right)  }_{=\left(  \text{the number of all $E$-coeven }G\text{-orbits on }
\Par_q \left(  E,>_{1},<_{2}\right)  \right)  }\\
&  =\left(  -1\right)  ^{\left\vert E\right\vert }\left(  \text{the number of
all $E$-coeven }G\text{-orbits on }
\Par_q \left( E,>_{1},<_{2}\right)  \right)  .
\end{align*}

In order to prove Corollary \ref{cor.reciprocity.GammawG} (b), it thus remains
to show that
\begin{align}
&  \left(  \text{the number of all $E$-coeven }G\text{-orbits on }%
\Par_{q} \left(  E,>_{1},<_{2}\right)  \right)
\nonumber\\
&  =\left(  \text{the number of all $E$-coeven }G\text{-orbits on }%
\Par_{q} \left(  E,<_{1},>_{2}\right)  \right)
\label{pf.cor.reciprocity.GammawG.b.goal5}
\end{align}
for every $q\in \NN $.

\textit{Proof of (\ref{pf.cor.reciprocity.GammawG.b.goal5}):} Let
$q\in \NN $. Let $w_{0}:\left\{  1,2,\ldots,q\right\}  \rightarrow
\left\{  1,2,\ldots,q\right\}  $ be the map sending each $i\in\left\{
1,2,\ldots,q\right\}  $ to $q+1-i$. Then, the map
\[
 \Par_q \left(  E,>_{1},<_{2}\right)  \rightarrow
 \Par_q \left(  E,<_{1},>_{2}\right)
,\ \ \ \ \ \ \ \ \ \ \pi\mapsto w_{0}\circ\pi
\]
is an isomorphism of $G$-sets (this is easy to check). Thus,
$\Par_q \left(  E,>_{1},<_{2}\right)  \cong
 \Par_q \left(  E,<_{1},>_{2}\right)  $ as $G$-sets.
From this, (\ref{pf.cor.reciprocity.GammawG.b.goal5}) follows (by
functoriality, if one wishes).

The proof of Corollary~\ref{cor.reciprocity.GammawG} (b) is now complete.
\end{proof}

Now, the second formula of \cite[Theorem 2.13]{Joch} follows from our
(\ref{eq.cor.reciprocity.GammawG.b.2}), applied to $\EE = \left(
P,\prec,<_{\omega}\right)  $ (where $<_{\omega}$ is the partial order on $P$
given by $\left(  p<_{\omega}q\right)  \Longleftrightarrow\left(
\omega\left(  p\right)  <\omega\left(  q\right)  \right)  $). The first
formula of \cite[Theorem 2.13]{Joch} can also be derived from our above
arguments. We leave the details to the reader.

\section{A final question}

With the results proven above (specifically, Theorems
\ref{thm.antipode.Gammaw} and \ref{thm.antipode.GammawG}), we have obtained
formulas for a large class of quasisymmetric generating functions for maps
from a double poset to $\left\{  1,2,3,\ldots\right\}  $. At least one
question arises:

\begin{question}
In \cite{Gri-nbc}, I have studied generalizations of Whitney's famous
non-broken-circuit theorem for graphs and matroids. One of the cornerstones of
that study is the bijection $\Phi$ in \cite[proofs of Lemma 2.8, Lemma
5.29 and Lemma 8.25]{Gri-nbc}, which is uncannily reminiscent of the
involution $T$ in the proof
of Theorem~\ref{thm.antipode.Gammaw}. (Actually, this bijection $\Phi$ can be
extended to an involution, thus making the analogy even more palpable.) Both
$\Phi$ and $T$ are defined by toggling a certain element in or out of a
subset; and this element is chosen as the argmin or argmax of a function
defined on the ground set. Is there a connection between the two results, or
even a common generalization?
\end{question}

\begin{verlong}
\begin{footnotesize}

\section{Appendix: Proofs of some basic properties of quasisymmetric
functions}

In this final section, we are going to restate and prove (in detail) some
foundational facts that were stated without proof in the first few sections of
this note. Most of these facts are well-known, and all are pretty obvious to
anyone with some experience in this subject (although sometimes, formalizing
the intuitively clear arguments is a nontrivial task); the reasons why I
nevertheless have chosen to prove them here are twofold: One is to make this
paper more self-contained (although this is not completely achieved, as some
other results from places such as \cite{Reiner} are used without proof);
another is to do (some of) the groundwork for an eventual formalization of the
theory of quasisymmetric functions in a formal proof system (such as Coq). I
do not expect much of the following to be useful to the reader; most likely,
she will be able to reconstruct at least the proofs herself easily, if not the
theorems as well.

\subsection{Monomial quasisymmetric functions}

We begin with a fact that was used in the definition of $M_{\alpha}$ given in
Section \ref{sect.qsym-intro}:

\begin{proposition}
\label{prop.Malpha.equivalent}Let $\alpha=\left(  \alpha_{1},\alpha_{2}%
,\ldots,\alpha_{\ell}\right)  $ be a composition. Then,%
\[
\sum_{i_{1}<i_{2}<\cdots<i_{\ell}}x_{i_{1}}^{\alpha_{1}}x_{i_{2}}^{\alpha_{2}%
}\cdots x_{i_{\ell}}^{\alpha_{\ell}}=\sum_{\substack{\mathfrak{m}\text{ is a
monomial pack-equivalent}\\\text{to }x_{1}^{\alpha_{1}}x_{2}^{\alpha_{2}%
}\cdots x_{\ell}^{\alpha_{\ell}}}}\mathfrak{m}.
\]

\end{proposition}

\begin{proof}
[Proof of Proposition \ref{prop.Malpha.equivalent}.]If $\left(  i_{1}%
<i_{2}<\cdots<i_{\ell}\right)  $ is a length-$\ell$ strictly increasing
sequence of positive integers, then $\left(  i_{1}<i_{2}<\cdots<i_{\ell
}\right)  $ can be uniquely reconstructed from the monomial $x_{i_{1}}%
^{\alpha_{1}}x_{i_{2}}^{\alpha_{2}}\cdots x_{i_{\ell}}^{\alpha_{\ell}}%
$\ \ \ \ \footnote{Namely, $\left(  i_{1}<i_{2}<\cdots<i_{\ell}\right)  $ is
the list of all positive integers $j$ such that $x_{j}$ appears in this
monomial, written in increasing order.}. Hence, we conclude the following:

\begin{statement}
\textit{Claim 1:} If a monomial $\mathfrak{m}$ can be written in the form
$x_{i_{1}}^{\alpha_{1}}x_{i_{2}}^{\alpha_{2}}\cdots x_{i_{\ell}}^{\alpha
_{\ell}}$ for some length-$\ell$ strictly increasing sequence $\left(
i_{1}<i_{2}<\cdots<i_{\ell}\right)  $ of positive integers, then
$\mathfrak{m}$ can be written in this form in a \textbf{unique} way.
\end{statement}

On the other hand, the definition of \textquotedblleft
pack-equivalent\textquotedblright\ yields the following: The monomials which
are pack-equivalent to $x_{1}^{\alpha_{1}}x_{2}^{\alpha_{2}}\cdots x_{\ell
}^{\alpha_{\ell}}$ are precisely the monomials of the form $x_{i_{1}}%
^{\alpha_{1}}x_{i_{2}}^{\alpha_{2}}\cdots x_{i_{\ell}}^{\alpha_{\ell}}$ where
$\left(  i_{1}<i_{2}<\cdots<i_{\ell}\right)  $ is a length-$\ell$ strictly
increasing sequence of positive integers. Hence,%
\[
\sum_{\substack{\mathfrak{m}\text{ is a monomial pack-equivalent}\\\text{to
}x_{1}^{\alpha_{1}}x_{2}^{\alpha_{2}}\cdots x_{\ell}^{\alpha_{\ell}}}%
}=\sum_{\substack{\mathfrak{m}\text{ is a monomial of the form }x_{i_{1}%
}^{\alpha_{1}}x_{i_{2}}^{\alpha_{2}}\cdots x_{i_{\ell}}^{\alpha_{\ell}%
}\\\text{for some length-}\ell\text{ strictly increasing sequence}\\\left(
i_{1}<i_{2}<\cdots<i_{\ell}\right)  \text{ of positive integers}}}
\]
(an equality between summation signs). Thus,%
\begin{align*}
\sum_{\substack{\mathfrak{m}\text{ is a monomial pack-equivalent}\\\text{to
}x_{1}^{\alpha_{1}}x_{2}^{\alpha_{2}}\cdots x_{\ell}^{\alpha_{\ell}}%
}}\mathfrak{m}  &  =\sum_{\substack{\mathfrak{m}\text{ is a monomial of the
form }x_{i_{1}}^{\alpha_{1}}x_{i_{2}}^{\alpha_{2}}\cdots x_{i_{\ell}}%
^{\alpha_{\ell}}\\\text{for some length-}\ell\text{ strictly increasing
sequence}\\\left(  i_{1}<i_{2}<\cdots<i_{\ell}\right)  \text{ of positive
integers}}}\mathfrak{m}\\
&  =\sum_{i_{1}<i_{2}<\cdots<i_{\ell}}x_{i_{1}}^{\alpha_{1}}x_{i_{2}}%
^{\alpha_{2}}\cdots x_{i_{\ell}}^{\alpha_{\ell}}%
\end{align*}
(by Claim 1). This proves Proposition \ref{prop.Malpha.equivalent}.
\end{proof}

\begin{definition}
\label{def.k}Let $k\in\mathbb{Z}$. Then, $\left[  k\right]  $ will denote the
subset $\left\{  1,2,\ldots,k\right\}  =\left\{  a\in\mathbb{Z}\ \mid\ 0<a\leq
k\right\}  $ of $\left\{  1,2,3,\ldots\right\}  $. Notice that $\left\vert
\left[  k\right]  \right\vert =k$ when $k\in\mathbb{N}$. For negative $k$, we
have $\left[  k\right]  =\varnothing$.
\end{definition}

We shall now prove a result that will be used further below. We recall the
definition of $D\left(  \alpha\right)  $ given in Example \ref{exam.Gamma} (c):

\begin{definition}
\label{def.Dalpha}Let $\alpha=\left(  \alpha_{1},\alpha_{2},\ldots
,\alpha_{\ell}\right)  $ be a composition of a nonnegative integer $n$. Let
$D\left(  \alpha\right)  $ denote the set
\begin{align*}
&  \left\{  \alpha_{1},\alpha_{1}+\alpha_{2},\alpha_{1}+\alpha_{2}+\alpha
_{3},\ldots,\alpha_{1}+\alpha_{2}+\cdots+\alpha_{\ell-1}\right\} \\
&  =\left\{  \alpha_{1}+\alpha_{2}+\cdots+\alpha_{i}\ \mid\ i\in\left[
\ell-1\right]  \right\}  .
\end{align*}

\end{definition}

(We notice that this definition of $D\left(  \alpha\right)  $ is identical
with that given in \cite[Definition 5.1.10]{Reiner}.)

\begin{lemma}
\label{lem.Dalpha.n-1}Let $\alpha$ be a composition of a nonnegative integer
$n$. Then, $D\left(  \alpha\right)  \subseteq\left[  n-1\right]  $.
\end{lemma}

\begin{proof}
[Proof of Lemma \ref{lem.Dalpha.n-1}.]Write $\alpha$ in the form $\left(
\alpha_{1},\alpha_{2},\ldots,\alpha_{\ell}\right)  $. Let $k\in D\left(
\alpha\right)  $.

We know that $\left(  \alpha_{1},\alpha_{2},\ldots,\alpha_{\ell}\right)
=\alpha$ is a composition of $n$. Thus, $\alpha_{1},\alpha_{2},\ldots
,\alpha_{\ell}$ are positive integers, and their sum is $\alpha_{1}+\alpha
_{2}+\cdots+\alpha_{\ell}=n$.

We have $k\in D\left(  \alpha\right)  =\left\{  \alpha_{1}+\alpha_{2}%
+\cdots+\alpha_{i}\ \mid\ i\in\left[  \ell-1\right]  \right\}  $. In other
words, there exists some $i\in\left[  \ell-1\right]  $ such that $k=\alpha
_{1}+\alpha_{2}+\cdots+\alpha_{i}$. Consider this $i$.

We have $i\in\left[  \ell-1\right]  $, and thus $1\leq i\leq\ell-1$. From
$i\leq\ell-1$, we obtain%
\[
\alpha_{i+1}+\alpha_{i+2}+\cdots+\alpha_{\ell}=\underbrace{\left(
\alpha_{i+1}+\alpha_{i+2}+\cdots+\alpha_{\ell-1}\right)  }_{\geq0}%
+\alpha_{\ell}\geq\alpha_{\ell}>0.
\]
Now,%
\begin{align*}
n  &  =\alpha_{1}+\alpha_{2}+\cdots+\alpha_{\ell}\\
&  =\left(  \alpha_{1}+\alpha_{2}+\cdots+\alpha_{i}\right)
+\underbrace{\left(  \alpha_{i+1}+\alpha_{i+2}+\cdots+\alpha_{\ell}\right)
}_{>0}\ \ \ \ \ \ \ \ \ \ \left(  \text{since }i\leq\ell-1\leq\ell\right) \\
&  >\alpha_{1}+\alpha_{2}+\cdots+\alpha_{i}=k,
\end{align*}
so that $k<n$. Combining this with%
\begin{align*}
k  &  =\alpha_{1}+\alpha_{2}+\cdots+\alpha_{i}=\alpha_{1}+\underbrace{\left(
\alpha_{2}+\alpha_{3}+\cdots+\alpha_{i}\right)  }_{\geq0}%
\ \ \ \ \ \ \ \ \ \ \left(  \text{since }i\geq1\right) \\
&  \geq\alpha_{1}>0,
\end{align*}
we find that $0<k<n$. In other words, $k\in\left[  n-1\right]  $.

Now, forget that we fixed $k$. We thus have proven that $k\in\left[
n-1\right]  $ for every $k\in D\left(  \alpha\right)  $. In other words,
$D\left(  \alpha\right)  \subseteq\left[  n-1\right]  $. This proves Lemma
\ref{lem.Dalpha.n-1}.
\end{proof}

\begin{lemma}
\label{lem.Dalpha.s}Let $\alpha=\left(  \alpha_{1},\alpha_{2},\ldots
,\alpha_{\ell}\right)  $ be a composition of a nonnegative integer $n$. For
every $i\in\left\{  0,1,\ldots,\ell\right\}  $, define a nonnegative integer
$s_{i}$ by
\[
s_{i}=\alpha_{1}+\alpha_{2}+\cdots+\alpha_{i}.
\]

\begin{enumerate}
\item[(a)] We have $s_{i}\in\left[  n\right]  $ for every $i\in\left[
\ell\right]  $.

\item[(b)] We have $s_{0}<s_{1}<\cdots<s_{\ell}$.

\item[(c)] We have $D\left(  \alpha\right)  =\left\{  s_{1},s_{2}%
,\ldots,s_{\ell-1}\right\}  $.

\item[(d)] We have $s_{j}-s_{j-1}=\alpha_{j}$ for every $j\in\left[
\ell\right]  $.

\item[(e)] We have $s_{\ell}=n$.

\item[(f)] We have $s_{0}=0$.

\item[(g)] For every $k\in\left[  n\right]  $, the element $\min\left\{
p\in\left[  \ell\right]  \ \mid\ s_{p}\geq k\right\}  $ is a well-defined
element of $\left[  \ell\right]  $.
\end{enumerate}
\end{lemma}

\begin{proof}
[Proof of Lemma \ref{lem.Dalpha.s}.]We have $\alpha=\left(  \alpha_{1}%
,\alpha_{2},\ldots,\alpha_{\ell}\right)  $. Hence, $\left(  \alpha_{1}%
,\alpha_{2},\ldots,\alpha_{\ell}\right)  $ is a composition of $n$ (since
$\alpha$ is a composition of $n$). Thus, $\alpha_{1},\alpha_{2},\ldots
,\alpha_{\ell}$ are positive integers, and their sum is $\alpha_{1}+\alpha
_{2}+\cdots+\alpha_{\ell}=n$.

(a) Let $i\in\left[  \ell\right]  $. Thus, $1\leq i\leq\ell$. Now,%
\begin{align*}
s_{i}  &  =\alpha_{1}+\alpha_{2}+\cdots+\alpha_{i}=\alpha_{1}%
+\underbrace{\left(  \alpha_{2}+\alpha_{3}+\cdots+\alpha_{i}\right)  }_{\geq
0}\ \ \ \ \ \ \ \ \ \ \left(  \text{since }i\geq1\right) \\
&  \geq\alpha_{1}>0.
\end{align*}
Also, $i\leq\ell$, so that%
\[
\alpha_{1}+\alpha_{2}+\cdots+\alpha_{\ell}=\underbrace{\left(  \alpha
_{1}+\alpha_{2}+\cdots+\alpha_{i}\right)  }_{=s_{i}}+\underbrace{\left(
\alpha_{i+1}+\alpha_{i+2}+\cdots+\alpha_{\ell}\right)  }_{\geq0}\geq s_{i},
\]
and thus
\[
s_{i}\leq\alpha_{1}+\alpha_{2}+\cdots+\alpha_{\ell}=n.
\]
Combined with $s_{i}>0$, this yields $s_{i}\in\left\{  1,2,\ldots,n\right\}
=\left[  n\right]  $. This proves Lemma \ref{lem.Dalpha.s} (a).

(b) Let $k\in\left\{  0,1,\ldots,\ell-1\right\}  $. Then, the definition of
$s_{k}$ yields $s_{k}=\alpha_{1}+\alpha_{2}+\cdots+\alpha_{k}$. Also, the
definition of $s_{k+1}$ yields%
\[
s_{k+1}=\alpha_{1}+\alpha_{2}+\cdots+\alpha_{k+1}=\underbrace{\left(
\alpha_{1}+\alpha_{2}+\cdots+\alpha_{k}\right)  }_{=s_{k}}+\underbrace{\alpha
_{k+1}}_{>0}>s_{k}.
\]
In other words, $s_{k}<s_{k+1}$.

Now, let us forget that we fixed $k$. We thus have shown that $s_{k}<s_{k+1}$
for every $k\in\left\{  0,1,\ldots,\ell-1\right\}  $. In other words,
$s_{0}<s_{1}<\cdots<s_{\ell}$. This proves Lemma \ref{lem.Dalpha.s} (b).

(c) We have%
\begin{align*}
\left\{  s_{1},s_{2},\ldots,s_{\ell-1}\right\}   &  =\left\{
\underbrace{s_{i}}_{=\alpha_{1}+\alpha_{2}+\cdots+\alpha_{i}}\ \mid
\ i\in\left[  \ell-1\right]  \right\} \\
&  =\left\{  \alpha_{1}+\alpha_{2}+\cdots+\alpha_{i}\ \mid\ i\in\left[
\ell-1\right]  \right\}  =D\left(  \alpha\right)
\end{align*}
(because this is how $D\left(  \alpha\right)  $ is defined). This proves Lemma
\ref{lem.Dalpha.s} (c).

(d) Let $j\in\left[  \ell\right]  $. The definition of $s_{j-1}$ yields
$s_{j-1}=\alpha_{1}+\alpha_{2}+\cdots+\alpha_{j-1}$. But the definition of
$s_{j}$ yields%
\[
s_{j}=\alpha_{1}+\alpha_{2}+\cdots+\alpha_{j}=\underbrace{\left(  \alpha
_{1}+\alpha_{2}+\cdots+\alpha_{j-1}\right)  }_{=s_{j-1}}+\alpha_{j}%
=s_{j-1}+\alpha_{j}.
\]
Hence, $s_{j}-s_{j-1}=\alpha_{j}$. This proves Lemma \ref{lem.Dalpha.s} (d).

(e) The definition of $s_{\ell}$ yields $s_{\ell}=\alpha_{1}+\alpha_{2}%
+\cdots+\alpha_{\ell}=n$. This proves Lemma \ref{lem.Dalpha.s} (e).

(f) The definition of $s_{0}$ yields $s_{0}=\alpha_{1}+\alpha_{2}%
+\cdots+\alpha_{0}=\left(  \text{empty sum}\right)  =0$. This proves Lemma
\ref{lem.Dalpha.s} (f).

(g) Let $k\in\left[  n\right]  $. Hence, $1\leq k\leq n$. Thus, $n\geq1$,
hence $n\neq0$, so that $\alpha_{1}+\alpha_{2}+\cdots+\alpha_{\ell}=n\neq0$.
If we had $\ell=0$, then we would have $\alpha_{1}+\alpha_{2}+\cdots
+\alpha_{\ell}=\left(  \text{empty sum}\right)  =0$, which would contradict
$\alpha_{1}+\alpha_{2}+\cdots+\alpha_{\ell}\neq0$. Thus, we cannot have
$\ell=0$. Therefore, we have $\ell>0$, so that $\ell\in\left[  \ell\right]  $.

Lemma \ref{lem.Dalpha.s} (e) shows that $s_{\ell}=n\geq k$. Now, $\ell$ is an
element of $\left[  \ell\right]  $ and satisfies $s_{\ell}\geq k$. In other
words, $\ell$ is an element $p$ of $\left[  \ell\right]  $ satisfying
$s_{p}\geq k$. In other words, $\ell\in\left\{  p\in\left[  \ell\right]
\ \mid\ s_{p}\geq k\right\}  $. Hence, the set $\left\{  p\in\left[
\ell\right]  \ \mid\ s_{p}\geq k\right\}  $ is nonempty (since it contains
$\ell$) and finite, and thus has a minimum (since every nonempty finite set of
integers has a minimum). In other words, the minimum $\min\left\{  p\in\left[
\ell\right]  \ \mid\ s_{p}\geq k\right\}  $ is well-defined. This minimum
clearly is an element of $\left[  \ell\right]  $. This proves Lemma
\ref{lem.Dalpha.s} (g).
\end{proof}

Let us next prove a basic lemma about integers:

\begin{lemma}
\label{lem.Dalpha.s-interval}Let $\ell\in\mathbb{N}$. Let $s_{0},s_{1}%
,\ldots,s_{\ell}$ be $\ell+1$ integers satisfying $s_{0}<s_{1}<\cdots<s_{\ell
}$. Let $a\in\left[  \ell\right]  $ and $b\in\left[  \ell\right]  $ and
$u\in\mathbb{Z}$.

\begin{enumerate}
\item[(a)] If $s_{a-1}<u$ and $u\leq s_{b}$, then $a\leq b$.

\item[(b)] If $s_{a-1}<u\leq s_{a}$ and $s_{b-1}<u\leq s_{b}$, then $a=b$.
\end{enumerate}
\end{lemma}

\begin{proof}
[Proof of Lemma \ref{lem.Dalpha.s-interval}.](a) Assume that $s_{a-1}<u$ and
$u\leq s_{b}$. We must prove that $a\leq b$.

Indeed, assume the contrary. Thus, we don't have $a\leq b$. Hence, we have
$a>b$. In other words, $b<a$. Hence, $b\leq a-1$ (since $b$ and $a$ are
integers). Notice that $b\in\left[  \ell\right]  \subseteq\left\{
0,1,\ldots,\ell\right\}  $. Also, from $a\in\left[  \ell\right]  $, we obtain
$a-1\in\left\{  0,1,\ldots,\ell-1\right\}  \subseteq\left\{  0,1,\ldots
,\ell\right\}  $.

But $s_{0}<s_{1}<\cdots<s_{\ell}$. Hence, $s_{u}\leq s_{v}$ for any
$u\in\left\{  0,1,\ldots,\ell\right\}  $ and $v\in\left\{  0,1,\ldots
,\ell\right\}  $ satisfying $u\leq v$. Applying this to $u=b$ and $v=a-1$, we
obtain $s_{b}\leq s_{a-1}$ (since $b\leq a-1$). Thus, $u\leq s_{b}\leq
s_{a-1}<u$, which is absurd. This contradiction shows that our assumption was
wrong. Thus, $a\leq b$ is proven. This proves Lemma
\ref{lem.Dalpha.s-interval} (a).

(b) Assume that $s_{a-1}<u\leq s_{a}$ and $s_{b-1}<u\leq s_{b}$. We must prove
that $a=b$.

We have $s_{b-1}<u$ and $u\leq s_{a}$. Hence, Lemma
\ref{lem.Dalpha.s-interval} (a) (applied to $b$ and $a$ instead of $a$ and
$b$) yields $b\leq a$.

But $s_{a-1}<u$ and $u\leq s_{b}$. Hence, Lemma \ref{lem.Dalpha.s-interval}
(a) yields $a\leq b$. Combining this with $b\leq a$, we obtain $a=b$. This
proves Lemma \ref{lem.Dalpha.s-interval} (b).
\end{proof}

\begin{lemma}
\label{lem.Dalpha.s2}Let $\alpha=\left(  \alpha_{1},\alpha_{2},\ldots
,\alpha_{\ell}\right)  $ be a composition of a nonnegative integer $n$. For
every $i\in\left\{  0,1,\ldots,\ell\right\}  $, define a nonnegative integer
$s_{i}$ by
\[
s_{i}=\alpha_{1}+\alpha_{2}+\cdots+\alpha_{i}.
\]

Lemma \ref{lem.Dalpha.s} (g) says the following: For every $k\in\left[
n\right]  $, the element $\min\left\{  p\in\left[  \ell\right]  \ \mid
\ s_{p}\geq k\right\}  $ is a well-defined element of $\left[  \ell\right]  $.
Hence, we can define a map $f:\left[  n\right]  \rightarrow\left[
\ell\right]  $ by%
\[
\left(  f\left(  k\right)  =\min\left\{  p\in\left[  \ell\right]
\ \mid\ s_{p}\geq k\right\}  \ \ \ \ \ \ \ \ \ \ \text{for every }k\in\left[
n\right]  \right)  .
\]
Consider this map $f$.

\begin{enumerate}
\item[(a)] We have
\begin{equation}
s_{f\left(  k\right)  -1}<k\leq s_{f\left(  k\right)  }%
\ \ \ \ \ \ \ \ \ \ \text{for every }k\in\left[  n\right]  .
\label{pf.prop.Malpha.D.f.1}%
\end{equation}

\item[(b)] Moreover,%
\begin{equation}
k=s_{f\left(  k\right)  }\ \ \ \ \ \ \ \ \ \ \text{for every }k\in D\left(
\alpha\right)  . \label{pf.prop.Malpha.D.f.4}%
\end{equation}

\item[(c)] Also,%
\begin{equation}
f\left(  s_{i}\right)  =i\ \ \ \ \ \ \ \ \ \ \text{for every }i\in\left[
\ell\right]  . \label{pf.prop.Malpha.D.f.7}%
\end{equation}

\item[(d)] Furthermore,%
\begin{equation}
f\left(  k\right)  \leq f\left(  k+1\right)  \ \ \ \ \ \ \ \ \ \ \text{for
every }k\in\left[  n-1\right]  . \label{pf.prop.Malpha.D.f.2}%
\end{equation}

\item[(e)] Also,%
\begin{equation}
f\left(  k\right)  <f\left(  k+1\right)  \ \ \ \ \ \ \ \ \ \ \text{for every
}k\in D\left(  \alpha\right)  . \label{pf.prop.Malpha.D.f.3}%
\end{equation}

\item[(f)] Moreover,%
\begin{equation}
f\left(  k\right)  =f\left(  k+1\right)  \ \ \ \ \ \ \ \ \ \ \text{for every
}k\in\left[  n-1\right]  \setminus D\left(  \alpha\right)  .
\label{pf.prop.Malpha.D.f.5}%
\end{equation}

\item[(g)] We have%
\begin{equation}
f^{-1}\left(  j\right)  =\left[  s_{j}\right]  \setminus\left[  s_{j-1}%
\right]  \ \ \ \ \ \ \ \ \ \ \text{for every }j\in\left[  \ell\right]  .
\label{pf.prop.Malpha.D.f-1.1}%
\end{equation}

\item[(h)] We have%
\begin{equation}
\left\vert f^{-1}\left(  j\right)  \right\vert =\alpha_{j}%
\ \ \ \ \ \ \ \ \ \ \text{for every }j\in\left[  \ell\right]  .
\label{pf.prop.Malpha.D.f-1.2}%
\end{equation}

\end{enumerate}
\end{lemma}

\begin{proof}
[Proof of Lemma \ref{lem.Dalpha.s2}.]We recall two fundamental properties of
minima of sets:

\begin{itemize}
\item If $A$ is a subset of $\mathbb{Z}$ for which $\min A$ is well-defined,
then%
\begin{equation}
\min A\in A. \label{pf.prop.Malpha.D.min.1}%
\end{equation}

\item If $A$ is a subset of $\mathbb{Z}$ for which $\min A$ is well-defined,
and if $a$ is an element of $A$, then
\begin{equation}
a\geq\min A. \label{pf.prop.Malpha.D.min.2}%
\end{equation}
(In other words, any element of $A$ is greater or equal to the minimum of $A$.)
\end{itemize}

(a) Let $k\in\left[  n\right]  $. Hence, $0<k\leq n$. We have%
\[
f\left(  k\right)  =\min\left\{  p\in\left[  \ell\right]  \ \mid\ s_{p}\geq
k\right\}  \in\left\{  p\in\left[  \ell\right]  \ \mid\ s_{p}\geq k\right\}
\]
(by (\ref{pf.prop.Malpha.D.min.1}), applied to $A=\left\{  p\in\left[
\ell\right]  \ \mid\ s_{p}\geq k\right\}  $). In other words, $f\left(
k\right)  $ is an element of $\left[  \ell\right]  $ and satisfies
$s_{f\left(  k\right)  }\geq k$.

On the other hand, let us prove that $s_{f\left(  k\right)  -1}<k$. Indeed,
assume the contrary (for the sake of contradiction). Hence, $s_{f\left(
k\right)  -1}\geq k$. But Lemma \ref{lem.Dalpha.s} (f) shows that $s_{0}=0<k$,
so that $k>s_{0}$. Hence, $s_{f\left(  k\right)  -1}\geq k>s_{0}$, so that
$s_{f\left(  k\right)  -1}\neq s_{0}$, and therefore $f\left(  k\right)
-1\neq0$. In other words, $f\left(  k\right)  \neq1$. Combined with $f\left(
k\right)  \in\left[  \ell\right]  $, this shows that $f\left(  k\right)
\in\left[  \ell\right]  \setminus\left\{  1\right\}  $. Hence, $f\left(
k\right)  -1\in\left[  \ell-1\right]  \subseteq\left[  \ell\right]  $. Now,
$f\left(  k\right)  -1$ is an element of $\left[  \ell\right]  $ and satisfies
$s_{f\left(  k\right)  -1}\geq k$. In other words, $f\left(  k\right)  -1$ is
an element of the set $\left\{  p\in\left[  \ell\right]  \ \mid\ s_{p}\geq
k\right\}  $. Hence, (\ref{pf.prop.Malpha.D.min.2}) (applied to $A=\left\{
p\in\left[  \ell\right]  \ \mid\ s_{p}\geq k\right\}  $ and $a=f\left(
k\right)  -1$) shows that%
\[
f\left(  k\right)  -1\geq\min\left\{  p\in\left[  \ell\right]  \ \mid
\ s_{p}\geq k\right\}  =f\left(  k\right)  .
\]
In other words, $-1\geq0$. This is absurd. This contradiction proves that our
assumption was wrong; thus, the proof of $s_{f\left(  k\right)  -1}<k$ is
complete. Now, we know that $s_{f\left(  k\right)  -1}<k\leq s_{f\left(
k\right)  }$ (since $s_{f\left(  k\right)  }\geq k$). This proves Lemma
\ref{lem.Dalpha.s2} (a).

(b) Let $k\in D\left(  \alpha\right)  $. Thus, $k\in D\left(  \alpha\right)
\subseteq\left[  n-1\right]  $ (by Lemma \ref{lem.Dalpha.n-1}). Hence,
$k\in\left[  n-1\right]  \subseteq\left[  n\right]  $. Thus, $f\left(
k\right)  \in\left[  \ell\right]  $.

We have $k\in D\left(  \alpha\right)  =\left\{  s_{1},s_{2},\ldots,s_{\ell
-1}\right\}  $ (by Lemma \ref{lem.Dalpha.s} (c)). In other words, $k=s_{j}$
for some $j\in\left[  \ell-1\right]  $. Consider this $j$. Thus, $j\in\left[
\ell-1\right]  \subseteq\left[  \ell\right]  $.

But (\ref{pf.prop.Malpha.D.f.1}) yields $s_{f\left(  k\right)  -1}<k\leq
s_{f\left(  k\right)  }$.

Lemma \ref{lem.Dalpha.s} (b) shows that $s_{0}<s_{1}<\cdots<s_{\ell}$. Hence,
$s_{j-1}<s_{j}$ (since $j\in\left[  \ell\right]  $). Hence, $s_{j-1}<s_{j}=k$
and $k\leq k=s_{j}$. Thus, we know that $s_{j-1}<k\leq s_{j}$ and $s_{f\left(
k\right)  -1}<k\leq s_{f\left(  k\right)  }$. Consequently, Lemma
\ref{lem.Dalpha.s-interval} (b) (applied to $a=j$, $b=f\left(  k\right)  $ and
$u=k$) shows that $j=f\left(  k\right)  $. Hence, $s_{j}=s_{f\left(  k\right)
}$, so that $k=s_{j}=s_{f\left(  k\right)  }$. This proves Lemma
\ref{lem.Dalpha.s2} (b).

(c) Let $i\in\left[  \ell\right]  $. Clearly, $i$ is an element of the set
$\left\{  p\in\left[  \ell\right]  \ \mid\ s_{p}\geq s_{i}\right\}  $ (since
$i\in\left[  \ell\right]  $ and $s_{i}\geq s_{i}$).

Now, the definition of $f\left(  s_{i}\right)  $ yields $f\left(
s_{i}\right)  =\min\left\{  p\in\left[  \ell\right]  \ \mid\ s_{p}\geq
s_{i}\right\}  $. But (\ref{pf.prop.Malpha.D.min.2}) (applied to $A=\left\{
p\in\left[  \ell\right]  \ \mid\ s_{p}\geq s_{i}\right\}  $ and $a=i$) yields
$i\geq\min\left\{  p\in\left[  \ell\right]  \ \mid\ s_{p}\geq s_{i}\right\}  $
(since $i$ is an element of the set $\left\{  p\in\left[  \ell\right]
\ \mid\ s_{p}\geq s_{i}\right\}  $). In other words, $i\geq f\left(
s_{i}\right)  $ (since $f\left(  s_{i}\right)  =\min\left\{  p\in\left[
\ell\right]  \ \mid\ s_{p}\geq s_{i}\right\}  $).

Now, assume (for the sake of contradiction) that $i\neq f\left(  s_{i}\right)
$. Then, $i>f\left(  s_{i}\right)  $ (since $i\geq f\left(  s_{i}\right)  $).
In other words, $f\left(  s_{i}\right)  <i$.

But Lemma \ref{lem.Dalpha.s} (b) shows that $s_{0}<s_{1}<\cdots<s_{\ell}$. In
other words, $s_{u}<s_{v}$ for any $u\in\left\{  0,1,\ldots,\ell\right\}  $
and $v\in\left\{  0,1,\ldots,\ell\right\}  $ satisfying $u<v$. Applying this
to $u=f\left(  s_{i}\right)  $ and $v=i$, we obtain $s_{f\left(  s_{i}\right)
}<s_{i}$.

From Lemma \ref{lem.Dalpha.s} (a), we obtain $s_{i}\in\left[  n\right]  $.
Thus, (\ref{pf.prop.Malpha.D.f.1}) (applied to $k=s_{i}$) shows that
$s_{f\left(  s_{i}\right)  -1}<s_{i}\leq s_{f\left(  s_{i}\right)  }$. Hence,
$s_{i}\leq s_{f\left(  s_{i}\right)  }<s_{i}$. But this is absurd. This
contradiction shows that our assumption (that $i\neq f\left(  s_{i}\right)  $)
was false. We therefore have $i=f\left(  s_{i}\right)  $. This proves Lemma
\ref{lem.Dalpha.s2} (c).

(d) Let $k\in\left[  n-1\right]  $.

From $k\in\left[  n-1\right]  $, we see that both $k$ and $k+1$ are elements
of $\left[  n\right]  $. Thus, $f\left(  k\right)  $ and $f\left(  k+1\right)
$ are well-defined elements of $\left[  \ell\right]  $.

But (\ref{pf.prop.Malpha.D.f.1}) yields $s_{f\left(  k\right)  -1}<k\leq
s_{f\left(  k\right)  }$. Also, (\ref{pf.prop.Malpha.D.f.1}) (applied to $k+1$
instead of $k$) yields $s_{f\left(  k+1\right)  -1}<k+1\leq s_{f\left(
k+1\right)  }$.

Now, $s_{f\left(  k\right)  -1}<k$ and $k\leq k+1\leq s_{f\left(  k+1\right)
}$. Also, Lemma \ref{lem.Dalpha.s} (b) shows that $s_{0}<s_{1}<\cdots<s_{\ell
}$. Thus, Lemma \ref{lem.Dalpha.s-interval} (a) (applied to $a=f\left(
k\right)  $, $b=f\left(  k+1\right)  $ and $u=k$) yields $f\left(  k\right)
\leq f\left(  k+1\right)  $. This completes the proof of Lemma
\ref{lem.Dalpha.s2} (d).

(e) Let $k\in D\left(  \alpha\right)  $. Thus, $k\in D\left(  \alpha\right)
\subseteq\left[  n-1\right]  $ (by Lemma \ref{lem.Dalpha.n-1}). Hence,
$f\left(  k\right)  \leq f\left(  k+1\right)  $ (by
(\ref{pf.prop.Malpha.D.f.2})).

We want to prove that $f\left(  k\right)  <f\left(  k+1\right)  $. Indeed,
assume the contrary (for the sake of contradiction). Thus, $f\left(  k\right)
\geq f\left(  k+1\right)  $. Combined with $f\left(  k\right)  \leq f\left(
k+1\right)  $, this shows that $f\left(  k\right)  =f\left(  k+1\right)  $.

We have $k=s_{f\left(  k\right)  }$ (by (\ref{pf.prop.Malpha.D.f.4})).

But $k+1\in\left[  n\right]  $ (since $k\in\left[  n-1\right]  $). Hence,
(\ref{pf.prop.Malpha.D.f.1}) (applied to $k+1$ instead of $k$) yields
$s_{f\left(  k+1\right)  -1}<k+1\leq s_{f\left(  k+1\right)  }$. Hence,
$k+1\leq s_{f\left(  k+1\right)  }=s_{f\left(  k\right)  }$ (since $f\left(
k+1\right)  =f\left(  k\right)  $). This contradicts $s_{f\left(  k\right)
}=k<k+1$. This contradiction proves that our assumption was false. Hence,
$f\left(  k\right)  <f\left(  k+1\right)  $ is proven. This completes the
proof of Lemma \ref{lem.Dalpha.s2} (e).

(f) Let $k\in\left[  n-1\right]  \setminus D\left(  \alpha\right)  $. Thus,
$k\in\left[  n-1\right]  $ but $k\notin D\left(  \alpha\right)  $.

From $k\in\left[  n-1\right]  $, we see that both $k$ and $k+1$ are elements
of $\left[  n\right]  $. From (\ref{pf.prop.Malpha.D.f.2}), we obtain
$f\left(  k\right)  \leq f\left(  k+1\right)  $.

We must prove that $f\left(  k\right)  =f\left(  k+1\right)  $. Indeed, assume
the contrary (for the sake of contradiction). Thus, $f\left(  k\right)  \neq
f\left(  k+1\right)  $. Combined with $f\left(  k\right)  \leq f\left(
k+1\right)  $, this shows that $f\left(  k\right)  <f\left(  k+1\right)  $.

From (\ref{pf.prop.Malpha.D.f.1}), we obtain $s_{f\left(  k\right)  -1}<k\leq
s_{f\left(  k\right)  }$. From (\ref{pf.prop.Malpha.D.f.1}) (applied to $k+1$
instead of $k$), we obtain $s_{f\left(  k+1\right)  -1}<k+1\leq s_{f\left(
k+1\right)  }$.

But $f\left(  k\right)  <f\left(  k+1\right)  \leq\ell$ (since $f\left(
k+1\right)  \in\left[  \ell\right]  $). Hence, $f\left(  k\right)  \leq\ell-1$
(since $f\left(  k\right)  $ and $\ell$ are integers). Thus, $f\left(
k\right)  \in\left[  \ell-1\right]  $. Hence, $s_{f\left(  k\right)  }%
\in\left\{  s_{1},s_{2},\ldots,s_{\ell-1}\right\}  =D\left(  \alpha\right)  $
(by Lemma \ref{lem.Dalpha.s} (c)).

Also, $f\left(  k\right)  <f\left(  k+1\right)  $, so that $f\left(  k\right)
\leq f\left(  k+1\right)  -1$ (since $f\left(  k\right)  $ and $f\left(
k+1\right)  $ are integers). But Lemma \ref{lem.Dalpha.s} (b) shows that
$s_{0}<s_{1}<\cdots<s_{\ell}$. Thus, $s_{u}\leq s_{v}$ for any $u\in\left\{
0,1,\ldots,\ell\right\}  $ and $v\in\left\{  0,1,\ldots,\ell\right\}  $
satisfying $u\leq v$. Applying this to $u=f\left(  k\right)  $ and $v=f\left(
k+1\right)  -1$, we obtain $s_{f\left(  k\right)  }\leq s_{f\left(
k+1\right)  -1}<k+1$. In other words, $s_{f\left(  k\right)  }\leq\left(
k+1\right)  -1$ (since $s_{f\left(  k\right)  }$ and $k+1$ are integers).

Now, combining $k\leq s_{f\left(  k\right)  }$ with $s_{f\left(  k\right)
}\leq\left(  k+1\right)  -1=k$, we obtain $k=s_{f\left(  k\right)  }\in
D\left(  \alpha\right)  $. This contradicts $k\notin D\left(  \alpha\right)
$. This contradiction proves that our assumption was wrong. Hence, $f\left(
k\right)  =f\left(  k+1\right)  $ is proven. This completes the proof of Lemma
\ref{lem.Dalpha.s2} (f).

(g) Let $j\in\left[  \ell\right]  $.

From Lemma \ref{lem.Dalpha.s} (b), we have $s_{0}<s_{1}<\cdots<s_{\ell}$.
Thus, $s_{j-1}<s_{j}$.

Let $k\in f^{-1}\left(  j\right)  $. Thus, $k\in\left[  n\right]  $ and
$f\left(  k\right)  =j$. From (\ref{pf.prop.Malpha.D.f.1}), we obtain
$s_{f\left(  k\right)  -1}<k\leq s_{f\left(  k\right)  }$. Since $f\left(
k\right)  =j$, this rewrites as follows: $s_{j-1}<k\leq s_{j}$. In other
words, $k\in\left\{  s_{j-1}+1,s_{j-1}+2,\ldots,s_{j}\right\}  =\left[
s_{j}\right]  \setminus\left[  s_{j-1}\right]  $ (since $0\leq s_{j-1}<s_{j}$).

Now, forget that we fixed $k$. We thus have shown that $k\in\left[
s_{j}\right]  \setminus\left[  s_{j-1}\right]  $ for every $k\in f^{-1}\left(
j\right)  $. In other words, $f^{-1}\left(  j\right)  \subseteq\left[
s_{j}\right]  \setminus\left[  s_{j-1}\right]  $.

On the other hand, let $g\in\left[  s_{j}\right]  \setminus\left[
s_{j-1}\right]  $. Thus, $g\in\left[  s_{j}\right]  \setminus\left[
s_{j-1}\right]  =\left\{  s_{j-1}+1,s_{j-1}+2,\ldots,s_{j}\right\}  $ (since
$0\leq s_{j-1}<s_{j}$). In other words, $s_{j-1}<g\leq s_{j}$.

Lemma \ref{lem.Dalpha.s} (a) (applied to $i=j$) yields $s_{j}\in\left[
n\right]  $. Hence, $s_{j}\leq n$. Now, $g\in\left[  s_{j}\right]
\subseteq\left[  n\right]  $ (since $s_{j}\leq n$). Hence,
(\ref{pf.prop.Malpha.D.f.1}) (applied to $k=g$) shows that $s_{f\left(
g\right)  -1}<g\leq s_{f\left(  g\right)  }$.

We have $f\left(  g\right)  \in\left[  \ell\right]  $ and $j\in\left[
\ell\right]  $, and we have $s_{f\left(  g\right)  -1}<g\leq s_{f\left(
g\right)  }$ and $s_{j-1}<g\leq s_{j}$. Hence, Lemma
\ref{lem.Dalpha.s-interval} (b) (applied to $a=f\left(  g\right)  $, $b=j$ and
$c=g$) yields $f\left(  g\right)  =j$. Hence, $g\in f^{-1}\left(  j\right)  $.

Now, forget that we fixed $g$. We thus have shown that $g\in f^{-1}\left(
j\right)  $ for every $g\in\left[  s_{j}\right]  \setminus\left[
s_{j-1}\right]  $. In other words, $\left[  s_{j}\right]  \setminus\left[
s_{j-1}\right]  \subseteq f^{-1}\left(  j\right)  $. Combined with
$f^{-1}\left(  j\right)  \subseteq\left[  s_{j}\right]  \setminus\left[
s_{j-1}\right]  $, this yields $f^{-1}\left(  j\right)  =\left[  s_{j}\right]
\setminus\left[  s_{j-1}\right]  $. This proves Lemma \ref{lem.Dalpha.s2} (g).

(h) Let $j\in\left[  \ell\right]  $. From Lemma \ref{lem.Dalpha.s} (b), we
have $s_{0}<s_{1}<\cdots<s_{\ell}$. Thus, $s_{j-1}<s_{j}$. Consequently,
$0\leq s_{j-1}<s_{j}$. Hence, $\left[  s_{j}\right]  \setminus\left[
s_{j-1}\right]  =\left\{  s_{j-1}+1,s_{j-1}+2,\ldots,s_{j}\right\}  $, so that%
\[
\left\vert \left[  s_{j}\right]  \setminus\left[  s_{j-1}\right]  \right\vert
=\left\vert \left\{  s_{j-1}+1,s_{j-1}+2,\ldots,s_{j}\right\}  \right\vert
=s_{j}-s_{j-1}\ \ \ \ \ \ \ \ \ \ \left(  \text{since }s_{j-1}<s_{j}\right)
.
\]

But (\ref{pf.prop.Malpha.D.f-1.1}) shows that $f^{-1}\left(  j\right)
=\left[  s_{j}\right]  \setminus\left[  s_{j-1}\right]  $. Therefore,
$\left\vert f^{-1}\left(  j\right)  \right\vert =\left\vert \left[
s_{j}\right]  \setminus\left[  s_{j-1}\right]  \right\vert =s_{j}%
-s_{j-1}=\alpha_{j}$ (by Lemma \ref{lem.Dalpha.s} (d)). This proves Lemma
\ref{lem.Dalpha.s2} (h).
\end{proof}

\begin{lemma}
\label{lem.Dalpha.ItoJ}Let $\alpha=\left(  \alpha_{1},\alpha_{2},\ldots
,\alpha_{\ell}\right)  $ be a composition of a nonnegative integer $n$.

For every $i\in\left\{  0,1,\ldots,\ell\right\}  $, define a nonnegative
integer $s_{i}$ by
\[
s_{i}=\alpha_{1}+\alpha_{2}+\cdots+\alpha_{i}.
\]

Define a map $f:\left[  n\right]  \rightarrow\left[  \ell\right]  $ as in
Lemma \ref{lem.Dalpha.s2}.

Let $\left(  i_{1},i_{2},\ldots,i_{n}\right)  $ be an element of $\left\{
1,2,3,\ldots\right\}  ^{n}$ satisfying $i_{1}\leq i_{2}\leq\cdots\leq i_{n}$
and $\left\{  j\in\left[  n-1\right]  \ \mid\ i_{j}<i_{j+1}\right\}  \subseteq
D\left(  \alpha\right)  $. Then:

\begin{enumerate}
\item[(a)] We have $i_{s_{f\left(  k\right)  }}=i_{k}$ for every $k\in\left[
n\right]  $.

\item[(b)] We have $\left(  i_{s_{1}},i_{s_{2}},\ldots,i_{s_{\ell}}\right)
\in\left\{  1,2,3,\ldots\right\}  ^{\ell}$ and $i_{s_{1}}\leq i_{s_{2}}%
\leq\cdots\leq i_{s_{\ell}}$.

\item[(c)] If $\left\{  j\in\left[  n-1\right]  \ \mid\ i_{j}<i_{j+1}\right\}
=D\left(  \alpha\right)  $, then $i_{s_{1}}<i_{s_{2}}<\cdots<i_{s_{\ell}}$.
\end{enumerate}
\end{lemma}

\begin{proof}
[Proof of Lemma \ref{lem.Dalpha.ItoJ}.]Lemma \ref{lem.Dalpha.s} (b) shows that
$s_{0}<s_{1}<\cdots<s_{\ell}$.

Recall that $i_{1}\leq i_{2}\leq\cdots\leq i_{n}$. In other words,%
\begin{equation}
i_{u}\leq i_{v} \label{pf.lem.Dalpha.ItoJ.rise}%
\end{equation}
for any two elements $u$ and $v$ of $\left[  n\right]  $ satisfying $u\leq v$.

(a) First of all, we notice that $i_{s_{f\left(  k\right)  }}$ is well-defined
for every $k\in\left[  n\right]  $\ \ \ \ \footnote{\textit{Proof.} Let
$k\in\left[  n\right]  $. Thus, $f\left(  k\right)  \in\left[  \ell\right]  $.
Therefore, Lemma \ref{lem.Dalpha.s} (a) (applied to $i=f\left(  k\right)  $)
yields $s_{f\left(  k\right)  }\in\left[  n\right]  $. Thus, $i_{s_{f\left(
k\right)  }}$ is well-defined.}.

We must show that $i_{s_{f\left(  k\right)  }}=i_{k}$ for every $k\in\left[
n\right]  $.

Assume the contrary (for the sake of contradiction). Thus, $i_{s_{f\left(
k\right)  }}=i_{k}$ holds \textbf{not} for every $k\in\left[  n\right]  $. In
other words, there exists some $k\in\left[  n\right]  $ satisfying
$i_{s_{f\left(  k\right)  }}\neq i_{k}$. Let $g$ be the \textbf{highest such
}$k$. Thus, $g$ is a $k\in\left[  n\right]  $ satisfying $i_{s_{f\left(
k\right)  }}\neq i_{k}$. In other words, $g$ is an element of $\left[
n\right]  $ and satisfies $i_{s_{f\left(  g\right)  }}\neq i_{g}$. From
$i_{s_{f\left(  g\right)  }}\neq i_{g}$, we obtain $s_{f\left(  g\right)
}\neq g$. If we had $g\in D\left(  \alpha\right)  $, then we would have%
\begin{align*}
g  &  =s_{f\left(  g\right)  }\ \ \ \ \ \ \ \ \ \ \left(  \text{by
(\ref{pf.prop.Malpha.D.f.4}), applied to }k=g\right) \\
&  \neq g;
\end{align*}
this would be absurd. Hence, we cannot have $g\in D\left(  \alpha\right)  $.
We thus have $g\notin D\left(  \alpha\right)  $.

We have $g\in\left[  n\right]  $, so that $1\leq g\leq n$. Hence, $n\geq1$.
Consequently, $n\neq0$, so that $\alpha_{1}+\alpha_{2}+\cdots+\alpha_{\ell
}=n\neq0$. If we had $\ell=0$, then we would have $\alpha_{1}+\alpha
_{2}+\cdots+\alpha_{\ell}=\left(  \text{empty sum}\right)  =0$, which would
contradict $\alpha_{1}+\alpha_{2}+\cdots+\alpha_{\ell}\neq0$. Thus, we cannot
have $\ell=0$. Therefore, we have $\ell>0$, so that $\ell\in\left[
\ell\right]  $. Thus, $f\left(  s_{\ell}\right)  =\ell$ (by
(\ref{pf.prop.Malpha.D.f.7}), applied to $i=\ell$). Also, $n\in\left[
n\right]  $ (since $n\geq1$).

From Lemma \ref{lem.Dalpha.s} (e), we obtain $n=s_{\ell}$. If we had $g=n$,
then we would have $f\left(  \underbrace{g}_{=n=s_{\ell}}\right)  =f\left(
s_{\ell}\right)  =\ell$ and therefore $s_{f\left(  g\right)  }=s_{\ell}=n=g$;
but this would contradict $s_{f\left(  g\right)  }\neq g$. Hence, we cannot
have $g=n$. Thus, $g\neq n$.

From $g\in\left[  n\right]  $ and $g\neq n$, we obtain $g\in\left[  n\right]
\setminus\left\{  n\right\}  =\left[  n-1\right]  $. Combining this with
$g\notin D\left(  \alpha\right)  $, we find that $g\in\left[  n-1\right]
\setminus D\left(  \alpha\right)  $. Thus, (\ref{pf.prop.Malpha.D.f.5})
(applied to $k=g$) shows that $f\left(  g\right)  =f\left(  g+1\right)  $.

Recall that $g$ is the \textbf{highest} $k\in\left[  n\right]  $ satisfying
$i_{s_{f\left(  k\right)  }}\neq i_{k}$. Hence, for every $k\in\left[
n\right]  $ satisfying $i_{s_{f\left(  k\right)  }}\neq i_{k}$, we have%
\begin{equation}
k\leq g. \label{pf.lem.Dalpha.ItoJ.a.1}%
\end{equation}

From $g\in\left[  n-1\right]  $, we obtain $g+1\in\left[  n\right]  $. Now, if
we had $i_{s_{f\left(  g+1\right)  }}\neq i_{g+1}$, then we would have
$g+1\leq g$ (by (\ref{pf.lem.Dalpha.ItoJ.a.1}), applied to $k=g+1$); but this
would contradict $g+1>g$. Hence, we cannot have $i_{s_{f\left(  g+1\right)  }%
}\neq i_{g+1}$. We therefore must have $i_{s_{f\left(  g+1\right)  }}=i_{g+1}%
$. But $f\left(  g\right)  =f\left(  g+1\right)  $, so that $i_{s_{f\left(
g\right)  }}=i_{s_{f\left(  g+1\right)  }}=i_{g+1}$. From $i_{s_{f\left(
g\right)  }}\neq i_{g}$, we now obtain $i_{g}\neq i_{s_{f\left(  g\right)  }%
}=i_{g+1}$.

Applying (\ref{pf.lem.Dalpha.ItoJ.rise}) to $u=g$ and $v=g+1$, we obtain
$i_{g}\leq i_{g+1}$ (since $g\leq g+1$). Combined with $i_{g}\neq i_{g+1}$,
this yields $i_{g}<i_{g+1}$.

Now, $g$ is an element of $\left[  n-1\right]  $ and satisfies $i_{g}<i_{g+1}%
$. In other words, $g\in\left\{  j\in\left[  n-1\right]  \ \mid\ i_{j}%
<i_{j+1}\right\}  $. Hence, $g \in
\left\{  j\in\left[  n-1\right]  \ \mid\ i_{j}<i_{j+1}\right\}  \subseteq
D\left(  \alpha\right)  $. This contradicts $g\notin D\left(  \alpha\right)  $.
This contradiction proves that our assumption was wrong. Hence, Lemma
\ref{lem.Dalpha.ItoJ} (a) is proven.

(b) From Lemma \ref{lem.Dalpha.s} (a), we see that $s_{i}\in\left[  n\right]
$ for every $i\in\left[  \ell\right]  $. Thus, from $\left(  i_{1}%
,i_{2},\ldots,i_{n}\right)  \in\left\{  1,2,3,\ldots\right\}  ^{n}$, we obtain
$\left(  i_{s_{1}},i_{s_{2}},\ldots,i_{s_{\ell}}\right)  \in\left\{
1,2,3,\ldots\right\}  ^{\ell}$.

It remains to prove that $i_{s_{1}}\leq i_{s_{2}}\leq\cdots\leq i_{s_{\ell}}$.

Let $k\in\left[  \ell-1\right]  $ be arbitrary. Then, both $k$ and $k+1$
belong to $\left[  \ell\right]  $. Hence, both $s_{k}$ and $s_{k+1}$ belong to
$\left[  n\right]  $ (by Lemma \ref{lem.Dalpha.s} (a)).

Now, $s_{k}<s_{k+1}$ (because of $s_{0}<s_{1}<\cdots<s_{\ell}$). Hence,
$s_{k}\leq s_{k+1}$.

Applying (\ref{pf.lem.Dalpha.ItoJ.rise}) to $u=s_{k}$ and $v=s_{k+1}$, we
obtain $i_{s_{k}}\leq i_{s_{k+1}}$ (since $s_{k}\leq s_{k+1}$).

Now, forget that we fixed $k$. We thus have proven that $i_{s_{k}}\leq
i_{s_{k+1}}$ for every $k\in\left[  \ell-1\right]  $. In other words,
$i_{s_{1}}\leq i_{s_{2}}\leq\cdots\leq i_{s_{\ell}}$. This completes the proof
of Lemma \ref{lem.Dalpha.ItoJ} (b).

(c) Assume that $\left\{  j\in\left[  n-1\right]  \ \mid\ i_{j}<i_{j+1}%
\right\}  =D\left(  \alpha\right)  $.

Now, let $k\in\left[  \ell-1\right]  $ be arbitrary. Then, $k+1\in\left[
\ell\right]  $ (since $k\in\left[  \ell-1\right]  $). Thus, $s_{k+1}\in\left[
n\right]  $ (by Lemma \ref{lem.Dalpha.s} (a) (applied to $k+1$ instead of
$i$)), and thus $s_{k+1}\leq n$.

Furthermore, $s_{k}<s_{k+1}$ (because of $s_{0}<s_{1}<\cdots<s_{\ell}$).
Hence, $s_{k}+1\leq s_{k+1}$ (since $s_{k}$ and $s_{k+1}$ are integers).
Hence, $s_{k}+1\leq s_{k+1}\leq n$. Combining this with $\underbrace{s_{k}%
}_{\geq0}+1\geq1$, we obtain $s_{k}+1\in\left[  n\right]  $.

Applying (\ref{pf.lem.Dalpha.ItoJ.rise}) to $u=s_{k}+1$ and $v=s_{k+1}$, we
obtain $i_{s_{k}+1}\leq i_{s_{k+1}}$ (since $s_{k}+1\leq s_{k+1}$).

On the other hand, $k\in\left[  \ell-1\right]  $ and thus
\begin{align*}
s_{k}  &  \in\left\{  s_{1},s_{2},\ldots,s_{\ell-1}\right\}  =D\left(
\alpha\right)  \ \ \ \ \ \ \ \ \ \ \left(  \text{by Lemma \ref{lem.Dalpha.s}
(c)}\right) \\
&  =\left\{  j\in\left[  n-1\right]  \ \mid\ i_{j}<i_{j+1}\right\}  .
\end{align*}
In other words, $s_{k}$ is an element of $\left[  n-1\right]  $ and satisfies
$i_{s_{k}}<i_{s_{k}+1}$. Hence, $i_{s_{k}}<i_{s_{k}+1}\leq i_{s_{k+1}}$.

Now, forget that we fixed $k$. We thus have proven that $i_{s_{k}}<i_{s_{k+1}%
}$ for every $k\in\left[  \ell-1\right]  $. In other words, $i_{s_{1}%
}<i_{s_{2}}<\cdots<i_{s_{\ell}}$. This proves Lemma \ref{lem.Dalpha.ItoJ} (c).
\end{proof}

\begin{lemma}
\label{lem.Dalpha.JtoI}Let $\alpha=\left(  \alpha_{1},\alpha_{2},\ldots
,\alpha_{\ell}\right)  $ be a composition of a nonnegative integer $n$.

For every $i\in\left\{  0,1,\ldots,\ell\right\}  $, define a nonnegative
integer $s_{i}$ by
\[
s_{i}=\alpha_{1}+\alpha_{2}+\cdots+\alpha_{i}.
\]

Define a map $f:\left[  n\right]  \rightarrow\left[  \ell\right]  $ as in
Lemma \ref{lem.Dalpha.s2}.

Let $\left(  h_{1},h_{2},\ldots,h_{\ell}\right)  $ be an $\ell$-tuple in
$\left\{  1,2,3,\ldots\right\}  ^{\ell}$. Assume that $h_{1}\leq h_{2}%
\leq\cdots\leq h_{\ell}$. Then:

\begin{enumerate}
\item[(a)] We have $\left(  h_{f\left(  1\right)  },h_{f\left(  2\right)
},\ldots,h_{f\left(  n\right)  }\right)  \in\left\{  1,2,3,\ldots\right\}
^{n}$ and $h_{f\left(  1\right)  }\leq h_{f\left(  2\right)  }\leq\cdots\leq
h_{f\left(  n\right)  }$.

\item[(b)] We have $\left\{  j\in\left[  n-1\right]  \ \mid\ h_{f\left(
j\right)  }<h_{f\left(  j+1\right)  }\right\}  \subseteq D\left(
\alpha\right)  $.

\item[(c)] If $h_{1}<h_{2}<\cdots<h_{\ell}$, then $\left\{  j\in\left[
n-1\right]  \ \mid\ h_{f\left(  j\right)  }<h_{f\left(  j+1\right)  }\right\}
=D\left(  \alpha\right)  $.
\end{enumerate}
\end{lemma}

\begin{proof}
[Proof of Lemma \ref{lem.Dalpha.JtoI}.]We have $\left(  h_{1},h_{2}%
,\ldots,h_{\ell}\right)  \in\left\{  1,2,3,\ldots\right\}  ^{\ell}$. Hence,
\newline
$\left(  h_{f\left(  1\right)  },h_{f\left(  2\right)  },\ldots,h_{f\left(
n\right)  }\right)  \in\left\{  1,2,3,\ldots\right\}  ^{n}$.

(a) We have $h_{1}\leq h_{2}\leq\cdots\leq h_{\ell}$. Thus,%
\begin{equation}
h_{u}\leq h_{v}\ \ \ \ \ \ \ \ \ \ \text{for any }u\in\left[  \ell\right]
\text{ and }v\in\left[  \ell\right]  \text{ satisfying }u\leq v.
\label{pf.prop.SMalpha.Psi-wd.pf.2}%
\end{equation}

Let $k\in\left[  n-1\right]  $. Then, $f\left(  k\right)  \leq f\left(
k+1\right)  $ (by (\ref{pf.prop.Malpha.D.f.2})), and therefore $h_{f\left(
k\right)  }\leq h_{f\left(  k+1\right)  }$ (by
(\ref{pf.prop.SMalpha.Psi-wd.pf.2}), applied to $u=f\left(  k\right)  $ and
$v=f\left(  k+1\right)  $). Now, let us forget that we fixed $k$. Thus we have
shown that $h_{f\left(  k\right)  }\leq h_{f\left(  k+1\right)  }$ for every
$k\in\left[  n-1\right]  $. In other words, $h_{f\left(  1\right)  }\leq
h_{f\left(  2\right)  }\leq\cdots\leq h_{f\left(  n\right)  }$. Combined with
$\left(  h_{f\left(  1\right)  },h_{f\left(  2\right)  },\ldots,h_{f\left(
n\right)  }\right)  \in\left\{  1,2,3,\ldots\right\}  ^{n}$, this completes
the proof of Lemma \ref{lem.Dalpha.JtoI} (a).

(b) Let $k\in\left\{  j\in\left[  n-1\right]  \ \mid\ h_{f\left(  j\right)
}<h_{f\left(  j+1\right)  }\right\}  $. Thus, $k$ is an element of $\left[
n-1\right]  $ and satisfies $h_{f\left(  k\right)  }<h_{f\left(  k+1\right)
}$. If we had $k\notin D\left(  \alpha\right)  $, then we would have
$k\in\left[  n-1\right]  \setminus D\left(  \alpha\right)  $ (since
$k\in\left[  n-1\right]  $ and $k\notin D\left(  \alpha\right)  $) and
therefore $f\left(  k\right)  =f\left(  k+1\right)  $ (by
(\ref{pf.prop.Malpha.D.f.5})), and thus $h_{f\left(  k\right)  }=h_{f\left(
k+1\right)  }$; but this would contradict $h_{f\left(  k\right)  }<h_{f\left(
k+1\right)  }$. Hence, we cannot have $k\notin D\left(  \alpha\right)  $.
Thus, we have $k\in D\left(  \alpha\right)  $.

Now, let us forget that we fixed $k$. We thus have shown that $k\in D\left(
\alpha\right)  $ for every \newline
$k\in\left\{  j\in\left[  n-1\right]
\ \mid\ h_{f\left(  j\right)  }<h_{f\left(  j+1\right)  }\right\}  $. In other
words, $\left\{  j\in\left[  n-1\right]  \ \mid\ h_{f\left(  j\right)
}<h_{f\left(  j+1\right)  }\right\}  \subseteq D\left(  \alpha\right)  $. This
proves Lemma \ref{lem.Dalpha.JtoI} (b).

(c) Assume that $h_{1}<h_{2}<\cdots<h_{\ell}$. Thus,%
\begin{equation}
h_{u}<h_{v}\ \ \ \ \ \ \ \ \ \ \text{for any }u\in\left[  \ell\right]  \text{
and }v\in\left[  \ell\right]  \text{ satisfying }u<v.
\label{pf.prop.Malpha.Psi-wd.pf.1}%
\end{equation}

From Lemma \ref{lem.Dalpha.JtoI} (b), we obtain%
\begin{equation}
\left\{  j\in\left[  n-1\right]  \ \mid\ h_{f\left(  j\right)  }<h_{f\left(
j+1\right)  }\right\}  \subseteq D\left(  \alpha\right)  .
\label{pf.prop.Malpha.Psi-wd.pf.4}%
\end{equation}

On the other hand, we have%
\begin{equation}
D\left(  \alpha\right)  \subseteq\left\{  j\in\left[  n-1\right]
\ \mid\ h_{f\left(  j\right)  }<h_{f\left(  j+1\right)  }\right\}  .
\label{pf.prop.Malpha.Psi-wd.pf.5}%
\end{equation}

[\textit{Proof of (\ref{pf.prop.Malpha.Psi-wd.pf.5}):} Let $k\in D\left(
\alpha\right)  $. Then, $f\left(  k\right)  <f\left(  k+1\right)  $ (by
(\ref{pf.prop.Malpha.D.f.3})). Therefore, $h_{f\left(  k\right)  }<h_{f\left(
k+1\right)  }$ (by (\ref{pf.prop.Malpha.Psi-wd.pf.1}), applied to $u=f\left(
k\right)  $ and $v=f\left(  k+1\right)  $). Also, $k\in D\left(
\alpha\right)  \subseteq\left[  n-1\right]  $ (by Lemma \ref{lem.Dalpha.n-1}).
Thus, $k$ is an element of $\left[  n-1\right]  $ and satisfies $h_{f\left(
k\right)  }<h_{f\left(  k+1\right)  }$. In other words, $k\in\left\{
j\in\left[  n-1\right]  \ \mid\ h_{f\left(  j\right)  }<h_{f\left(
j+1\right)  }\right\}  $.

Now, let us forget that we fixed $k$. We thus have shown that $k\in\left\{
j\in\left[  n-1\right]  \ \mid\ h_{f\left(  j\right)  }<h_{f\left(
j+1\right)  }\right\}  $ for every $k\in D\left(  \alpha\right)  $. In other
words, $D\left(  \alpha\right)  \subseteq\left\{  j\in\left[  n-1\right]
\ \mid\ h_{f\left(  j\right)  }<h_{f\left(  j+1\right)  }\right\}  $. This
proves (\ref{pf.prop.Malpha.Psi-wd.pf.5}).]

Combining (\ref{pf.prop.Malpha.Psi-wd.pf.4}) with
(\ref{pf.prop.Malpha.Psi-wd.pf.5}), we obtain $\left\{  j\in\left[
n-1\right]  \ \mid\ h_{f\left(  j\right)  }<h_{f\left(  j+1\right)  }\right\}
=D\left(  \alpha\right)  $. This proves Lemma \ref{lem.Dalpha.JtoI} (c).
\end{proof}

\begin{proposition}
\label{prop.Malpha.D}Let $\alpha$ be a composition of a nonnegative integer
$n$. Then,%
\[
M_{\alpha}=\sum_{\substack{i_{1}\leq i_{2}\leq\cdots\leq i_{n};\\\left\{
j\in\left[  n-1\right]  \ \mid\ i_{j}<i_{j+1}\right\}  =D\left(
\alpha\right)  }}x_{i_{1}}x_{i_{2}}\cdots x_{i_{n}}.
\]

\end{proposition}

\begin{proof}
[Proof of Proposition \ref{prop.Malpha.D}.]Let $\mathcal{J}$ denote the set of
all length-$\ell$ strictly increasing sequences of positive integers. In other
words,%
\begin{equation}
\mathcal{J}=\left\{  \left(  i_{1},i_{2},\ldots,i_{\ell}\right)  \in\left\{
1,2,3,\ldots\right\}  ^{\ell}\ \mid\ i_{1}<i_{2}<\cdots<i_{\ell}\right\}  .
\label{pf.prop.Malpha.D.J=}%
\end{equation}
Renaming the index $\left(  i_{1},i_{2},\ldots,i_{\ell}\right)  $ as $\left(
j_{1},j_{2},\ldots,j_{\ell}\right)  $ in this formula, we obtain
\[
\mathcal{J}=\left\{  \left(  j_{1},j_{2},\ldots,j_{\ell}\right)  \in\left\{
1,2,3,\ldots\right\}  ^{\ell}\ \mid\ j_{1}<j_{2}<\cdots<j_{\ell}\right\}  .
\]

The definition of $M_{\alpha}$ yields%
\begin{align}
M_{\alpha}  &  =\underbrace{\sum_{i_{1}<i_{2}<\cdots<i_{\ell}}}%
_{\substack{=\sum_{\substack{\left(  i_{1},i_{2},\ldots,i_{\ell}\right)
\in\left\{  1,2,3,\ldots\right\}  ^{\ell};\\i_{1}<i_{2}<\cdots<i_{\ell}}%
}=\sum_{\left(  i_{1},i_{2},\ldots,i_{\ell}\right)  \in\mathcal{J}%
}\\\text{(since }\left\{  \left(  i_{1},i_{2},\ldots,i_{\ell}\right)
\in\left\{  1,2,3,\ldots\right\}  ^{\ell}\ \mid\ i_{1}<i_{2}<\cdots<i_{\ell
}\right\}  =\mathcal{J}\text{)}}}x_{i_{1}}^{\alpha_{1}}x_{i_{2}}^{\alpha_{2}%
}\cdots x_{i_{\ell}}^{\alpha_{\ell}}\nonumber\\
&  =\sum_{\left(  i_{1},i_{2},\ldots,i_{\ell}\right)  \in\mathcal{J}}x_{i_{1}%
}^{\alpha_{1}}x_{i_{2}}^{\alpha_{2}}\cdots x_{i_{\ell}}^{\alpha_{\ell}}%
=\sum_{\left(  j_{1},j_{2},\ldots,j_{\ell}\right)  \in\mathcal{J}}x_{j_{1}%
}^{\alpha_{1}}x_{j_{2}}^{\alpha_{2}}\cdots x_{j_{\ell}}^{\alpha_{\ell}}
\label{pf.prop.Malpha.D.Malpha=}%
\end{align}
(here, we have renamed $\left(  i_{1},i_{2},\ldots,i_{\ell}\right)  $ as
$\left(  j_{1},j_{2},\ldots,j_{\ell}\right)  $ in the sum).

Define a set $\mathcal{I}$ by%
\begin{align}
\mathcal{I}  &  =\left\{  \left(  i_{1},i_{2},\ldots,i_{n}\right)  \in\left\{
1,2,3,\ldots\right\}  ^{n}\ \mid\ i_{1}\leq i_{2}\leq\cdots\leq i_{n}\right.
\nonumber\\
&  \ \ \ \ \ \ \ \ \ \ \left.  \text{and }\left\{  j\in\left[  n-1\right]
\ \mid\ i_{j}<i_{j+1}\right\}  =D\left(  \alpha\right)  \right\}  .
\label{pf.prop.Malpha.D.I=}%
\end{align}
Thus, $\sum_{\substack{i_{1}\leq i_{2}\leq\cdots\leq i_{n};\\\left\{
j\in\left[  n-1\right]  \ \mid\ i_{j}<i_{j+1}\right\}  =D\left(
\alpha\right)  }}=\sum_{\left(  i_{1},i_{2},\ldots,i_{n}\right)
\in\mathcal{I}}$ (an equality between summation signs). Hence,%
\begin{equation}
\sum_{\substack{i_{1}\leq i_{2}\leq\cdots\leq i_{n};\\\left\{  j\in\left[
n-1\right]  \ \mid\ i_{j}<i_{j+1}\right\}  =D\left(  \alpha\right)  }%
}x_{i_{1}}x_{i_{2}}\cdots x_{i_{n}}=\sum_{\left(  i_{1},i_{2},\ldots
,i_{n}\right)  \in\mathcal{I}}x_{i_{1}}x_{i_{2}}\cdots x_{i_{n}}.
\label{pf.prop.Malpha.D.RHS=}%
\end{equation}

The definition of $\mathcal{I}$ shows that%
\begin{align*}
\mathcal{I}  &  =\left\{  \left(  i_{1},i_{2},\ldots,i_{n}\right)  \in\left\{
1,2,3,\ldots\right\}  ^{n}\ \mid\ i_{1}\leq i_{2}\leq\cdots\leq i_{n}\right.
\\
&  \ \ \ \ \ \ \ \ \ \ \left.  \text{and }\left\{  j\in\left[  n-1\right]
\ \mid\ i_{j}<i_{j+1}\right\}  =D\left(  \alpha\right)  \right\} \\
&  =\left\{  \left(  k_{1},k_{2},\ldots,k_{n}\right)  \in\left\{
1,2,3,\ldots\right\}  ^{n}\ \mid\ k_{1}\leq k_{2}\leq\cdots\leq k_{n}\right.
\\
&  \ \ \ \ \ \ \ \ \ \ \left.  \text{and }\left\{  j\in\left[  n-1\right]
\ \mid\ k_{j}<k_{j+1}\right\}  =D\left(  \alpha\right)  \right\}
\end{align*}
(here, we have renamed the index $\left(  i_{1},i_{2},\ldots,i_{n}\right)  $
as $\left(  k_{1},k_{2},\ldots,k_{n}\right)  $).

Now, for every $i\in\left\{  0,1,\ldots,\ell\right\}  $, define a nonnegative
integer $s_{i}$ by
\[
s_{i}=\alpha_{1}+\alpha_{2}+\cdots+\alpha_{i}.
\]

Define a map $f:\left[  n\right]  \rightarrow\left[  \ell\right]  $ as in
Lemma \ref{lem.Dalpha.s2}.

Now, for every $\left(  i_{1},i_{2},\ldots,i_{n}\right)  \in\mathcal{I}$, we
have $\left(  i_{s_{1}},i_{s_{2}},\ldots,i_{s_{\ell}}\right)  \in\mathcal{J}%
$\ \ \ \ \footnote{\textit{Proof.} Let $\left(  i_{1},i_{2},\ldots
,i_{n}\right)  \in\mathcal{I}$. Thus,%
\begin{align*}
\left(  i_{1},i_{2},\ldots,i_{n}\right)   &  \in\mathcal{I}\\
&  =\left\{  \left(  k_{1},k_{2},\ldots,k_{n}\right)  \in\left\{
1,2,3,\ldots\right\}  ^{n}\ \mid\ k_{1}\leq k_{2}\leq\cdots\leq k_{n}\right.
\\
&  \ \ \ \ \ \ \ \ \ \ \left.  \text{and }\left\{  j\in\left[  n-1\right]
\ \mid\ k_{j}<k_{j+1}\right\}  =D\left(  \alpha\right)  \right\}  .
\end{align*}
In other words, $\left(  i_{1},i_{2},\ldots,i_{n}\right)  $ is an element of
$\left\{  1,2,3,\ldots\right\}  ^{n}$ satisfying $i_{1}\leq i_{2}\leq
\cdots\leq i_{n}$ and $\left\{  j\in\left[  n-1\right]  \ \mid\ i_{j}%
<i_{j+1}\right\}  =D\left(  \alpha\right)  $.
\par
Thus, $\left\{  j\in\left[  n-1\right]  \ \mid\ i_{j}<i_{j+1}\right\}
=D\left(  \alpha\right)  \subseteq D\left(  \alpha\right)  $. Hence, Lemma
\ref{lem.Dalpha.ItoJ} (b) shows that we have $\left(  i_{s_{1}},i_{s_{2}%
},\ldots,i_{s_{\ell}}\right)  \in\left\{  1,2,3,\ldots\right\}  ^{\ell}$ and
$i_{s_{1}}\leq i_{s_{2}}\leq\cdots\leq i_{s_{\ell}}$. Furthermore, Lemma
\ref{lem.Dalpha.ItoJ} (c) shows that $i_{s_{1}}<i_{s_{2}}<\cdots<i_{s_{\ell}}$
(since $\left\{  j\in\left[  n-1\right]  \ \mid\ i_{j}<i_{j+1}\right\}
=D\left(  \alpha\right)  $).
\par
Now, $\left(  i_{s_{1}},i_{s_{2}},\ldots,i_{s_{\ell}}\right)  $ is an element
of $\left\{  1,2,3,\ldots\right\}  ^{\ell}$ (since $\left(  i_{s_{1}}%
,i_{s_{2}},\ldots,i_{s_{\ell}}\right)  \in\left\{  1,2,3,\ldots\right\}
^{\ell}$) and satisfies $i_{s_{1}}<i_{s_{2}}<\cdots<i_{s_{\ell}}$. In other
words, $\left(  i_{s_{1}},i_{s_{2}},\ldots,i_{s_{\ell}}\right)  $ is a
$\left(  j_{1},j_{2},\ldots,j_{\ell}\right)  \in\left\{  1,2,3,\ldots\right\}
^{\ell}$ satisfying $j_{1}<j_{2}<\cdots<j_{\ell}$. In other words,%
\[
\left(  i_{s_{1}},i_{s_{2}},\ldots,i_{s_{\ell}}\right)  \in\left\{  \left(
j_{1},j_{2},\ldots,j_{\ell}\right)  \in\left\{  1,2,3,\ldots\right\}  ^{\ell
}\ \mid\ j_{1}<j_{2}<\cdots<j_{\ell}\right\}  =\mathcal{J},
\]
qed.}. Hence, we can define a map $\Phi:\mathcal{I}\rightarrow\mathcal{J}$ by
setting%
\[
\left(  \Phi\left(  i_{1},i_{2},\ldots,i_{n}\right)  =\left(  i_{s_{1}%
},i_{s_{2}},\ldots,i_{s_{\ell}}\right)  \ \ \ \ \ \ \ \ \ \ \text{for every
}\left(  i_{1},i_{2},\ldots,i_{n}\right)  \in\mathcal{I}\right)  .
\]
Consider this $\Phi$.

For every $\left(  i_{1},i_{2},\ldots,i_{n}\right)  \in\mathcal{I}$, we have%
\begin{equation}
i_{s_{f\left(  k\right)  }}=i_{k}\ \ \ \ \ \ \ \ \ \ \text{for every }%
k\in\left[  n\right]  \label{pf.prop.Malpha.D.sf.1}%
\end{equation}
\footnote{\textit{Proof of (\ref{pf.prop.Malpha.D.sf.1}):} Fix $\left(
i_{1},i_{2},\ldots,i_{n}\right)  \in\mathcal{I}$. Thus,%
\begin{align*}
\left(  i_{1},i_{2},\ldots,i_{n}\right)   &  \in\mathcal{I}\\
&  =\left\{  \left(  k_{1},k_{2},\ldots,k_{n}\right)  \in\left\{
1,2,3,\ldots\right\}  ^{n}\ \mid\ k_{1}\leq k_{2}\leq\cdots\leq k_{n}\right.
\\
&  \ \ \ \ \ \ \ \ \ \ \left.  \text{and }\left\{  j\in\left[  n-1\right]
\ \mid\ k_{j}<k_{j+1}\right\}  =D\left(  \alpha\right)  \right\}  .
\end{align*}
In other words, $\left(  i_{1},i_{2},\ldots,i_{n}\right)  $ is an element of
$\left\{  1,2,3,\ldots\right\}  ^{n}$ satisfying $i_{1}\leq i_{2}\leq
\cdots\leq i_{n}$ and $\left\{  j\in\left[  n-1\right]  \ \mid\ i_{j}%
<i_{j+1}\right\}  =D\left(  \alpha\right)  $.
\par
Hence, Lemma \ref{lem.Dalpha.ItoJ} (a) shows that we have $i_{s_{f\left(
k\right)  }}=i_{k}$ for every $k\in\left[  n\right]  $. Hence,
(\ref{pf.prop.Malpha.D.sf.1}) is proven.}.

Now, for every $\left(  h_{1},h_{2},\ldots,h_{\ell}\right)  \in\mathcal{J}$,
we have $\left(  h_{f\left(  1\right)  },h_{f\left(  2\right)  }%
,\ldots,h_{f\left(  n\right)  }\right)  \in\mathcal{I}$%
\ \ \ \ \footnote{\textit{Proof.} Let $\left(  h_{1},h_{2},\ldots,h_{\ell
}\right)  \in\mathcal{J}$. Thus, $\left(  h_{1},h_{2},\ldots,h_{\ell}\right)
\in\mathcal{J}=\left\{  \left(  i_{1},i_{2},\ldots,i_{\ell}\right)
\in\left\{  1,2,3,\ldots\right\}  ^{\ell}\ \mid\ i_{1}<i_{2}<\cdots<i_{\ell
}\right\}  $. In other words, $\left(  h_{1},h_{2},\ldots,h_{\ell}\right)  $
is an $\ell$-tuple in $\left\{  1,2,3,\ldots\right\}  ^{\ell}$ and satisfies
$h_{1}<h_{2}<\cdots<h_{\ell}$.
\par
From $h_{1}<h_{2}<\cdots<h_{\ell}$, we obtain $h_{1}\leq h_{2}\leq\cdots\leq
h_{\ell}$. Hence, Lemma \ref{lem.Dalpha.JtoI} (a) shows that we have $\left(
h_{f\left(  1\right)  },h_{f\left(  2\right)  },\ldots,h_{f\left(  n\right)
}\right)  \in\left\{  1,2,3,\ldots\right\}  ^{n}$ and $h_{f\left(  1\right)
}\leq h_{f\left(  2\right)  }\leq\cdots\leq h_{f\left(  n\right)  }$.
Furthermore, Lemma \ref{lem.Dalpha.JtoI} (c) yields $\left\{  j\in\left[
n-1\right]  \ \mid\ h_{f\left(  j\right)  }<h_{f\left(  j+1\right)  }\right\}
=D\left(  \alpha\right)  $.
\par
Thus, $\left(  h_{f\left(  1\right)  },h_{f\left(  2\right)  },\ldots
,h_{f\left(  n\right)  }\right)  $ is an element of $\left\{  1,2,3,\ldots
\right\}  ^{n}$ which satisfies $h_{f\left(  1\right)  }\leq h_{f\left(
2\right)  }\leq\cdots\leq h_{f\left(  n\right)  }$ and $\left\{  j\in\left[
n-1\right]  \ \mid\ h_{f\left(  j\right)  }<h_{f\left(  j+1\right)  }\right\}
=D\left(  \alpha\right)  $. In other words,%
\begin{align*}
\left(  h_{f\left(  1\right)  },h_{f\left(  2\right)  },\ldots,h_{f\left(
n\right)  }\right)   &  \in\left\{  \left(  i_{1},i_{2},\ldots,i_{n}\right)
\in\left\{  1,2,3,\ldots\right\}  ^{n}\ \mid\ i_{1}\leq i_{2}\leq\cdots\leq
i_{n}\right. \\
&  \ \ \ \ \ \ \ \ \ \ \left.  \text{and }\left\{  j\in\left[  n-1\right]
\ \mid\ i_{j}<i_{j+1}\right\}  =D\left(  \alpha\right)  \right\}  .
\end{align*}
In light of (\ref{pf.prop.Malpha.D.I=}), this rewrites as $\left(  h_{f\left(
1\right)  },h_{f\left(  2\right)  },\ldots,h_{f\left(  n\right)  }\right)
\in\mathcal{I}$. Qed.}. Hence, we can define a map $\Psi:\mathcal{J}%
\rightarrow\mathcal{I}$ by setting%
\[
\left(  \Psi\left(  h_{1},h_{2},\ldots,h_{\ell}\right)  =\left(  h_{f\left(
1\right)  },h_{f\left(  2\right)  },\ldots,h_{f\left(  n\right)  }\right)
\ \ \ \ \ \ \ \ \ \ \text{for every }\left(  h_{1},h_{2},\ldots,h_{\ell
}\right)  \in\mathcal{J}\right)  .
\]
Consider this $\Psi$.

Now, $\Phi\circ\Psi=\operatorname*{id}$\ \ \ \ \footnote{\textit{Proof.} Let
$\left(  h_{1},h_{2},\ldots,h_{\ell}\right)  \in\mathcal{J}$. For every
$i\in\left[  \ell\right]  $, we have $f\left(  s_{i}\right)  =i$ (by
(\ref{pf.prop.Malpha.D.f.7})) and thus $h_{f\left(  s_{i}\right)  }=h_{i}$.
Now,%
\begin{align*}
\left(  \Phi\circ\Psi\right)  \left(  h_{1},h_{2},\ldots,h_{\ell}\right)   &
=\Phi\left(  \underbrace{\Psi\left(  h_{1},h_{2},\ldots,h_{\ell}\right)
}_{=\left(  h_{f\left(  1\right)  },h_{f\left(  2\right)  },\ldots,h_{f\left(
n\right)  }\right)  }\right)  =\Phi\left(  h_{f\left(  1\right)  },h_{f\left(
2\right)  },\ldots,h_{f\left(  n\right)  }\right) \\
&  =\left(  h_{f\left(  s_{1}\right)  },h_{f\left(  s_{2}\right)  }%
,\ldots,h_{f\left(  s_{\ell}\right)  }\right)  \ \ \ \ \ \ \ \ \ \ \left(
\text{by the definition of }\Phi\right) \\
&  =\left(  h_{1},h_{2},\ldots,h_{\ell}\right)
\end{align*}
(since $h_{f\left(  s_{i}\right)  }=h_{i}$ for every $i\in\left[  \ell\right]
$).
\par
Now, forget that we fixed $\left(  h_{1},h_{2},\ldots,h_{\ell}\right)  $. We
thus have shown that $\left(  \Phi\circ\Psi\right)  \left(  h_{1},h_{2}%
,\ldots,h_{\ell}\right)  =\left(  h_{1},h_{2},\ldots,h_{\ell}\right)  $ for
every $\left(  h_{1},h_{2},\ldots,h_{\ell}\right)  \in\mathcal{J}$. In other
words, $\Phi\circ\Psi=\operatorname*{id}$, qed.} and $\Psi\circ\Phi
=\operatorname*{id}$ \ \ \ \footnote{\textit{Proof.} For every $\left(
i_{1},i_{2},\ldots,i_{n}\right)  \in\mathcal{I}$, we have%
\begin{align*}
\left(  \Psi\circ\Phi\right)  \left(  i_{1},i_{2},\ldots,i_{n}\right)   &
=\Psi\left(  \underbrace{\Phi\left(  i_{1},i_{2},\ldots,i_{n}\right)
}_{\substack{=\left(  i_{s_{1}},i_{s_{2}},\ldots,i_{s_{\ell}}\right)
\\\text{(by the definition of }\Phi\text{)}}}\right)  =\Psi\left(  i_{s_{1}%
},i_{s_{2}},\ldots,i_{s_{\ell}}\right) \\
&  =\left(  i_{s_{f\left(  1\right)  }},i_{s_{f\left(  2\right)  }}%
,\ldots,i_{s_{f\left(  n\right)  }}\right)  \ \ \ \ \ \ \ \ \ \ \left(
\text{by the definition of }\Psi\right) \\
&  =\left(  i_{1},i_{2},\ldots,i_{n}\right)  \ \ \ \ \ \ \ \ \ \ \left(
\text{by (\ref{pf.prop.Malpha.D.sf.1})}\right)  .
\end{align*}
In other words, $\Psi\circ\Phi=\operatorname*{id}$, qed.}. Hence, the maps
$\Phi$ and $\Psi$ are mutually inverse. Thus, the map $\Phi$ is a bijection.
In other words, the map
\[
\mathcal{I}\rightarrow\mathcal{J},\ \ \ \ \ \ \ \ \ \ \left(  i_{1}%
,i_{2},\ldots,i_{n}\right)  \mapsto\left(  i_{s_{1}},i_{s_{2}},\ldots
,i_{s_{\ell}}\right)
\]
is a bijection\footnote{since $\Phi$ is the map
\[
\mathcal{I}\rightarrow\mathcal{J},\ \ \ \ \ \ \ \ \ \ \left(  i_{1}%
,i_{2},\ldots,i_{n}\right)  \mapsto\left(  i_{s_{1}},i_{s_{2}},\ldots
,i_{s_{\ell}}\right)
\]
(by the definition of $\Phi$)}.

Now, for every $\left(  i_{1},i_{2},\ldots,i_{n}\right)  \in\mathcal{I}$, we
have
\begin{equation}
x_{i_{1}}x_{i_{2}}\cdots x_{i_{n}}=x_{i_{s_{1}}}^{\alpha_{1}}x_{i_{s_{2}}%
}^{\alpha_{2}}\cdots x_{i_{s_{\ell}}}^{\alpha_{\ell}}
\label{pf.prop.Malpha.D.f-1.3}%
\end{equation}
\footnote{\textit{Proof of (\ref{pf.prop.Malpha.D.f-1.3}):} Let $\left(
i_{1},i_{2},\ldots,i_{n}\right)  \in\mathcal{I}$. Then,%
\begin{align*}
x_{i_{1}}x_{i_{2}}\cdots x_{i_{n}}  &  =\prod_{k\in\left[  n\right]  }%
x_{i_{k}}=\prod_{j\in\left[  \ell\right]  }\underbrace{\prod_{\substack{k\in
\left[  n\right]  ;\\f\left(  k\right)  =j}}}_{=\prod_{k\in f^{-1}\left(
j\right)  }}\underbrace{x_{i_{k}}}_{\substack{=x_{i_{s_{j}}}\\\text{(because
(\ref{pf.prop.Malpha.D.sf.1}) shows that}\\i_{k}=i_{s_{f\left(  k\right)  }%
}=i_{s_{j}}\text{ (since }f\left(  k\right)  =j\text{))}}%
}\ \ \ \ \ \ \ \ \ \ \left(  \text{since }f\left(  k\right)  \in\left[
\ell\right]  \text{ for every }k\in\left[  n\right]  \right) \\
&  =\prod_{j\in\left[  \ell\right]  }\underbrace{\prod_{k\in f^{-1}\left(
j\right)  }x_{i_{s_{j}}}}_{\substack{=x_{i_{s_{j}}}^{\left\vert f^{-1}\left(
j\right)  \right\vert }=x_{i_{s_{j}}}^{\alpha_{j}}\\\text{(by
(\ref{pf.prop.Malpha.D.f-1.2}))}}}=\prod_{j\in\left[  \ell\right]
}x_{i_{s_{j}}}^{\alpha_{j}}=x_{i_{s_{1}}}^{\alpha_{1}}x_{i_{s_{2}}}%
^{\alpha_{2}}\cdots x_{i_{s_{\ell}}}^{\alpha_{\ell}}.
\end{align*}
This proves (\ref{pf.prop.Malpha.D.f-1.3}).}. But
(\ref{pf.prop.Malpha.D.Malpha=}) becomes%
\begin{align*}
M_{\alpha}  &  =\sum_{\left(  j_{1},j_{2},\ldots,j_{\ell}\right)
\in\mathcal{J}}x_{j_{1}}^{\alpha_{1}}x_{j_{2}}^{\alpha_{2}}\cdots x_{j_{\ell}%
}^{\alpha_{\ell}}=\sum_{\left(  i_{1},i_{2},\ldots,i_{n}\right)
\in\mathcal{I}}\underbrace{x_{i_{s_{1}}}^{\alpha_{1}}x_{i_{s_{2}}}^{\alpha
_{2}}\cdots x_{i_{s_{\ell}}}^{\alpha_{\ell}}}_{\substack{=x_{i_{1}}x_{i_{2}%
}\cdots x_{i_{n}}\\\text{(by (\ref{pf.prop.Malpha.D.f-1.3}))}}}\\
&  \ \ \ \ \ \ \ \ \ \ \left(
\begin{array}
[c]{c}%
\text{here, we have substituted }\left(  i_{s_{1}},i_{s_{2}},\ldots
,i_{s_{\ell}}\right)  \text{ for }\left(  j_{1},j_{2},\ldots,j_{\ell}\right)
\text{ in the}\\
\text{sum, since the map }\mathcal{I}\rightarrow\mathcal{J},\ \left(
i_{1},i_{2},\ldots,i_{n}\right)  \mapsto\left(  i_{s_{1}},i_{s_{2}}%
,\ldots,i_{s_{\ell}}\right) \\
\text{is a bijection}%
\end{array}
\right) \\
&  =\sum_{\left(  i_{1},i_{2},\ldots,i_{n}\right)  \in\mathcal{I}}x_{i_{1}%
}x_{i_{2}}\cdots x_{i_{n}}=\sum_{\substack{i_{1}\leq i_{2}\leq\cdots\leq
i_{n};\\\left\{  j\in\left[  n-1\right]  \ \mid\ i_{j}<i_{j+1}\right\}
=D\left(  \alpha\right)  }}x_{i_{1}}x_{i_{2}}\cdots x_{i_{n}}%
\end{align*}
(by (\ref{pf.prop.Malpha.D.RHS=})). This proves Proposition
\ref{prop.Malpha.D}.
\end{proof}

\subsection{More on $D\left(  \alpha\right)  $}

We shall now prove another property of the sets $D\left(  \alpha\right)  $
defined in Definition \ref{def.Dalpha}; this property will be used later on.
Let us start with some definitions.

\begin{definition}
For every set $A$, we let $\mathcal{P}\left(  A\right)  $ denote the powerset
of $A$ (that is, the set of all subsets of $A$).
\end{definition}

\begin{definition}
\label{def.Compn}For every $n\in\mathbb{N}$, we let $\operatorname*{Comp}%
\nolimits_{n}$ denote the set of all compositions of $n$.
\end{definition}

\begin{definition}
\label{def.Dalpha.D}Let $n\in\mathbb{N}$. Let $\alpha\in\operatorname*{Comp}%
\nolimits_{n}$. Thus, $\alpha$ is a composition of $n$ (since
$\operatorname*{Comp}\nolimits_{n}$ is the set of all compositions of $n$).
Hence, Lemma \ref{lem.Dalpha.n-1} shows that $D\left(  \alpha\right)
\subseteq\left[  n-1\right]  $. In other words, $D\left(  \alpha\right)
\in\mathcal{P}\left(  \left[  n-1\right]  \right)  $.

Now, forget that we fixed $\alpha$. Thus, we have defined a $D\left(
\alpha\right)  \in\mathcal{P}\left(  \left[  n-1\right]  \right)  $ for every
$\alpha\in\operatorname*{Comp}\nolimits_{n}$. In other words, we have defined
a map $D:\operatorname*{Comp}\nolimits_{n}\rightarrow\mathcal{P}\left(
\left[  n-1\right]  \right)  $ which sends every $\alpha\in
\operatorname*{Comp}\nolimits_{n}$ to $D\left(  \alpha\right)  \in
\mathcal{P}\left(  \left[  n-1\right]  \right)  $. Consider this map $D$.
\end{definition}

\begin{definition}
If $J$ is a finite subset of $\mathbb{Z}$, then we let $\operatorname*{ilis}J$
be the list of all elements of $J$ in increasing order (with each element
appearing only once). For example, $\operatorname*{ilis}\left\{
2,5,1\right\}  =\left(  1,2,5\right)  $, whereas $\operatorname*{ilis}%
\varnothing$ is the empty list.
\end{definition}

\begin{lemma}
\label{lem.Dalpha.ilis}Let $n\in\mathbb{N}$. Let $I\in\mathcal{P}\left(
\left[  n-1\right]  \right)  $. Write the list $\operatorname*{ilis}\left(
I\cup\left\{  0,n\right\}  \right)  $ in the form $\left(  i_{0},i_{1}%
,\ldots,i_{m}\right)  $ for some integer $m\geq-1$.

\begin{enumerate}
\item[(a)] We have $m\geq0$.

\item[(b)] We have $i_{0}=0$.

\item[(c)] We have $i_{m}=n$.

\item[(d)] We have $\left(  i_{1}-i_{0},i_{2}-i_{1},\ldots,i_{m}%
-i_{m-1}\right)  \in\operatorname*{Comp}\nolimits_{n}$.

\item[(e)] We have $I=\left\{  i_{1},i_{2},\ldots,i_{m-1}\right\}  $.

\item[(f)] We have $D\left(  i_{1}-i_{0},i_{2}-i_{1},\ldots,i_{m}%
-i_{m-1}\right)  =I$.
\end{enumerate}
\end{lemma}

\begin{proof}
[Proof of Lemma \ref{lem.Dalpha.ilis}.]We have $\left(  i_{0},i_{1}%
,\ldots,i_{m}\right)  =\operatorname*{ilis}\left(  I\cup\left\{  0,n\right\}
\right)  $ (by the definition of $\left(  i_{0},i_{1},\ldots,i_{m}\right)  $).
In other words, $\left(  i_{0},i_{1},\ldots,i_{m}\right)  $ is the list of all
elements of $I\cup\left\{  0,n\right\}  $ in increasing order (since this is
what we defined $\operatorname*{ilis}\left(  I\cup\left\{  0,n\right\}
\right)  $ to be). Thus,
\[
\left\{  i_{0},i_{1},\ldots,i_{m}\right\}  =I\cup\left\{  0,n\right\}
\]
(since the list $\left(  i_{0},i_{1},\ldots,i_{m}\right)  $ is a list of all
elements of $I\cup\left\{  0,n\right\}  $). Also,
\[
i_{0}<i_{1}<\cdots<i_{m}%
\]
(since the list $\left(  i_{0},i_{1},\ldots,i_{m}\right)  $ is in increasing
order). Thus,%
\begin{equation}
i_{u}<i_{v}\ \ \ \ \ \ \ \ \ \ \text{for any two elements }u\text{ and
}v\text{ of }\left\{  0,1,\ldots,m\right\}  \text{ satisfying }u<v.
\label{pf.lem.Dalpha.ilis.1}%
\end{equation}

But $I\in\mathcal{P}\left(  \left[  n-1\right]  \right)  $, so that
$I\subseteq\left[  n-1\right]  \subseteq\left[  n\right]  \subseteq\left\{
0,1,\ldots,n\right\}  $. Hence,%
\[
\underbrace{I}_{\subseteq\left\{  0,1,\ldots,n\right\}  }\cup
\underbrace{\left\{  0,n\right\}  }_{\subseteq\left\{  0,1,\ldots,n\right\}
}\subseteq\left\{  0,1,\ldots,n\right\}  \cup\left\{  0,1,\ldots,n\right\}
=\left\{  0,1,\ldots,n\right\}  .
\]

(a) Assume the contrary. Thus, $m<0$. Now, $0 \in\left\{  0, n\right\}
\subseteq I\cup\left\{  0,n\right\}  = \left\{  i_{0},i_{1},\ldots
,i_{m}\right\}  = \left\{  \right\}  $ (since $m < 0$). This contradicts the
fact that the set $\left\{  \right\}  $ is empty. This contradiction proves
that our assumption was false. Hence, Lemma \ref{lem.Dalpha.ilis} (a) is proven.

(b) We have $0\in\left\{  0,n\right\}  \subseteq I\cup\left\{  0,n\right\}
=\left\{  i_{0},i_{1},\ldots,i_{m}\right\}  $. Hence, there exists some
$k\in\left\{  0,1,\ldots,m\right\}  $ such that $0=i_{k}$. Consider this $k$.

From Lemma \ref{lem.Dalpha.ilis} (a), we have $m\geq0$. Hence, $0\in\left\{
0,1,\ldots,m\right\}  $. Thus, $i_{0}$ is well-defined. Assume (for the sake
of contradiction) that $i_{0}\neq0$. Thus, $i_{0}\neq0=i_{k}$, so that $0\neq
k$ and thus $0<k$ (since $k\in\left\{  0,1,\ldots,m\right\}  $). Therefore,
(\ref{pf.lem.Dalpha.ilis.1}) (applied to $u=0$ and $v=k$) shows that
$i_{0}<i_{k}=0$.

But $i_{0}\in\left\{  i_{0},i_{1},\ldots,i_{m}\right\}  =I\cup\left\{
0,n\right\}  \subseteq\left\{  0,1,\ldots,n\right\}  $, so that $i_{0}\geq0$.
This contradicts $i_{0}<0$. This contradiction proves that our assumption
(that $i_{0}\neq0$) was false. Hence, we have $i_{0}=0$. This proves Lemma
\ref{lem.Dalpha.ilis} (b).

(c) We have $n\in\left\{  0,n\right\}  \subseteq I\cup\left\{  0,n\right\}
=\left\{  i_{0},i_{1},\ldots,i_{m}\right\}  $. Hence, there exists some
$k\in\left\{  0,1,\ldots,m\right\}  $ such that $n=i_{k}$. Consider this $k$.

From Lemma \ref{lem.Dalpha.ilis} (a), we have $m\geq0$. Hence, $m\in\left\{
0,1,\ldots,m\right\}  $. Thus, $i_{m}$ is well-defined. Assume (for the sake
of contradiction) that $i_{m}\neq n$. Thus, $i_{m}\neq n=i_{k}$, so that
$m\neq k$ and thus $k\neq m$. Hence, $k<m$ (since $k\in\left\{  0,1,\ldots
,m\right\}  $). Therefore, (\ref{pf.lem.Dalpha.ilis.1}) (applied to $u=k$ and
$v=m$) shows that $i_{k}<i_{m}$. Hence, $i_{m}>i_{k}=n$.

But $i_{m}\in\left\{  i_{0},i_{1},\ldots,i_{m}\right\}  =I\cup\left\{
0,n\right\}  \subseteq\left\{  0,1,\ldots,n\right\}  $, so that $i_{m}\leq n$.
This contradicts $i_{m}>n$. This contradiction proves that our assumption
(that $i_{m}\neq n$) was false. Hence, we have $i_{m}=n$. This proves Lemma
\ref{lem.Dalpha.ilis} (c).

(d) Lemma \ref{lem.Dalpha.ilis} (a) shows that $m\geq0$. Lemma
\ref{lem.Dalpha.ilis} (c) yields $i_{m}=n$. Lemma \ref{lem.Dalpha.ilis} (b)
yields $i_{0}=0$. Now,%
\begin{equation}
\left(  i_{1}-i_{0}\right)  +\left(  i_{2}-i_{1}\right)  +\ldots+\left(
i_{m}-i_{m-1}\right)  =n \label{pf.lem.Dalpha.ilis.d.1}%
\end{equation}
\footnote{\textit{Proof of (\ref{pf.lem.Dalpha.ilis.d.1}):} If $m=0$, then%
\begin{align*}
\left(  i_{1}-i_{0}\right)  +\left(  i_{2}-i_{1}\right)  +\ldots+\left(
i_{m}-i_{m-1}\right)   &  =\left(  \text{empty sum}\right)  =0=i_{0}%
=i_{m}\ \ \ \ \ \ \ \ \ \ \left(  \text{since }0=m\right) \\
&  =n.
\end{align*}
Hence, (\ref{pf.lem.Dalpha.ilis.d.1}) is proven in the case when $m=0$. We
thus WLOG assume that we don't have $m=0$.
\par
So we have $m\neq0$ (since we don't have $m=0$) but $m\geq0$. Consequently,
$m>0$. Now,%
\begin{align}
&  \left(  i_{1}-i_{0}\right)  +\left(  i_{2}-i_{1}\right)  +\ldots+\left(
i_{m}-i_{m-1}\right) \nonumber\\
&  =\sum_{r=1}^{m}\left(  i_{r}-i_{r-1}\right)  =\sum_{r=1}^{m}i_{r}%
-\sum_{r=1}^{m}i_{r-1}\nonumber\\
&  =\underbrace{\sum_{r=1}^{m}i_{r}}_{\substack{=\sum_{r=1}^{m-1}i_{r}%
+i_{m}\\\text{(here, we have split off the addend for }r=m\\\text{from the
sum, since }m>0\text{)}}}-\underbrace{\sum_{r=0}^{m-1}i_{r}}_{\substack{=i_{0}%
+\sum_{r=1}^{m-1}i_{r}\\\text{(here, we have split off the addend for
}r=0\\\text{from the sum, since }m>0\text{)}}}\nonumber\\
&  \ \ \ \ \ \ \ \ \ \ \left(  \text{here, we have substituted }r\text{ for
}r-1\text{ in the second sum}\right) \nonumber\\
&  =\left(  \sum_{r=1}^{m-1}i_{r}+i_{m}\right)  -\left(  i_{0}+\sum
_{r=1}^{m-1}i_{r}\right)  =\underbrace{i_{m}}_{=n}-\underbrace{i_{0}}%
_{=0}=n.\nonumber
\end{align}
This proves (\ref{pf.lem.Dalpha.ilis.d.1}).}.

Now, let $k\in\left\{  1,2,\ldots,m\right\}  $. Then, $k-1$ and $k$ are two
elements of $\left\{  0,1,\ldots,m\right\}  $. These two elements satisfy
$k-1<k$. Hence, (\ref{pf.lem.Dalpha.ilis.1}) (applied to $u=k-1$ and $v=k$)
shows that $i_{k-1}<i_{k}$. Thus, $i_{k}-i_{k-1}$ is a positive integer.

Now, forget that we fixed $k$. Thus, we have shown that for every
$k\in\left\{  1,2,\ldots,m\right\}  $, the number $i_{k}-i_{k-1}$ is a
positive integer. Hence, $\left(  i_{1}-i_{0},i_{2}-i_{1},\ldots,i_{m}%
-i_{m-1}\right)  $ is a finite list of positive integers, i.e., a composition.
The entries of this composition sum to $n$ (because of
(\ref{pf.lem.Dalpha.ilis.d.1})). Therefore, $\left(  i_{1}-i_{0},i_{2}%
-i_{1},\ldots,i_{m}-i_{m-1}\right)  $ is a composition of $n$ (by the
definition of a \textquotedblleft composition of $n$\textquotedblright). In
other words, $\left(  i_{1}-i_{0},i_{2}-i_{1},\ldots,i_{m}-i_{m-1}\right)
\in\operatorname*{Comp}\nolimits_{n}$ (since $\operatorname*{Comp}%
\nolimits_{n}$ is the set of all compositions of $n$). This proves Lemma
\ref{lem.Dalpha.ilis} (d).

(e) Let $g\in I$. We shall show that $g\in\left\{  i_{1},i_{2},\ldots
,i_{m-1}\right\}  $.

We have $g\in I\subseteq\left[  n-1\right]  $. Hence, $0<g<n$. But $g\in
I\subseteq I\cup\left\{  0,n\right\}  =\left\{  i_{0},i_{1},\ldots
,i_{m}\right\}  $. Thus, there exists some $p\in\left\{  0,1,\ldots,m\right\}
$ such that $g=i_{p}$. Consider this $p$.

We have $0<g$, thus $g\neq0$. Hence, $i_{p}=g\neq0=i_{0}$ (by Lemma
\ref{lem.Dalpha.ilis} (b)), so that $p\neq0$. Combining $p\in\left\{
0,1,\ldots,m\right\}  $ with $p\neq0$, we obtain $p\in\left\{  0,1,\ldots
,m\right\}  \setminus\left\{  0\right\}  =\left[  m\right]  $.

We have $g<n$, thus $g\neq n$. Hence, $i_{p}=g\neq n=i_{m}$ (by Lemma
\ref{lem.Dalpha.ilis} (c)), so that $p\neq m$. Combining $p\in\left[
m\right]  $ with $p\neq m$, we obtain $p\in\left[  m\right]  \setminus\left\{
m\right\}  =\left[  m-1\right]  $. Therefore, $i_{p}\in\left\{  i_{1}%
,i_{2},\ldots,i_{m-1}\right\}  $. Hence, $g=i_{p}\in\left\{  i_{1}%
,i_{2},\ldots,i_{m-1}\right\}  $.

Now, forget that we fixed $g$. Thus, we have proven that $g\in\left\{
i_{1},i_{2},\ldots,i_{m-1}\right\}  $ for every $g\in I$. In other words,%
\begin{equation}
I\subseteq\left\{  i_{1},i_{2},\ldots,i_{m-1}\right\}  .
\label{pf.lem.Dalpha.ilis.e.dir1}%
\end{equation}

On the other hand, fix $h\in\left\{  i_{1},i_{2},\ldots,i_{m-1}\right\}  $.
Thus, $h=i_{q}$ for some $q\in\left[  m-1\right]  $. Consider this $q$. We
shall show that $h\in I$.

We have $h\in\left\{  i_{1},i_{2},\ldots,i_{m-1}\right\}  \subseteq\left\{
i_{0},i_{1},\ldots,i_{m}\right\}  =I\cup\left\{  0,n\right\}  $.

But $q\in\left[  m-1\right]  $, so that $0<q$. Therefore,
(\ref{pf.lem.Dalpha.ilis.1}) (applied to $u=0$ and $v=q$) yields $i_{0}<i_{q}%
$. But Lemma \ref{lem.Dalpha.ilis} (b) shows that $0=i_{0}<i_{q}=h$. Hence,
$h\neq0$.

Also, $q\in\left[  m-1\right]  $, so that $q<m$. Therefore,
(\ref{pf.lem.Dalpha.ilis.1}) (applied to $u=q$ and $v=m$) yields $i_{q}%
<i_{m}=n$ (by Lemma \ref{lem.Dalpha.ilis} (c)). Hence, $h=i_{q}<n$, so that
$h\neq n$.

So $h$ is none of the two elements $0$ and $n$ (since $h\neq0$ and $h\neq n$).
In other words, $h\notin\left\{  0,n\right\}  $. Combined with $h\in
I\cup\left\{  0,n\right\}  $, this yields $h\in\left(  I\cup\left\{
0,n\right\}  \right)  \setminus\left\{  0,n\right\}  \subseteq I$.

Now, forget that we fixed $h$. We thus have proven that $h\in I$ for every
$h\in\left\{  i_{1},i_{2},\ldots,i_{m-1}\right\}  $. In other words,
\[
\left\{  i_{1},i_{2},\ldots,i_{m-1}\right\}  \subseteq I.
\]
Combining this with (\ref{pf.lem.Dalpha.ilis.e.dir1}), we obtain $I=\left\{
i_{1},i_{2},\ldots,i_{m-1}\right\}  $. This proves Lemma \ref{lem.Dalpha.ilis} (e).

(f) Lemma \ref{lem.Dalpha.ilis} (d) shows that $\left(  i_{1}-i_{0}%
,i_{2}-i_{1},\ldots,i_{m}-i_{m-1}\right)  \in\operatorname*{Comp}%
\nolimits_{n}$. In other words, \newline
$\left(  i_{1}-i_{0},i_{2}-i_{1},\ldots
,i_{m}-i_{m-1}\right)  $ is a composition of $n$ (since $\operatorname*{Comp}%
\nolimits_{n}$ is the set of all compositions of $n$).

If $\alpha=\left(  \alpha_{1},\alpha_{2},\ldots,\alpha_{\ell}\right)  $ is a
composition of $n$, then%
\begin{align*}
D\left(  \alpha\right)   &  =\left\{  \underbrace{\alpha_{1}+\alpha_{2}%
+\cdots+\alpha_{i}}_{=\sum_{r=1}^{i}\alpha_{r}}\ \mid\ i\in\left[
\ell-1\right]  \right\}  \ \ \ \ \ \ \ \ \ \ \left(  \text{by the definition
of }D\left(  \alpha\right)  \right) \\
&  =\left\{  \sum_{r=1}^{i}\alpha_{r}\ \mid\ i\in\left[  \ell-1\right]
\right\}  =\left\{  \sum_{r=1}^{g}\alpha_{r}\ \mid\ g\in\left[  \ell-1\right]
\right\}
\end{align*}
(here, we renamed the index $i$ as $g$). Applying this to $\alpha=\left(
i_{1}-i_{0},i_{2}-i_{1},\ldots,i_{m}-i_{m-1}\right)  $, $\ell=m$ and
$\alpha_{k}=i_{k}-i_{k-1}$, we obtain the following:%
\begin{equation}
D\left(  i_{1}-i_{0},i_{2}-i_{1},\ldots,i_{m}-i_{m-1}\right)  =\left\{
\sum_{r=1}^{g}\left(  i_{r}-i_{r-1}\right)  \ \mid\ g\in\left[  m-1\right]
\right\}  . \label{pf.lem.Dalpha.ilis.f.1}%
\end{equation}

But Lemma~\ref{lem.Dalpha.ilis} (b) yields $i_{0} = 0$. Now, every
$g\in\left[  m-1\right]  $ satisfies%
\begin{align}
\sum_{r=1}^{g}\left(  i_{r}-i_{r-1}\right)   &  =\sum_{r=1}^{g}i_{r}%
-\sum_{r=1}^{g}i_{r-1}\nonumber\\
&  =\underbrace{\sum_{r=1}^{g}i_{r}}_{\substack{=\sum_{r=1}^{g-1}i_{r}%
+i_{g}\\\text{(here, we have split off the addend for }r=g\\\text{from the
sum, since }g>0\text{)}}}-\underbrace{\sum_{r=0}^{g-1}i_{r}}_{\substack{=i_{0}%
+\sum_{r=1}^{g-1}i_{r}\\\text{(here, we have split off the addend for
}r=0\\\text{from the sum, since }g>0\text{)}}}\nonumber\\
&  \ \ \ \ \ \ \ \ \ \ \left(  \text{here, we have substituted }r\text{ for
}r-1\text{ in the second sum}\right) \nonumber\\
&  =\left(  \sum_{r=1}^{g-1}i_{r}+i_{g}\right)  -\left(  i_{0}+\sum
_{r=1}^{g-1}i_{r}\right)  =i_{g}-\underbrace{i_{0}}_{=0}=i_{g}.\nonumber
\end{align}
Thus, (\ref{pf.lem.Dalpha.ilis.f.1}) becomes%
\begin{align*}
D\left(  i_{1}-i_{0},i_{2}-i_{1},\ldots,i_{m}-i_{m-1}\right)   &  =\left\{
\underbrace{\sum_{r=1}^{g}\left(  i_{r}-i_{r-1}\right)  }_{=i_{g}}\ \mid
\ g\in\left[  m-1\right]  \right\} \\
&  =\left\{  i_{g}\ \mid\ g\in\left[  m-1\right]  \right\}  =\left\{
i_{1},i_{2},\ldots,i_{m-1}\right\}  =I
\end{align*}
(by Lemma \ref{lem.Dalpha.ilis} (e)). This proves Lemma \ref{lem.Dalpha.ilis} (f).
\end{proof}

\begin{definition}
\label{def.Dalpha.comp}Let $n\in\mathbb{N}$. We define a map
$\operatorname*{comp}:\mathcal{P}\left(  \left[  n-1\right]  \right)
\rightarrow\operatorname*{Comp}\nolimits_{n}$ as follows: Let $I\in
\mathcal{P}\left(  \left[  n-1\right]  \right)  $. Write the list
$\operatorname*{ilis}\left(  I\cup\left\{  0,n\right\}  \right)  $ in the form
$\left(  i_{0},i_{1},\ldots,i_{m}\right)  $ for some integer $m\geq-1$. Lemma
\ref{lem.Dalpha.ilis} (d) shows that $\left(  i_{1}-i_{0},i_{2}-i_{1}%
,\ldots,i_{m}-i_{m-1}\right)  \in\operatorname*{Comp}\nolimits_{n}$. Define
$\operatorname*{comp}I$ to be $\left(  i_{1}-i_{0},i_{2}-i_{1},\ldots
,i_{m}-i_{m-1}\right)  $.

Hence, a map $\operatorname*{comp}:\mathcal{P}\left(  \left[  n-1\right]
\right)  \rightarrow\operatorname*{Comp}\nolimits_{n}$ is defined.
\end{definition}

\begin{proposition}
\label{prop.Dalpha.comp}Let $n\in\mathbb{N}$. Consider the map
$D:\operatorname*{Comp}\nolimits_{n}\rightarrow\mathcal{P}\left(  \left[
n-1\right]  \right)  $ defined in Definition \ref{def.Dalpha.D}. Consider the
map $\operatorname*{comp}:\mathcal{P}\left(  \left[  n-1\right]  \right)
\rightarrow\operatorname*{Comp}\nolimits_{n}$ introduced in Definition
\ref{def.Dalpha.comp}.

These maps $D$ and $\operatorname*{comp}$ are mutually inverse.
\end{proposition}

\begin{proof}
[Proof of Proposition \ref{prop.Dalpha.comp}.]Let us first show that
$\operatorname*{comp}\circ D=\operatorname*{id}$.

Indeed, let $\alpha\in\operatorname*{Comp}\nolimits_{n}$. Thus, $\alpha$ is a
composition of $n$ (since $\operatorname*{Comp}\nolimits_{n}$ is the set of
all compositions of $n$). Write the composition $\alpha$ in the form $\left(
\alpha_{1},\alpha_{2},\ldots,\alpha_{\ell}\right)  $. Hence, $\alpha=\left(
\alpha_{1},\alpha_{2},\ldots,\alpha_{\ell}\right)  $. Lemma
\ref{lem.Dalpha.n-1} shows that $D\left(  \alpha\right)  \subseteq\left[
n-1\right]  $, so that $D\left(  \alpha\right)  \in\mathcal{P}\left(  \left[
n-1\right]  \right)  $. Hence, $\operatorname*{comp}\left(  D\left(
\alpha\right)  \right)  $ is well-defined. The definition of
$\operatorname*{comp}\left(  D\left(  \alpha\right)  \right)  $ shows that if
the list $\operatorname*{ilis}\left(  D\left(  \alpha\right)  \cup\left\{
0,n\right\}  \right)  $ is written in the form $\left(  i_{0},i_{1}%
,\ldots,i_{m}\right)  $ for some integer $m\geq-1$, then%
\begin{equation}
\operatorname*{comp}\left(  D\left(  \alpha\right)  \right)  =\left(
i_{1}-i_{0},i_{2}-i_{1},\ldots,i_{m}-i_{m-1}\right)  .
\label{pf.prop.Dalpha.comp.cD=}%
\end{equation}

Now, for every $i\in\left\{  0,1,\ldots,\ell\right\}  $, define a nonnegative
integer $s_{i}$ by
\[
s_{i}=\alpha_{1}+\alpha_{2}+\cdots+\alpha_{i}.
\]
Lemma \ref{lem.Dalpha.s} (d) yields $s_{j}-s_{j-1}=\alpha_{j}$ for every
$j\in\left[  \ell\right]  $. In other words,%
\begin{equation}
\left(  s_{1}-s_{0},s_{2}-s_{1},\ldots,s_{\ell}-s_{\ell-1}\right)  =\left(
\alpha_{1},\alpha_{2},\ldots,\alpha_{\ell}\right)  .
\label{pf.prop.Dalpha.comp.6}%
\end{equation}

Now, we have%
\begin{equation}
D\left(  \alpha\right)  \cup\left\{  0,n\right\}  =\left\{  s_{0},s_{1}%
,\ldots,s_{\ell}\right\}  \label{pf.prop.Dalpha.comp.11}%
\end{equation}
\footnote{\textit{Proof of (\ref{pf.prop.Dalpha.comp.11}):} From Lemma
\ref{lem.Dalpha.s} (c), we have $D\left(  \alpha\right)  =\left\{  s_{1}%
,s_{2},\ldots,s_{\ell-1}\right\}  $. Lemma \ref{lem.Dalpha.s} (f) yields $0 =
s_{0}$. Lemma \ref{lem.Dalpha.s} (e) yields $n = s_{\ell}$. Thus,%
\begin{align*}
&  \underbrace{D\left(  \alpha\right)  }_{=\left\{  s_{1},s_{2},\ldots
,s_{\ell-1}\right\}  }\cup\left\{  \underbrace{0}_{=s_{0}}, \underbrace{n}%
_{=s_{\ell}}\right\}  =\left\{  s_{1},s_{2},\ldots,s_{\ell-1}\right\}
\cup\left\{  s_{0},s_{\ell}\right\}  =\left\{  s_{0},s_{1},\ldots,s_{\ell
}\right\}  .
\end{align*}
This proves (\ref{pf.prop.Dalpha.comp.11}).}. Thus, the list $\left(
s_{0},s_{1},\ldots,s_{\ell}\right)  $ contains precisely the elements of the
set $D\left(  \alpha\right)  \cup\left\{  0,n\right\}  $. Since this list
$\left(  s_{0},s_{1},\ldots,s_{\ell}\right)  $ is furthermore strictly
increasing (because of Lemma \ref{lem.Dalpha.s} (b)), we thus see that the
list $\left(  s_{0},s_{1},\ldots,s_{\ell}\right)  $ is the list of all
elements of $D\left(  \alpha\right)  \cup\left\{  0,n\right\}  $ in increasing
order (with each element appearing only once). In other words, the list
$\left(  s_{0},s_{1},\ldots,s_{\ell}\right)  $ is $\operatorname*{ilis}\left(
D\left(  \alpha\right)  \cup\left\{  0,n\right\}  \right)  $ (since
$\operatorname*{ilis}\left(  D\left(  \alpha\right)  \cup\left\{  0,n\right\}
\right)  $ is defined to be the list of all elements of $D\left(
\alpha\right)  \cup\left\{  0,n\right\}  $ in increasing order (with each
element appearing only once)). In other words,%
\[
\operatorname*{ilis}\left(  D\left(  \alpha\right)  \cup\left\{  0,n\right\}
\right)  =\left(  s_{0},s_{1},\ldots,s_{\ell}\right)  .
\]
Hence, (\ref{pf.prop.Dalpha.comp.cD=}) (applied to $m=\ell$ and $i_{k}=s_{k}$)
yields%
\begin{align*}
\operatorname*{comp}\left(  D\left(  \alpha\right)  \right)   &  =\left(
s_{1}-s_{0},s_{2}-s_{1},\ldots,s_{\ell}-s_{\ell-1}\right)  =\left(  \alpha
_{1},\alpha_{2},\ldots,\alpha_{\ell}\right)  \ \ \ \ \ \ \ \ \ \ \left(
\text{by (\ref{pf.prop.Dalpha.comp.6})}\right) \\
&  =\alpha.
\end{align*}
Thus, $\left(  \operatorname*{comp}\circ D\right)  \left(  \alpha\right)
=\operatorname*{comp}\left(  D\left(  \alpha\right)  \right)  =\alpha$.

Now, forget that we fixed $\alpha$. We thus have proven that $\left(
\operatorname*{comp}\circ D\right)  \left(  \alpha\right)  =\alpha$ for every
$\alpha\in\operatorname*{Comp}\nolimits_{n}$. In other words,%
\begin{equation}
\operatorname*{comp}\circ D=\operatorname*{id}.
\label{pf.prop.Dalpha.comp.one-side}%
\end{equation}

On the other hand, let us show that $D\circ\operatorname*{comp}%
=\operatorname*{id}$. Indeed, let $I\in\mathcal{P}\left(  \left[  n-1\right]
\right)  $. Write the list $\operatorname*{ilis}\left(  I\cup\left\{
0,n\right\}  \right)  $ in the form $\left(  i_{0},i_{1},\ldots,i_{m}\right)
$ for some integer $m\geq-1$. The definition of $\operatorname*{comp}$ then
shows that $\operatorname*{comp}I=\left(  i_{1}-i_{0},i_{2}-i_{1},\ldots
,i_{m}-i_{m-1}\right)  $. Now,%
\[
\left(  D\circ\operatorname*{comp}\right)  \left(  I\right)  =D\left(
\underbrace{\operatorname*{comp}I}_{=\left(  i_{1}-i_{0},i_{2}-i_{1}%
,\ldots,i_{m}-i_{m-1}\right)  }\right)  =D\left(  i_{1}-i_{0},i_{2}%
-i_{1},\ldots,i_{m}-i_{m-1}\right)  =I
\]
(by Lemma \ref{lem.Dalpha.ilis} (f)).

Now, forget that we fixed $I$. We thus have proven that $\left(
D\circ\operatorname*{comp}\right)  \left(  I\right)  =I$ for every
$I\in\mathcal{P}\left(  \left[  n-1\right]  \right)  $. In other words,
$D\circ\operatorname*{comp}=\operatorname*{id}$. Combining this with
(\ref{pf.prop.Dalpha.comp.one-side}), we conclude that the maps $D$ and
$\operatorname*{comp}$ are mutually inverse. Proposition
\ref{prop.Dalpha.comp} is proven.
\end{proof}

From what we have proven so far, we obtain the following corollary:

\begin{corollary}
\label{cor.Dalpha.F}Let $\alpha$ be a composition of a nonnegative integer
$n$. Then,
\[
\sum_{\substack{i_{1}\leq i_{2}\leq\cdots\leq i_{n};\\i_{j}<i_{j+1}\text{
whenever }j\in D\left(  \alpha\right)  }}x_{i_{1}}x_{i_{2}}\cdots x_{i_{n}%
}=\sum_{\beta\text{ is a composition of }n;\ D\left(  \beta\right)  \supseteq
D\left(  \alpha\right)  }M_{\beta}.
\]
(Here, we are using the notations of Definition \ref{def.k} and Definition
\ref{def.Dalpha}.)
\end{corollary}

\begin{proof}
[Proof of Corollary \ref{cor.Dalpha.F}.]Lemma \ref{lem.Dalpha.n-1} shows that
$D\left(  \alpha\right)  \subseteq\left[  n-1\right]  $.

Proposition \ref{prop.Dalpha.comp} yields that the maps $D$ and
$\operatorname*{comp}$ are mutually inverse. Thus, the map $D$ is invertible,
i.e., a bijection.

If $\left(  i_{1},i_{2},\ldots,i_{n}\right)  $ is an $n$-tuple of positive
integers, then the condition \newline$\left(  i_{j}<i_{j+1}\text{ whenever
}j\in D\left(  \alpha\right)  \right)  $ is equivalent to the condition
\newline$\left(  D\left(  \alpha\right)  \subseteq\left\{  j\in\left[
n-1\right]  \ \mid\ i_{j}<i_{j+1}\right\}  \right)  $%
\ \ \ \ \footnote{\textit{Proof.} Let $\left(  i_{1},i_{2},\ldots
,i_{n}\right)  $ be an $n$-tuple of positive integers. Then, we have the
following chain of logical equivalences:%
\begin{align*}
&  \ \left(  i_{j}<i_{j+1}\text{ whenever }j\in D\left(  \alpha\right)
\right) \\
&  \Longleftrightarrow\ \left(  i_{h}<i_{h+1}\text{ whenever }h\in D\left(
\alpha\right)  \right) \\
&  \ \ \ \ \ \ \ \ \ \ \left(  \text{here, we have renamed the index }j\text{
as }h\right) \\
&  \Longleftrightarrow\ \left(  \text{every }h\in D\left(  \alpha\right)
\text{ satisfies }\underbrace{i_{h}<i_{h+1}}_{\substack{\text{this is
equivalent to}\\h\in\left\{  j\in\left[  n-1\right]  \ \mid\ i_{j}%
<i_{j+1}\right\}  \\\text{(because }h\text{ is always an element of }\left[
n-1\right]  \\\text{(since }h\in D\left(  \alpha\right)  \subseteq\left[
n-1\right]  \text{))}}}\right) \\
&  \Longleftrightarrow\ \left(  \text{every }h\in D\left(  \alpha\right)
\text{ satisfies }h\in\left\{  j\in\left[  n-1\right]  \ \mid\ i_{j}%
<i_{j+1}\right\}  \right) \\
&  \Longleftrightarrow\ \left(  D\left(  \alpha\right)  \subseteq\left\{
j\in\left[  n-1\right]  \ \mid\ i_{j}<i_{j+1}\right\}  \right)  ,
\end{align*}
Thus, the condition $\left(  i_{j}<i_{j+1}\text{ whenever }j\in D\left(
\alpha\right)  \right)  $ is equivalent to the condition $\left(  D\left(
\alpha\right)  \subseteq\left\{  j\in\left[  n-1\right]  \ \mid\ i_{j}%
<i_{j+1}\right\}  \right)  $. Qed.}. Hence, we have the following equality
between summation signs:%
\[
\sum_{\substack{i_{1}\leq i_{2}\leq\cdots\leq i_{n};\\i_{j}<i_{j+1}\text{
whenever }j\in D\left(  \alpha\right)  }}=\sum_{\substack{i_{1}\leq i_{2}%
\leq\cdots\leq i_{n};\\D\left(  \alpha\right)  \subseteq\left\{  j\in\left[
n-1\right]  \ \mid\ i_{j}<i_{j+1}\right\}  }}.
\]
Thus,%
\begin{align*}
&  \sum_{\substack{i_{1}\leq i_{2}\leq\cdots\leq i_{n};\\i_{j}<i_{j+1}\text{
whenever }j\in D\left(  \alpha\right)  }}x_{i_{1}}x_{i_{2}}\cdots x_{i_{n}}\\
&  =\sum_{\substack{i_{1}\leq i_{2}\leq\cdots\leq i_{n};\\D\left(
\alpha\right)  \subseteq\left\{  j\in\left[  n-1\right]  \ \mid\ i_{j}%
<i_{j+1}\right\}  }}x_{i_{1}}x_{i_{2}}\cdots x_{i_{n}}\\
&  =\underbrace{\sum_{\substack{G\subseteq\left[  n-1\right]  ;\\D\left(
\alpha\right)  \subseteq G}}}_{=\sum_{\substack{G\subseteq\left[  n-1\right]
;\\G\supseteq D\left(  \alpha\right)  }}=\sum_{\substack{G\in\mathcal{P}%
\left(  \left[  n-1\right]  \right)  ;\\G\supseteq D\left(  \alpha\right)  }%
}}\sum_{\substack{i_{1}\leq i_{2}\leq\cdots\leq i_{n};\\\left\{  j\in\left[
n-1\right]  \ \mid\ i_{j}<i_{j+1}\right\}  =G}}x_{i_{1}}x_{i_{2}}\cdots
x_{i_{n}}\\
&  \ \ \ \ \ \ \ \ \ \ \left(  \text{since }\left\{  j\in\left[  n-1\right]
\ \mid\ i_{j}<i_{j+1}\right\}  \subseteq\left[  n-1\right]  \text{ for every
}\left(  i_{1}\leq i_{2}\leq\cdots\leq i_{n}\right)  \right) \\
&  =\sum_{\substack{G\in\mathcal{P}\left(  \left[  n-1\right]  \right)
;\\G\supseteq D\left(  \alpha\right)  }}\sum_{\substack{i_{1}\leq i_{2}%
\leq\cdots\leq i_{n};\\\left\{  j\in\left[  n-1\right]  \ \mid\ i_{j}%
<i_{j+1}\right\}  =G}}x_{i_{1}}x_{i_{2}}\cdots x_{i_{n}}\\
&  =\sum_{\substack{\beta\in\operatorname*{Comp}\nolimits_{n};\\D\left(
\beta\right)  \supseteq D\left(  \alpha\right)  }}\sum_{\substack{i_{1}\leq
i_{2}\leq\cdots\leq i_{n};\\\left\{  j\in\left[  n-1\right]  \ \mid
\ i_{j}<i_{j+1}\right\}  =D\left(  \beta\right)  }}x_{i_{1}}x_{i_{2}}\cdots
x_{i_{n}}%
\end{align*}
(here, we have substituted $D\left(  \beta\right)  $ for $G$ in the outer sum,
since the map $D:\operatorname*{Comp}\nolimits_{n}\rightarrow\mathcal{P}%
\left(  \left[  n-1\right]  \right)  $ is a bijection). Comparing this with%
\begin{align*}
&  \underbrace{\sum_{\beta\text{ is a composition of }n;\ D\left(
\beta\right)  \supseteq D\left(  \alpha\right)  }}_{\substack{=\sum
_{\substack{\beta\in\operatorname*{Comp}\nolimits_{n};\\D\left(  \beta\right)
\supseteq D\left(  \alpha\right)  }}\\\text{(since }\operatorname*{Comp}%
\nolimits_{n}\text{ is the set}\\\text{of all compositions of }n\text{)}%
}}\underbrace{M_{\beta}}_{\substack{=\sum_{\substack{i_{1}\leq i_{2}\leq
\cdots\leq i_{n};\\\left\{  j\in\left[  n-1\right]  \ \mid\ i_{j}%
<i_{j+1}\right\}  =D\left(  \beta\right)  }}x_{i_{1}}x_{i_{2}}\cdots x_{i_{n}%
}\\\text{(by Proposition \ref{prop.Malpha.D} (applied to }\beta\text{ instead
of }\alpha\text{))}}}\\
&  =\sum_{\substack{\beta\in\operatorname*{Comp}\nolimits_{n};\\D\left(
\beta\right)  \supseteq D\left(  \alpha\right)  }}\sum_{\substack{i_{1}\leq
i_{2}\leq\cdots\leq i_{n};\\\left\{  j\in\left[  n-1\right]  \ \mid
\ i_{j}<i_{j+1}\right\}  =D\left(  \beta\right)  }}x_{i_{1}}x_{i_{2}}\cdots
x_{i_{n}},
\end{align*}
we obtain%
\[
\sum_{\substack{i_{1}\leq i_{2}\leq\cdots\leq i_{n};\\i_{j}<i_{j+1}\text{
whenever }j\in D\left(  \alpha\right)  }}x_{i_{1}}x_{i_{2}}\cdots x_{i_{n}%
}=\sum_{\beta\text{ is a composition of }n;\ D\left(  \beta\right)  \supseteq
D\left(  \alpha\right)  }M_{\beta}.
\]
This proves Corollary \ref{cor.Dalpha.F}.
\end{proof}

\subsection{$\Gamma\left(  \mathbf{E},w\right)  $ is well-defined}

Let us next show a really simple fact that was left unproven in Definition
\ref{def.Gammaw}:

\begin{proposition}
\label{prop.Gammaw.welldef}Let ${\mathbf{E}}=\left(  E,<_{1},<_{2}\right)  $
be a double poset. Let $w:E\rightarrow\left\{  1,2,3,\ldots\right\}  $ be a
map. Then, the sum $\sum_{\pi\text{ is an }{\mathbf{E}}\text{-partition}%
}{\mathbf{x}}_{\pi,w}$ in $\mathbf{k}\left[  \left[  x_{1},x_{2},x_{3}%
,\ldots\right]  \right]  $ converges (with respect to the topology on
$\mathbf{k}\left[  \left[  x_{1},x_{2},x_{3},\ldots\right]  \right]  $).
\end{proposition}

Proposition \ref{prop.Gammaw.welldef} shows that the power series
$\Gamma\left(  {\mathbf{E}},w\right)  $ in Definition \ref{def.Gammaw} is well-defined.

\begin{proof}
[Proof of Proposition \ref{prop.Gammaw.welldef}.]We know that $\left(
E,<_{1},<_{2}\right)  $ is a double poset. Hence, $E$ is a finite set.

For every power series $f\in\mathbf{k}\left[  \left[  x_{1},x_{2},x_{3}%
,\ldots\right]  \right]  $ and every monomial $\mathfrak{m}$, we denote by
$\left[  \mathfrak{m}\right]  \left(  f\right)  $ the coefficient of
$\mathfrak{m}$ in $f$. Notice that any two monomials $\mathfrak{m}$ and
$\mathfrak{n}$ satisfy%
\begin{equation}
\left[  \mathfrak{m}\right]  \left(  \mathfrak{n}\right)  =%
\begin{cases}
1, & \text{if }\mathfrak{m}=\mathfrak{n};\\
0, & \text{if }\mathfrak{m}\neq\mathfrak{n}%
\end{cases}
. \label{pf.prop.Gammaw.welldef.mn}%
\end{equation}

The definition of the topology on $\mathbf{k}\left[  \left[  x_{1},x_{2}%
,x_{3},\ldots\right]  \right]  $ has the following consequence:

\begin{statement}
\textit{Fact 1:} Let $P$ be a set. Let $\left(  \alpha_{\pi}\right)  _{\pi\in
P}$ be a family of elements of $\mathbf{k}\left[  \left[  x_{1},x_{2}%
,x_{3},\ldots\right]  \right]  $. Assume that for every monomial
$\mathfrak{m}$, all but finitely many $\pi\in P$ satisfy $\left[
\mathfrak{m}\right]  \left(  \alpha_{\pi}\right)  =0$. Then, the sum
$\sum_{\pi\in P}\alpha_{\pi}$ converges.
\end{statement}

Now, let $\mathfrak{m}$ be a monomial. Let $Z$ be the set of all positive
integers $i$ such that the indeterminate $x_{i}$ appears in the monomial
$\mathfrak{m}$. Thus, $Z$ is a finite subset of $\left\{  1,2,3,\ldots
\right\}  $. Thus, $\left\vert Z\right\vert <\infty$. Also, $\left\vert
E\right\vert <\infty$ (since $E$ is finite).

If $A$ and $B$ are sets, and if $C$ is a subset of $B$, then
\begin{equation}
\left\{  \pi\in B^{A}\ \mid\ \pi\left(  A\right)  \subseteq C\right\}  \cong
C^{A}\text{ as sets} \label{pf.prop.Gammaw.welldef.1}%
\end{equation}
\footnote{\textit{Proof of (\ref{pf.prop.Gammaw.welldef.1}):} This is a
well-known fact about sets. The proof proceeds as follows:
\par
Let $A$ and $B$ be sets. Let $C$ be a subset of $B$. Let $\iota:C\rightarrow
B$ be the canonical inclusion map.
\par
\begin{itemize}
\item Define a map $\Phi:\left\{  \pi\in B^{A}\ \mid\ \pi\left(  A\right)
\subseteq C\right\}  \rightarrow C^{A}$ as follows:
\par
Let $f\in\left\{  \pi\in B^{A}\ \mid\ \pi\left(  A\right)  \subseteq
C\right\}  $. Thus, $f$ is an element of $B^{A}$ and satisfies $f \left(
A\right)  \subseteq C$. Hence, we can define a map $f^{\prime}:A\rightarrow C$
by $\left(  f^{\prime}\left(  a\right)  =f\left(  a\right)  \text{ for every
}a\in A\right)  $. Define $\Phi\left(  f\right)  $ to be this map $f^{\prime}%
$. Hence, a map $\Phi:\left\{  \pi\in B^{A}\ \mid\ \pi\left(  A\right)
\subseteq C\right\}  \rightarrow C^{A}$ is defined.
\par
\item Define a map $\Psi:C^{A}\rightarrow\left\{  \pi\in B^{A}\ \mid
\ \pi\left(  A\right)  \subseteq C\right\}  $ by
\[
\left(  \Psi\left(  g\right)  = \iota\circ g \qquad\text{for each } g \in
C^{A} \right)  .
\]
\end{itemize}
\par
It is easy to see that the maps $\Phi$ and $\Psi$ are mutually inverse
bijections. Hence, there is a bijection from $\left\{  \pi\in B^{A}\ \mid
\ \pi\left(  A\right)  \subseteq C\right\}  $ to $C^{A}$ (namely, the map
$\Phi$). This proves (\ref{pf.prop.Gammaw.welldef.1}).}.

Let $\operatorname*{Par}\mathbf{E}$ be the set of all $\mathbf{E}$-partitions.
Thus, every element of $\operatorname*{Par}\mathbf{E}$ is an $\mathbf{E}%
$-partition, hence a map $E\rightarrow\left\{  1,2,3,\ldots\right\}  $, hence
an element of $\left\{  1,2,3,\ldots\right\}  ^{E}$. In other words,
$\operatorname*{Par}\mathbf{E}\subseteq\left\{  1,2,3,\ldots\right\}  ^{E}$.

Let $Q$ be the subset%
\[
\left\{  \pi\in\operatorname*{Par}\mathbf{E}\ \mid\ \pi\left(  E\right)
\subseteq Z\right\}
\]
of $\operatorname*{Par}\mathbf{E}$. The set $Q$ is
finite\footnote{\textit{Proof.} From (\ref{pf.prop.Gammaw.welldef.1}) (applied
to $A=E$, $B=\left\{  1,2,3,\ldots\right\}  $ and $C=Z$), we obtain the fact
that%
\[
\left\{  \pi\in\left\{  1,2,3,\ldots\right\}  ^{E}\ \mid\ \pi\left(  E\right)
\subseteq Z\right\}  \cong Z^{E}\text{ as sets.}%
\]
Hence,%
\[
\left\vert \left\{  \pi\in\left\{  1,2,3,\ldots\right\}  ^{E}\ \mid
\ \pi\left(  E\right)  \subseteq Z\right\}  \right\vert =\left\vert
Z^{E}\right\vert =\left\vert Z\right\vert ^{\left\vert E\right\vert }<\infty
\]
(since $\left\vert Z\right\vert <\infty$ and $\left\vert E\right\vert <\infty
$). But%
\[
Q=\left\{  \pi\in\underbrace{\operatorname*{Par}\mathbf{E}}_{\subseteq\left\{
1,2,3,\ldots\right\}  ^{E}}\ \mid\ \pi\left(  E\right)  \subseteq Z\right\}
\subseteq\left\{  \pi\in\left\{  1,2,3,\ldots\right\}  ^{E}\ \mid\ \pi\left(
E\right)  \subseteq Z\right\}
\]
and thus%
\[
\left\vert Q\right\vert \leq\left\vert \left\{  \pi\in\left\{  1,2,3,\ldots
\right\}  ^{E}\ \mid\ \pi\left(  E\right)  \subseteq Z\right\}  \right\vert
<\infty.
\]
Hence, the set $Q$ is finite, qed.}.

We have $\left[  \mathfrak{m}\right]  \left(  {\mathbf{x}}_{\phi,w}\right)
=0$ for every $\phi\in\left(  \operatorname*{Par}\mathbf{E}\right)  \setminus
Q$\ \ \ \ \footnote{\textit{Proof.} Let $\phi\in\left(  \operatorname*{Par}%
\mathbf{E}\right)  \setminus Q$. We shall first show that $\mathfrak{m}%
\neq{\mathbf{x}}_{\phi,w}$ (as monomials).
\par
Indeed, assume the contrary. Thus, $\mathfrak{m}={\mathbf{x}}_{\phi,w}$ (as
monomials). But the definition of $\mathbf{x}_{\phi,w}$ shows that
$\mathbf{x}_{\phi,w}=\prod_{e\in E}x_{\phi\left(  e\right)  }^{w\left(
e\right)  }$.
\par
Now, let $f\in E$. Then, $w\left(  f\right)  \in\left\{  1,2,3,\ldots\right\}
$ (since $w$ is a map $E\rightarrow\left\{  1,2,3,\ldots\right\}  $). Hence,
$x_{\phi\left(  f\right)  }\mid x_{\phi\left(  f\right)  }^{w\left(  f\right)
}$ (as monomials). But $x_{\phi\left(  f\right)  }^{w\left(  f\right)  }$ is a
factor in the product $\prod_{e\in E}x_{\phi\left(  e\right)  }^{w\left(
e\right)  }$. Therefore, $x_{\phi\left(  f\right)  }^{w\left(  f\right)  }%
\mid\prod_{e\in E}x_{\phi\left(  e\right)  }^{w\left(  e\right)  }$ (as
monomials). Hence, $x_{\phi\left(  f\right)  }\mid x_{\phi\left(  f\right)
}^{w\left(  f\right)  }\mid\prod_{e\in E}x_{\phi\left(  e\right)  }^{w\left(
e\right)  }=\mathfrak{m}$ (as monomials). Thus, the indeterminate
$x_{\phi\left(  f\right)  }$ appears in the monomial $\mathfrak{m}$. Thus,
$\phi\left(  f\right)  $ is a positive integer $i$ such that the indeterminate
$x_{i}$ appears in the monomial $\mathfrak{m}$. In other words, $\phi\left(
f\right)  \in Z$ (since $Z$ is the set of all positive integers $i$ such that
the indeterminate $x_{i}$ appears in the monomial $\mathfrak{m}$).
\par
Now, forget that we fixed $f$. We thus have shown that $\phi\left(  f\right)
\in Z$ for each $f\in E$. In other words, $\phi\left(  E\right)  \subseteq Z$.
\par
Now, $\phi\in\left(  \operatorname*{Par}\mathbf{E}\right)  \setminus Q$. In
other words, $\phi\in\operatorname*{Par}\mathbf{E}$ but $\phi\notin Q$. We now
know that $\phi$ is an element of $\operatorname*{Par}\mathbf{E}$ satisfying
$\phi\left(  E\right)  \subseteq Z$. In other words, $\phi\in\left\{  \pi
\in\operatorname*{Par}\mathbf{E}\ \mid\ \pi\left(  E\right)  \subseteq
Z\right\}  $. In other words, $\phi\in Q$ (since $Q=\left\{  \pi
\in\operatorname*{Par}\mathbf{E}\ \mid\ \pi\left(  E\right)  \subseteq
Z\right\}  $). This contradicts $\phi\notin Q$.
\par
This contradiction proves that our assumption was wrong. Hence, $\mathfrak{m}%
\neq{\mathbf{x}}_{\phi,w}$ is proven. Now, (\ref{pf.prop.Gammaw.welldef.mn})
(applied to $\mathfrak{n}={\mathbf{x}}_{\phi,w}$) shows that%
\[
\left[  \mathfrak{m}\right]  \left(  {\mathbf{x}}_{\phi,w}\right)  =%
\begin{cases}
1, & \text{if }\mathfrak{m}={\mathbf{x}}_{\phi,w};\\
0, & \text{if }\mathfrak{m}\neq{\mathbf{x}}_{\phi,w}%
\end{cases}
=0\ \ \ \ \ \ \ \ \ \ \left(  \text{since }\mathfrak{m}\neq{\mathbf{x}}%
_{\phi,w}\right)  ,
\]
qed.}. Since $Q$ is a finite set, this shows that we have $\left[
\mathfrak{m}\right]  \left(  {\mathbf{x}}_{\phi,w}\right)  =0$ for all but
finitely many $\phi\in\operatorname*{Par}\mathbf{E}$. If we rename $\phi$ as
$\pi$ in this statement, we obtain the following: We have $\left[
\mathfrak{m}\right]  \left(  {\mathbf{x}}_{\pi,w}\right)  =0$ for all but
finitely many $\pi\in\operatorname*{Par}\mathbf{E}$. In other words, all but
finitely many $\pi\in\operatorname*{Par}\mathbf{E}$ satisfy $\left[
\mathfrak{m}\right]  \left(  {\mathbf{x}}_{\pi,w}\right)  =0$.

Let us now forget that we fixed $\mathfrak{m}$. We therefore have shown that
for every monomial $\mathfrak{m}$, all but finitely many $\pi\in
\operatorname*{Par}\mathbf{E}$ satisfy $\left[  \mathfrak{m}\right]  \left(
{\mathbf{x}}_{\pi,w}\right)  =0$. Thus, Fact 1 (applied to
$P=\operatorname*{Par}\mathbf{E}$ and $\alpha_{\pi}={\mathbf{x}}_{\pi,w}$)
shows that the sum $\sum_{\pi\in\operatorname*{Par}\mathbf{E}}{\mathbf{x}%
}_{\pi,w}$ converges. Since%
\[
\sum_{\pi\in\operatorname*{Par}\mathbf{E}}=\sum_{\pi\text{ is an }%
\mathbf{E}\text{-partition}}%
\]
(because $\operatorname*{Par}\mathbf{E}$ is the set of all $\mathbf{E}%
$-partitions), this rewrites as follows: The sum $\sum_{\pi\text{ is an
}{\mathbf{E}}\text{-partition}}{\mathbf{x}}_{\pi,w}$ converges. This proves
Proposition \ref{prop.Gammaw.welldef}.
\end{proof}

Let us also show a more detailed proof of Lemma \ref{lem.Gammaw.empty}:

\begin{proof}
[Proof of Lemma \ref{lem.Gammaw.empty}.](a) Assume that $E=\varnothing$. We
need to show that $\Gamma\left(  \mathbf{E},w\right)  =1$.

Let $\operatorname*{Par}\mathbf{E}$ be the set of all $\mathbf{E}$-partitions.
Let $g$ be the unique map $\varnothing\rightarrow\left\{  1,2,3,\ldots
\right\}  $. Then, $g\in\operatorname*{Par}\mathbf{E}$%
\ \ \ \ \footnote{\textit{Proof.} Recall the definition of an $\mathbf{E}%
$-partition. This definition shows that $g$ is an $\mathbf{E}$-partition if
and only if $g$ is a map $E\rightarrow\left\{  1,2,3,\ldots\right\}  $
satisfying the following two assertions:
\par
\textit{Assertion }$\mathcal{A}_{1}$\textit{:} Every $e\in E$ and $f\in E$
satisfying $e<_{1}f$ satisfy $g\left(  e\right)  \leq g\left(  f\right)  $.
\par
\textit{Assertion }$\mathcal{A}_{2}$\textit{:} Every $e\in E$ and $f\in E$
satisfying $e<_{1}f$ and $f<_{2}e$ satisfy $g\left(  e\right)  <g\left(
f\right)  $.
\par
Now, $g$ is a map $\varnothing\rightarrow\left\{  1,2,3,\ldots\right\}  $. In
other words, $g$ is a map $E\rightarrow\left\{  1,2,3,\ldots\right\}  $ (since
$E=\varnothing$). Also, there exists no $e\in E$ (since $E=\varnothing$).
Hence, Assertion $\mathcal{A}_{1}$ is vacuously true. Also, Assertion
$\mathcal{A}_{2}$ is vacuously true (for the same reason). Thus, $g$ is a map
$E\rightarrow\left\{  1,2,3,\ldots\right\}  $ satisfying the two Assertions
$\mathcal{A}_{1}$ and $\mathcal{A}_{2}$. In other words, $g$ is an
$\mathbf{E}$-partition (since $g$ is an $\mathbf{E}$-partition if and only if
$g$ is a map $E\rightarrow\left\{  1,2,3,\ldots\right\}  $ satisfying the two
Assertions $\mathcal{A}_{1}$ and $\mathcal{A}_{2}$). In other words, $g$
belongs to the set of all $\mathbf{E}$-partitions. In other words, $g$ belongs
to $\operatorname*{Par}\mathbf{E}$ (since $\operatorname*{Par}\mathbf{E}$ is
the set of all $\mathbf{E}$-partitions). In other words, $g\in
\operatorname*{Par}\mathbf{E}$. Qed.}. Thus, $\left\{  g\right\}
\subseteq\operatorname*{Par}\mathbf{E}$.

Now, $\operatorname*{Par}\mathbf{E}=\left\{  g\right\}  $%
\ \ \ \ \footnote{\textit{Proof.} Let $\phi\in\operatorname*{Par}\mathbf{E}$.
We have $\operatorname*{Par}\mathbf{E}\subseteq\left\{  1,2,3,\ldots\right\}
^{E}$ (indeed, we have shown this in the proof of Proposition
\ref{prop.Gammaw.welldef}). Since $E = \varnothing$, this rewrites as
$\operatorname*{Par}\mathbf{E}\subseteq\left\{  1,2,3,\ldots\right\}
^{\varnothing} = \left\{  g\right\}  $. Combining this with $\left\{
g\right\}  \subseteq\operatorname*{Par}\mathbf{E}$, we obtain
$\operatorname*{Par}\mathbf{E}=\left\{  g\right\}  $, qed.} and $\mathbf{x}%
_{g,w}=1$\ \ \ \ \footnote{\textit{Proof.} The definition of $\mathbf{x}%
_{g,w}$ yields
\begin{align*}
{\mathbf{x}}_{g,w}  &  =\prod_{e\in E}x_{g\left(  e\right)  }^{w\left(
e\right)  }=\prod_{e\in\varnothing}x_{g\left(  e\right)  }^{w\left(  e\right)
}\ \ \ \ \ \ \ \ \ \ \left(  \text{since }E=\varnothing\right) \\
&  =\left(  \text{empty product}\right)  =1,
\end{align*}
qed.}. Now, the definition of $\Gamma\left(  \mathbf{E},w\right)  $ yields%
\begin{align*}
\Gamma\left(  {\mathbf{E}},w\right)   &  =\underbrace{\sum_{\pi\text{ is an
}{\mathbf{E}}\text{-partition}}}_{\substack{=\sum_{\pi\in\operatorname*{Par}%
\mathbf{E}}\\\text{(since }\operatorname*{Par}\mathbf{E}\text{ is the set of
all }\mathbf{E}\text{-partitions)}}}{\mathbf{x}}_{\pi,w}=\sum_{\pi
\in\operatorname*{Par}\mathbf{E}}\mathbf{x}_{\pi,w}\\
&  =\sum_{\pi\in\left\{  g\right\}  }\mathbf{x}_{\pi,w}%
\ \ \ \ \ \ \ \ \ \ \left(  \text{since }\operatorname*{Par}\mathbf{E}%
=\left\{  g\right\}  \right) \\
&  =\mathbf{x}_{g,w}=1.
\end{align*}
This proves Lemma \ref{lem.Gammaw.empty} (a).

(b) Assume that $E\neq\varnothing$. We need to show that $\varepsilon\left(
\Gamma\left(  \mathbf{E},w\right)  \right)  =0$.

For every $\mathbf{E}$-partition $\pi$, we have%
\begin{equation}
\left(  \text{the constant term of }\mathbf{x}_{\pi,w}\right)  =0
\label{pf.lem.Gammaw.empty.b.1}%
\end{equation}
\footnote{\textit{Proof of (\ref{pf.lem.Gammaw.empty.b.1}):} Let $\pi$ be an
$\mathbf{E}$-partition. The definition of $\mathbf{x}_{\pi,w}$ yields
${\mathbf{x}}_{\pi,w}=\prod_{e\in E}x_{\pi\left(  e\right)  }^{w\left(
e\right)  }$. Thus, ${\mathbf{x}}_{\pi,w}$ is a monomial, and its degree is%
\[
\deg\left(  \underbrace{\mathbf{x}_{\pi,w}}_{=\prod_{e\in E}x_{\pi\left(
e\right)  }^{w\left(  e\right)  }}\right)  =\deg\left(  \prod_{e\in E}%
x_{\pi\left(  e\right)  }^{w\left(  e\right)  }\right)  =\sum_{e\in
E}\underbrace{\deg\left(  x_{\pi\left(  e\right)  }^{w\left(  e\right)
}\right)  }_{=w\left(  e\right)  }=\sum_{e\in E}w\left(  e\right)  .
\]
Now, the sum $\sum_{e\in E}w\left(  e\right)  $ is a nonempty sum (since
$E\neq\varnothing$), and all its addends are positive integers (since
$w\left(  e\right)  \in\left\{  1,2,3,\ldots\right\}  $ for each $e\in E$).
Thus, $\sum_{e\in E}w\left(  e\right)  $ is a nonempty sum of positive
integers, and therefore is itself a positive integer (since any nonempty sum
of positive integers is a positive integer). In other words, $\deg\left(
\mathbf{x}_{\pi,w}\right)  $ is a positive integer (since $\deg\left(
\mathbf{x}_{\pi,w}\right)  =\sum_{e\in E}w\left(  e\right)  $). Hence,
$\deg\left(  \mathbf{x}_{\pi,w}\right)  \neq0=\deg1$.
\par
But if $\mathfrak{m}$ is a monomial such that $\mathfrak{m}\neq1$, then
$\left(  \text{the constant term of }\mathfrak{m}\right)  =0$. Applying this
to $\mathfrak{m}=\mathbf{x}_{\pi,w}$, we obtain $\left(  \text{the constant
term of }\mathbf{x}_{\pi,w}\right)  =0$ (since $\mathbf{x}_{\pi,w}\neq1$).
This proves (\ref{pf.lem.Gammaw.empty.b.1}).}.

It is well-known that
\[
\varepsilon\left(  f\right)  =\left(  \text{the constant term of }f\right)
\ \ \ \ \ \ \ \ \ \ \text{for every }f\in\operatorname*{QSym}%
\]
(where the constant term of $f$ makes sense because $f$ is a power series).
Applying this to $f=\Gamma\left(  \mathbf{E},w\right)  $, we obtain%
\begin{align*}
\varepsilon\left(  \Gamma\left(  \mathbf{E},w\right)  \right)   &  =\left(
\text{the constant term of }\underbrace{\Gamma\left(  \mathbf{E},w\right)
}_{=\sum_{\pi\text{ is an }{\mathbf{E}}\text{-partition}}{\mathbf{x}}_{\pi,w}%
}\right) \\
&  =\left(  \text{the constant term of }\sum_{\pi\text{ is an }{\mathbf{E}%
}\text{-partition}}{\mathbf{x}}_{\pi,w}\right) \\
&  =\sum_{\pi\text{ is an }{\mathbf{E}}\text{-partition}}\underbrace{\left(
\text{the constant term of }\mathbf{x}_{\pi,w}\right)  }%
_{\substack{=0\\\text{(by (\ref{pf.lem.Gammaw.empty.b.1}))}}}=\sum_{\pi\text{
is an }{\mathbf{E}}\text{-partition}}0=0.
\end{align*}
This proves Lemma \ref{lem.Gammaw.empty} (b).
\end{proof}

\subsection{Increasing and strictly increasing maps as $\mathbf{E}%
$-partitions}

Now, we shall prove two claims that were left unproven in Example
\ref{exam.dp}:

\begin{proposition}
\label{prop.example.weaklinc}Let $\mathbf{E}=\left(  E,<_{1},<_{2}\right)  $
be a double poset. Assume that the order $<_{2}$ is an extension of the order
$<_{1}$ (that is, we have $u<_{2}v$ for every two elements $u$ and $v$ of $E$
satisfying $u<_{1}v$). Then, the $\mathbf{E}$-partitions are precisely the
weakly increasing maps from the poset $\left(  E,<_{1}\right)  $ to the
totally ordered set $\left\{  1,2,3,\ldots\right\}  $.
\end{proposition}

\begin{proposition}
\label{prop.example.strictinc}Let $\mathbf{E}=\left(  E,<_{1},<_{2}\right)  $
be a double poset. Let $>_{1}$ denote the opposite relation of $<_{1}$. Assume
that the order $<_{2}$ is an extension of the order $>_{1}$ (that is, we have
$u<_{2}v$ for every two elements $u$ and $v$ of $E$ satisfying $u>_{1}v$).
Then, the $\mathbf{E}$-partitions are precisely the strictly increasing maps
from the poset $\left(  E,<_{1}\right)  $ to the totally ordered set $\left\{
1,2,3,\ldots\right\}  $.
\end{proposition}

\begin{proof}
[Proof of Proposition \ref{prop.example.weaklinc}.]Let $\phi:E\rightarrow
\left\{  1,2,3,\ldots\right\}  $ be a map. We need to show the following
logical equivalence:%
\begin{align}
&  \left(  \phi\text{ is an }\mathbf{E}\text{-partition}\right) \nonumber\\
&  \Longleftrightarrow\ \left(  \phi\text{ is a weakly increasing map from
}\left(  E,<_{1}\right)  \text{ to }\left\{  1,2,3,\ldots\right\}  \right)  .
\label{pf.prop.example.weaklinc.goal}%
\end{align}

Recall the definition of an $\mathbf{E}$-partition. This definition shows that
$\phi$ is an $\mathbf{E}$-partition if and only if it satisfies the following
two assertions:

\textit{Assertion }$\mathcal{A}_{1}$\textit{:} Every $e\in E$ and $f\in E$
satisfying $e<_{1}f$ satisfy $\phi\left(  e\right)  \leq\phi\left(  f\right)
$.

\textit{Assertion }$\mathcal{A}_{2}$\textit{:} Every $e\in E$ and $f\in E$
satisfying $e<_{1}f$ and $f<_{2}e$ satisfy $\phi\left(  e\right)  <\phi\left(
f\right)  $.

In other words, we have the following logical equivalence:%
\[
\ \left(  \phi\text{ is an }\mathbf{E}\text{-partition}\right)
\ \Longleftrightarrow\ \left(  \text{Assertions }\mathcal{A}_{1}\text{ and
}\mathcal{A}_{2}\text{ hold}\right)  .
\]
But Assertion $\mathcal{A}_{2}$ always holds\footnote{\textit{Proof.} We shall
show that Assertion $\mathcal{A}_{2}$ is vacuously true. In other words, we
shall show that there exist no $e\in E$ and $f\in E$ satisfying $e<_{1}f$ and
$f<_{2}e$. Indeed, assume the contrary. Thus, there exist two $e\in E$ and
$f\in E$ satisfying $e<_{1}f$ and $f<_{2}e$. Consider these $e$ and $f$. From
$e<_{1}f$, we obtain $e<_{2}f$ (since the order $<_{2}$ is an extension of the
order $<_{1}$). This contradicts $f<_{2}e$ (since $<_{2}$ is a strict partial
order). This contradiction proves that our assumption was false. Hence, there
exist no $e\in E$ and $f\in E$ satisfying $e<_{1}f$ and $f<_{2}e$. Thus,
Assertion $\mathcal{A}_{2}$ is vacuously true, qed.}. Now, we have the
following chain of equivalences:%
\begin{align*}
&  \ \left(  \phi\text{ is an }\mathbf{E}\text{-partition}\right) \\
&  \Longleftrightarrow\ \left(  \text{Assertions }\mathcal{A}_{1}\text{ and
}\mathcal{A}_{2}\text{ hold}\right) \\
&  \Longleftrightarrow\ \left(  \text{Assertion }\mathcal{A}_{1}\text{
holds}\right)  \ \ \ \ \ \ \ \ \ \ \left(  \text{since Assertion }%
\mathcal{A}_{2}\text{ always holds}\right) \\
&  \Longleftrightarrow\ \left(  \text{every }e\in E\text{ and }f\in E\text{
satisfying }e<_{1}f\text{ satisfy }\phi\left(  e\right)  \leq\phi\left(
f\right)  \right) \\
&  \ \ \ \ \ \ \ \ \ \ \left(  \text{because this is Assertion }%
\mathcal{A}_{1}\right) \\
&  \Longleftrightarrow\ \left(  \phi\text{ is a weakly increasing map from
}\left(  E,<_{1}\right)  \text{ to }\left\{  1,2,3,\ldots\right\}  \right) \\
&  \ \ \ \ \ \ \ \ \ \ \left(  \text{by the definition of a \textquotedblleft
weakly increasing map\textquotedblright}\right)  .
\end{align*}
Thus, (\ref{pf.prop.example.weaklinc.goal}) is proven. This concludes the
proof of Proposition \ref{prop.example.weaklinc}.
\end{proof}

\begin{proof}
[Proof of Proposition \ref{prop.example.strictinc}.]Let $\phi:E\rightarrow
\left\{  1,2,3,\ldots\right\}  $ be a map. We need to show the following
logical equivalence:%
\begin{align}
&  \ \left(  \phi\text{ is an }\mathbf{E}\text{-partition}\right) \nonumber\\
&  \Longleftrightarrow\ \left(  \phi\text{ is a strictly increasing map from
}\left(  E,<_{1}\right)  \text{ to }\left\{  1,2,3,\ldots\right\}  \right)  .
\label{pf.prop.example.strictlinc.goal}%
\end{align}

Recall the definition of an $\mathbf{E}$-partition. This definition shows that
$\phi$ is an $\mathbf{E}$-partition if and only if it satisfies the following
two assertions:

\textit{Assertion }$\mathcal{A}_{1}$\textit{:} Every $e\in E$ and $f\in E$
satisfying $e<_{1}f$ satisfy $\phi\left(  e\right)  \leq\phi\left(  f\right)
$.

\textit{Assertion }$\mathcal{A}_{2}$\textit{:} Every $e\in E$ and $f\in E$
satisfying $e<_{1}f$ and $f<_{2}e$ satisfy $\phi\left(  e\right)  <\phi\left(
f\right)  $.

In other words, we have the following logical equivalence:%
\[
\ \left(  \phi\text{ is an }\mathbf{E}\text{-partition}\right)
\ \Longleftrightarrow\ \left(  \text{Assertions }\mathcal{A}_{1}\text{ and
}\mathcal{A}_{2}\text{ hold}\right)  .
\]

Now, consider the following assertion:

\textit{Assertion }$\mathcal{A}_{3}$\textit{:} Every $e\in E$ and $f\in E$
satisfying $e<_{1}f$ satisfy $\phi\left(  e\right)  <\phi\left(  f\right)  $.

The following logical implication is obvious:%
\begin{equation}
\left(  \text{Assertion }\mathcal{A}_{3}\text{ holds}\right)
\ \Longrightarrow\ \left(  \text{Assertions }\mathcal{A}_{1}\text{ and
}\mathcal{A}_{2}\text{ hold}\right)  . \label{pf.prop.example.strictlinc.3}%
\end{equation}
On the other hand, we also have the following logical implication:%
\begin{equation}
\left(  \text{Assertions }\mathcal{A}_{1}\text{ and }\mathcal{A}_{2}\text{
hold}\right)  \ \Longrightarrow\ \left(  \text{Assertion }\mathcal{A}%
_{3}\text{ holds}\right)  \label{pf.prop.example.strictlinc.4}%
\end{equation}
\footnote{\textit{Proof of (\ref{pf.prop.example.strictlinc.4}):} Assume that
Assertions $\mathcal{A}_{1}$ and $\mathcal{A}_{2}$ hold. We need to show that
Assertion $\mathcal{A}_{3}$ holds.
\par
Let $e\in E$ and $f\in E$ be such that $e<_{1}f$. From $e<_{1}f$, we conclude
that $f>_{1}e$, so that $f<_{2}e$ (since the order $<_{2}$ is an extension of
the order $>_{1}$). Hence, Assertion $\mathcal{A}_{2}$ shows that $\phi\left(
e\right)  <\phi\left(  f\right)  $.
\par
Now, forget that we fixed $e$ and $f$. We thus have proven that every $e\in E$
and $f\in E$ satisfying $e<_{1}f$ satisfy $\phi\left(  e\right)  <\phi\left(
f\right)  $. In other words, Assertion $\mathcal{A}_{3}$ holds. This proves
the implication (\ref{pf.prop.example.strictlinc.4}).}. Combining this
implication with (\ref{pf.prop.example.strictlinc.3}), we obtain the
equivalence%
\[
\left(  \text{Assertions }\mathcal{A}_{1}\text{ and }\mathcal{A}_{2}\text{
hold}\right)  \ \Longleftrightarrow\ \left(  \text{Assertion }\mathcal{A}%
_{3}\text{ holds}\right)  .
\]

Now, we have the following chain of equivalences:%
\begin{align*}
&  \ \left(  \phi\text{ is an }\mathbf{E}\text{-partition}\right) \\
&  \Longleftrightarrow\ \left(  \text{Assertions }\mathcal{A}_{1}\text{ and
}\mathcal{A}_{2}\text{ hold}\right) \\
&  \Longleftrightarrow\ \left(  \text{Assertion }\mathcal{A}_{3}\text{
holds}\right) \\
&  \Longleftrightarrow\ \left(  \text{every }e\in E\text{ and }f\in E\text{
satisfying }e<_{1}f\text{ satisfy }\phi\left(  e\right)  <\phi\left(
f\right)  \right) \\
&  \ \ \ \ \ \ \ \ \ \ \left(  \text{because this is Assertion }%
\mathcal{A}_{3}\right) \\
&  \Longleftrightarrow\ \left(  \phi\text{ is a strictly increasing map from
}\left(  E,<_{1}\right)  \text{ to }\left\{  1,2,3,\ldots\right\}  \right) \\
&  \ \ \ \ \ \ \ \ \ \ \left(  \text{by the definition of a \textquotedblleft
strictly increasing map\textquotedblright}\right)  .
\end{align*}
Thus, (\ref{pf.prop.example.strictlinc.goal}) is proven. This concludes the
proof of Proposition \ref{prop.example.strictinc}.
\end{proof}

\subsection{Semistandard tableaux as $\mathbf{E}$-partitions}

Let us now verify two further claims made in Example \ref{exam.dp} -- namely,
the claims about semistandard tableaux. We recall the definition of a
semistandard tableau:

\begin{definition}
\label{def.sst}Let $\mu$ and $\lambda$ be two partitions such that
$\mu\subseteq\lambda$. Define the set $Y\left(  \lambda/\mu\right)  $ as in
Example \ref{exam.dp}. A \textit{semistandard tableau of shape }$\lambda/\mu$
means a map $\phi:Y\left(  \lambda/\mu\right)  \rightarrow\left\{
1,2,3,\ldots\right\}  $ satisfying the following two assertions:

\textit{Assertion }$\mathcal{T}_{1}$: For any $\left(  i,a\right)  \in
Y\left(  \lambda/\mu\right)  $ and $\left(  i,b\right)  \in Y\left(
\lambda/\mu\right)  $ with $a<b$, we have $\phi\left(  i,a\right)  \leq
\phi\left(  i,b\right)  $.

\textit{Assertion }$\mathcal{T}_{2}$: For any $\left(  a,j\right)  \in
Y\left(  \lambda/\mu\right)  $ and $\left(  b,j\right)  \in Y\left(
\lambda/\mu\right)  $ with $a<b$, we have $\phi\left(  a,j\right)
<\phi\left(  b,j\right)  $.

(It is usual to visualize the set $Y\left(  \lambda/\mu\right)  $ as a set of
$1\times1$-squares on the integer lattice $\mathbb{Z}^{2}$; then, a map
$\phi:Y\left(  \lambda/\mu\right)  \rightarrow\left\{  1,2,3,\ldots\right\}  $
can be regarded as a filling of these squares with numbers in $\left\{
1,2,3,\ldots\right\}  $. In this visual representation, Assertion
$\mathcal{T}_{1}$ claims that the entries of the filling $\phi$ are weakly
increasing from left to right along each row of the tableau, whereas Assertion
$\mathcal{T}_{2}$ says that the entries of $\phi$ are strictly increasing from
top to bottom along each column of the tableau. See \cite[\S 2.2]{Reiner} for
more about semistandard tableaux (which are called column-strict tableaux in
\cite{Reiner}) as well as for examples of this visual representation.)
\end{definition}

The following proposition contains two unproven claims made in Example
\ref{exam.dp}:

\begin{proposition}
\label{prop.example.sst}Let $\mu$ and $\lambda$ be two partitions such that
$\mu\subseteq\lambda$. Define the double posets $\mathbf{Y}\left(  \lambda
/\mu\right)  $ and $\mathbf{Y}_{h}\left(  \lambda/\mu\right)  $ as in Example
\ref{exam.dp}. Then:

\begin{enumerate}
\item[(a)] The $\mathbf{Y}\left(  \lambda/\mu\right)  $-partitions are
precisely the semistandard tableaux of shape $\lambda/\mu$.

\item[(b)] The $\mathbf{Y}_{h}\left(  \lambda/\mu\right)  $-partitions are
precisely the semistandard tableaux of shape $\lambda/\mu$.
\end{enumerate}
\end{proposition}

We shall prove a slightly more general fact:

\begin{proposition}
\label{prop.example.sst.gen}Let $\mu$ and $\lambda$ be two partitions such
that $\mu\subseteq\lambda$. Define the set $Y\left(  \lambda/\mu\right)  $ and
the relation $<_{1}$ as in Example \ref{exam.dp}. Let $\prec$ be a strict
partial order on the set $Y\left(  \lambda/\mu\right)  $ such that the
following two conditions hold:

\textit{Condition $\mathcal{O}$}$_{1}$\textit{:} For any $\left(  i,a\right)
\in Y\left(  \lambda/\mu\right)  $ and $\left(  i,b\right)  \in Y\left(
\lambda/\mu\right)  $ with $a<b$, we have $\left(  i,a\right)  \prec\left(
i,b\right)  $.

\textit{Condition $\mathcal{O}$}$_{2}$\textit{:} For any $\left(  a,j\right)
\in Y\left(  \lambda/\mu\right)  $ and $\left(  b,j\right)  \in Y\left(
\lambda/\mu\right)  $ with $a<b$, we have $\left(  b,j\right)  \prec\left(
a,j\right)  $.

Then, the $\left(  Y\left(  \lambda/\mu\right)  ,<_{1},\prec\right)
$-partitions are precisely the semistandard tableaux of shape $\lambda/\mu$.
\end{proposition}

Before we begin proving this, let us however reach back and verify a really
basic result about partitions:

\begin{lemma}
\label{lem.example.sst.convex}Let $\mu$ and $\lambda$ be two partitions such
that $\mu\subseteq\lambda$. Let $\left(  a_{1} ,b_{1}\right)  $, $\left(
a_{2},b_{2}\right)  $ and $\left(  a_{3},b_{3}\right)  $ be three elements of
$\mathbb{Z}^{2}$ such that $\left(  a_{1},b_{1}\right)  \in Y\left(
\lambda/\mu\right)  $ and $\left(  a_{3},b_{3}\right)  \in Y\left(
\lambda/\mu\right)  $ and $a_{1}\leq a_{2}\leq a_{3}$ and $b_{1}\leq b_{2}\leq
b_{3}$. Then,
\begin{equation}
\left(  a_{2},b_{2}\right)  \in Y\left(  \lambda/\mu\right)  .
\label{pf.prop.example.sst.gen.2}%
\end{equation}

\end{lemma}

\begin{proof}
[Proof of Lemma~\ref{lem.example.sst.convex}.]Write the partition $\mu$ in the
form $\left(  \mu_{1},\mu_{2},\mu_{3},\ldots\right)  $.

Write the partition $\lambda$ in the form $\left(  \lambda_{1},\lambda
_{2},\lambda_{3},\ldots\right)  $.

Now, $Y\left(  \lambda/\mu\right)  $ is the set of all $\left(  i,j\right)
\in\left\{  1,2,3,\ldots\right\}  ^{2}$ satisfying $\mu_{i}<j\leq\lambda_{i}$
(by the definition of $Y\left(  \lambda/\mu\right)  $). In other words,%
\[
Y\left(  \lambda/\mu\right)  =\left\{  \left(  i,j\right)  \in\left\{
1,2,3,\ldots\right\}  ^{2}\ \mid\ \mu_{i}<j\leq\lambda_{i}\right\}  .
\]

Now,
\[
\left(  a_{1},b_{1}\right)  \in Y\left(  \lambda/\mu\right)  =\left\{  \left(
i,j\right)  \in\left\{  1,2,3,\ldots\right\}  ^{2}\ \mid\ \mu_{i}<j\leq
\lambda_{i}\right\}  .
\]
In other words, $\left(  a_{1},b_{1}\right)  $ is an element of $\left\{
1,2,3,\ldots\right\}  ^{2}$ and satisfies $\mu_{a_{1}}<b_{1}\leq\lambda
_{a_{1}}$.

Also,
\[
\left(  a_{3},b_{3}\right)  \in Y\left(  \lambda/\mu\right)  =\left\{  \left(
i,j\right)  \in\left\{  1,2,3,\ldots\right\}  ^{2}\ \mid\ \mu_{i}<j\leq
\lambda_{i}\right\}  .
\]
In other words, $\left(  a_{3},b_{3}\right)  $ is an element of $\left\{
1,2,3,\ldots\right\}  ^{2}$ and satisfies $\mu_{a_{3}}<b_{3}\leq\lambda
_{a_{3}}$.

From $\left(  a_{1},b_{1}\right)  \in\left\{  1,2,3,\ldots\right\}  ^{2}$, we
obtain $a_{1} \geq1$ and $b_{1} \geq1$. Now, $\left(  a_{2},b_{2}\right)  \in\left\{
1,2,3,\ldots\right\}  ^{2}$ (since $a_{2} \geq a_{1} \geq1$ and $b_{2} \geq
b_{1} \geq1$). Hence, $\lambda_{a_{2}}$ and $\mu_{a_{2}}$ are well-defined.

The sequence $\left(  \mu_{1},\mu_{2},\mu_{3},\ldots\right)  =\mu$ is a
partition. Thus, $\mu_{1}\geq\mu_{2}\geq\mu_{3}\geq\cdots$. In other words,
any two positive integers $u$ and $v$ satisfying $u\leq v$ satisfy $\mu
_{u}\geq\mu_{v}$. Applying this to $u=a_{1}$ and $v=a_{2}$, we obtain
$\mu_{a_{1}}\geq\mu_{a_{2}}$ (since $a_{1}\leq a_{2}$). Thus, $\mu_{a_{2}}%
\leq\mu_{a_{1}}<b_{1}\leq b_{2}$.

The sequence $\left(  \lambda_{1},\lambda_{2},\lambda_{3},\ldots\right)
=\lambda$ is a partition. Hence, $\lambda_{1}\geq\lambda_{2}\geq\lambda
_{3}\geq\cdots$. In other words, any two positive integers $u$ and $v$
satisfying $u\leq v$ satisfy $\lambda_{u}\geq\lambda_{v}$. Applying this to
$u=a_{2}$ and $v=a_{3}$, we obtain $\lambda_{a_{2}}\geq\lambda_{a_{3}}$ (since
$a_{2}\leq a_{3}$). Thus, $\lambda_{a_{3}}\leq\lambda_{a_{2}}$, so that
$b_{2}\leq b_{3}\leq\lambda_{a_{3}}\leq\lambda_{a_{2}}$. Hence, $\mu_{a_{2}%
}<b_{2}\leq\lambda_{a_{2}}$.

Now, we know that $\left(  a_{2},b_{2}\right)  $ is an element of $\left\{
1,2,3,\ldots\right\}  ^{2}$ and satisfies $\mu_{a_{2}}<b_{2}\leq\lambda
_{a_{2}}$. Hence,
\[
\left(  a_{2},b_{2}\right)  \in\left\{  \left(  i,j\right)  \in\left\{
1,2,3,\ldots\right\}  ^{2}\ \mid\ \mu_{i}<j\leq\lambda_{i}\right\}  =Y\left(
\lambda/\mu\right)  .
\]
This proves Lemma~\ref{lem.example.sst.convex}.
\end{proof}

\begin{proof}
[Proof of Proposition \ref{prop.example.sst.gen}.]Both $\left(  Y\left(
\lambda/\mu\right)  ,<_{1},\prec\right)  $-partitions and semistandard
tableaux of shape $\lambda/\mu$ are maps from $Y\left(  \lambda/\mu\right)  $
to $\left\{  1,2,3,\ldots\right\}  $. Our goal is to show that the $\left(
Y\left(  \lambda/\mu\right)  ,<_{1},\prec\right)  $-partitions are precisely
the semistandard tableaux of shape $\lambda/\mu$. In other words, our goal is
to prove that, for every map $\phi:Y\left(  \lambda/\mu\right)  \rightarrow
\left\{  1,2,3,\ldots\right\}  $, we have the following equivalence:%
\begin{align}
&  \left(  \phi\text{ is a }\left(  Y\left(  \lambda/\mu\right)  ,<_{1}%
,\prec\right)  \text{-partition}\right) \nonumber\\
&  \Longleftrightarrow\ \left(  \phi\text{ is a semistandard tableau of shape
}\lambda/\mu\right)  . \label{pf.prop.example.sst.gen.goal}%
\end{align}

Let $\phi:Y\left(  \lambda/\mu\right)  \rightarrow\left\{  1,2,3,\ldots
\right\}  $ be a map.

Recall the definition of a $\left(  Y\left(  \lambda/\mu\right)  ,<_{1}%
,\prec\right)  $-partition. This definition shows that $\phi$ is a $\left(
Y\left(  \lambda/\mu\right)  ,<_{1},\prec\right)  $-partition if and only if
it satisfies the following two assertions:

\textit{Assertion }$\mathcal{P}_{1}$\textit{:} Every $e\in Y\left(
\lambda/\mu\right)  $ and $f\in Y\left(  \lambda/\mu\right)  $ satisfying
$e<_{1}f$ satisfy $\phi\left(  e\right)  \leq\phi\left(  f\right)  $.

\textit{Assertion }$\mathcal{P}_{2}$\textit{:} Every $e\in Y\left(
\lambda/\mu\right)  $ and $f\in Y\left(  \lambda/\mu\right)  $ satisfying
$e<_{1}f$ and $f\prec e$ satisfy $\phi\left(  e\right)  <\phi\left(  f\right)
$.

On the other hand, recall the definition of a semistandard tableau of shape
$\lambda/\mu$. This definition shows that $\phi$ is a semistandard tableau of
shape $\lambda/\mu$ if and only if it satisfies the following two assertions:

\textit{Assertion }$\mathcal{T}_{1}$: For any $\left(  i,a\right)  \in
Y\left(  \lambda/\mu\right)  $ and $\left(  i,b\right)  \in Y\left(
\lambda/\mu\right)  $ with $a<b$, we have $\phi\left(  i,a\right)  \leq
\phi\left(  i,b\right)  $.

\textit{Assertion }$\mathcal{T}_{2}$: For any $\left(  a,j\right)  \in
Y\left(  \lambda/\mu\right)  $ and $\left(  b,j\right)  \in Y\left(
\lambda/\mu\right)  $ with $a<b$, we have $\phi\left(  a,j\right)
<\phi\left(  b,j\right)  $.

Now, we shall prove the logical implication%
\begin{align}
&  \ \left(  \phi\text{ is a }\left(  Y\left(  \lambda/\mu\right)
,<_{1},\prec\right)  \text{-partition}\right) \nonumber\\
&  \Longrightarrow\ \left(  \phi\text{ is a semistandard tableau of shape
}\lambda/\mu\right)  . \label{pf.prop.example.sst.gen.goal1}%
\end{align}

\textit{Proof of (\ref{pf.prop.example.sst.gen.goal1}):} Assume that $\phi$ is
a $\left(  Y\left(  \lambda/\mu\right)  ,<_{1},\prec\right)  $-partition. We
shall now prove that $\phi$ is a semistandard tableau of shape $\lambda/\mu$.

We know that $\phi$ is a $\left(  Y\left(  \lambda/\mu\right)  ,<_{1}%
,\prec\right)  $-partition if and only if $\phi$ satisfies Assertions
$\mathcal{P}_{1}$ and $\mathcal{P}_{2}$. Thus, $\phi$ satisfies Assertions
$\mathcal{P}_{1}$ and $\mathcal{P}_{2}$ (since $\phi$ is a $\left(  Y\left(
\lambda/\mu\right)  ,<_{1},\prec\right)  $-partition).

We now notice that $\phi$ satisfies Assertion $\mathcal{T}_{1}$%
\ \ \ \ \footnote{\textit{Proof.} Let $\left(  i,a\right)  \in Y\left(
\lambda/\mu\right)  $ and $\left(  i,b\right)  \in Y\left(  \lambda
/\mu\right)  $ be such that $a<b$. From $a<b$, we obtain $a\neq b$, hence
$\left(  i,a\right)  \neq\left(  i,b\right)  $. Also, $a \leq b$ (since $a <
b$).
\par
The definition of the relation $<_{1}$ shows that $\left(  i,a\right)
<_{1}\left(  i,b\right)  $ holds if and only if $\left(  i\leq i\text{ and
}a\leq b\text{ and }\left(  i,a\right)  \neq\left(  i,b\right)  \right)  $.
Thus, $\left(  i,a\right)  <_{1}\left(  i,b\right)  $ holds (since we have
$\left(  i\leq i\text{ and }a\leq b\text{ and }\left(  i,a\right)  \neq\left(
i,b\right)  \right)  $). But we know that $\phi$ satisfies Assertion
$\mathcal{P}_{1}$. Thus, Assertion $\mathcal{P}_{1}$ (applied to $e=\left(
i,a\right)  $ and $f=\left(  i,b\right)  $) shows that $\phi\left(
i,a\right)  \leq\phi\left(  i,b\right)  $ (since $\left(  i,a\right)
<_{1}\left(  i,b\right)  $).
\par
Now, forget that we fixed $\left(  i,a\right)  $ and $\left(  i,b\right)  $.
We thus have shown that for any $\left(  i,a\right)  \in Y\left(  \lambda
/\mu\right)  $ and $\left(  i,b\right)  \in Y\left(  \lambda/\mu\right)  $
with $a<b$, we have $\phi\left(  i,a\right)  \leq\phi\left(  i,b\right)  $. In
other words, $\phi$ satisfies Assertion $\mathcal{T}_{1}$.} and satisfies
Assertion $\mathcal{T}_{2}$\ \ \ \ \footnote{\textit{Proof.} Let $\left(
a,j\right)  \in Y\left(  \lambda/\mu\right)  $ and $\left(  b,j\right)  \in
Y\left(  \lambda/\mu\right)  $ be such that $a<b$. From $a<b$, we obtain
$a\neq b$, hence $\left(  a,j\right)  \neq\left(  b,j\right)  $. Also, $a \leq
b$ (since $a < b$).
\par
The definition of the relation $<_{1}$ shows that $\left(  a,j\right)
<_{1}\left(  b,j\right)  $ holds if and only if $\left(  a\leq b\text{ and
}j\leq j\text{ and }\left(  a,j\right)  \neq\left(  b,j\right)  \right)  $.
Thus, $\left(  a,j\right)  <_{1}\left(  b,j\right)  $ holds (since we have
$\left(  a\leq b\text{ and }j\leq j\text{ and }\left(  a,j\right)  \neq\left(
b,j\right)  \right)  $). Also, $\left(  b,j\right)  \prec\left(  a,j\right)  $
(because of Condition $\mathcal{O}_{2}$ in the statement of Proposition
\ref{prop.example.sst.gen}). But we know that $\phi$ satisfies Assertion
$\mathcal{P}_{2}$. Thus, Assertion $\mathcal{P}_{2}$ (applied to $e=\left(
a,j\right)  $ and $f=\left(  b,j\right)  $) shows that $\phi\left(
a,j\right)  <\phi\left(  b,j\right)  $ (since $\left(  a,j\right)
<_{1}\left(  b,j\right)  $ and $\left(  b,j\right)  \prec\left(  a,j\right)
$).
\par
Now, forget that we fixed $\left(  a,j\right)  $ and $\left(  b,j\right)  $.
We thus have shown that for any $\left(  a,j\right)  \in Y\left(  \lambda
/\mu\right)  $ and $\left(  b,j\right)  \in Y\left(  \lambda/\mu\right)  $
with $a<b$, we have $\phi\left(  a,j\right)  <\phi\left(  b,j\right)  $. In
other words, $\phi$ satisfies Assertion $\mathcal{T}_{2}$.}.

Recall that $\phi$ is a semistandard tableau of shape $\lambda/\mu$ if and
only if it satisfies Assertions $\mathcal{T}_{1}$ and $\mathcal{T}_{2}$. Thus,
$\phi$ is a semistandard tableau of shape $\lambda/\mu$ (since $\phi$
satisfies Assertions $\mathcal{T}_{1}$ and $\mathcal{T}_{2}$).

Now, forget that we have assumed that $\phi$ is a $\left(  Y\left(
\lambda/\mu\right)  ,<_{1},\prec\right)  $-partition. We thus have shown that
if $\phi$ is a $\left(  Y\left(  \lambda/\mu\right)  ,<_{1},\prec\right)
$-partition, then $\phi$ is a semistandard tableau of shape $\lambda/\mu$.
Thus, the implication (\ref{pf.prop.example.sst.gen.goal1}) is proven.

Let us next prove the logical implication%
\begin{align}
&  \ \left(  \phi\text{ is a semistandard tableau of shape }\lambda/\mu\right)
\nonumber\\
&  \Longrightarrow\ \left(  \phi\text{ is a }\left(  Y\left(  \lambda
/\mu\right)  ,<_{1},\prec\right)  \text{-partition}\right)  .
\label{pf.prop.example.sst.gen.goal2}%
\end{align}

\textit{Proof of (\ref{pf.prop.example.sst.gen.goal2}):} Assume that $\phi$ is
a semistandard tableau of shape $\lambda/\mu$. We shall now prove that $\phi$
is a $\left(  Y\left(  \lambda/\mu\right)  ,<_{1},\prec\right)  $-partition.

Recall that $\phi$ is a semistandard tableau of shape $\lambda/\mu$ if and
only if it satisfies Assertions $\mathcal{T}_{1}$ and $\mathcal{T}_{2}$. Thus,
$\phi$ satisfies Assertions $\mathcal{T}_{1}$ and $\mathcal{T}_{2}$ (since
$\phi$ is a semistandard tableau of shape $\lambda/\mu$).

Let us make some simple observations:

\begin{itemize}
\item For any $\left(  i,a\right)  \in Y\left(  \lambda/\mu\right)  $ and
$\left(  i,b\right)  \in Y\left(  \lambda/\mu\right)  $ with $a\leq b$, we
have
\begin{equation}
\phi\left(  i,a\right)  \leq\phi\left(  i,b\right)
\label{pf.prop.example.sst.gen.goal2.pf.1}%
\end{equation}
\footnote{\textit{Proof of (\ref{pf.prop.example.sst.gen.goal2.pf.1}):} Let
$\left(  i,a\right)  \in Y\left(  \lambda/\mu\right)  $ and $\left(
i,b\right)  \in Y\left(  \lambda/\mu\right)  $ be such that $a\leq b$. We must
prove the inequality (\ref{pf.prop.example.sst.gen.goal2.pf.1}). If $a=b$,
then this inequality holds (because if $a=b$, then $\phi\left(
i,\underbrace{a}_{=b}\right)  =\phi\left(  i,b\right)  \leq\phi\left(
i,b\right)  $). Hence, for the rest of this proof, we can WLOG assume that we
don't have $a=b$. Assume this.
\par
We have $a\neq b$ (since we don't have $a=b$). Combining this with $a\leq b$,
we obtain $a<b$. Now, recall that $\phi$ satisfies Assertion $\mathcal{T}_{1}%
$. Hence, Assertion $\mathcal{T}_{1}$ shows that $\phi\left(  i,a\right)
\leq\phi\left(  i,b\right)  $. This proves
(\ref{pf.prop.example.sst.gen.goal2.pf.1}).}.

\item For any $\left(  a,j\right)  \in Y\left(  \lambda/\mu\right)  $ and
$\left(  b,j\right)  \in Y\left(  \lambda/\mu\right)  $ with $a\leq b$, we
have%
\begin{equation}
\phi\left(  a,j\right)  \leq\phi\left(  b,j\right)
\label{pf.prop.example.sst.gen.goal2.pf.2}%
\end{equation}
\footnote{\textit{Proof of (\ref{pf.prop.example.sst.gen.goal2.pf.2}):} Let
$\left(  a,j\right)  \in Y\left(  \lambda/\mu\right)  $ and $\left(
b,j\right)  \in Y\left(  \lambda/\mu\right)  $ be such that $a\leq b$. We must
prove the inequality (\ref{pf.prop.example.sst.gen.goal2.pf.2}). If $a=b$,
then this inequality holds (because if $a=b$, then $\phi\left(  \underbrace{a}%
_{=b},j\right)  =\phi\left(  b,j\right)  \leq\phi\left(  b,j\right)  $).
Hence, for the rest of this proof, we can WLOG assume that we don't have
$a=b$. Assume this.
\par
We have $a\neq b$ (since we don't have $a=b$). Combining this with $a\leq b$,
we obtain $a<b$. Now, recall that $\phi$ satisfies Assertion $\mathcal{T}_{2}%
$. Hence, Assertion $\mathcal{T}_{2}$ shows that $\phi\left(  a,j\right)
<\phi\left(  b,j\right)  $. Thus, $\phi\left(  a,j\right)  \leq\phi\left(
b,j\right)  $. This proves (\ref{pf.prop.example.sst.gen.goal2.pf.2}).}.
\end{itemize}

Now, we can easily see that $\phi$ satisfies Assertion $\mathcal{P}_{1}%
$\ \ \ \ \footnote{\textit{Proof.} Let $e\in Y\left(  \lambda/\mu\right)  $
and $f\in Y\left(  \lambda/\mu\right)  $ be such that $e<_{1}f$. We shall
prove that $\phi\left(  e\right)  \leq\phi\left(  f\right)  $.
\par
We have $e\in Y\left(  \lambda/\mu\right)  \subseteq\left\{  1,2,3,\ldots
\right\}  ^{2}$. Hence, there exist two positive integers $a_{1}$ and $b_{1}$
such that $e=\left(  a_{1},b_{1}\right)  $. Consider these $a_{1}$ and $b_{1}%
$. Thus, $\left(  a_{1},b_{1}\right)  =e\in Y\left(  \lambda/\mu\right)  $.
\par
We have $f\in Y\left(  \lambda/\mu\right)  \subseteq\left\{  1,2,3,\ldots
\right\}  ^{2}$. Hence, there exist two positive integers $a_{3}$ and $b_{3}$
such that $f=\left(  a_{3},b_{3}\right)  $. Consider these $a_{3}$ and $b_{3}%
$. Thus, $\left(  a_{3},b_{3}\right)  =f\in Y\left(  \lambda/\mu\right)  $.
\par
We have $\left(  a_{1},b_{1}\right)  =e<_{1}f=\left(  a_{3},b_{3}\right)  $.
On the other hand, $\left(  a_{1},b_{1}\right)  <_{1}\left(  a_{3}%
,b_{3}\right)  $ holds if and only if $\left(  a_{1}\leq a_{3}\text{ and
}b_{1}\leq b_{3}\text{ and }\left(  a_{1},b_{1}\right)  \neq\left(
a_{3},b_{3}\right)  \right)  $ (by the definition of the relation $<_{1}$).
Thus, we have $\left(  a_{1}\leq a_{3}\text{ and }b_{1}\leq b_{3}\text{ and
}\left(  a_{1},b_{1}\right)  \neq\left(  a_{3},b_{3}\right)  \right)  $ (since
$\left(  a_{1},b_{1}\right)  <_{1}\left(  a_{3},b_{3}\right)  $ holds).
\par
Now, $a_{1}\leq a_{3}\leq a_{3}$ and $b_{1}\leq b_{1}\leq b_{3}$. Hence,
(\ref{pf.prop.example.sst.gen.2}) (applied to $\left(  a_{2},b_{2}\right)
=\left(  a_{3},b_{1}\right)  $) shows that $\left(  a_{3},b_{1}\right)  \in
Y\left(  \lambda/\mu\right)  $.
\par
Now, (\ref{pf.prop.example.sst.gen.goal2.pf.1}) (applied to $a_{3}$, $b_{1}$
and $b_{3}$ instead of $i$, $a$ and $b$) shows that $\phi\left(  a_{3}%
,b_{1}\right)  \leq\phi\left(  a_{3},b_{3}\right)  $ (since $\left(
a_{3},b_{1}\right)  \in Y\left(  \lambda/\mu\right)  $, $\left(  a_{3}%
,b_{3}\right)  \in Y\left(  \lambda/\mu\right)  $ and $b_{1}\leq b_{3}$).
Thus, $\phi\left(  a_{3},b_{1}\right)  \leq\phi\underbrace{\left(  a_{3}%
,b_{3}\right)  }_{=f}=\phi\left(  f\right)  $.
\par
Also, (\ref{pf.prop.example.sst.gen.goal2.pf.2}) (applied to $a_{1}$, $a_{3}$
and $b_{1}$ instead of $a$, $b$ and $j$) shows that $\phi\left(  a_{1}%
,b_{1}\right)  \leq\phi\left(  a_{3},b_{1}\right)  $ (since $\left(
a_{1},b_{1}\right)  \in Y\left(  \lambda/\mu\right)  $, $\left(  a_{3}%
,b_{1}\right)  \in Y\left(  \lambda/\mu\right)  $ and $a_{1}\leq a_{3}$).
Thus, $\phi\underbrace{\left(  e\right)  }_{=\left(  a_{1},b_{1}\right)
}=\phi\left(  a_{1},b_{1}\right)  \leq\phi\left(  a_{3},b_{1}\right)  $.
\par
Now, $\phi\left(  e\right)  \leq\phi\left(  a_{3},b_{1}\right)  \leq
\phi\left(  f\right)  $.
\par
Now, forget that we fixed $e$ and $f$. We thus have shown that every $e\in
Y\left(  \lambda/\mu\right)  $ and $f\in Y\left(  \lambda/\mu\right)  $
satisfying $e<_{1}f$ satisfy $\phi\left(  e\right)  \leq\phi\left(  f\right)
$. In other words, $\phi$ satisfies Assertion $\mathcal{P}_{1}$.} and
Assertion $\mathcal{P}_{2}$\ \ \ \ \footnote{\textit{Proof.} Let $e\in
Y\left(  \lambda/\mu\right)  $ and $f\in Y\left(  \lambda/\mu\right)  $ be
such that $e<_{1}f$ and $f\prec e$. We shall prove that $\phi\left(  e\right)
<\phi\left(  f\right)  $.
\par
Assume the contrary. Thus, $\phi\left(  e\right)  \geq\phi\left(  f\right)  $.
\par
We have $e\in Y\left(  \lambda/\mu\right)  \subseteq\left\{  1,2,3,\ldots
\right\}  ^{2}$. Hence, there exist two positive integers $a_{1}$ and $b_{1}$
such that $e=\left(  a_{1},b_{1}\right)  $. Consider these $a_{1}$ and $b_{1}%
$. Thus, $\left(  a_{1},b_{1}\right)  =e\in Y\left(  \lambda/\mu\right)  $.
\par
We have $f\in Y\left(  \lambda/\mu\right)  \subseteq\left\{  1,2,3,\ldots
\right\}  ^{2}$. Hence, there exist two positive integers $a_{3}$ and $b_{3}$
such that $f=\left(  a_{3},b_{3}\right)  $. Consider these $a_{3}$ and $b_{3}%
$. Thus, $\left(  a_{3},b_{3}\right)  =f\in Y\left(  \lambda/\mu\right)  $.
\par
We have $\left(  a_{1},b_{1}\right)  =e<_{1}f=\left(  a_{3},b_{3}\right)  $.
On the other hand, $\left(  a_{1},b_{1}\right)  <_{1}\left(  a_{3}%
,b_{3}\right)  $ holds if and only if $\left(  a_{1}\leq a_{3}\text{ and
}b_{1}\leq b_{3}\text{ and }\left(  a_{1},b_{1}\right)  \neq\left(
a_{3},b_{3}\right)  \right)  $ (by the definition of the relation $<_{1}$).
Thus, we have $\left(  a_{1}\leq a_{3}\text{ and }b_{1}\leq b_{3}\text{ and
}\left(  a_{1},b_{1}\right)  \neq\left(  a_{3},b_{3}\right)  \right)  $ (since
$\left(  a_{1},b_{1}\right)  <_{1}\left(  a_{3},b_{3}\right)  $ holds).
\par
Now, $a_{1}\leq a_{3}\leq a_{3}$ and $b_{1}\leq b_{1}\leq b_{3}$. Hence,
(\ref{pf.prop.example.sst.gen.2}) (applied to $\left(  a_{2},b_{2}\right)
=\left(  a_{3},b_{1}\right)  $) shows that $\left(  a_{3},b_{1}\right)  \in
Y\left(  \lambda/\mu\right)  $.
\par
Now, (\ref{pf.prop.example.sst.gen.goal2.pf.1}) (applied to $a_{3}$, $b_{1}$
and $b_{3}$ instead of $i$, $a$ and $b$) shows that $\phi\left(  a_{3}%
,b_{1}\right)  \leq\phi\left(  a_{3},b_{3}\right)  $ (since $\left(
a_{3},b_{1}\right)  \in Y\left(  \lambda/\mu\right)  $, $\left(  a_{3}%
,b_{3}\right)  \in Y\left(  \lambda/\mu\right)  $ and $b_{1}\leq b_{3}$).
Thus,
\begin{align*}
\phi\left(  a_{3},b_{1}\right)   &  \leq\phi\underbrace{\left(  a_{3}%
,b_{3}\right)  }_{=f}=\phi\left(  f\right)  \leq\phi\underbrace{\left(
e\right)  }_{=\left(  a_{1},b_{1}\right)  }\ \ \ \ \ \ \ \ \ \ \left(
\text{since }\phi\left(  e\right)  \geq\phi\left(  f\right)  \right) \\
&  =\phi\left(  a_{1},b_{1}\right)  .
\end{align*}
\par
Now, assume (for the sake of contradiction) that $a_{1}<a_{3}$. Recall that
$\phi$ satisfies Assertion $\mathcal{T}_{2}$. Hence, Assertion $\mathcal{T}%
_{2}$ (applied to $a_{1}$, $a_{3}$ and $b_{1}$ instead of $a$, $b$ and $j$)
shows that $\phi\left(  a_{1},b_{1}\right)  <\phi\left(  a_{3},b_{1}\right)  $
(since $\left(  a_{1},b_{1}\right)  \in Y\left(  \lambda/\mu\right)  $,
$\left(  a_{3},b_{1}\right)  \in Y\left(  \lambda/\mu\right)  $ and
$a_{1}<a_{3}$). This contradicts $\phi\left(  a_{3},b_{1}\right)  \leq
\phi\left(  a_{1},b_{1}\right)  $. This contradiction shows that our
assumption (that $a_{1}<a_{3}$) was false. Hence, $a_{1}\geq a_{3}$. Combining
this with $a_{1}\leq a_{3}$, we obtain $a_{1}=a_{3}$. Thus, $\left(
\underbrace{a_{1}}_{=a_{3}},b_{3}\right)  =\left(  a_{3},b_{3}\right)  =f\in
Y\left(  \lambda/\mu\right)  $.
\par
If we had $b_{1}=b_{3}$, then we would have $\left(  \underbrace{a_{1}%
}_{=a_{3}},\underbrace{b_{1}}_{=b_{3}}\right)  =\left(  a_{3},b_{3}\right)  $,
which would contradict $\left(  a_{1},b_{1}\right)  \neq\left(  a_{3}%
,b_{3}\right)  $. Hence, we cannot have $b_{1}=b_{3}$. Thus, we have
$b_{1}\neq b_{3}$. Combined with $b_{1}\leq b_{3}$, this shows that
$b_{1}<b_{3}$.
\par
The relation $\prec$ is a strict partial order, and thus is antisymmetric.
\par
However, recall that the relation $\prec$ satisfies Condition $\mathcal{O}%
_{1}$ in the statement of Proposition \ref{prop.example.sst.gen}. This
Condition $\mathcal{O}_{1}$ (applied to $a_{1}$, $b_{1}$ and $b_{3}$ instead
of $i$, $a$ and $b$) shows that $\left(  a_{1},b_{1}\right)  \prec\left(
a_{1},b_{3}\right)  $ (since $\left(  a_{1},b_{1}\right)  \in Y\left(
\lambda/\mu\right)  $, $\left(  a_{1},b_{3}\right)  \in Y\left(  \lambda
/\mu\right)  $ and $b_{1}<b_{3}$). Thus, $e=\left(  a_{1},b_{1}\right)
\prec\left(  \underbrace{a_{1}}_{=a_{3}},b_{3}\right)  =\left(  a_{3}%
,b_{3}\right)  =f$. Hence, we cannot have $f\prec e$ (since the relation
$\prec$ is antisymmetric). This contradicts $f\prec e$. This contradiction
proves that our assumption was wrong. Hence, the proof of $\phi\left(
e\right)  <\phi\left(  f\right)  $ is complete.
\par
Now, forget that we fixed $e$ and $f$. We thus have shown that every $e\in
Y\left(  \lambda/\mu\right)  $ and $f\in Y\left(  \lambda/\mu\right)  $
satisfying $e<_{1}f$ and $f\prec e$ satisfy $\phi\left(  e\right)
<\phi\left(  f\right)  $. In other words, $\phi$ satisfies Assertion
$\mathcal{P}_{2}$.}.

Recall that $\phi$ is a $\left(  Y\left(  \lambda/\mu\right)  ,<_{1}%
,\prec\right)  $-partition if and only if $\phi$ satisfies Assertions
$\mathcal{P}_{1}$ and $\mathcal{P}_{2}$. Thus, $\phi$ is a $\left(  Y\left(
\lambda/\mu\right)  ,<_{1},\prec\right)  $-partition (since $\phi$ satisfies
Assertions $\mathcal{P}_{1}$ and $\mathcal{P}_{2}$).

Now, forget that we have assumed that $\phi$ is a semistandard tableau of
shape $\lambda/\mu$. We thus have shown that if $\phi$ is a semistandard
tableau of shape $\lambda/\mu$, then $\phi$ is a $\left(  Y\left(  \lambda
/\mu\right)  ,<_{1},\prec\right)  $-partition. Thus, the implication
(\ref{pf.prop.example.sst.gen.goal2}) is proven.

Now, combining the implications (\ref{pf.prop.example.sst.gen.goal1}) and
(\ref{pf.prop.example.sst.gen.goal2}), we obtain the equivalence%
\begin{align*}
&  \left(  \phi\text{ is a }\left(  Y\left(  \lambda/\mu\right)  ,<_{1}%
,\prec\right)  \text{-partition}\right) \\
&  \Longleftrightarrow\ \left(  \phi\text{ is a semistandard tableau of shape
}\lambda/\mu\right)  .
\end{align*}
Thus, the equivalence (\ref{pf.prop.example.sst.gen.goal}) is proven.

Now, forget that we fixed $\phi$. Thus, for every map $\phi:Y\left(
\lambda/\mu\right)  \rightarrow\left\{  1,2,3,\ldots\right\}  $, we have
proven the equivalence (\ref{pf.prop.example.sst.gen.goal}). Thus, a map
$\phi:Y\left(  \lambda/\mu\right)  \rightarrow\left\{  1,2,3,\ldots\right\}  $
is a $\left(  Y\left(  \lambda/\mu\right)  ,<_{1},\prec\right)  $-partition if
and only if it is a semistandard tableau of shape $\lambda/\mu$. Hence, the
$\left(  Y\left(  \lambda/\mu\right)  ,<_{1},\prec\right)  $-partitions are
precisely the semistandard tableaux of shape $\lambda/\mu$ (because both
$\left(  Y\left(  \lambda/\mu\right)  ,<_{1},\prec\right)  $-partitions and
semistandard tableaux of shape $\lambda/\mu$ are maps from $Y\left(
\lambda/\mu\right)  $ to $\left\{  1,2,3,\ldots\right\}  $).
\end{proof}

\begin{proof}
[Proof of Proposition \ref{prop.example.sst}.]Define the set $Y\left(
\lambda/\mu\right)  $ as in Example \ref{exam.dp}.

Define the relations $<_{1}$, $<_{2}$ and $<_{h}$ as in Example \ref{exam.dp}.
It is straightforward to see that $<_{1}$, $<_{2}$ and $<_{h}$ are strict
partial orders. The definition of $\mathbf{Y}\left(  \lambda/\mu\right)  $
yields $\mathbf{Y}\left(  \lambda/\mu\right)  =\left(  Y\left(  \lambda
/\mu\right)  ,<_{1},<_{2}\right)  $. The definition of $\mathbf{Y}_{h}\left(
\lambda/\mu\right)  $ yields $\mathbf{Y}_{h}\left(  \lambda/\mu\right)
=\left(  Y\left(  \lambda/\mu\right)  ,<_{1},<_{h}\right)  $.

(a) It is straightforward to see that the relation $<_{2}$ satisfies the
Conditions $\mathcal{O}_{1}$ and $\mathcal{O}_{2}$ in the statement of
Proposition \ref{prop.example.sst.gen} (with $\prec$ replaced by $<_{2}$).
Hence, Proposition \ref{prop.example.sst.gen} (applied to $<_{2}$ instead of
$\prec$) shows that the $\left(  Y\left(  \lambda/\mu\right)  ,<_{1}%
,<_{2}\right)  $-partitions are precisely the semistandard tableaux of shape
$\lambda/\mu$. In other words, the $\mathbf{Y}\left(  \lambda/\mu\right)
$-partitions are precisely the semistandard tableaux of shape $\lambda/\mu$
(since $\mathbf{Y}\left(  \lambda/\mu\right)  =\left(  Y\left(  \lambda
/\mu\right)  ,<_{1},<_{2}\right)  $). This proves Proposition
\ref{prop.example.sst} (a).

(b) The proof of Proposition \ref{prop.example.sst} (b) is analogous to that
of Proposition \ref{prop.example.sst} (a) (but now, the relation $<_{h}$ and
the double poset $\mathbf{Y}_{h}\left(  \lambda/\mu\right)  $ replace the
relation $<_{2}$ and the double poset $\mathbf{Y}\left(  \lambda/\mu\right)  $).
\end{proof}

\subsection{$M_{\alpha}$ as $\Gamma\left(  \mathbf{E},w\right)  $}

Next, let us prove the claim of Example \ref{exam.Gamma} (b):

\begin{proposition}
\label{prop.example.Gamma.b}Let $\ell\in\mathbb{N}$. Let $E=\left\{
1,2,\ldots,\ell\right\}  $. Let $<_{1}$ be the restriction of the standard
relation $<$ on $\mathbb{Z}$ to the subset $E$. (Thus, two elements $e$ and
$f$ of $E$ satisfy $e<_{1}f$ if and only if $e<f$.) Let $>_{1}$ be the
opposite relation of $<_{1}$. (Thus, two elements $e$ and $f$ of $E$ satisfy
$e>_{1}f$ if and only if $f<_{1}e$.) Let ${\mathbf{E}}=\left(  E,<_{1}%
,>_{1}\right)  $.

\begin{enumerate}
\item[(a)] Then, $\mathbf{E}$ is a special double poset.

\item[(b)] Let $w:E\rightarrow\left\{  1,2,3,\ldots\right\}  $ be any map. Set
$\alpha=\left(  w\left(  1\right)  ,w\left(  2\right)  ,\ldots,w\left(
\ell\right)  \right)  $. Then, $\alpha$ is a composition and satisfies
$\Gamma\left(  {\mathbf{E}},w\right)  =M_{\alpha}$.
\end{enumerate}
\end{proposition}

\begin{proof}
[Proof of Proposition \ref{prop.example.Gamma.b}.](a) The relation $<_{1}$ is
a total order (since it is a restriction of the relation $<$ on $\mathbb{Z}$,
which is a total order). Hence, the relation $>_{1}$ is a total order as well
(since it is the opposite relation of the total order $<_{1}$). Thus, $\left(
E,<_{1},>_{1}\right)  $ is a special double poset (by the definition of
\textquotedblleft special\textquotedblright). In other words, $\mathbf{E}$ is
a special double poset (since ${\mathbf{E}}=\left(  E,<_{1},>_{1}\right)  $).
This proves Proposition \ref{prop.example.Gamma.b} (a).

(b) The map $w$ is a map $E\rightarrow\left\{  1,2,3,\ldots\right\}  $. In
other words, the map $w$ is a map $\left\{  1,2,\ldots,\ell\right\}
\rightarrow\left\{  1,2,3,\ldots\right\}  $ (since $E=\left\{  1,2,\ldots
,\ell\right\}  $).
Thus, $\left(  w\left(  1\right)  ,w\left(  2\right)  ,\ldots,w\left(
\ell\right)  \right)  $ is a sequence of positive integers, i.e., a
composition. In other words, $\alpha$ is a composition (since $\alpha=\left(
w\left(  1\right)  ,w\left(  2\right)  ,\ldots,w\left(  \ell\right)  \right)
$).

We have $\alpha=\left(  w\left(  1\right)  ,w\left(  2\right)  ,\ldots
,w\left(  \ell\right)  \right)  $. Thus, the definition of $M_{\alpha}$ yields%
\begin{equation}
M_{\alpha}=\sum_{i_{1}<i_{2}<\cdots<i_{\ell}}x_{i_{1}}^{w\left(  1\right)
}x_{i_{2}}^{w\left(  2\right)  }\cdots x_{i_{\ell}}^{w\left(  \ell\right)  }.
\label{pf.prop.example.Gamma.b.b.Malpha=}%
\end{equation}

It remains to prove that $\Gamma\left(  \mathbf{E},w\right)  =M_{\alpha}$. The
order $>_{1}$ is an extension of the order $>_{1}$ (obviously). Thus,
Proposition \ref{prop.example.strictinc} (applied to $>_{1}$ instead of
$<_{2}$) shows that the $\mathbf{E}$-partitions are precisely the strictly
increasing maps from the poset $\left(  E,<_{1}\right)  $ to the totally
ordered set $\left\{  1,2,3,\ldots\right\}  $.

On the other hand, let $\mathcal{J}$ denote the set of all length-$\ell$
strictly increasing sequences of positive integers. In other words,%
\[
\mathcal{J}=\left\{  \left(  i_{1},i_{2},\ldots,i_{\ell}\right)  \in\left\{
1,2,3,\ldots\right\}  ^{\ell}\ \mid\ i_{1}<i_{2}<\cdots<i_{\ell}\right\}  .
\]
Thus,%
\[
\sum_{\left(  i_{1},i_{2},\ldots,i_{\ell}\right)  \in\mathcal{J}}%
=\sum_{\substack{\left(  i_{1},i_{2},\ldots,i_{\ell}\right)  \in\left\{
1,2,3,\ldots\right\}  ^{\ell};\\i_{1}<i_{2}<\cdots<i_{\ell}}}=\sum
_{i_{1}<i_{2}<\cdots<i_{\ell}}%
\]
(an equality between summation signs).

Let $Z$ denote the set of all $\mathbf{E}$-partitions.

For every $\phi\in Z$, we have $\left(  \phi\left(  1\right)  ,\phi\left(
2\right)  ,\ldots,\phi\left(  \ell\right)  \right)  \in\mathcal{J}%
$\ \ \ \ \footnote{\textit{Proof.} Let $\phi\in Z$. Thus, $\phi$ is an element
of $Z$. In other words, $\phi$ is an $\mathbf{E}$-partition (since $Z$ is the
set of all $\mathbf{E}$-partitions). In other words, $\phi$ is a strictly
increasing map from the poset $\left(  E,<_{1}\right)  $ to the totally
ordered set $\left\{  1,2,3,\ldots\right\}  $ (since the $\mathbf{E}%
$-partitions are precisely the strictly increasing maps from the poset
$\left(  E,<_{1}\right)  $ to the totally ordered set $\left\{  1,2,3,\ldots
\right\}  $). In other words, $\phi$ is a map $E\rightarrow\left\{
1,2,3,\ldots\right\}  $ which has the property that if $e$ and $f$ are two
elements of $E$ satisfying $e<_{1}f$, then%
\begin{equation}
\phi\left(  e\right)  <\phi\left(  f\right)
\label{pf.prop.example.Gamma.b.b.fn1.1}%
\end{equation}
(by the definition of a \textquotedblleft strictly increasing
map\textquotedblright).
\par
The map $\phi$ is a map $E\rightarrow\left\{  1,2,3,\ldots\right\}  $. In
other words, the map $\phi$ is a map $\left\{  1,2,\ldots,\ell\right\}
\rightarrow\left\{  1,2,3,\ldots\right\}  $ (since $E=\left\{  1,2,\ldots
,\ell\right\}  $).
Hence, $\left(  \phi\left(  1\right)  ,\phi\left(  2\right)  ,\ldots
,\phi\left(  \ell\right)  \right)  $ is an element of $\left\{  1,2,3,\ldots
\right\}  ^{\ell}$.
\par
Now, let $i$ and $j$ be two elements of $\left\{  1,2,\ldots,\ell\right\}  $
satisfying $i<j$. The elements $\phi\left(  i\right)  $ and $\phi\left(
j\right)  $ are well-defined (since $i$ and $j$ belong to $\left\{
1,2,\ldots,\ell\right\}  = E$). We furthermore have $i < j$. In other words,
$i <_{1} j$ (since the relation $<_{1}$ is the restriction of the standard
relation $<$ on $\mathbb{Z}$ to the subset $E$). Thus,
(\ref{pf.prop.example.Gamma.b.b.fn1.1}) (applied to $e=i$ and $f=j$) shows
that $\phi\left(  i\right)  <\phi\left(  j\right)  $.
\par
Now, forget that we fixed $i$ and $j$. We thus have shown that if $i$ and $j$
are two elements of $\left\{  1,2,\ldots,\ell\right\}  $ satisfying $i<j$,
then $\phi\left(  i\right)  <\phi\left(  j\right)  $. In other words,
$\phi\left(  1\right)  <\phi\left(  2\right)  <\cdots<\phi\left(  \ell\right)
$.
\par
Now, $\left(  \phi\left(  1\right)  ,\phi\left(  2\right)  ,\ldots,\phi\left(
\ell\right)  \right)  $ is an element of $\left\{  1,2,3,\ldots\right\}
^{\ell}$ and satisfies $\phi\left(  1\right)  <\phi\left(  2\right)
<\cdots<\phi\left(  \ell\right)  $. Hence,
\begin{align*}
&  \left(  \phi\left(  1\right)  ,\phi\left(  2\right)  ,\ldots,\phi\left(
\ell\right)  \right)  \in\left\{  \left(  i_{1},i_{2},\ldots,i_{\ell}\right)
\in\left\{  1,2,3,\ldots\right\}  ^{\ell}\ \mid\ i_{1}<i_{2}<\cdots<i_{\ell
}\right\}  =\mathcal{J},
\end{align*}
qed.}. Hence, we can define a map $\Phi:Z\rightarrow\mathcal{J}$ by%
\[
\left(  \Phi\left(  \phi\right)  =\left(  \phi\left(  1\right)  ,\phi\left(
2\right)  ,\ldots,\phi\left(  \ell\right)  \right)
\ \ \ \ \ \ \ \ \ \ \text{for every }\phi\in Z\right)  .
\]
Consider this map $\Phi$. This map $\Phi$ is
injective\footnote{\textit{Proof.} Let $\phi_{1}$ and $\phi_{2}$ be two
elements of $Z$ such that $\Phi\left(  \phi_{1}\right)  =\Phi\left(  \phi
_{2}\right)  $. We shall show that $\phi_{1}=\phi_{2}$.
\par
The definition of $\Phi$ shows that $\Phi\left(  \phi_{1}\right)  =\left(
\phi_{1}\left(  1\right)  ,\phi_{1}\left(  2\right)  ,\ldots,\phi_{1}\left(
\ell\right)  \right)  $. The definition of $\Phi$ shows that $\Phi\left(
\phi_{2}\right)  =\left(  \phi_{2}\left(  1\right)  ,\phi_{2}\left(  2\right)
,\ldots,\phi_{2}\left(  \ell\right)  \right)  $. Hence,%
\[
\left(  \phi_{1}\left(  1\right)  ,\phi_{1}\left(  2\right)  ,\ldots,\phi
_{1}\left(  \ell\right)  \right)  =\Phi\left(  \phi_{1}\right)  =\Phi\left(
\phi_{2}\right)  =\left(  \phi_{2}\left(  1\right)  ,\phi_{2}\left(  2\right)
,\ldots,\phi_{2}\left(  \ell\right)  \right)  .
\]
In other words, $\phi_{1}\left(  i\right)  =\phi_{2}\left(  i\right)  $ for
each $i\in\left\{  1,2,\ldots,\ell\right\}  $. In other words, $\phi
_{1}\left(  i\right)  =\phi_{2}\left(  i\right)  $ for each $i\in E$ (since
$E=\left\{  1,2,\ldots,\ell\right\}  $). In other words, $\phi_{1}=\phi_{2}$.
\par
Now, forget that we fixed $\phi_{1}$ and $\phi_{2}$. We thus have shown that
if $\phi_{1}$ and $\phi_{2}$ are two elements of $Z$ such that $\Phi\left(
\phi_{1}\right)  =\Phi\left(  \phi_{2}\right)  $, then $\phi_{1}=\phi_{2}$. In
other words, the map $\Phi$ is injective, qed.} and
surjective\footnote{\textit{Proof.} Let $\mathbf{j}\in\mathcal{J}$. We shall
show that $\mathbf{j}\in\Phi\left(  Z\right)  $.
\par
We have $\mathbf{j}\in\mathcal{J}=\left\{  \left(  i_{1},i_{2},\ldots,i_{\ell
}\right)  \in\left\{  1,2,3,\ldots\right\}  ^{\ell}\ \mid\ i_{1}<i_{2}%
<\cdots<i_{\ell}\right\}  $. In other words, $\mathbf{j}$ has the form
$\left(  i_{1},i_{2},\ldots,i_{\ell}\right)  $ for some $\left(  i_{1}%
,i_{2},\ldots,i_{\ell}\right)  \in\left\{  1,2,3,\ldots\right\}  ^{\ell}$
satisfying $i_{1}<i_{2}<\cdots<i_{\ell}$. Consider this $\left(  i_{1}%
,i_{2},\ldots,i_{\ell}\right)  $. Thus, $\mathbf{j}=\left(  i_{1},i_{2}%
,\ldots,i_{\ell}\right)  $.
\par
We have $i_{e}\in\left\{  1,2,3,\ldots\right\}  $ for every $e\in\left\{
1,2,\ldots,\ell\right\}  $ (since $\left(  i_{1},i_{2},\ldots,i_{\ell}\right)
\in\left\{  1,2,3,\ldots\right\}  ^{\ell}$). In other words, $i_{e}\in\left\{
1,2,3,\ldots\right\}  $ for every $e\in E$ (since $E=\left\{  1,2,\ldots
,\ell\right\}  $). Thus, we can define a map $\phi:E\rightarrow\left\{
1,2,3,\ldots\right\}  $ by $\left(  \phi\left(  e\right)  =i_{e}\text{ for
every }e\in E\right)  $. Consider this map $\phi$.
\par
We have $i_{1}<i_{2}<\cdots<i_{\ell}$. In other words, if $e$ and $f$ are two
elements of $\left\{  1,2,\ldots,\ell\right\}  $ such that $e<f$, then%
\begin{equation}
i_{e}<i_{f}. \label{pf.prop.example.Gamma.b.b.fn3.1}%
\end{equation}
\par
Let $e$ and $f$ be two elements of $E$ satisfying $e<_{1}f$. From $e <_{1} f$,
we obtain $e < f$ (since the relation $<_{1}$ is the restriction of the
standard relation $<$ on $\mathbb{Z}$ to the subset $E$). Thus,
(\ref{pf.prop.example.Gamma.b.b.fn3.1}) shows that $i_{e}<i_{f}$. But the
definition of $\phi$ shows that $\phi\left(  e\right)  =i_{e}$ and
$\phi\left(  f\right)  =i_{f}$. Hence, $\phi\left(  e\right)  =i_{e}%
<i_{f}=\phi\left(  f\right)  $.
\par
Now, forget that we fixed $e$ and $f$. We thus have shown that if $e$ and $f$
are two elements of $E$ satisfying $e<_{1}f$, then $\phi\left(  e\right)
<\phi\left(  f\right)  $. In other words, $\phi$ is a strictly increasing map
from the poset $\left(  E,<_{1}\right)  $ to the totally ordered set $\left\{
1,2,3,\ldots\right\}  $ (by the definition of a \textquotedblleft strictly
increasing map\textquotedblright). In other words, $\phi$ is an $\mathbf{E}%
$-partition (since the $\mathbf{E}$-partitions are precisely the strictly
increasing maps from the poset $\left(  E,<_{1}\right)  $ to the totally
ordered set $\left\{  1,2,3,\ldots\right\}  $). In other words, $\phi\in Z$
(since $Z$ is the set of all $\mathbf{E}$-partitions).
\par
We have $\phi\left(  e\right)  =i_{e}$ for every $e\in E$ (by the definition
of $\phi$). In other words, $\phi\left(  e\right)  =i_{e}$ for every
$e\in\left\{  1,2,\ldots,\ell\right\}  $ (since $E=\left\{  1,2,\ldots
,\ell\right\}  $).
\par
Now, the definition of $\Phi$ yields%
\begin{align*}
\Phi\left(  \phi\right)   &  =\left(  \phi\left(  1\right)  ,\phi\left(
2\right)  ,\ldots,\phi\left(  \ell\right)  \right)  =\left(  i_{1}%
,i_{2},\ldots,i_{\ell}\right) \\
&  \ \ \ \ \ \ \ \ \ \ \left(  \text{since }\phi\left(  e\right)  =i_{e}\text{
for every }e\in\left\{  1,2,\ldots,\ell\right\}  \right) \\
&  =\mathbf{j}.
\end{align*}
Thus, $\mathbf{j}=\Phi\left(  \underbrace{\phi}_{\in Z}\right)  \in\Phi\left(
Z\right)  $.
\par
Now, forget that we fixed $\mathbf{j}$. We thus have shown that $\mathbf{j}%
\in\Phi\left(  Z\right)  $ for every $\mathbf{j}\in\mathcal{J}$. In other
words, $\mathcal{J}\subseteq\Phi\left(  Z\right)  $. In other words, the map
$\Phi$ is surjective, qed.}. In other words, the map $\Phi$ is bijective.
Thus, $\Phi$ is a bijection. In other words, the map%
\begin{equation}
Z\rightarrow\mathcal{J},\ \ \ \ \ \ \ \ \ \ \phi\mapsto\left(  \phi\left(
1\right)  ,\phi\left(  2\right)  ,\ldots,\phi\left(  \ell\right)  \right)
\label{pf.prop.example.Gamma.b.b.themap}%
\end{equation}
is a bijection\footnote{since the map (\ref{pf.prop.example.Gamma.b.b.themap})
is the map $\Phi$ (because $\Phi\left(  \phi\right)  =\left(  \phi\left(
1\right)  ,\phi\left(  2\right)  ,\ldots,\phi\left(  \ell\right)  \right)  $
for every $\phi\in Z$)}.

For every $\pi\in Z$, we have%
\begin{equation}
\mathbf{x}_{\pi,w}=x_{\pi\left(  1\right)  }^{w\left(  1\right)  }%
x_{\pi\left(  2\right)  }^{w\left(  2\right)  }\cdots x_{\pi\left(
\ell\right)  }^{w\left(  \ell\right)  } \label{pf.prop.example.Gamma.b.b.x=}%
\end{equation}
\footnote{\textit{Proof of (\ref{pf.prop.example.Gamma.b.b.x=}):} Let $\pi\in
Z$. Then, the definition of $\mathbf{x}_{\pi,w}$ yields%
\begin{align*}
{\mathbf{x}}_{\pi,w}  &  =\prod_{e\in E}x_{\pi\left(  e\right)  }^{w\left(
e\right)  } =\underbrace{\prod_{e\in\left\{  1,2,\ldots,\ell\right\}  }%
}_{=\prod_{e=1}^{\ell}}x_{\pi\left(  e\right)  }^{w\left(  e\right)
}\ \ \ \ \ \ \ \ \ \ \left(  \text{since }E=\left\{  1,2,\ldots,\ell\right\}
\right) \\
&  =\prod_{e=1}^{\ell}x_{\pi\left(  e\right)  }^{w\left(  e\right)  }%
=x_{\pi\left(  1\right)  }^{w\left(  1\right)  }x_{\pi\left(  2\right)
}^{w\left(  2\right)  }\cdots x_{\pi\left(  \ell\right)  }^{w\left(
\ell\right)  }.
\end{align*}
This proves (\ref{pf.prop.example.Gamma.b.b.x=}).}.

Now, the definition of $\Gamma\left(  \mathbf{E},w\right)  $ yields%
\begin{align*}
\Gamma\left(  {\mathbf{E}},w\right)   &  =\underbrace{\sum_{\pi\text{ is an
}{\mathbf{E}}\text{-partition}}}_{\substack{=\sum_{\pi\in Z}\\\text{(since
}Z\text{ is the set of}\\\text{all }\mathbf{E}\text{-partitions)}}%
}{\mathbf{x}}_{\pi,w}=\sum_{\pi\in Z}\underbrace{{\mathbf{x}}_{\pi,w}%
}_{\substack{=x_{\pi\left(  1\right)  }^{w\left(  1\right)  }x_{\pi\left(
2\right)  }^{w\left(  2\right)  }\cdots x_{\pi\left(  \ell\right)  }^{w\left(
\ell\right)  }\\\text{(by (\ref{pf.prop.example.Gamma.b.b.x=}))}}}=\sum
_{\pi\in Z}x_{\pi\left(  1\right)  }^{w\left(  1\right)  }x_{\pi\left(
2\right)  }^{w\left(  2\right)  }\cdots x_{\pi\left(  \ell\right)  }^{w\left(
\ell\right)  }\\
&  =\underbrace{\sum_{\left(  i_{1},i_{2},\ldots,i_{\ell}\right)
\in\mathcal{J}}}_{=\sum_{i_{1}<i_{2}<\cdots<i_{\ell}}}x_{i_{1}}^{w\left(
1\right)  }x_{i_{2}}^{w\left(  2\right)  }\cdots x_{i_{\ell}}^{w\left(
\ell\right)  }\\
&  \ \ \ \ \ \ \ \ \ \ \left(
\begin{array}
[c]{c}%
\text{here, we have substituted }\left(  i_{1},i_{2},\ldots,i_{\ell}\right)
\text{ for }\left(  \pi\left(  1\right)  ,\pi\left(  2\right)  ,\ldots
,\pi\left(  \ell\right)  \right)  \text{,}\\
\text{since the map }Z\rightarrow\mathcal{J},\ \phi\mapsto\left(  \phi\left(
1\right)  ,\phi\left(  2\right)  ,\ldots,\phi\left(  \ell\right)  \right) \\
\text{is a bijection}%
\end{array}
\right) \\
&  =\sum_{i_{1}<i_{2}<\cdots<i_{\ell}}x_{i_{1}}^{w\left(  1\right)  }x_{i_{2}%
}^{w\left(  2\right)  }\cdots x_{i_{\ell}}^{w\left(  \ell\right)  }=M_{\alpha
}\ \ \ \ \ \ \ \ \ \ \left(  \text{by (\ref{pf.prop.example.Gamma.b.b.Malpha=}%
)}\right)  .
\end{align*}
This completes the proof of Proposition \ref{prop.example.Gamma.b} (b).
\end{proof}

For future reference, let us state a consequence of Proposition
\ref{prop.example.Gamma.b}:

\begin{corollary}
\label{cor.Malpha.Gamma}Let $\alpha$ be a composition. Then, there exist a set
$E$, a special double poset ${\mathbf{E}}=\left(  E,<_{1},>_{1}\right)  $, and
a map $w:E\rightarrow\left\{  1,2,3,\ldots\right\}  $ satisfying
$\Gamma\left(  {\mathbf{E}},w\right)  =M_{\alpha}$.
\end{corollary}

\begin{proof}
[Proof of Corollary \ref{cor.Malpha.Gamma}.]Write the composition $\alpha$ in
the form $\left(  \alpha_{1},\alpha_{2},\ldots,\alpha_{\ell}\right)  $.
Hence, $\left(  \alpha_{1},\alpha_{2},\ldots,\alpha_{\ell}\right)  $ is a
composition (since $\alpha$ is a composition). Therefore, $\alpha_{i}%
\in\left\{  1,2,3,\ldots\right\}  $ for every $i\in\left\{  1,2,\ldots
,\ell\right\}  $.

Define the set $E$, the relations $<_{1}$ and $>_{1}$, and the double poset
${\mathbf{E}} = \left(  E, <_{1}, >_{1}\right)  $ as in Proposition
\ref{prop.example.Gamma.b} (a). Then, Proposition \ref{prop.example.Gamma.b}
(a) shows that $\mathbf{E}$ is a special double poset.

Define a map $w:\left\{  1,2,\ldots,\ell\right\}  \rightarrow\left\{
1,2,3,\ldots\right\}  $ by $\left(  w\left(  i\right)  =\alpha_{i}\text{ for
every }i\in\left\{  1,2,\ldots,\ell\right\}  \right)  $. (This is
well-defined, since $\alpha_{i}\in\left\{  1,2,3,\ldots\right\}  $ for every
$i\in\left\{  1,2,\ldots,\ell\right\}  $). Then, $w$ is a map $\left\{
1,2,\ldots,\ell\right\}  \rightarrow\left\{  1,2,3,\ldots\right\}  $. In other
words, $w$ is a map $E\rightarrow\left\{  1,2,3,\ldots\right\}  $ (since
$E=\left\{  1,2,\ldots,\ell\right\}  $).

Recall that $w\left(  i\right)  =\alpha_{i}$ for every $i\in\left\{
1,2,\ldots,\ell\right\}  $. Thus,
\begin{align*}
\left(  w\left(  1\right)  ,w\left(  2\right)  ,\ldots,w\left(  \ell\right)
\right)   &  =\left(  \alpha_{1},\alpha_{2},\ldots,\alpha_{\ell}\right)
=\alpha.
\end{align*}
In other words, $\alpha=\left(  w\left(  1\right)  ,w\left(  2\right)
,\ldots,w\left(  \ell\right)  \right)  $. Proposition
\ref{prop.example.Gamma.b} (b) thus shows that $\alpha$ is a composition and
satisfies $\Gamma\left(  {\mathbf{E}},w\right)  =M_{\alpha}$.

We thus have constructed a set $E$, a special double poset ${\mathbf{E}%
}=\left(  E,<_{1},>_{1}\right)  $, and a map $w:E\rightarrow\left\{
1,2,3,\ldots\right\}  $ satisfying $\Gamma\left(  {\mathbf{E}},w\right)
=M_{\alpha}$. Hence, there exist a set $E$, a special double poset
${\mathbf{E}}=\left(  E,<_{1},>_{1}\right)  $, and a map $w:E\rightarrow
\left\{  1,2,3,\ldots\right\}  $ satisfying $\Gamma\left(  {\mathbf{E}%
},w\right)  =M_{\alpha}$. This proves Corollary \ref{cor.Malpha.Gamma}.
\end{proof}

\subsection{Disjoint unions of double posets, and the algebra
$\operatorname*{QSym}$}

We shall now study disjoint unions of double posets. As a result of this
study, we will give a new proof of the fact that $\operatorname*{QSym}$ is a
$\mathbf{k}$-algebra.

Let us first recall the classical definition of the disjoint union of several sets:

\begin{definition}
\label{def.djun.sets}Let $I$ be a set. For each $i\in I$, let $E_{i}$ be a
set. Then, we define a set $\bigsqcup_{i\in I}E_{i}$ by%
\[
\bigsqcup_{i\in I}E_{i}=\bigcup_{i\in I}\left(  \left\{  i\right\}  \times
E_{i}\right)  .
\]
(Notice that the sets $\left\{  i\right\}  \times E_{i}$ for distinct $i\in I$
are disjoint.) The set $\bigsqcup_{i\in I}E_{i}$ is called the
\textit{disjoint union} of the sets $E_{i}$. Thus, an element of this disjoint
union $\bigsqcup_{i\in I}E_{i}$ is a pair $\left(  i,e\right)  $, where $i$ is
an element of $I$ and where $e$ is an element of $E_{i}$.

For each $j\in I$, we let $\operatorname*{inc}\nolimits_{j}$ be the map%
\[
E_{j}\rightarrow\bigsqcup_{i\in I}E_{i},\ \ \ \ \ \ \ \ \ \ e\mapsto\left(
j,e\right)  .
\]
This map $\operatorname*{inc}\nolimits_{j}$ is injective, and satisfies
$\operatorname*{inc}\nolimits_{j}\left(  E_{j}\right)  =\left\{  j\right\}
\times E_{j}$. This map $\operatorname*{inc}\nolimits_{j}$ is called the
\textit{canonical inclusion of the }$j$\textit{-th term into the disjoint
union }$\bigsqcup_{i\in I}E_{i}$.

When the sets $E_{i}$ are disjoint, their disjoint union $\bigsqcup_{i\in
I}E_{i}$ is often identified with their union $\bigcup_{i\in I}E_{i}$ via the
bijection%
\begin{equation}
\bigsqcup_{i\in I}E_{i}\rightarrow\bigcup_{i\in I}E_{i}%
,\ \ \ \ \ \ \ \ \ \ \left(  i,e\right)  \mapsto e.
\label{eq.def.djun.sets.identif}%
\end{equation}
We shall not make this identification, however.

When the sets $E_{i}$ are not disjoint, the map
(\ref{eq.def.djun.sets.identif}) is no longer a bijection (but just a
surjection), and thus our above definition of $\bigsqcup_{i\in I}E_{i}$ really
is the simplest way to define a \textquotedblleft disjoint
union\textquotedblright\ of these sets $E_{i}$ (i.e., a big set into which
each $E_{i}$ is canonically embedded in such a way that all the embeddings
have disjoint images).
\end{definition}

We have thus defined the disjoint union of arbitrarily many sets. As a
particular case of this construction, we can define the disjoint union of two sets:

\begin{definition}
\label{def.djun.sets2}Let $E$ and $F$ be two sets. Then, a set $E\sqcup F$ is
defined as follows: Define a family $\left(  E_{i}\right)  _{i\in\left\{
0,1\right\}  }$ of sets by setting $E_{0}=E$ and $E_{1}=F$. Then, define
$E\sqcup F$ to be the set $\bigsqcup_{i\in\left\{  0,1\right\}  }E_{i}$. Thus,
explicitly, we have $E\sqcup F=\left(  \left\{  0\right\}  \times E\right)
\cup\left(  \left\{  1\right\}  \times F\right)  $. Again, this set $E\sqcup
F$ is often identified with $E\cup F$ when $E$ and $F$ are already disjoint;
we shall not make this identification, however.
\end{definition}

Here are some fundamental properties of the disjoint union of sets:

\begin{remark}
\label{rmk.djun.elt}Let $I$ be a set. For each $i\in I$, let $E_{i}$ be a set.

\begin{enumerate}
\item[(a)] The sets $\left\{  i\right\}  \times E_{i}$ for distinct $i\in I$
are disjoint.

\item[(b)] If $j\in I$ and $e\in E_{j}$, then $\left(  j,e\right)
\in\bigsqcup_{i\in I}E_{i}$.

\item[(c)] Let $x\in\bigsqcup_{i\in I}E_{i}$. Then, there exist an $i\in I$
and an $e\in E_{i}$ such that $x=\left(  i,e\right)  $.
\end{enumerate}
\end{remark}

A basic map associated to any disjoint union is its \textit{index map}:

\begin{definition}
\label{def.djun.ind}Let $I$ be a set. For each $i\in I$, let $E_{i}$ be a set.
Define a map $\operatorname*{ind}:\bigsqcup_{i\in I}E_{i}\rightarrow I$ as
follows: Let $x\in\bigsqcup_{i\in I}E_{i}$. Then, there exist an $i\in I$ and
an $e\in E_{i}$ such that $x=\left(  i,e\right)  $ (according to Remark
\ref{rmk.djun.elt} (c)). Consider these $i$ and $e$. Clearly, $\left(
i,e\right)  $ is uniquely determined by $x$ (since $\left(  i,e\right)  =x$).
Hence, both $i$ and $e$ are uniquely determined by $x$ (since $i$ and $e$ are
the two entries of the pair $\left(  i,e\right)  $). Set $\operatorname*{ind}%
\left(  x\right)  =i$. Thus, a map $\operatorname*{ind}:\bigsqcup_{i\in
I}E_{i}\rightarrow I$ is defined.

We call $\operatorname*{ind}$ the \textit{index map} of the family $\left(
E_{i}\right)  _{i\in I}$.
\end{definition}

\begin{remark}
\label{rmk.djun.ind.explicit}Let $I$ be a set. For each $i\in I$, let $E_{i}$
be a set. For each $j\in I$ and $e\in E_{j}$, we have $\operatorname*{ind}%
\left(  j,e\right)  =j$.
\end{remark}

\begin{proposition}
\label{prop.djun.ind.bij}Let $I$ be a set. For each $i\in I$, let $E_{i}$ be a
set. Consider the map $\operatorname*{ind}:\bigsqcup_{i\in I}E_{i}\rightarrow
I$ defined in Definition \ref{def.djun.ind}.

Let $j\in I$. Then, the map%
\[
E_{j}\rightarrow\operatorname*{ind}\nolimits^{-1}\left(  j\right)
,\ \ \ \ \ \ \ \ \ \ e\mapsto\left(  j,e\right)
\]
is well-defined and a bijection.
\end{proposition}

Remark~\ref{rmk.djun.elt}, Remark~\ref{rmk.djun.ind.explicit} and
Proposition~\ref{prop.djun.ind.bij} are standard facts about sets, and their
proofs are straightforward; we will not dwell on them here any longer.

Here is a sample application of Proposition \ref{prop.djun.ind.bij} that we
will actually use:

\begin{proposition}
\label{prop.djun.ind.prod}Let $I$ be a finite set. For each $i\in I$, let
$E_{i}$ be a finite set.

\begin{enumerate}
\item[(a)] The set $\bigsqcup_{i\in I}E_{i}$ is finite.

\item[(b)] Let $A$ be a commutative ring. Let $E = \bigsqcup_{i\in I}E_{i}$.
Let $a:E\rightarrow A$ be a map. Then,%
\[
\prod_{e\in E}a\left(  e\right)  =\prod_{i\in I}\prod_{e\in E_{i}}a\left(
\operatorname*{inc}\nolimits_{i}\left(  e\right)  \right)  .
\]

\end{enumerate}
\end{proposition}

An analogue of Proposition \ref{prop.djun.ind.prod} (b) holds for sums instead
of products, of course.

\begin{proof}
[Proof of Proposition \ref{prop.djun.ind.prod}.](a) For each $i\in I$, the set
$\left\{  i\right\}  \times E_{i}$ is finite (since it is the Cartesian
product of the two finite sets $\left\{  i\right\}  $ and $E_{i}$). Hence,
$\bigcup_{i\in I}\left(  \left\{  i\right\}  \times E_{i}\right)  $ is a
finite union of finite sets (since $I$ is finite), and thus itself a finite
set. In other words, $\bigsqcup_{i\in I}E_{i}$ is a finite set (since
$\bigsqcup_{i\in I}E_{i}=\bigcup_{i\in I}\left(  \left\{  i\right\}  \times
E_{i}\right)  $ (by the definition of $\bigsqcup_{i\in I}E_{i}$)). This proves
Proposition \ref{prop.djun.ind.prod} (a).

(b) Let $i\in I$. Then, Proposition \ref{prop.djun.ind.bij} (applied to $j=i$)
shows that the map%
\[
E_{i}\rightarrow\operatorname*{ind}\nolimits^{-1}\left(  i\right)
,\ \ \ \ \ \ \ \ \ \ e\mapsto\left(  i,e\right)
\]
is well-defined and a bijection. Thus, we can substitute $\left(  i,e\right)
$ for $g$ in the product $\prod_{g\in\operatorname*{ind}\nolimits^{-1}\left(
i\right)  }a\left(  g\right)  $. We thus obtain%
\begin{equation}
\prod_{g\in\operatorname*{ind}\nolimits^{-1}\left(  i\right)  }a\left(
g\right)  =\prod_{e\in E_{i}}a\underbrace{\left(  i,e\right)  }%
_{\substack{=\operatorname*{inc}\nolimits_{i}\left(  e\right)  \\\text{(since
}\operatorname*{inc}\nolimits_{i}\left(  e\right)  =\left(  i,e\right)
\\\text{(by the definition of }\operatorname*{inc}\nolimits_{i}\text{))}%
}}=\prod_{e\in E_{i}}a\left(  \operatorname*{inc}\nolimits_{i}\left(
e\right)  \right)  . \label{pf.prop.djun.ind.prod.b.1}%
\end{equation}
Now, forget that we fixed $i$. We thus have proven
(\ref{pf.prop.djun.ind.prod.b.1}) for every $i\in I$.

We have%
\begin{align*}
\prod_{e\in E}a\left(  e\right)   &  =\prod_{g\in E}a\left(  g\right)
\ \ \ \ \ \ \ \ \ \ \left(  \text{here, we have renamed the index }e\text{ as
}g\text{ in the product}\right) \\
&  =\prod_{i\in I}\underbrace{\prod_{\substack{g\in E;\\\operatorname*{ind}%
g=i}}}_{=\prod_{g\in\operatorname*{ind}\nolimits^{-1}\left(  i\right)  }%
}a\left(  g\right)  \ \ \ \ \ \ \ \ \ \ \left(  \text{since }%
\operatorname*{ind}g\in I\text{ for each }g\in E\right) \\
&  =\prod_{i\in I}\underbrace{\prod_{g\in\operatorname*{ind}\nolimits^{-1}%
\left(  i\right)  }a\left(  g\right)  }_{\substack{=\prod_{e\in E_{i}}a\left(
\operatorname*{inc}\nolimits_{i}\left(  e\right)  \right)  \\\text{(by
(\ref{pf.prop.djun.ind.prod.b.1}))}}}=\prod_{i\in I}\prod_{e\in E_{i}}a\left(
\operatorname*{inc}\nolimits_{i}\left(  e\right)  \right)  .
\end{align*}
This proves Proposition \ref{prop.djun.ind.prod} (b).
\end{proof}

The following fact is the universal property of disjoint unions:

\begin{proposition}
\label{prop.djun.uniprop}Let $I$ be a set. For each $i\in I$, let $E_{i}$ be a
set. Let $F$ be a set. For each $i\in I$, let $f_{i}:E_{i}\rightarrow F$ be a
map. Then, there exists a unique map $f:\bigsqcup_{i\in I}E_{i}\rightarrow F$
such that every $j\in I$ satisfies $f_{j}=f\circ\operatorname*{inc}%
\nolimits_{j}$.
\end{proposition}

\begin{proof}
[Proof of Proposition \ref{prop.djun.uniprop}.]Proposition
\ref{prop.djun.uniprop} is another basic property of sets that we are not
going to prove; let us merely exhibit the unique map $f:\bigsqcup_{i\in
I}E_{i}\rightarrow F$ such that every $j\in I$ satisfies $f_{j}=f\circ
\operatorname*{inc}\nolimits_{j}$. Namely, this is the map
\[
\bigsqcup_{i\in I}E_{i}\rightarrow F,\ \ \ \ \ \ \ \ \ \ \left(  i,e\right)
\mapsto f_{i}\left(  e\right)  .
\]

\end{proof}

\begin{definition}
\label{def.djun.restr}Let $I$ be a set. For each $i\in I$, let $E_{i}$ be a
set. Let $F$ be a set. We define a map%
\[
\operatorname*{Restr}:F^{\bigsqcup_{i\in I}E_{i}}\rightarrow\prod_{i\in
I}F^{E_{i}}%
\]
by%
\[
\left(  \operatorname*{Restr}\left(  f\right)  =\left(  f\circ
\operatorname*{inc}\nolimits_{i}\right)  _{i\in I}%
\ \ \ \ \ \ \ \ \ \ \text{for every }f\in F^{\bigsqcup_{i\in I}E_{i}}\right)
.
\]

\end{definition}

\begin{corollary}
\label{cor.djun.restr}Let $I$ be a set. For each $i\in I$, let $E_{i}$ be a
set. Let $F$ be a set. Then, the map $\operatorname*{Restr}:F^{\bigsqcup_{i\in
I}E_{i}}\rightarrow\prod_{i\in I}F^{E_{i}}$ is a bijection.
\end{corollary}

\begin{proof}
[Proof of Corollary \ref{cor.djun.restr}.]The map $\operatorname*{Restr}%
:F^{\bigsqcup_{i\in I}E_{i}}\rightarrow\prod_{i\in I}F^{E_{i}}$ is
injective\footnote{\textit{Proof.} Let $\alpha$ and $\beta$ be two elements of
$F^{\bigsqcup_{i\in I}E_{i}}$ such that $\operatorname*{Restr}\alpha
=\operatorname*{Restr}\beta$. We shall prove that $\alpha=\beta$.
\par
Both $\alpha$ and $\beta$ are elements of $F^{\bigsqcup_{i\in I}E_{i}}$, hence
are maps $\bigsqcup_{i\in I}E_{i}\rightarrow F$.
\par
We have $\operatorname*{Restr}\alpha=\left(  \alpha\circ\operatorname*{inc}%
\nolimits_{i}\right)  _{i\in I}$ (by the definition of $\operatorname*{Restr}%
$) and $\operatorname*{Restr}\beta=\left(  \beta\circ\operatorname*{inc}%
\nolimits_{i}\right)  _{i\in I}$ (by the definition of $\operatorname*{Restr}%
$). Hence,%
\[
\left(  \alpha\circ\operatorname*{inc}\nolimits_{i}\right)  _{i\in
I}=\operatorname*{Restr}\alpha=\operatorname*{Restr}\beta=\left(  \beta
\circ\operatorname*{inc}\nolimits_{i}\right)  _{i\in I}.
\]
In other words,
\begin{equation}
\alpha\circ\operatorname*{inc}\nolimits_{i}=\beta\circ\operatorname*{inc}%
\nolimits_{i}\ \ \ \ \ \ \ \ \ \ \text{for every }i\in I.
\label{pf.cor.djun.restr.fn1.1}%
\end{equation}
\par
Now, let $x\in\bigsqcup_{i\in I}E_{i}$. We shall prove that $\alpha\left(
x\right)  =\beta\left(  x\right)  $.
\par
There exist an $i\in I$ and an $e\in E_{i}$ such that $x=\left(  i,e\right)  $
(by Remark \ref{rmk.djun.elt} (c)). Consider these $i$ and $e$. The definition
of $\operatorname*{inc}\nolimits_{i}$ shows that $\operatorname*{inc}%
\nolimits_{i}\left(  e\right)  =\left(  i,e\right)  =x$. Now,%
\[
\left(  \alpha\circ\operatorname*{inc}\nolimits_{i}\right)  \left(  e\right)
=\alpha\left(  \underbrace{\operatorname*{inc}\nolimits_{i}\left(  e\right)
}_{=x}\right)  =\alpha\left(  x\right)  ,
\]
so that%
\[
\alpha\left(  x\right)  =\underbrace{\left(  \alpha\circ\operatorname*{inc}%
\nolimits_{i}\right)  }_{\substack{=\beta\circ\operatorname*{inc}%
\nolimits_{i}\\\text{(by (\ref{pf.cor.djun.restr.fn1.1}))}}}\left(  e\right)
=\left(  \beta\circ\operatorname*{inc}\nolimits_{i}\right)  \left(  e\right)
=\beta\left(  \underbrace{\operatorname*{inc}\nolimits_{i}\left(  e\right)
}_{=x}\right)  =\beta\left(  x\right)  .
\]
\par
Now, forget that we fixed $x$. We thus have proven that $\alpha\left(
x\right)  =\beta\left(  x\right)  $ for every $x\in\bigsqcup_{i\in I}E_{i}$.
In other words, $\alpha=\beta$.
\par
Now, forget that we fixed $\alpha$ and $\beta$. We thus have shown that if
$\alpha$ and $\beta$ are two elements of $F^{\bigsqcup_{i\in I}E_{i}}$ such
that $\operatorname*{Restr}\alpha=\operatorname*{Restr}\beta$, then
$\alpha=\beta$. In other words, the map $\operatorname*{Restr}$ is injective.
Qed.} and surjective\footnote{\textit{Proof.} Let $\sigma\in\prod_{i\in
I}F^{E_{i}}$. Thus, $\sigma$ can be written in the form $\left(  f_{i}\right)
_{i\in I}$, where $f_{i}$ is an element of $F^{E_{i}}$ for every $i\in I$.
Consider these $f_{i}$. Thus, $\sigma=\left(  f_{i}\right)  _{i\in I}$.
\par
For each $i\in I$, the element $f_{i}$ is an element of $F^{E_{i}}$, thus a
map from $E_{i}$ to $F$. Thus, Proposition \ref{prop.djun.uniprop} shows that
there exists a unique map $f:\bigsqcup_{i\in I}E_{i}\rightarrow F$ such that
every $j\in I$ satisfies $f_{j}=f\circ\operatorname*{inc}\nolimits_{j}$.
Consider this $f$.
\par
Every $j\in I$ satisfies $f_{j}=f\circ\operatorname*{inc}\nolimits_{j}$.
Renaming $j$ as $i$ in this statement, we obtain the following: Every $i\in I$
satisfies $f_{i}=f\circ\operatorname*{inc}\nolimits_{i}$. In other words,
$\left(  f_{i}\right)  _{i\in I}=\left(  f\circ\operatorname*{inc}%
\nolimits_{i}\right)  _{i\in I}$.
\par
But $f$ is a map $\bigsqcup_{i\in I}E_{i}\rightarrow F$, thus an element of
$F^{\bigsqcup_{i\in I}E_{i}}$. Hence, $\operatorname*{Restr}\left(  f\right)
$ is well-defined. The definition of $\operatorname*{Restr}$ shows that
$\operatorname*{Restr}\left(  f\right)  =\left(  f\circ\operatorname*{inc}%
\nolimits_{i}\right)  _{i\in I}$. Comparing this with $\sigma=\left(
f_{i}\right)  _{i\in I}=\left(  f\circ\operatorname*{inc}\nolimits_{i}\right)
_{i\in I}$, we obtain $\sigma=\operatorname*{Restr}\left(  f\right)
\in\operatorname*{Restr}\left(  F^{\bigsqcup_{i\in I}E_{i}}\right)  $.
\par
Now, forget that we fixed $\sigma$. We thus have shown that $\sigma
\in\operatorname*{Restr}\left(  F^{\bigsqcup_{i\in I}E_{i}}\right)  $ for
every $\sigma\in\prod_{i\in I}F^{E_{i}}$. In other words, $\prod_{i\in
I}F^{E_{i}}\subseteq\operatorname*{Restr}\left(  F^{\bigsqcup_{i\in I}E_{i}%
}\right)  $. In other words, the map $\operatorname*{Restr}$ is surjective.
Qed.}. Thus, this map $\operatorname*{Restr}$ is bijective, i.e., a bijection.
This proves Corollary \ref{cor.djun.restr}.
\end{proof}

Next, let us define the direct sum of relations:

\Needspace{4cm}
\begin{definition}
\label{def.djun.rels}Let $I$ be a set. For each $i\in I$, let $E_{i}$ be a
set. For each $i\in I$, let $\rho_{i}$ be a binary relation on the set $E_{i}%
$. We shall write the relation $\rho_{i}$ in infix notation (i.e., we shall
write $e\rho_{i}f$ instead of $\left(  e,f\right)  \in\rho_{i}$ when we want
to say that two elements $e$ and $f$ of $E_{i}$ are related by $\rho$). We
define a binary relation $\rho$ on the set $\bigsqcup_{i\in I}E_{i}$ (again,
written in infix notation) by the rule%
\[
\left(
\begin{array}
[c]{l}%
\left(  \left(  j,e\right)  \rho\left(  k,f\right)  \right)
\ \Longleftrightarrow\ \left(  j=k\text{ and }e\rho_{j}f\right) \\
\ \ \ \ \ \ \ \ \ \ \text{for any two elements }\left(  j,e\right)  \text{ and
}\left(  k,f\right)  \text{ of }\bigsqcup_{i\in I}E_{i}%
\end{array}
\right)  .
\]
This relation $\rho$ will be denoted by $\bigoplus_{i\in I}\rho_{i}$ and
called the \textit{direct sum} of the relations $\rho_{i}$.
\end{definition}

Notice that we are denoting the relation $\rho$ in Definition
\ref{def.djun.rels} by $\bigoplus_{i\in I}\rho_{i}$ and not by $\bigsqcup
_{i\in I}\rho_{i}$. The (rather pedantic) reason for this is that the
relations $\rho_{i}$ are sets (of pairs of elements of $E_{i}$), and thus the
expression $\bigsqcup_{i\in I}\rho_{i}$ already has a meaning, which is not
the meaning we want to give $\bigoplus_{i\in I}\rho_{i}$.

\begin{proposition}
\label{prop.djun.rels.strict}Let $I$ be a set. For each $i\in I$, let $E_{i}$
be a set. For each $i\in I$, let $\rho_{i}$ be a strict partial order on the
set $E_{i}$. Then, $\bigoplus_{i\in I}\rho_{i}$ is a strict partial order on
the set $\bigsqcup_{i\in I}E_{i}$.
\end{proposition}

\begin{proof}
[Proof of Proposition \ref{prop.djun.rels.strict}.]Let us denote the relation
$\bigoplus_{i\in I}\rho_{i}$ by $\rho$. Thus, $\rho$ is a binary relation on
the set $\bigsqcup_{i\in I}E_{i}$.

We shall write the relations $\rho_{i}$ and also the relations $\bigoplus
_{i\in I}\rho_{i}$ and $\rho$ in infix notation.

The definition of $\bigoplus_{i\in I}\rho_{i}$ shows that we have the
equivalence%
\[
\left(  \left(  j,e\right)  \left(  \bigoplus_{i\in I}\rho_{i}\right)  \left(
k,f\right)  \right)  \ \Longleftrightarrow\ \left(  j=k\text{ and }e\rho
_{j}f\right)
\]
for any two elements $\left(  j,e\right)  $ and $\left(  k,f\right)  $ of
$\bigsqcup_{i\in I}E_{i}$. In other words, we have the equivalence%
\begin{equation}
\left(  \left(  j,e\right)  \rho\left(  k,f\right)  \right)
\ \Longleftrightarrow\ \left(  j=k\text{ and }e\rho_{j}f\right)
\label{pf.prop.djun.rels.strict.rhodef}%
\end{equation}
for any two elements $\left(  j,e\right)  $ and $\left(  k,f\right)  $ of
$\bigsqcup_{i\in I}E_{i}$ (since $\rho=\bigoplus_{i\in I}\rho_{i}$).

The relation $\rho$ is transitive\footnote{\textit{Proof.} Let $u$, $v$ and
$w$ be three elements of $\bigsqcup_{i\in I}E_{i}$ such that $u\rho v$ and
$v\rho w$. We shall prove that $u\rho w$.
\par
Remark \ref{rmk.djun.elt} (c) (applied to $x=u$) shows that there exist an
$i\in I$ and an $e\in E_{i}$ such that $u=\left(  i,e\right)  $. Denote these
$i$ and $e$ by $a$ and $\alpha$. Thus, $a\in I$ and $\alpha\in E_{a}$ are such
that $u=\left(  a,\alpha\right)  $. Thus, $\left(  a,\alpha\right)
=u\in\bigsqcup_{i\in I}E_{i}$.
\par
Remark \ref{rmk.djun.elt} (c) (applied to $x=v$) shows that there exist an
$i\in I$ and an $e\in E_{i}$ such that $v=\left(  i,e\right)  $. Denote these
$i$ and $e$ by $b$ and $\beta$. Thus, $b\in I$ and $\beta\in E_{b}$ are such
that $v=\left(  b,\beta\right)  $. Thus, $\left(  b,\beta\right)
=v\in\bigsqcup_{i\in I}E_{i}$.
\par
Remark \ref{rmk.djun.elt} (c) (applied to $x=w$) shows that there exist an
$i\in I$ and an $e\in E_{i}$ such that $w=\left(  i,e\right)  $. Denote these
$i$ and $e$ by $c$ and $\gamma$. Thus, $c\in I$ and $\gamma\in E_{c}$ are such
that $w=\left(  c,\gamma\right)  $. Thus, $\left(  c,\gamma\right)
=w\in\bigsqcup_{i\in I}E_{i}$.
\par
We have $u\rho v$. This rewrites as $\left(  a,\alpha\right)  \rho\left(
b,\beta\right)  $ (since $u=\left(  a,\alpha\right)  $ and $v=\left(
b,\beta\right)  $). But (\ref{pf.prop.djun.rels.strict.rhodef}) (applied to
$\left(  j,e\right)  =\left(  a,\alpha\right)  $ and $\left(  k,f\right)
=\left(  b,\beta\right)  $) shows that we have the equivalence%
\[
\left(  \left(  a,\alpha\right)  \rho\left(  b,\beta\right)  \right)
\ \Longleftrightarrow\ \left(  a=b\text{ and }\alpha\rho_{a}\beta\right)  .
\]
Thus, we have $\left(  a=b\text{ and }\alpha\rho_{a}\beta\right)  $ (since we
have $\left(  a,\alpha\right)  \rho\left(  b,\beta\right)  $).
\par
We have $v\rho w$. This rewrites as $\left(  b,\beta\right)  \rho\left(
c,\gamma\right)  $ (since $v=\left(  b,\beta\right)  $ and $w=\left(
c,\gamma\right)  $). But (\ref{pf.prop.djun.rels.strict.rhodef}) (applied to
$\left(  j,e\right)  =\left(  b,\beta\right)  $ and $\left(  k,f\right)
=\left(  c,\gamma\right)  $) shows that we have the equivalence%
\[
\left(  \left(  b,\beta\right)  \rho\left(  c,\gamma\right)  \right)
\ \Longleftrightarrow\ \left(  b=c\text{ and }\beta\rho_{b}\gamma\right)  .
\]
Thus, we have $\left(  b=c\text{ and }\beta\rho_{b}\gamma\right)  $ (since we
have $\left(  b,\beta\right)  \rho\left(  c,\gamma\right)  $). This rewrites
as $\left(  a=c\text{ and }\beta\rho_{a}\gamma\right)  $ (since $a=b$).
\par
The relation $\rho_{a}$ is a strict partial order on the set $E_{a}$ (since
$\rho_{i}$ is a strict partial order on the set $E_{i}$ for each $i\in I$),
and thus is transitive. Hence, from $\alpha\rho_{a}\beta$ and $\beta\rho
_{a}\gamma$, we obtain $\alpha\rho_{a}\gamma$. Hence, $\left(  a=c\text{ and
}\alpha\rho_{a}\gamma\right)  $. But (\ref{pf.prop.djun.rels.strict.rhodef})
(applied to $\left(  j,e\right)  =\left(  a,\alpha\right)  $ and $\left(
k,f\right)  =\left(  c,\gamma\right)  $) shows that we have the equivalence%
\[
\left(  \left(  a,\alpha\right)  \rho\left(  c,\gamma\right)  \right)
\ \Longleftrightarrow\ \left(  a=c\text{ and }\alpha\rho_{a}\gamma\right)  .
\]
Thus, we have $\left(  a,\alpha\right)  \rho\left(  c,\gamma\right)  $ (since
we have $\left(  a=c\text{ and }\alpha\rho_{a}\gamma\right)  $). This rewrites
as $u\rho w$ (since $u=\left(  a,\alpha\right)  $ and $w=\left(
c,\gamma\right)  $).
\par
Now, forget that we fixed $u$, $v$ and $w$. We thus have proven that if $u$,
$v$ and $w$ are three elements of $\bigsqcup_{i\in I}E_{i}$ such that $u\rho
v$ and $v\rho w$, then $u\rho w$. In other words, the relation $\rho$ is
transitive, qed.}, irreflexive\footnote{\textit{Proof.} Let $u\in
\bigsqcup_{i\in I}E_{i}$ be such that $u\rho u$. We shall derive a
contradiction.
\par
Remark \ref{rmk.djun.elt} (c) (applied to $x=u$) shows that there exist an
$i\in I$ and an $e\in E_{i}$ such that $u=\left(  i,e\right)  $. Denote these
$i$ and $e$ by $a$ and $\alpha$. Thus, $a\in I$ and $\alpha\in E_{a}$ are such
that $u=\left(  a,\alpha\right)  $. Thus, $\left(  a,\alpha\right)
=u\in\bigsqcup_{i\in I}E_{i}$.
\par
We have $u\rho u$. This rewrites as $\left(  a,\alpha\right)  \rho\left(
a,\alpha\right)  $ (since $u=\left(  a,\alpha\right)  $). Hence,
(\ref{pf.prop.djun.rels.strict.rhodef}) (applied to $\left(  j,e\right)
=\left(  a,\alpha\right)  $ and $\left(  k,f\right)  =\left(  a,\alpha\right)
$) shows that we have the equivalence%
\[
\left(  \left(  a,\alpha\right)  \rho\left(  a,\alpha\right)  \right)
\ \Longleftrightarrow\ \left(  a=a\text{ and }\alpha\rho_{a}\alpha\right)  .
\]
Thus, we have $\left(  a=a\text{ and }\alpha\rho_{a}\alpha\right)  $ (since we
have $\left(  a,\alpha\right)  \rho\left(  a,\alpha\right)  $). Hence,
$\alpha\rho_{a}\alpha$. Thus, there exists an $e\in E_{a}$ such that
$e\rho_{a}e$ (namely, $e=\alpha$).
\par
The relation $\rho_{a}$ is a strict partial order on the set $E_{a}$ (since
$\rho_{i}$ is a strict partial order on the set $E_{i}$ for each $i\in I$),
and thus is irreflexive. Hence, there exists no $e\in E_{a}$ such that
$e\rho_{a}e$. This contradicts the fact that there exists an $e\in E_{a}$ such
that $e\rho_{a}e$.
\par
Now, forget that we fixed $u$. We thus have derived a contradiction for each
$u\in\bigsqcup_{i\in I}E_{i}$ satisfying $u\rho u$. Hence, there exists no
$u\in\bigsqcup_{i\in I}E_{i}$ satisfying $u\rho u$. In other words, the
relation $\rho$ is irreflexive, qed.} and
antisymmetric\footnote{\textit{Proof.} Every binary relation which is
transitive and irreflexive must be antisymmetric (this is well-known).
Applying this to the binary relation $\rho$, we conclude that $\rho$ is
antisymmetric (since $\rho$ is transitive and irreflexive). Qed.}. Thus,
$\rho$ is an irreflexive, transitive and antisymmetric binary relation on the
set $\bigsqcup_{i\in I}E_{i}$. In other words, $\rho$ is a strict partial
order on the set $\bigsqcup_{i\in I}E_{i}$ (because this is how strict partial
orders are defined). In other words, $\bigoplus_{i\in I}\rho_{i}$ is a strict
partial order on the set $\bigsqcup_{i\in I}E_{i}$ (since $\rho=\bigoplus
_{i\in I}\rho_{i}$). This proves Proposition \ref{prop.djun.rels.strict}.
\end{proof}

Now, we can define the disjoint union of double posets:

\begin{proposition}
\label{prop.djun.doubs}Let $I$ be a finite set. For each $i\in I$, let
$\mathbf{E}_{i}=\left(  E_{i},<_{1,i},<_{2,i}\right)  $ be a double poset.
Then, $\left(  \bigsqcup_{i\in I}E_{i},\bigoplus_{i\in I}\left(
<_{1,i}\right)  ,\bigoplus_{i\in I}\left(  <_{2,i}\right)  \right)  $ is a
double poset.
\end{proposition}

\begin{proof}
[Proof of Proposition \ref{prop.djun.doubs}.]For each $i\in I$, we know that
$E_{i}$ is a finite set (since $\left(  E_{i},<_{1,i},<_{2,i}\right)  $ is a
double poset), and that $<_{1,i}$ and $<_{2,i}$ are two strict partial orders
on the set $E_{i}$ (for the same reason).

For each $i\in I$, the relation $<_{1,i}$ is a strict partial order on the set
$E_{i}$. Thus, $\bigoplus_{i\in I}\left(  <_{1,i}\right)  $ is a strict
partial order on the set $\bigsqcup_{i\in I}E_{i}$ (by Proposition
\ref{prop.djun.rels.strict}, applied to $<_{1,i}$ instead of $\rho_{i}$).
Similarly, $\bigoplus_{i\in I}\left(  <_{2,i}\right)  $ is a strict partial
order on the set $\bigsqcup_{i\in I}E_{i}$.

Proposition \ref{prop.djun.ind.prod} (a) shows that the set $\bigsqcup_{i\in
I}E_{i}$ is finite. Thus, $\bigsqcup_{i\in I}E_{i}$ is a finite set, and
$\bigoplus_{i\in I}\left(  <_{1,i}\right)  $ and $\bigoplus_{i\in I}\left(
<_{2,i}\right)  $ are two strict partial orders on the set $\bigsqcup_{i\in
I}E_{i}$. In other words, \newline
$\left(  \bigsqcup_{i\in I}E_{i},\bigoplus_{i\in
I}\left(  <_{1,i}\right)  ,\bigoplus_{i\in I}\left(  <_{2,i}\right)  \right)
$ is a double poset (by the definition of a \textquotedblleft double
poset\textquotedblright). This proves Proposition \ref{prop.djun.doubs}.
\end{proof}

\begin{definition}
\label{def.djun.doubs}Let $I$ be a finite set. For each $i\in I$, let
$\mathbf{E}_{i}=\left(  E_{i},<_{1,i},<_{2,i}\right)  $ be a double poset.
Proposition \ref{prop.djun.doubs} shows that $\left(  \bigsqcup_{i\in I}%
E_{i},\bigoplus_{i\in I}\left(  <_{1,i}\right)  ,\bigoplus_{i\in I}\left(
<_{2,i}\right)  \right)  $ is a double poset. We denote this double poset by
$\bigsqcup_{i\in I}\mathbf{E}_{i}$. We call it the \textit{disjoint union} of
the double posets $\mathbf{E}_{i}$.
\end{definition}

In particular, let us define the disjoint union of two double posets:

\begin{definition}
\label{def.djun.doubs2}Let $\mathbf{E}$ and $\mathbf{F}$ be two double posets.
Then, a double poset $\mathbf{E}\sqcup\mathbf{F}$ is defined as follows:
Define a family $\left(  \mathbf{E}_{i}\right)  _{i\in\left\{  0,1\right\}  }$
of double posets by setting $\mathbf{E}_{0}=\mathbf{E}$ and $\mathbf{E}%
_{1}=\mathbf{F}$. Then, define $\mathbf{E}\sqcup\mathbf{F}$ to be the double
poset $\bigsqcup_{i\in\left\{  0,1\right\}  }\mathbf{E}_{i}$.
\end{definition}

The next proposition characterizes the $\bigsqcup_{i\in I}\mathbf{E}_{i}%
$-partitions when $\left(  \mathbf{E}_{i}\right)  _{i\in I}$ is a finite
family of double posets:

\begin{proposition}
\label{prop.djun.Epars}For every double poset $\mathbf{E}$, let
$\operatorname*{Par}\mathbf{E}$ be the set of all $\mathbf{E}$-partitions.

Let $F=\left\{  1,2,3,\ldots\right\}  $. Thus, $\operatorname*{Par}%
\mathbf{E}\subseteq F^{E}$ for each double poset $\mathbf{E}=\left(
E,<_{1},<_{2}\right)  $.

Let $I$ be a finite set. For each $i\in I$, let $\mathbf{E}_{i}=\left(
E_{i},<_{1,i},<_{2,i}\right)  $ be a double poset. Corollary
\ref{cor.djun.restr} shows that the map $\operatorname*{Restr}:F^{\bigsqcup
_{i\in I}E_{i}}\rightarrow\prod_{i\in I}F^{E_{i}}$ is a bijection. Let
$\phi\in F^{\bigsqcup_{i\in I}E_{i}}$.

\begin{enumerate}
\item[(a)] Then, we have the following logical equivalence:%
\[
\left(  \phi\in\operatorname*{Par}\left(  \bigsqcup_{i\in I}\mathbf{E}%
_{i}\right)  \right)  \ \Longleftrightarrow\ \left(  \operatorname*{Restr}%
\left(  \phi\right)  \in\prod_{i\in I}\operatorname*{Par}\left(
\mathbf{E}_{i}\right)  \right)  .
\]

\item[(b)] Let $w:\bigsqcup_{i\in I}E_{i}\rightarrow\left\{  1,2,3,\ldots
\right\}  $ be a map. Then,%
\[
\mathbf{x}_{\phi,w}=\prod_{i\in I}\mathbf{x}_{\phi\circ\operatorname*{inc}%
\nolimits_{i},w\circ\operatorname*{inc}\nolimits_{i}}.
\]

\end{enumerate}
\end{proposition}

\begin{proof}
[Proof of Proposition \ref{prop.djun.Epars}.]Let $\mathbf{E}$ denote the
double poset $\bigsqcup_{i\in I}\mathbf{E}_{i}$. Thus, \newline
$\mathbf{E}%
=\bigsqcup_{i\in I}\mathbf{E}_{i}=\left(  \bigsqcup_{i\in I}E_{i}%
,\bigoplus_{i\in I}\left(  <_{1,i}\right)  ,\bigoplus_{i\in I}\left(
<_{2,i}\right)  \right)  $ (by the definition of $\bigsqcup_{i\in I}%
\mathbf{E}_{i}$).

Let $E$ be the set $\bigsqcup_{i\in I}E_{i}$. The element $\phi$ is an element
of $F^{\bigsqcup_{i\in I}E_{i}}=F^{E}$ (since $\bigsqcup_{i\in I}E_{i}=E$),
thus a map from $E$ to $F$. In other words, $\phi$ is a map from $E$ to
$\left\{  1,2,3,\ldots\right\}  $ (since $F=\left\{  1,2,3,\ldots\right\}  $).

Let $<_{1}$ denote the binary relation $\bigoplus_{i\in I}\left(
<_{1,i}\right)  $. Thus, $\bigoplus_{i\in I}\left(  <_{1,i}\right)  =\left(
<_{1}\right)  $.

Let $<_{2}$ denote the binary relation $\bigoplus_{i\in I}\left(
<_{2,i}\right)  $. Thus, $\bigoplus_{i\in I}\left(  <_{2,i}\right)  =\left(
<_{2}\right)  $.

(a) Now, $\phi$ is a map $E\rightarrow\left\{  1,2,3,\ldots\right\}  $, and we
have%
\[
\mathbf{E}=\left(  \underbrace{\bigsqcup_{i\in I}E_{i}}_{=E}%
,\underbrace{\bigoplus_{i\in I}\left(  <_{1,i}\right)  }_{=\left(
<_{1}\right)  },\underbrace{\bigoplus_{i\in I}\left(  <_{2,i}\right)
}_{=\left(  <_{2}\right)  }\right)  =\left(  E,<_{1},<_{2}\right)  .
\]
Hence, the definition of an \textquotedblleft$\mathbf{E}$%
-partition\textquotedblright\ shows that $\phi$ is an $\mathbf{E}$-partition
if and only if $\phi$ satisfies the following two conditions:

\textit{Condition $\mathcal{P}$}$_{1}$\textit{:} Every $e\in E$ and $f\in E$
satisfying $e<_{1}f$ satisfy $\phi\left(  e\right)  \leq\phi\left(  f\right)
$.

\textit{Condition $\mathcal{P}$}$_{2}$\textit{:} Every $e\in E$ and $f\in E$
satisfying $e<_{1}f$ and $f<_{2}e$ satisfy $\phi\left(  e\right)  <\phi\left(
f\right)  $.

We have the following chain of logical equivalences:%
\begin{align*}
\ \left(  \phi\in\operatorname*{Par}\mathbf{E}\right)   &  \Longleftrightarrow
\ \left(  \phi\text{ belongs to }\operatorname*{Par}\mathbf{E}\right) \\
&  \Longleftrightarrow\ \left(  \phi\text{ is an }\mathbf{E}\text{-partition}%
\right) \\
&  \ \ \ \ \ \ \ \ \ \ \left(  \text{since }\operatorname*{Par}\mathbf{E}%
\text{ is the set of all }\mathbf{E}\text{-partitions}\right) \\
&  \Longleftrightarrow\ \left(  \phi\text{ satisfies Conditions }%
\mathcal{P}_{1}\text{ and }\mathcal{P}_{2}\right)
\end{align*}
(since $\phi$ is an $\mathbf{E}$-partition if and only if $\phi$ satisfies
Conditions $\mathcal{P}_{1}$ and $\mathcal{P}_{2}$). Thus, we have the
following chain of logical equivalences:%
\begin{align}
&  \ \left(  \phi\in\operatorname*{Par}\mathbf{E}\right) \nonumber\\
&  \Longleftrightarrow\ \left(  \phi\text{ satisfies Conditions }%
\mathcal{P}_{1}\text{ and }\mathcal{P}_{2}\right) \nonumber\\
&  \Longleftrightarrow\ \left(  \text{Conditions }\mathcal{P}_{1}\text{ and
}\mathcal{P}_{2}\text{ hold}\right) \nonumber\\
&  \Longleftrightarrow\ \left(  \text{Condition }\mathcal{P}_{1}\text{
holds}\right)  \wedge\left(  \text{Condition }\mathcal{P}_{2}\text{
holds}\right)  . \label{pf.prop.djun.Epars.eq1}%
\end{align}

On the other hand, let $j\in I$. Then, $\operatorname*{inc}\nolimits_{j}$ is a
map $E_{j}\rightarrow\bigsqcup_{i\in I}E_{i}$. In other words,
$\operatorname*{inc}\nolimits_{j}$ is a map $E_{j}\rightarrow E$ (since
$\bigsqcup_{i\in I}E_{i}=E$). Hence, $\phi\circ\operatorname*{inc}%
\nolimits_{j}$ is a map $E_{j}\rightarrow\left\{  1,2,3,\ldots\right\}  $
(since $\operatorname*{inc}\nolimits_{j}$ is a map $E_{j}\rightarrow E$, and
since $\phi$ is a map $E\rightarrow\left\{  1,2,3,\ldots\right\}  $).

Furthermore, $\mathbf{E}_{i}=\left(  E_{i},<_{1,i},<_{2,i}\right)  $ is a
double poset for each $i\in I$. Applying this to $i=j$, we see that
$\mathbf{E}_{j}=\left(  E_{j},<_{1,j},<_{2,j}\right)  $ is a double poset.

So we know that $\phi\circ\operatorname*{inc}\nolimits_{j}$ is a map
$E_{j}\rightarrow\left\{  1,2,3,\ldots\right\}  $, and we have $\mathbf{E}%
_{j}=\left(  E_{j},<_{1,j},<_{2,j}\right)  $. Hence, the definition of an
\textquotedblleft$\mathbf{E}_{j}$-partition\textquotedblright\ shows that
$\phi\circ\operatorname*{inc}\nolimits_{j}$ is an $\mathbf{E}_{j}$-partition
if and only if $\phi\circ\operatorname*{inc}\nolimits_{j}$ satisfies the
following two conditions:

\textit{Condition $\mathcal{Q}$}$_{1}\left(  j\right)  $\textit{:} Every $e\in
E_{j}$ and $f\in E_{j}$ satisfying $e<_{1,j}f$ satisfy $\left(  \phi
\circ\operatorname*{inc}\nolimits_{j}\right)  \left(  e\right)  \leq\left(
\phi\circ\operatorname*{inc}\nolimits_{j}\right)  \left(  f\right)  $.

\textit{Condition $\mathcal{Q}$}$_{2}\left(  j\right)  $\textit{:} Every $e\in
E_{j}$ and $f\in E_{j}$ satisfying $e<_{1,j}f$ and $f<_{2,j}e$ satisfy
$\left(  \phi\circ\operatorname*{inc}\nolimits_{j}\right)  \left(  e\right)
<\left(  \phi\circ\operatorname*{inc}\nolimits_{j}\right)  \left(  f\right)  $.

We have the following chain of logical equivalences:%
\begin{align}
&  \ \left(  \phi\circ\operatorname*{inc}\nolimits_{j}\in\operatorname*{Par}%
\left(  \mathbf{E}_{j}\right)  \right) \nonumber\\
&  \Longleftrightarrow\ \left(  \phi\circ\operatorname*{inc}\nolimits_{j}%
\text{ belongs to }\operatorname*{Par}\left(  \mathbf{E}_{j}\right)  \right)
\nonumber\\
&  \Longleftrightarrow\ \left(  \phi\circ\operatorname*{inc}\nolimits_{j}%
\text{ is an }\mathbf{E}_{j}\text{-partition}\right) \nonumber\\
&  \ \ \ \ \ \ \ \ \ \ \left(  \text{since }\operatorname*{Par}\left(
\mathbf{E}_{j}\right)  \text{ is the set of all }\mathbf{E}_{j}%
\text{-partitions}\right) \nonumber\\
&  \Longleftrightarrow\ \left(  \phi\circ\operatorname*{inc}\nolimits_{j}%
\text{ satisfies Conditions }\mathcal{Q}_{1}\left(  j\right)  \text{ and
}\mathcal{Q}_{2}\left(  j\right)  \right)  \label{pf.prop.djun.Epars.eq2a}%
\end{align}
(since $\phi\circ\operatorname*{inc}\nolimits_{j}$ is an $\mathbf{E}_j
$-partition if and only if $\phi\circ\operatorname*{inc}\nolimits_{j}$
satisfies Conditions $\mathcal{Q}_{1}\left(  j\right)  $ and $\mathcal{Q}%
_{2}\left(  j\right)  $).

Now, let us forget that we fixed $j$. We thus have introduced two Conditions
$\mathcal{Q}_{1}\left(  j\right)  $ and $\mathcal{Q}_{2}\left(  j\right)  $
for each $j\in I$, and we have proven the logical equivalence
(\ref{pf.prop.djun.Epars.eq2a}) for each $j\in I$.

The definition of $\operatorname*{Restr}$ yields $\operatorname*{Restr}\left(
\phi\right)  =\left(  \phi\circ\operatorname*{inc}\nolimits_{i}\right)  _{i\in
I}$. Now, we have the following chain of logical equivalences:%
\begin{align}
&  \ \left(  \underbrace{\operatorname*{Restr}\left(  \phi\right)  }_{=\left(
\phi\circ\operatorname*{inc}\nolimits_{i}\right)  _{i\in I}}\in\prod_{i\in
I}\operatorname*{Par}\left(  \mathbf{E}_{i}\right)  \right) \nonumber\\
&  \Longleftrightarrow\ \left(  \left(  \phi\circ\operatorname*{inc}%
\nolimits_{i}\right)  _{i\in I}\in\prod_{i\in I}\operatorname*{Par}\left(
\mathbf{E}_{i}\right)  \right) \nonumber\\
&  \Longleftrightarrow\ \left(  \phi\circ\operatorname*{inc}\nolimits_{i}%
\in\operatorname*{Par}\left(  \mathbf{E}_{i}\right)  \text{ for each }i\in
I\right) \nonumber\\
&  \Longleftrightarrow\ \left(  \underbrace{\phi\circ\operatorname*{inc}%
\nolimits_{j}\in\operatorname*{Par}\left(  \mathbf{E}_{j}\right)
}_{\substack{\Longleftrightarrow\ \left(  \phi\circ\operatorname*{inc}%
\nolimits_{j}\text{ satisfies Conditions }\mathcal{Q}_{1}\left(  j\right)
\text{ and }\mathcal{Q}_{2}\left(  j\right)  \right)  \\\text{(by
(\ref{pf.prop.djun.Epars.eq2a}))}}}\text{ for each }j\in I\right) \nonumber\\
&  \ \ \ \ \ \ \ \ \ \ \left(  \text{here, we have renamed the index }i\text{
as }j\right) \nonumber\\
&  \Longleftrightarrow\ \left(  \underbrace{\phi\circ\operatorname*{inc}%
\nolimits_{j}\text{ satisfies Conditions }\mathcal{Q}_{1}\left(  j\right)
\text{ and }\mathcal{Q}_{2}\left(  j\right)  }_{\Longleftrightarrow\ \left(
\text{Conditions }\mathcal{Q}_{1}\left(  j\right)  \text{ and }\mathcal{Q}%
_{2}\left(  j\right)  \text{ hold}\right)  }\text{ for each }j\in I\right)
\nonumber\\
&  \Longleftrightarrow\ \left(  \text{Conditions }\mathcal{Q}_{1}\left(
j\right)  \text{ and }\mathcal{Q}_{2}\left(  j\right)  \text{ hold for each
}j\in I\right) \nonumber\\
&  \Longleftrightarrow\ \left(  \text{Condition }\mathcal{Q}_{1}\left(
j\right)  \text{ holds for each }j\in I\right) \nonumber\\
&  \ \ \ \ \ \ \ \ \ \ \wedge\left(  \text{Condition }\mathcal{Q}_{2}\left(
j\right)  \text{ holds for each }j\in I\right)  .
\label{pf.prop.djun.Epars.eq2b}%
\end{align}

But for any two elements $\left(  j,e\right)  $ and $\left(  k,f\right)  $ of
$E$, we have the equivalence%
\begin{equation}
\left(  \left(  j,e\right)  <_{1}\left(  k,f\right)  \right)
\ \Longleftrightarrow\ \left(  j=k\text{ and }e<_{1,j}f\right)
\label{pf.prop.djun.Epars.le1}%
\end{equation}
\footnote{\textit{Proof of (\ref{pf.prop.djun.Epars.le1}):} Recall that
$\bigoplus_{i\in I}\left(  <_{1,i}\right)  =\left(  <_{1}\right)  $.
\par
Now, the definition of the relation $\bigoplus_{i\in I}\left(  <_{1,i}\right)
$ shows that%
\[
\left(
\begin{array}
[c]{l}%
\left(  \left(  j,e\right)  \left(  \bigoplus_{i\in I}\left(  <_{1,i}\right)
\right)  \left(  k,f\right)  \right)  \ \Longleftrightarrow\ \left(  j=k\text{
and }e<_{1,j}f\right) \\
\ \ \ \ \ \ \ \ \ \ \text{for any two elements }\left(  j,e\right)  \text{ and
}\left(  k,f\right)  \text{ of }\bigsqcup_{i\in I}E_{i}%
\end{array}
\right)  .
\]
In other words,%
\[
\left(
\begin{array}
[c]{l}%
\left(  \left(  j,e\right)  <_{1}\left(  k,f\right)  \right)
\ \Longleftrightarrow\ \left(  j=k\text{ and }e<_{1,j}f\right) \\
\ \ \ \ \ \ \ \ \ \ \text{for any two elements }\left(  j,e\right)  \text{ and
}\left(  k,f\right)  \text{ of }E
\end{array}
\right)
\]
(since $\bigoplus_{i\in I}\left(  <_{1,i}\right)  =\left(  <_{1}\right)  $ and
$\bigsqcup_{i\in I}E_{i}=E$). This proves (\ref{pf.prop.djun.Epars.le1}).}.
Similarly, for any two elements $\left(  j,e\right)  $ and $\left(
k,f\right)  $ of $E$, we have the equivalence%
\begin{equation}
\left(  \left(  j,e\right)  <_{2}\left(  k,f\right)  \right)
\ \Longleftrightarrow\ \left(  j=k\text{ and }e<_{2,j}f\right)  .
\label{pf.prop.djun.Epars.le2}%
\end{equation}

We observe that every $j\in I$ and every $e\in E_{j}$ satisfy%
\begin{equation}
\phi\left(  j,e\right)  =\left(  \phi\circ\operatorname*{inc}\nolimits_{j}%
\right)  \left(  e\right)  \label{pf.prop.djun.Epars.phije}%
\end{equation}
\footnote{\textit{Proof of (\ref{pf.prop.djun.Epars.phije}):} Let $j\in I$ and
$e\in E_{j}$. The definition of $\operatorname*{inc}\nolimits_{j}$ yields
$\operatorname*{inc}\nolimits_{j}\left(  e\right)  =\left(  j,e\right)  $.
Now, $\left(  \phi\circ\operatorname*{inc}\nolimits_{j}\right)  \left(
e\right)  =\phi\left(  \underbrace{\operatorname*{inc}\nolimits_{j}\left(
e\right)  }_{=\left(  j,e\right)  }\right)  =\phi\left(  j,e\right)  $. This
proves (\ref{pf.prop.djun.Epars.phije}).}.

Now, let us prove the implication%
\begin{equation}
\left(  \text{Condition }\mathcal{P}_{1}\text{ holds}\right)  \Longrightarrow
\left(  \text{Condition }\mathcal{Q}_{1}\left(  j\right)  \text{ holds for
each }j\in I\right)  . \label{pf.prop.djun.Epars.dir1a}%
\end{equation}

[\textit{Proof of (\ref{pf.prop.djun.Epars.dir1a}):} Assume that Condition
$\mathcal{P}_{1}$ holds. We must prove that Condition $\mathcal{Q}_{1}\left(
j\right)  $ holds for each $j\in I$.

Indeed, fix $j\in I$. Let $e\in E_{j}$ and $f\in E_{j}$ be such that
$e<_{1,j}f$. We shall show that $\left(  \phi\circ\operatorname*{inc}%
\nolimits_{j}\right)  \left(  e\right)  \leq\left(  \phi\circ
\operatorname*{inc}\nolimits_{j}\right)  \left(  f\right)  $.

The definition of $\operatorname*{inc}\nolimits_{j}$ yields
$\operatorname*{inc}\nolimits_{j}\left(  e\right)  =\left(  j,e\right)  $ and
$\operatorname*{inc}\nolimits_{j}\left(  f\right)  =\left(  j,f\right)  $.

Recall that $\operatorname*{inc}\nolimits_{j}$ is a map $E_{j}\rightarrow
\bigsqcup_{i\in I}E_{i}$. In other words, $\operatorname*{inc}\nolimits_{j}$
is a map $E_{j}\rightarrow E$ (since $\bigsqcup_{i\in I}E_{i}=E$).

We have $\left(  j,e\right)  =\operatorname*{inc}\nolimits_{j}\left(
e\right)  \in E$ (since $\operatorname*{inc}\nolimits_{j}$ is a map
$E_{j}\rightarrow E$). Also, $\left(  j,f\right)  =\operatorname*{inc}%
\nolimits_{j}\left(  f\right)  \in E$ (since $\operatorname*{inc}%
\nolimits_{j}$ is a map $E_{j}\rightarrow E$). Thus,
(\ref{pf.prop.djun.Epars.le1}) (applied to $j$ instead of $k$) shows that we
have the equivalence%
\[
\left(  \left(  j,e\right)  <_{1}\left(  j,f\right)  \right)
\ \Longleftrightarrow\ \left(  j=j\text{ and }e<_{1,j}f\right)  .
\]
Thus, we have $\left(  j,e\right)  <_{1}\left(  j,f\right)  $ (since we have
$\left(  j=j\text{ and }e<_{1,j}f\right)  $).

Now, recall that Condition $\mathcal{P}_{1}$ holds. Hence, Condition
$\mathcal{P}_{1}$ (applied to $\left(  j,e\right)  $ and $\left(  j,f\right)
$ instead of $e$ and $f$) shows that $\phi\left(  j,e\right)  \leq\phi\left(
j,f\right)  $ (since $\left(  j,e\right)  <_{1}\left(  j,f\right)  $). But
(\ref{pf.prop.djun.Epars.phije}) yields $\phi\left(  j,e\right)  =\left(
\phi\circ\operatorname*{inc}\nolimits_{j}\right)  \left(  e\right)  $. Also,
(\ref{pf.prop.djun.Epars.phije}) (applied to $f$ instead of $e$) yields
$\phi\left(  j,f\right)  =\left(  \phi\circ\operatorname*{inc}\nolimits_{j}%
\right)  \left(  f\right)  $. Now,%
\[
\left(  \phi\circ\operatorname*{inc}\nolimits_{j}\right)  \left(  e\right)
=\phi\left(  j,e\right)  \leq\phi\left(  j,f\right)  =\left(  \phi
\circ\operatorname*{inc}\nolimits_{j}\right)  \left(  f\right)  .
\]

Now, forget that we fixed $e$ and $f$. We thus have shown that every $e\in
E_{j}$ and $f\in E_{j}$ satisfying $e<_{1,j}f$ satisfy $\left(  \phi
\circ\operatorname*{inc}\nolimits_{j}\right)  \left(  e\right)  \leq\left(
\phi\circ\operatorname*{inc}\nolimits_{j}\right)  \left(  f\right)  $. In
other words, Condition $\mathcal{Q}_{1}\left(  j\right)  $ holds.

Now, forget that we fixed $j$. We thus have shown that Condition
$\mathcal{Q}_{1}\left(  j\right)  $ holds for each $j\in I$. This completes
the proof of the implication (\ref{pf.prop.djun.Epars.dir1a}).]

Next, let us prove the implication%
\begin{equation}
\left(  \text{Condition }\mathcal{Q}_{1}\left(  j\right)  \text{ holds for
each }j\in I\right)  \Longrightarrow\left(  \text{Condition }\mathcal{P}%
_{1}\text{ holds}\right)  . \label{pf.prop.djun.Epars.dir1b}%
\end{equation}

[\textit{Proof of (\ref{pf.prop.djun.Epars.dir1b}):} Assume that Condition
$\mathcal{Q}_{1}\left(  j\right)  $ holds for each $j\in I$. We must prove
that Condition $\mathcal{P}_{1}$ holds.

We have assumed that%
\begin{equation}
\text{Condition }\mathcal{Q}_{1}\left(  j\right)  \text{ holds for each }j\in
I\text{.} \label{pf.prop.djun.Epars.dir1b.pf.1}%
\end{equation}

Now, let $u\in E$ and $v\in E$ be such that $u<_{1}v$. We shall show that
$\phi\left(  u\right)  \leq\phi\left(  v\right)  $.

We have $v\in E=\bigsqcup_{i\in I}E_{i}$. Thus, Remark \ref{rmk.djun.elt} (c)
(applied to $x=v$) shows that there exist an $i\in I$ and an $e\in E_{i}$ such
that $v=\left(  i,e\right)  $. Denote these $i$ and $e$ by $k$ and $f$. Thus,
$k\in I$ and $f\in E_{k}$ are such that $v=\left(  k,f\right)  $.

We have $u\in E=\bigsqcup_{i\in I}E_{i}$. Thus, Remark \ref{rmk.djun.elt} (c)
(applied to $x=u$) shows that there exist an $i\in I$ and an $e\in E_{i}$ such
that $u=\left(  i,e\right)  $. Denote these $i$ and $e$ by $j$ and $e$. Thus,
$j\in I$ and $e\in E_{j}$ are such that $u=\left(  j,e\right)  $.

Now, $u<_{1}v$. In other words, $\left(  j,e\right)  <_{1}\left(  k,f\right)
$ (since $u=\left(  j,e\right)  $ and $v=\left(  k,f\right)  $).

We have $\left(  j,e\right)  =u\in E$ and $\left(  k,f\right)  =v\in E$. Thus,
(\ref{pf.prop.djun.Epars.le1}) shows that we have the equivalence%
\[
\left(  \left(  j,e\right)  <_{1}\left(  k,f\right)  \right)
\ \Longleftrightarrow\ \left(  j=k\text{ and }e<_{1,j}f\right)  .
\]
Thus, we have $\left(  j=k\text{ and }e<_{1,j}f\right)  $ (since we have
$\left(  j,e\right)  <_{1}\left(  k,f\right)  $). Hence, $j=k$ and $e<_{1,j}%
f$. Now, $f\in E_{k}=E_{j}$ (since $k=j$).

But (\ref{pf.prop.djun.Epars.dir1b.pf.1}) shows that Condition $\mathcal{Q}%
_{1}\left(  j\right)  $ holds. Thus, this condition shows that $\left(
\phi\circ\operatorname*{inc}\nolimits_{j}\right)  \left(  e\right)
\leq\left(  \phi\circ\operatorname*{inc}\nolimits_{j}\right)  \left(
f\right)  $ (since $e<_{1,j}f$). But (\ref{pf.prop.djun.Epars.phije}) yields
$\phi\left(  j,e\right)  =\left(  \phi\circ\operatorname*{inc}\nolimits_{j}%
\right)  \left(  e\right)  $. Also, (\ref{pf.prop.djun.Epars.phije}) (applied
to $f$ instead of $e$) yields $\phi\left(  j,f\right)  =\left(  \phi
\circ\operatorname*{inc}\nolimits_{j}\right)  \left(  f\right)  $. Now,%
\begin{align*}
\phi\left(  \underbrace{u}_{=\left(  j,e\right)  }\right)   &  =\phi\left(
j,e\right)  =\left(  \phi\circ\operatorname*{inc}\nolimits_{j}\right)  \left(
e\right)  \leq\left(  \phi\circ\operatorname*{inc}\nolimits_{j}\right)
\left(  f\right)  =\phi\left(  \underbrace{j}_{=k},f\right) \\
&  =\phi\underbrace{\left(  k,f\right)  }_{=v}=\phi\left(  v\right)  .
\end{align*}

Now, forget that we fixed $u$ and $v$. We thus have shown that every $u\in E$
and $v\in E$ satisfying $u<_{1}v$ satisfy $\phi\left(  u\right)  \leq
\phi\left(  v\right)  $. Renaming $u$ and $v$ as $e$ and $f$ in this
statement, we obtain the following: Every $e\in E$ and $f\in E$ satisfying
$e<_{1}f$ satisfy $\phi\left(  e\right)  \leq\phi\left(  f\right)  $. In other
words, Condition $\mathcal{P}_{1}$ holds. This proves the implication
(\ref{pf.prop.djun.Epars.dir1b}).]

Combining the two implications (\ref{pf.prop.djun.Epars.dir1a}) and
(\ref{pf.prop.djun.Epars.dir1b}), we obtain the logical equivalence%
\begin{equation}
\left(  \text{Condition }\mathcal{P}_{1}\text{ holds}\right)
\Longleftrightarrow\left(  \text{Condition }\mathcal{Q}_{1}\left(  j\right)
\text{ holds for each }j\in I\right)  . \label{pf.prop.djun.Epars.dir1}%
\end{equation}

Next, we shall prove the implication%
\begin{equation}
\left(  \text{Condition }\mathcal{P}_{2}\text{ holds}\right)  \Longrightarrow
\left(  \text{Condition }\mathcal{Q}_{2}\left(  j\right)  \text{ holds for
each }j\in I\right)  . \label{pf.prop.djun.Epars.dir2a}%
\end{equation}

[\textit{Proof of (\ref{pf.prop.djun.Epars.dir2a}):} Assume that Condition
$\mathcal{P}_{2}$ holds. We must prove that Condition $\mathcal{Q}_{2}\left(
j\right)  $ holds for each $j\in I$.

Indeed, fix $j\in I$. Let $e\in E_{j}$ and $f\in E_{j}$ be such that
$e<_{1,j}f$ and $f<_{2,j}e$. We shall show that $\left(  \phi\circ
\operatorname*{inc}\nolimits_{j}\right)  \left(  e\right)  <\left(  \phi
\circ\operatorname*{inc}\nolimits_{j}\right)  \left(  f\right)  $.

Just as in the proof of (\ref{pf.prop.djun.Epars.dir1a}), we can

\begin{itemize}
\item show that $\left(  j,e\right)  \in E$ and $\left(  j,f\right)  \in E$
and $\left(  j,e\right)  <_{1}\left(  j,f\right)  $;

\item show that $\phi\left(  j,e\right)  =\left(  \phi\circ\operatorname*{inc}%
\nolimits_{j}\right)  \left(  e\right)  $ and $\phi\left(  j,f\right)
=\left(  \phi\circ\operatorname*{inc}\nolimits_{j}\right)  \left(  f\right)  $.
\end{itemize}

Furthermore, recall that $\left(  j,f\right)  \in E$ and $\left(  j,e\right)
\in E$. Thus, (\ref{pf.prop.djun.Epars.le2}) (applied to $j$, $f$, $j$ and $e$
instead of $j$, $e$, $k$ and $f$) shows that we have the equivalence%
\[
\left(  \left(  j,f\right)  <_{2}\left(  j,e\right)  \right)
\ \Longleftrightarrow\ \left(  j=j\text{ and }f<_{2,j}e\right)  .
\]
Thus, we have $\left(  j,f\right)  <_{2}\left(  j,e\right)  $ (since we have
$\left(  j=j\text{ and }f<_{2,j}e\right)  $).

Now, recall that Condition $\mathcal{P}_{2}$ holds. Hence, Condition
$\mathcal{P}_{2}$ (applied to $\left(  j,e\right)  $ and $\left(  j,f\right)
$ instead of $e$ and $f$) shows that $\phi\left(  j,e\right)  <\phi\left(
j,f\right)  $ (since $\left(  j,e\right)  <_{1}\left(  j,f\right)  $ and
$\left(  j,f\right)  <_{2}\left(  j,e\right)  $). Now,%
\[
\left(  \phi\circ\operatorname*{inc}\nolimits_{j}\right)  \left(  e\right)
=\phi\left(  j,e\right)  <\phi\left(  j,f\right)  =\left(  \phi\circ
\operatorname*{inc}\nolimits_{j}\right)  \left(  f\right)  .
\]

Now, forget that we fixed $e$ and $f$. We thus have shown that every $e\in
E_{j}$ and $f\in E_{j}$ satisfying $e<_{1,j}f$ and $f<_{2,j}e$ satisfy
$\left(  \phi\circ\operatorname*{inc}\nolimits_{j}\right)  \left(  e\right)
<\left(  \phi\circ\operatorname*{inc}\nolimits_{j}\right)  \left(  f\right)
$. In other words, Condition $\mathcal{Q}_{2}\left(  j\right)  $ holds.

Now, forget that we fixed $j$. We thus have shown that Condition
$\mathcal{Q}_{2}\left(  j\right)  $ holds for each $j\in I$. This proves the
implication (\ref{pf.prop.djun.Epars.dir2a}).]

Next, let us prove the implication%
\begin{equation}
\left(  \text{Condition }\mathcal{Q}_{2}\left(  j\right)  \text{ holds for
each }j\in I\right)  \Longrightarrow\left(  \text{Condition }\mathcal{P}%
_{2}\text{ holds}\right)  . \label{pf.prop.djun.Epars.dir2b}%
\end{equation}

[\textit{Proof of (\ref{pf.prop.djun.Epars.dir2b}):} Assume that Condition
$\mathcal{Q}_{2}\left(  j\right)  $ holds for each $j\in I$. We must prove
that Condition $\mathcal{P}_{2}$ holds.

We have assumed that%
\begin{equation}
\text{Condition }\mathcal{Q}_{2}\left(  j\right)  \text{ holds for each }j\in
I\text{.} \label{pf.prop.djun.Epars.dir2b.pf.1}%
\end{equation}

Now, let $u\in E$ and $v\in E$ be such that $u<_{1}v$ and $v<_{2}u$. We shall
show that $\phi\left(  u\right)  <\phi\left(  v\right)  $.

Just as in the proof of (\ref{pf.prop.djun.Epars.dir1b}), we can

\begin{itemize}
\item construct $k\in I$ and $f\in E_{k}$ such that $v=\left(  k,f\right)  $;

\item construct $j\in I$ and $e\in E_{j}$ such that $u=\left(  j,e\right)  $;

\item show that $\left(  j,e\right)  =u\in E$ and $\left(  k,f\right)  =v\in
E$;

\item show that $j=k$ and $e<_{1,j}f$ and $f\in E_{j}$;

\item show that $\phi\left(  j,e\right)  =\left(  \phi\circ\operatorname*{inc}%
\nolimits_{j}\right)  \left(  e\right)  $ and $\phi\left(  j,f\right)
=\left(  \phi\circ\operatorname*{inc}\nolimits_{j}\right)  \left(  f\right)  $.
\end{itemize}

On the other hand, $v<_{2}u$. In other words, $\left(  k,f\right)
<_{2}\left(  j,e\right)  $ (since $u=\left(  j,e\right)  $ and $v=\left(
k,f\right)  $).

We have $\left(  k,f\right)  \in E$ and $\left(  j,e\right)  \in E$. Thus,
(\ref{pf.prop.djun.Epars.le2}) (applied to $k$, $f$, $j$ and $e$ instead of
$j$, $e$, $k$ and $f$) shows that we have the equivalence%
\[
\left(  \left(  k,f\right)  <_{2}\left(  j,e\right)  \right)
\ \Longleftrightarrow\ \left(  k=j\text{ and }f<_{2,k}e\right)  .
\]
Thus, we have $\left(  k=j\text{ and }f<_{2,k}e\right)  $ (since we have
$\left(  k,f\right)  <_{2}\left(  j,e\right)  $). Hence, $f<_{2,k}e$. In other
words, $f<_{2,j}e$ (since $k=j$).

But (\ref{pf.prop.djun.Epars.dir2b.pf.1}) shows that Condition $\mathcal{Q}%
_{2}\left(  j\right)  $ holds. Thus, this condition shows that $\left(
\phi\circ\operatorname*{inc}\nolimits_{j}\right)  \left(  e\right)  <\left(
\phi\circ\operatorname*{inc}\nolimits_{j}\right)  \left(  f\right)  $ (since
$e<_{1,j}f$ and $f<_{2,j}e$). Now,%
\begin{align*}
\phi\left(  \underbrace{u}_{=\left(  j,e\right)  }\right)   &  =\phi\left(
j,e\right)  =\left(  \phi\circ\operatorname*{inc}\nolimits_{j}\right)  \left(
e\right)  <\left(  \phi\circ\operatorname*{inc}\nolimits_{j}\right)  \left(
f\right)  =\phi\left(  \underbrace{j}_{=k},f\right) \\
&  =\phi\underbrace{\left(  k,f\right)  }_{=v}=\phi\left(  v\right)  .
\end{align*}

Now, forget that we fixed $u$ and $v$. We thus have shown that every $u\in E$
and $v\in E$ satisfying $u<_{1}v$ and $v<_{2}u$ satisfy $\phi\left(  u\right)
<\phi\left(  v\right)  $. Renaming $u$ and $v$ as $e$ and $f$ in this
statement, we obtain the following: Every $e\in E$ and $f\in E$ satisfying
$e<_{1}f$ and $f<_{2}e$ satisfy $\phi\left(  e\right)  <\phi\left(  f\right)
$. In other words, Condition $\mathcal{P}_{2}$ holds. This proves the
implication (\ref{pf.prop.djun.Epars.dir2b}).]

Combining the two implications (\ref{pf.prop.djun.Epars.dir2a}) and
(\ref{pf.prop.djun.Epars.dir2b}), we obtain the logical equivalence%
\begin{equation}
\left(  \text{Condition }\mathcal{P}_{2}\text{ holds}\right)
\Longleftrightarrow\left(  \text{Condition }\mathcal{Q}_{2}\left(  j\right)
\text{ holds for each }j\in I\right)  . \label{pf.prop.djun.Epars.dir2}%
\end{equation}

Now, we have the following chain of logical equivalences:%
\begin{align*}
&  \ \left(  \phi\in\operatorname*{Par}\underbrace{\left(  \bigsqcup_{i\in
I}\mathbf{E}_{i}\right)  }_{=\mathbf{E}}\right) \\
&  \Longleftrightarrow\ \left(  \phi\in\operatorname*{Par}\mathbf{E}\right) \\
&  \Longleftrightarrow\ \underbrace{\left(  \text{Condition }\mathcal{P}%
_{1}\text{ holds}\right)  }_{\substack{\Longleftrightarrow\left(
\text{Condition }\mathcal{Q}_{1}\left(  j\right)  \text{ holds for each }j\in
I\right)  \\\text{(by (\ref{pf.prop.djun.Epars.dir1}))}}}\wedge
\underbrace{\left(  \text{Condition }\mathcal{P}_{2}\text{ holds}\right)
}_{\substack{\Longleftrightarrow\left(  \text{Condition }\mathcal{Q}%
_{2}\left(  j\right)  \text{ holds for each }j\in I\right)  \\\text{(by
(\ref{pf.prop.djun.Epars.dir2}))}}}\\
&  \ \ \ \ \ \ \ \ \ \ \left(  \text{by (\ref{pf.prop.djun.Epars.eq1})}\right)
\\
&  \Longleftrightarrow\ \left(  \text{Condition }\mathcal{Q}_{1}\left(
j\right)  \text{ holds for each }j\in I\right) \\
&  \ \ \ \ \ \ \ \ \ \ \wedge\left(  \text{Condition }\mathcal{Q}_{2}\left(
j\right)  \text{ holds for each }j\in I\right) \\
&  \Longleftrightarrow\ \left(  \operatorname*{Restr}\left(  \phi\right)
\in\prod_{i\in I}\operatorname*{Par}\left(  \mathbf{E}_{i}\right)  \right)
\ \ \ \ \ \ \ \ \ \ \left(  \text{by (\ref{pf.prop.djun.Epars.eq2b})}\right)
.
\end{align*}
This proves Proposition \ref{prop.djun.Epars} (a).

(b) We know that $w$ is a map $\bigsqcup_{i\in I}E_{i}\rightarrow\left\{
1,2,3,\ldots\right\}  $. In other words, $w$ is a map $E\rightarrow\left\{
1,2,3,\ldots\right\}  $ (since $E=\bigsqcup_{i\in I}E_{i}$).

Both $w$ and $\phi$ are maps $E\rightarrow\left\{  1,2,3,\ldots\right\}  $.
Thus, $w\left(  e\right)  $ and $\phi\left(  e\right)  $ are well-defined
elements of $\left\{  1,2,3,\ldots\right\}  $ for each $e\in E$.

Let $A$ be the commutative ring $\mathbf{k}\left[  \left[  x_{1},x_{2}%
,x_{3},\ldots\right]  \right]  $. Define a map $a:E\rightarrow A$ by%
\begin{equation}
\left(  a\left(  e\right)  =x_{\phi\left(  e\right)  }^{w\left(  e\right)
}\ \ \ \ \ \ \ \ \ \ \text{for every }e\in E\right)  .
\label{pf.prop.djun.Epars.b.defa}%
\end{equation}

The definition of $\mathbf{x}_{\phi,w}$ yields%
\begin{align*}
\mathbf{x}_{\phi,w}  &  =\prod_{e\in E}\underbrace{x_{\phi\left(  e\right)
}^{w\left(  e\right)  }}_{\substack{=a\left(  e\right)  \\\text{(by
(\ref{pf.prop.djun.Epars.b.defa}))}}}=\prod_{e\in E}a\left(  e\right) \\
&  =\prod_{i\in I}\prod_{e\in E_{i}}\underbrace{a\left(  \operatorname*{inc}%
\nolimits_{i}\left(  e\right)  \right)  }_{\substack{=x_{\phi\left(
\operatorname*{inc}\nolimits_{i}\left(  e\right)  \right)  }^{w\left(
\operatorname*{inc}\nolimits_{i}\left(  e\right)  \right)  }\\\text{(by the
definition}\\\text{of the map }a\text{)}}}\ \ \ \ \ \ \ \ \ \ \left(  \text{by
Proposition \ref{prop.djun.ind.prod} (b)}\right) \\
&  =\prod_{i\in I}\prod_{e\in E_{i}}\underbrace{x_{\phi\left(
\operatorname*{inc}\nolimits_{i}\left(  e\right)  \right)  }^{w\left(
\operatorname*{inc}\nolimits_{i}\left(  e\right)  \right)  }}_{=x_{\left(
\phi\circ\operatorname*{inc}\nolimits_{i}\right)  \left(  e\right)  }^{\left(
w\circ\operatorname*{inc}\nolimits_{i}\right)  \left(  e\right)  }}%
=\prod_{i\in I}\prod_{e\in E_{i}}x_{\left(  \phi\circ\operatorname*{inc}%
\nolimits_{i}\right)  \left(  e\right)  }^{\left(  w\circ\operatorname*{inc}%
\nolimits_{i}\right)  \left(  e\right)  }.
\end{align*}
Comparing this with%
\[
\prod_{i\in I}\underbrace{\mathbf{x}_{\phi\circ\operatorname*{inc}%
\nolimits_{i},w\circ\operatorname*{inc}\nolimits_{i}}}_{\substack{=\prod_{e\in
E_{i}}x_{\left(  \phi\circ\operatorname*{inc}\nolimits_{i}\right)  \left(
e\right)  }^{\left(  w\circ\operatorname*{inc}\nolimits_{i}\right)  \left(
e\right)  }\\\text{(by the definition of }\mathbf{x}_{\phi\circ
\operatorname*{inc}\nolimits_{i},w\circ\operatorname*{inc}\nolimits_{i}%
}\text{)}}}=\prod_{i\in I}\prod_{e\in E_{i}}x_{\left(  \phi\circ
\operatorname*{inc}\nolimits_{i}\right)  \left(  e\right)  }^{\left(
w\circ\operatorname*{inc}\nolimits_{i}\right)  \left(  e\right)  },
\]
we obtain%
\[
\mathbf{x}_{\phi,w}=\prod_{i\in I}\mathbf{x}_{\phi\circ\operatorname*{inc}%
\nolimits_{i},w\circ\operatorname*{inc}\nolimits_{i}}.
\]
This proves Proposition \ref{prop.djun.Epars} (b).
\end{proof}

Now, we can prove the following fact:

\begin{proposition}
\label{prop.djun.Gamma}Let $I$ be a finite set. For each $i\in I$, let
$\mathbf{E}_{i}=\left(  E_{i},<_{1,i},<_{2,i}\right)  $ be a double poset.
Every map $w:\bigsqcup_{i\in I}E_{i}\rightarrow\left\{  1,2,3,\ldots\right\}
$ satisfies%
\[
\Gamma\left(  \bigsqcup_{i\in I}\mathbf{E}_{i},w\right)  =\prod_{i\in I}%
\Gamma\left(  \mathbf{E}_{i},w\circ\operatorname*{inc}\nolimits_{i}\right)  .
\]

\end{proposition}

\begin{proof}
[Proof of Proposition \ref{prop.djun.Gamma}.]For every double poset
$\mathbf{E}$, let $\operatorname*{Par}\mathbf{E}$ be the set of all
$\mathbf{E}$-partitions. For every double poset $\mathbf{E}=\left(
E,<_{1},<_{2}\right)  $ and every map $w:E\rightarrow\left\{  1,2,3,\ldots
\right\}  $, we have%
\begin{align}
\Gamma\left(  \mathbf{E},w\right)   &  =\underbrace{\sum_{\pi\text{ is an
}\mathbf{E}\text{-partition}}}_{\substack{=\sum_{\pi\in\operatorname*{Par}%
\mathbf{E}}\\\text{(since }\operatorname*{Par}\mathbf{E}\text{ is
the}\\\text{set of all }\mathbf{E}\text{-partitions)}}}\mathbf{x}_{\pi
,w}\ \ \ \ \ \ \ \ \ \ \left(  \text{by the definition of }\Gamma\left(
\mathbf{E},w\right)  \right) \nonumber\\
&  =\sum_{\pi\in\operatorname*{Par}\mathbf{E}}\mathbf{x}_{\pi,w}.
\label{pf.prop.djun.Gamma.1}%
\end{align}

Let $F=\left\{  1,2,3,\ldots\right\}  $. Thus, $\operatorname*{Par}%
\mathbf{E}\subseteq F^{E}$ for each double poset $\mathbf{E}=\left(
E,<_{1},<_{2}\right)  $. In particular, $\operatorname*{Par}\left(
\bigsqcup_{i\in I}\mathbf{E}_{i}\right)  \subseteq F^{\bigsqcup_{i\in I}E_{i}%
}$, and every $i\in I$ satisfies $\operatorname*{Par}\left(  \mathbf{E}%
_{i}\right)  \subseteq F^{E_{i}}$.

Let $w:\bigsqcup_{i\in I}E_{i}\rightarrow\left\{  1,2,3,\ldots\right\}  $ be
any map.

Corollary \ref{cor.djun.restr} shows that the map $\operatorname*{Restr}%
:F^{\bigsqcup_{i\in I}E_{i}}\rightarrow\prod_{i\in I}F^{E_{i}}$ is a
bijection. Every $\pi\in\operatorname*{Par}\left(  \bigsqcup_{i\in
I}\mathbf{E}_{i}\right)  $ satisfies $\operatorname*{Restr}\left(  \pi\right)
\in\prod_{i\in I}\operatorname*{Par}\left(  \mathbf{E}_{i}\right)
$\ \ \ \ \footnote{\textit{Proof.} Let $\pi\in\operatorname*{Par}\left(
\bigsqcup_{i\in I}\mathbf{E}_{i}\right)  $. Then, $\pi\in\operatorname*{Par}%
\left(  \bigsqcup_{i\in I}\mathbf{E}_{i}\right)  \subseteq F^{\bigsqcup_{i\in
I}E_{i}}$. Proposition \ref{prop.djun.Epars} (a) (applied to $\phi=\pi$) thus
shows that we have the following logical equivalence:%
\[
\left(  \pi\in\operatorname*{Par}\left(  \bigsqcup_{i\in I}\mathbf{E}%
_{i}\right)  \right)  \ \Longleftrightarrow\ \left(  \operatorname*{Restr}%
\left(  \pi\right)  \in\prod_{i\in I}\operatorname*{Par}\left(  \mathbf{E}%
_{i}\right)  \right)  .
\]
Thus, we have $\operatorname*{Restr}\left(  \pi\right)  \in\prod_{i\in
I}\operatorname*{Par}\left(  \mathbf{E}_{i}\right)  $ (since we have $\pi
\in\operatorname*{Par}\left(  \bigsqcup_{i\in I}\mathbf{E}_{i}\right)  $).
Qed.}. Hence, we can define a map $\rho:\operatorname*{Par}\left(
\bigsqcup_{i\in I}\mathbf{E}_{i}\right)  \rightarrow\prod_{i\in I}%
\operatorname*{Par}\left(  \mathbf{E}_{i}\right)  $ by $\left(  \rho\left(
\pi\right)  =\operatorname*{Restr}\left(  \pi\right)  \text{ for every }\pi
\in\operatorname*{Par}\left(  \bigsqcup_{i\in I}\mathbf{E}_{i}\right)
\right)  $. Consider this map $\rho$.

Every $\pi\in\operatorname*{Par}\left(  \bigsqcup_{i\in I}\mathbf{E}%
_{i}\right)  $ satisfies%
\begin{equation}
\rho\left(  \pi\right)  =\operatorname*{Restr}\left(  \pi\right)  =\left(
\pi\circ\operatorname*{inc}\nolimits_{i}\right)  _{i\in I}
\label{pf.prop.djun.Gamma.rho}%
\end{equation}
(by the definition of $\operatorname*{Restr}$). Thus, $\rho$ is the map%
\[
\operatorname*{Par}\left(  \bigsqcup_{i\in I}\mathbf{E}_{i}\right)
\rightarrow\prod_{i\in I}\operatorname*{Par}\left(  \mathbf{E}_{i}\right)
,\ \ \ \ \ \ \ \ \ \ \pi\mapsto\left(  \pi\circ\operatorname*{inc}%
\nolimits_{i}\right)  _{i\in I}.
\]

The map $\rho:\operatorname*{Par}\left(  \bigsqcup_{i\in I}\mathbf{E}%
_{i}\right)  \rightarrow\prod_{i\in I}\operatorname*{Par}\left(
\mathbf{E}_{i}\right)  $ is injective\footnote{\textit{Proof.} Let $\pi_{1}$
and $\pi_{2}$ be two elements of $\operatorname*{Par}\left(  \bigsqcup_{i\in
I}\mathbf{E}_{i}\right)  $ such that $\rho\left(  \pi_{1}\right)  =\rho\left(
\pi_{2}\right)  $. We must show that $\pi_{1}=\pi_{2}$.
\par
The map $\operatorname*{Restr}$ is a bijection, thus injective.
\par
The definition of $\rho$ yields $\rho\left(  \pi_{1}\right)
=\operatorname*{Restr}\left(  \pi_{1}\right)  $ and $\rho\left(  \pi
_{2}\right)  =\operatorname*{Restr}\left(  \pi_{2}\right)  $. Thus,
$\operatorname*{Restr}\left(  \pi_{1}\right)  =\rho\left(  \pi_{1}\right)
=\rho\left(  \pi_{2}\right)  =\operatorname*{Restr}\left(  \pi_{2}\right)  $.
Hence, $\pi_{1}=\pi_{2}$ (since the map $\operatorname*{Restr}$ is injective).
\par
Now, forget that we fixed $\pi_{1}$ and $\pi_{2}$. We thus have shown that if
$\pi_{1}$ and $\pi_{2}$ are two elements of $\operatorname*{Par}\left(
\bigsqcup_{i\in I}\mathbf{E}_{i}\right)  $ such that $\rho\left(  \pi
_{1}\right)  =\rho\left(  \pi_{2}\right)  $, then $\pi_{1}=\pi_{2}$. In other
words, the map $\rho$ is injective. Qed.} and
surjective\footnote{\textit{Proof.} Let $\gamma\in\prod_{i\in I}%
\operatorname*{Par}\left(  \mathbf{E}_{i}\right)  $. We shall prove that
$\gamma\in\rho\left(  \operatorname*{Par}\left(  \bigsqcup_{i\in I}%
\mathbf{E}_{i}\right)  \right)  $.
\par
The map $\operatorname*{Restr}:F^{\bigsqcup_{i\in I}E_{i}}\rightarrow
\prod_{i\in I}F^{E_{i}}$ is a bijection, thus surjective. Hence, there exists
a $\phi\in F^{\bigsqcup_{i\in I}E_{i}}$ such that $\gamma
=\operatorname*{Restr}\left(  \phi\right)  $ (since $\gamma\in\prod_{i\in
I}\underbrace{\operatorname*{Par}\left(  \mathbf{E}_{i}\right)  }_{\subseteq
F^{E_{i}}}\subseteq\prod_{i\in I}F^{E_{i}}$). Consider this $\phi$.
\par
Proposition \ref{prop.djun.Epars} (a) shows that we have the following logical
equivalence:%
\[
\left(  \phi\in\operatorname*{Par}\left(  \bigsqcup_{i\in I}\mathbf{E}%
_{i}\right)  \right)  \ \Longleftrightarrow\ \left(  \operatorname*{Restr}%
\left(  \phi\right)  \in\prod_{i\in I}\operatorname*{Par}\left(
\mathbf{E}_{i}\right)  \right)  .
\]
Thus, we have $\phi\in\operatorname*{Par}\left(  \bigsqcup_{i\in I}%
\mathbf{E}_{i}\right)  $ (since we have $\operatorname*{Restr}\left(
\phi\right)  =\gamma\in\prod_{i\in I}\operatorname*{Par}\left(  \mathbf{E}%
_{i}\right)  $).
\par
The definition of $\rho$ yields $\rho\left(  \phi\right)
=\operatorname*{Restr}\left(  \phi\right)  =\gamma$, so that $\gamma
=\rho\left(  \underbrace{\phi}_{\in\operatorname*{Par}\left(  \bigsqcup_{i\in
I}\mathbf{E}_{i}\right)  }\right)  \in\rho\left(  \operatorname*{Par}\left(
\bigsqcup_{i\in I}\mathbf{E}_{i}\right)  \right)  $.
\par
Now, forget that we fixed $\gamma$. We thus have shown that $\gamma\in
\rho\left(  \operatorname*{Par}\left(  \bigsqcup_{i\in I}\mathbf{E}%
_{i}\right)  \right)  $ for each $\gamma\in\prod_{i\in I}\operatorname*{Par}%
\left(  \mathbf{E}_{i}\right)  $. In other words, $\prod_{i\in I}%
\operatorname*{Par}\left(  \mathbf{E}_{i}\right)  \subseteq\rho\left(
\operatorname*{Par}\left(  \bigsqcup_{i\in I}\mathbf{E}_{i}\right)  \right)
$. In other words, the map $\rho$ is surjective. Qed.}. Hence, this map $\rho$
is bijective, i.e., a bijection. In other words, the map%
\[
\operatorname*{Par}\left(  \bigsqcup_{i\in I}\mathbf{E}_{i}\right)
\rightarrow\prod_{i\in I}\operatorname*{Par}\left(  \mathbf{E}_{i}\right)
,\ \ \ \ \ \ \ \ \ \ \pi\mapsto\left(  \pi\circ\operatorname*{inc}%
\nolimits_{i}\right)  _{i\in I}%
\]
is a bijection\footnote{since $\rho$ is the map%
\[
\operatorname*{Par}\left(  \bigsqcup_{i\in I}\mathbf{E}_{i}\right)
\rightarrow\prod_{i\in I}\operatorname*{Par}\left(  \mathbf{E}_{i}\right)
,\ \ \ \ \ \ \ \ \ \ \pi\mapsto\left(  \pi\circ\operatorname*{inc}%
\nolimits_{i}\right)  _{i\in I}%
\]
}.

Every $\pi\in\operatorname*{Par}\left(  \bigsqcup_{i\in I}\mathbf{E}%
_{i}\right)  $ satisfies%
\begin{equation}
\mathbf{x}_{\pi,w}=\prod_{i\in I}\mathbf{x}_{\pi\circ\operatorname*{inc}%
\nolimits_{i},w\circ\operatorname*{inc}\nolimits_{i}}
\label{pf.prop.djun.Gamma.5}%
\end{equation}
\footnote{\textit{Proof of (\ref{pf.prop.djun.Gamma.5}):} Let $\pi
\in\operatorname*{Par}\left(  \bigsqcup_{i\in I}\mathbf{E}_{i}\right)  $.
Then, $\pi\in\operatorname*{Par}\left(  \bigsqcup_{i\in I}\mathbf{E}%
_{i}\right)  \subseteq F^{\bigsqcup_{i\in I}E_{i}}$. Hence, Proposition
\ref{prop.djun.Epars} (b) (applied to $\phi=\pi$) yields%
\[
\mathbf{x}_{\pi,w}=\prod_{i\in I}\mathbf{x}_{\pi\circ\operatorname*{inc}%
\nolimits_{i},w\circ\operatorname*{inc}\nolimits_{i}}.
\]
}.

Now, recall that $\bigsqcup_{i\in I}\mathbf{E}_{i}=\left(  \bigsqcup_{i\in
I}E_{i},\bigoplus_{i\in I}\left(  <_{1,i}\right)  ,\bigoplus_{i\in I}\left(
<_{2,i}\right)  \right)  $ (by the definition of $\bigsqcup_{i\in I}%
\mathbf{E}_{i}$). Hence, (\ref{pf.prop.djun.Gamma.1}) (applied to
$\bigsqcup_{i\in I}\mathbf{E}_{i}$, $\bigsqcup_{i\in I}E_{i}$, $\bigoplus
_{i\in I}\left(  <_{1,i}\right)  $ and $\bigoplus_{i\in I}\left(
<_{2,i}\right)  $ instead of $\mathbf{E}$, $E$, $<_{1}$ and $<_{2}$) shows
that%
\begin{align*}
\Gamma\left(  \bigsqcup_{i\in I}\mathbf{E}_{i},w\right)   &  =\sum_{\pi
\in\operatorname*{Par}\left(  \bigsqcup_{i\in I}\mathbf{E}_{i}\right)
}\underbrace{\mathbf{x}_{\pi,w}}_{\substack{=\prod_{i\in I}\mathbf{x}%
_{\pi\circ\operatorname*{inc}\nolimits_{i},w\circ\operatorname*{inc}%
\nolimits_{i}}\\\text{(by (\ref{pf.prop.djun.Gamma.5}))}}}=\sum_{\pi
\in\operatorname*{Par}\left(  \bigsqcup_{i\in I}\mathbf{E}_{i}\right)  }%
\prod_{i\in I}\mathbf{x}_{\pi\circ\operatorname*{inc}\nolimits_{i}%
,w\circ\operatorname*{inc}\nolimits_{i}}\\
&  =\sum_{\left(  \pi_{i}\right)  _{i\in I}\in\prod_{i\in I}%
\operatorname*{Par}\left(  \mathbf{E}_{i}\right)  }\prod_{i\in I}%
\mathbf{x}_{\pi_{i},w\circ\operatorname*{inc}\nolimits_{i}}%
\end{align*}
(here, we have substituted $\left(  \pi_{i}\right)  _{i\in I}$ for $\left(
\pi\circ\operatorname*{inc}\nolimits_{i}\right)  _{i\in I}$ in the sum, since
the map $\operatorname*{Par}\left(  \bigsqcup_{i\in I}\mathbf{E}_{i}\right)
\rightarrow\prod_{i\in I}\operatorname*{Par}\left(  \mathbf{E}_{i}\right)
,\ \pi\mapsto\left(  \pi\circ\operatorname*{inc}\nolimits_{i}\right)  _{i\in
I}$ is a bijection). Comparing this with%
\begin{align*}
\prod_{i\in I}\underbrace{\Gamma\left(  \mathbf{E}_{i},w\circ
\operatorname*{inc}\nolimits_{i}\right)  }_{\substack{=\sum_{\pi
\in\operatorname*{Par}\left(  \mathbf{E}_{i}\right)  }\mathbf{x}_{\pi
,w\circ\operatorname*{inc}\nolimits_{i}}\\\text{(by
(\ref{pf.prop.djun.Gamma.1}), applied to}\\\mathbf{E}_{i}\text{, }E_{i}\text{,
}<_{1,i}\text{, }<_{2,i}\text{ and }w\circ\operatorname*{inc}\nolimits_{i}%
\\\text{instead of }\mathbf{E}\text{, }E\text{, }<_{1}\text{, }<_{2}\text{ and
}w\text{)}}}  &  =\prod_{i\in I}\sum_{\pi\in\operatorname*{Par}\left(
\mathbf{E}_{i}\right)  }\mathbf{x}_{\pi,w\circ\operatorname*{inc}%
\nolimits_{i}}=\sum_{\left(  \pi_{i}\right)  _{i\in I}\in\prod_{i\in
I}\operatorname*{Par}\left(  \mathbf{E}_{i}\right)  }\prod_{i\in I}%
\mathbf{x}_{\pi_{i},w\circ\operatorname*{inc}\nolimits_{i}}\\
&  \ \ \ \ \ \ \ \ \ \ \left(  \text{by the product rule}\right)  ,
\end{align*}
we obtain $\Gamma\left(  \bigsqcup_{i\in I}\mathbf{E}_{i},w\right)
=\prod_{i\in I}\Gamma\left(  \mathbf{E}_{i},w\circ\operatorname*{inc}%
\nolimits_{i}\right)  $. This proves Proposition \ref{prop.djun.Gamma}.
\end{proof}

\begin{corollary}
\label{cor.djun.Gamma.2}Let $I$ be a finite set. For each $i\in I$, let
$\mathbf{E}_{i}=\left(  E_{i},<_{1,i},<_{2,i}\right)  $ be a double poset. For
each $i\in I$, let $w_{i}:E_{i}\rightarrow\left\{  1,2,3,\ldots\right\}  $ be
a map.

Let $F=\left\{  1,2,3,\ldots\right\}  $. For each $i\in I$, we have $w_{i}%
\in\left\{  1,2,3,\ldots\right\}  ^{E_{i}}=F^{E_{i}}$ (since $\left\{
1,2,3,\ldots\right\}  =F$). Thus, $\left(  w_{i}\right)  _{i\in I}\in
\prod_{i\in I}F^{E_{i}}$.

Corollary \ref{cor.djun.restr} shows that the map $\operatorname*{Restr}%
:F^{\bigsqcup_{i\in I}E_{i}}\rightarrow\prod_{i\in I}F^{E_{i}}$ is a
bijection. Hence, $\operatorname*{Restr}\nolimits^{-1}\left(  \left(
w_{i}\right)  _{i\in I}\right)  $ is a well-defined element of $F^{\bigsqcup
_{i\in I}E_{i}}$. Denote this element by $w$. Then,%
\[
\prod_{i\in I}\Gamma\left(  \mathbf{E}_{i},w_{i}\right)  =\Gamma\left(
\bigsqcup_{i\in I}\mathbf{E}_{i},w\right)  .
\]

\end{corollary}

\begin{proof}
[Proof of Corollary \ref{cor.djun.Gamma.2}.]We have $w=\operatorname*{Restr}%
\nolimits^{-1}\left(  \left(  w_{i}\right)  _{i\in I}\right)  $ (by the
definition of $w$) and thus%
\[
\left(  w_{i}\right)  _{i\in I}=\operatorname*{Restr}\left(  w\right)
=\left(  w\circ\operatorname*{inc}\nolimits_{i}\right)  _{i\in I}%
\ \ \ \ \ \ \ \ \ \ \left(  \text{by the definition of }\operatorname*{Restr}%
\right)  .
\]
In other words,%
\begin{equation}
w_{i}=w\circ\operatorname*{inc}\nolimits_{i}\ \ \ \ \ \ \ \ \ \ \text{for each
}i\in I. \label{pf.cor.djun.Gamma.2}%
\end{equation}

We have $w\in F^{\bigsqcup_{i\in I}E_{i}}$. Thus, $w$ is a map $\bigsqcup
_{i\in I}E_{i}\rightarrow F$. In other words, $w$ is a map $\bigsqcup_{i\in
I}E_{i}\rightarrow\left\{  1,2,3,\ldots\right\}  $ (since $F=\left\{
1,2,3,\ldots\right\}  $). Hence, Proposition \ref{prop.djun.Gamma} yields%
\[
\Gamma\left(  \bigsqcup_{i\in I}\mathbf{E}_{i},w\right)  =\prod_{i\in I}%
\Gamma\left(  \mathbf{E}_{i},\underbrace{w\circ\operatorname*{inc}%
\nolimits_{i}}_{\substack{=w_{i}\\\text{(by (\ref{pf.cor.djun.Gamma.2}))}%
}}\right)  =\prod_{i\in I}\Gamma\left(  \mathbf{E}_{i},w_{i}\right)  .
\]
This proves Corollary \ref{cor.djun.Gamma.2}.
\end{proof}

As a consequence of Corollary \ref{cor.djun.Gamma.2} in the case when
$I=\left\{  0,1\right\}  $, we obtain the following:

\begin{corollary}
\label{cor.djun.Gamma.EF}Let $\mathbf{E}=\left(  E,<_{1,0},<_{2,0}\right)  $
and $\mathbf{F}=\left(  F,<_{1,1},<_{2,1}\right)  $ be two double posets.
Recall that a double poset $\mathbf{E}\sqcup\mathbf{F}$ is defined (in
Definition \ref{def.djun.doubs2}).

Define a family $\left(  E_{i}\right)  _{i\in\left\{  0,1\right\}  }$ of sets
by setting $E_{0}=E$ and $E_{1}=F$. Recall that $E\sqcup F=\bigsqcup
_{i\in\left\{  0,1\right\}  }E_{i}$ (by the definition of $E\sqcup F$).

Recall that there is a map $\operatorname*{inc}\nolimits_{j}:E_{j}%
\rightarrow\bigsqcup_{i\in\left\{  0,1\right\}  }E_{i}$ defined for each
$j\in\left\{  0,1\right\}  $. Thus, we have two maps $\operatorname*{inc}%
\nolimits_{0}:E_{0}\rightarrow\bigsqcup_{i\in\left\{  0,1\right\}  }E_{i}$ and
$\operatorname*{inc}\nolimits_{1}:E_{1}\rightarrow\bigsqcup_{i\in\left\{
0,1\right\}  }E_{i}$. In other words, we have two maps $\operatorname*{inc}%
\nolimits_{0}:E\rightarrow E\sqcup F$ and $\operatorname*{inc}\nolimits_{1}%
:F\rightarrow E\sqcup F$ (since $E_{0}=E$, $E_{1}=F$ and $E\sqcup
F=\bigsqcup_{i\in\left\{  0,1\right\}  }E_{i}$). (Explicitly, these maps are
given as follows: The map $\operatorname*{inc}\nolimits_{0}$ sends each $e\in
E$ to $\left(  0,e\right)  \in E\sqcup F$; the map $\operatorname*{inc}%
\nolimits_{1}$ sends each $f\in F$ to $\left(  1,f\right)  \in E\sqcup F$.)

Every map $w:E\sqcup F\rightarrow\left\{  1,2,3,\ldots\right\}  $ satisfies%
\[
\Gamma\left(  \mathbf{E}\sqcup\mathbf{F},w\right)  =\Gamma\left(
\mathbf{E},w\circ\operatorname*{inc}\nolimits_{0}\right)  \Gamma\left(
\mathbf{F},w\circ\operatorname*{inc}\nolimits_{1}\right)  .
\]

\end{corollary}

Corollary \ref{cor.djun.Gamma.EF} is the \textquotedblleft rule for
multiplying quasisymmetric functions of the form $\Gamma\left(  {\mathbf{E}%
},w\right)  $\textquotedblright\ mentioned at the end of Section
\ref{sect.lemmas} (but here we are denoting by $\mathbf{E}\sqcup\mathbf{F}$
what had been called $\mathbf{EF}$ back there).

\begin{proof}
[Proof of Corollary \ref{cor.djun.Gamma.EF}.]Define a family $\left(
\mathbf{E}_{i}\right)  _{i\in\left\{  0,1\right\}  }$ of double posets by
setting $\mathbf{E}_{0}=\mathbf{E}$ and $\mathbf{E}_{1}=\mathbf{F}$. Then,
$\mathbf{E}\sqcup\mathbf{F}=\bigsqcup_{i\in\left\{  0,1\right\}  }%
\mathbf{E}_{i}$ (by the definition of $\mathbf{E}\sqcup\mathbf{F}$).

We have $\mathbf{E}_{i}=\left(  E_{i},<_{1,i},<_{2,i}\right)  $ for every
$i\in\left\{  0,1\right\}  $\ \ \ \ \footnote{\textit{Proof:}
\par
\begin{itemize}
\item For $i=0$, this follows from $\mathbf{E}_{0}=\mathbf{E}=\left(
\underbrace{E}_{=E_{0}},<_{1,0},<_{2,0}\right)  =\left(  E_{0},<_{1,0}%
,<_{2,0}\right)  $.
\par
\item For $i=1$, this follows from $\mathbf{E}_{1}=\mathbf{F}=\left(
\underbrace{F}_{=E_{1}},<_{1,1},<_{2,1}\right)  =\left(  E_{1},<_{1,1}%
,<_{2,1}\right)  $.
\end{itemize}
}. Also, $w$ is a map $E\sqcup F\rightarrow\left\{  1,2,3,\ldots\right\}  $.
In other words, $w$ is a map $\bigsqcup_{i\in\left\{  0,1\right\}  }%
E_{i}\rightarrow\left\{  1,2,3,\ldots\right\}  $ (since $E\sqcup
F=\bigsqcup_{i\in\left\{  0,1\right\}  }E_{i}$). Thus, Proposition
\ref{prop.djun.Gamma} (applied to $I=\left\{  0,1\right\}  $) yields%
\begin{align*}
\Gamma\left(  \bigsqcup_{i\in\left\{  0,1\right\}  }\mathbf{E}_{i},w\right)
&  =\prod_{i\in\left\{  0,1\right\}  }\Gamma\left(  \mathbf{E}_{i}%
,w\circ\operatorname*{inc}\nolimits_{i}\right)  =\Gamma\left(
\underbrace{\mathbf{E}_{0}}_{=\mathbf{E}},w\circ\operatorname*{inc}%
\nolimits_{0}\right)  \Gamma\left(  \underbrace{\mathbf{E}_{1}}_{=\mathbf{F}%
},w\circ\operatorname*{inc}\nolimits_{1}\right) \\
&  =\Gamma\left(  \mathbf{E},w\circ\operatorname*{inc}\nolimits_{0}\right)
\Gamma\left(  \mathbf{F},w\circ\operatorname*{inc}\nolimits_{1}\right)  .
\end{align*}
Thus,%
\[
\Gamma\left(  \underbrace{\mathbf{E}\sqcup\mathbf{F}}_{=\bigsqcup
_{i\in\left\{  0,1\right\}  }\mathbf{E}_{i}},w\right)  =\Gamma\left(
\bigsqcup_{i\in\left\{  0,1\right\}  }\mathbf{E}_{i},w\right)  =\Gamma\left(
\mathbf{E},w\circ\operatorname*{inc}\nolimits_{0}\right)  \Gamma\left(
\mathbf{F},w\circ\operatorname*{inc}\nolimits_{1}\right)  .
\]
This proves Corollary \ref{cor.djun.Gamma.EF}.
\end{proof}

Here is a further corollary of Corollary \ref{cor.djun.Gamma.2}:

\begin{corollary}
\label{cor.djun.Gamma.2.cor}Let $I$ be a finite set. For each $i\in I$, let
$\mathbf{E}_{i}=\left(  E_{i},<_{1,i},<_{2,i}\right)  $ be a double poset. For
each $i\in I$, let $w_{i}:E_{i}\rightarrow\left\{  1,2,3,\ldots\right\}  $ be
a map. Then,%
\[
\prod_{i\in I}\Gamma\left(  \mathbf{E}_{i},w_{i}\right)  \in
\operatorname*{QSym}.
\]

\end{corollary}

\begin{proof}
[Proof of Corollary \ref{cor.djun.Gamma.2.cor}.]Let $F=\left\{  1,2,3,\ldots
\right\}  $. For each $i\in I$, we have $w_{i}\in\left\{  1,2,3,\ldots
\right\}  ^{E_{i}}=F^{E_{i}}$ (since $\left\{  1,2,3,\ldots\right\}  =F$).
Thus, $\left(  w_{i}\right)  _{i\in I}\in\prod_{i\in I}F^{E_{i}}$.

Corollary \ref{cor.djun.restr} shows that the map $\operatorname*{Restr}%
:F^{\bigsqcup_{i\in I}E_{i}}\rightarrow\prod_{i\in I}F^{E_{i}}$ is a
bijection. Hence, $\operatorname*{Restr}\nolimits^{-1}\left(  \left(
w_{i}\right)  _{i\in I}\right)  $ is a well-defined element of $F^{\bigsqcup
_{i\in I}E_{i}}$. Denote this element by $w$. Thus, $w$ is an element of
$F^{\bigsqcup_{i\in I}E_{i}}$. In other words, $w$ is a map $\bigsqcup_{i\in
I}E_{i}\rightarrow F$. In other words, $w$ is a map $\bigsqcup_{i\in I}%
E_{i}\rightarrow\left\{  1,2,3,\ldots\right\}  $ (since $F=\left\{
1,2,3,\ldots\right\}  $).

Recall that $\bigsqcup_{i\in I}\mathbf{E}_{i}$ is a double poset. Its
definition shows that%
\[
\bigsqcup_{i\in I}\mathbf{E}_{i}=\left(  \bigsqcup_{i\in I}E_{i}%
,\bigoplus_{i\in I}\left(  <_{1,i}\right)  ,\bigoplus_{i\in I}\left(
<_{2,i}\right)  \right)  .
\]
Proposition~\ref{prop.Gammaw.qsym} (applied to $\bigsqcup_{i\in I}%
\mathbf{E}_{i}$, $\bigsqcup_{i\in I}E_{i}$, $\bigoplus_{i\in I}\left(
<_{1,i}\right)  $ and $\bigoplus_{i\in I}\left(  <_{2,i}\right)  $ instead of
$\mathbf{E}$, $E$, $<_{1}$ and $<_{2}$) thus yields that $\Gamma\left(
\bigsqcup_{i\in I}\mathbf{E}_{i},w\right)  \in\operatorname*{QSym}$.

But Corollary \ref{cor.djun.Gamma.2} yields%
\[
\prod_{i\in I}\Gamma\left(  \mathbf{E}_{i},w_{i}\right)  =\Gamma\left(
\bigsqcup_{i\in I}\mathbf{E}_{i},w\right)  \in\operatorname*{QSym}.
\]
This proves Corollary \ref{cor.djun.Gamma.2.cor}.
\end{proof}

Now, we can prove the following basic fact:

\begin{proposition}
\label{prop.QSym.alg}The subset $\operatorname*{QSym}$ of $\mathbf{k}\left[
\left[  x_{1},x_{2},x_{3},\ldots\right]  \right]  $ is a $\mathbf{k}%
$-subalgebra of $\mathbf{k}\left[  \left[  x_{1},x_{2},x_{3},\ldots\right]
\right]  $.
\end{proposition}

\begin{proof}
[Proof of Proposition \ref{prop.QSym.alg}.]It is well-known that
$\operatorname*{QSym}$ is a $\mathbf{k}$-submodule of $\mathbf{k}\left[
\left[  x_{1},x_{2},x_{3},\ldots\right]  \right]  $, and that $\left(
M_{\alpha}\right)  _{\alpha\in{\operatorname{Comp}}}$ is a basis of this
$\mathbf{k}$-module $\operatorname*{QSym}$.

We shall now prove that
\begin{equation}
M_{\alpha}M_{\beta}\in\operatorname*{QSym}\ \ \ \ \ \ \ \ \ \ \text{for any
}\alpha\in\operatorname*{Comp}\text{ and }\beta\in\operatorname*{Comp}.
\label{pf.prop.QSym.alg.prod.1}%
\end{equation}

[\textit{Proof of (\ref{pf.prop.QSym.alg.prod.1}):} Let $\alpha\in
\operatorname*{Comp}$ and $\beta\in\operatorname*{Comp}$.

We know that $\beta\in\operatorname*{Comp}$. In other words, $\beta$ is a
composition (since $\operatorname*{Comp}$ is the set of all compositions).
Corollary \ref{cor.Malpha.Gamma} (applied to $\beta$ instead of $\alpha$) thus
shows that there exist a set $E$, a special double poset ${\mathbf{E}}=\left(
E,<_{1},>_{1}\right)  $, and a map $w:E\rightarrow\left\{  1,2,3,\ldots
\right\}  $ satisfying $\Gamma\left(  {\mathbf{E}},w\right)  =M_{\beta}$.
Renaming $E$, $\mathbf{E}$, $<_{1}$, $>_{1}$ and $w$ as $E_{1}$,
$\mathbf{E}_{1}$, $<_{1,1}$, $>_{1,1}$ and $w_{1}$ in this statement, we
obtain the following: There exist a set $E_{1}$, a special double poset
$\mathbf{E}_{1}=\left(  E_{1},<_{1,1},>_{1,1}\right)  $, and a map
$w_{1}:E_{1}\rightarrow\left\{  1,2,3,\ldots\right\}  $ satisfying
$\Gamma\left(  {\mathbf{E}}_{1},w_{1}\right)  =M_{\beta}$. Consider these
$E_{1}$, $\mathbf{E}_{1}$, $<_{1,1}$, $>_{1,1}$ and $w_{1}$.

We know that $\alpha\in\operatorname*{Comp}$. In other words, $\alpha$ is a
composition (since $\operatorname*{Comp}$ is the set of all compositions).
Corollary \ref{cor.Malpha.Gamma} thus shows that there exist a set $E$, a
special double poset ${\mathbf{E}}=\left(  E,<_{1},>_{1}\right)  $, and a map
$w:E\rightarrow\left\{  1,2,3,\ldots\right\}  $ satisfying $\Gamma\left(
{\mathbf{E}},w\right)  =M_{\alpha}$. Renaming $E$, $\mathbf{E}$, $<_{1}$,
$>_{1}$ and $w$ as $E_{0}$, $\mathbf{E}_{0}$, $<_{1,0}$, $>_{1,0}$ and $w_{0}$
in this statement, we obtain the following: There exist a set $E_{0}$, a
special double poset $\mathbf{E}_{0}=\left(  E_{0},<_{1,0},>_{1,0}\right)  $,
and a map $w_{0}:E_{0}\rightarrow\left\{  1,2,3,\ldots\right\}  $ satisfying
$\Gamma\left(  {\mathbf{E}}_{0},w_{0}\right)  =M_{\alpha}$. Consider these
$E_{0}$, $\mathbf{E}_{0}$, $<_{1,0}$, $>_{1,0}$ and $w_{0}$.

Now, $\mathbf{E}_{i}=\left(  E_{i},<_{1,i},>_{1,i}\right)  $ for every
$i\in\left\{  0,1\right\}  $\ \ \ \ \footnote{\textit{Proof:}
\par
\begin{itemize}
\item For $i=0$, this follows from $\mathbf{E}_{0}=\left(  E_{0}%
,<_{1,0},>_{1,0}\right)  $.
\par
\item For $i=1$, this follows from $\mathbf{E}_{1}=\left(  E_{1}%
,<_{1,1},>_{1,1}\right)  $.
\end{itemize}
}. Moreover, $w_{i}$ is a map $E_{i}\rightarrow\left\{  1,2,3,\ldots\right\}
$ for each $i\in\left\{  0,1\right\}  $\ \ \ \ \footnote{\textit{Proof:}
\par
\begin{itemize}
\item For $i=0$, this holds because $w_{0}$ is a map $E_{0}\rightarrow\left\{
1,2,3,\ldots\right\}  $.
\par
\item For $i=1$, this holds because $w_{1}$ is a map $E_{1}\rightarrow\left\{
1,2,3,\ldots\right\}  $.
\end{itemize}
}. Hence, Corollary \ref{cor.djun.Gamma.2.cor} (applied to $\left\{
0,1\right\}  $ and $>_{1,i}$ instead of $I$ and $<_{2,i}$) shows that
\[
\prod_{i\in\left\{  0,1\right\}  }\Gamma\left(  \mathbf{E}_{i},w_{i}\right)
\in\operatorname*{QSym}.
\]
Since
\[
\prod_{i\in\left\{  0,1\right\}  }\Gamma\left(  \mathbf{E}_{i},w_{i}\right)
=\underbrace{\Gamma\left(  {\mathbf{E}}_{0},w_{0}\right)  }_{=M_{\alpha}%
}\underbrace{\Gamma\left(  {\mathbf{E}}_{1},w_{1}\right)  }_{=M_{\beta}%
}=M_{\alpha}M_{\beta},
\]
this rewrites as $M_{\alpha}M_{\beta}\in\operatorname*{QSym}$. Thus,
(\ref{pf.prop.QSym.alg.prod.1}) is proven.]

Now, we can see that%
\begin{equation}
ab\in\operatorname*{QSym}\ \ \ \ \ \ \ \ \ \ \text{for any }a\in
\operatorname*{QSym}\text{ and }b\in\operatorname*{QSym}.
\label{pf.prop.QSym.alg.prod.3}%
\end{equation}

[\textit{Proof of (\ref{pf.prop.QSym.alg.prod.3}):} Let $a\in
\operatorname*{QSym}$ and $b\in\operatorname*{QSym}$. We must prove the
relation $ab\in\operatorname*{QSym}$.

This relation is $\mathbf{k}$-linear in $b$ (since $\operatorname*{QSym}$ is a
$\mathbf{k}$-submodule of $\mathbf{k}\left[  \left[  x_{1},x_{2},x_{3}%
,\ldots\right]  \right]  $). Hence, we can WLOG assume that $b$ belongs to the
basis $\left(  M_{\alpha}\right)  _{\alpha\in\operatorname*{Comp}}$ of the
$\mathbf{k}$-module $\operatorname*{QSym}$. Assume this. Thus, $b=M_{\beta}$
for some $\beta\in\operatorname*{Comp}$. Consider this $\beta$.

We must prove the relation $ab\in\operatorname*{QSym}$. This relation is
$\mathbf{k}$-linear in $a$ (since $\operatorname*{QSym}$ is a $\mathbf{k}%
$-submodule of $\mathbf{k}\left[  \left[  x_{1},x_{2},x_{3},\ldots\right]
\right]  $). Hence, we can WLOG assume that $a$ belongs to the basis $\left(
M_{\alpha}\right)  _{\alpha\in\operatorname*{Comp}}$ of the $\mathbf{k}%
$-module $\operatorname*{QSym}$. Assume this. Thus, $a=M_{\alpha}$ for some
$\alpha\in\operatorname*{Comp}$. Consider this $\alpha$.

Now, $\underbrace{a}_{=M_{\alpha}}\underbrace{b}_{=M_{\beta}}=M_{\alpha
}M_{\beta}\in\operatorname*{QSym}$ (by (\ref{pf.prop.QSym.alg.prod.1})). This
proves (\ref{pf.prop.QSym.alg.prod.3}).]

Finally, let us prove that%
\begin{equation}
1\in\operatorname*{QSym}. \label{pf.prop.QSym.alg.prod.5}%
\end{equation}

[\textit{Proof of (\ref{pf.prop.QSym.alg.prod.5}):} There are many reasons why
(\ref{pf.prop.QSym.alg.prod.5}) is obvious (for example, it follows from Lemma
\ref{lem.Gammaw.empty} (a) or from $M_{\varnothing}=1$), but let us derive
(\ref{pf.prop.QSym.alg.prod.5}) from Corollary \ref{cor.djun.Gamma.2.cor}:

For each $i\in\varnothing$, we define a double poset $\mathbf{E}_{i}=\left(
E_{i},<_{1,i},<_{2,i}\right)  $ and a map $w_{i}:E_{i}\rightarrow\left\{
1,2,3,\ldots\right\}  $ as follows: There is nothing to define, because there
exists no $i\in\varnothing$.

Thus, Corollary \ref{cor.djun.Gamma.2.cor} (applied to $I=\varnothing$)
yields
\[
\prod_{i\in\varnothing}\Gamma\left(  \mathbf{E}_{i},w_{i}\right)
\in\operatorname*{QSym}.
\]
Since $\prod_{i\in\varnothing}\Gamma\left(  \mathbf{E}_{i},w_{i}\right)
=\left(  \text{empty product}\right)  =1$, this rewrites as $1\in
\operatorname*{QSym}$. Thus, (\ref{pf.prop.QSym.alg.prod.5}) is proven.]

Now, recall that $\operatorname*{QSym}$ is a $\mathbf{k}$-submodule of
$\mathbf{k}\left[  \left[  x_{1},x_{2},x_{3},\ldots\right]  \right]  $.
Combining this with (\ref{pf.prop.QSym.alg.prod.3}) and
(\ref{pf.prop.QSym.alg.prod.5}), we conclude that $\operatorname*{QSym}$ is a
$\mathbf{k}$-subalgebra of $\mathbf{k}\left[  \left[  x_{1},x_{2},x_{3}%
,\ldots\right]  \right]  $. This proves Proposition \ref{prop.QSym.alg}.
\end{proof}

\subsection{Restrictions and disjoint unions}

In this short section, we shall prove the following straightforward fact,
which will be used in the next section.

\begin{proposition}
\label{prop.djun.restrict}Let $I$ be a finite set. For each $i\in I$, let
$\mathbf{E}_{i}=\left(  E_{i},<_{1,i},<_{2,i}\right)  $ be a double poset.

Let $\mathbf{E}$ be the double poset $\bigsqcup_{i\in I}\mathbf{E}_{i}$.

For each $i\in I$, let $T_{i}$ be a subset of $E_{i}$.

Then, $\mathbf{E}\mid_{\bigsqcup_{i\in I}T_{i}}=\bigsqcup_{i\in I}\left(
\mathbf{E}_{i}\mid_{T_{i}}\right)  $.
\end{proposition}

The proof will rely on a rather pedantic notation:

\begin{definition}
Let $X$ be a set. Let $\rho$ be a binary relation on the set $X$. Let $Y$ be a
subset of $X$. Then, $\rho\mid_{Y}$ shall denote the restriction of the
relation $\rho$ to $Y$.
\end{definition}

Using this notation, we can state the following trivial fact:

\begin{proposition}
\label{prop.djun.restrict-rewr}Let $\mathbf{E}=\left(  E,<_{1},<_{2}\right)  $
be a double poset. Let $T$ be a subset of $E$. Then, $\mathbf{E}\mid
_{T}=\left(  T,\left(  <_{1}\right)  \mid_{T},\left(  <_{2}\right)  \mid
_{T}\right)  $.
\end{proposition}

\begin{proof}
[Proof of Proposition \ref{prop.djun.restrict-rewr}.]Recall that the double
poset ${\mathbf{E}}\mid_{T}$ is defined as the double poset $\left(
T,<_{1},<_{2}\right)  $, where $<_{1}$ and $<_{2}$ (by abuse of notation)
denote the restrictions of the relations $<_{1}$ and $<_{2}$ to $T$. Avoiding
the abuse of notation, this rewrites as follows: The double poset
${\mathbf{E}}\mid_{T}$ is defined as the double poset $\left(  T,\left(
<_{1}\right)  \mid_{T},\left(  <_{2}\right)  \mid_{T}\right)  $. This proves
Proposition \ref{prop.djun.restrict-rewr}.
\end{proof}

\begin{lemma}
\label{lem.djun.restrict.rel}Let $I$ be a set. For each $i\in I$, let $E_{i}$
be a set. For each $i\in I$, let $T_{i}$ be a subset of $E_{i}$.

\begin{enumerate}
\item[(a)] The set $\bigsqcup_{i\in I}T_{i}$ is a subset of $\bigsqcup_{i\in
I}E_{i}$.

\item[(b)] For each $i\in I$, let $\rho_{i}$ be a binary relation on the set
$E_{i}$. Then, $\left(  \bigoplus_{i\in I}\rho_{i}\right)  \mid_{\bigsqcup
_{i\in I}T_{i}}=\bigoplus_{i\in I}\left(  \rho_{i}\mid_{T_{i}}\right)  $.
\end{enumerate}
\end{lemma}

\begin{proof}
[Proof of Lemma \ref{lem.djun.restrict.rel}.]This is a straightforward
consequence of the definitions.
\end{proof}

\begin{proof}
[Proof of Proposition \ref{prop.djun.restrict}.]Let $E$ denote the set
$\bigsqcup_{i\in I}E_{i}$. Let $T$ denote the set $\bigsqcup_{i\in I}T_{i}$.

Let $<_{1}$ denote the relation $\bigoplus_{i\in I}\left(  <_{1,i}\right)  $.
Let $<_{2}$ denote the relation $\bigoplus_{i\in I}\left(  <_{2,i}\right)  $.

We have
\begin{align*}
\mathbf{E}  &  =\bigsqcup_{i\in I}\mathbf{E}_{i}=\left(  \underbrace{\bigsqcup
_{i\in I}E_{i}}_{=E},\underbrace{\bigoplus_{i\in I}\left(  <_{1,i}\right)
}_{=\left(  <_{1}\right)  },\underbrace{\bigoplus_{i\in I}\left(
<_{2,i}\right)  }_{=\left(  <_{2}\right)  }\right)
\ \ \ \ \ \ \ \ \ \ \left(  \text{by the definition of }\bigsqcup_{i\in
I}\mathbf{E}_{i}\right) \\
&  =\left(  E,<_{1},<_{2}\right)  .
\end{align*}
Lemma \ref{lem.djun.restrict.rel} (a) shows that the set $\bigsqcup_{i\in
I}T_{i}$ is a subset of $\bigsqcup_{i\in I}E_{i}$. In other words, the set $T$
is a subset of $E$ (since $\bigsqcup_{i\in I}T_{i}=T$ and $\bigsqcup_{i\in
I}E_{i}=E$). Hence, Proposition \ref{prop.djun.restrict-rewr} shows that
\begin{equation}
\mathbf{E}\mid_{T}=\left(  T,\left(  <_{1}\right)  \mid_{T},\left(
<_{2}\right)  \mid_{T}\right)  . \label{pf.prop.djun.restrict.3}%
\end{equation}

We have%
\begin{equation}
\bigoplus_{i\in I}\left(  \left(  <_{1,i}\right)  \mid_{T_{i}}\right)
=\left(  <_{1}\right)  \mid_{T} \label{pf.prop.djun.restrict.5a}%
\end{equation}
\footnote{\textit{Proof of (\ref{pf.prop.djun.restrict.5a}):} For each $i\in
I$, the relation $<_{1,i}$ is a binary relation on the set $E_{i}$ (since
$\left(  E_{i},<_{1,i},<_{2,i}\right)  $ is a double poset). Recall
furthermore that $T_{i}$ is a subset of $E_{i}$ for each $i\in I$. Hence,
Lemma \ref{lem.djun.restrict.rel} (b) (applied to $<_{1,i}$ instead of
$\rho_{i}$) shows that $\left(  \bigoplus_{i\in I}\left(  <_{1,i}\right)
\right)  \mid_{\bigsqcup_{i\in I}T_{i}}=\bigoplus_{i\in I}\left(  \left(
<_{1,i}\right)  \mid_{T_{i}}\right)  $. Thus,%
\[
\bigoplus_{i\in I}\left(  \left(  <_{1,i}\right)  \mid_{T_{i}}\right)
=\underbrace{\left(  \bigoplus_{i\in I}\left(  <_{1,i}\right)  \right)
}_{\substack{=\left(  <_{1}\right)  \\\text{(since }<_{1}\text{ was
defined}\\\text{as }\bigoplus_{i\in I}\left(  <_{1,i}\right)  \text{)}}%
}\mid_{\bigsqcup_{i\in I}T_{i}}=\left(  <_{1}\right)  \mid_{\bigsqcup_{i\in
I}T_{i}}=\left(  <_{1}\right)  \mid_{T}\ \ \ \ \ \ \ \ \ \ \left(  \text{since
}\bigsqcup_{i\in I}T_{i}=T\right)  .
\]
This proves (\ref{pf.prop.djun.restrict.5a}).}. The same argument (applied to
$<_{2,i}$ and $<_{2}$ instead of $<_{1,i}$ and $<_{1}$) shows that%
\begin{equation}
\bigoplus_{i\in I}\left(  \left(  <_{2,i}\right)  \mid_{T_{i}}\right)
=\left(  <_{2}\right)  \mid_{T}. \label{pf.prop.djun.restrict.5b}%
\end{equation}

Furthermore, for each $i\in I$, we have $\mathbf{E}_{i}\mid_{T_{i}}=\left(
T_{i},\left(  <_{1,i}\right)  \mid_{T_{i}},\left(  <_{2,i}\right)  \mid
_{T_{i}}\right)  $\ \ \ \ \footnote{\textit{Proof.} Let $i\in I$. Then,
$\mathbf{E}_{i}=\left(  E_{i},<_{1,i},<_{2,i}\right)  $ is a double poset,
whereas $T_{i}$ is a subset of $E_{i}$. Hence, Proposition
\ref{prop.djun.restrict-rewr} (applied to $\mathbf{E}_{i}$, $E_{i}$, $<_{1,i}%
$, $<_{2,i}$ and $T_{i}$ instead of $\mathbf{E}$, $E$, $<_{1}$, $<_{2}$ and
$T$) shows that $\mathbf{E}_{i}\mid_{T_{i}}=\left(  T_{i},\left(
<_{1,i}\right)  \mid_{T_{i}},\left(  <_{2,i}\right)  \mid_{T_{i}}\right)  $.}.
Hence, the definition of $\bigsqcup_{i\in I}\left(  \mathbf{E}_{i}\mid_{T_{i}%
}\right)  $ yields%
\begin{align*}
\bigsqcup_{i\in I}\left(  \mathbf{E}_{i}\mid_{T_{i}}\right)   &  =\left(
\underbrace{\bigsqcup_{i\in I}T_{i}}_{=T},\underbrace{\bigoplus_{i\in
I}\left(  \left(  <_{1,i}\right)  \mid_{T_{i}}\right)  }_{\substack{=\left(
<_{1}\right)  \mid_{T}\\\text{(by (\ref{pf.prop.djun.restrict.5a}))}%
}},\underbrace{\bigoplus_{i\in I}\left(  \left(  <_{2,i}\right)  \mid_{T_{i}%
}\right)  }_{\substack{=\left(  <_{2}\right)  \mid_{T}\\\text{(by
(\ref{pf.prop.djun.restrict.5b}))}}}\right) \\
&  =\left(  T,\left(  <_{1}\right)  \mid_{T},\left(  <_{2}\right)  \mid
_{T}\right)  =\mathbf{E}\mid_{T}=\mathbf{E}\mid_{\bigsqcup_{i\in I}T_{i}}%
\end{align*}
(since $T=\bigsqcup_{i\in I}T_{i}$). This proves Proposition
\ref{prop.djun.restrict}.
\end{proof}

\subsection{$\operatorname*{Adm}\left(  \bigsqcup_{i\in I}\mathbf{E}%
_{i}\right)  $ and the bialgebra $\operatorname*{QSym}$}

In this section, we shall continue analyzing the disjoint union of several
double posets. This will result in a new proof of the fact that
$\operatorname*{QSym}$ is a $\mathbf{k}$-bialgebra.

We begin with some simple facts:

\begin{proposition}
\label{prop.djun.Adm1}Let $I$ be a finite set. For each $i\in I$, let
$\mathbf{E}_{i}=\left(  E_{i},<_{1,i},<_{2,i}\right)  $ be a double poset.

Let $\mathbf{E}$ be the double poset $\bigsqcup_{i\in I}\mathbf{E}_{i}$.

\begin{enumerate}
\item[(a)] For each $\left(  P,Q\right)  \in\operatorname*{Adm}\mathbf{E}$, we
have%
\[
\left(  \left(  \left(  \operatorname*{inc}\nolimits_{i}\right)  ^{-1}\left(
P\right)  ,\left(  \operatorname*{inc}\nolimits_{i}\right)  ^{-1}\left(
Q\right)  \right)  \right)  _{i\in I}\in\prod_{i\in I}\operatorname*{Adm}%
\left(  \mathbf{E}_{i}\right)  .
\]

\item[(b)] For each $\left(  \left(  P_{i},Q_{i}\right)  \right)  _{i\in I}%
\in\prod_{i\in I}\operatorname*{Adm}\left(  \mathbf{E}_{i}\right)  $, we have%
\[
\left(  \bigsqcup_{i\in I}P_{i},\bigsqcup_{i\in I}Q_{i}\right)  \in
\operatorname*{Adm}\mathbf{E}.
\]

\end{enumerate}
\end{proposition}

\begin{proof}
[Proof of Proposition \ref{prop.djun.Adm1}.]Write the double poset
$\mathbf{E}$ as $\left(  E,<_{1},<_{2}\right)  $. Thus,%
\[
\left(  E,<_{1},<_{2}\right)  =\mathbf{E}=\bigsqcup_{i\in I}\mathbf{E}%
_{i}=\left(  \bigsqcup_{i\in I}E_{i},\bigoplus_{i\in I}\left(  <_{1,i}\right)
,\bigoplus_{i\in I}\left(  <_{2,i}\right)  \right)
\]
(by the definition of the double poset $\bigsqcup_{i\in I}\mathbf{E}_{i}$). In
other words,%
\[
E=\bigsqcup_{i\in I}E_{i},\ \ \ \ \ \ \ \ \ \ \left(  <_{1}\right)
=\bigoplus_{i\in I}\left(  <_{1,i}\right)  \ \ \ \ \ \ \ \ \ \ \text{and}%
\ \ \ \ \ \ \ \ \ \ \left(  <_{2}\right)  =\bigoplus_{i\in I}\left(
<_{2,i}\right)  .
\]

For any two elements $\left(  j,e\right)  $ and $\left(  k,f\right)  $ of $E$,
we have the equivalence%
\begin{equation}
\left(  \left(  j,e\right)  <_{1}\left(  k,f\right)  \right)
\ \Longleftrightarrow\ \left(  j=k\text{ and }e<_{1,j}f\right)
\label{pf.prop.djun.Adm1.le1}%
\end{equation}
\footnote{\textit{Proof of (\ref{pf.prop.djun.Adm1.le1}):} Recall that
$\left(  <_{1}\right)  =\bigoplus_{i\in I}\left(  <_{1,i}\right)  $.
\par
Now, the definition of the relation $\bigoplus_{i\in I}\left(  <_{1,i}\right)
$ shows that%
\[
\left(
\begin{array}
[c]{l}%
\left(  \left(  j,e\right)  \left(  \bigoplus_{i\in I}\left(  <_{1,i}\right)
\right)  \left(  k,f\right)  \right)  \ \Longleftrightarrow\ \left(  j=k\text{
and }e<_{1,j}f\right) \\
\ \ \ \ \ \ \ \ \ \ \text{for any two elements }\left(  j,e\right)  \text{ and
}\left(  k,f\right)  \text{ of }\bigsqcup_{i\in I}E_{i}%
\end{array}
\right)  .
\]
In other words,%
\[
\left(
\begin{array}
[c]{l}%
\left(  \left(  j,e\right)  <_{1}\left(  k,f\right)  \right)
\ \Longleftrightarrow\ \left(  j=k\text{ and }e<_{1,j}f\right) \\
\ \ \ \ \ \ \ \ \ \ \text{for any two elements }\left(  j,e\right)  \text{ and
}\left(  k,f\right)  \text{ of }E
\end{array}
\right)
\]
(since $\bigoplus_{i\in I}\left(  <_{1,i}\right)  =\left(  <_{1}\right)  $ and
$\bigsqcup_{i\in I}E_{i}=E$). This proves (\ref{pf.prop.djun.Adm1.le1}).}.

(a) Let $\left(  P,Q\right)  \in\operatorname*{Adm}\mathbf{E}$. We shall show
that \newline$\left(  \left(  \left(  \operatorname*{inc}\nolimits_{i}\right)
^{-1}\left(  P\right)  ,\left(  \operatorname*{inc}\nolimits_{i}\right)
^{-1}\left(  Q\right)  \right)  \right)  _{i\in I}\in\prod_{i\in
I}\operatorname*{Adm}\left(  \mathbf{E}_{i}\right)  $.

Indeed, fix $i\in I$. We shall prove that $\left(  \left(  \operatorname*{inc}%
\nolimits_{i}\right)  ^{-1}\left(  P\right)  ,\left(  \operatorname*{inc}%
\nolimits_{i}\right)  ^{-1}\left(  Q\right)  \right)  \in\operatorname*{Adm}%
\left(  \mathbf{E}_{i}\right)  $.

The sets $\left(  \operatorname*{inc}\nolimits_{i}\right)  ^{-1}\left(
P\right)  $ and $\left(  \operatorname*{inc}\nolimits_{i}\right)  ^{-1}\left(
Q\right)  $ are subsets of $E$ (since $\operatorname*{inc}\nolimits_{i}$ is a
map $E_{i}\rightarrow E$).

We have $\left(  P,Q\right)  \in\operatorname*{Adm}\mathbf{E}$. In other
words, $P$ and $Q$ are subsets of $E$ satisfying $P\cap Q=\varnothing$ and
$P\cup Q=E$ and having the property that
\begin{equation}
\text{no }p\in P\text{ and }q\in Q\text{ satisfy }q<_{1}p
\label{pf.prop.djun.Adm1.a.1}%
\end{equation}
(by the definition of $\operatorname*{Adm}\mathbf{E}$ (since $\mathbf{E}%
=\left(  E,<_{1},<_{2}\right)  $).

If $p\in P$ and $q\in Q$, then%
\begin{equation}
\text{we do not have }q<_{1}p. \label{pf.prop.djun.Adm1.a.1a}%
\end{equation}
(Indeed, this is merely a restatement of (\ref{pf.prop.djun.Adm1.a.1}).)

We have
\[
\left(  \operatorname*{inc}\nolimits_{i}\right)  ^{-1}\left(  P\right)
\cap\left(  \operatorname*{inc}\nolimits_{i}\right)  ^{-1}\left(  Q\right)
=\left(  \operatorname*{inc}\nolimits_{i}\right)  ^{-1}\left(
\underbrace{P\cap Q}_{=\varnothing}\right)  =\left(  \operatorname*{inc}%
\nolimits_{i}\right)  ^{-1}\left(  \varnothing\right)  =\varnothing
\]
and%
\[
\left(  \operatorname*{inc}\nolimits_{i}\right)  ^{-1}\left(  P\right)
\cup\left(  \operatorname*{inc}\nolimits_{i}\right)  ^{-1}\left(  Q\right)
=\left(  \operatorname*{inc}\nolimits_{i}\right)  ^{-1}\left(
\underbrace{P\cup Q}_{=E}\right)  =\left(  \operatorname*{inc}\nolimits_{i}%
\right)  ^{-1}\left(  E\right)  =E_{i}.
\]

Moreover, no $p\in\left(  \operatorname*{inc}\nolimits_{i}\right)
^{-1}\left(  P\right)  $ and $q\in\left(  \operatorname*{inc}\nolimits_{i}%
\right)  ^{-1}\left(  Q\right)  $ satisfy $q<_{1,i}p$%
\ \ \ \ \footnote{\textit{Proof.} Assume the contrary. Thus, there exist
$p\in\left(  \operatorname*{inc}\nolimits_{i}\right)  ^{-1}\left(  P\right)  $
and $q\in\left(  \operatorname*{inc}\nolimits_{i}\right)  ^{-1}\left(
Q\right)  $ satisfying $q<_{1,i}p$. Consider these $p$ and $q$.
\par
We have $p\in\left(  \operatorname*{inc}\nolimits_{i}\right)  ^{-1}\left(
P\right)  $. In other words, $p$ is an element of $E_{i}$ satisfying
$\operatorname*{inc}\nolimits_{i}\left(  p\right)  \in P$. The definition of
$\operatorname*{inc}\nolimits_{i}$ yields $\operatorname*{inc}\nolimits_{i}%
\left(  p\right)  =\left(  i,p\right)  $; thus, $\left(  i,p\right)
=\operatorname*{inc}\nolimits_{i}\left(  p\right)  \in P\subseteq E$.
\par
We have $q\in\left(  \operatorname*{inc}\nolimits_{i}\right)  ^{-1}\left(
Q\right)  $. In other words, $q$ is an element of $E_{i}$ satisfying
$\operatorname*{inc}\nolimits_{i}\left(  q\right)  \in Q$. The definition of
$\operatorname*{inc}\nolimits_{i}$ yields $\operatorname*{inc}\nolimits_{i}%
\left(  q\right)  =\left(  i,q\right)  $; thus, $\left(  i,q\right)
=\operatorname*{inc}\nolimits_{i}\left(  q\right)  \in Q\subseteq E$.
\par
Now, (\ref{pf.prop.djun.Adm1.le1}) (applied to $\left(  j,e\right)  =\left(
i,q\right)  $ and $\left(  k,f\right)  =\left(  i,p\right)  $) yields the
equivalence%
\[
\left(  \left(  i,q\right)  <_{1}\left(  i,p\right)  \right)
\ \Longleftrightarrow\ \left(  i=i\text{ and }q<_{1,i}p\right)  .
\]
Hence, we have $\left(  i,q\right)  <_{1}\left(  i,p\right)  $ (since we have
$\left(  i=i\text{ and }q<_{1,i}p\right)  $). Moreover, recall that $\left(
i,p\right)  \in P$ and $\left(  i,q\right)  \in Q$. Hence,
(\ref{pf.prop.djun.Adm1.a.1a}) (applied to $\left(  i,p\right)  $ and $\left(
i,q\right)  $ instead of $p$ and $q$) shows that we do not have $\left(
i,q\right)  <_{1}\left(  i,p\right)  $. This contradicts $\left(  i,q\right)
<_{1}\left(  i,p\right)  $. This contradiction shows that our assumption was
false. This completes the proof.}.

Thus, $\left(  \operatorname*{inc}\nolimits_{i}\right)  ^{-1}\left(  P\right)
$ and $\left(  \operatorname*{inc}\nolimits_{i}\right)  ^{-1}\left(  Q\right)
$ are subsets of $E_{i}$ satisfying $\left(  \operatorname*{inc}%
\nolimits_{i}\right)  ^{-1}\left(  P\right)  \cap\left(  \operatorname*{inc}%
\nolimits_{i}\right)  ^{-1}\left(  Q\right)  =\varnothing$ and $\left(
\operatorname*{inc}\nolimits_{i}\right)  ^{-1}\left(  P\right)  \cup\left(
\operatorname*{inc}\nolimits_{i}\right)  ^{-1}\left(  Q\right)  =E_{i}$ and
having the property that no $p\in\left(  \operatorname*{inc}\nolimits_{i}%
\right)  ^{-1}\left(  P\right)  $ and $q\in\left(  \operatorname*{inc}%
\nolimits_{i}\right)  ^{-1}\left(  Q\right)  $ satisfy $q<_{1,i}p$. In other
words, $\left(  \left(  \operatorname*{inc}\nolimits_{i}\right)  ^{-1}\left(
P\right)  ,\left(  \operatorname*{inc}\nolimits_{i}\right)  ^{-1}\left(
Q\right)  \right)  \in\operatorname*{Adm}\left(  \mathbf{E}_{i}\right)  $ (by
the definition of $\operatorname*{Adm}\left(  \mathbf{E}_{i}\right)  $ (since
$\mathbf{E}_{i}=\left(  E_{i},<_{1,i},<_{2,i}\right)  $)).

Now, forget that we fixed $i$. We thus have shown that $\left(  \left(
\operatorname*{inc}\nolimits_{i}\right)  ^{-1}\left(  P\right)  ,\left(
\operatorname*{inc}\nolimits_{i}\right)  ^{-1}\left(  Q\right)  \right)
\in\operatorname*{Adm}\left(  \mathbf{E}_{i}\right)  $ for each $i\in I$. In
other words, $\left(  \left(  \left(  \operatorname*{inc}\nolimits_{i}\right)
^{-1}\left(  P\right)  ,\left(  \operatorname*{inc}\nolimits_{i}\right)
^{-1}\left(  Q\right)  \right)  \right)  _{i\in I}\in\prod_{i\in
I}\operatorname*{Adm}\left(  \mathbf{E}_{i}\right)  $. This proves Proposition
\ref{prop.djun.Adm1} (a).

(b) Let $\left(  \left(  P_{i},Q_{i}\right)  \right)  _{i\in I}\in\prod_{i\in
I}\operatorname*{Adm}\left(  \mathbf{E}_{i}\right)  $. We shall show that
$\left(  \bigsqcup_{i\in I}P_{i},\bigsqcup_{i\in I}Q_{i}\right)
\in\operatorname*{Adm}\mathbf{E}$.

We have $\left(  \left(  P_{i},Q_{i}\right)  \right)  _{i\in I}\in\prod_{i\in
I}\operatorname*{Adm}\left(  \mathbf{E}_{i}\right)  $. In other words,%
\begin{equation}
\left(  P_{i},Q_{i}\right)  \in\operatorname*{Adm}\left(  \mathbf{E}%
_{i}\right)  \ \ \ \ \ \ \ \ \ \ \text{for each }i\in I.
\label{pf.prop.djun.Adm1.b.1}%
\end{equation}

From this, we can easily obtain the following observation:

\begin{statement}
\textit{Observation 1:} Let $i\in I$. Then, the sets $P_{i}$ and $Q_{i}$ are
subsets of $E_{i}$ satisfying $P_{i}\cap Q_{i}=\varnothing$ and $P_{i}\cup
Q_{i}=E_{i}$ and having the property that
\begin{equation}
\text{no }p\in P_{i}\text{ and }q\in Q_{i}\text{ satisfy }q<_{1,i}p.
\label{pf.prop.djun.Adm1.b.2}%
\end{equation}

\end{statement}

[\textit{Proof of Observation 1:} From (\ref{pf.prop.djun.Adm1.b.1}), we
obtain $\left(  P_{i},Q_{i}\right)  \in\operatorname*{Adm}\left(
\mathbf{E}_{i}\right)  $. From this, Observation 1 immediately follows by the
definition of $\operatorname*{Adm}\left(  \mathbf{E}_{i}\right)  $.]

Now, $\bigsqcup_{i\in I}P_{i}$ and $\bigsqcup_{i\in I}Q_{i}$ are subsets of
$E$\ \ \ \ \footnote{\textit{Proof.} For each $i\in I$, the set $P_{i}$ is a
subset of $E_{i}$ (by Observation 1). In other words, for each $i\in I$, we
have $P_{i}\subseteq E_{i}$. Thus, $\bigsqcup_{i\in I}\underbrace{P_{i}%
}_{\subseteq E_{i}}\subseteq\bigsqcup_{i\in I}E_{i}=E$. The same argument (but
applied to $Q_{i}$ instead of $P_{i}$) shows that $\bigsqcup_{i\in I}%
Q_{i}\subseteq E$. Thus, $\bigsqcup_{i\in I}P_{i}$ and $\bigsqcup_{i\in
I}Q_{i}$ are subsets of $E$.} satisfying $\left(  \bigsqcup_{i\in I}%
P_{i}\right)  \cap\left(  \bigsqcup_{i\in I}Q_{i}\right)  =\varnothing
$\ \ \ \ \footnote{\textit{Proof.} Let $x\in\left(  \bigsqcup_{i\in I}%
P_{i}\right)  \cap\left(  \bigsqcup_{i\in I}Q_{i}\right)  $. We shall derive a
contradiction.
\par
We have $x\in\left(  \bigsqcup_{i\in I}P_{i}\right)  \cap\left(
\bigsqcup_{i\in I}Q_{i}\right)  \subseteq\bigsqcup_{i\in I}P_{i}$. Hence,
Remark \ref{rmk.djun.elt} (c) (applied to $P_{i}$ instead of $E_{i}$) shows
that there exist an $i\in I$ and an $e\in P_{i}$ such that $x=\left(
i,e\right)  $. Denote these $i$ and $e$ by $j$ and $f$. Thus, $j\in I$ and
$f\in P_{j}$ satisfy $x=\left(  j,f\right)  $.
\par
We have $x\in\left(  \bigsqcup_{i\in I}P_{i}\right)  \cap\left(
\bigsqcup_{i\in I}Q_{i}\right)  \subseteq\bigsqcup_{i\in I}Q_{i}$. Hence,
Remark \ref{rmk.djun.elt} (c) (applied to $Q_{i}$ instead of $E_{i}$) shows
that there exist an $i\in I$ and an $e\in Q_{i}$ such that $x=\left(
i,e\right)  $. Denote these $i$ and $e$ by $k$ and $g$. Thus, $k\in I$ and
$g\in Q_{k}$ satisfy $x=\left(  k,g\right)  $.
\par
We have $\left(  j,f\right)  =x=\left(  k,g\right)  $. In other words, $j=k$
and $f=g$. Hence, $f=g\in Q_{k}=Q_{j}$ (since $k=j$).
\par
Observation 1 (applied to $i=j$) yields $P_{j}\cap Q_{j}=\varnothing$. But
combining $f\in P_{j}$ with $f\in Q_{j}$, we obtain $f\in P_{j}\cap
Q_{j}=\varnothing$. This shows that the set $\varnothing$ is nonempty (because
it contains the element $f$). But this is absurd.
\par
Now, forget that we have fixed $x$. We thus have derived a contradiction for
each $x\in\left(  \bigsqcup_{i\in I}P_{i}\right)  \cap\left(  \bigsqcup_{i\in
I}Q_{i}\right)  $. Hence, there exists no $x\in\left(  \bigsqcup_{i\in I}%
P_{i}\right)  \cap\left(  \bigsqcup_{i\in I}Q_{i}\right)  $. In other words,
the set $\left(  \bigsqcup_{i\in I}P_{i}\right)  \cap\left(  \bigsqcup_{i\in
I}Q_{i}\right)  $ is empty. In other words, $\left(  \bigsqcup_{i\in I}%
P_{i}\right)  \cap\left(  \bigsqcup_{i\in I}Q_{i}\right)  =\varnothing$.} and
$\left(  \bigsqcup_{i\in I}P_{i}\right)  \cup\left(  \bigsqcup_{i\in I}%
Q_{i}\right)  =E$\ \ \ \ \footnote{\textit{Proof.} We have $E=\bigsqcup_{i\in
I}E_{i}=\bigcup_{i\in I}\left(  \left\{  i\right\}  \times E_{i}\right)  $ (by
the definition of $\bigsqcup_{i\in I}E_{i}$) and $\bigsqcup_{i\in I}%
P_{i}=\bigcup_{i\in I}\left(  \left\{  i\right\}  \times P_{i}\right)  $ (by
the definition of $\bigsqcup_{i\in I}P_{i}$) and $\bigsqcup_{i\in I}%
Q_{i}=\bigcup_{i\in I}\left(  \left\{  i\right\}  \times Q_{i}\right)  $ (by
the definition of $\bigsqcup_{i\in I}Q_{i}$). Now,%
\begin{align*}
\underbrace{\left(  \bigsqcup_{i\in I}P_{i}\right)  }_{=\bigcup_{i\in
I}\left(  \left\{  i\right\}  \times P_{i}\right)  }\cup\underbrace{\left(
\bigsqcup_{i\in I}Q_{i}\right)  }_{=\bigcup_{i\in I}\left(  \left\{
i\right\}  \times Q_{i}\right)  }  &  =\left(  \bigcup_{i\in I}\left(
\left\{  i\right\}  \times P_{i}\right)  \right)  \cup\left(  \bigcup_{i\in
I}\left(  \left\{  i\right\}  \times Q_{i}\right)  \right)  =\bigcup_{i\in
I}\underbrace{\left(  \left(  \left\{  i\right\}  \times P_{i}\right)
\cup\left(  \left\{  i\right\}  \times Q_{i}\right)  \right)  }%
_{\substack{=\left\{  i\right\}  \times\left(  P_{i}\cup Q_{i}\right)
\\\text{(since any three sets }X\text{, }Y\text{ and }Z\\\text{satisfy
}\left(  X\times Y\right)  \cup\left(  X\times Z\right)  =X\times\left(  Y\cup
Z\right)  \text{)}}}\\
&  =\bigcup_{i\in I}\left(  \left\{  i\right\}  \times\underbrace{\left(
P_{i}\cup Q_{i}\right)  }_{\substack{=E_{i}\\\text{(by Observation 1)}%
}}\right)  =\bigcup_{i\in I}\left(  \left\{  i\right\}  \times E_{i}\right)
=E,
\end{align*}
qed.} and having the property that%
\[
\text{no }p\in\bigsqcup_{i\in I}P_{i}\text{ and }q\in\bigsqcup_{i\in I}%
Q_{i}\text{ satisfy }q<_{1}p
\]
\footnote{\textit{Proof.} Assume the contrary. Thus, there exist
$p\in\bigsqcup_{i\in I}P_{i}$ and $q\in\bigsqcup_{i\in I}Q_{i}$ satisfying
$q<_{1}p$. Fix these $p$ and $q$.
\par
We have $p\in\bigsqcup_{i\in I}P_{i}$. Hence, Remark \ref{rmk.djun.elt} (c)
(applied to $p$ and $P_{i}$ instead of $x$ and $E_{i}$) shows that there exist
an $i\in I$ and an $e\in P_{i}$ such that $p=\left(  i,e\right)  $. Denote
these $i$ and $e$ by $k$ and $f$. Thus, $k\in I$ and $f\in P_{k}$ satisfy
$p=\left(  k,f\right)  $.
\par
We have $q\in\bigsqcup_{i\in I}Q_{i}$. Hence, Remark \ref{rmk.djun.elt} (c)
(applied to $q$ and $Q_{i}$ instead of $x$ and $E_{i}$) shows that there exist
an $i\in I$ and an $e\in Q_{i}$ such that $q=\left(  i,e\right)  $. Denote
these $i$ and $e$ by $j$ and $e$. Thus, $j\in I$ and $e\in Q_{j}$ satisfy
$q=\left(  j,e\right)  $.
\par
We have $\left(  k,f\right)  =p\in\bigsqcup_{i\in I}P_{i}\subseteq E$ and
$\left(  j,e\right)  =q\in\bigsqcup_{i\in I}Q_{i}\subseteq E$. Also, $\left(
j,e\right)  =q<_{1}p=\left(  k,f\right)  $. But (\ref{pf.prop.djun.Adm1.le1})
yields the equivalence%
\[
\left(  \left(  j,e\right)  <_{1}\left(  k,f\right)  \right)
\ \Longleftrightarrow\ \left(  j=k\text{ and }e<_{1,j}f\right)  .
\]
Hence, we have $\left(  j=k\text{ and }e<_{1,j}f\right)  $ (since we have
$\left(  j,e\right)  <_{1}\left(  k,f\right)  $).
\par
We have $e\in Q_{j}=Q_{k}$ (since $j=k$). Also, we have $e<_{1,j}f$. This
rewrites as $e<_{1,k}f$ (since $j=k$).
\par
But (\ref{pf.prop.djun.Adm1.b.2}) (applied to $i=k$) shows that no $p\in
P_{k}$ and $q\in Q_{k}$ satisfy $q<_{1,k}p$. In other words, if $p\in P_{k}$
and $q\in Q_{k}$, then we do not have $q<_{1,k}p$. Applying this to $p=f$ and
$q=e$, we conclude that we do not have $e<_{1,k}f$ (since $f\in P_{k}$ and
$e\in Q_{k}$). This contradicts $e<_{1,k}f$. This contradiction shows that our
assumption was wrong. This completes the proof.}. In other words, $\left(
\bigsqcup_{i\in I}P_{i},\bigsqcup_{i\in I}Q_{i}\right)  \in\operatorname*{Adm}%
\mathbf{E}$ (by the definition of $\operatorname*{Adm}\mathbf{E}$ (since
$\mathbf{E}=\left(  E,<_{1},<_{2}\right)  $). This proves Proposition
\ref{prop.djun.Adm1} (b).
\end{proof}

\begin{definition}
\label{def.djun.Adm-bij}Let $I$ be a finite set. For each $i\in I$, let
$\mathbf{E}_{i}=\left(  E_{i},<_{1,i},<_{2,i}\right)  $ be a double poset.

Let $\mathbf{E}$ be the double poset $\bigsqcup_{i\in I}\mathbf{E}_{i}$.

\begin{enumerate}
\item[(a)] We define a map
\[
\operatorname*{Split}:\operatorname*{Adm}\mathbf{E}\rightarrow\prod_{i\in
I}\operatorname*{Adm}\left(  \mathbf{E}_{i}\right)
\]
by%
\[
\left(
\begin{array}
[c]{c}%
\operatorname*{Split}\left(  \left(  P,Q\right)  \right)  =\left(  \left(
\left(  \operatorname*{inc}\nolimits_{i}\right)  ^{-1}\left(  P\right)
,\left(  \operatorname*{inc}\nolimits_{i}\right)  ^{-1}\left(  Q\right)
\right)  \right)  _{i\in I}\\
\ \ \ \ \ \ \ \ \ \ \text{for each }\left(  P,Q\right)  \in\operatorname*{Adm}%
\mathbf{E}%
\end{array}
\right)  .
\]
(This is well-defined, because of Proposition \ref{prop.djun.Adm1} (a).)

\item[(b)] We define a map
\[
\operatorname*{Combine}:\prod_{i\in I}\operatorname*{Adm}\left(
\mathbf{E}_{i}\right)  \rightarrow\operatorname*{Adm}\mathbf{E}%
\]
by%
\[
\left(
\begin{array}
[c]{c}%
\operatorname*{Combine}\left(  \left(  \left(  P_{i},Q_{i}\right)  \right)
_{i\in I}\right)  =\left(  \bigsqcup_{i\in I}P_{i},\bigsqcup_{i\in I}%
Q_{i}\right) \\
\ \ \ \ \ \ \ \ \ \ \text{for every }\left(  \left(  P_{i},Q_{i}\right)
\right)  _{i\in I}\in\prod_{i\in I}\operatorname*{Adm}\left(  \mathbf{E}%
_{i}\right)
\end{array}
\right)  .
\]
(This is well-defined, because of Proposition \ref{prop.djun.Adm1} (b).)
\end{enumerate}
\end{definition}

Next, we need another basic lemma about sets:

\begin{lemma}
\label{lem.djun.Adm-bij.1}Let $I$ be a finite set. For each $i\in I$, let
$E_{i}$ be a set. For each $j\in I$, consider the map $\operatorname*{inc}%
\nolimits_{j}:E_{j}\rightarrow\bigsqcup_{i\in I}E_{i}$.

\begin{enumerate}
\item[(a)] If $R$ is a subset of $\bigsqcup_{i\in I}E_{i}$, then
$\bigsqcup_{i\in I}\left(  \operatorname*{inc}\nolimits_{i}\right)
^{-1}\left(  R\right)  =R$.

\item[(b)] For each $i\in I$, let $R_{i}$ be a subset of $E_{i}$. Let $j\in
I$. Then, $\left(  \operatorname*{inc}\nolimits_{j}\right)  ^{-1}\left(
\bigsqcup_{i\in I}R_{i}\right)  =R_{j}$.
\end{enumerate}
\end{lemma}

\begin{proof}
[Proof of Lemma \ref{lem.djun.Adm-bij.1}.]This is, again, a straightforward
fact about sets, and its proof is left to the reader.
\end{proof}

\begin{proposition}
\label{prop.djun.Adm-bij}Let $I$ be a finite set. For each $i\in I$, let
$\mathbf{E}_{i}=\left(  E_{i},<_{1,i},<_{2,i}\right)  $ be a double poset.

Let $\mathbf{E}$ be the double poset $\bigsqcup_{i\in I}\mathbf{E}_{i}$.

The maps $\operatorname*{Split}$ and $\operatorname*{Combine}$ are mutually
inverse bijections.
\end{proposition}

\begin{proof}
[Proof of Proposition \ref{prop.djun.Adm-bij}.]For each $\left(  P,Q\right)
\in\operatorname*{Adm}\mathbf{E}$, we have
\begin{align}
\operatorname*{Split}\left(  \left(  P,Q\right)  \right)   &  =\left(  \left(
\left(  \operatorname*{inc}\nolimits_{i}\right)  ^{-1}\left(  P\right)
,\left(  \operatorname*{inc}\nolimits_{i}\right)  ^{-1}\left(  Q\right)
\right)  \right)  _{i\in I}\ \ \ \ \ \ \ \ \ \ \left(  \text{by the definition
of }\operatorname*{Split}\right) \nonumber\\
&  =\left(  \left(  \left(  \operatorname*{inc}\nolimits_{j}\right)
^{-1}\left(  P\right)  ,\left(  \operatorname*{inc}\nolimits_{j}\right)
^{-1}\left(  Q\right)  \right)  \right)  _{j\in I}
\label{pf.prop.djun.Adm-bij.Split-rewr}%
\end{align}
(here, we have renamed the index $i$ as $j$).

We have $\operatorname*{Split}\circ\operatorname*{Combine}=\operatorname*{id}%
$\ \ \ \ \ \footnote{\textit{Proof.} Let $\alpha\in\prod_{i\in I}%
\operatorname*{Adm}\left(  \mathbf{E}_{i}\right)  $. Thus, $\alpha$ has the
form $\alpha=\left(  \alpha_{i}\right)  _{i\in I}$, where each $\alpha_{i}$ is
an element of $\operatorname*{Adm}\left(  \mathbf{E}_{i}\right)  $. Consider
these $\alpha_{i}$.
\par
For each $i\in I$, the element $\alpha_{i}\in\operatorname*{Adm}\left(
\mathbf{E}_{i}\right)  $ has the form $\alpha_{i}=\left(  P_{i},Q_{i}\right)
$ for two subsets $P_{i}$ and $Q_{i}$ of $E_{i}$ (by the definition of
$\operatorname*{Adm}\left(  \mathbf{E}_{i}\right)  $ (since $\mathbf{E}%
_{i}=\left(  E_{i},<_{1,i},<_{2,i}\right)  $)). Consider these $P_{i}$ and
$Q_{i}$.
\par
Now, $\alpha=\left(  \alpha_{i} \right)  _{i\in I} =\left(  \left(
P_{i},Q_{i}\right)  \right)  _{i\in I} $ (since $\alpha_{i}=\left(
P_{i},Q_{i}\right)  $ for all $i \in I$). Applying the map
$\operatorname*{Combine}$ to both sides of this equality, we find%
\[
\operatorname*{Combine}\left(  \alpha\right)  =\operatorname*{Combine}\left(
\left(  \left(  P_{i},Q_{i}\right)  \right)  _{i\in I}\right)  =\left(
\bigsqcup_{i\in I}P_{i},\bigsqcup_{i\in I}Q_{i}\right)
\]
(by the definition of $\operatorname*{Combine}$). Hence, $\left(
\bigsqcup_{i\in I}P_{i},\bigsqcup_{i\in I}Q_{i}\right)
=\operatorname*{Combine}\left(  \alpha\right)  \in\operatorname*{Adm}%
\mathbf{E}$. Thus, (\ref{pf.prop.djun.Adm-bij.Split-rewr}) (applied to
$\left(  P,Q\right)  =\left(  \bigsqcup_{i\in I}P_{i},\bigsqcup_{i\in I}%
Q_{i}\right)  $) yields%
\begin{align*}
\operatorname*{Split}\left(  \left(  \bigsqcup_{i\in I}P_{i},\bigsqcup_{i\in
I}Q_{i}\right)  \right)   &  =\left(  \left(  \underbrace{\left(
\operatorname*{inc}\nolimits_{j}\right)  ^{-1}\left(  \bigsqcup_{i\in I}%
P_{i}\right)  }_{\substack{=P_{j}\\\text{(by Lemma \ref{lem.djun.Adm-bij.1}
(b)}\\\text{(applied to }R_{i}=P_{i}\text{))}}},\underbrace{\left(
\operatorname*{inc}\nolimits_{j}\right)  ^{-1}\left(  \bigsqcup_{i\in I}%
Q_{i}\right)  }_{\substack{=Q_{j}\\\text{(by Lemma \ref{lem.djun.Adm-bij.1}
(b)}\\\text{(applied to }R_{i}=Q_{i}\text{))}}}\right)  \right)  _{j\in I}
=\left(  \left(  P_{j},Q_{j}\right)  \right)  _{j\in I}\\
&  =\left(  \left(  P_{i},Q_{i}\right)  \right)  _{i\in I}
\ \ \ \ \ \ \ \ \ \ \left(  \text{here, we have renamed the index }j\text{ as
}i\right) \\
&  =\alpha.
\end{align*}
\par
But
\begin{align*}
\left(  \operatorname*{Split}\circ\operatorname*{Combine}\right)  \left(
\alpha\right)   &  =\operatorname*{Split}\left(
\underbrace{\operatorname*{Combine}\left(  \alpha\right)  }_{=\left(
\bigsqcup_{i\in I}P_{i},\bigsqcup_{i\in I}Q_{i}\right)  }\right)
=\operatorname*{Split}\left(  \left(  \bigsqcup_{i\in I}P_{i},\bigsqcup_{i\in
I}Q_{i}\right)  \right)  =\alpha=\operatorname*{id}\left(  \alpha\right)  .
\end{align*}
\par
Now, forget that we fixed $\alpha$. We thus have shown that $\left(
\operatorname*{Split}\circ\operatorname*{Combine}\right)  \left(
\alpha\right)  =\operatorname*{id}\left(  \alpha\right)  $ for each $\alpha
\in\prod_{i\in I}\operatorname*{Adm}\left(  \mathbf{E}_{i}\right)  $. In other
words, $\operatorname*{Split}\circ\operatorname*{Combine}=\operatorname*{id}%
$.} and $\operatorname*{Combine}\circ\operatorname*{Split}=\operatorname*{id}%
$\ \ \ \ \footnote{\textit{Proof.} Let $\alpha\in\operatorname*{Adm}%
\mathbf{E}$. Recall that
\[
\mathbf{E}=\bigsqcup_{i\in I}\mathbf{E}_{i}=\left(  \bigsqcup_{i\in I}%
E_{i},\bigoplus_{i\in I}\left(  <_{1,i}\right)  ,\bigoplus_{i\in I}\left(
<_{2,i}\right)  \right)
\]
(by the definition of the double poset $\bigsqcup_{i\in I}\mathbf{E}_{i}$).
\par
The element $\alpha\in\operatorname*{Adm}\mathbf{E}$ has the form
$\alpha=\left(  P,Q\right)  $ for two subsets $P$ and $Q$ of $\bigsqcup_{i\in
I}E_{i}$ (by the definition of $\operatorname*{Adm}\mathbf{E}$ (since
$\mathbf{E}=\left(  \bigsqcup_{i\in I}E_{i},\bigoplus_{i\in I}\left(
<_{1,i}\right)  ,\bigoplus_{i\in I}\left(  <_{2,i}\right)  \right)  $)).
Consider these $P$ and $Q$.
\par
Now, $\alpha=\left(  P,Q\right)  $. Applying the map $\operatorname*{Split}$
to both sides of this equality, we find%
\[
\operatorname*{Split}\left(  \alpha\right)  =\operatorname*{Split}\left(
\left(  P,Q\right)  \right)  =\left(  \left(  \left(  \operatorname*{inc}%
\nolimits_{i}\right)  ^{-1}\left(  P\right)  ,\left(  \operatorname*{inc}%
\nolimits_{i}\right)  ^{-1}\left(  Q\right)  \right)  \right)  _{i\in I}%
\]
(by the definition of $\operatorname*{Split}$). Hence, $\left(  \left(
\left(  \operatorname*{inc}\nolimits_{i}\right)  ^{-1}\left(  P\right)
,\left(  \operatorname*{inc}\nolimits_{i}\right)  ^{-1}\left(  Q\right)
\right)  \right)  _{i\in I}=\operatorname*{Split}\left(  \alpha\right)
\in\prod_{i\in I}\operatorname*{Adm}\left(  \mathbf{E}_{i}\right)  $. Thus,
the definition of the map $\operatorname*{Combine}$ yields
\begin{align*}
\operatorname*{Combine}\left(  \left(  \left(  \left(  \operatorname*{inc}%
\nolimits_{i}\right)  ^{-1}\left(  P\right)  ,\left(  \operatorname*{inc}%
\nolimits_{i}\right)  ^{-1}\left(  Q\right)  \right)  \right)  _{i\in
I}\right)   &  =\left(  \underbrace{\bigsqcup_{i\in I}\left(
\operatorname*{inc}\nolimits_{i}\right)  ^{-1}\left(  P\right)  }%
_{\substack{=P\\\text{(by Lemma \ref{lem.djun.Adm-bij.1} (a)}\\\text{(applied
to }R=P\text{))}}},\underbrace{\bigsqcup_{i\in I}\left(  \operatorname*{inc}%
\nolimits_{i}\right)  ^{-1}\left(  Q\right)  }_{\substack{=Q\\\text{(by Lemma
\ref{lem.djun.Adm-bij.1} (a)}\\\text{(applied to }R=Q\text{))}}}\right) \\
&  =\left(  P,Q\right)  =\alpha.
\end{align*}
\par
But
\begin{align*}
\left(  \operatorname*{Combine}\circ\operatorname*{Split}\right)  \left(
\alpha\right)   &  =\operatorname*{Combine}\left(
\underbrace{\operatorname*{Split}\left(  \alpha\right)  }_{=\left(  \left(
\left(  \operatorname*{inc}\nolimits_{i}\right)  ^{-1}\left(  P\right)
,\left(  \operatorname*{inc}\nolimits_{i}\right)  ^{-1}\left(  Q\right)
\right)  \right)  _{i\in I}}\right) \\
&  =\operatorname*{Combine}\left(  \left(  \left(  \left(  \operatorname*{inc}%
\nolimits_{i}\right)  ^{-1}\left(  P\right)  ,\left(  \operatorname*{inc}%
\nolimits_{i}\right)  ^{-1}\left(  Q\right)  \right)  \right)  _{i\in
I}\right)  =\alpha=\operatorname*{id}\left(  \alpha\right)  .
\end{align*}
\par
Now, forget that we fixed $\alpha$. We thus have shown that $\left(
\operatorname*{Combine}\circ\operatorname*{Split}\right)  \left(
\alpha\right)  =\operatorname*{id}\left(  \alpha\right)  $ for each $\alpha
\in\operatorname*{Adm}\mathbf{E}$. In other words, $\operatorname*{Combine}%
\circ\operatorname*{Split}=\operatorname*{id}$.}. Hence, the maps
$\operatorname*{Split}$ and $\operatorname*{Combine}$ are mutually inverse.
Thus, these maps $\operatorname*{Split}$ and $\operatorname*{Combine}$ are
mutually inverse bijections. This proves Proposition \ref{prop.djun.Adm-bij}.
\end{proof}

\begin{lemma}
\label{lem.djun.restrict.2lem}Let $I$ be a finite set. For each $i\in I$, let
$\mathbf{E}_{i}=\left(  E_{i},<_{1,i},<_{2,i}\right)  $ be a double poset.
Let $w:\bigsqcup_{i \in I} E_{i}\rightarrow\left\{  1,2,3,\ldots\right\}  $ be
a map. For each $i\in I$, let $T_{i}$ be a subset of $E_{i}$.

Let $\mathbf{E}$ be the double poset $\bigsqcup_{i\in I}\mathbf{E}_{i}$. Then,%
\[
\Gamma\left(  {\mathbf{E}}\mid_{\bigsqcup_{i\in I}T_{i}},w\mid_{\bigsqcup
_{i\in I}T_{i}}\right)  =\prod_{i\in I}\Gamma\left(  \mathbf{E}_{i}\mid
_{T_{i}},\left(  w\circ\operatorname*{inc}\nolimits_{i}\right)  \mid_{T_{i}%
}\right)  .
\]

\end{lemma}

\begin{proof}
[Proof of Lemma \ref{lem.djun.restrict.2lem}.]We notice that the notation
$\operatorname*{inc}\nolimits_{j}$ (for $j\in I$) is slightly ambiguous: It
may mean both the map $\operatorname*{inc}\nolimits_{j}:E_{j}\rightarrow
\bigsqcup_{i\in I}E_{i}$ and the map $\operatorname*{inc}\nolimits_{j}%
:T_{j}\rightarrow\bigsqcup_{i\in I}T_{i}$. In order to resolve this ambiguity,
let us agree to denote the latter map by $\operatorname*{inc}\nolimits_{j,T}$
(instead of just calling it $\operatorname*{inc}\nolimits_{j}$). Thus,
$\operatorname*{inc}\nolimits_{j}$ shall only mean the map
$\operatorname*{inc}\nolimits_{j}:E_{j}\rightarrow\bigsqcup_{i\in I}E_{i}$.

Set $T=\bigsqcup_{i\in I}T_{i}$. This is a subset of $\bigsqcup_{i\in I} E_i$.

Each $j\in I$ satisfies%
\begin{equation}
\left(  w\mid_{T}\right)  \circ\operatorname*{inc}\nolimits_{j,T}=\left(
w\circ\operatorname*{inc}\nolimits_{j}\right)  \mid_{T_{j}}
\label{pf.lem.djun.restrict.2lem.1}%
\end{equation}
\footnote{\textit{Proof of (\ref{pf.lem.djun.restrict.2lem.1}):} Let $j\in I$.
\par
Notice that the map $\left(  w\mid_{T}\right)  \circ\operatorname*{inc}%
\nolimits_{j,T}$ is well-defined, since $\operatorname*{inc}\nolimits_{j,T}$
is a map from $T_{j}$ to $\bigsqcup_{i\in I}T_{i}=T$.
\par
Let $g\in T_{j}$. Then, $\operatorname*{inc}\nolimits_{j}\left(  g\right)
=\left(  j,g\right)  $ (by the definition of the map $\operatorname*{inc}%
\nolimits_{j}$). Applying the map $w$ to both sides of this equality, we
obtain $w \left(  \operatorname{inc}_{j} \left(  g\right)  \right)  = w
\left(  \left(  j, g\right)  \right)  $. On the other hand,
$\operatorname*{inc}\nolimits_{j,T}\left(  g\right)  =\left(  j,g\right)  $
(by the definition of the map $\operatorname*{inc}\nolimits_{j,T}$). Applying
the map $w$ to both sides of this equality, we obtain $w \left(
\operatorname{inc}_{j,T} \left(  g\right)  \right)  = w \left(  \left(  j,
g\right)  \right)  $. Now,
\[
\left(  \left(  w\circ\operatorname*{inc}\nolimits_{j}\right)  \mid_{T_{j}%
}\right)  \left(  g\right)  =\left(  w\circ\operatorname*{inc}\nolimits_{j}%
\right)  \left(  g\right)  =w\left(  \operatorname*{inc}\nolimits_{j}\left(
g\right)  \right)  =w\left(  \left(  j,g\right)  \right)  .
\]
Comparing this with%
\[
\left(  \left(  w\mid_{T}\right)  \circ\operatorname*{inc}\nolimits_{j,T}%
\right)  \left(  g\right)  =\left(  w\mid_{T}\right)  \left(
\operatorname*{inc}\nolimits_{j,T}\left(  g\right)  \right)  =w\left(
\operatorname*{inc}\nolimits_{j,T}\left(  g\right)  \right)  =w\left(  \left(
j,g\right)  \right)  ,
\]
we obtain $\left(  \left(  w\circ\operatorname*{inc}\nolimits_{j}\right)
\mid_{T_{j}}\right)  \left(  g\right)  =\left(  \left(  w\mid_{T}\right)
\circ\operatorname*{inc}\nolimits_{j,T}\right)  \left(  g\right)  $.
\par
Now, forget that we fixed $g$. We thus have proven that $\left(  \left(
w\circ\operatorname*{inc}\nolimits_{j}\right)  \mid_{T_{j}}\right)  \left(
g\right)  =\left(  \left(  w\mid_{T}\right)  \circ\operatorname*{inc}%
\nolimits_{j,T}\right)  \left(  g\right)  $ for each $g\in T_{j}$. In other
words, $\left(  w\circ\operatorname*{inc}\nolimits_{j}\right)  \mid_{T_{j}%
}=\left(  w\mid_{T}\right)  \circ\operatorname*{inc}\nolimits_{j,T}$. Thus,
the proof of (\ref{pf.lem.djun.restrict.2lem.1}) is complete.}.

But for each $i\in I$, we have $\mathbf{E}_{i}=\left(  E_{i},<_{1,i}%
,<_{2,i}\right)  $ and therefore $\mathbf{E}_{i}\mid_{T_{i}}=\left(
T_{i},<_{1,i},<_{2,i}\right)  $ (by the definition of $\mathbf{E}_{i}%
\mid_{T_{i}}$). Furthermore, $w\mid_{T}$ is a map from $\bigsqcup_{i\in
I}T_{i}$ to $\left\{  1,2,3,\ldots\right\}  $ (because the domain of
$w\mid_{T}$ is $T=\bigsqcup_{i\in I}T_{i}$). Now,%
\begin{align*}
&  \Gamma\left(  \underbrace{{\mathbf{E}}\mid_{\bigsqcup_{i\in I}T_{i}}%
}_{\substack{=\bigsqcup_{i\in I}\left(  \mathbf{E}_{i}\mid_{T_{i}}\right)
\\\text{(by Proposition \ref{prop.djun.restrict})}}},\underbrace{w\mid
_{\bigsqcup_{i\in I}T_{i}}}_{\substack{=w\mid_{T}\\\text{(since }%
\bigsqcup_{i\in I}T_{i}=T\text{)}}}\right) \\
&  =\Gamma\left(  \bigsqcup_{i\in I}\left(  \mathbf{E}_{i}\mid_{T_{i}}\right)
,w\mid_{T}\right)  =\prod_{i\in I}\Gamma\left(  \mathbf{E}_{i}\mid_{T_{i}%
},\underbrace{\left(  w\mid_{T}\right)  \circ\operatorname*{inc}%
\nolimits_{i,T}}_{\substack{=\left(  w\circ\operatorname*{inc}\nolimits_{i}%
\right)  \mid_{T_{i}}\\\text{(by (\ref{pf.lem.djun.restrict.2lem.1}%
)}\\\text{(applied to }j=i\text{))}}}\right) \\
&  \ \ \ \ \ \ \ \ \ \ \left(
\begin{array}
[c]{c}%
\text{by Proposition \ref{prop.djun.Gamma} (applied to }\mathbf{E}_{i}%
\mid_{T_{i}}\text{, }T_{i}\text{, }<_{1,i}\text{, }<_{2,i}\text{ and }%
w\mid_{T}\\
\text{instead of }\mathbf{E}_{i}\text{, }E_{i}\text{, }<_{1,i}\text{, }%
<_{2,i}\text{ and }w \text{)}%
\end{array}
\right) \\
&  =\prod_{i\in I}\Gamma\left(  \mathbf{E}_{i}\mid_{T_{i}},\left(
w\circ\operatorname*{inc}\nolimits_{i}\right)  \mid_{T_{i}}\right)  .
\end{align*}
This proves Lemma \ref{lem.djun.restrict.2lem}.
\end{proof}

\Needspace{3cm}

\begin{corollary}
\label{cor.djun.restrict.2}Let $I$ be a finite set. For each $i\in I$, let
$\mathbf{E}_{i}=\left(  E_{i},<_{1,i},<_{2,i}\right)  $ be a double poset. For
each $i\in I$, let $w_{i}:E_{i}\rightarrow\left\{  1,2,3,\ldots\right\}  $ be
a map. Then,%
\[
\Delta\left(  \prod_{i\in I}\Gamma\left(  \mathbf{E}_{i},w_{i}\right)
\right)  =\prod_{i\in I}\Delta\left(  \Gamma\left(  \mathbf{E}_{i}%
,w_{i}\right)  \right)  .
\]

\end{corollary}

\begin{proof}
[Proof of Corollary \ref{cor.djun.restrict.2}.]Let $\mathbf{E}$ be the double
poset $\bigsqcup_{i\in I}\mathbf{E}_{i}$. Then,%
\[
\mathbf{E}=\bigsqcup_{i\in I}\mathbf{E}_{i}=\left(  \bigsqcup_{i\in I}%
E_{i},\bigoplus_{i\in I}\left(  <_{1,i}\right)  ,\bigoplus_{i\in I}\left(
<_{2,i}\right)  \right)
\]
(by the definition of $\bigsqcup_{i\in I}\mathbf{E}_{i}$).

Proposition \ref{prop.djun.Adm-bij} shows that the maps $\operatorname*{Split}%
$ and $\operatorname*{Combine}$ are mutually inverse bijections. In
particular, the map $\operatorname*{Combine}$ is a bijection. In other words,
the map%
\begin{equation}
\prod_{i\in I}\operatorname*{Adm}\left(  \mathbf{E}_{i}\right)  \rightarrow
\operatorname*{Adm}\mathbf{E},\ \ \ \ \ \ \ \ \ \ \left(  \left(  P_{i}%
,Q_{i}\right)  \right)  _{i\in I}\mapsto\left(  \bigsqcup_{i\in I}%
P_{i},\bigsqcup_{i\in I}Q_{i}\right)  \label{pf.cor.djun.restrict.2.combine}%
\end{equation}
is a bijection (since this map is the map $\operatorname*{Combine}%
$\ \ \ \ \footnote{because $\operatorname*{Combine}\left(  \left(  \left(
P_{i},Q_{i}\right)  \right)  _{i\in I}\right)  =\left(  \bigsqcup_{i\in
I}P_{i},\bigsqcup_{i\in I}Q_{i}\right)  $ for every $\left(  \left(
P_{i},Q_{i}\right)  \right)  _{i\in I}\in\prod_{i\in I}\operatorname*{Adm}%
\left(  \mathbf{E}_{i}\right)  $}).

Let $F=\left\{  1,2,3,\ldots\right\}  $. For each $i\in I$, we have $w_{i}%
\in\left\{  1,2,3,\ldots\right\}  ^{E_{i}}=F^{E_{i}}$ (since $\left\{
1,2,3,\ldots\right\}  =F$). Thus, $\left(  w_{i}\right)  _{i\in I}\in
\prod_{i\in I}F^{E_{i}}$.

Corollary \ref{cor.djun.restr} shows that the map $\operatorname*{Restr}%
:F^{\bigsqcup_{i\in I}E_{i}}\rightarrow\prod_{i\in I}F^{E_{i}}$ is a
bijection. Hence, $\operatorname*{Restr}\nolimits^{-1}\left(  \left(
w_{i}\right)  _{i\in I}\right)  $ is a well-defined element of $F^{\bigsqcup
_{i\in I}E_{i}}$. Denote this element by $w$. Then, Corollary
\ref{cor.djun.Gamma.2} shows that%
\[
\prod_{i\in I}\Gamma\left(  \mathbf{E}_{i},w_{i}\right)  =\Gamma\left(
\underbrace{\bigsqcup_{i\in I}\mathbf{E}_{i}}_{=\mathbf{E}},w\right)
=\Gamma\left(  \mathbf{E},w\right)  .
\]
Applying the map $\Delta$ to both sides of this equality, we obtain%
\begin{align}
&  \Delta\left(  \prod_{i\in I}\Gamma\left(  \mathbf{E}_{i},w_{i}\right)
\right) \nonumber\\
&  =\Delta\left(  \Gamma\left(  \mathbf{E},w\right)  \right) \nonumber\\
&  =\sum_{\left(  P,Q\right)  \in\operatorname{Adm}\mathbf{E}}\Gamma\left(
{\mathbf{E}}\mid_{P},w\mid_{P}\right)  \otimes\Gamma\left(  {\mathbf{E}}%
\mid_{Q},w\mid_{Q}\right)  \ \ \ \ \ \ \ \ \ \ \left(  \text{by Proposition
\ref{prop.Gammaw.coprod}}\right) \nonumber\\
&  =\sum_{\left(  \left(  P_{i},Q_{i}\right)  \right)  _{i\in I}\in\prod_{i\in
I}\operatorname*{Adm}\left(  \mathbf{E}_{i}\right)  }\Gamma\left(
{\mathbf{E}}\mid_{\bigsqcup_{i\in I}P_{i}},w\mid_{\bigsqcup_{i\in I}P_{i}%
}\right)  \otimes\Gamma\left(  {\mathbf{E}}\mid_{\bigsqcup_{i\in I}Q_{i}%
},w\mid_{\bigsqcup_{i\in I}Q_{i}}\right) \label{pf.cor.djun.restrict.2.3}\\
&  \ \ \ \ \ \ \ \ \ \ \left(
\begin{array}
[c]{c}%
\text{here, we have substituted }\left(  \bigsqcup_{i\in I}P_{i}%
,\bigsqcup_{i\in I}Q_{i}\right)  \text{ for }\left(  P,Q\right)  \text{ in the
sum,}\\
\text{since the map (\ref{pf.cor.djun.restrict.2.combine}) is a bijection}%
\end{array}
\right)  .\nonumber
\end{align}

We have $w=\operatorname*{Restr}\nolimits^{-1}\left(  \left(  w_{i}\right)
_{i\in I}\right)  $ (by the definition of $w$) and thus%
\[
\left(  w_{i}\right)  _{i\in I}=\operatorname*{Restr}\left(  w\right)
=\left(  w\circ\operatorname*{inc}\nolimits_{i}\right)  _{i\in I}%
\ \ \ \ \ \ \ \ \ \ \left(  \text{by the definition of }\operatorname*{Restr}%
\right)  .
\]
In other words,%
\begin{equation}
w_{i}=w\circ\operatorname*{inc}\nolimits_{i}\ \ \ \ \ \ \ \ \ \ \text{for each
}i\in I. \label{pf.cor.djun.restrict.2.wi=}%
\end{equation}

But the following holds:

\begin{statement}
\textit{Observation 1:} Let $\left(  \left(  P_{i},Q_{i}\right)  \right)
_{i\in I}\in\prod_{i\in I}\operatorname*{Adm}\left(  \mathbf{E}_{i}\right)  $.
Then,%
\begin{equation}
\Gamma\left(  {\mathbf{E}}\mid_{\bigsqcup_{i\in I}P_{i}},w\mid_{\bigsqcup
_{i\in I}P_{i}}\right)  =\prod_{i\in I}\Gamma\left(  \mathbf{E}_{i}\mid
_{P_{i}},w_{i}\mid_{P_{i}}\right)  \label{pf.cor.djun.restrict.2.factor1}%
\end{equation}
and%
\begin{equation}
\Gamma\left(  {\mathbf{E}}\mid_{\bigsqcup_{i\in I}Q_{i}},w\mid_{\bigsqcup
_{i\in I}Q_{i}}\right)  =\prod_{i\in I}\Gamma\left(  \mathbf{E}_{i}\mid
_{Q_{i}},w_{i}\mid_{Q_{i}}\right)  \label{pf.cor.djun.restrict.2.factor2}%
\end{equation}

\end{statement}

[\textit{Proof of Observation 1:} Let $\left(  \left(  P_{i},Q_{i}\right)
\right)  _{i\in I}\in\prod_{i\in I}\operatorname*{Adm}\left(  \mathbf{E}%
_{i}\right)  $. Thus, for each $i\in I$, we have $\left(  P_{i},Q_{i}\right)
\in\operatorname*{Adm}\left(  \mathbf{E}_{i}\right)  $. Hence, for each $i\in
I$, the sets $P_{i}$ and $Q_{i}$ are two subsets of $E_{i}$ (by the definition
of $\operatorname*{Adm}\left(  \mathbf{E}_{i}\right)  $ (since $\mathbf{E}%
_{i}=\left(  E_{i},<_{1,i},<_{2,i}\right)  $)). Hence, Lemma
\ref{lem.djun.restrict.2lem} (applied to $T_{i}=P_{i}$) yields%
\[
\Gamma\left(  {\mathbf{E}}\mid_{\bigsqcup_{i\in I}P_{i}},w\mid_{\bigsqcup
_{i\in I}P_{i}}\right)
=\prod_{i\in I}\Gamma\left(  \mathbf{E}_{i}\mid
_{P_{i}},\left(\underbrace{w\circ\operatorname{inc}_{i}}_{\substack{
=w_i \\ \text{(by (\ref{pf.cor.djun.restrict.2.wi=}))}}}\right)\mid_{P_{i}}\right)
=\prod_{i\in I}\Gamma\left(  \mathbf{E}_{i}\mid
_{P_{i}},w_{i}\mid_{P_{i}}\right)  .
\]
Furthermore, Lemma \ref{lem.djun.restrict.2lem} (applied to $T_{i}=Q_{i}$)
yields%
\[
\Gamma\left(  {\mathbf{E}}\mid_{\bigsqcup_{i\in I}Q_{i}},w\mid_{\bigsqcup
_{i\in I}Q_{i}}\right)
=\prod_{i\in I}\Gamma\left(  \mathbf{E}_{i}\mid
_{Q_{i}},\left(\underbrace{w\circ\operatorname{inc}_{i}}_{\substack{
=w_i \\ \text{(by (\ref{pf.cor.djun.restrict.2.wi=}))}}}\right)\mid_{Q_{i}}\right)
=\prod_{i\in I}\Gamma\left(  \mathbf{E}_{i}\mid
_{Q_{i}},w_{i}\mid_{Q_{i}}\right)  .
\]
This completes the proof of Observation 1.]

Now, (\ref{pf.cor.djun.restrict.2.3}) rewrites as follows:%
\begin{align}
&  \Delta\left(  \prod_{i\in I}\Gamma\left(  \mathbf{E}_{i},w_{i}\right)
\right) \nonumber\\
&  =\sum_{\left(  \left(  P_{i},Q_{i}\right)  \right)  _{i\in I}\in\prod_{i\in
I}\operatorname*{Adm}\left(  \mathbf{E}_{i}\right)  }\underbrace{\Gamma\left(
{\mathbf{E}}\mid_{\bigsqcup_{i\in I}P_{i}},w\mid_{\bigsqcup_{i\in I}P_{i}%
}\right)  }_{\substack{=\prod_{i\in I}\Gamma\left(  \mathbf{E}_{i}\mid_{P_{i}%
},w_{i}\mid_{P_{i}}\right)  \\\text{(by (\ref{pf.cor.djun.restrict.2.factor1}%
))}}}\otimes\underbrace{\Gamma\left(  {\mathbf{E}}\mid_{\bigsqcup_{i\in
I}Q_{i}},w\mid_{\bigsqcup_{i\in I}Q_{i}}\right)  }_{\substack{=\prod_{i\in
I}\Gamma\left(  \mathbf{E}_{i}\mid_{Q_{i}},w_{i}\mid_{Q_{i}}\right)
\\\text{(by (\ref{pf.cor.djun.restrict.2.factor2}))}}}\nonumber\\
&  =\sum_{\left(  \left(  P_{i},Q_{i}\right)  \right)  _{i\in I}\in\prod_{i\in
I}\operatorname*{Adm}\left(  \mathbf{E}_{i}\right)  }\left(  \prod_{i\in
I}\Gamma\left(  \mathbf{E}_{i}\mid_{P_{i}},w_{i}\mid_{P_{i}}\right)  \right)
\otimes\left(  \prod_{i\in I}\Gamma\left(  \mathbf{E}_{i}\mid_{Q_{i}}%
,w_{i}\mid_{Q_{i}}\right)  \right)  . \label{pf.cor.djun.restrict.2.4}%
\end{align}

On the other hand, for each $i\in I$, we have%
\begin{equation}
\Delta\left(  \Gamma\left(  \mathbf{E}_{i},w_{i}\right)  \right)
=\sum_{\left(  P,Q\right)  \in\operatorname{Adm}\left(  \mathbf{E}_{i}\right)
}\Gamma\left(  \mathbf{E}_{i}\mid_{P},w_{i}\mid_{P}\right)  \otimes
\Gamma\left(  \mathbf{E}_{i}\mid_{Q},w_{i}\mid_{Q}\right)
\label{pf.cor.djun.restrict.2.5}%
\end{equation}
\footnote{\textit{Proof of (\ref{pf.cor.djun.restrict.2.5}):} Let $i\in I$.
Recall that $\mathbf{E}_{i}=\left(  E_{i},<_{1,i},<_{2,i}\right)  $ is a
double poset. Hence, Proposition \ref{prop.Gammaw.coprod} (applied to
$\mathbf{E}_{i}$, $<_{1,i}$, $<_{2,i}$ and $w_{i}$ instead of $\mathbf{E}$,
$<_{1}$, $<_{2}$ and $w$) yields%
\[
\Delta\left(  \Gamma\left(  \mathbf{E}_{i},w_{i}\right)  \right)
=\sum_{\left(  P,Q\right)  \in\operatorname{Adm}\left(  \mathbf{E}_{i}\right)
}\Gamma\left(  \mathbf{E}_{i}\mid_{P},w_{i}\mid_{P}\right)  \otimes
\Gamma\left(  \mathbf{E}_{i}\mid_{Q},w_{i}\mid_{Q}\right)  .
\]
This proves (\ref{pf.cor.djun.restrict.2.5}).}. Hence,%
\begin{align*}
&  \prod_{i\in I}\underbrace{\Delta\left(  \Gamma\left(  \mathbf{E}_{i}%
,w_{i}\right)  \right)  }_{\substack{=\sum_{\left(  P,Q\right)  \in
\operatorname{Adm}\left(  \mathbf{E}_{i}\right)  }\Gamma\left(  \mathbf{E}%
_{i}\mid_{P},w_{i}\mid_{P}\right)  \otimes\Gamma\left(  \mathbf{E}_{i}\mid
_{Q},w_{i}\mid_{Q}\right)  \\\text{(by (\ref{pf.cor.djun.restrict.2.5}))}}}\\
&  =\prod_{i\in I}\sum_{\left(  P,Q\right)  \in\operatorname{Adm}\left(
\mathbf{E}_{i}\right)  }\Gamma\left(  \mathbf{E}_{i}\mid_{P},w_{i}\mid
_{P}\right)  \otimes\Gamma\left(  \mathbf{E}_{i}\mid_{Q},w_{i}\mid_{Q}\right)
\\
&  =\sum_{\left(  \left(  P_{i},Q_{i}\right)  \right)  _{i\in I}\in\prod_{i\in
I}\operatorname*{Adm}\left(  \mathbf{E}_{i}\right)  }\underbrace{\prod_{i\in
I}\left(  \Gamma\left(  \mathbf{E}_{i}\mid_{P_{i}},w_{i}\mid_{P_{i}}\right)
\otimes\Gamma\left(  \mathbf{E}_{i}\mid_{Q_{i}},w_{i}\mid_{Q_{i}}\right)
\right)  }_{\substack{=\left(  \prod_{i\in I}\Gamma\left(  \mathbf{E}_{i}%
\mid_{P_{i}},w_{i}\mid_{P_{i}}\right)  \right)  \otimes\left(  \prod_{i\in
I}\Gamma\left(  \mathbf{E}_{i}\mid_{Q_{i}},w_{i}\mid_{Q_{i}}\right)  \right)
\\\text{(by the definition of the multiplication on }\operatorname*{QSym}%
\otimes\operatorname*{QSym}\text{)}}}\\
&  \ \ \ \ \ \ \ \ \ \ \left(  \text{by the product rule}\right) \\
&  =\sum_{\left(  \left(  P_{i},Q_{i}\right)  \right)  _{i\in I}\in\prod_{i\in
I}\operatorname*{Adm}\left(  \mathbf{E}_{i}\right)  }\left(  \prod_{i\in
I}\Gamma\left(  \mathbf{E}_{i}\mid_{P_{i}},w_{i}\mid_{P_{i}}\right)  \right)
\otimes\left(  \prod_{i\in I}\Gamma\left(  \mathbf{E}_{i}\mid_{Q_{i}}%
,w_{i}\mid_{Q_{i}}\right)  \right) \\
&  =\Delta\left(  \prod_{i\in I}\Gamma\left(  \mathbf{E}_{i},w_{i}\right)
\right)  \ \ \ \ \ \ \ \ \ \ \left(  \text{by (\ref{pf.cor.djun.restrict.2.4}%
)}\right)  .
\end{align*}
This proves Corollary \ref{cor.djun.restrict.2}.
\end{proof}

\begin{corollary}
\label{cor.djun.restrict.3}Let $a\in\operatorname*{QSym}$ and $b\in
\operatorname*{QSym}$. Then, $\Delta\left(  ab\right)  =\Delta\left(
a\right)  \Delta\left(  b\right)  $.
\end{corollary}

\begin{proof}
[Proof of Corollary \ref{cor.djun.restrict.3}.]It is well-known that
$\operatorname*{QSym}$ is a $\mathbf{k}$-submodule of $\mathbf{k}\left[
\left[  x_{1},x_{2},x_{3},\ldots\right]  \right]  $, and that $\left(
M_{\alpha}\right)  _{\alpha\in{\operatorname{Comp}}}$ is a basis of this
$\mathbf{k}$-module $\operatorname*{QSym}$.

We shall now prove that
\begin{equation}
\Delta\left(  M_{\alpha}M_{\beta}\right)  =\Delta\left(  M_{\alpha}\right)
\Delta\left(  M_{\beta}\right)  \ \ \ \ \ \ \ \ \ \ \text{for any }\alpha
\in\operatorname*{Comp}\text{ and }\beta\in\operatorname*{Comp}.
\label{pf.cor.djun.restrict.3.1}%
\end{equation}

[\textit{Proof of (\ref{pf.cor.djun.restrict.3.1}):} Let $\alpha
\in\operatorname*{Comp}$ and $\beta\in\operatorname*{Comp}$.

As in the proof of (\ref{pf.prop.QSym.alg.prod.1}), we can

\begin{itemize}
\item find a set $E_{1}$, a special double poset $\mathbf{E}_{1}=\left(
E_{1},<_{1,1},>_{1,1}\right)  $, and a map $w_{1}:E_{1}\rightarrow\left\{
1,2,3,\ldots\right\}  $ satisfying $\Gamma\left(  {\mathbf{E}}_{1}%
,w_{1}\right)  =M_{\beta}$;

\item find a set $E_{0}$, a special double poset $\mathbf{E}_{0}=\left(
E_{0},<_{1,0},>_{1,0}\right)  $, and a map $w_{0}:E_{0}\rightarrow\left\{
1,2,3,\ldots\right\}  $ satisfying $\Gamma\left(  {\mathbf{E}}_{0}%
,w_{0}\right)  =M_{\alpha}$;

\item show that $\mathbf{E}_{i}=\left(  E_{i},<_{1,i},>_{1,i}\right)  $ for
every $i\in\left\{  0,1\right\}  $;

\item show that $w_{i}$ is a map $E_{i}\rightarrow\left\{  1,2,3,\ldots
\right\}  $ for each $i\in\left\{  0,1\right\}  $.
\end{itemize}

Now, Corollary \ref{cor.djun.restrict.2} (applied to $\left\{  0,1\right\}  $
and $>_{1,i}$ instead of $I$ and $<_{2,i}$) shows that
\[
\Delta\left(  \prod_{i\in\left\{  0,1\right\}  }\Gamma\left(  \mathbf{E}%
_{i},w_{i}\right)  \right)  =\prod_{i\in\left\{  0,1\right\}  }\Delta\left(
\Gamma\left(  \mathbf{E}_{i},w_{i}\right)  \right)  .
\]
Since
\[
\prod_{i\in\left\{  0,1\right\}  }\Gamma\left(  \mathbf{E}_{i},w_{i}\right)
=\underbrace{\Gamma\left(  {\mathbf{E}}_{0},w_{0}\right)  }_{=M_{\alpha}%
}\underbrace{\Gamma\left(  {\mathbf{E}}_{1},w_{1}\right)  }_{=M_{\beta}%
}=M_{\alpha}M_{\beta},
\]
this rewrites as $\Delta\left(  M_{\alpha}M_{\beta}\right)  =\prod
_{i\in\left\{  0,1\right\}  }\Delta\left(  \Gamma\left(  \mathbf{E}_{i}%
,w_{i}\right)  \right)  $. Hence,%
\[
\Delta\left(  M_{\alpha}M_{\beta}\right)  =\prod_{i\in\left\{  0,1\right\}
}\Delta\left(  \Gamma\left(  \mathbf{E}_{i},w_{i}\right)  \right)
=\Delta\left(  \underbrace{\Gamma\left(  {\mathbf{E}}_{0},w_{0}\right)
}_{=M_{\alpha}}\right)  \Delta\left(  \underbrace{\Gamma\left(  {\mathbf{E}%
}_{1},w_{1}\right)  }_{=M_{\beta}}\right)  =\Delta\left(  M_{\alpha}\right)
\Delta\left(  M_{\beta}\right)  .
\]
Thus, (\ref{pf.cor.djun.restrict.3.1}) is proven.]

Now, we can see that%
\begin{equation}
\Delta\left(  ab\right)  =\Delta\left(  a\right)  \Delta\left(  b\right)
\ \ \ \ \ \ \ \ \ \ \text{for any }a\in\operatorname*{QSym}\text{ and }%
b\in\operatorname*{QSym}. \label{pf.cor.djun.restrict.3.3}%
\end{equation}

[\textit{Proof of (\ref{pf.cor.djun.restrict.3.3}):} Let $a\in
\operatorname*{QSym}$ and $b\in\operatorname*{QSym}$. We must prove the
relation $\Delta\left(  ab\right)  =\Delta\left(  a\right)  \Delta\left(
b\right)  $.

This relation is $\mathbf{k}$-linear in $b$ (since $\Delta$ is a $\mathbf{k}%
$-linear map). Hence, we can WLOG assume that $b$ belongs to the basis
$\left(  M_{\alpha}\right)  _{\alpha\in\operatorname*{Comp}}$ of the
$\mathbf{k}$-module $\operatorname*{QSym}$. Assume this. Thus, $b=M_{\beta}$
for some $\beta\in\operatorname*{Comp}$. Consider this $\beta$.

We must prove the relation $\Delta\left(  ab\right)  =\Delta\left(  a\right)
\Delta\left(  b\right)  $. This relation is $\mathbf{k}$-linear in $a$ (since
$\Delta$ is a $\mathbf{k}$-linear map). Hence, we can WLOG assume that $a$
belongs to the basis $\left(  M_{\alpha}\right)  _{\alpha\in
\operatorname*{Comp}}$ of the $\mathbf{k}$-module $\operatorname*{QSym}$.
Assume this. Thus, $a=M_{\alpha}$ for some $\alpha\in\operatorname*{Comp}$.
Consider this $\alpha$.

Now, \eqref{pf.cor.djun.restrict.3.1} yields $\Delta\left(  M_{\alpha}%
M_{\beta}\right)  =\Delta\left(  M_{\alpha}\right)  \Delta\left(  M_{\beta
}\right)  $. Since $a = M_{\alpha}$ and $b = M_{\beta}$, this rewrites as
$\Delta\left(  ab \right)  = \Delta\left(  a \right)  \Delta\left(  b \right)
$. This proves (\ref{pf.cor.djun.restrict.3.3}).]

Corollary \ref{cor.djun.restrict.3} immediately follows from
(\ref{pf.cor.djun.restrict.3.3}).
\end{proof}

Now, we can prove that $\operatorname*{QSym}$ is a bialgebra:

\begin{proposition}
\label{prop.QSym.bialg}The $\mathbf{k}$-algebra $\operatorname*{QSym}$,
equipped with the comultiplication $\Delta$ and the counit $\varepsilon$,
becomes a $\mathbf{k}$-bialgebra.
\end{proposition}

\begin{proof}
[Proof of Proposition \ref{prop.QSym.bialg}.]Let $\varepsilon^{\prime}$ be the
map%
\[
\mathbf{k}\left[  \left[  x_{1},x_{2},x_{3},\ldots\right]  \right]
\rightarrow\mathbf{k},\ \ \ \ \ \ \ \ \ \ f\mapsto f\left(  0,0,0,\ldots
\right)  .
\]
Then, $\varepsilon^{\prime}$ is a substitution homomorphism (since it
substitutes $\left(  0,0,0,\ldots\right)  $ for the indeterminates
$x_{1},x_{2},x_{3},\ldots$), and therefore is a $\mathbf{k}$-algebra homomorphism.

Now, $\varepsilon=\varepsilon^{\prime}\mid_{\operatorname*{QSym}}%
$\ \ \ \ \footnote{\textit{Proof.} Recall that the map $\varepsilon$ sends
every power series $f\in{\operatorname{QSym}}$ to the result $f\left(
0,0,0,\ldots\right)  $ of substituting zeroes for the variables $x_{1}%
,x_{2},x_{3},\ldots$ in $f$. Thus, for each $f\in\operatorname*{QSym}$, we
have%
\begin{align*}
\varepsilon\left(  f\right)   &  =f\left(  0,0,0,\ldots\right)  =\varepsilon
^{\prime}\left(  f\right)  \ \ \ \ \ \ \ \ \ \ \left(  \text{since
}\varepsilon^{\prime}\left(  f\right)  =f\left(  0,0,0,\ldots\right)  \text{
(by the definition of }\varepsilon^{\prime}\text{)}\right) \\
&  =\left(  \varepsilon^{\prime}\mid_{\operatorname*{QSym}}\right)  \left(
f\right)  .
\end{align*}
In other words, $\varepsilon=\varepsilon^{\prime}\mid_{\operatorname*{QSym}}%
$.}. Hence, $\varepsilon$ is a restriction of the map $\varepsilon^{\prime}$.
Thus, $\varepsilon$ is a restriction of a $\mathbf{k}$-algebra homomorphism
(since $\varepsilon^{\prime}$ is a $\mathbf{k}$-algebra homomorphism), and
thus itself is a $\mathbf{k}$-algebra homomorphism.

It is easy to see that $\Delta\left(  1\right)  =1\otimes1$%
\ \ \ \ \footnote{\textit{Proof.} There are many ways to prove $\Delta\left(
1\right)  =1\otimes1$ (for example, it follows from $1=M_{\varnothing}$ using
the definition of $\Delta$), but let us derive $\Delta\left(  1\right)
=1\otimes1$ from Corollary \ref{cor.djun.restrict.2}:
\par
For each $i\in\varnothing$, we define a double poset $\mathbf{E}_{i}=\left(
E_{i},<_{1,i},<_{2,i}\right)  $ and a map $w_{i}:E_{i}\rightarrow\left\{
1,2,3,\ldots\right\}  $ as follows: There is nothing to define, because there
exists no $i\in\varnothing$.
\par
Thus, Corollary \ref{cor.djun.restrict.2} (applied to $I=\varnothing$) yields
\begin{align*}
\Delta\left(  \prod_{i\in\varnothing}\Gamma\left(  \mathbf{E}_{i}%
,w_{i}\right)  \right)   &  =\prod_{i\in\varnothing}\Delta\left(
\Gamma\left(  \mathbf{E}_{i},w_{i}\right)  \right)  =\left(  \text{empty
product}\right) \\
&  =\left(  \text{the unity of the }\mathbf{k}\text{-algebra }%
\operatorname*{QSym}\otimes\operatorname*{QSym}\right)  =1\otimes1.
\end{align*}
Since $\prod_{i\in\varnothing}\Gamma\left(  \mathbf{E}_{i},w_{i}\right)
=\left(  \text{empty product}\right)  =1$, this rewrites as $\Delta\left(
1\right)  =1\otimes1$. Qed.}. Thus, the map $\Delta$ sends the unity of the
$\mathbf{k}$-algebra $\operatorname*{QSym}$ to the unity of the $\mathbf{k}%
$-algebra $\operatorname*{QSym}\otimes\operatorname*{QSym}$ (since the latter
unity is $1\otimes1$). Combining this with Corollary \ref{cor.djun.restrict.3}%
, we conclude that $\Delta$ is a $\mathbf{k}$-algebra homomorphism (since
$\Delta$ is a $\mathbf{k}$-linear map).

Now, the $\mathbf{k}$-algebra $\operatorname*{QSym}$, equipped with the
comultiplication $\Delta$ and the counit $\varepsilon$, becomes a $\mathbf{k}%
$-bialgebra (because it is a $\mathbf{k}$-algebra and a $\mathbf{k}%
$-coalgebra, and because $\Delta$ and $\varepsilon$ are $\mathbf{k}$-algebra
homomorphisms). This proves Proposition \ref{prop.QSym.bialg}.
\end{proof}

Finally, using \cite[Proposition 1.4.16]{Reiner}, we can leverage Proposition
\ref{prop.QSym.bialg} to a proof of the following fact:

\begin{proposition}
\label{prop.QSym.hopfalg}The $\mathbf{k}$-bialgebra $\operatorname*{QSym}$ is
a Hopf algebra (i.e., it has an antipode).
\end{proposition}

\begin{proof}
[Proof of Proposition \ref{prop.QSym.hopfalg} (sketched).]It is not hard to
see that the $\mathbf{k}$-bialgebra $\operatorname*{QSym}$ is a connected
graded $\mathbf{k}$-bialgebra, where the grading is given by the degree of the
power series (i.e., for each $n\in\mathbb{N}$, the $n$-th graded component of
$\operatorname*{QSym}$ is the $\mathbf{k}$-submodule%
\[
\left\{  f\in\operatorname*{QSym}\ \mid\ \text{the power series }f\text{ is
homogeneous of degree }n\right\}
\]
of $\operatorname*{QSym}$\ \ \ \ \footnote{This $\mathbf{k}$-submodule has
basis $\left(  M_{\alpha}\right)  _{\alpha\in\operatorname*{Comp}%
\nolimits_{n}}$.}).

A well-known result states that every connected graded $\mathbf{k}$-bialgebra
is a Hopf algebra (i.e., it has an antipode).\footnote{This is proven, for
example, in \cite[Proposition 1.4.16]{Reiner}.} Applying this to the connected
graded $\mathbf{k}$-bialgebra $\operatorname*{QSym}$, we thus conclude that
the $\mathbf{k}$-bialgebra $\operatorname*{QSym}$ is a Hopf algebra (i.e., it
has an antipode). This proves Proposition \ref{prop.QSym.hopfalg}.
\end{proof}

There is also an alternative way to prove Proposition \ref{prop.QSym.hopfalg},
by constructing the antipode explicitly (e.g., using Proposition
\ref{prop.exam.antipode.Gammaw.b} as the \textbf{definition} of the antipode)
and then showing that it satisfies the axioms of an antipode. We shall not
dwell on this.

\subsection{$F_{\alpha}$ as $\Gamma\left(  \mathbf{E}\right)  $}

Next, let us prove a claim made in Example \ref{exam.Gamma} (c):

\begin{proposition}
\label{prop.example.Gamma.c1}Let $\alpha=\left(  \alpha_{1},\alpha_{2}%
,\ldots,\alpha_{\ell}\right)  $ be a composition of a nonnegative integer $n$.
Define a set $D\left(  \alpha\right)  $ as in Definition \ref{def.Dalpha}. Let
$E=\left\{  1,2,\ldots,n\right\}  $. Then, there exists a total order $<_{2}$
on the set $E$ satisfying%
\begin{equation}
\left(  i+1<_{2}i\qquad\text{ for every }i\in D\left(  \alpha\right)  \right)
\label{eq.prop.example.Gamma.c1.cond1}%
\end{equation}
and
\begin{equation}
\left(  i<_{2}i+1\qquad\text{ for every }i\in\left\{  1,2,\ldots,n-1\right\}
\setminus D\left(  \alpha\right)  \right)  .
\label{eq.prop.example.Gamma.c1.cond2}%
\end{equation}

\end{proposition}

We shall actually prove the following fact first (which is easily seen to be
equivalent to Proposition~\ref{prop.example.Gamma.c1}):

\begin{proposition}
\label{prop.example.Gamma.c2}Let $n\in\mathbb{N}$. Let $E=\left\{
1,2,\ldots,n\right\}  $. Let $Q$ be a subset of $\left\{  1,2,\ldots
,n-1\right\}  $. Then, there exists a total order $<_{2}$ on the set $E$
satisfying%
\begin{equation}
\left(  i+1<_{2}i\qquad\text{ for every }i\in Q\right)
\label{eq.prop.example.Gamma.c2.cond1}%
\end{equation}
and
\begin{equation}
\left(  i<_{2}i+1\qquad\text{ for every }i\in\left\{  1,2,\ldots,n-1\right\}
\setminus Q\right)  . \label{eq.prop.example.Gamma.c2.cond2}%
\end{equation}

\end{proposition}

Usually, there are, in fact, several total orders $<_{2}$ satisfying the
conditions of Proposition \ref{prop.example.Gamma.c2}, but of course it
suffices to construct one of them in order to prove the proposition.

Before we prove Proposition \ref{prop.example.Gamma.c2}, let us record a
really simple fact:

\begin{proposition}
\label{prop.order.pullback}Let $E$ and $F$ be two sets. Let $\omega
:E\rightarrow F$ be a map. Let $<$ be a strict partial order on the set $F$.
We define a binary relation $\prec$ on the set $E$ as follows: For any $a\in
E$ and $b\in E$, we set $a\prec b$ if and only if $\omega\left(  a\right)
<\omega\left(  b\right)  $.

\begin{enumerate}
\item[(a)] The relation $\prec$ is a strict partial order on $E$.

\item[(b)] Assume that the relation $<$ is a total order. Assume that the map
$\omega$ is injective. Then, the relation $\prec$ is a total order.
\end{enumerate}
\end{proposition}

Proposition \ref{prop.order.pullback} is a basic fact about partial and total
orders; its easy proof is left to the reader.

\begin{noncompile}

\begin{proof}
[Proof of Proposition \ref{prop.order.pullback}.]The relation $<$ is a strict
partial order on $F$. In other words, the relation $<$ is an irreflexive,
transitive and antisymmetric binary relation on $F$ (because this is how a
\textquotedblleft strict partial order\textquotedblright\ is defined).

Now, the relation $\prec$ is irreflexive\footnote{\textit{Proof.} Let $a\in E$
be such that $a\prec a$. We shall derive a contradiction.
\par
We have $a\prec a$ if and only if $\omega\left(  a\right)  <\omega\left(
a\right)  $ (by the definition of the relation $\prec$). Thus, we have
$\omega\left(  a\right)  <\omega\left(  a\right)  $ (since $a\prec a$). But
this is absurd, since the relation $<$ is irreflexive. Thus, we have obtained
a contradiction.
\par
Let us now forget that we fixed $a$. We thus have derived a contradiction for
every $a\in E$ satisfying $a\prec a$. Therefore, no $a\in E$ satisfies $a\prec
a$. In other words, the relation $\prec$ is irreflexive. Qed.},
transitive\footnote{\textit{Proof.} Let $a\in E$, $b\in E$ and $c\in E$ be
such that $a\prec b$ and $b\prec c$. We shall prove that $a\prec c$.
\par
We have $a\prec b$ if and only if $\omega\left(  a\right)  <\omega\left(
b\right)  $ (by the definition of the relation $\prec$). Thus, we have
$\omega\left(  a\right)  <\omega\left(  b\right)  $ (since $a\prec b$).
\par
We have $b\prec c$ if and only if $\omega\left(  b\right)  <\omega\left(
c\right)  $ (by the definition of the relation $\prec$). Thus, we have
$\omega\left(  b\right)  <\omega\left(  c\right)  $ (since $b\prec c$).
\par
From $\omega\left(  a\right)  <\omega\left(  b\right)  $ and $\omega\left(
b\right)  <\omega\left(  c\right)  $, we obtain $\omega\left(  a\right)
<\omega\left(  c\right)  $ (since the relation $<$ is transitive).
\par
We have $a\prec c$ if and only if $\omega\left(  a\right)  <\omega\left(
c\right)  $ (by the definition of the relation $\prec$). Thus, we have $a\prec
c$ (since $\omega\left(  a\right)  <\omega\left(  c\right)  $).
\par
Now, forget that we fixed $a$, $b$ and $c$. Thus, we have shown that every
$a\in E$, $b\in E$ and $c\in E$ satisfying $a\prec b$ and $b\prec c$ must
satisfy $a\prec c$. In other words, the relation $\prec$ is transitive. Qed.}
and antisymmetric\footnote{\textit{Proof.} The relation $\prec$ is irreflexive
and transitive. Hence, the relation $\prec$ is antisymmetric (since every
irreflexive and transitive binary relation is antisymmetric). Qed.}. Thus, the
relation $\prec$ is an irreflexive, transitive and antisymmetric binary
relation on $E$. In other words, the relation $\prec$ is a strict partial
order on $E$ (because this is how a \textquotedblleft strict partial
order\textquotedblright\ is defined). This proves Proposition
\ref{prop.order.pullback} (a).

(b) The relation $\prec$ is a strict partial order (by Proposition
\ref{prop.order.pullback} (a)). Hence, the relation $\prec$ is a total order
if and only if every two elements of $E$ are $\prec$-comparable (by the
definition of a \textquotedblleft total order\textquotedblright).

Now, let $a$ and $b$ be two elements of $E$. We shall show that $a$ and $b$
are $\prec$-comparable.

The relation $<$ is a strict partial order. Hence, the relation $<$ is a total
order if and only if every two elements of $F$ are $<$-comparable (by the
definition of a \textquotedblleft total order\textquotedblright). Thus, every
two elements of $F$ are $<$-comparable (since the relation $<$ is a total
order). Applying this to the two elements $\omega\left(  a\right)  $ and
$\omega\left(  b\right)  $, we conclude that the two elements $\omega\left(
a\right)  $ and $\omega\left(  b\right)  $ of $F$ are $<$-comparable. In other
words, we have either $\omega\left(  a\right)  <\omega\left(  b\right)  $ or
$\omega\left(  a\right)  =\omega\left(  b\right)  $ or $\omega\left(
b\right)  <\omega\left(  a\right)  $ (because $\omega\left(  a\right)  $ and
$\omega\left(  b\right)  $ are $<$-comparable if and only if we have either
$\omega\left(  a\right)  <\omega\left(  b\right)  $ or $\omega\left(
a\right)  =\omega\left(  b\right)  $ or $\omega\left(  b\right)
<\omega\left(  a\right)  $ (by the definition of \textquotedblleft%
$<$-comparable\textquotedblright)). Thus, we are in one of the following three cases:

\textit{Case 1:} We have $\omega\left(  a\right)  <\omega\left(  b\right)  $.

\textit{Case 2:} We have $\omega\left(  a\right)  =\omega\left(  b\right)  $.

\textit{Case 3:} We have $\omega\left(  b\right)  <\omega\left(  a\right)  $.

Let us first consider Case 1. In this case, we have $\omega\left(  a\right)
<\omega\left(  b\right)  $. But we have $a\prec b$ if and only if
$\omega\left(  a\right)  <\omega\left(  b\right)  $ (by the definition of the
relation $\prec$). Thus, $a\prec b$ (since $\omega\left(  a\right)
<\omega\left(  b\right)  $). Thus, we have either $a\prec b$ or $a=b$ or
$b\prec a$. In other words, the two elements $a$ and $b$ of $E$ are $\prec
$-comparable (because $a$ and $b$ are $\prec$-comparable if and only if we
have either $a\prec b$ or $a=b$ or $b\prec a$ (by the definition of
\textquotedblleft$\prec$-comparable\textquotedblright)). Thus, in Case 1, we
have shown that $a$ and $b$ are $\prec$-comparable.

Let us now consider Case 2. In this case, we have $\omega\left(  a\right)
=\omega\left(  b\right)  $. Hence, $a=b$ (since the map $\omega$ is
injective). Thus, we have either $a\prec b$ or $a=b$ or $b\prec a$. In other
words, the two elements $a$ and $b$ of $E$ are $\prec$-comparable (because $a$
and $b$ are $\prec$-comparable if and only if we have either $a\prec b$ or
$a=b$ or $b\prec a$ (by the definition of \textquotedblleft$\prec
$-comparable\textquotedblright)). Thus, in Case 2, we have shown that $a$ and
$b$ are $\prec$-comparable.

Let us now consider Case 3. In this case, we have $\omega\left(  b\right)
<\omega\left(  a\right)  $. But we have $b\prec a$ if and only if
$\omega\left(  b\right)  <\omega\left(  a\right)  $ (by the definition of the
relation $\prec$). Thus, $b\prec a$ (since $\omega\left(  b\right)
<\omega\left(  a\right)  $). Thus, we have either $a\prec b$ or $a=b$ or
$b\prec a$. In other words, the two elements $a$ and $b$ of $E$ are $\prec
$-comparable (because $a$ and $b$ are $\prec$-comparable if and only if we
have either $a\prec b$ or $a=b$ or $b\prec a$ (by the definition of
\textquotedblleft$\prec$-comparable\textquotedblright)). Thus, in Case 3, we
have shown that $a$ and $b$ are $\prec$-comparable.

Hence, in each of the three Cases 1, 2 and 3, we have shown that $a$ and $b$
are $\prec$-comparable. Thus, $a$ and $b$ are always $\prec$-comparable (since
the three Cases 1, 2 and 3 cover all possibilities).

Let us now forget that we fixed $a$ and $b$. We thus have shown that if $a$
and $b$ are two elements of $E$, then $a$ and $b$ are $\prec$-comparable. In
other words, every two elements of $E$ are $\prec$-comparable. In other words,
the relation $\prec$ is a total order (since the relation $\prec$ is a total
order if and only if every two elements of $E$ are $\prec$-comparable). This
proves Proposition \ref{prop.order.pullback} (b).
\end{proof}
\end{noncompile}

\begin{proof}
[Proof of Proposition \ref{prop.example.Gamma.c2}.]We shall use the notation
introduced in Definition \ref{def.k} (that is, we shall write $\left[
k\right]  $ for $\left\{  1,2,\ldots,k\right\}  $ when $k\in\mathbb{Z}$). In
particular, $\left[  n\right]  =\left\{  1,2,\ldots,n\right\}  =E$.

We shall also use the so-called \textit{Iverson bracket notation}: If
$\mathcal{A}$ is any logical statement, then we shall write $\left[
\mathcal{A}\right]  $ for the integer $%
\begin{cases}
1, & \text{if }\mathcal{A}\text{ is true;}\\
0, & \text{if }\mathcal{A}\text{ is false}%
\end{cases}
$.

Now, we define a map $\rho:\left[  n\right]  \rightarrow\mathbb{Z}$ by%
\[
\left(  \rho\left(  i\right)  =\left\vert Q\cap\left[  i-1\right]  \right\vert
\ \ \ \ \ \ \ \ \ \ \text{for every }i\in\left[  n\right]  \right)  .
\]
We also define a map $\omega:\left[  n\right]  \rightarrow\mathbb{Z}$ by%
\[
\left(  \omega\left(  i\right)  =-n\rho\left(  i\right)
+i\ \ \ \ \ \ \ \ \ \ \text{for every }i\in\left[  n\right]  \right)  .
\]
The map $\omega$ is injective\footnote{\textit{Proof.} Let $i$ and $j$ be two
elements of $\left[  n\right]  $ such that $\omega\left(  i\right)
=\omega\left(  j\right)  $. We shall show that $i=j$.
\par
Assume the contrary. Thus, $i\neq j$. We can WLOG assume that $i\geq j$ (since
otherwise, we can just switch $i$ with $j$). Assume this. Combining $i\geq j$
with $i\neq j$, we obtain $i>j$, so that $i-j>0$.
\par
We have $i\in\left[  n\right]  $, so that $1\leq i\leq n$. Thus, $n\geq1$, so
that $n>0$. Combined with $i-j>0$, this yields $\dfrac{i-j}{n}>0$.
\par
The definition of $\omega$ shows that $\omega\left(  i\right)
=\underbrace{-n\rho\left(  i\right)  }_{\equiv0\operatorname{mod}n}+i\equiv
i\operatorname{mod}n$. The same argument (applied to $j$ instead of $i$)
yields $\omega\left(  j\right)  \equiv j\operatorname{mod}n$. Thus,
$i\equiv\omega\left(  i\right)  =\omega\left(  j\right)  \equiv
j\operatorname{mod}n$. In other words, $n\mid i-j$. Thus, $\dfrac{i-j}{n}$ is
an integer. Hence, $\dfrac{i-j}{n}\geq1$ (since $\dfrac{i-j}{n}>0$), so that
$i-j\geq n$ (since $n>0$). Hence, $i\geq n+\underbrace{j}%
_{\substack{>0\\\text{(since }j\in\left[  n\right]  \text{)}}}>n$. This
contradicts $i\leq n$. This contradiction proves that our assumption was
wrong. Hence, $i=j$ is proven.
\par
Now, forget that we fixed $i$ and $j$. We thus have shown that if $i$ and $j$
are two elements of $\left[  n\right]  $ such that $\omega\left(  i\right)
=\omega\left(  j\right)  $, then $i=j$. In other words, the map $\omega$ is
injective. Qed.}. Notice that $\omega$ is a map $\left[  n\right]
\rightarrow\mathbb{Z}$. In other words, $\omega$ is a map $E\rightarrow
\mathbb{Z}$ (since $\left[  n\right]  =E$).

Consider the total order $<$ on the set $\mathbb{Z}$ (that is, the usual
smaller relation on $\mathbb{Z}$). Define a binary relation $\prec$ on the set
$E$ as follows: For any $a\in E$ and $b\in E$, we set $a\prec b$ if and only
if $\omega\left(  a\right)  <\omega\left(  b\right)  $. Proposition
\ref{prop.order.pullback} (b) (applied to $F=\mathbb{Z}$) thus shows that the
relation $\prec$ is a total order.

Now, it is easy to show that%
\begin{equation}
\rho\left(  i+1\right)  -\rho\left(  i\right)  =\left[  i\in Q\right]
\ \ \ \ \ \ \ \ \ \ \text{for every }i\in\left[  n-1\right]
\label{pf.prop.example.Gamma.c2.Deltarho}%
\end{equation}
\footnote{\textit{Proof of (\ref{pf.prop.example.Gamma.c2.Deltarho}):} Let
$i\in\left[  n-1\right]  $. We must prove
(\ref{pf.prop.example.Gamma.c2.Deltarho}). We are in one of the following two
cases:
\par
\textit{Case 1:} We have $i\in Q$.
\par
\textit{Case 2:} We have $i\notin Q$.
\par
Let us first consider Case 1. In this case, we have $i\in Q$. Hence, $\left\{
i\right\}  \subseteq Q$, so that $Q\cap\left\{  i\right\}  =\left\{
i\right\}  $. But%
\[
\left(  Q\cap\left[  i-1\right]  \right)  \cap\left\{  i\right\}
=Q\cap\underbrace{\left[  i-1\right]  \cap\left\{  i\right\}  }%
_{\substack{=\varnothing\\\text{(since }i\notin\left[  i-1\right]  \text{)}%
}}=Q\cap\varnothing=\varnothing.
\]
Also,%
\[
Q\cap\underbrace{\left[  i\right]  }_{=\left[  i-1\right]  \cup\left\{
i\right\}  }=Q\cap\left(  \left[  i-1\right]  \cup\left\{  i\right\}  \right)
=\left(  Q\cap\left[  i-1\right]  \right)  \cup\underbrace{\left(
Q\cap\left\{  i\right\}  \right)  }_{=\left\{  i\right\}  }=\left(
Q\cap\left[  i-1\right]  \right)  \cup\left\{  i\right\}  .
\]
Hence,%
\begin{align*}
\left\vert \underbrace{Q\cap\left[  i\right]  }_{=\left(  Q\cap\left[
i-1\right]  \right)  \cup\left\{  i\right\}  }\right\vert  &  =\left\vert
\left(  Q\cap\left[  i-1\right]  \right)  \cup\left\{  i\right\}  \right\vert
\\
&  =\underbrace{\left\vert Q\cap\left[  i-1\right]  \right\vert }%
_{\substack{=\rho\left(  i\right)  \\\text{(since }\rho\left(  i\right)
=\left\vert Q\cap\left[  i-1\right]  \right\vert \text{)}}%
}+\underbrace{\left\vert \left\{  i\right\}  \right\vert }_{=1}%
\ \ \ \ \ \ \ \ \ \ \left(  \text{since }\left(  Q\cap\left[  i-1\right]
\right)  \cap\left\{  i\right\}  =\varnothing\right) \\
&  =\rho\left(  i\right)  +1.
\end{align*}
Now, the definition of $\rho$ yields%
\[
\rho\left(  i+1\right)  =\left\vert Q\cap\left[  \underbrace{\left(
i+1\right)  -1}_{=i}\right]  \right\vert =\left\vert Q\cap\left[  i\right]
\right\vert =\rho\left(  i\right)  +1,
\]
so that $\rho\left(  i+1\right)  -\rho\left(  i\right)  =1=\left[  i\in
Q\right]  $ (since $\left[  i\in Q\right]  =1$ (since $i\in Q$)). Thus,
(\ref{pf.prop.example.Gamma.c2.Deltarho}) is proven in Case 1.
\par
Let us now consider Case 2. In this case, we have $i\notin Q$. Thus, $\left\{
i\right\}  \cap Q=\varnothing$. Now,
\begin{align*}
Q\cap\underbrace{\left[  i\right]  }_{=\left[  i-1\right]  \cup\left\{
i\right\}  }  &  =Q\cap\left(  \left[  i-1\right]  \cup\left\{  i\right\}
\right)  =\left(  Q\cap\left[  i-1\right]  \right)  \cup\underbrace{\left(
Q\cap\left\{  i\right\}  \right)  }_{=\left\{  i\right\}  \cap Q=\varnothing
}\\
&  =\left(  Q\cap\left[  i-1\right]  \right)  \cup\varnothing=Q\cap\left[
i-1\right]  .
\end{align*}
Hence,%
\[
\left\vert \underbrace{Q\cap\left[  i\right]  }_{=Q\cap\left[  i-1\right]
}\right\vert =\left\vert Q\cap\left[  i-1\right]  \right\vert =\rho\left(
i\right)  \ \ \ \ \ \ \ \ \ \ \left(  \text{since }\rho\left(  i\right)
=\left\vert Q\cap\left[  i-1\right]  \right\vert \right)  .
\]
Now, the definition of $\rho$ yields%
\[
\rho\left(  i+1\right)  =\left\vert Q\cap\left[  \underbrace{\left(
i+1\right)  -1}_{=i}\right]  \right\vert =\left\vert Q\cap\left[  i\right]
\right\vert =\rho\left(  i\right)  ,
\]
so that $\rho\left(  i+1\right)  -\rho\left(  i\right)  =0=\left[  i\in
Q\right]  $ (since $\left[  i\in Q\right]  =0$ (since $i\notin Q$)). Thus,
(\ref{pf.prop.example.Gamma.c2.Deltarho}) is proven in Case 2.
\par
We have now proven (\ref{pf.prop.example.Gamma.c2.Deltarho}) in each of the
two Cases 1 and 2. Therefore, (\ref{pf.prop.example.Gamma.c2.Deltarho}) always
holds.}.

We now claim that%
\begin{equation}
\left(  i+1\prec i\qquad\text{ for every }i\in Q\right)  .
\label{pf.prop.example.Gamma.c2.part1}%
\end{equation}

[\textit{Proof of (\ref{pf.prop.example.Gamma.c2.part1}):} Let $i\in Q$. Thus,
$i\in Q\subseteq\left[  n-1\right]  $. Thus, $1\leq i\leq n-1$, so that
$n-1\geq1$. Also, $i\in\left[  n-1\right]  $, so that
(\ref{pf.prop.example.Gamma.c2.Deltarho}) yields $\rho\left(  i+1\right)
-\rho\left(  i\right)  =\left[  i\in Q\right]  =1$ (since $i\in Q$). Now, the
definition of $\omega$ yields $\omega\left(  i+1\right)  =-n\rho\left(
i+1\right)  +\left(  i+1\right)  $ and $\omega\left(  i\right)  =-n\rho\left(
i\right)  +i$. Hence,%
\begin{align*}
\underbrace{\omega\left(  i+1\right)  }_{=-n\rho\left(  i+1\right)  +\left(
i+1\right)  }-\underbrace{\omega\left(  i\right)  }_{=-n\rho\left(  i\right)
+i}  &  =\left(  -n\rho\left(  i+1\right)  +\left(  i+1\right)  \right)
-\left(  -n\rho\left(  i\right)  +i\right) \\
&  =-n\underbrace{\left(  \rho\left(  i+1\right)  -\rho\left(  i\right)
\right)  }_{=1}+\underbrace{\left(  i+1\right)  -i}_{=1}\\
&  =-n+1=-\underbrace{\left(  n-1\right)  }_{\geq1>0}<-0=0.
\end{align*}
In other words, $\omega\left(  i+1\right)  <\omega\left(  i\right)  $.

But $i+1\prec i$ holds if and only if $\omega\left(  i+1\right)
<\omega\left(  i\right)  $ (by the definition of the relation $\prec$). Hence,
$i+1\prec i$ holds (since $\omega\left(  i+1\right)  <\omega\left(  i\right)
$). This proves (\ref{pf.prop.example.Gamma.c2.part1}).]

Furthermore, we claim that%
\begin{equation}
\left(  i\prec i+1\qquad\text{ for every }i\in\left\{  1,2,\ldots,n-1\right\}
\setminus Q\right)  . \label{pf.prop.example.Gamma.c2.part2}%
\end{equation}

[\textit{Proof of (\ref{pf.prop.example.Gamma.c2.part2}):} Let $i\in\left\{
1,2,\ldots,n-1\right\}  \setminus Q$. Thus, $i\in\left\{  1,2,\ldots
,n-1\right\}  $ but $i\notin Q$. We have $i\in\left\{  1,2,\ldots,n-1\right\}
=\left[  n-1\right]  $. Thus, (\ref{pf.prop.example.Gamma.c2.Deltarho}) yields
$\rho\left(  i+1\right)  -\rho\left(  i\right)  =\left[  i\in Q\right]  =0$
(since $i\notin Q$). Now, the definition of $\omega$ yields $\omega\left(
i+1\right)  =-n\rho\left(  i+1\right)  +\left(  i+1\right)  $ and
$\omega\left(  i\right)  =-n\rho\left(  i\right)  +i$. Hence,%
\begin{align*}
\underbrace{\omega\left(  i+1\right)  }_{=-n\rho\left(  i+1\right)  +\left(
i+1\right)  }-\underbrace{\omega\left(  i\right)  }_{=-n\rho\left(  i\right)
+i}  &  =\left(  -n\rho\left(  i+1\right)  +\left(  i+1\right)  \right)
-\left(  -n\rho\left(  i\right)  +i\right) \\
&  =-n\underbrace{\left(  \rho\left(  i+1\right)  -\rho\left(  i\right)
\right)  }_{=0}+\underbrace{\left(  i+1\right)  -i}_{=1}\\
&  =-0+1=1>0.
\end{align*}
In other words, $\omega\left(  i+1\right)  >\omega\left(  i\right)  $, so that
$\omega\left(  i\right)  <\omega\left(  i+1\right)  $.

But $i\prec i+1$ holds if and only if $\omega\left(  i\right)  <\omega\left(
i+1\right)  $ (by the definition of the relation $\prec$). Hence, $i\prec i+1$
holds (since $\omega\left(  i\right)  <\omega\left(  i+1\right)  $). This
proves (\ref{pf.prop.example.Gamma.c2.part2}).]

Now, we know that $\prec$ is a total order on the set $E$ and satisfies
(\ref{pf.prop.example.Gamma.c2.part1}) and
(\ref{pf.prop.example.Gamma.c2.part2}). Hence, there exists a total order
$<_{2}$ on the set $E$ satisfying%
\[
\left(  i+1<_{2}i\qquad\text{ for every }i\in Q\right)
\]
and
\[
\left(  i<_{2}i+1\qquad\text{ for every }i\in\left\{  1,2,\ldots,n-1\right\}
\setminus Q\right)
\]
(namely, $\prec$ is such a total order). This proves Proposition
\ref{prop.example.Gamma.c2}.
\end{proof}

\begin{proof}
[Proof of Proposition \ref{prop.example.Gamma.c1}.]We shall use the notation
introduced in Definition \ref{def.k} (that is, we shall write $\left[
k\right]  $ for $\left\{  1,2,\ldots,k\right\}  $ when $k\in\mathbb{Z}$). In
particular, $\left[  n\right]  =\left\{  1,2,\ldots,n\right\}  =E$.

Lemma \ref{lem.Dalpha.n-1} shows that $D\left(  \alpha\right)  \subseteq
\left[  n-1\right]  =\left\{  1,2,\ldots,n-1\right\}  $. In other words,
$D\left(  \alpha\right)  $ is a subset of $\left\{  1,2,\ldots,n-1\right\}  $.
Proposition \ref{prop.example.Gamma.c2} (applied to $Q=D\left(  \alpha\right)
$) thus shows that there exists a total order $<_{2}$ on the set $E$
satisfying%
\[
\left(  i+1<_{2}i\qquad\text{ for every }i\in D\left(  \alpha\right)  \right)
\]
and
\[
\left(  i<_{2}i+1\qquad\text{ for every }i\in\left\{  1,2,\ldots,n-1\right\}
\setminus D\left(  \alpha\right)  \right)  .
\]
This proves Proposition \ref{prop.example.Gamma.c1}.
\end{proof}

Next, let us prove another claim made in Example \ref{exam.Gamma} (c):

\begin{proposition}
\label{prop.example.Gamma.c3}Let $\alpha=\left(  \alpha_{1},\alpha_{2}%
,\ldots,\alpha_{\ell}\right)  $ be a composition of a nonnegative integer $n$.
Define a set $D\left(  \alpha\right)  $ as in Definition \ref{def.Dalpha}. Let
$E=\left\{  1,2,\ldots,n\right\}  $. Let $<_{1}$ be the total order on the set
$E$ inherited from $\mathbb{Z}$ (thus, two elements $a$ and $b$ of $E$ satisfy
$a<_{1}b$ if and only if they satisfy $a<b$). Let $<_{2}$ be a strict partial
order on the set $E$ satisfying (\ref{eq.prop.example.Gamma.c1.cond1}) and
(\ref{eq.prop.example.Gamma.c1.cond2}). Then,%
\begin{align*}
\Gamma\left(  \left(  E,<_{1},<_{2}\right)  \right)   &  =\sum
_{\substack{i_{1}\leq i_{2}\leq\cdots\leq i_{n};\\i_{j}<i_{j+1}\text{ whenever
}j\in D\left(  \alpha\right)  }}x_{i_{1}}x_{i_{2}}\cdots x_{i_{n}}\\
&  =\sum_{\beta\text{ is a composition of }n;\ D\left(  \beta\right)
\supseteq D\left(  \alpha\right)  }M_{\beta}.
\end{align*}

\end{proposition}

\begin{proof}
[Proof of Proposition \ref{prop.example.Gamma.c3}.]Let us first observe a
simple fact: If $u$ and $v$ are two elements of $\left\{  1,2,\ldots
,n\right\}  $ such that $u<v$ and $v<_{2}u$, then
\begin{equation}
\left\{  u,u+1,\ldots,v-1\right\}  \cap D\left(  \alpha\right)  \neq
\varnothing\label{pf.prop.example.Gamma.c3.0}%
\end{equation}
\footnote{\textit{Proof of (\ref{pf.prop.example.Gamma.c3.0}):} Let $u$ and
$v$ be two elements of $\left\{  1,2,\ldots,n\right\}  $ such that $u<v$ and
$v<_{2}u$. We must prove (\ref{pf.prop.example.Gamma.c3.0}).
\par
Indeed, assume (for the sake of contradiction) that $\left\{  u,u+1,\ldots
,v-1\right\}  \cap D\left(  \alpha\right)  =\varnothing$.
\par
We have $u\geq1$ (since $u\in\left\{  1,2,\ldots,n\right\}  $) and $v\leq n$
(since $v\in\left\{  1,2,\ldots,n\right\}  $). Also, $u\neq v$ (since $u<v$).
\par
Now, let $k\in\left\{  u,u+1,\ldots,v-1\right\}  $ be arbitrary. If we had
$k\in D\left(  \alpha\right)  $, then we would have%
\begin{align*}
k  &  \in\left\{  u,u+1,\ldots,v-1\right\}  \cap D\left(  \alpha\right)
\ \ \ \ \ \ \ \ \ \ \left(  \text{since }k\in\left\{  u,u+1,\ldots
,v-1\right\}  \text{ and }k\in D\left(  \alpha\right)  \right) \\
&  =\varnothing,
\end{align*}
which would imply that the empty set $\varnothing$ has at least one element
(namely, the element $k$); but this is clearly absurd. Thus, we cannot have
$k\in D\left(  \alpha\right)  $. Therefore, we must have $k\notin D\left(
\alpha\right)  $. But $k\in\left\{  u,u+1,\ldots,v-1\right\}  \subseteq
\left\{  1,2,\ldots,n-1\right\}  $ (since $u\geq1$ and $v\leq n$). Combining
this with $k\notin D\left(  \alpha\right)  $, we obtain $k\in\left\{
1,2,\ldots,n-1\right\}  \setminus D\left(  \alpha\right)  $. Hence,
(\ref{eq.prop.example.Gamma.c1.cond2}) (applied to $i=k$) shows that
$k<_{2}k+1$.
\par
Now, forget that we fixed $k$. We thus have shown that $k<_{2}k+1$ for each
$k\in\left\{  u,u+1,\ldots,v-1\right\}  $. In other words, we have the
relations $u<_{2}u+1$, $u+1<_{2}u+2$, $\ldots$, $v-1<_{2}v$. Since $\left(
E,<_{2}\right)  $ is a poset (because $<_{2}$ is a strict partial order on
$E$), we can combine these relations into a chain of inequalities:%
\[
u<_{2}u+1<_{2}u+2<_{2}\cdots<_{2}v.
\]
Thus, $u<_{2}v$ (since $u\neq v$). This contradicts $v<_{2}u$ (since $<_{2}$
is a strict partial order on $E$).
\par
This contradiction shows that our assumption (that $\left\{  u,u+1,\ldots
,v-1\right\}  \cap D\left(  \alpha\right)  =\varnothing$) was false. Hence, we
cannot have $\left\{  u,u+1,\ldots,v-1\right\}  \cap D\left(  \alpha\right)
=\varnothing$. We thus must have $\left\{  u,u+1,\ldots,v-1\right\}  \cap
D\left(  \alpha\right)  \neq\varnothing$. This proves
(\ref{pf.prop.example.Gamma.c3.0}).}.

Let $Z$ be the set of all $\left(  E,<_{1},<_{2}\right)  $-partitions. The
definition of $\Gamma\left(  \left(  E,<_{1},<_{2}\right)  \right)  $ yields%
\[
\Gamma\left(  \left(  E,<_{1},<_{2}\right)  \right)  =\sum_{\pi\text{ is an
}\left(  E,<_{1},<_{2}\right)  \text{-partition}}{\mathbf{x}}_{\pi},
\]
where ${\mathbf{x}}_{\pi}=\prod_{e\in E}x_{\pi\left(  e\right)  }$. Thus,%
\begin{equation}
\Gamma\left(  \left(  E,<_{1},<_{2}\right)  \right)  =\underbrace{\sum
_{\pi\text{ is an }\left(  E,<_{1},<_{2}\right)  \text{-partition}}%
}_{\substack{=\sum_{\pi\in Z}\\\text{(since }Z\text{ is the set of
all}\\\left(  E,<_{1},<_{2}\right)  \text{-partitions)}}}{\mathbf{x}}_{\pi
}=\sum_{\pi\in Z}\mathbf{x}_{\pi}. \label{pf.prop.example.Gamma.c3.Gamma=}%
\end{equation}

On the other hand, define a set $W$ by%
\begin{align}
W  &  =\left\{  \left(  i_{1},i_{2},\ldots,i_{n}\right)  \in\left\{
1,2,3,\ldots\right\}  ^{n}\ \mid\ i_{1}\leq i_{2}\leq\cdots\leq i_{n}\right.
\nonumber\\
&  \ \ \ \ \ \ \ \ \ \ \left.  \text{and }\left(  i_{j}<i_{j+1}\text{ whenever
}j\in D\left(  \alpha\right)  \right)  \right\}  .
\label{pf.prop.example.Gamma.c3.W=}%
\end{align}
Thus, we have the following equality between summation signs:%
\[
\sum_{\left(  i_{1},i_{2},\ldots,i_{n}\right)  \in W}=\sum_{\substack{\left(
i_{1},i_{2},\ldots,i_{n}\right)  \in\left\{  1,2,3,\ldots\right\}
^{n};\\i_{1}\leq i_{2}\leq\cdots\leq i_{n};\\i_{j}<i_{j+1}\text{ whenever
}j\in D\left(  \alpha\right)  }}=\sum_{\substack{i_{1}\leq i_{2}\leq\cdots\leq
i_{n};\\i_{j}<i_{j+1}\text{ whenever }j\in D\left(  \alpha\right)  }}.
\]

For every $\pi\in Z$, we have $\left(  \pi\left(  1\right)  ,\pi\left(
2\right)  ,\ldots,\pi\left(  n\right)  \right)  \in W$%
\ \ \ \ \footnote{\textit{Proof.} Let $\pi\in Z$. We must show that $\left(
\pi\left(  1\right)  ,\pi\left(  2\right)  ,\ldots,\pi\left(  n\right)
\right)  \in W$.
\par
First, let us see that $\left(  \pi\left(  1\right)  ,\pi\left(  2\right)
,\ldots,\pi\left(  n\right)  \right)  $ is well-defined. We have $\pi\in Z$.
In other words, $\pi$ is an $\left(  E,<_{1},<_{2}\right)  $-partition (since
$Z$ is the set of all $\left(  E,<_{1},<_{2}\right)  $-partitions). Hence,
$\pi$ is a map $E\rightarrow\left\{  1,2,3,\ldots\right\}  $. In other words,
$\pi$ is a map $\left\{  1,2,\ldots,n\right\}  \rightarrow\left\{
1,2,3,\ldots\right\}  $ (since $E=\left\{  1,2,\ldots,n\right\}  $).
Hence, $\left(  \pi\left(  1\right)  ,\pi\left(  2\right)  ,\ldots,\pi\left(
n\right)  \right)  $ is well-defined and an element of $\left\{
1,2,3,\ldots\right\}  ^{n}$.
\par
Recall the definition of an $\left(  E,<_{1},<_{2}\right)  $-partition. This
definition shows that $\pi$ is an $\left(  E,<_{1},<_{2}\right)  $-partition
if and only if it satisfies the following two assertions:
\par
\textit{Assertion }$\mathcal{A}_{1}$\textit{:} Every $e\in E$ and $f\in E$
satisfying $e<_{1}f$ satisfy $\pi\left(  e\right)  \leq\pi\left(  f\right)  $.
\par
\textit{Assertion }$\mathcal{A}_{2}$\textit{:} Every $e\in E$ and $f\in E$
satisfying $e<_{1}f$ and $f<_{2}e$ satisfy $\pi\left(  e\right)  <\pi\left(
f\right)  $.
\par
Thus, $\pi$ satisfies Assertions $\mathcal{A}_{1}$ and $\mathcal{A}_{2}$
(since $\pi$ is an $\left(  E,<_{1},<_{2}\right)  $-partition).
\par
We now shall show that
\begin{equation}
\pi\left(  k\right)  \leq\pi\left(  k+1\right)  \ \ \ \ \ \ \ \ \ \ \text{for
every }k\in\left\{  1,2,\ldots,n-1\right\}
.\label{pf.prop.example.Gamma.c3.ZW.1}%
\end{equation}
\par
[\textit{Proof of (\ref{pf.prop.example.Gamma.c3.ZW.1}):} Let $k\in\left\{
1,2,\ldots,n-1\right\}  $ be arbitrary. We shall show that $\pi\left(
k\right)  \leq\pi\left(  k+1\right)  $.
\par
We have $k\in\left\{  1,2,\ldots,n-1\right\}  $. Thus, both $k$ and $k+1$
belong to the set $\left\{  1,2,\ldots,n\right\}  $. In other words, both $k$
and $k+1$ belong to the set $E$ (since $E=\left\{  1,2,\ldots,n\right\}  $).
\par
Recall that $<_{1}$ is the total order on the set $E$ inherited from
$\mathbb{Z}$. Hence, $k<_{1}k+1$ (since $k<k+1$). Therefore, Assertion
$\mathcal{A}_{1}$ (applied to $e=k$ and $f=k+1$) yields $\pi\left(  k\right)
\leq\pi\left(  k+1\right)  $. This proves (\ref{pf.prop.example.Gamma.c3.ZW.1}%
).]
\par
Now, (\ref{pf.prop.example.Gamma.c3.ZW.1}) shows that $\pi\left(  1\right)
\leq\pi\left(  2\right)  \leq\cdots\leq\pi\left(  n\right)  $.
\par
Next, let us show that
\begin{equation}
\pi\left(  j\right)  <\pi\left(  j+1\right)
\ \ \ \ \ \ \ \ \ \ \text{whenever }j\in D\left(  \alpha\right)
.\label{pf.prop.example.Gamma.c3.ZW.2}%
\end{equation}
\par
[\textit{Proof of (\ref{pf.prop.example.Gamma.c3.ZW.2}):} Let $j\in D\left(
\alpha\right)  $.
\par
We shall use the notation introduced in Definition \ref{def.k} (that is, we
shall write $\left[  k\right]  $ for $\left\{  1,2,\ldots,k\right\}  $ when
$k\in\mathbb{Z}$). In particular, $\left[  n-1\right]  =\left\{
1,2,\ldots,n-1\right\}  $. Lemma \ref{lem.Dalpha.n-1} shows that $D\left(
\alpha\right)  \subseteq\left[  n-1\right]  =\left\{  1,2,\ldots,n-1\right\}
$. Hence, $j\in D\left(  \alpha\right)  \subseteq\left\{  1,2,\ldots
,n-1\right\}  $. Thus, both $j$ and $j+1$ belong to the set $\left\{
1,2,\ldots,n\right\}  $. In other words, both $j$ and $j+1$ belong to the set
$E$ (since $E=\left\{  1,2,\ldots,n\right\}  $).
\par
Recall that $<_{1}$ is the total order on the set $E$ inherited from
$\mathbb{Z}$. Hence, $j<_{1}j+1$ (since $j<j+1$). Also, $j+1<_{2}j$ (by
(\ref{eq.prop.example.Gamma.c1.cond1}), applied to $i=j$). Therefore,
Assertion $\mathcal{A}_{2}$ (applied to $e=j$ and $f=j+1$) yields $\pi\left(
j\right)  <\pi\left(  j+1\right)  $. This proves
(\ref{pf.prop.example.Gamma.c3.ZW.2}).]
\par
Now, we know that $\left(  \pi\left(  1\right)  ,\pi\left(  2\right)
,\ldots,\pi\left(  n\right)  \right)  $ is an element of $\left\{
1,2,3,\ldots\right\}  ^{n}$ and satisfies $\pi\left(  1\right)  \leq\pi\left(
2\right)  \leq\cdots\leq\pi\left(  n\right)  $ and $\left(  \pi\left(
j\right)  <\pi\left(  j+1\right)  \text{ whenever }j\in D\left(
\alpha\right)  \right)  $. Thus,
\begin{align*}
\left(  \pi\left(  1\right)  ,\pi\left(  2\right)  ,\ldots,\pi\left(
n\right)  \right)   &  \in\left\{  \left(  i_{1},i_{2},\ldots,i_{n}\right)
\in\left\{  1,2,3,\ldots\right\}  ^{n}\ \mid\ i_{1}\leq i_{2}\leq\cdots\leq
i_{n}\right.  \\
&  \ \ \ \ \ \ \ \ \ \ \left.  \text{and }\left(  i_{j}<i_{j+1}\text{ whenever
}j\in D\left(  \alpha\right)  \right)  \right\}  \\
&  =W.
\end{align*}
}. Hence, we can define a map $\Phi:Z\rightarrow W$ by setting%
\[
\left(  \Phi\left(  \pi\right)  =\left(  \pi\left(  1\right)  ,\pi\left(
2\right)  ,\ldots,\pi\left(  n\right)  \right)  \ \ \ \ \ \ \ \ \ \ \text{for
every }\pi\in Z\right)  .
\]
Consider this map $\Phi$. Thus, $\Phi$ is the map $Z\rightarrow W,\ \pi
\mapsto\left(  \pi\left(  1\right)  ,\pi\left(  2\right)  ,\ldots,\pi\left(
n\right)  \right)  $.

The map $\Phi$ is surjective\footnote{\textit{Proof.} Let $\mathbf{g}\in W$.
We shall show that $\mathbf{g}\in\Phi\left(  Z\right)  $.
\par
We have%
\begin{align*}
\mathbf{g}  &  \in W=\left\{  \left(  i_{1},i_{2},\ldots,i_{n}\right)
\in\left\{  1,2,3,\ldots\right\}  ^{n}\ \mid\ i_{1}\leq i_{2}\leq\cdots\leq
i_{n}\right. \\
&  \ \ \ \ \ \ \ \ \ \ \ \ \ \ \ \ \ \ \ \ \left.  \text{and }\left(
i_{j}<i_{j+1}\text{ whenever }j\in D\left(  \alpha\right)  \right)  \right\}
.
\end{align*}
In other words, $\mathbf{g}$ is an $\left(  i_{1},i_{2},\ldots,i_{n}\right)
\in\left\{  1,2,3,\ldots\right\}  ^{n}$ satisfying $i_{1}\leq i_{2}\leq
\cdots\leq i_{n}$ and $\left(  i_{j}<i_{j+1}\text{ whenever }j\in D\left(
\alpha\right)  \right)  $. Consider this $\left(  i_{1},i_{2},\ldots
,i_{n}\right)  $. Thus, $\mathbf{g}=\left(  i_{1},i_{2},\ldots,i_{n}\right)
$.
\par
We have $i_{1}\leq i_{2}\leq\cdots\leq i_{n}$. In other words, for any
$u\in\left\{  1,2,\ldots,n\right\}  $ and $v\in\left\{  1,2,\ldots,n\right\}
$ satisfying $u\leq v$, we have%
\begin{equation}
i_{u}\leq i_{v}. \label{pf.prop.example.Gamma.c3.surj.1}%
\end{equation}
Notice also that we have%
\begin{equation}
\left(  i_{j}<i_{j+1}\text{ whenever }j\in D\left(  \alpha\right)  \right)  .
\label{pf.prop.example.Gamma.c3.surj.2}%
\end{equation}
\par
Define a map $\pi:\left\{  1,2,\ldots,n\right\}  \rightarrow\left\{
1,2,3,\ldots\right\}  $ by setting%
\[
\left(  \pi\left(  k\right)  =i_{k}\ \ \ \ \ \ \ \ \ \ \text{for every }%
k\in\left\{  1,2,\ldots,n\right\}  \right)  .
\]
Thus,
$\pi$ is a map from $E$ to $\left\{  1,2,3,\ldots\right\}  $ (since
$E=\left\{  1,2,\ldots,n\right\}  $).
\par
Recall the definition of an $\left(  E,<_{1},<_{2}\right)  $-partition. This
definition shows that $\pi$ is an $\left(  E,<_{1},<_{2}\right)  $-partition
if and only if it satisfies the following two assertions:
\par
\textit{Assertion }$\mathcal{A}_{1}$\textit{:} Every $e\in E$ and $f\in E$
satisfying $e<_{1}f$ satisfy $\pi\left(  e\right)  \leq\pi\left(  f\right)  $.
\par
\textit{Assertion }$\mathcal{A}_{2}$\textit{:} Every $e\in E$ and $f\in E$
satisfying $e<_{1}f$ and $f<_{2}e$ satisfy $\pi\left(  e\right)  <\pi\left(
f\right)  $.
\par
We shall now show that $\pi$ satisfies Assertions $\mathcal{A}_{1}$ and
$\mathcal{A}_{2}$.
\par
\textit{Proof of Assertion $\mathcal{A}$}$_{1}$\textit{:} Let $e\in E$ and
$f\in E$ be such that $e<_{1}f$. Recall that $<_{1}$ is the total order on the
set $E$ inherited from $\mathbb{Z}$. Hence, $e<_{1}f$ holds if and only if
$e<f$. Thus, $e<f$ (since $e<_{1}f$ holds). Thus, $e\leq f$. Hence,
(\ref{pf.prop.example.Gamma.c3.surj.1}) (applied to $u=e$ and $v=f$) yields
$i_{e}\leq i_{f}$. Now, the definition of $\pi$ yields $\pi\left(  e\right)
=i_{e}$ and $\pi\left(  f\right)  =i_{f}$. Hence, $\pi\left(  e\right)
=i_{e}\leq i_{f}=\pi\left(  f\right)  $. This proves Assertion $\mathcal{A}%
_{1}$.
\par
\textit{Proof of Assertion $\mathcal{A}$}$_{2}$\textit{:} Let $e\in E$ and
$f\in E$ be such that $e<_{1}f$ and $f<_{2}e$. Recall that $<_{1}$ is the
total order on the set $E$ inherited from $\mathbb{Z}$. Hence, $e<_{1}f$ holds
if and only if $e<f$. Thus, $e<f$ (since $e<_{1}f$ holds). Thus, $e\leq f$.
\par
Both $e$ and $f$ are elements of $E$. In other words, both $e$ and $f$ are
elements of $\left\{  1,2,\ldots,n\right\}  $ (since $E=\left\{
1,2,\ldots,n\right\}  $). We have $e<f$ and $f<_{2}e$. Thus,
(\ref{pf.prop.example.Gamma.c3.0}) (applied to $u=e$ and $v=f$) yields
$\left\{  e,e+1,\ldots,f-1\right\}  \cap D\left(  \alpha\right)
\neq\varnothing$. In other words, the set $\left\{  e,e+1,\ldots,f-1\right\}
\cap D\left(  \alpha\right)  $ is nonempty; hence, this set contains an
element. In other words, there exists some $j\in\left\{  e,e+1,\ldots
,f-1\right\}  \cap D\left(  \alpha\right)  $. Consider such $j$.
\par
We have $j\in\left\{  e,e+1,\ldots,f-1\right\}  \cap D\left(  \alpha\right)
\subseteq\left\{  e,e+1,\ldots,f-1\right\}  $. Thus, $e\leq j\leq f-1$. From
$j\leq f-1$, we obtain $j+1\leq f$.
\par
From $e\in\left\{  1,2,\ldots,n\right\}  $, we obtain $e\geq1$. From
$f\in\left\{  1,2,\ldots,n\right\}  $, we obtain $f\leq n$. Thus,
$j\in\left\{  e,e+1,\ldots,f-1\right\}  \subseteq\left\{  1,2,\ldots
,n-1\right\}  $ (since $e\geq1$ and $f\leq n$). Hence, both $j$ and $j+1$ are
elements of $\left\{  1,2,\ldots,n\right\}  $.
\par
We have $j\in\left\{  e,e+1,\ldots,f-1\right\}  \cap D\left(  \alpha\right)
\subseteq D\left(  \alpha\right)  $. Thus, $i_{j}<i_{j+1}$ (by
(\ref{pf.prop.example.Gamma.c3.surj.2})). Applying
(\ref{pf.prop.example.Gamma.c3.surj.1}) to $u=e$ and $v=j$, we obtain
$i_{e}\leq i_{j}$ (since $e\leq j$). Applying
(\ref{pf.prop.example.Gamma.c3.surj.1}) to $u=j+1$ and $v=f$, we obtain
$i_{j+1}\leq i_{f}$ (since $j+1\leq f$). Hence, $i_{e}\leq i_{j}<i_{j+1}\leq
i_{f}$.
\par
Now, the definition of $\pi$ yields $\pi\left(  e\right)  =i_{e}$ and
$\pi\left(  f\right)  =i_{f}$. Hence, $\pi\left(  e\right)  =i_{e}<i_{f}%
=\pi\left(  f\right)  $. This proves Assertion $\mathcal{A}_{2}$.
\par
Now, we have shown that $\pi$ satisfies Assertions $\mathcal{A}_{1}$ and
$\mathcal{A}_{2}$. Thus, $\pi$ is an $\left(  E,<_{1},<_{2}\right)
$-partition (since we know that $\pi$ is an $\left(  E,<_{1},<_{2}\right)
$-partition if and only if it satisfies Assertions $\mathcal{A}_{1}$ and
$\mathcal{A}_{2}$). In other words, $\pi\in Z$ (since $Z$ is the set of all
$\left(  E,<_{1},<_{2}\right)  $-partitions).
\par
Now, the definition of $\Phi$ yields%
\begin{align*}
\Phi\left(  \pi\right)   &  =\left(  \pi\left(  1\right)  ,\pi\left(
2\right)  ,\ldots,\pi\left(  n\right)  \right)  =\left(  i_{1},i_{2}%
,\ldots,i_{n}\right) \\
&  \ \ \ \ \ \ \ \ \ \ \left(  \text{since }\pi\left(  k\right)  =i_{k}\text{
for every }k\in\left\{  1,2,\ldots,n\right\}  \right) \\
&  =\mathbf{g}.
\end{align*}
Hence, $\mathbf{g}=\Phi\left(  \underbrace{\pi}_{\in Z}\right)  \in\Phi\left(
Z\right)  $.
\par
Let us now forget that we fixed $\mathbf{g}$. We thus have shown that
$\mathbf{g}\in\Phi\left(  Z\right)  $ for every $\mathbf{g}\in W$. In other
words, $W\subseteq\Phi\left(  Z\right)  $. In other words, the map $\Phi$ is
surjective, qed.} and injective\footnote{\textit{Proof.} Let $\pi_{1}$ and
$\pi_{2}$ be two elements of $Z$ such that $\Phi\left(  \pi_{1}\right)
=\Phi\left(  \pi_{2}\right)  $. We shall prove that $\pi_{1}=\pi_{2}$.
\par
The definition of $\Phi$ yields $\Phi\left(  \pi_{1}\right)  =\left(  \pi
_{1}\left(  1\right)  ,\pi_{1}\left(  2\right)  ,\ldots,\pi_{1}\left(
n\right)  \right)  $ and $\Phi\left(  \pi_{2}\right)  =\left(  \pi_{2}\left(
1\right)  ,\pi_{2}\left(  2\right)  ,\ldots,\pi_{2}\left(  n\right)  \right)
$. Thus,%
\[
\left(  \pi_{1}\left(  1\right)  ,\pi_{1}\left(  2\right)  ,\ldots,\pi
_{1}\left(  n\right)  \right)  =\Phi\left(  \pi_{1}\right)  =\Phi\left(
\pi_{2}\right)  =\left(  \pi_{2}\left(  1\right)  ,\pi_{2}\left(  2\right)
,\ldots,\pi_{2}\left(  n\right)  \right)  .
\]
In other words,%
\[
\pi_{1}\left(  i\right)  =\pi_{2}\left(  i\right)
\ \ \ \ \ \ \ \ \ \ \text{for every }i\in\left\{  1,2,\ldots,n\right\}  .
\]
In other words,%
\[
\pi_{1}\left(  i\right)  =\pi_{2}\left(  i\right)
\ \ \ \ \ \ \ \ \ \ \text{for every }i\in E
\]
(since $E=\left\{  1,2,\ldots,n\right\}  $). In other words, $\pi_{1}=\pi_{2}%
$.
\par
Now, forget that we fixed $\pi_{1}$ and $\pi_{2}$. We thus have shown that if
$\pi_{1}$ and $\pi_{2}$ are two elements of $Z$ such that $\Phi\left(  \pi
_{1}\right)  =\Phi\left(  \pi_{2}\right)  $, then $\pi_{1}=\pi_{2}$. In other
words, the map $\Phi$ is injective, qed.}. In other words, the map $\Phi$ is
bijective, i.e., a bijection. In other words, the map $Z\rightarrow
W,\ \pi\mapsto\left(  \pi\left(  1\right)  ,\pi\left(  2\right)  ,\ldots
,\pi\left(  n\right)  \right)  $ is a bijection (since $\Phi$ is the map
$Z\rightarrow W,\ \pi\mapsto\left(  \pi\left(  1\right)  ,\pi\left(  2\right)
,\ldots,\pi\left(  n\right)  \right)  $).

Every $\pi\in Z$ satisfies%
\begin{equation}
\mathbf{x}_{\pi}=x_{\pi\left(  1\right)  }x_{\pi\left(  2\right)  }\cdots
x_{\pi\left(  n\right)  } \label{pf.prop.example.Gamma.c3.x}%
\end{equation}
\footnote{\textit{Proof of (\ref{pf.prop.example.Gamma.c3.x}):} Let $\pi\in
Z$. The definition of $\mathbf{x}_{\pi}$ yields%
\begin{align*}
{\mathbf{x}}_{\pi}  &  =\prod_{e\in E}x_{\pi\left(  e\right)  }=\prod
_{e\in\left\{  1,2,\ldots,n\right\}  }x_{\pi\left(  e\right)  }%
\ \ \ \ \ \ \ \ \ \ \left(  \text{since }E=\left\{  1,2,\ldots,n\right\}
\right) \\
&  =x_{\pi\left(  1\right)  }x_{\pi\left(  2\right)  }\cdots x_{\pi\left(
n\right)  }.
\end{align*}
This proves (\ref{pf.prop.example.Gamma.c3.x}).}. Now,
(\ref{pf.prop.example.Gamma.c3.Gamma=}) becomes%
\begin{align*}
\Gamma\left(  \left(  E,<_{1},<_{2}\right)  \right)   &  =\sum_{\pi\in
Z}\underbrace{\mathbf{x}_{\pi}}_{\substack{=x_{\pi\left(  1\right)  }%
x_{\pi\left(  2\right)  }\cdots x_{\pi\left(  n\right)  }\\\text{(by
(\ref{pf.prop.example.Gamma.c3.x}))}}}=\sum_{\pi\in Z}x_{\pi\left(  1\right)
}x_{\pi\left(  2\right)  }\cdots x_{\pi\left(  n\right)  }\\
&  =\underbrace{\sum_{\left(  i_{1},i_{2},\ldots,i_{n}\right)  \in W}}%
_{=\sum_{\substack{i_{1}\leq i_{2}\leq\cdots\leq i_{n};\\i_{j}<i_{j+1}\text{
whenever }j\in D\left(  \alpha\right)  }}}x_{i_{1}}x_{i_{2}}\cdots x_{i_{n}}\\
&  \ \ \ \ \ \ \ \ \ \ \left(
\begin{array}
[c]{c}%
\text{here, we have substituted }\left(  i_{1},i_{2},\ldots,i_{n}\right)
\text{ for }\left(  \pi\left(  1\right)  ,\pi\left(  2\right)  ,\ldots
,\pi\left(  n\right)  \right)  \text{,}\\
\text{since the map }Z\rightarrow W,\ \pi\mapsto\left(  \pi\left(  1\right)
,\pi\left(  2\right)  ,\ldots,\pi\left(  n\right)  \right) \\
\text{is a bijection}%
\end{array}
\right) \\
&  =\sum_{\substack{i_{1}\leq i_{2}\leq\cdots\leq i_{n};\\i_{j}<i_{j+1}\text{
whenever }j\in D\left(  \alpha\right)  }}x_{i_{1}}x_{i_{2}}\cdots x_{i_{n}}\\
&  =\sum_{\beta\text{ is a composition of }n;\ D\left(  \beta\right)
\supseteq D\left(  \alpha\right)  }M_{\beta}%
\end{align*}
(by Corollary \ref{cor.Dalpha.F}). This proves Proposition
\ref{prop.example.Gamma.c3}.
\end{proof}

\begin{definition}
\label{def.Falpha}Let $\alpha=\left(  \alpha_{1},\alpha_{2},\ldots
,\alpha_{\ell}\right)  $ be a composition of a nonnegative integer $n$. Define
a set $D\left(  \alpha\right)  $ as in Definition \ref{def.Dalpha}. Define a
formal power series $F_{\alpha}\in\mathbf{k}\left[  \left[  x_{1},x_{2}%
,x_{3},\ldots\right]  \right]  $ by%
\begin{equation}
F_{\alpha}=\sum_{\substack{i_{1}\leq i_{2}\leq\cdots\leq i_{n};\\i_{j}%
<i_{j+1}\text{ whenever }j\in D\left(  \alpha\right)  }}x_{i_{1}}x_{i_{2}%
}\cdots x_{i_{n}}. \label{eq.def.Falpha.1}%
\end{equation}
This formal power series $F_{\alpha}$ is called the \textit{fundamental
quasisymmetric function} corresponding to the composition $\alpha$. We have%
\[
F_{\alpha}=\sum_{\substack{i_{1}\leq i_{2}\leq\cdots\leq i_{n};\\i_{j}%
<i_{j+1}\text{ whenever }j\in D\left(  \alpha\right)  }}x_{i_{1}}x_{i_{2}%
}\cdots x_{i_{n}}=\sum_{\beta\text{ is a composition of }n;\ D\left(
\beta\right)  \supseteq D\left(  \alpha\right)  }M_{\beta}%
\]
(by Proposition \ref{prop.example.Gamma.c3}), so that $F_{\alpha}%
\in\operatorname*{QSym}$.
\end{definition}

\subsection{The antipode of $M_{\alpha}$}

Next, we shall prove a fact which is (in a sense) similar to Corollary
\ref{cor.Dalpha.F}:

\begin{corollary}
\label{cor.Dalpha.SM}Let $\alpha$ be a composition of a nonnegative integer
$n$. Then,
\[
\sum_{\substack{i_{1}\leq i_{2}\leq\cdots\leq i_{n};\\\left\{  j\in\left[
n-1\right]  \ \mid\ i_{j}<i_{j+1}\right\}  \subseteq D\left(  \alpha\right)
}}x_{i_{1}}x_{i_{2}}\cdots x_{i_{n}}=\sum_{\substack{\beta\text{ is a
composition of }n;\\D\left(  \beta\right)  \subseteq D\left(  \alpha\right)
}}M_{\beta}.
\]
(Here, we are using the notations of Definition \ref{def.k} and Definition
\ref{def.Dalpha}.)
\end{corollary}

\begin{proof}
[Proof of Corollary \ref{cor.Dalpha.SM}.]Proposition \ref{prop.Dalpha.comp}
yields that the maps $D$ and $\operatorname*{comp}$ are mutually inverse.
Thus, the map $D$ is invertible, i.e., a bijection.

Now,%
\begin{align*}
&  \sum_{\substack{i_{1}\leq i_{2}\leq\cdots\leq i_{n};\\\left\{  j\in\left[
n-1\right]  \ \mid\ i_{j}<i_{j+1}\right\}  \subseteq D\left(  \alpha\right)
}}x_{i_{1}}x_{i_{2}}\cdots x_{i_{n}}\\
&  =\underbrace{\sum_{\substack{G\subseteq\left[  n-1\right]  ;\\G\subseteq
D\left(  \alpha\right)  }}}_{=\sum_{\substack{G\in\mathcal{P}\left(  \left[
n-1\right]  \right)  ;\\G\subseteq D\left(  \alpha\right)  }}}\sum
_{\substack{i_{1}\leq i_{2}\leq\cdots\leq i_{n};\\\left\{  j\in\left[
n-1\right]  \ \mid\ i_{j}<i_{j+1}\right\}  =G}}x_{i_{1}}x_{i_{2}}\cdots
x_{i_{n}}\\
&  \ \ \ \ \ \ \ \ \ \ \left(  \text{since }\left\{  j\in\left[  n-1\right]
\ \mid\ i_{j}<i_{j+1}\right\}  \subseteq\left[  n-1\right]  \text{ for every
}\left(  i_{1}\leq i_{2}\leq\cdots\leq i_{n}\right)  \right) \\
&  =\sum_{\substack{G\in\mathcal{P}\left(  \left[  n-1\right]  \right)
;\\G\subseteq D\left(  \alpha\right)  }}\sum_{\substack{i_{1}\leq i_{2}%
\leq\cdots\leq i_{n};\\\left\{  j\in\left[  n-1\right]  \ \mid\ i_{j}%
<i_{j+1}\right\}  =G}}x_{i_{1}}x_{i_{2}}\cdots x_{i_{n}}\\
&  =\sum_{\substack{\beta\in\operatorname*{Comp}\nolimits_{n};\\D\left(
\beta\right)  \subseteq D\left(  \alpha\right)  }}\sum_{\substack{i_{1}\leq
i_{2}\leq\cdots\leq i_{n};\\\left\{  j\in\left[  n-1\right]  \ \mid
\ i_{j}<i_{j+1}\right\}  =D\left(  \beta\right)  }}x_{i_{1}}x_{i_{2}}\cdots
x_{i_{n}}%
\end{align*}
(here, we have substituted $D\left(  \beta\right)  $ for $G$ in the outer sum,
since the map $D:\operatorname*{Comp}\nolimits_{n}\rightarrow\mathcal{P}%
\left(  \left[  n-1\right]  \right)  $ is a bijection). Comparing this with%
\begin{align*}
&  \underbrace{\sum_{\substack{\beta\text{ is a composition of }n;\\D\left(
\beta\right)  \subseteq D\left(  \alpha\right)  }}}_{\substack{=\sum
_{\substack{\beta\in\operatorname*{Comp}\nolimits_{n};\\D\left(  \beta\right)
\subseteq D\left(  \alpha\right)  }}\\\text{(since }\operatorname*{Comp}%
\nolimits_{n}\text{ is the set}\\\text{of all compositions of }n\text{)}%
}}\underbrace{M_{\beta}}_{\substack{=\sum_{\substack{i_{1}\leq i_{2}\leq
\cdots\leq i_{n};\\\left\{  j\in\left[  n-1\right]  \ \mid\ i_{j}%
<i_{j+1}\right\}  =D\left(  \beta\right)  }}x_{i_{1}}x_{i_{2}}\cdots x_{i_{n}%
}\\\text{(by Proposition \ref{prop.Malpha.D} (applied to }\beta\text{ instead
of }\alpha\text{))}}}\\
&  =\sum_{\substack{\beta\in\operatorname*{Comp}\nolimits_{n};\\D\left(
\beta\right)  \subseteq D\left(  \alpha\right)  }}\sum_{\substack{i_{1}\leq
i_{2}\leq\cdots\leq i_{n};\\\left\{  j\in\left[  n-1\right]  \ \mid
\ i_{j}<i_{j+1}\right\}  =D\left(  \beta\right)  }}x_{i_{1}}x_{i_{2}}\cdots
x_{i_{n}},
\end{align*}
we obtain%
\[
\sum_{\substack{i_{1}\leq i_{2}\leq\cdots\leq i_{n};\\\left\{  j\in\left[
n-1\right]  \ \mid\ i_{j}<i_{j+1}\right\}  \subseteq D\left(  \alpha\right)
}}x_{i_{1}}x_{i_{2}}\cdots x_{i_{n}}=\sum_{\substack{\beta\text{ is a
composition of }n;\\D\left(  \beta\right)  \subseteq D\left(  \alpha\right)
}}M_{\beta}.
\]
This proves Corollary \ref{cor.Dalpha.SM}.
\end{proof}

Next, we shall prove an analogue of Proposition \ref{prop.example.Gamma.b}:

\begin{proposition}
\label{prop.example.Gamma.Sb}Let $\ell\in\mathbb{N}$. Let $E=\left\{
1,2,\ldots,\ell\right\}  $. Let $<_{1}$ be the restriction of the standard
relation $<$ on $\mathbb{Z}$ to the subset $E$. (Thus, two elements $e$ and
$f$ of $E$ satisfy $e<_{1}f$ if and only if $e<f$.) Let $>_{1}$ be the
opposite relation of $<_{1}$. (Thus, two elements $e$ and $f$ of $E$ satisfy
$e>_{1}f$ if and only if $f<_{1}e$.) Let ${\mathbf{E}}^{\prime}=\left(
E,>_{1},>_{1}\right)  $.

\begin{enumerate}
\item[(a)] Then, $\mathbf{E}^{\prime}$ is a special double poset.

\item[(b)] Let $w:E\rightarrow\left\{  1,2,3,\ldots\right\}  $ be any map. Set
$\alpha=\left(  w\left(  1\right)  ,w\left(  2\right)  ,\ldots,w\left(
\ell\right)  \right)  $. Then, $\alpha$ is a composition. Write $\alpha$ in
the form $\alpha=\left(  \alpha_{1},\alpha_{2},\ldots,\alpha_{\ell}\right)  $.
We have%
\[
\Gamma\left(  {\mathbf{E}}^{\prime},w\right)  =\sum_{i_{1}\geq i_{2}\geq
\cdots\geq i_{\ell}}x_{i_{1}}^{\alpha_{1}}x_{i_{2}}^{\alpha_{2}}\cdots
x_{i_{\ell}}^{\alpha_{\ell}}.
\]

\end{enumerate}
\end{proposition}

\begin{proof}
[Proof of Proposition \ref{prop.example.Gamma.Sb}.](a) The relation $<_{1}$ is
a total order (since it is a restriction of the relation $<$ on $\mathbb{Z}$,
which is a total order). Hence, the relation $>_{1}$ is a total order as well
(since it is the opposite relation of the total order $<_{1}$). Thus, $\left(
E,>_{1},>_{1}\right)  $ is a special double poset (by the definition of
\textquotedblleft special\textquotedblright). In other words, $\mathbf{E}%
^{\prime}$ is a special double poset (since ${\mathbf{E}}^{\prime}=\left(
E,>_{1},>_{1}\right)  $). This proves Proposition \ref{prop.example.Gamma.Sb} (a).

(b) The map $w$ is a map $E\rightarrow\left\{  1,2,3,\ldots\right\}  $. In
other words, the map $w$ is a map $\left\{  1,2,\ldots,\ell\right\}
\rightarrow\left\{  1,2,3,\ldots\right\}  $ (since $E=\left\{  1,2,\ldots
,\ell\right\}  $). Hence, for every $i\in\left\{  1,2,\ldots,\ell\right\}  $,
we have $w\left(  i\right)  \in\left\{  1,2,3,\ldots\right\}  $. Thus,
$\left(  w\left(  1\right)  ,w\left(  2\right)  ,\ldots,w\left(  \ell\right)
\right)  $ is a sequence of positive integers, i.e., a composition. In other
words, $\alpha$ is a composition (since $\alpha=\left(  w\left(  1\right)
,w\left(  2\right)  ,\ldots,w\left(  \ell\right)  \right)  $).

We have $\alpha=\left(  w\left(  1\right)  ,w\left(  2\right)  ,\ldots
,w\left(  \ell\right)  \right)  $. Thus, $\left(  \alpha_{1},\alpha_{2}%
,\ldots,\alpha_{\ell}\right)  =\alpha=\left(  w\left(  1\right)  ,w\left(
2\right)  ,\ldots,w\left(  \ell\right)  \right)  $. In other words,
$\alpha_{k}=w\left(  k\right)  $ for each $k\in\left\{  1,2,\ldots
,\ell\right\}  $. Hence,
\begin{equation}
\sum_{i_{1}\geq i_{2}\geq\cdots\geq i_{\ell}}\underbrace{x_{i_{1}}^{\alpha
_{1}}x_{i_{2}}^{\alpha_{2}}\cdots x_{i_{\ell}}^{\alpha_{\ell}}}%
_{\substack{=x_{i_{1}}^{w\left(  1\right)  }x_{i_{2}}^{w\left(  2\right)
}\cdots x_{i_{\ell}}^{w\left(  \ell\right)  }\\\text{(since }\alpha
_{k}=w\left(  k\right)  \text{ for each }k\in\left\{  1,2,\ldots,\ell\right\}
\text{)}}}=\sum_{i_{1}\geq i_{2}\geq\cdots\geq i_{\ell}}x_{i_{1}}^{w\left(
1\right)  }x_{i_{2}}^{w\left(  2\right)  }\cdots x_{i_{\ell}}^{w\left(
\ell\right)  }. \label{pf.prop.example.Gamma.Sb.b.Malpha=}%
\end{equation}

It remains to prove that $\Gamma\left(  \mathbf{E}^{\prime},w\right)
=\sum_{i_{1}\geq i_{2}\geq\cdots\geq i_{\ell}}x_{i_{1}}^{\alpha_{1}}x_{i_{2}%
}^{\alpha_{2}}\cdots x_{i_{\ell}}^{\alpha_{\ell}}$. The order $>_{1}$ is an
extension of the order $>_{1}$ (obviously). Thus, Proposition
\ref{prop.example.weaklinc} (applied to $\mathbf{E}^{\prime}$, $>_{1}$ and
$>_{1}$ instead of $\mathbf{E}$, $<_{1}$ and $<_{2}$) shows that the
$\mathbf{E}^{\prime}$-partitions are precisely the weakly increasing maps from
the poset $\left(  E,>_{1}\right)  $ to the totally ordered set $\left\{
1,2,3,\ldots\right\}  $.

On the other hand, let $\mathcal{J}$ denote the set of all length-$\ell$
weakly decreasing sequences of positive integers. In other words,%
\[
\mathcal{J}=\left\{  \left(  i_{1},i_{2},\ldots,i_{\ell}\right)  \in\left\{
1,2,3,\ldots\right\}  ^{\ell}\ \mid\ i_{1}\geq i_{2}\geq\cdots\geq i_{\ell
}\right\}  .
\]
Thus,%
\[
\sum_{\left(  i_{1},i_{2},\ldots,i_{\ell}\right)  \in\mathcal{J}}%
=\sum_{\substack{\left(  i_{1},i_{2},\ldots,i_{\ell}\right)  \in\left\{
1,2,3,\ldots\right\}  ^{\ell};\\i_{1}\geq i_{2}\geq\cdots\geq i_{\ell}}%
}=\sum_{i_{1}\geq i_{2}\geq\cdots\geq i_{\ell}}%
\]
(an equality between summation signs).

Let $Z$ denote the set of all $\mathbf{E}^{\prime}$-partitions.

For every $\phi\in Z$, we have $\left(  \phi\left(  1\right)  ,\phi\left(
2\right)  ,\ldots,\phi\left(  \ell\right)  \right)  \in\mathcal{J}%
$\ \ \ \ \footnote{\textit{Proof.} Let $\phi\in Z$. Thus, $\phi$ is an
$\mathbf{E}^{\prime}$-partition (since $Z$ is the set of all $\mathbf{E}%
^{\prime}$-partitions). In other words, $\phi$ is a weakly increasing map from
the poset $\left(  E,>_{1}\right)  $ to the totally ordered set $\left\{
1,2,3,\ldots\right\}  $ (since the $\mathbf{E}^{\prime}$-partitions are
precisely the weakly increasing maps from the poset $\left(  E,>_{1}\right)  $
to the totally ordered set $\left\{  1,2,3,\ldots\right\}  $). In other words,
$\phi$ is a map $E\rightarrow\left\{  1,2,3,\ldots\right\}  $ which has the
property that if $e$ and $f$ are two elements of $E$ satisfying $e>_{1}f$,
then%
\begin{equation}
\phi\left(  e\right)  \leq\phi\left(  f\right)
\label{pf.prop.example.Gamma.Sb.b.fn1.1}%
\end{equation}
(by the definition of a \textquotedblleft weakly increasing
map\textquotedblright).
\par
The map $\phi$ is a map $E\rightarrow\left\{  1,2,3,\ldots\right\}  $. In
other words, the map $\phi$ is a map $\left\{  1,2,\ldots,\ell\right\}
\rightarrow\left\{  1,2,3,\ldots\right\}  $ (since $E=\left\{  1,2,\ldots
,\ell\right\}  $). Thus, $\left(  \phi\left(  1\right)  ,\phi\left(  2\right)
,\ldots,\phi\left(  \ell\right)  \right)  $ is an element of $\left\{
1,2,3,\ldots\right\}  ^{\ell}$.
\par
Now, let $i$ and $j$ be two elements of $\left\{  1,2,\ldots,\ell\right\}  $
satisfying $i<j$. We have $i\in\left\{  1,2,\ldots,\ell\right\}  =E$; thus,
$\phi\left(  i\right)  $ is well-defined. Similarly, $\phi\left(  j\right)  $
is well-defined. The definition of the relation $<_{1}$ shows that the
relation $<_{1}$ is the restriction of the standard relation $<$ on
$\mathbb{Z}$ to the subset $E$. Thus, $i<_{1}j$ if and only if $i<j$. Hence,
$i<_{1}j$ (since $i<j$). Thus, $j>_{1}i$ (since $>_{1}$ is the opposite
relation of $<_{1}$). Hence, (\ref{pf.prop.example.Gamma.Sb.b.fn1.1}) (applied
to $e=j$ and $f=i$) shows that $\phi\left(  j\right)  \leq\phi\left(
i\right)  $. In other words, $\phi\left(  i\right)  \geq\phi\left(  j\right)
$.
\par
Now, forget that we fixed $i$ and $j$. We thus have shown that if $i$ and $j$
are two elements of $\left\{  1,2,\ldots,\ell\right\}  $ satisfying $i<j$,
then $\phi\left(  i\right)  \geq\phi\left(  j\right)  $. In other words,
$\phi\left(  1\right)  \geq\phi\left(  2\right)  \geq\cdots\geq\phi\left(
\ell\right)  $.
\par
Now, $\left(  \phi\left(  1\right)  ,\phi\left(  2\right)  ,\ldots,\phi\left(
\ell\right)  \right)  $ is an element of $\left\{  1,2,3,\ldots\right\}
^{\ell}$ and satisfies $\phi\left(  1\right)  \geq\phi\left(  2\right)
\geq\cdots\geq\phi\left(  \ell\right)  $. In other words,%
\begin{align*}
&  \left(  \phi\left(  1\right)  ,\phi\left(  2\right)  ,\ldots,\phi\left(
\ell\right)  \right) \\
&  \in\left\{  \left(  i_{1},i_{2},\ldots,i_{\ell}\right)  \in\left\{
1,2,3,\ldots\right\}  ^{\ell}\ \mid\ i_{1}\geq i_{2}\geq\cdots\geq i_{\ell
}\right\}  =\mathcal{J},
\end{align*}
qed.}. Hence, we can define a map $\Phi:Z\rightarrow\mathcal{J}$ by%
\[
\left(  \Phi\left(  \phi\right)  =\left(  \phi\left(  1\right)  ,\phi\left(
2\right)  ,\ldots,\phi\left(  \ell\right)  \right)
\ \ \ \ \ \ \ \ \ \ \text{for every }\phi\in Z\right)  .
\]
Consider this map $\Phi$. This map $\Phi$ is
injective\footnote{\textit{Proof.} Let $\phi_{1}$ and $\phi_{2}$ be two
elements of $Z$ such that $\Phi\left(  \phi_{1}\right)  =\Phi\left(  \phi
_{2}\right)  $. We shall show that $\phi_{1}=\phi_{2}$.
\par
The definition of $\Phi$ shows that $\Phi\left(  \phi_{1}\right)  =\left(
\phi_{1}\left(  1\right)  ,\phi_{1}\left(  2\right)  ,\ldots,\phi_{1}\left(
\ell\right)  \right)  $. The definition of $\Phi$ shows that $\Phi\left(
\phi_{2}\right)  =\left(  \phi_{2}\left(  1\right)  ,\phi_{2}\left(  2\right)
,\ldots,\phi_{2}\left(  \ell\right)  \right)  $. Hence,%
\[
\left(  \phi_{1}\left(  1\right)  ,\phi_{1}\left(  2\right)  ,\ldots,\phi
_{1}\left(  \ell\right)  \right)  =\Phi\left(  \phi_{1}\right)  =\Phi\left(
\phi_{2}\right)  =\left(  \phi_{2}\left(  1\right)  ,\phi_{2}\left(  2\right)
,\ldots,\phi_{2}\left(  \ell\right)  \right)  .
\]
In other words, $\phi_{1}\left(  i\right)  =\phi_{2}\left(  i\right)  $ for
each $i\in\left\{  1,2,\ldots,\ell\right\}  $. In other words, $\phi
_{1}\left(  i\right)  =\phi_{2}\left(  i\right)  $ for each $i\in E$ (since
$E=\left\{  1,2,\ldots,\ell\right\}  $). In other words, $\phi_{1}=\phi_{2}$.
\par
Now, forget that we fixed $\phi_{1}$ and $\phi_{2}$. We thus have shown that
if $\phi_{1}$ and $\phi_{2}$ are two elements of $Z$ such that $\Phi\left(
\phi_{1}\right)  =\Phi\left(  \phi_{2}\right)  $, then $\phi_{1}=\phi_{2}$. In
other words, the map $\Phi$ is injective, qed.} and
surjective\footnote{\textit{Proof.} Let $\mathbf{j}\in\mathcal{J}$. We shall
show that $\mathbf{j}\in\Phi\left(  Z\right)  $.
\par
We have $\mathbf{j}\in\mathcal{J}=\left\{  \left(  i_{1},i_{2},\ldots,i_{\ell
}\right)  \in\left\{  1,2,3,\ldots\right\}  ^{\ell}\ \mid\ i_{1}\geq i_{2}%
\geq\cdots\geq i_{\ell}\right\}  $. In other words, $\mathbf{j}$ has the form
$\left(  i_{1},i_{2},\ldots,i_{\ell}\right)  $ for some $\left(  i_{1}%
,i_{2},\ldots,i_{\ell}\right)  \in\left\{  1,2,3,\ldots\right\}  ^{\ell}$
satisfying $i_{1}\geq i_{2}\geq\cdots\geq i_{\ell}$. Consider this $\left(
i_{1},i_{2},\ldots,i_{\ell}\right)  $. Thus, $\mathbf{j}=\left(  i_{1}%
,i_{2},\ldots,i_{\ell}\right)  $.
\par
We have $i_{e}\in\left\{  1,2,3,\ldots\right\}  $ for every $e\in\left\{
1,2,\ldots,\ell\right\}  $ (since $\left(  i_{1},i_{2},\ldots,i_{\ell}\right)
\in\left\{  1,2,3,\ldots\right\}  ^{\ell}$). In other words, $i_{e}\in\left\{
1,2,3,\ldots\right\}  $ for every $e\in E$ (since $E=\left\{  1,2,\ldots
,\ell\right\}  $). Thus, we can define a map $\phi:E\rightarrow\left\{
1,2,3,\ldots\right\}  $ by $\left(  \phi\left(  e\right)  =i_{e}\text{ for
every }e\in E\right)  $. Consider this map $\phi$.
\par
We have $i_{1}\geq i_{2}\geq\cdots\geq i_{\ell}$. In other words, if $e$ and
$f$ are two elements of $\left\{  1,2,\ldots,\ell\right\}  $ such that $e<f$,
then%
\begin{equation}
i_{e}\geq i_{f}. \label{pf.prop.example.Gamma.Sb.b.fn3.1}%
\end{equation}
\par
Let $e$ and $f$ be two elements of $E$ satisfying $e>_{1}f$. We have $e>_{1}%
f$. In other words, $f<_{1}e$ (since $>_{1}$ is the opposite relation of
$<_{1}$). In other words, $f<e$ (since $<_{1}$ is the restriction of the
standard relation $<$ on $\mathbb{Z}$ to the subset $E$). Thus,
(\ref{pf.prop.example.Gamma.Sb.b.fn3.1}) (applied to $f$ and $e$ instead of
$e$ and $f$) shows that $i_{f}\geq i_{e}$. In other words, $i_{e}\leq i_{f}$.
But the definition of $\phi$ shows that $\phi\left(  e\right)  =i_{e}$ and
$\phi\left(  f\right)  =i_{f}$. Hence, $\phi\left(  e\right)  =i_{e}\leq
i_{f}=\phi\left(  f\right)  $.
\par
Now, forget that we fixed $e$ and $f$. We thus have shown that if $e$ and $f$
are two elements of $E$ satisfying $e>_{1}f$, then $\phi\left(  e\right)
\leq\phi\left(  f\right)  $. In other words, $\phi$ is a weakly increasing map
from the poset $\left(  E,>_{1}\right)  $ to the totally ordered set $\left\{
1,2,3,\ldots\right\}  $ (by the definition of a \textquotedblleft weakly
increasing map\textquotedblright). In other words, $\phi$ is an $\mathbf{E}%
^{\prime}$-partition (since the $\mathbf{E}^{\prime}$-partitions are precisely
the weakly increasing maps from the poset $\left(  E,>_{1}\right)  $ to the
totally ordered set $\left\{  1,2,3,\ldots\right\}  $). In other words,
$\phi\in Z$ (since $Z$ is the set of all $\mathbf{E}^{\prime}$-partitions).
\par
We have $\phi\left(  e\right)  =i_{e}$ for every $e\in E$ (by the definition
of $\phi$). In other words, $\phi\left(  e\right)  =i_{e}$ for every
$e\in\left\{  1,2,\ldots,\ell\right\}  $ (since $E=\left\{  1,2,\ldots
,\ell\right\}  $). In other words, $\left(  \phi\left(  1\right)  ,\phi\left(
2\right)  ,\ldots,\phi\left(  \ell\right)  \right)  =\left(  i_{1}%
,i_{2},\ldots,i_{\ell}\right)  $.
\par
Now, the definition of $\Phi$ yields%
\[
\Phi\left(  \phi\right)  =\left(  \phi\left(  1\right)  ,\phi\left(  2\right)
,\ldots,\phi\left(  \ell\right)  \right)  =\left(  i_{1},i_{2},\ldots,i_{\ell
}\right)  =\mathbf{j}.
\]
Thus, $\mathbf{j}=\Phi\left(  \underbrace{\phi}_{\in Z}\right)  \in\Phi\left(
Z\right)  $.
\par
Now, forget that we fixed $\mathbf{j}$. We thus have shown that $\mathbf{j}%
\in\Phi\left(  Z\right)  $ for every $\mathbf{j}\in\mathcal{J}$. In other
words, $\mathcal{J}\subseteq\Phi\left(  Z\right)  $. In other words, the map
$\Phi$ is surjective, qed.}. In other words, the map $\Phi$ is bijective.
Thus, $\Phi$ is a bijection. In other words, the map%
\begin{equation}
Z\rightarrow\mathcal{J},\ \ \ \ \ \ \ \ \ \ \pi\mapsto\left(  \pi\left(
1\right)  ,\pi\left(  2\right)  ,\ldots,\pi\left(  \ell\right)  \right)
\label{pf.prop.example.Gamma.Sb.b.themap}%
\end{equation}
is a bijection (since the map (\ref{pf.prop.example.Gamma.Sb.b.themap}) is the
map $\Phi$).

For every $\pi\in Z$, we have%
\begin{equation}
\mathbf{x}_{\pi,w}=x_{\pi\left(  1\right)  }^{w\left(  1\right)  }%
x_{\pi\left(  2\right)  }^{w\left(  2\right)  }\cdots x_{\pi\left(
\ell\right)  }^{w\left(  \ell\right)  } \label{pf.prop.example.Gamma.Sb.b.x=}%
\end{equation}
\footnote{\textit{Proof of (\ref{pf.prop.example.Gamma.Sb.b.x=}):} Let $\pi\in
Z$. Then, the definition of $\mathbf{x}_{\pi,w}$ yields%
\begin{align*}
{\mathbf{x}}_{\pi,w}  &  =\prod_{e\in E}x_{\pi\left(  e\right)  }^{w\left(
e\right)  }=\underbrace{\prod_{e\in\left\{  1,2,\ldots,\ell\right\}  }%
}_{=\prod_{e=1}^{\ell}}x_{\pi\left(  e\right)  }^{w\left(  e\right)
}\ \ \ \ \ \ \ \ \ \ \left(  \text{since }E=\left\{  1,2,\ldots,\ell\right\}
\right) \\
&  =\prod_{e=1}^{\ell}x_{\pi\left(  e\right)  }^{w\left(  e\right)  }%
=x_{\pi\left(  1\right)  }^{w\left(  1\right)  }x_{\pi\left(  2\right)
}^{w\left(  2\right)  }\cdots x_{\pi\left(  \ell\right)  }^{w\left(
\ell\right)  }.
\end{align*}
This proves (\ref{pf.prop.example.Gamma.Sb.b.x=}).}.

Now, the definition of $\Gamma\left(  \mathbf{E}^{\prime},w\right)  $ yields%
\begin{align*}
\Gamma\left(  {\mathbf{E}}^{\prime},w\right)   &  =\underbrace{\sum_{\pi\text{
is an }{\mathbf{E}}^{\prime}\text{-partition}}}_{\substack{=\sum_{\pi\in
Z}\\\text{(since }Z\text{ is the set of}\\\text{all }\mathbf{E}^{\prime
}\text{-partitions)}}}{\mathbf{x}}_{\pi,w}=\sum_{\pi\in Z}%
\underbrace{{\mathbf{x}}_{\pi,w}}_{\substack{=x_{\pi\left(  1\right)
}^{w\left(  1\right)  }x_{\pi\left(  2\right)  }^{w\left(  2\right)  }\cdots
x_{\pi\left(  \ell\right)  }^{w\left(  \ell\right)  }\\\text{(by
(\ref{pf.prop.example.Gamma.Sb.b.x=}))}}}=\sum_{\pi\in Z}x_{\pi\left(
1\right)  }^{w\left(  1\right)  }x_{\pi\left(  2\right)  }^{w\left(  2\right)
}\cdots x_{\pi\left(  \ell\right)  }^{w\left(  \ell\right)  }\\
&  =\underbrace{\sum_{\left(  i_{1},i_{2},\ldots,i_{\ell}\right)
\in\mathcal{J}}}_{=\sum_{i_{1}\geq i_{2}\geq\cdots\geq i_{\ell}}}x_{i_{1}%
}^{w\left(  1\right)  }x_{i_{2}}^{w\left(  2\right)  }\cdots x_{i_{\ell}%
}^{w\left(  \ell\right)  }\\
&  \ \ \ \ \ \ \ \ \ \ \left(
\begin{array}
[c]{c}%
\text{here, we have substituted }\left(  i_{1},i_{2},\ldots,i_{\ell}\right)
\text{ for }\left(  \pi\left(  1\right)  ,\pi\left(  2\right)  ,\ldots
,\pi\left(  \ell\right)  \right)  \text{,}\\
\text{since the map }Z\rightarrow\mathcal{J},\ \pi\mapsto\left(  \pi\left(
1\right)  ,\pi\left(  2\right)  ,\ldots,\pi\left(  \ell\right)  \right) \\
\text{is a bijection}%
\end{array}
\right) \\
&  =\sum_{i_{1}\geq i_{2}\geq\cdots\geq i_{\ell}}x_{i_{1}}^{w\left(  1\right)
}x_{i_{2}}^{w\left(  2\right)  }\cdots x_{i_{\ell}}^{w\left(  \ell\right)
}=\sum_{i_{1}\geq i_{2}\geq\cdots\geq i_{\ell}}x_{i_{1}}^{\alpha_{1}}x_{i_{2}%
}^{\alpha_{2}}\cdots x_{i_{\ell}}^{\alpha_{\ell}}\ \ \ \ \ \ \ \ \ \ \left(
\text{by (\ref{pf.prop.example.Gamma.Sb.b.Malpha=})}\right)  .
\end{align*}
This completes the proof of Proposition \ref{prop.example.Gamma.Sb} (b).
\end{proof}

Our next claim is an analogue of Proposition \ref{prop.Malpha.D}:

\begin{proposition}
\label{prop.SMalpha.D}Let $\alpha=\left(  \alpha_{1},\alpha_{2},\ldots
,\alpha_{\ell}\right)  $ be a composition of a nonnegative integer $n$. Then,%
\[
\sum_{i_{1}\leq i_{2}\leq\cdots\leq i_{\ell}}x_{i_{1}}^{\alpha_{1}}x_{i_{2}%
}^{\alpha_{2}}\cdots x_{i_{\ell}}^{\alpha_{\ell}}=\sum_{\substack{i_{1}\leq
i_{2}\leq\cdots\leq i_{n};\\\left\{  j\in\left[  n-1\right]  \ \mid
\ i_{j}<i_{j+1}\right\}  \subseteq D\left(  \alpha\right)  }}x_{i_{1}}%
x_{i_{2}}\cdots x_{i_{n}}.
\]

\end{proposition}

\begin{proof}
[Proof of Proposition \ref{prop.SMalpha.D}.]Let $\mathcal{J}$ denote the set
of all length-$\ell$ weakly increasing sequences of positive integers. In
other words,%
\begin{equation}
\mathcal{J}=\left\{  \left(  i_{1},i_{2},\ldots,i_{\ell}\right)  \in\left\{
1,2,3,\ldots\right\}  ^{\ell}\ \mid\ i_{1}\leq i_{2}\leq\cdots\leq i_{\ell
}\right\}  . \label{pf.prop.SMalpha.D.J=}%
\end{equation}
Renaming the index $\left(  i_{1},i_{2},\ldots,i_{\ell}\right)  $ as $\left(
j_{1},j_{2},\ldots,j_{\ell}\right)  $ in this formula, we obtain%
\[
\mathcal{J}=\left\{  \left(  j_{1},j_{2},\ldots,j_{\ell}\right)  \in\left\{
1,2,3,\ldots\right\}  ^{\ell}\ \mid\ j_{1}\leq j_{2}\leq\cdots\leq j_{\ell
}\right\}  .
\]

Now,%
\begin{align}
&  \underbrace{\sum_{i_{1}\leq i_{2}\leq\cdots\leq i_{\ell}}}_{\substack{=\sum
_{\substack{\left(  i_{1},i_{2},\ldots,i_{\ell}\right)  \in\left\{
1,2,3,\ldots\right\}  ^{\ell};\\i_{1}\leq i_{2}\leq\cdots\leq i_{\ell}}%
}=\sum_{\left(  i_{1},i_{2},\ldots,i_{\ell}\right)  \in\mathcal{J}%
}\\\text{(since }\left\{  \left(  i_{1},i_{2},\ldots,i_{\ell}\right)
\in\left\{  1,2,3,\ldots\right\}  ^{\ell}\ \mid\ i_{1}\leq i_{2}\leq\cdots\leq
i_{\ell}\right\}  =\mathcal{J}\text{)}}}x_{i_{1}}^{\alpha_{1}}x_{i_{2}%
}^{\alpha_{2}}\cdots x_{i_{\ell}}^{\alpha_{\ell}}\nonumber\\
&  =\sum_{\left(  i_{1},i_{2},\ldots,i_{\ell}\right)  \in\mathcal{J}}x_{i_{1}%
}^{\alpha_{1}}x_{i_{2}}^{\alpha_{2}}\cdots x_{i_{\ell}}^{\alpha_{\ell}}%
=\sum_{\left(  j_{1},j_{2},\ldots,j_{\ell}\right)  \in\mathcal{J}}x_{j_{1}%
}^{\alpha_{1}}x_{j_{2}}^{\alpha_{2}}\cdots x_{j_{\ell}}^{\alpha_{\ell}}
\label{pf.prop.SMalpha.D.Malpha=}%
\end{align}
(here, we have renamed $\left(  i_{1},i_{2},\ldots,i_{\ell}\right)  $ as
$\left(  j_{1},j_{2},\ldots,j_{\ell}\right)  $ in the sum).

Define a set $\mathcal{I}$ by%
\begin{align}
\mathcal{I}  &  =\left\{  \left(  i_{1},i_{2},\ldots,i_{n}\right)  \in\left\{
1,2,3,\ldots\right\}  ^{n}\ \mid\ i_{1}\leq i_{2}\leq\cdots\leq i_{n}\right.
\nonumber\\
&  \ \ \ \ \ \ \ \ \ \ \left.  \text{and }\left\{  j\in\left[  n-1\right]
\ \mid\ i_{j}<i_{j+1}\right\}  \subseteq D\left(  \alpha\right)  \right\}  .
\label{pf.prop.SMalpha.D.I=}%
\end{align}
Thus, $\sum_{\substack{i_{1}\leq i_{2}\leq\cdots\leq i_{n};\\\left\{
j\in\left[  n-1\right]  \ \mid\ i_{j}<i_{j+1}\right\}  \subseteq D\left(
\alpha\right)  }}=\sum_{\left(  i_{1},i_{2},\ldots,i_{n}\right)
\in\mathcal{I}}$ (an equality between summation signs). Hence,%
\begin{equation}
\sum_{\substack{i_{1}\leq i_{2}\leq\cdots\leq i_{n};\\\left\{  j\in\left[
n-1\right]  \ \mid\ i_{j}<i_{j+1}\right\}  \subseteq D\left(  \alpha\right)
}}x_{i_{1}}x_{i_{2}}\cdots x_{i_{n}}=\sum_{\left(  i_{1},i_{2},\ldots
,i_{n}\right)  \in\mathcal{I}}x_{i_{1}}x_{i_{2}}\cdots x_{i_{n}}.
\label{pf.prop.SMalpha.D.RHS=}%
\end{equation}

The definition of $\mathcal{I}$ shows that%
\begin{align*}
\mathcal{I}  &  =\left\{  \left(  i_{1},i_{2},\ldots,i_{n}\right)  \in\left\{
1,2,3,\ldots\right\}  ^{n}\ \mid\ i_{1}\leq i_{2}\leq\cdots\leq i_{n}\right.
\\
&  \ \ \ \ \ \ \ \ \ \ \left.  \text{and }\left\{  j\in\left[  n-1\right]
\ \mid\ i_{j}<i_{j+1}\right\}  \subseteq D\left(  \alpha\right)  \right\} \\
&  =\left\{  \left(  k_{1},k_{2},\ldots,k_{n}\right)  \in\left\{
1,2,3,\ldots\right\}  ^{n}\ \mid\ k_{1}\leq k_{2}\leq\cdots\leq k_{n}\right.
\\
&  \ \ \ \ \ \ \ \ \ \ \left.  \text{and }\left\{  j\in\left[  n-1\right]
\ \mid\ k_{j}<k_{j+1}\right\}  \subseteq D\left(  \alpha\right)  \right\}
\end{align*}
(here, we have renamed the index $\left(  i_{1},i_{2},\ldots,i_{n}\right)  $
as $\left(  k_{1},k_{2},\ldots,k_{n}\right)  $).

Now, for every $i\in\left\{  0,1,\ldots,\ell\right\}  $, define a nonnegative
integer $s_{i}$ by
\[
s_{i}=\alpha_{1}+\alpha_{2}+\cdots+\alpha_{i}.
\]

Define a map $f:\left[  n\right]  \rightarrow\left[  \ell\right]  $ as in
Lemma \ref{lem.Dalpha.s2}.

Now, for every $\left(  i_{1},i_{2},\ldots,i_{n}\right)  \in\mathcal{I}$, we
have $\left(  i_{s_{1}},i_{s_{2}},\ldots,i_{s_{\ell}}\right)  \in\mathcal{J}%
$\ \ \ \ \footnote{\textit{Proof.} Let $\left(  i_{1},i_{2},\ldots
,i_{n}\right)  \in\mathcal{I}$. Thus,%
\begin{align*}
\left(  i_{1},i_{2},\ldots,i_{n}\right)   &  \in\mathcal{I}\\
&  =\left\{  \left(  k_{1},k_{2},\ldots,k_{n}\right)  \in\left\{
1,2,3,\ldots\right\}  ^{n}\ \mid\ k_{1}\leq k_{2}\leq\cdots\leq k_{n}\right.
\\
&  \ \ \ \ \ \ \ \ \ \ \left.  \text{and }\left\{  j\in\left[  n-1\right]
\ \mid\ k_{j}<k_{j+1}\right\}  \subseteq D\left(  \alpha\right)  \right\}  .
\end{align*}
In other words, $\left(  i_{1},i_{2},\ldots,i_{n}\right)  $ is an element of
$\left\{  1,2,3,\ldots\right\}  ^{n}$ satisfying $i_{1}\leq i_{2}\leq
\cdots\leq i_{n}$ and $\left\{  j\in\left[  n-1\right]  \ \mid\ i_{j}%
<i_{j+1}\right\}  \subseteq D\left(  \alpha\right)  $.
\par
Lemma \ref{lem.Dalpha.ItoJ} (b) shows that we have $\left(  i_{s_{1}}%
,i_{s_{2}},\ldots,i_{s_{\ell}}\right)  \in\left\{  1,2,3,\ldots\right\}
^{\ell}$ and $i_{s_{1}}\leq i_{s_{2}}\leq\cdots\leq i_{s_{\ell}}$. In other
words,%
\[
\left(  i_{s_{1}},i_{s_{2}},\ldots,i_{s_{\ell}}\right)  \in\left\{  \left(
j_{1},j_{2},\ldots,j_{\ell}\right)  \in\left\{  1,2,3,\ldots\right\}  ^{\ell
}\ \mid\ j_{1}\leq j_{2}\leq\cdots\leq j_{\ell}\right\}  =\mathcal{J},
\]
qed.}. Hence, we can define a map $\Phi:\mathcal{I}\rightarrow\mathcal{J}$ by
setting%
\[
\left(  \Phi\left(  i_{1},i_{2},\ldots,i_{n}\right)  =\left(  i_{s_{1}%
},i_{s_{2}},\ldots,i_{s_{\ell}}\right)  \ \ \ \ \ \ \ \ \ \ \text{for every
}\left(  i_{1},i_{2},\ldots,i_{n}\right)  \in\mathcal{I}\right)  .
\]
Consider this $\Phi$.

For every $\left(  i_{1},i_{2},\ldots,i_{n}\right)  \in\mathcal{I}$, we have%
\begin{equation}
i_{s_{f\left(  k\right)  }}=i_{k}\ \ \ \ \ \ \ \ \ \ \text{for every }%
k\in\left[  n\right]  \label{pf.prop.SMalpha.D.sf.1}%
\end{equation}
\footnote{\textit{Proof of (\ref{pf.prop.SMalpha.D.sf.1}):} Fix $\left(
i_{1},i_{2},\ldots,i_{n}\right)  \in\mathcal{I}$. We need to prove the
equality (\ref{pf.prop.SMalpha.D.sf.1}).
\par
We have%
\begin{align*}
\left(  i_{1},i_{2},\ldots,i_{n}\right)   &  \in\mathcal{I}\\
&  =\left\{  \left(  k_{1},k_{2},\ldots,k_{n}\right)  \in\left\{
1,2,3,\ldots\right\}  ^{n}\ \mid\ k_{1}\leq k_{2}\leq\cdots\leq k_{n}\right.
\\
&  \ \ \ \ \ \ \ \ \ \ \left.  \text{and }\left\{  j\in\left[  n-1\right]
\ \mid\ k_{j}<k_{j+1}\right\}  \subseteq D\left(  \alpha\right)  \right\}  .
\end{align*}
In other words, $\left(  i_{1},i_{2},\ldots,i_{n}\right)  $ is an element of
$\left\{  1,2,3,\ldots\right\}  ^{n}$ satisfying $i_{1}\leq i_{2}\leq
\cdots\leq i_{n}$ and $\left\{  j\in\left[  n-1\right]  \ \mid\ i_{j}%
<i_{j+1}\right\}  \subseteq D\left(  \alpha\right)  $. Hence, Lemma
\ref{lem.Dalpha.ItoJ} (a) shows that we have $i_{s_{f\left(  k\right)  }%
}=i_{k}$ for every $k\in\left[  n\right]  $. Hence,
(\ref{pf.prop.SMalpha.D.sf.1}) is proven.}.

Now, for every $\left(  h_{1},h_{2},\ldots,h_{\ell}\right)  \in\mathcal{J}$,
we have $\left(  h_{f\left(  1\right)  },h_{f\left(  2\right)  }%
,\ldots,h_{f\left(  n\right)  }\right)  \in\mathcal{I}$%
\ \ \ \ \footnote{\textit{Proof.} Let $\left(  h_{1},h_{2},\ldots,h_{\ell
}\right)  \in\mathcal{J}$. Thus, $\left(  h_{1},h_{2},\ldots,h_{\ell}\right)
\in\mathcal{J}=\left\{  \left(  i_{1},i_{2},\ldots,i_{\ell}\right)
\in\left\{  1,2,3,\ldots\right\}  ^{\ell}\ \mid\ i_{1}\leq i_{2}\leq\cdots\leq
i_{\ell}\right\}  $. In other words, $\left(  h_{1},h_{2},\ldots,h_{\ell
}\right)  $ is an $\ell$-tuple in $\left\{  1,2,3,\ldots\right\}  ^{\ell}$ and
satisfies $h_{1}\leq h_{2}\leq\cdots\leq h_{\ell}$. Hence, Lemma
\ref{lem.Dalpha.JtoI} (a) shows that we have $\left(  h_{f\left(  1\right)
},h_{f\left(  2\right)  },\ldots,h_{f\left(  n\right)  }\right)  \in\left\{
1,2,3,\ldots\right\}  ^{n}$ and $h_{f\left(  1\right)  }\leq h_{f\left(
2\right)  }\leq\cdots\leq h_{f\left(  n\right)  }$. Furthermore, Lemma
\ref{lem.Dalpha.JtoI} (b) yields $\left\{  j\in\left[  n-1\right]
\ \mid\ h_{f\left(  j\right)  }<h_{f\left(  j+1\right)  }\right\}  \subseteq
D\left(  \alpha\right)  $.
\par
Thus, $\left(  h_{f\left(  1\right)  },h_{f\left(  2\right)  },\ldots
,h_{f\left(  n\right)  }\right)  $ is an $n$-tuple in $\left\{  1,2,3,\ldots
\right\}  ^{n}$ which satisfies $h_{f\left(  1\right)  }\leq h_{f\left(
2\right)  }\leq\cdots\leq h_{f\left(  n\right)  }$ and $\left\{  j\in\left[
n-1\right]  \ \mid\ h_{f\left(  j\right)  }<h_{f\left(  j+1\right)  }\right\}
\subseteq D\left(  \alpha\right)  $. In other words,%
\begin{align*}
&  \left(  h_{f\left(  1\right)  },h_{f\left(  2\right)  },\ldots,h_{f\left(
n\right)  }\right) \\
&  \in\left\{  \left(  i_{1},i_{2},\ldots,i_{n}\right)  \in\left\{
1,2,3,\ldots\right\}  ^{n}\ \mid\ i_{1}\leq i_{2}\leq\cdots\leq i_{n}\right.
\\
&  \ \ \ \ \ \ \ \ \ \ \left.  \text{and }\left\{  j\in\left[  n-1\right]
\ \mid\ i_{j}<i_{j+1}\right\}  \subseteq D\left(  \alpha\right)  \right\}  .
\end{align*}
In light of (\ref{pf.prop.SMalpha.D.I=}), this rewrites as $\left(
h_{f\left(  1\right)  },h_{f\left(  2\right)  },\ldots,h_{f\left(  n\right)
}\right)  \in\mathcal{I}$. Qed.}. Hence, we can define a map $\Psi
:\mathcal{J}\rightarrow\mathcal{I}$ by setting%
\[
\left(  \Psi\left(  h_{1},h_{2},\ldots,h_{\ell}\right)  =\left(  h_{f\left(
1\right)  },h_{f\left(  2\right)  },\ldots,h_{f\left(  n\right)  }\right)
\ \ \ \ \ \ \ \ \ \ \text{for every }\left(  h_{1},h_{2},\ldots,h_{\ell
}\right)  \in\mathcal{J}\right)  .
\]
Consider this $\Psi$.

Now, $\Phi\circ\Psi=\operatorname*{id}$\ \ \ \ \footnote{\textit{Proof.} Let
$\left(  h_{1},h_{2},\ldots,h_{\ell}\right)  \in\mathcal{J}$. For every
$i\in\left[  \ell\right]  $, we have $f\left(  s_{i}\right)  =i$ (by
(\ref{pf.prop.Malpha.D.f.7})) and thus $h_{f\left(  s_{i}\right)  }=h_{i}$.
Now,%
\begin{align*}
\left(  \Phi\circ\Psi\right)  \left(  h_{1},h_{2},\ldots,h_{\ell}\right)   &
=\Phi\left(  \underbrace{\Psi\left(  h_{1},h_{2},\ldots,h_{\ell}\right)
}_{=\left(  h_{f\left(  1\right)  },h_{f\left(  2\right)  },\ldots,h_{f\left(
n\right)  }\right)  }\right)  =\Phi\left(  h_{f\left(  1\right)  },h_{f\left(
2\right)  },\ldots,h_{f\left(  n\right)  }\right) \\
&  =\left(  h_{f\left(  s_{1}\right)  },h_{f\left(  s_{2}\right)  }%
,\ldots,h_{f\left(  s_{\ell}\right)  }\right)  \ \ \ \ \ \ \ \ \ \ \left(
\text{by the definition of }\Phi\right) \\
&  =\left(  h_{1},h_{2},\ldots,h_{\ell}\right)
\end{align*}
(since $h_{f\left(  s_{i}\right)  }=h_{i}$ for every $i\in\left[  \ell\right]
$).
\par
Now, forget that we fixed $\left(  h_{1},h_{2},\ldots,h_{\ell}\right)  $. We
thus have shown that $\left(  \Phi\circ\Psi\right)  \left(  h_{1},h_{2}%
,\ldots,h_{\ell}\right)  =\left(  h_{1},h_{2},\ldots,h_{\ell}\right)  $ for
every $\left(  h_{1},h_{2},\ldots,h_{\ell}\right)  \in\mathcal{J}$. In other
words, $\Phi\circ\Psi=\operatorname*{id}$, qed.} and $\Psi\circ\Phi
=\operatorname*{id}$ \ \ \ \footnote{\textit{Proof.} For every $\left(
i_{1},i_{2},\ldots,i_{n}\right)  \in\mathcal{I}$, we have%
\begin{align*}
\left(  \Psi\circ\Phi\right)  \left(  i_{1},i_{2},\ldots,i_{n}\right)   &
=\Psi\left(  \underbrace{\Phi\left(  i_{1},i_{2},\ldots,i_{n}\right)
}_{\substack{=\left(  i_{s_{1}},i_{s_{2}},\ldots,i_{s_{\ell}}\right)
\\\text{(by the definition of }\Phi\text{)}}}\right)  =\Psi\left(  i_{s_{1}%
},i_{s_{2}},\ldots,i_{s_{\ell}}\right) \\
&  =\left(  i_{s_{f\left(  1\right)  }},i_{s_{f\left(  2\right)  }}%
,\ldots,i_{s_{f\left(  n\right)  }}\right)  \ \ \ \ \ \ \ \ \ \ \left(
\text{by the definition of }\Psi\right) \\
&  =\left(  i_{1},i_{2},\ldots,i_{n}\right)  \ \ \ \ \ \ \ \ \ \ \left(
\text{by (\ref{pf.prop.SMalpha.D.sf.1})}\right)  .
\end{align*}
In other words, $\Psi\circ\Phi=\operatorname*{id}$, qed.}. Hence, the maps
$\Phi$ and $\Psi$ are mutually inverse. Thus, the map $\Phi$ is a bijection.
In other words, the map
\[
\mathcal{I}\rightarrow\mathcal{J},\ \ \ \ \ \ \ \ \ \ \left(  i_{1}%
,i_{2},\ldots,i_{n}\right)  \mapsto\left(  i_{s_{1}},i_{s_{2}},\ldots
,i_{s_{\ell}}\right)
\]
is a bijection\footnote{since $\Phi$ is the map
\[
\mathcal{I}\rightarrow\mathcal{J},\ \ \ \ \ \ \ \ \ \ \left(  i_{1}%
,i_{2},\ldots,i_{n}\right)  \mapsto\left(  i_{s_{1}},i_{s_{2}},\ldots
,i_{s_{\ell}}\right)
\]
(by the definition of $\Phi$)}.

Now, for every $\left(  i_{1},i_{2},\ldots,i_{n}\right)  \in\mathcal{I}$, we
have
\begin{equation}
x_{i_{1}}x_{i_{2}}\cdots x_{i_{n}}=x_{i_{s_{1}}}^{\alpha_{1}}x_{i_{s_{2}}%
}^{\alpha_{2}}\cdots x_{i_{s_{\ell}}}^{\alpha_{\ell}}
\label{pf.prop.SMalpha.D.f-1.3}%
\end{equation}
\footnote{\textit{Proof of (\ref{pf.prop.SMalpha.D.f-1.3}):} Let $\left(
i_{1},i_{2},\ldots,i_{n}\right)  \in\mathcal{I}$. Then,%
\begin{align*}
x_{i_{1}}x_{i_{2}}\cdots x_{i_{n}}  &  =\prod_{k\in\left[  n\right]  }%
x_{i_{k}}=\prod_{j\in\left[  \ell\right]  }\underbrace{\prod_{\substack{k\in
\left[  n\right]  ;\\f\left(  k\right)  =j}}}_{=\prod_{k\in f^{-1}\left(
j\right)  }}\underbrace{x_{i_{k}}}_{\substack{=x_{i_{s_{j}}}\\\text{(because
(\ref{pf.prop.SMalpha.D.sf.1}) shows that}\\i_{k}=i_{s_{f\left(  k\right)  }%
}=i_{s_{j}}\text{ (since }f\left(  k\right)  =j\text{))}}%
}\ \ \ \ \ \ \ \ \ \ \left(  \text{since }f\left(  k\right)  \in\left[
\ell\right]  \text{ for every }k\in\left[  n\right]  \right) \\
&  =\prod_{j\in\left[  \ell\right]  }\underbrace{\prod_{k\in f^{-1}\left(
j\right)  }x_{i_{s_{j}}}}_{\substack{=x_{i_{s_{j}}}^{\left\vert f^{-1}\left(
j\right)  \right\vert }=x_{i_{s_{j}}}^{\alpha_{j}}\\\text{(by
(\ref{pf.prop.Malpha.D.f-1.2}))}}}=\prod_{j\in\left[  \ell\right]
}x_{i_{s_{j}}}^{\alpha_{j}}=x_{i_{s_{1}}}^{\alpha_{1}}x_{i_{s_{2}}}%
^{\alpha_{2}}\cdots x_{i_{s_{\ell}}}^{\alpha_{\ell}}.
\end{align*}
This proves (\ref{pf.prop.SMalpha.D.f-1.3}).}. But
(\ref{pf.prop.SMalpha.D.Malpha=}) becomes%
\begin{align*}
\sum_{i_{1}\leq i_{2}\leq\cdots\leq i_{\ell}}
x_{i_{1}}^{\alpha_{1}}x_{i_{2}}^{\alpha_{2}}\cdots x_{i_{\ell}}^{\alpha_{\ell}}
&  =\sum_{\left(  j_{1},j_{2},\ldots,j_{\ell}\right)
\in\mathcal{J}}x_{j_{1}}^{\alpha_{1}}x_{j_{2}}^{\alpha_{2}}\cdots x_{j_{\ell}%
}^{\alpha_{\ell}}=\sum_{\left(  i_{1},i_{2},\ldots,i_{n}\right)
\in\mathcal{I}}\underbrace{x_{i_{s_{1}}}^{\alpha_{1}}x_{i_{s_{2}}}^{\alpha
_{2}}\cdots x_{i_{s_{\ell}}}^{\alpha_{\ell}}}_{\substack{=x_{i_{1}}x_{i_{2}%
}\cdots x_{i_{n}}\\\text{(by (\ref{pf.prop.SMalpha.D.f-1.3}))}}}\\
&  \ \ \ \ \ \ \ \ \ \ \left(
\begin{array}
[c]{c}%
\text{here, we have substituted }\left(  i_{s_{1}},i_{s_{2}},\ldots
,i_{s_{\ell}}\right)  \text{ for }\left(  j_{1},j_{2},\ldots,j_{\ell}\right)
\text{ in the}\\
\text{sum, since the map }\mathcal{I}\rightarrow\mathcal{J},\ \left(
i_{1},i_{2},\ldots,i_{n}\right)  \mapsto\left(  i_{s_{1}},i_{s_{2}}%
,\ldots,i_{s_{\ell}}\right) \\
\text{is a bijection}%
\end{array}
\right) \\
&  =\sum_{\left(  i_{1},i_{2},\ldots,i_{n}\right)  \in\mathcal{I}}x_{i_{1}%
}x_{i_{2}}\cdots x_{i_{n}}=\sum_{\substack{i_{1}\leq i_{2}\leq\cdots\leq
i_{n};\\\left\{  j\in\left[  n-1\right]  \ \mid\ i_{j}<i_{j+1}\right\}
\subseteq D\left(  \alpha\right)  }}x_{i_{1}}x_{i_{2}}\cdots x_{i_{n}}%
\end{align*}
(by (\ref{pf.prop.SMalpha.D.RHS=})). This proves Proposition
\ref{prop.SMalpha.D}.
\end{proof}

Let us now prove the main statements made in Example
\ref{exam.antipode.Gammaw} (b):

\begin{proposition}
\label{prop.exam.antipode.Gammaw.b}Let $\alpha=\left(  \alpha_{1},\alpha
_{2},\ldots,\alpha_{\ell}\right)  $ be a composition of a nonnegative integer
$n$. Then,%
\begin{align*}
S\left(  M_{\alpha}\right)   &  =\left(  -1\right)  ^{\ell}\sum_{i_{1}\geq
i_{2}\geq\cdots\geq i_{\ell}}x_{i_{1}}^{\alpha_{1}}x_{i_{2}}^{\alpha_{2}%
}\cdots x_{i_{\ell}}^{\alpha_{\ell}}=\left(  -1\right)  ^{\ell}\sum_{i_{1}\leq
i_{2}\leq\cdots\leq i_{\ell}}x_{i_{1}}^{\alpha_{\ell}}x_{i_{2}}^{\alpha
_{\ell-1}}\cdots x_{i_{\ell}}^{\alpha_{1}}\\
&  =\left(  -1\right)  ^{\ell}\sum_{\substack{\gamma\text{ is a composition of
}n;\\D\left(  \gamma\right)  \subseteq D\left(  \left(  \alpha_{\ell}%
,\alpha_{\ell-1},\ldots,\alpha_{1}\right)  \right)  }}M_{\gamma}.
\end{align*}

\end{proposition}

\begin{proof}
[Proof of Proposition \ref{prop.exam.antipode.Gammaw.b}.]Let $E=\left\{
1,2,\ldots,\ell\right\}  $. Thus, $\left\vert E\right\vert =\ell$.

Let $w:\left\{  1,2,\ldots,\ell\right\}  \rightarrow\left\{  1,2,3,\ldots
\right\}  $ be the map sending every $i$ to $\alpha_{i}$. Thus, $w$ is a map
from $\left\{  1,2,\ldots,\ell\right\}  $ to $\left\{  1,2,3,\ldots\right\}
$. In other words, $w$ is a map from $E$ to $\left\{  1,2,3,\ldots\right\}  $
(since $E=\left\{  1,2,\ldots,\ell\right\}  $). Now, the definition of $w$
yields $w\left(  i\right)  =\alpha_{i}$ for every $i\in\left\{  1,2,\ldots
,\ell\right\}  $. In other words, $\alpha_{i}=w\left(  i\right)  $ for every
$i\in\left\{  1,2,\ldots,\ell\right\}  $. 

Let $<_{1}$ be the restriction of the standard relation $<$ on $\mathbb{Z}$ to
the subset $E$. (Thus, two elements $e$ and $f$ of $E$ satisfy $e<_{1}f$ if
and only if $e<f$.) Let $>_{1}$ be the opposite relation of $<_{1}$. (Thus,
two elements $e$ and $f$ of $E$ satisfy $e>_{1}f$ if and only if $f<_{1}e$.)
Let ${\mathbf{E}}=\left(  E,<_{1},>_{1}\right)  $. Then, Proposition
\ref{prop.example.Gamma.b} (a) shows that $\mathbf{E}$ is a special double
poset. The double poset $\mathbf{E}$ is special, thus
semispecial\footnote{since every special double poset is semispecial}, and
therefore tertispecial\footnote{since every semispecial double poset is
tertispecial}. Hence, Theorem \ref{thm.antipode.Gammaw} (applied to $>_{1}$
instead of $<_{2}$) yields%
\begin{align}
S\left(  \Gamma\left(  \left(  E,<_{1},>_{1}\right)  ,w\right)  \right)   &
=\underbrace{\left(  -1\right)  ^{\left\vert E\right\vert }}%
_{\substack{=\left(  -1\right)  ^{\ell}\\\text{(since }\left\vert E\right\vert
=\ell\text{)}}}\Gamma\left(  \left(  E,>_{1},>_{1}\right)  ,w\right)
\nonumber\\
&  =\left(  -1\right)  ^{\ell}\Gamma\left(  \left(  E,>_{1},>_{1}\right)
,w\right)  . \label{pf.prop.exam.antipode.Gammaw.b.1}%
\end{align}

We have%
\[
\alpha=\left(  \alpha_{1},\alpha_{2},\ldots,\alpha_{\ell}\right)  =\left(
w\left(  1\right)  ,w\left(  2\right)  ,\ldots,w\left(  \ell\right)  \right)
\]
(since $\alpha_{i}=w\left(  i\right)  $ for every $i\in\left\{  1,2,\ldots
,\ell\right\}  $). Hence, Proposition \ref{prop.example.Gamma.b} (b) shows
that $\alpha$ is a composition and satisfies $\Gamma\left(  {\mathbf{E}%
},w\right)  =M_{\alpha}$. Thus, $M_{\alpha}=\Gamma\left(
\underbrace{{\mathbf{E}}}_{=\left(  E,<_{1},>_{1}\right)  },w\right)
=\Gamma\left(  \left(  E,<_{1},>_{1}\right)  ,w\right)  $. Applying the map
$S$ to both sides of this equality, we obtain%
\begin{equation}
S\left(  M_{\alpha}\right)  =S\left(  \Gamma\left(  \left(  E,<_{1}%
,>_{1}\right)  ,w\right)  \right)  =\left(  -1\right)  ^{\ell}\Gamma\left(
\left(  E,>_{1},>_{1}\right)  ,w\right)
\label{pf.prop.exam.antipode.Gammaw.b.2}%
\end{equation}
(by (\ref{pf.prop.exam.antipode.Gammaw.b.1})).

Let ${\mathbf{E}}^{\prime}=\left(  E,>_{1},>_{1}\right)  $. Then, Proposition
\ref{prop.example.Gamma.Sb} (a) shows that $\mathbf{E}^{\prime}$ is a special
double poset. Proposition \ref{prop.example.Gamma.Sb} (b) furthermore shows
that%
\[
\Gamma\left(  {\mathbf{E}}^{\prime},w\right)  =\sum_{i_{1}\geq i_{2}\geq
\cdots\geq i_{\ell}}x_{i_{1}}^{\alpha_{1}}x_{i_{2}}^{\alpha_{2}}\cdots
x_{i_{\ell}}^{\alpha_{\ell}}.
\]
Now, (\ref{pf.prop.exam.antipode.Gammaw.b.2}) becomes%
\begin{align}
S\left(  M_{\alpha}\right)   &  =\left(  -1\right)  ^{\ell}\Gamma\left(
\underbrace{\left(  E,>_{1},>_{1}\right)  }_{=\mathbf{E}^{\prime}},w\right)
=\left(  -1\right)  ^{\ell}\underbrace{\Gamma\left(  \mathbf{E}^{\prime
},w\right)  }_{=\sum_{i_{1}\geq i_{2}\geq\cdots\geq i_{\ell}}x_{i_{1}}%
^{\alpha_{1}}x_{i_{2}}^{\alpha_{2}}\cdots x_{i_{\ell}}^{\alpha_{\ell}}%
}\nonumber\\
&  =\left(  -1\right)  ^{\ell}\sum_{i_{1}\geq i_{2}\geq\cdots\geq i_{\ell}%
}x_{i_{1}}^{\alpha_{1}}x_{i_{2}}^{\alpha_{2}}\cdots x_{i_{\ell}}^{\alpha
_{\ell}}\label{pf.prop.exam.antipode.Gammaw.b.result1}\\
&  =\left(  -1\right)  ^{\ell}\underbrace{\sum_{i_{\ell}\geq i_{\ell-1}%
\geq\cdots\geq i_{1}}}_{=\sum_{i_{1}\leq i_{2}\leq\cdots\leq i_{\ell}}%
}\underbrace{x_{i_{\ell}}^{\alpha_{1}}x_{i_{\ell-1}}^{\alpha_{2}}\cdots
x_{i_{1}}^{\alpha_{\ell}}}_{=x_{i_{1}}^{\alpha_{\ell}}x_{i_{2}}^{\alpha
_{\ell-1}}\cdots x_{i_{\ell}}^{\alpha_{1}}}\nonumber\\
&  \ \ \ \ \ \ \ \ \ \ \left(
\begin{array}
[c]{c}%
\text{here, we have substituted }\left(  i_{\ell},i_{\ell-1},\ldots
,i_{1}\right) \\
\text{for }\left(  i_{1},i_{2},\ldots,i_{\ell}\right)  \text{ in the sum}%
\end{array}
\right) \nonumber\\
&  =\left(  -1\right)  ^{\ell}\sum_{i_{1}\leq i_{2}\leq\cdots\leq i_{\ell}%
}x_{i_{1}}^{\alpha_{\ell}}x_{i_{2}}^{\alpha_{\ell-1}}\cdots x_{i_{\ell}%
}^{\alpha_{1}}. \label{pf.prop.exam.antipode.Gammaw.b.result2}%
\end{align}

Now, recall that $\left(  \alpha_{1},\alpha_{2},\ldots,\alpha_{\ell}\right)
=\alpha$ is a composition of $n$. Thus, $\alpha_{1}+\alpha_{2}+\cdots
+\alpha_{\ell}=n$. Hence, $\alpha_{\ell}+\alpha_{\ell-1}+\cdots+\alpha
_{1}=\alpha_{1}+\alpha_{2}+\cdots+\alpha_{\ell}=n$. Thus, $\left(
\alpha_{\ell},\alpha_{\ell-1},\ldots,\alpha_{1}\right)  $ is a composition of
$n$. Therefore, Corollary \ref{cor.Dalpha.SM} (applied to $\left(
\alpha_{\ell},\alpha_{\ell-1},\ldots,\alpha_{1}\right)  $ instead of $\alpha$)
yields
\begin{equation}
\sum_{\substack{i_{1}\leq i_{2}\leq\cdots\leq i_{n};\\\left\{  j\in\left[
n-1\right]  \ \mid\ i_{j}<i_{j+1}\right\}  \subseteq D\left(  \left(
\alpha_{\ell},\alpha_{\ell-1},\ldots,\alpha_{1}\right)  \right)  }}x_{i_{1}%
}x_{i_{2}}\cdots x_{i_{n}}=\sum_{\substack{\beta\text{ is a composition of
}n;\\D\left(  \beta\right)  \subseteq D\left(  \left(  \alpha_{\ell}%
,\alpha_{\ell-1},\ldots,\alpha_{1}\right)  \right)  }}M_{\beta}.
\label{pf.prop.exam.antipode.Gammaw.b.8}%
\end{equation}
(Here, we are using the notations of Definition \ref{def.k} and Definition
\ref{def.Dalpha}.) On the other hand, Proposition \ref{prop.SMalpha.D}
(applied to $\left(  \alpha_{\ell},\alpha_{\ell-1},\ldots,\alpha_{1}\right)  $
and $\alpha_{\ell+1-k}$ instead of $\alpha$ and $\alpha_{k}$) yields%
\begin{align*}
\sum_{i_{1}\leq i_{2}\leq\cdots\leq i_{\ell}}x_{i_{1}}^{\alpha_{\ell}}%
x_{i_{2}}^{\alpha_{\ell-1}}\cdots x_{i_{\ell}}^{\alpha_{1}}  &  =\sum
_{\substack{i_{1}\leq i_{2}\leq\cdots\leq i_{n};\\\left\{  j\in\left[
n-1\right]  \ \mid\ i_{j}<i_{j+1}\right\}  \subseteq D\left(  \left(
\alpha_{\ell},\alpha_{\ell-1},\ldots,\alpha_{1}\right)  \right)  }}x_{i_{1}%
}x_{i_{2}}\cdots x_{i_{n}}\\
&  =\sum_{\substack{\beta\text{ is a composition of }n;\\D\left(
\beta\right)  \subseteq D\left(  \left(  \alpha_{\ell},\alpha_{\ell-1}%
,\ldots,\alpha_{1}\right)  \right)  }}M_{\beta}\ \ \ \ \ \ \ \ \ \ \left(
\text{by (\ref{pf.prop.exam.antipode.Gammaw.b.8})}\right)  .
\end{align*}
Hence, (\ref{pf.prop.exam.antipode.Gammaw.b.result2}) becomes%
\[
S\left(  M_{\alpha}\right)  =\left(  -1\right)  ^{\ell}\underbrace{\sum
_{i_{1}\leq i_{2}\leq\cdots\leq i_{\ell}}x_{i_{1}}^{\alpha_{\ell}}x_{i_{2}%
}^{\alpha_{\ell-1}}\cdots x_{i_{\ell}}^{\alpha_{1}}}_{=\sum_{\substack{\beta
\text{ is a composition of }n;\\D\left(  \beta\right)  \subseteq D\left(
\left(  \alpha_{\ell},\alpha_{\ell-1},\ldots,\alpha_{1}\right)  \right)
}}M_{\beta}}=\left(  -1\right)  ^{\ell}\sum_{\substack{\beta\text{ is a
composition of }n;\\D\left(  \beta\right)  \subseteq D\left(  \left(
\alpha_{\ell},\alpha_{\ell-1},\ldots,\alpha_{1}\right)  \right)  }}M_{\beta}.
\]
This equality, combined with (\ref{pf.prop.exam.antipode.Gammaw.b.result1})
and (\ref{pf.prop.exam.antipode.Gammaw.b.result2}), proves Proposition
\ref{prop.exam.antipode.Gammaw.b}.
\end{proof}

\subsection{The antipode of $F_{\alpha}$}

Recall how we defined a power series $F_{\alpha}\in\operatorname*{QSym}$ for
every composition $\alpha$ in Definition \ref{def.Falpha}. We shall now study
the antipode of this $F_{\alpha}$. First, let us introduce an operation on compositions:

\begin{definition}
\label{def.rev}Let $\alpha=\left(  \alpha_{1},\alpha_{2},\ldots,\alpha_{\ell
}\right)  $ be a composition of a nonnegative integer $n$. Thus, $\left(
\alpha_{1},\alpha_{2},\ldots,\alpha_{\ell}\right)  $ is a composition of $n$.
In other words, $\alpha_{1},\alpha_{2},\ldots,\alpha_{\ell}$ are positive
integers, and their sum is $\alpha_{1}+\alpha_{2}+\cdots+\alpha_{\ell}=n$.
Hence, $\alpha_{\ell}+\alpha_{\ell-1}+\cdots+\alpha_{1}=\alpha_{1}+\alpha
_{2}+\cdots+\alpha_{\ell}=n$. Thus, $\left(  \alpha_{\ell},\alpha_{\ell
-1},\ldots,\alpha_{1}\right)  $ is a composition of $n$ as well. We denote
this composition $\left(  \alpha_{\ell},\alpha_{\ell-1},\ldots,\alpha
_{1}\right)  $ by $\operatorname*{rev}\alpha$, and we call it the
\textit{reversal} of $\alpha$.
\end{definition}

Let us pause to see how this allows rewriting part of Proposition
\ref{prop.exam.antipode.Gammaw.b}:

\begin{proposition}
\label{prop.exam.antipode.Gammaw.b.rev}Let $\alpha$ be a composition of a
nonnegative integer $n$. Let $\ell\in\mathbb{N}$ be such that $\alpha$ is an
$\ell$-tuple. Consider the map $D:\operatorname*{Comp}\nolimits_{n}%
\rightarrow\mathcal{P}\left(  \left[  n-1\right]  \right)  $ defined in
Definition \ref{def.Dalpha.D}. Then,%
\[
S\left(  M_{\alpha}\right)  =\left(  -1\right)  ^{\ell}\sum_{\substack{\gamma
\text{ is a composition of }n;\\D\left(  \gamma\right)  \subseteq D\left(
\operatorname*{rev}\alpha\right)  }}M_{\gamma}.
\]

\end{proposition}

\begin{proof}
[Proof of Proposition \ref{prop.exam.antipode.Gammaw.b.rev}.]Write $\alpha$ in
the form $\alpha=\left(  \alpha_{1},\alpha_{2},\ldots,\alpha_{\ell}\right)  $
(this is possible since $\alpha$ is an $\ell$-tuple). Then, the definition of
$\operatorname*{rev}\alpha$ yields $\operatorname*{rev}\alpha=\left(
\alpha_{\ell},\alpha_{\ell-1},\ldots,\alpha_{1}\right)  $. But Proposition
\ref{prop.exam.antipode.Gammaw.b} yields%
\[
S\left(  M_{\alpha}\right)  =\left(  -1\right)  ^{\ell}\sum_{\substack{\gamma
\text{ is a composition of }n;\\D\left(  \gamma\right)  \subseteq D\left(
\left(  \alpha_{\ell},\alpha_{\ell-1},\ldots,\alpha_{1}\right)  \right)
}}M_{\gamma}=\left(  -1\right)  ^{\ell}\sum_{\substack{\gamma\text{ is a
composition of }n;\\D\left(  \gamma\right)  \subseteq D\left(
\operatorname*{rev}\alpha\right)  }}M_{\gamma}%
\]
(since $\left(  \alpha_{\ell},\alpha_{\ell-1},\ldots,\alpha_{1}\right)
=\operatorname*{rev}\alpha$). This proves Proposition
\ref{prop.exam.antipode.Gammaw.b.rev}.
\end{proof}

Proposition \ref{prop.exam.antipode.Gammaw.b.rev} is easily seen to be
equivalent to \cite[Theorem 5.1.11]{Reiner}.\footnote{Indeed, the $\ell\left(
\alpha\right)  $ in \cite[Theorem 5.1.11]{Reiner} is precisely what we call
$\ell$ in Proposition \ref{prop.exam.antipode.Gammaw.b.rev}, whereas the
condition \textquotedblleft$\gamma$ coarsens $\operatorname*{rev}\left(
\alpha\right)  $\textquotedblright\ in \cite[Theorem 5.1.11]{Reiner} is
equivalent to \textquotedblleft$D\left(  \gamma\right)  \subseteq D\left(
\operatorname*{rev}\alpha\right)  $\textquotedblright.}

Next, let us define the conjugate of a composition (following \cite[Definition
5.2.14]{Reiner}):

\begin{definition}
\label{def.conjugate}Let $\alpha$ be a composition of a nonnegative integer
$n$. We shall use the notations that were used in Proposition
\ref{prop.Dalpha.comp}.

We know that $\operatorname*{rev}\alpha$ is a composition of $n$; in other
words, $\operatorname*{rev}\alpha\in\operatorname*{Comp}\nolimits_{n}$. Hence,
$D\left(  \operatorname*{rev}\alpha\right)  \in\mathcal{P}\left(  \left[
n-1\right]  \right)  $. Thus, $\left[  n-1\right]  \setminus D\left(
\operatorname*{rev}\alpha\right)  \in\mathcal{P}\left(  \left[  n-1\right]
\right)  $. Hence, $\operatorname*{comp}\left(  \left[  n-1\right]  \setminus
D\left(  \operatorname*{rev}\alpha\right)  \right)  \in\operatorname*{Comp}%
\nolimits_{n}$. In other words, $\operatorname*{comp}\left(  \left[
n-1\right]  \setminus D\left(  \operatorname*{rev}\alpha\right)  \right)  $ is
a composition of $n$.

We denote this composition $\operatorname*{comp}\left(  \left[  n-1\right]
\setminus D\left(  \operatorname*{rev}\alpha\right)  \right)  $ by
$\omega\left(  \alpha\right)  $, and call it the \textit{conjugate} of
$\alpha$.
\end{definition}

\begin{remark}
\label{rmk.conjugate.D}Let $\alpha$ be a composition of a nonnegative integer
$n$. Then,%
\[
D\left(  \omega\left(  \alpha\right)  \right)  =\left[  n-1\right]  \setminus
D\left(  \operatorname*{rev}\alpha\right)
\]
(where we are using the notations that were used in Proposition
\ref{prop.Dalpha.comp}).
\end{remark}

\begin{proof}
[Proof of Remark \ref{rmk.conjugate.D}.]Proposition \ref{prop.Dalpha.comp}
shows that the maps $D$ and $\operatorname*{comp}$ are mutually inverse. Thus,
$D\circ\operatorname*{comp}=\operatorname*{id}$. Now, the definition of
$\omega\left(  \alpha\right)  $ shows that $\omega\left(  \alpha\right)
=\operatorname*{comp}\left(  \left[  n-1\right]  \setminus D\left(
\operatorname*{rev}\alpha\right)  \right)  $. Applying the map $D$ to both
sides of this equality, we obtain%
\begin{align*}
D\left(  \omega\left(  \alpha\right)  \right)   &  =D\left(
\operatorname*{comp}\left(  \left[  n-1\right]  \setminus D\left(
\operatorname*{rev}\alpha\right)  \right)  \right)  =\underbrace{\left(
D\circ\operatorname*{comp}\right)  }_{=\operatorname*{id}}\left(  \left[
n-1\right]  \setminus D\left(  \operatorname*{rev}\alpha\right)  \right) \\
&  =\left[  n-1\right]  \setminus D\left(  \operatorname*{rev}\alpha\right)  .
\end{align*}
This proves Remark \ref{rmk.conjugate.D}.
\end{proof}

\begin{lemma}
\label{lem.Dalpha.rev}Let $\alpha$ be a composition of a nonnegative integer
$n$. Consider the map $D:\operatorname*{Comp}\nolimits_{n}\rightarrow
\mathcal{P}\left(  \left[  n-1\right]  \right)  $ defined in Definition
\ref{def.Dalpha.D}.

We have $D\left(  \operatorname*{rev}\alpha\right)  =\left\{  n-u\ \mid\ u\in
D\left(  \alpha\right)  \right\}  $.
\end{lemma}

\begin{proof}
[Proof of Lemma \ref{lem.Dalpha.rev}.]Write the composition $\alpha$ in the
form $\left(  \alpha_{1},\alpha_{2},\ldots,\alpha_{\ell}\right)  $. Thus,
$\alpha=\left(  \alpha_{1},\alpha_{2},\ldots,\alpha_{\ell}\right)  $. Hence,
$\left(  \alpha_{1},\alpha_{2},\ldots,\alpha_{\ell}\right)  $ is a composition
of $n$ (since $\alpha$ is a composition of $n$). Therefore, $\alpha_{1}%
,\alpha_{2},\ldots,\alpha_{\ell}$ are positive integers whose sum is $n$.
Thus, $\alpha_{1}+\alpha_{2}+\cdots+\alpha_{\ell}=n$.

For every $i\in\left\{  0,1,\ldots,\ell\right\}  $, define a nonnegative
integer $s_{i}$ by
\[
s_{i}=\alpha_{1}+\alpha_{2}+\cdots+\alpha_{i}.
\]
Lemma \ref{lem.Dalpha.s} (c) thus shows that $D\left(  \alpha\right)
=\left\{  s_{1},s_{2},\ldots,s_{\ell-1}\right\}  $.

For every $i\in\left\{  0,1,\ldots,\ell\right\}  $, define a nonnegative
integer $t_{i}$ by
\[
t_{i}=\alpha_{\ell}+\alpha_{\ell-1}+\cdots+\alpha_{\left(  \ell+1\right)
-i}.
\]
Lemma \ref{lem.Dalpha.s} (c) (applied to $\operatorname*{rev}\alpha$,
$\alpha_{\ell+1-k}$ and $t_{i}$ instead of $\alpha$, $\alpha_{k}$ and $s_{i}$)
thus shows that $D\left(  \operatorname*{rev}\alpha\right)  =\left\{
t_{1},t_{2},\ldots,t_{\ell-1}\right\}  $ (since $\operatorname*{rev}%
\alpha=\left(  \alpha_{\ell},\alpha_{\ell-1},\ldots,\alpha_{1}\right)  $ (by
the definition of $\operatorname*{rev}\alpha$)). But each $i\in\left\{
1,2,\ldots,\ell-1\right\}  $ satisfies $t_{i}=n-s_{\ell-i}$%
\ \ \ \ \footnote{\textit{Proof.} Let $i\in\left\{  1,2,\ldots,\ell-1\right\}
$. The definition of $s_{\ell-i}$ yields $s_{\ell-i}=\alpha_{1}+\alpha
_{2}+\cdots+\alpha_{\ell-i}$. But from $\alpha_{1}+\alpha_{2}+\cdots
+\alpha_{\ell}=n$, we obtain%
\[
n=\alpha_{1}+\alpha_{2}+\cdots+\alpha_{\ell}=\underbrace{\left(  \alpha
_{1}+\alpha_{2}+\cdots+\alpha_{\ell-i}\right)  }_{=s_{\ell-i}}%
+\underbrace{\left(  \alpha_{\ell-i+1}+\alpha_{\ell-i+2}+\cdots+\alpha_{\ell
}\right)  }_{\substack{=\alpha_{\ell}+\alpha_{\ell-1}+\cdots+\alpha_{\ell
-i+1}\\=\alpha_{\ell}+\alpha_{\ell-1}+\cdots+\alpha_{\left(  \ell+1\right)
-i}=t_{i}}}=s_{\ell-i}+t_{i}.
\]
Hence, $t_{i}=n-s_{\ell-i}$, qed.}. Hence,
\[
\left(  t_{1},t_{2},\ldots,t_{\ell-1}\right)  =\left(  n-s_{\ell-1}%
,n-s_{\ell-2},\ldots,n-s_{\ell-\left(  \ell-1\right)  }\right)  =\left(
n-s_{\ell-1},n-s_{\ell-2},\ldots,n-s_{1}\right)  ,
\]
so that $\left\{  t_{1},t_{2},\ldots,t_{\ell-1}\right\}  =\left\{
n-s_{\ell-1},n-s_{\ell-2},\ldots,n-s_{1}\right\}  $. Now,%
\begin{align*}
D\left(  \operatorname*{rev}\alpha\right)   &  =\left\{  t_{1},t_{2}%
,\ldots,t_{\ell-1}\right\}  =\left\{  n-s_{\ell-1},n-s_{\ell-2},\ldots
,n-s_{1}\right\} \\
&  =\left\{  n-u\ \mid\ u\in\underbrace{\left\{  s_{\ell-1},s_{\ell-2}%
,\ldots,s_{1}\right\}  }_{=\left\{  s_{1},s_{2},\ldots,s_{\ell-1}\right\}
=D\left(  \alpha\right)  }\right\}  =\left\{  n-u\ \mid\ u\in D\left(
\alpha\right)  \right\}  .
\end{align*}
This proves Lemma \ref{lem.Dalpha.rev}.
\end{proof}

Now, let us prove an analogue of Proposition \ref{prop.example.Gamma.c3}:

\begin{proposition}
\label{prop.example.Gamma.Sc3}Let $\alpha=\left(  \alpha_{1},\alpha_{2}%
,\ldots,\alpha_{\ell}\right)  $ be a composition of a nonnegative integer $n$.
Define a set $D\left(  \alpha\right)  $ as in Definition \ref{def.Dalpha}. Let
$E=\left\{  1,2,\ldots,n\right\}  $. Let $<_{1}$ be the total order on the set
$E$ inherited from $\mathbb{Z}$ (thus, two elements $a$ and $b$ of $E$ satisfy
$a<_{1}b$ if and only if they satisfy $a<b$). Let $>_{1}$ be the opposite
relation of $<_{1}$. (Thus, two elements $e$ and $f$ of $E$ satisfy $e>_{1}f$
if and only if $f<_{1}e$.) Let $<_{2}$ be a strict partial order on the set
$E$ satisfying (\ref{eq.prop.example.Gamma.c1.cond1}) and
(\ref{eq.prop.example.Gamma.c1.cond2}). We shall use the notation introduced
in Definition \ref{def.k} (that is, we shall write $\left[  k\right]  $ for
$\left\{  1,2,\ldots,k\right\}  $ when $k\in\mathbb{Z}$).

\begin{enumerate}
\item[(a)] We have%
\[
\Gamma\left(  \left(  E,>_{1},<_{2}\right)  \right)  =\sum_{\substack{i_{1}%
\geq i_{2}\geq\cdots\geq i_{n};\\i_{j}>i_{j+1}\text{ whenever }j\in\left[
n-1\right]  \setminus D\left(  \alpha\right)  }}x_{i_{1}}x_{i_{2}}\cdots
x_{i_{n}}.
\]

\item[(b)] We have%
\[
\Gamma\left(  \left(  E,>_{1},<_{2}\right)  \right)  =\sum_{\substack{i_{1}%
\leq i_{2}\leq\cdots\leq i_{n};\\i_{j}<i_{j+1}\text{ whenever }j\in\left[
n-1\right]  \setminus D\left(  \operatorname*{rev}\alpha\right)  }}x_{i_{1}%
}x_{i_{2}}\cdots x_{i_{n}}=F_{\omega\left(  \alpha\right)  }.
\]

\end{enumerate}
\end{proposition}

\begin{proof}
[Proof of Proposition \ref{prop.example.Gamma.Sc3}.](a) Let us first observe a
simple fact: If $u$ and $v$ are two elements of $\left\{  1,2,\ldots
,n\right\}  $ such that $u<v$ and $u<_{2}v$, then
\begin{equation}
\left\{  u,u+1,\ldots,v-1\right\}  \not \subseteq D\left(  \alpha\right)
\label{pf.prop.example.Gamma.Sc3.0}%
\end{equation}
\footnote{\textit{Proof of (\ref{pf.prop.example.Gamma.Sc3.0}):} Let $u$ and
$v$ be two elements of $\left\{  1,2,\ldots,n\right\}  $ such that $u<v$ and
$u<_{2}v$. We must prove (\ref{pf.prop.example.Gamma.Sc3.0}).
\par
Indeed, assume (for the sake of contradiction) that $\left\{  u,u+1,\ldots
,v-1\right\}  \subseteq D\left(  \alpha\right)  $.
\par
We have $u\geq1$ (since $u\in\left\{  1,2,\ldots,n\right\}  $) and $v\leq n$
(since $v\in\left\{  1,2,\ldots,n\right\}  $). Also, $u\neq v$ (since $u<v$),
so that $v\neq u$.
\par
Now, let $k\in\left\{  u,u+1,\ldots,v-1\right\}  $ be arbitrary. Thus,
$k\in\left\{  u,u+1,\ldots,v-1\right\}  \subseteq D\left(  \alpha\right)  $.
Hence, (\ref{eq.prop.example.Gamma.c1.cond1}) (applied to $i=k$) shows that
$k+1<_{2}k$.
\par
Now, forget that we fixed $k$. We thus have shown that $k+1<_{2}k$ for each
$k\in\left\{  u,u+1,\ldots,v-1\right\}  $. In other words, we have the
relations $v<_{2}v-1$, $v-1<_{2}v-2$, $\ldots$, $u+1<_{2}u$. Since $\left(
E,<_{2}\right)  $ is a poset (because $<_{2}$ is a strict partial order on
$E$), we can combine these relations into a chain of inequalities:%
\[
v<_{2}v-1<_{2}v-2<_{2}\cdots<_{2}u.
\]
Thus, $v<_{2}u$ (since $v\neq u$). This contradicts $u<_{2}v$ (since $<_{2}$
is a strict partial order on $E$).
\par
This contradiction shows that our assumption (that $\left\{  u,u+1,\ldots
,v-1\right\}  \subseteq D\left(  \alpha\right)  $) was false. Hence, we cannot
have $\left\{  u,u+1,\ldots,v-1\right\}  \subseteq D\left(  \alpha\right)  $.
We thus must have $\left\{  u,u+1,\ldots,v-1\right\}  \not \subseteq D\left(
\alpha\right)  $. This proves (\ref{pf.prop.example.Gamma.Sc3.0}).}.

Let $Z$ be the set of all $\left(  E,>_{1},<_{2}\right)  $-partitions. The
definition of $\Gamma\left(  \left(  E,>_{1},<_{2}\right)  \right)  $ yields%
\[
\Gamma\left(  \left(  E,>_{1},<_{2}\right)  \right)  =\sum_{\pi\text{ is an
}\left(  E,>_{1},<_{2}\right)  \text{-partition}}{\mathbf{x}}_{\pi},
\]
where ${\mathbf{x}}_{\pi}=\prod_{e\in E}x_{\pi\left(  e\right)  }$. Thus,%
\begin{equation}
\Gamma\left(  \left(  E,>_{1},<_{2}\right)  \right)  =\underbrace{\sum
_{\pi\text{ is an }\left(  E,>_{1},<_{2}\right)  \text{-partition}}%
}_{\substack{=\sum_{\pi\in Z}\\\text{(since }Z\text{ is the set of
all}\\\left(  E,>_{1},<_{2}\right)  \text{-partitions)}}}{\mathbf{x}}_{\pi
}=\sum_{\pi\in Z}\mathbf{x}_{\pi}. \label{pf.prop.example.Gamma.Sc3.Gamma=}%
\end{equation}

On the other hand, define a set $W$ by%
\begin{align}
W  &  =\left\{  \left(  i_{1},i_{2},\ldots,i_{n}\right)  \in\left\{
1,2,3,\ldots\right\}  ^{n}\ \mid\ i_{1}\geq i_{2}\geq\cdots\geq i_{n}\right.
\nonumber\\
&  \ \ \ \ \ \ \ \ \ \ \left.  \text{and }\left(  i_{j}>i_{j+1}\text{ whenever
}j\in\left[  n-1\right]  \setminus D\left(  \alpha\right)  \right)  \right\}
. \label{pf.prop.example.Gamma.Sc3.W=}%
\end{align}
Thus, we have the following equality between summation signs:%
\[
\sum_{\left(  i_{1},i_{2},\ldots,i_{n}\right)  \in W}=\sum_{\substack{\left(
i_{1},i_{2},\ldots,i_{n}\right)  \in\left\{  1,2,3,\ldots\right\}
^{n};\\i_{1}\geq i_{2}\geq\cdots\geq i_{n};\\i_{j}>i_{j+1}\text{ whenever
}j\in\left[  n-1\right]  \setminus D\left(  \alpha\right)  }}=\sum
_{\substack{i_{1}\geq i_{2}\geq\cdots\geq i_{n};\\i_{j}>i_{j+1}\text{ whenever
}j\in\left[  n-1\right]  \setminus D\left(  \alpha\right)  }}.
\]

For every $\pi\in Z$, we have $\left(  \pi\left(  1\right)  ,\pi\left(
2\right)  ,\ldots,\pi\left(  n\right)  \right)  \in W$%
\ \ \ \ \footnote{\textit{Proof.} Let $\pi\in Z$. We must show that $\left(
\pi\left(  1\right)  ,\pi\left(  2\right)  ,\ldots,\pi\left(  n\right)
\right)  \in W$.
\par
First, let us see that $\left(  \pi\left(  1\right)  ,\pi\left(  2\right)
,\ldots,\pi\left(  n\right)  \right)  $ is well-defined. We have $\pi\in Z$.
In other words, $\pi$ is an $\left(  E,>_{1},<_{2}\right)  $-partition (since
$Z$ is the set of all $\left(  E,>_{1},<_{2}\right)  $-partitions). Hence,
$\pi$ is a map $E\rightarrow\left\{  1,2,3,\ldots\right\}  $. In other words,
$\pi$ is a map $\left\{  1,2,\ldots,n\right\}  \rightarrow\left\{
1,2,3,\ldots\right\}  $ (since $E=\left\{  1,2,\ldots,n\right\}  $). Thus,
$\left(  \pi\left(  1\right)  ,\pi\left(  2\right)  ,\ldots,\pi\left(
n\right)  \right)  $ is well-defined and an element of $\left\{
1,2,3,\ldots\right\}  ^{n}$.
\par
Recall the definition of an $\left(  E,>_{1},<_{2}\right)  $-partition. This
definition shows that $\pi$ is an $\left(  E,>_{1},<_{2}\right)  $-partition
if and only if it satisfies the following two assertions:
\par
\textit{Assertion }$\mathcal{A}_{1}$\textit{:} Every $e\in E$ and $f\in E$
satisfying $e>_{1}f$ satisfy $\pi\left(  e\right)  \leq\pi\left(  f\right)  $.
\par
\textit{Assertion }$\mathcal{A}_{2}$\textit{:} Every $e\in E$ and $f\in E$
satisfying $e>_{1}f$ and $f<_{2}e$ satisfy $\pi\left(  e\right)  <\pi\left(
f\right)  $.
\par
Thus, $\pi$ satisfies Assertions $\mathcal{A}_{1}$ and $\mathcal{A}_{2}$
(since $\pi$ is an $\left(  E,>_{1},<_{2}\right)  $-partition).
\par
We now shall show that
\begin{equation}
\pi\left(  k\right)  \geq\pi\left(  k+1\right)  \ \ \ \ \ \ \ \ \ \ \text{for
every }k\in\left\{  1,2,\ldots,n-1\right\}
.\label{pf.prop.example.Gamma.Sc3.ZW.1}%
\end{equation}
\par
[\textit{Proof of (\ref{pf.prop.example.Gamma.Sc3.ZW.1}):} Let $k\in\left\{
1,2,\ldots,n-1\right\}  $ be arbitrary. We shall show that $\pi\left(
k\right)  \geq\pi\left(  k+1\right)  $.
\par
We have $k\in\left\{  1,2,\ldots,n-1\right\}  $. Thus, both $k$ and $k+1$
belong to the set $\left\{  1,2,\ldots,n\right\}  $. In other words, both $k$
and $k+1$ belong to the set $E$ (since $E=\left\{  1,2,\ldots,n\right\}  $).
\par
Recall that $<_{1}$ is the total order on the set $E$ inherited from
$\mathbb{Z}$. Hence, $k<_{1}k+1$ (since $k<k+1$).
\par
But $k+1>_{1}k$ holds if and only if $k<_{1}k+1$ (since $>_{1}$ is the
opposite relation of $<_{1}$). Thus, $k+1>_{1}k$ holds (since $k<_{1}k+1$).
Therefore, Assertion $\mathcal{A}_{1}$ (applied to $e=k+1$ and $f=k$) yields
$\pi\left(  k+1\right)  \leq\pi\left(  k\right)  $. Hence, $\pi\left(
k\right)  \geq\pi\left(  k+1\right)  $. This proves
(\ref{pf.prop.example.Gamma.Sc3.ZW.1}).]
\par
Now, (\ref{pf.prop.example.Gamma.Sc3.ZW.1}) shows that $\pi\left(  1\right)
\geq\pi\left(  2\right)  \geq\cdots\geq\pi\left(  n\right)  $.
\par
Next, let us show that
\begin{equation}
\pi\left(  j\right)  >\pi\left(  j+1\right)
\ \ \ \ \ \ \ \ \ \ \text{whenever }j\in\left[  n-1\right]  \setminus D\left(
\alpha\right)  .\label{pf.prop.example.Gamma.Sc3.ZW.2}%
\end{equation}
\par
[\textit{Proof of (\ref{pf.prop.example.Gamma.Sc3.ZW.2}):} Let $j\in\left[
n-1\right]  \setminus D\left(  \alpha\right)  $. Thus, $j\in\left[
n-1\right]  $ and $j\notin D\left(  \alpha\right)  $.
\par
From $j\in\left[  n-1\right]  =\left\{  1,2,\ldots,n-1\right\}  $, we conclude
that both $j$ and $j+1$ belong to the set $\left\{  1,2,\ldots,n\right\}  $.
In other words, both $j$ and $j+1$ belong to the set $E$ (since $E=\left\{
1,2,\ldots,n\right\}  $). We have $j\in\underbrace{\left[  n-1\right]
}_{=\left\{  1,2,\ldots,n-1\right\}  }\setminus D\left(  \alpha\right)
=\left\{  1,2,\ldots,n-1\right\}  \setminus D\left(  \alpha\right)  $, so that
$j<_{2}j+1$ (by (\ref{eq.prop.example.Gamma.c1.cond2}), applied to $i=j$).
\par
Recall that $<_{1}$ is the total order on the set $E$ inherited from
$\mathbb{Z}$. Hence, $j<_{1}j+1$ (since $j<j+1$). But $j+1>_{1}j$ holds if and
only if $j<_{1}j+1$ (since $>_{1}$ is the opposite relation of $<_{1}$). Thus,
$j+1>_{1}j$ holds (since $j<_{1}j+1$). Also, $j<_{2}j+1$. Therefore, Assertion
$\mathcal{A}_{2}$ (applied to $e=j+1$ and $f=j$) yields $\pi\left(
j+1\right)  <\pi\left(  j\right)  $. In other words, $\pi\left(  j\right)
>\pi\left(  j+1\right)  $. This proves (\ref{pf.prop.example.Gamma.Sc3.ZW.2}%
).]
\par
Now, we know that $\left(  \pi\left(  1\right)  ,\pi\left(  2\right)
,\ldots,\pi\left(  n\right)  \right)  $ is an element of $\left\{
1,2,3,\ldots\right\}  ^{n}$ and satisfies $\pi\left(  1\right)  \geq\pi\left(
2\right)  \geq\cdots\geq\pi\left(  n\right)  $ and $\left(  \pi\left(
j\right)  >\pi\left(  j+1\right)  \text{ whenever }j\in\left[  n-1\right]
\setminus D\left(  \alpha\right)  \right)  $. In other words,%
\begin{align*}
\left(  \pi\left(  1\right)  ,\pi\left(  2\right)  ,\ldots,\pi\left(
n\right)  \right)   &  \in\left\{  \left(  i_{1},i_{2},\ldots,i_{n}\right)
\in\left\{  1,2,3,\ldots\right\}  ^{n}\ \mid\ i_{1}\geq i_{2}\geq\cdots\geq
i_{n}\right.  \\
&  \ \ \ \ \ \ \ \ \ \ \left.  \text{and }\left(  i_{j}>i_{j+1}\text{ whenever
}j\in\left[  n-1\right]  \setminus D\left(  \alpha\right)  \right)  \right\}
\\
&  =W.
\end{align*}
}. Hence, we can define a map $\Phi:Z\rightarrow W$ by setting%
\[
\left(  \Phi\left(  \pi\right)  =\left(  \pi\left(  1\right)  ,\pi\left(
2\right)  ,\ldots,\pi\left(  n\right)  \right)  \ \ \ \ \ \ \ \ \ \ \text{for
every }\pi\in Z\right)  .
\]
Consider this map $\Phi$. Thus, $\Phi$ is the map $Z\rightarrow W,\ \pi
\mapsto\left(  \pi\left(  1\right)  ,\pi\left(  2\right)  ,\ldots,\pi\left(
n\right)  \right)  $ (since $\Phi\left(  \pi\right)  =\left(  \pi\left(
1\right)  ,\pi\left(  2\right)  ,\ldots,\pi\left(  n\right)  \right)  $ for
every $\pi\in Z$).

The map $\Phi$ is surjective\footnote{\textit{Proof.} Let $\mathbf{g}\in W$.
We shall show that $\mathbf{g}\in\Phi\left(  Z\right)  $.
\par
We have%
\begin{align*}
\mathbf{g} &  \in W=\left\{  \left(  i_{1},i_{2},\ldots,i_{n}\right)
\in\left\{  1,2,3,\ldots\right\}  ^{n}\ \mid\ i_{1}\geq i_{2}\geq\cdots\geq
i_{n}\right.  \\
&  \ \ \ \ \ \ \ \ \ \ \left.  \text{and }\left(  i_{j}>i_{j+1}\text{ whenever
}j\in\left[  n-1\right]  \setminus D\left(  \alpha\right)  \right)  \right\}
.
\end{align*}
In other words, $\mathbf{g}$ is an $\left(  i_{1},i_{2},\ldots,i_{n}\right)
\in\left\{  1,2,3,\ldots\right\}  ^{n}$ satisfying $i_{1}\geq i_{2}\geq
\cdots\geq i_{n}$ and $\left(  i_{j}>i_{j+1}\text{ whenever }j\in\left[
n-1\right]  \setminus D\left(  \alpha\right)  \right)  $. Consider this
$\left(  i_{1},i_{2},\ldots,i_{n}\right)  $. Thus, $\mathbf{g}=\left(
i_{1},i_{2},\ldots,i_{n}\right)  $.
\par
We have $i_{1}\geq i_{2}\geq\cdots\geq i_{n}$. In other words, for any
$u\in\left\{  1,2,\ldots,n\right\}  $ and $v\in\left\{  1,2,\ldots,n\right\}
$ satisfying $u\leq v$, we have%
\begin{equation}
i_{u}\geq i_{v}.\label{pf.prop.example.Gamma.Sc3.surj.1}%
\end{equation}
Notice also that we have%
\begin{equation}
\left(  i_{j}>i_{j+1}\text{ whenever }j\in\left[  n-1\right]  \setminus
D\left(  \alpha\right)  \right)  .\label{pf.prop.example.Gamma.Sc3.surj.2}%
\end{equation}
\par
Define a map $\pi:\left\{  1,2,\ldots,n\right\}  \rightarrow\left\{
1,2,3,\ldots\right\}  $ by setting%
\[
\left(  \pi\left(  k\right)  =i_{k}\ \ \ \ \ \ \ \ \ \ \text{for every }%
k\in\left\{  1,2,\ldots,n\right\}  \right)  .
\]
Thus, $\pi$ is a map from $E$ to $\left\{  1,2,3,\ldots\right\}  $ (since
$E=\left\{  1,2,\ldots,n\right\}  $).
\par
Recall the definition of an $\left(  E,>_{1},<_{2}\right)  $-partition. This
definition shows that $\pi$ is an $\left(  E,>_{1},<_{2}\right)  $-partition
if and only if it satisfies the following two assertions:
\par
\textit{Assertion }$\mathcal{A}_{1}$\textit{:} Every $e\in E$ and $f\in E$
satisfying $e>_{1}f$ satisfy $\pi\left(  e\right)  \leq\pi\left(  f\right)  $.
\par
\textit{Assertion }$\mathcal{A}_{2}$\textit{:} Every $e\in E$ and $f\in E$
satisfying $e>_{1}f$ and $f<_{2}e$ satisfy $\pi\left(  e\right)  <\pi\left(
f\right)  $.
\par
We shall now show that $\pi$ satisfies Assertions $\mathcal{A}_{1}$ and
$\mathcal{A}_{2}$.
\par
\textit{Proof of Assertion }$\mathcal{A}_{1}$\textit{:} Let $e\in E$ and $f\in
E$ be such that $e>_{1}f$. Recall that $>_{1}$ is the opposite relation of
$<_{1}$. Hence, $e>_{1}f$ holds if and only if $f<_{1}e$. Thus, $f<_{1}e$
(since $e>_{1}f$ holds). Now, recall that $<_{1}$ is the total order on the
set $E$ inherited from $\mathbb{Z}$. Hence, $f<_{1}e$ holds if and only if
$f<e$. Thus, $f<e$ (since $f<_{1}e$ holds). Thus, $f\leq e$. Hence,
(\ref{pf.prop.example.Gamma.Sc3.surj.1}) (applied to $u=f$ and $v=e$) yields
$i_{f}\geq i_{e}$. In other words, $i_{e}\leq i_{f}$. Now, the definition of
$\pi$ yields $\pi\left(  e\right)  =i_{e}$ and $\pi\left(  f\right)  =i_{f}$.
Hence, $\pi\left(  e\right)  =i_{e}\leq i_{f}=\pi\left(  f\right)  $. This
proves Assertion $\mathcal{A}_{1}$.
\par
\textit{Proof of Assertion }$\mathcal{A}_{2}$\textit{:} Let $e\in E$ and $f\in
E$ be such that $e>_{1}f$ and $f<_{2}e$. Thus, $f<e$ (indeed, this can be
proven just as in our proof of Assertion $\mathcal{A}_{1}$).
\par
Both $e$ and $f$ are elements of $E$. In other words, both $e$ and $f$ are
elements of $\left\{  1,2,\ldots,n\right\}  $ (since $E=\left\{
1,2,\ldots,n\right\}  $). We have $f<e$ and $f<_{2}e$. Thus,
(\ref{pf.prop.example.Gamma.Sc3.0}) (applied to $u=f$ and $v=e$) yields
$\left\{  f,f+1,\ldots,e-1\right\}  \not \subseteq D\left(  \alpha\right)  $.
In other words, there exists some $j\in\left\{  f,f+1,\ldots,e-1\right\}  $
such that $j\notin D\left(  \alpha\right)  $. Consider such $j$.
\par
We have $j\in\left\{  f,f+1,\ldots,e-1\right\}  $. Thus, $f\leq j\leq e-1$.
From $j\leq e-1$, we obtain $j+1\leq e$.
\par
From $f\in\left\{  1,2,\ldots,n\right\}  $, we obtain $f\geq1$. From
$e\in\left\{  1,2,\ldots,n\right\}  $, we obtain $e\leq n$. Thus,
$j\in\left\{  f,f+1,\ldots,e-1\right\}  \subseteq\left\{  1,2,\ldots
,n-1\right\}  $ (since $f\geq1$ and $e\leq n$). Hence, both $j$ and $j+1$ are
elements of $\left\{  1,2,\ldots,n\right\}  $.
\par
We have $j\in\left\{  1,2,\ldots,n-1\right\}  =\left[  n-1\right]  $.
Combining this with $j\notin D\left(  \alpha\right)  $, we obtain $j\in\left[
n-1\right]  \setminus D\left(  \alpha\right)  $. Thus, $i_{j}>i_{j+1}$ (by
(\ref{pf.prop.example.Gamma.Sc3.surj.2})). Applying
(\ref{pf.prop.example.Gamma.Sc3.surj.1}) to $u=f$ and $v=j$, we obtain
$i_{f}\geq i_{j}$ (since $f\leq j$). Applying
(\ref{pf.prop.example.Gamma.Sc3.surj.1}) to $u=j+1$ and $v=e$, we obtain
$i_{j+1}\geq i_{e}$ (since $j+1\leq e$). Hence, $i_{f}\geq i_{j}>i_{j+1}\geq
i_{e}$, so that $i_{e}<i_{f}$.
\par
Now, the definition of $\pi$ yields $\pi\left(  e\right)  =i_{e}$ and
$\pi\left(  f\right)  =i_{f}$. Hence, $\pi\left(  e\right)  =i_{e}<i_{f}%
=\pi\left(  f\right)  $. This proves Assertion $\mathcal{A}_{2}$.
\par
Now, we have shown that $\pi$ satisfies Assertions $\mathcal{A}_{1}$ and
$\mathcal{A}_{2}$. Thus, $\pi$ is an $\left(  E,>_{1},<_{2}\right)
$-partition (since we know that $\pi$ is an $\left(  E,>_{1},<_{2}\right)
$-partition if and only if it satisfies Assertions $\mathcal{A}_{1}$ and
$\mathcal{A}_{2}$). In other words, $\pi\in Z$ (since $Z$ is the set of all
$\left(  E,>_{1},<_{2}\right)  $-partitions).
\par
Now, the definition of $\Phi$ yields%
\begin{align*}
\Phi\left(  \pi\right)   &  =\left(  \pi\left(  1\right)  ,\pi\left(
2\right)  ,\ldots,\pi\left(  n\right)  \right)  =\left(  i_{1},i_{2}%
,\ldots,i_{n}\right)  \\
&  \ \ \ \ \ \ \ \ \ \ \left(  \text{since }\pi\left(  k\right)  =i_{k}\text{
for every }k\in\left\{  1,2,\ldots,n\right\}  \right)  \\
&  =\mathbf{g}.
\end{align*}
Hence, $\mathbf{g}=\Phi\left(  \underbrace{\pi}_{\in Z}\right)  \in\Phi\left(
Z\right)  $.
\par
Let us now forget that we fixed $\mathbf{g}$. We thus have shown that
$\mathbf{g}\in\Phi\left(  Z\right)  $ for every $\mathbf{g}\in W$. In other
words, $W\subseteq\Phi\left(  Z\right)  $. In other words, the map $\Phi$ is
surjective, qed.} and injective\footnote{\textit{Proof.} Let $\pi_{1}$ and
$\pi_{2}$ be two elements of $Z$ such that $\Phi\left(  \pi_{1}\right)
=\Phi\left(  \pi_{2}\right)  $. We shall prove that $\pi_{1}=\pi_{2}$.
\par
The definition of $\Phi$ yields $\Phi\left(  \pi_{1}\right)  =\left(  \pi
_{1}\left(  1\right)  ,\pi_{1}\left(  2\right)  ,\ldots,\pi_{1}\left(
n\right)  \right)  $ and $\Phi\left(  \pi_{2}\right)  =\left(  \pi_{2}\left(
1\right)  ,\pi_{2}\left(  2\right)  ,\ldots,\pi_{2}\left(  n\right)  \right)
$. Thus,%
\[
\left(  \pi_{1}\left(  1\right)  ,\pi_{1}\left(  2\right)  ,\ldots,\pi
_{1}\left(  n\right)  \right)  =\Phi\left(  \pi_{1}\right)  =\Phi\left(
\pi_{2}\right)  =\left(  \pi_{2}\left(  1\right)  ,\pi_{2}\left(  2\right)
,\ldots,\pi_{2}\left(  n\right)  \right)  .
\]
In other words,%
\[
\pi_{1}\left(  i\right)  =\pi_{2}\left(  i\right)
\ \ \ \ \ \ \ \ \ \ \text{for every }i\in\left\{  1,2,\ldots,n\right\}  .
\]
In other words,%
\[
\pi_{1}\left(  i\right)  =\pi_{2}\left(  i\right)
\ \ \ \ \ \ \ \ \ \ \text{for every }i\in E
\]
(since $E=\left\{  1,2,\ldots,n\right\}  $). In other words, $\pi_{1}=\pi_{2}%
$.
\par
Now, forget that we fixed $\pi_{1}$ and $\pi_{2}$. We thus have shown that if
$\pi_{1}$ and $\pi_{2}$ are two elements of $Z$ such that $\Phi\left(  \pi
_{1}\right)  =\Phi\left(  \pi_{2}\right)  $, then $\pi_{1}=\pi_{2}$. In other
words, the map $\Phi$ is injective, qed.}. In other words, the map $\Phi$ is
bijective, i.e., a bijection. In other words, the map $Z\rightarrow
W,\ \pi\mapsto\left(  \pi\left(  1\right)  ,\pi\left(  2\right)  ,\ldots
,\pi\left(  n\right)  \right)  $ is a bijection (since $\Phi$ is the map
$Z\rightarrow W,\ \pi\mapsto\left(  \pi\left(  1\right)  ,\pi\left(  2\right)
,\ldots,\pi\left(  n\right)  \right)  $).

Every $\pi\in Z$ satisfies%
\begin{equation}
\mathbf{x}_{\pi}=x_{\pi\left(  1\right)  }x_{\pi\left(  2\right)  }\cdots
x_{\pi\left(  n\right)  } \label{pf.prop.example.Gamma.Sc3.x}%
\end{equation}
\footnote{\textit{Proof of (\ref{pf.prop.example.Gamma.Sc3.x}):} Let $\pi\in
Z$. The definition of $\mathbf{x}_{\pi}$ yields%
\begin{align*}
{\mathbf{x}}_{\pi}  &  =\prod_{e\in E}x_{\pi\left(  e\right)  }=\prod
_{e\in\left\{  1,2,\ldots,n\right\}  }x_{\pi\left(  e\right)  }%
\ \ \ \ \ \ \ \ \ \ \left(  \text{since }E=\left\{  1,2,\ldots,n\right\}
\right) \\
&  =x_{\pi\left(  1\right)  }x_{\pi\left(  2\right)  }\cdots x_{\pi\left(
n\right)  }.
\end{align*}
This proves (\ref{pf.prop.example.Gamma.Sc3.x}).}. Now,
(\ref{pf.prop.example.Gamma.Sc3.Gamma=}) becomes%
\begin{align*}
&  \Gamma\left(  \left(  E,>_{1},<_{2}\right)  \right) \\
&  =\sum_{\pi\in Z}\underbrace{\mathbf{x}_{\pi}}_{\substack{=x_{\pi\left(
1\right)  }x_{\pi\left(  2\right)  }\cdots x_{\pi\left(  n\right)
}\\\text{(by (\ref{pf.prop.example.Gamma.Sc3.x}))}}}=\sum_{\pi\in Z}%
x_{\pi\left(  1\right)  }x_{\pi\left(  2\right)  }\cdots x_{\pi\left(
n\right)  }\\
&  =\underbrace{\sum_{\left(  i_{1},i_{2},\ldots,i_{n}\right)  \in W}}%
_{=\sum_{\substack{i_{1}\geq i_{2}\geq\cdots\geq i_{n};\\i_{j}>i_{j+1}\text{
whenever }j\in\left[  n-1\right]  \setminus D\left(  \alpha\right)  }%
}}x_{i_{1}}x_{i_{2}}\cdots x_{i_{n}}\\
&  \ \ \ \ \ \ \ \ \ \ \left(
\begin{array}
[c]{c}%
\text{here, we have substituted }\left(  i_{1},i_{2},\ldots,i_{n}\right)
\text{ for }\left(  \pi\left(  1\right)  ,\pi\left(  2\right)  ,\ldots
,\pi\left(  n\right)  \right)  \text{,}\\
\text{since the map }Z\rightarrow W,\ \pi\mapsto\left(  \pi\left(  1\right)
,\pi\left(  2\right)  ,\ldots,\pi\left(  n\right)  \right) \\
\text{is a bijection}%
\end{array}
\right) \\
&  =\sum_{\substack{i_{1}\geq i_{2}\geq\cdots\geq i_{n};\\i_{j}>i_{j+1}\text{
whenever }j\in\left[  n-1\right]  \setminus D\left(  \alpha\right)  }%
}x_{i_{1}}x_{i_{2}}\cdots x_{i_{n}}.
\end{align*}
This proves Proposition \ref{prop.example.Gamma.Sc3} (a).

(b) Write the composition $\alpha$ in the form $\left(  \alpha_{1},\alpha
_{2},\ldots,\alpha_{\ell}\right)  $. Thus, $\alpha=\left(  \alpha_{1}%
,\alpha_{2},\ldots,\alpha_{\ell}\right)  $. The definition of
$\operatorname*{rev}\alpha$ yields $\operatorname*{rev}\alpha=\left(
\alpha_{\ell},\alpha_{\ell-1},\ldots,\alpha_{1}\right)  $.

Proposition \ref{prop.example.Gamma.Sc3} (a) yields%
\begin{align}
\Gamma\left(  \left(  E,>_{1},<_{2}\right)  \right)   &  =\sum
_{\substack{i_{1}\geq i_{2}\geq\cdots\geq i_{n};\\i_{j}>i_{j+1}\text{ whenever
}j\in\left[  n-1\right]  \setminus D\left(  \alpha\right)  }}x_{i_{1}}%
x_{i_{2}}\cdots x_{i_{n}}\nonumber\\
&  =\underbrace{\sum_{\substack{i_{n}\geq i_{n-1}\geq\cdots\geq i_{1}%
;\\i_{n+1-j}>i_{n+1-\left(  j+1\right)  }\text{ whenever }j\in\left[
n-1\right]  \setminus D\left(  \alpha\right)  }}}_{\substack{=\sum
_{\substack{i_{1}\leq i_{2}\leq\cdots\leq i_{n};\\i_{n+1-j}>i_{n+1-\left(
j+1\right)  }\text{ whenever }j\in\left[  n-1\right]  \setminus D\left(
\alpha\right)  }}\\\text{(since the condition }i_{n}\geq i_{n-1}\geq\cdots\geq
i_{1}\text{ is equivalent}\\\text{to the condition }i_{1}\leq i_{2}\leq
\cdots\leq i_{n}\text{)}}}\underbrace{x_{i_{n}}x_{i_{n-1}}\cdots x_{i_{1}}%
}_{=x_{i_{1}}x_{i_{2}}\cdots x_{i_{n}}}\nonumber\\
&  \ \ \ \ \ \ \ \ \ \ \left(
\begin{array}
[c]{c}%
\text{here, we have renamed the summation index}\\
\left(  i_{1},i_{2},\ldots,i_{n}\right)  \text{ as }\left(  i_{n}%
,i_{n-1},\ldots,i_{1}\right)
\end{array}
\right) \nonumber\\
&  =\sum_{\substack{i_{1}\leq i_{2}\leq\cdots\leq i_{n};\\i_{n+1-j}%
>i_{n+1-\left(  j+1\right)  }\text{ whenever }j\in\left[  n-1\right]
\setminus D\left(  \alpha\right)  }}x_{i_{1}}x_{i_{2}}\cdots x_{i_{n}}.
\label{pf.prop.example.Gamma.Sc3.b.1}%
\end{align}

But for every $n$-tuple $\left(  i_{1},i_{2},\ldots,i_{n}\right)  \in\left\{
1,2,3,\ldots\right\}  ^{n}$, the condition
\begin{equation}
\left(  i_{n+1-j}>i_{n+1-\left(  j+1\right)  }\text{ whenever }j\in\left[
n-1\right]  \setminus D\left(  \alpha\right)  \right)
\label{pf.prop.example.Gamma.Sc3.b.eqL}%
\end{equation}
is equivalent to the condition
\begin{equation}
\left(  i_{j}<i_{j+1}\text{ whenever }j\in\left[  n-1\right]  \setminus
D\left(  \operatorname*{rev}\alpha\right)  \right)
\label{pf.prop.example.Gamma.Sc3.b.eqR}%
\end{equation}
\footnote{\textit{Proof.} Let $\left(  i_{1},i_{2},\ldots,i_{n}\right)
\in\left\{  1,2,3,\ldots\right\}  ^{n}$ be an $n$-tuple. We must show that the
condition (\ref{pf.prop.example.Gamma.Sc3.b.eqL}) is equivalent to the
condition (\ref{pf.prop.example.Gamma.Sc3.b.eqR}).
\par
Lemma \ref{lem.Dalpha.rev} yields $D\left(  \operatorname*{rev}\alpha\right)
=\left\{  n-u\ \mid\ u\in D\left(  \alpha\right)  \right\}  $.
\par
Let us first prove that (\ref{pf.prop.example.Gamma.Sc3.b.eqL}) implies
(\ref{pf.prop.example.Gamma.Sc3.b.eqR}).
\par
\textit{Proof that (\ref{pf.prop.example.Gamma.Sc3.b.eqL}) implies
(\ref{pf.prop.example.Gamma.Sc3.b.eqR}):} Assume that
(\ref{pf.prop.example.Gamma.Sc3.b.eqL}) holds. Now, let $j\in\left[
n-1\right]  \setminus D\left(  \operatorname*{rev}\alpha\right)  $. Thus,
$j\in\left[  n-1\right]  $ but $j\notin D\left(  \operatorname*{rev}%
\alpha\right)  $.
\par
Now, we shall prove that $n-j\notin D\left(  \alpha\right)  $. Indeed, assume
the contrary. Thus, $n-j\in D\left(  \alpha\right)  $, so that $n-\left(
n-j\right)  \in\left\{  n-u\ \mid\ u\in D\left(  \alpha\right)  \right\}  $
(since $n-\left(  n-j\right)  $ has the form $n-u$ for some $u\in D\left(
\alpha\right)  $ (namely, for $u=n-j$)). Therefore, $n-\left(  n-j\right)
\in\left\{  n-u\ \mid\ u\in D\left(  \alpha\right)  \right\}  =D\left(
\operatorname*{rev}\alpha\right)  $. Since $n-\left(  n-j\right)  =j$, this
rewrites as $j\in D\left(  \operatorname*{rev}\alpha\right)  $. This
contradicts $j\notin D\left(  \operatorname*{rev}\alpha\right)  $. This
contradiction proves that our assumption was wrong. Hence, $n-j\notin D\left(
\alpha\right)  $ is proven.
\par
Also, $j\in\left[  n-1\right]  $, so that $n-j\in\left[  n-1\right]  $.
Combining this with $n-j\notin D\left(  \alpha\right)  $, we obtain
$n-j\in\left[  n-1\right]  \setminus D\left(  \alpha\right)  $. Now, we
assumed that (\ref{pf.prop.example.Gamma.Sc3.b.eqL}) holds. Hence,
(\ref{pf.prop.example.Gamma.Sc3.b.eqL}) (applied to $n-j$ instead of $j$)
shows that $i_{n+1-\left(  n-j\right)  }>i_{n+1-\left(  \left(  n-j\right)
+1\right)  }$ (since $n-j\in\left[  n-1\right]  \setminus D\left(
\alpha\right)  $). In other words, $i_{j+1}>i_{j}$ (since $n+1-\left(
n-j\right)  =j+1$ and $n+1-\left(  \left(  n-j\right)  +1\right)  =j$). In
other words, $i_{j}<i_{j+1}$.
\par
Now, forget that we fixed $j$. We thus have shown that $i_{j}<i_{j+1}$
whenever $j\in\left[  n-1\right]  \setminus D\left(  \operatorname*{rev}%
\alpha\right)  $. In other words, (\ref{pf.prop.example.Gamma.Sc3.b.eqR})
holds.
\par
Now, forget that we assumed that (\ref{pf.prop.example.Gamma.Sc3.b.eqL})
holds. We thus have shown that if (\ref{pf.prop.example.Gamma.Sc3.b.eqL})
holds, then (\ref{pf.prop.example.Gamma.Sc3.b.eqR}) holds. In other words, we
have proven that (\ref{pf.prop.example.Gamma.Sc3.b.eqL}) implies
(\ref{pf.prop.example.Gamma.Sc3.b.eqR}).
\par
Let us now prove that (\ref{pf.prop.example.Gamma.Sc3.b.eqR}) implies
(\ref{pf.prop.example.Gamma.Sc3.b.eqL}).
\par
\textit{Proof that (\ref{pf.prop.example.Gamma.Sc3.b.eqR}) implies
(\ref{pf.prop.example.Gamma.Sc3.b.eqL}):} Assume that
(\ref{pf.prop.example.Gamma.Sc3.b.eqR}) holds. Now, let $j\in\left[
n-1\right]  \setminus D\left(  \alpha\right)  $. Thus, $j\in\left[
n-1\right]  $ but $j\notin D\left(  \alpha\right)  $.
\par
Now, we shall prove that $n-j\notin D\left(  \operatorname*{rev}\alpha\right)
$. Indeed, assume the contrary. Thus, $n-j\in D\left(  \operatorname*{rev}%
\alpha\right)  $. Hence, $n-j\in D\left(  \operatorname*{rev}\alpha\right)
=\left\{  n-u\ \mid\ u\in D\left(  \alpha\right)  \right\}  $. In other words,
$n-j$ has the form $n-u$ for some $u\in D\left(  \alpha\right)  $. Consider
this $u$. Thus, $n-j=n-u$, so that $j=u$. Hence, $j=u\in D\left(
\alpha\right)  $. This contradicts $j\notin D\left(  \alpha\right)  $. This
contradiction proves that our assumption was wrong. Hence, $n-j\notin D\left(
\operatorname*{rev}\alpha\right)  $ is proven.
\par
Also, $n-j\in\left[  n-1\right]  $ (since $j\in\left[  n-1\right]  $).
Combining this with $n-j\notin D\left(  \operatorname*{rev}\alpha\right)  $,
we obtain $n-j\in\left[  n-1\right]  \setminus D\left(  \operatorname*{rev}%
\alpha\right)  $. Now, we assumed that (\ref{pf.prop.example.Gamma.Sc3.b.eqR})
holds. Hence, (\ref{pf.prop.example.Gamma.Sc3.b.eqR}) (applied to $n-j$
instead of $j$) shows that $i_{n-j}<i_{\left(  n-j\right)  +1}$ (since
$n-j\in\left[  n-1\right]  \setminus D\left(  \operatorname*{rev}%
\alpha\right)  $). In other words, $i_{\left(  n-j\right)  +1}>i_{n-j}$.
Hence,%
\begin{align*}
i_{n+1-j}  &  =i_{\left(  n-j\right)  +1}\ \ \ \ \ \ \ \ \ \ \left(
\text{since }n+1-j=\left(  n-j\right)  +1\right) \\
&  >i_{n-j}=i_{n+1-\left(  j+1\right)  }\ \ \ \ \ \ \ \ \ \ \left(
\text{since }n-j=n+1-\left(  j+1\right)  \right)  .
\end{align*}
\par
Now, forget that we fixed $j$. We thus have shown that $i_{n+1-j}%
>i_{n+1-\left(  j+1\right)  }$ whenever $j\in\left[  n-1\right]  \setminus
D\left(  \alpha\right)  $. In other words,
(\ref{pf.prop.example.Gamma.Sc3.b.eqL}) holds.
\par
Now, forget that we assumed that (\ref{pf.prop.example.Gamma.Sc3.b.eqR})
holds. We thus have shown that if (\ref{pf.prop.example.Gamma.Sc3.b.eqR})
holds, then (\ref{pf.prop.example.Gamma.Sc3.b.eqL}) holds. In other words, we
have proven that (\ref{pf.prop.example.Gamma.Sc3.b.eqR}) implies
(\ref{pf.prop.example.Gamma.Sc3.b.eqL}).
\par
We have now shown the following two facts:
\par
\begin{itemize}
\item (\ref{pf.prop.example.Gamma.Sc3.b.eqL}) implies
(\ref{pf.prop.example.Gamma.Sc3.b.eqR}).
\par
\item (\ref{pf.prop.example.Gamma.Sc3.b.eqR}) implies
(\ref{pf.prop.example.Gamma.Sc3.b.eqL}).
\end{itemize}
\par
Combining these two facts, we conclude that the condition
(\ref{pf.prop.example.Gamma.Sc3.b.eqL}) is equivalent to the condition
(\ref{pf.prop.example.Gamma.Sc3.b.eqR}). Qed.}.

Now, (\ref{pf.prop.example.Gamma.Sc3.b.1}) becomes%
\begin{align}
\Gamma\left(  \left(  E,>_{1},<_{2}\right)  \right)   &  =\underbrace{\sum
_{\substack{i_{1}\leq i_{2}\leq\cdots\leq i_{n};\\i_{n+1-j}>i_{n+1-\left(
j+1\right)  }\text{ whenever }j\in\left[  n-1\right]  \setminus D\left(
\alpha\right)  }}}_{\substack{=\sum_{\substack{i_{1}\leq i_{2}\leq\cdots\leq
i_{n};\\i_{j}<i_{j+1}\text{ whenever }j\in\left[  n-1\right]  \setminus
D\left(  \operatorname*{rev}\alpha\right)  }}\\\text{(because for every
}n\text{-tuple }\left(  i_{1},i_{2},\ldots,i_{n}\right)  \in\left\{
1,2,3,\ldots\right\}  ^{n}\text{,}\\\text{the condition }\left(
i_{n+1-j}>i_{n+1-\left(  j+1\right)  }\text{ whenever }j\in\left[  n-1\right]
\setminus D\left(  \alpha\right)  \right)  \\\text{is equivalent to the
condition }\left(  i_{j}<i_{j+1}\text{ whenever }j\in\left[  n-1\right]
\setminus D\left(  \operatorname*{rev}\alpha\right)  \right)  \text{)}%
}}x_{i_{1}}x_{i_{2}}\cdots x_{i_{n}}\nonumber\\
&  =\sum_{\substack{i_{1}\leq i_{2}\leq\cdots\leq i_{n};\\i_{j}<i_{j+1}\text{
whenever }j\in\left[  n-1\right]  \setminus D\left(  \operatorname*{rev}%
\alpha\right)  }}x_{i_{1}}x_{i_{2}}\cdots x_{i_{n}}.
\label{pf.prop.example.Gamma.Sc3.b.result1}%
\end{align}

Now, recall that $\omega\left(  \alpha\right)  $ is a composition of $n$
(since $\alpha$ is a composition of $n$). Hence, the definition of
$F_{\omega\left(  \alpha\right)  }$ yields%
\[
F_{\omega\left(  \alpha\right)  }=\sum_{\substack{i_{1}\leq i_{2}\leq
\cdots\leq i_{n};\\i_{j}<i_{j+1}\text{ whenever }j\in D\left(  \omega\left(
\alpha\right)  \right)  }}x_{i_{1}}x_{i_{2}}\cdots x_{i_{n}}=\sum
_{\substack{i_{1}\leq i_{2}\leq\cdots\leq i_{n};\\i_{j}<i_{j+1}\text{ whenever
}j\in\left[  n-1\right]  \setminus D\left(  \operatorname*{rev}\alpha\right)
}}x_{i_{1}}x_{i_{2}}\cdots x_{i_{n}}%
\]
(since $D\left(  \omega\left(  \alpha\right)  \right)  =\left[  n-1\right]
\setminus D\left(  \operatorname*{rev}\alpha\right)  $ (by Remark
\ref{rmk.conjugate.D})). Comparing this with
(\ref{pf.prop.example.Gamma.Sc3.b.result1}), we obtain $\Gamma\left(  \left(
E,>_{1},<_{2}\right)  \right)  =F_{\omega\left(  \alpha\right)  }$. This
equality and (\ref{pf.prop.example.Gamma.Sc3.b.result1}) together complete the
proof of Proposition \ref{prop.example.Gamma.Sc3} (b).
\end{proof}

Finally, we can derive a formula for the antipode of $F_{\alpha}$:

\begin{proposition}
\label{prop.SF}Let $\alpha$ be a composition of a nonnegative integer $n$.
Then, $S\left(  F_{\alpha}\right)  =\left(  -1\right)  ^{n}F_{\omega\left(
\alpha\right)  }$.
\end{proposition}

Proposition \ref{prop.SF} is \cite[(5.2.7)]{Reiner} (although \cite{Reiner} uses
the notation $L_{\alpha}$ for what we call $F_{\alpha}$).

\begin{proof}
[Proof of Proposition \ref{prop.SF}.]Write the composition $\alpha$ in the
form $\left(  \alpha_{1},\alpha_{2},\ldots,\alpha_{\ell}\right)  $. Thus,
$\alpha=\left(  \alpha_{1},\alpha_{2},\ldots,\alpha_{\ell}\right)  $. Define a
set $D\left(  \alpha\right)  $ as in Definition \ref{def.Dalpha}.

Let $E=\left\{  1,2,\ldots,n\right\}  $. Thus, $\left\vert E\right\vert =n$.

Let $<_{1}$ be the total order on the set $E$ inherited from $\mathbb{Z}$
(thus, two elements $a$ and $b$ of $E$ satisfy $a<_{1}b$ if and only if they
satisfy $a<b$). Let $>_{1}$ be the opposite relation of $<_{1}$. (Thus, two
elements $e$ and $f$ of $E$ satisfy $e>_{1}f$ if and only if $f<_{1}e$.)

Proposition \ref{prop.example.Gamma.c1} shows that there exists a total order
$<_{2}$ on the set $E$ satisfying (\ref{eq.prop.example.Gamma.c1.cond1}) and
(\ref{eq.prop.example.Gamma.c1.cond2}). Consider such a $<_{2}$. Thus, $<_{2}$
is a total order, hence a strict partial order. Hence, Proposition
\ref{prop.example.Gamma.c3} yields%
\[
\Gamma\left(  \left(  E,<_{1},<_{2}\right)  \right)  =\sum_{\substack{i_{1}%
\leq i_{2}\leq\cdots\leq i_{n};\\i_{j}<i_{j+1}\text{ whenever }j\in D\left(
\alpha\right)  }}x_{i_{1}}x_{i_{2}}\cdots x_{i_{n}}.
\]
Comparing this with (\ref{eq.def.Falpha.1}), we obtain $\Gamma\left(  \left(
E,<_{1},<_{2}\right)  \right)  =F_{\alpha}$.

We shall use the notation introduced in Definition \ref{def.k} (that is, we
shall write $\left[  k\right]  $ for $\left\{  1,2,\ldots,k\right\}  $ when
$k\in\mathbb{Z}$). Proposition \ref{prop.example.Gamma.Sc3} (b) yields%
\[
\Gamma\left(  \left(  E,>_{1},<_{2}\right)  \right)  =\sum_{\substack{i_{1}%
\leq i_{2}\leq\cdots\leq i_{n};\\i_{j}<i_{j+1}\text{ whenever }j\in\left[
n-1\right]  \setminus D\left(  \operatorname*{rev}\alpha\right)  }}x_{i_{1}%
}x_{i_{2}}\cdots x_{i_{n}}=F_{\omega\left(  \alpha\right)  }.
\]

The double poset $\left(  E,<_{1},<_{2}\right)  $ is special (since $<_{2}$ is
a total order), thus semispecial\footnote{since every special double poset is
semispecial}, and therefore tertispecial\footnote{since every semispecial
double poset is tertispecial}. Hence, Corollary \ref{cor.antipode.Gamma}
yields $S\left(  \Gamma\left(  \left(  E,<_{1},<_{2}\right)  \right)  \right)
=\left(  -1\right)  ^{\left\vert E\right\vert }\Gamma\left(  \left(
E,>_{1},<_{2}\right)  \right)  $. Since $\Gamma\left(  \left(  E,<_{1}%
,<_{2}\right)  \right)  =F_{\alpha}$, $\left\vert E\right\vert =n$ and
$\Gamma\left(  \left(  E,>_{1},<_{2}\right)  \right)  =F_{\omega\left(
\alpha\right)  }$, this rewrites as follows: $S\left(  F_{\alpha}\right)
=\left(  -1\right)  ^{n}F_{\omega\left(  \alpha\right)  }$. This proves
Proposition \ref{prop.SF}.
\end{proof}

\end{footnotesize}
\end{verlong}


\begin{thebibliography}{99999999}                                                                                         %

\bibitem[Abe77]{Abe-HA}Eiichi Abe, \textit{Hopf Algebras}, CUP 1977.


\bibitem[BenSag14]{BenSag}Carolina Benedetti, Bruce Sagan,
\textit{Antipodes and involutions},
Journal of Combinatorial Theory, Series A,
Volume 148, May 2017, pp. 275--315.
\newline\url{http://doi.org/10.1016/j.jcta.2016.12.005}
\newline A preprint is available as arXiv:1410.5023v4.\newline
\url{http://arxiv.org/abs/1410.5023v4}

\bibitem[BBSSZ13]{BBSSZ}Chris Berg, Nantel Bergeron, Franco Saliola, Luis
Serrano, Mike Zabrocki, \textit{A lift of the Schur and Hall-Littlewood bases
to non-commutative symmetric functions}, Canadian Journal of Mathematics 66
(2014), pp. 525--565.\newline
\url{https://doi.org/10.4153/CJM-2013-013-0}
\newline Also available as arXiv:1208.5191v3.\newline
\url{http://arxiv.org/abs/1208.5191v3}

\bibitem[DNR01]{Dasca-HA}Sorin D\u{a}sc\u{a}lescu, Constantin
N\u{a}st\u{a}sescu, \c{S}erban Raianu, \textit{Hopf Algebras}, Marcel Dekker 2001.

\bibitem[Ehrenb96]{Ehrenb96}Richard Ehrenborg,
\textit{On Posets and Hopf Algebras},
Advances in Mathematics \textbf{119} (1996), pp. 1--25.
\newline\url{https://doi.org/10.1006/aima.1996.0026}

\bibitem[Foissy13]{Foissy13}
Lo\"ic Foissy,
\textit{Plane posets, special posets, and permutations},
Advances in Mathematics 240 (2013), pp. 24--60.
\newline\url{http://doi.org/10.1016/j.aim.2013.03.007}
\newline A preprint is available as arXiv:1109.1101v4.\newline
\url{https://arxiv.org/abs/1109.1101v4}

\bibitem[Fresse14]{Fresse-Op}
Benoit Fresse, \textit{Homotopy of operads and
Grothendieck--Teichm\"{u}ller groups, Part I},
Mathematical Surveys and Monographs \textbf{217},
AMS 2017.
A preprint can be found at
\url{http://web.archive.org/web/20160213160358/http://math.univ-lille1.fr/~fresse/OperadHomotopyBook/OperadHomotopy-FirstVolume.pdf} .

\bibitem[Gessel84]{Gessel}Ira M. Gessel, \textit{Multipartite
P-partitions and Inner Products of Skew Schur Functions}, Contemporary
Mathematics, vol. 34, 1984, pp. 289--301.\newline%
\url{http://people.brandeis.edu/~gessel/homepage/papers/multipartite.pdf}

\bibitem[Gessel15]{Gessel-Ppar}Ira M. Gessel, \textit{A Historical Survey of
$P$-Partitions}, to be published in Richard Stanley's 70th Birthday Festschrift,
arXiv:1506.03508v1.\newline
\url{http://arxiv.org/abs/1506.03508v1}
\newline (Published, possibly in a modified version, in:
Patricia Hersh, Thomas Lam, Pavlo Pylyavskyy, Victor Reiner (eds.),
\textit{The mathematical legacy of Richard P. Stanley},
AMS, Providence (RI) 2016.)


\bibitem[Grin14]{Gri-dimm}Darij Grinberg,
\textit{Dual immaculate creation operators and a dendriform algebra
structure on the quasisymmetric functions},
Canad. J. Math. \textbf{69} (2017), pp. 21--53,
\href{http://arxiv.org/abs/1410.0079v7}{arXiv:1410.0079v7}.

\bibitem[Grin16a]{Gri-nbc}Darij Grinberg,
\textit{Generalized Whitney formulas for
broken circuits in ambigraphs and
matroids},
\href{https://arxiv.org/abs/1604.03063v2}{arXiv:1604.03063v2}.

\bibitem[Grin16b]{Gri-extabs}
Darij Grinberg, \textit{Double posets and the antipode of $\QSym$
(extended abstract)},
extended abstract submitted to \href{https://sites.google.com/site/fpsac2017/}{FPSAC 2017}.
\url{http://www.cip.ifi.lmu.de/~grinberg/algebra/fpsac2017.pdf}

\bibitem[GriRei14]{Reiner}Darij Grinberg, Victor Reiner, \textit{Hopf algebras
in Combinatorics}, July 28, 2020,
\href{https://arxiv.org/abs/1409.8356v7}{arXiv:1409.8356v7}.\newline
\url{http://www.cip.ifi.lmu.de/~grinberg/algebra/HopfComb-sols.pdf}

\bibitem[HaGuKi10]{HGK}Michiel Hazewinkel, Nadiya Gubareni, V. V. Kirichenko,
\textit{Algebras, Rings and Modules: Lie Algebras and Hopf Algebras},
AMS 2010.


\bibitem[Joch13]{Joch}Katharina Jochemko, \textit{Order polynomials and
P\'{o}lya's enumeration theorem},
The Electronic Journal of Combinatorics 21(2) (2014), P2.52.
See also
\texttt{\href{http://arxiv.org/abs/1310.0838v2}{arXiv:1310.0838v2}}
for a preprint.


\bibitem[Malve93]{Malve-Thesis}Claudia Malvenuto, \textit{Produits et
coproduits des fonctions quasi-sym\'{e}triques et de l'alg\`{e}bre des
descentes}, thesis, defended November 1993.\newline
\url{http://www1.mat.uniroma1.it/people/malvenuto/Thesis.pdf}

\bibitem[MalReu95]{Mal-Reu-dua}
Claudia Malvenuto, Christophe Reutenauer,
\textit{Duality between Quasi-Symmetric Functions and the Solomon
Descent Algebra},
Journal of Algebra \textbf{177} (1995), pp. 967--982.
\newline\url{https://doi.org/10.1006/jabr.1995.1336}

\bibitem[MalReu98]{Mal-Reu}Claudia Malvenuto, Christophe Reutenauer,
\textit{Plethysm and conjugation of quasi-symmetric functions}, Discrete
Mathematics, Volume 193, Issues 1--3, 28 November 1998, pp.
225--233.\newline
\url{https://doi.org/10.1016/S0012-365X(98)00142-3}

\bibitem[MalReu09]{Mal-Reu-DP}Claudia Malvenuto, Christophe Reutenauer,
\textit{A self paired Hopf algebra on double posets and
a Littlewood-Richardson rule},
Journal of Combinatorial Theory, Series A 118 (2011), pp. 1322--1333.
\newline
\url{https://doi.org/10.1016/j.jcta.2010.10.010} . \\
A preprint version appeared as
\href{http://arxiv.org/abs/0905.3508v1}{arXiv:0905.3508v1}. \\
See \url{https://www.cip.ifi.lmu.de/~grinberg/algebra/doubleposets-errata.pdf}
for errata.

\bibitem[Manchon04]{Manchon-HA}Dominique Manchon, \textit{Hopf algebras, from
basics to applications to renormalization}, Comptes Rendus des Rencontres
Mathematiques de Glanon 2001 (published in 2003), arXiv:math/0408405v2.\newline%
\url{http://arxiv.org/abs/math/0408405v2}

\bibitem[Montg93]{Montg-Hopf}Susan Montgomery, \textit{Hopf Algebras and their
Actions on Rings}, Regional Conference Series in Mathematics \textbf{82},
AMS 1993.

\bibitem[NovThi05]{Nov-Thi} Jean-Christophe Novelli, Jean-Yves Thibon,
\textit{Hopf algebras and dendriform structures arising from parking functions},
Fundamenta Mathematicae 193 (2007), 189--241. A preprint also appears
on arXiv as \href{http://arxiv.org/abs/math/0511200v1}{arXiv:math/0511200v1}.

\bibitem[Sage16]{SageMath}The Sage Developers, \textit{SageMath,
the Sage Mathematics Software System} (Version 7.4), 2016.
\url{http://www.sagemath.org}

\bibitem[Stan11]{Stanley-EC1}Richard P. Stanley, \textit{Enumerative
Combinatorics, volume 1}, Cambridge University Press, 2nd edition 2011.
\newline \url{http://math.mit.edu/~rstan/ec/ec1/}

\bibitem[Stan99]{Stanley-EC2}Richard P. Stanley, \textit{Enumerative
Combinatorics, volume 2}, Cambridge University Press, 2nd edition 2024.
\newline \url{https://math.mit.edu/~rstan/ec/}

\bibitem[Stan71]{Stanley-Thes}Richard P. Stanley, \textit{Ordered Structures and
Partitions}, Memoirs of the American Mathematical Society, No. 119, American
Mathematical Society, Providence, R.I., 1972. \newline
\url{https://math.mit.edu/~rstan/pubs/pubfiles/9.pdf}

\bibitem[Stan84]{Stanley-Peck}Richard P. Stanley,
\textit{Quotients of Peck posets}, Order, 1 (1984), pp. 29--34. \newline
\url{http://dedekind.mit.edu/~rstan/pubs/pubfiles/60.pdf}

\bibitem[Sweed69]{Sweedler-HA}Moss E. Sweedler, \textit{Hopf Algebras},
W. A. Benjamin 1969.

\end{thebibliography}
\end{document}